\documentclass[11pt]{article}

\usepackage[margin=1in]{geometry}
\usepackage{amsmath,amsthm,amssymb,amsfonts,float,bbm}
\usepackage{listings}
\usepackage{pdfpages}
\usepackage{mathtools,slashed}
\usepackage{tikz}
\usetikzlibrary{patterns}
\usepackage{pgfplots}
\mathtoolsset{showonlyrefs}

\makeatletter
\newcommand{\subjclass}[2][2010]{%
  \let\@oldtitle\@title%
  \gdef\@title{\@oldtitle\footnotetext{#1 \emph{Mathematics subject classification.} #2}}%
}
\makeatother

\usepackage[colorlinks=true, pdfstartview=FitV, linkcolor=blue,citecolor=blue, urlcolor=blue]{hyperref}

\usepackage[abbrev,lite,nobysame]{amsrefs}
\usepackage{color}

\usepackage{esint}

\usepackage{mathtools,enumitem,mathrsfs}

\usepackage[compact]{titlesec}

\usepackage{comment}

\newcommand{\ep}{\epsilon}
\newcommand{\eps}{\epsilon}

\newcommand{\na}{\nabla}
\newcommand{\grad}{\nabla}

\newcommand{\norm}[1]{\left\| #1 \right\|}
\newcommand{\abs}[1]{\left| #1 \right|}
\newcommand{\set}[1]{\left\{ #1 \right\}}
\newcommand{\brak}[1]{\left\langle #1 \right\rangle} 

\newcommand{\rr}{\mathbb{R}}

\newcommand{\red}[1]{\textcolor{red}{#1}}

\DeclareMathOperator{\supp}{\mathrm{supp}}

\usepackage{bbm}

\newcommand{\p}{\partial}

\newtheorem{theorem}{Theorem}[section]
\newtheorem{proposition}[theorem]{Proposition}
\newtheorem{corollary}[theorem]{Corollary}
\newtheorem{lemma}[theorem]{Lemma}
\newtheorem*{lemma*}{Lemma}
\newtheorem{claim}[theorem]{Claim}

\theoremstyle{definition}
\newtheorem{definition}[theorem]{Definition}
\newtheorem{remark}[theorem]{Remark}

\def\jacob #1{\textcolor{red}{#1}}
\def\ZZ{\mathbb{Z}}

\def\nn{\nonumber\\}
\def\paren#1{\left(#1\right)}
\def\abs#1{\left|#1\right|}
\newcommand{\n}{\ensuremath{\nonumber}}
\newcommand{\pa}{\ensuremath{\partial}}
\newcommand{\sameer}[1]{\textcolor{magenta}{#1}}
\newcommand{\siming}[1]{}

\def\bb{\bold{a}}

\newcommand{\nnorm}[1]{\left\lVert #1\right\rVert}
\newcommand{\enorm}[1]{\left\lVert #1\right\rVert_{L^2}}
\newcommand{\brk}[1]{\left\langle #1\right\rangle}
\def\ss{{s^{-1}}}
\def\sss{{s}}

\newcommand{\vyi}{{\frac{1}{v_y}}}
\newcommand{\nq}{{\neq}}

\def\rom#1{\MakeUppercase{\romannumeral #1}}

\newcommand{\pv}{{v_y^{-1}\partial_y}}
\newcommand{\pav}{\overline{\partial_v}}
\newcommand{\myb}[1]{\textcolor{blue}{ #1 }}
\newcommand{\myr}[1]{{#1}}
\newcommand{\wt}{\widetilde}
\newcommand{\al}{\alpha}

\newcommand{\de}{\Delta}
\newcommand{\lan}{\langle}
\newcommand{\ran}{\rangle}
\newcommand{\lf}{\left}
\newcommand{\rg}{\right}
\newcommand{\vyn}{v_y^{-1}}
\def\cd{\mathcal{D}}

\newcommand{\phe}{\phi^{(E)}}

\newcommand{\wwe}{\widetilde{w}^{(E)}}
\newcommand{\oi}{\widetilde{\omega}^{(I)}}
\newcommand{\phmn}{{\phi_{m,n;k}^{(E)}}}
\newcommand{\omn}{{\widetilde{\omega}_{m,n;k}^{(E)}}}
\newcommand{\mf}{\mathfrak}

\newcommand{\phn}{\phi_k^{(I)}}

\newcommand{\GG}{{\mathfrak G}}

\newcommand{\wch}{\widetilde\chi_1}

\newcommand{\ci}{\mathbb{C}^{(I)}}
\newcommand{\ina}{\mathbbm{1}_{\substack{a+b+c\leq n\\ a\geq 1}}}

\newcommand{\omnk}{\mathring{\omega}_{m,n;k}}

\newcommand{\mr}{\mathring}

\newcommand{\ww}{w}

\allowdisplaybreaks

\numberwithin{equation}{section}

\begin{document}

\title{Pseudo-Gevrey Smoothing for the \\ Passive Scalar Equations near Couette}
\author{ Jacob Bedrossian\thanks{\footnotesize Department of Mathematics, University of California, Los Angeles, CA 90095, USA \href{mailto:jacob@math.ucla.edu}{\texttt{jacob@math.ucla.edu}}} \and Siming He\thanks{Department of Mathematics, University of South Carolina, Columbia, SC 29208, USA \href{mailto:siming@mailbox.sc.edu}{\texttt{siming@mailbox.sc.edu}}} \and Sameer Iyer\thanks{Department of Mathematics, University of California, Davis, Davis, CA 95616, USA \href{mailto:sameer@math.ucdavis.edu}{\texttt{sameer@math.ucdavis.edu}}} \and Fei Wang\thanks{School of Mathematical Sciences, CMA-Shanghai, Shanghai Jiao Tong University, Shanghai 200240, China \href{mailto:fwang256@sjtu.edu.cn}{\texttt{fwang256@sjtu.edu.cn}}}}

\maketitle

\begin{abstract}
In this article, we study the regularity theory for two linear equations that are important in fluid dynamics: the passive scalar equation for (time-varying) shear flows close to Couette in $\mathbb T \times [-1,1]$ with vanishing diffusivity $\nu \to 0$ and the Poisson equation with right-hand side behaving in similar function spaces to such a passive scalar.
The primary motivation for this work is to develop some of the main technical tools required for our treatment of the (nonlinear) 2D Navier-Stokes equations, carried out in our companion work.  
Both equations are studied with homogeneous Dirichlet conditions (the analogue of a Navier slip-type boundary condition) and the initial condition is taken to be compactly supported away from the walls. 
We develop smoothing estimates with the following three features: 
\begin{itemize}
\item[(1)] Uniform-in-$\nu$ regularity is with respect to $\p_x$ and a time-dependent adapted vector-field $\Gamma$ which approximately commutes with the passive scalar equation (as opposed to `flat' derivatives), and a scaled gradient $\sqrt{\nu} \nabla$; 
\item[(2)] $(\p_x, \Gamma)$-regularity estimates are performed in Gevrey spaces with regularity that depends on the spatial coordinate, $y$ (what we refer to as `pseudo-Gevrey');

\item[(3)] The regularity of these pseudo-Gevrey spaces degenerates to finite regularity near the center of the channel and hence standard Gevrey product rules and other amenable properties do not hold. 
\end{itemize}
Nonlinear analysis in such a delicate functional setting is one of the key ingredients to our companion paper, \cite{BHIW24a}, which proves the full nonlinear asymptotic stability of the Couette flow with slip boundary conditions.
The present article introduces new estimates for the associated linear problems in these degenerate pseudo-Gevrey spaces, which is of independent interest.  
\end{abstract}

\setcounter{tocdepth}{2}
{\small\tableofcontents}


\section{Introduction}

We study two distinct families of linear equations, each closely related to the solutions of the 2D Navier-Stokes solutions near the Couette flow in a periodic channel. 


Our investigation begins with an analysis of the passive scalar equation defined over the channel domain $\mathbb{T} \times [-1, 1]$:
\begin{subequations}
\begin{align}\label{PS}
&\pa_t \omega + (y + U^x_0(t, y))\pa_x\omega  = \nu \Delta \omega+f,\\
&\omega(t,x,y=\pm 1)=0,\\
&\omega(t=0,x,y)=\omega_{\text{in}}(x,y).
\end{align}
\end{subequations}
The perturbation of the background shear flow, $U^x_0(t, y)$, is assumed small in a suitable regularity norm (with smallness independent of $\nu$).
In this paper, $U^x_0(t, y)$ will be a \emph{passive} quantity. However, our eventual application of the methods developed in this paper will be for the nonlinear Navier-Stokes equations in our companion work \cite{BHIW24a}. For this reason, we do not assume arbitrary regularity on $U^x_0(t, y)$. Rather, a central difficulty here is to  track precisely how the results of the present paper depend on the regularity of $U^x_0(t, y)$. In a similar spirit, the external forcing can be considered as nonlinearities in the 2D Navier-Stokes equations.
For this reason, we will frequently refer to $\omega$ as the `vorticity'. 

We make the following assumptions on the initial data $\omega_{\text{in}}(x, y)$ (of course $1/4$ is arbitrary and can be replaced with any number $< 1$)
\begin{align} \label{in:data:1}
&\text{supp}(\omega_{\mathrm{in}}) \subset \lf(-\frac14, \frac14\rg), \\ \label{in:data:2}
&\omega_{\mathrm{in}} \in L^2(\mathbb{T} \times [-1,1]).
\end{align} 

The second equation we consider is the Poisson equation with right-hand side parameterized by $t$, also set on $\mathbb{T} \times [-1,1]$:
\begin{align}\label{BS}
(\pa_{x}^2+\pa_y^2) \psi(t,x,y) &= \ww(t,x,y),\quad  
\psi (t,x,y=\pm 1) =0.
\end{align}
Here the one-parameter family of abstract source terms $\{\ww(t,\cdot)\}_{t\geq 0}$ should be considered to satisfy similar a priori estimates as that of solutions to \eqref{PS}.     
Thanks to the Biot-Savart law, if the source $\ww$ is the vorticity of the fluid, the skew-gradient of the solution $\psi$ (stream function) yields the velocity in the Navier-Stokes. 

\subsection{Passive Scalar Equation}

Our main objective in the study of the passive scalar, \eqref{PS}, is to prove \textit{smoothing estimates}. Smoothing estimates in general refer to a gain in regularity over the assumed initial datum which potentially can become available after a trade-off (for example, integration in time, spatial localization, etc...). In fact for the Navier-Stokes equations, even if we only have a $L^2$ initial datum, the solution becomes analytic immediately for any positive time $t>0$ ( see~\cite{Che04, ben08} and references therein for instance). However, the radius of analyticity is proportional to $\sqrt{\nu}$, too small to be useful for our companion paper~\cite{BHIW24a}. Regarding the smoothing effects we are considering here, more relevant work could be found in~\cite{Maekawa14, KukNguVicWan22,	KukVicWan19,  KukVicWan22,	Wan20}, where basically a local solution analytic near the boundary and Sobolev regular far away is considered. Actually, our initial datum falls into this regime. Compared with these results, we develop a novel scheme to give a better description of this regularity discrepancy in a much larger time scale (up to $\nu^{-1/3}$). For local smoothing effects in Lebesgue spaces, we refer to~\cite{JiaSverak14, MaeMiuPra17} and references therein. We now proceed to indicate the precise type of smoothing we will prove in this work. 

\vspace{2 mm}

\noindent \textsc{\underline{Smoothing Estimates: Inspiration from Heat Equation:}} Of course, since only $L^2$ regularity on $\omega_{\text{in}}$ is assumed we cannot expect higher regularity than $L^\infty_t((0,\infty); L^2_{x,y} (\mathbb T \times [-1,1]))$ over the channel (uniformly in $t$). 
However, we might expect regularity gain as we depart from the support of $\omega_{in}$, which is compactly supported on the interior, as set in \eqref{in:data:1}. To see this, let us consider the heat equation, where simple calculations serve as the starting point for our technique (which is more elaborate for the passive scalar equation and the Navier-Stokes equations due to reasons we will describe).
Consider $\p_t h - \nu \Delta h = 0$ on $\mathbb{T} \times \mathbb{R}$, where $h|_{t = 0} = \omega_{\rm in}$.
One can easily observe from the Greens function formula that $h \in C^\infty( [0,\infty) \times \mathbb T \times (-1/4-\delta,1/4+\delta)^c)  $ for all $\delta > 0$ \emph{uniformly} in $\nu$. 
However, for our purposes, it will be better to see how to deduce this kind of regularization using an energy estimate. For example, consider deducing a uniform-in-$\nu$ $H^1$ estimate:
\begin{itemize}
\item Perform the basic energy estimate to prove $\| \sqrt{\nu} \nabla h \|_{L^2_{tx}}^2 \lesssim \| \omega_{\text{in}} \|_{L^2}^2$. 

\item Differentiate the equation to obtain $\p_t (\p_y h) - \nu \Delta (\p_y h) = 0$. 

\item Fix a cutoff $\chi_E(y)$ such that $\supp(\chi_E) \subset \lf(-\frac14, \frac14\rg)^c$. 

\item Multiply by $\p_y h \chi_E^2$ to obtain 
\begin{align*}
\frac{\p_t}{2} \| \p_y h \chi_E \|_{L^2}^2 + \| \sqrt{\nu} \nabla \p_y h \chi_E \|_{L^2}^2 \lesssim \nu \lf| \int_{\mathbb{T} \times \mathbb{R}} |\p_y h|^2 (\chi_E^2)''\rg|,
\end{align*}
and notice the right-hand side integrates as long as $\omega_{\text{in}} \in L^2$ (the time integral is bounded by $\| \omega_{\text{in}} \|_{L^2}^2$ by the first step). There is no appearance of the $H^1$ norm of the initial datum, $\omega_{in}$, due to the cutoff $\chi_E$. 
\end{itemize}
This is an example of a smoothing estimate which gains regularity pointwise in time (uniformly in $\nu$) due to the spatial localization.
We are insisting on this energy method approach because these are, in general, more robust (and specifically, they can be readily adapted to the full nonlinear setting in our companion work, \cite{BHIW24a}). 

\vspace{2 mm}

\noindent \textsc{\underline{Smoothing Estimates for the Passive Scalar Equation:}}
In this paper, the main goal is to establish analogues of the above smoothing estimates for the passive scalar equations on the bounded domain $\mathbb{T} \times [-1,1]$ (with Dirichlet boundary conditions).
However, doing so for passive scalar equations (versus the heat equation) as well as a domain with boundary creates several major difficulties and presents restrictions on the types of spaces we can hope for smoothing estimates.
We describe these here, and in so doing we introduce the reader to the main ingredients that will appear in our (relatively complicated) functionals.

\begin{itemize}

\item \underline{Gevrey smoothing and the $\set{\chi_n(y)}$:} In contrast to the $H^1$ Sobolev smoothing discussed above for the heat equation, we seek a much stronger type of smoothing, namely in Gevrey regularity. This requires an infinite cascade of cutoff functions which replace the one $\chi_E$, which are defined in \eqref{chi} -- \eqref{chi:prop:3}. Of course due to the infinite gain in regularity in a finite distance $y \in [-1,1]$, we need to decompose the physical domain in a very precise (eventually infinitely small) manner which is achieved through the introduction of $\{x_n, y_n\}$, \eqref{x1} -- \eqref{yn}. The motivation to use Gevrey regularity comes from the high regularity requirements that arise when studying uniform-in-$\nu$ inviscid damping results (see e.g. \cite{BM13,BMV14,HI20,JI20,MZ20,DM23}).

\item \underline{Gevrey coefficients, $B_{m,n}(t)$:} The Gevrey norms we work with are nearly analytic (as quantified by our choice $s > 1$ in \eqref{s:prime}), and are represented by the coefficient defined in \eqref{Bweight}. It is common to use Gevrey or analytic regularity with a time decaying radius that is bounded below. These factors give negative definite `Cauchy-Kovalevskaya' ($\mathcal{CK}$) terms that are useful in nonlinear applications to handle certain critical nonlinearities. In our case, this is the choice of $\lambda = \lambda(t)$ in \eqref{Gev:la}. 

\item \underline{Adapted Vector Fields $(\Gamma, \p_x)$:} Unlike the simple case of the heat equation, flat $\p_y$ derivatives do not commute with the linearized Navier-Stokes equations, and therefore $\p_y$ derivatives of $\omega$ are expected to grow in $t$, making uniform-in-$\nu$ estimates difficult or impossible (depending on precisely what one is after).
Moreover, the nonlinear analysis quantifies regularity in a very specific manner on the Fourier-side and it will be convenient to choose the compatible regularity for the analysis here. 
Therefore, we modify the notion of derivative that we obtain  ``smoothing" estimates for \eqref{PS}.
In the setting of $\mathbb{T} \times [-1,1]$, the Fourier analysis methods of \cite{BM13,HI20,BMV14} will not be readily available close to the boundary, and therefore we are forced to design a physical-side commuting vector field method.
Indeed, one of the core ideas is to design an \textit{approximate commuting vector-field} which is adapted to the transport operator appearing on the left-hand side of \eqref{PS}.
In comparison with the inviscid damping literature \cite{BM13,HI20,BMV14}, our vector field $\Gamma$ mimics the operation of changing to $(z, v)$ coordinates, taking $\p_v$ of the resulting profile, and then transforming back (which ensures the two notions of regularity work together well). 
Vector field methods have been used to quantify uniform mixing estimates in both fluid mechanics and kinetic theory, see e.g. \cite{CotiZelati20,WeiZhangZhu20,ChaturvediLukNguyen23,BedrossianCotiZelatiDolce22} for example. However, previous works have both been in finite regularity and have used vector fields which do not have non-trivial, time-variable coefficients. 

Let us now describe the objects in slightly more detail. Define the quantity $v(t,y)$ to solve the following evolution equation: 
\begin{subequations}
\begin{align} \label{eq:v}
&\pa_t (t (v - y)) = U^x_0(t,  y) + \nu t \pa_y^2 v, \\ \label{BC:v}
&\pa_y v|_{y = \pm 1} = 1.
\end{align}
\end{subequations}
Correspondingly, we define our two (approximately) commuting vector fields:
\begin{align}
\partial_x, \quad \Gamma :=\pav+t\pa_x :=\frac{1}{\partial_y v} \partial_y + t \partial_x.\label{vc_fld}
\end{align}
The vector-field $\Gamma$ is an \emph{approximate commuting vector field} for \eqref{PS}: 
\begin{align} \label{commutator:exp:1}
 [\Gamma, \left(\partial_t + (y +  U_0^x)\partial_x - \nu \Delta\right)] \omega =   {G}\pav \Gamma \omega  - \nu \pav(v_y^2-1)\pav^2\omega,
\end{align}
where the quantity $G$ is given by  
\begin{align} \label{defn:g:cap}
G(t, y)  =&  \pa_t v - \nu \pa_y^2 v = \frac{ U_0^x(t, y) - (v-y) }{t}.
\end{align}
In addition to $G$, we will need two additional quantities which appear in our analysis that measure the background shear flow: 
\begin{subequations} 
\begin{align} \label{defn:h:cap}
H(t, y) = & \pa_y v(t, y) - 1 ,\\ \label{defn:barh:cap}
\overline{H}(t, y) = & \pa_y G = \frac{\pa_y U^x_0(t, y) - (\pa_yv (t, y)- 1)}{t}. 
\end{align}
\end{subequations}
\vspace{2 mm}
The quantities $G, H$ appear in \eqref{commutator:exp:1} (for instance by writing $v_y^2 - 1 = H(H+2)$).
We introduce $\overline{H}$ as it is more convenient to work with $\pa_y G$ instead of $G$ itself.  We refer to the functions $G,H,\overline{H}$ as `coordinate system functions'. 

\item \underline{$\sqrt{\nu} \nabla$:} The presence of the $\sqrt{\nu}\nabla$ derivatives is used to control various $L^\infty$-based Gevrey norms that appear during the analysis. When implementing the classical $\dot{H}_y^1$-energy estimate for transport diffusion-type equations, the transport term typically drives the Sobolev norm to grow linearly. This extra growing term, when properly weighted by $\sqrt{\nu}$, can be absorbed by the dissipation term from the $L^2$-based Gevrey energy estimates.

\item \underline{Weighted Spaces \& Spread of Vorticity, $W(t, y)$:} Another feature we want to quantitatively capture is how quickly the vorticity can spread away from its initial support. For this reason, we design the weight function $W(t, y)$, \eqref{defndW}.  

\item \underline{Co-normal weighted spaces, $q(y)$:} In order to work in Gevrey spaces on $\mathbb{T} \times [-1,1]$ and not control overly-complicated and potentially singular boundary terms at $y = \pm 1$, we need to weaken our Gevrey spaces by including degenerate weights of $q(y)^n$, where $q$, defined in \eqref{q:defn}, vanishes linearly as $y \rightarrow \pm 1$. 

\item \underline{Decaying Anisotropic Regularity:} Introducing the three ingredients of co-normal weights, our adapted vector-field, and Gevrey regularity creates commutator terms which appear very challenging to estimate uniformly in time, essentially due to dangerous combinations of $\frac{1}{q}$, $(n + m)$ (this will imply a kind of derivative loss), and $t$ appearing in commutator terms.
To handle these commutators, we introduce a novel iterative scheme which we refer to as the \emph{Iterated Co-normal Commutator (ICC) method}.
This scheme of estimating commutators applies to both elliptic and parabolic equations.
Therefore, we postpone the discussion of the ICC method to the elliptic subsection below. However, we remark that part of the scheme uses an anisotropic radius of Gevrey regularity which decays like $1/t$ in the $\Gamma$ derivative direction (see \eqref{varphi}). 

\item \underline{Time-varying Background Shear Flow:} In order for our work to have maximal applicability to nonlinear problems (in particular, to our companion paper \cite{BHIW24a}), we do not take the background shear flow, $U^x_0(t, y)$, as having arbitrarily high regularity. Rather, we very precisely quantify functionals, \eqref{E_coord}, \eqref{D_coord} and  \eqref{CK_coord}, and estimates on these functionals that are required for our results to hold. These estimates suffice to close the nonlinear analysis in \cite{BHIW24a}. 

\end{itemize}

Of the many aspects discussed above, the inclusion of the weight $W(t, y)$, the decreasing radius $\lambda(t)$ and obtention of corresponding $\mathcal{CK}$ terms, and the precise quantification of norms on $U^x_0(t, y)$ are choices we have made in order to strengthen our results and to broaden their potential applicability to other problems (in particular, to the nonlinear setting in \cite{BHIW24a}). On the other hand, the use of the approximate vector-field $\Gamma$, the use of co-normal weights $q(y)$, and the implementation of the $ICC$ method with the time decaying weight $\varphi(t)$, are adapted to the study of the passive scalar equation on $\mathbb{T} \times [-1,1]$. 

We now give an informal statement of our main smoothing result on the passive scalar equations. The precise statement requires the definitions of various functionals and can be found below in Theorem \ref{thm:main:para}.

\begin{theorem}[Informal Version of the passive scalar equation theorem] \label{thm:PSinformal}
Assume the initial datum, $\omega_{in}$, satisfies \eqref{in:data:1} -- \eqref{in:data:2}. Assume further that the coordinate system functions satisfy appropriate bounds.
Then, as quantified by the coordinate system functionals and functionals on the source term, $f$, the solution $\omega$ to \eqref{PS} satisfies Gevrey smoothing bounds: 
\begin{itemize}
\item[a)] On $(-\frac12, \frac12)^c$, $\omega(t, \cdot)$, the solution is (uniformly) Gevrey regular in $(\partial_x,\Gamma)$, for any Gevrey class arbitrarily close to analytic. 
\item[b)] On the transition zone: $[-\frac12, -\frac14) \cup (\frac14, \frac12]$, the solution is ``Pseudo-Gevrey regular" (in the sense that the regularity depends on the location within the infinitely small slices as defined by \eqref{x1} -- \eqref{chi}).
\end{itemize}
In all cases, these Gevrey bounds are degenerate, weighted in both space and time. These estimates hold for all $ t \lesssim \nu^{-1/3-\eta}$ for some sufficiently small $\eta > 0$. 
\end{theorem}

\begin{remark}
We believe with some minor tweaking, the results can be extended $\forall t \lesssim \nu^{-1+\zeta}$ for any $\zeta > 0$. After the time-scale $\nu^{-1}$, the qualitative behavior should be significantly different. We did not pursue this sharper estimate as it is not required for our companion work \cite{BHIW24a}.
See also our earlier work \cite{BHIW23}, which obtains the results used to send $t \gtrsim \nu^{-1+\zeta}$ in the nonlinear problem.
This naturally implies some results at the linear level, however, we do not seek pseudo-Gevrey regularity on such long time-scales. 
\end{remark}

\subsection{Poisson's Equation}


The elliptic system \eqref{BS} is a family of Poisson equations with time dependent source terms. A concrete example that we have in mind is the following vorticity stream function relation in the Navier-Stokes equation
\begin{align*}
-\de\psi=\omega,\quad \psi\big|_{y=\pm 1}=0.
\end{align*}
The skew gradient of the stream function gives rise to the velocity of the fluid. We note that the $w$ in \eqref{BS} corresponds to the $-\omega$ in the companion paper \cite{BHIW24a}. 

The central theme in analyzing the system \eqref{BS} is to derive sharp \emph{inviscid damping} estimates in a \emph{bounded domain} setting. Since the pioneering work of Orr \cite{Orr07}, it has been known that the inviscid damping mechanism is closely related to the hydrodynamic stability problem. To introduce this mechanism, it is the easiest if we assume that the forcing $w$ solves the transport equation on the cylinder $\mathbb{T}\times \rr$:
\begin{align*}
\pa_t w+y\pa_x w=0.
\end{align*}
We observe that the quantities $(\pa_y+t\pa_x)^n w$ are transported by the flow and hence are bounded over time, i.e., $\|(\pa_y+t\pa_x )^nw\|_{L^2}^2\leq \|w_{\text{in}}\|_{\dot H^n}^2. $ 
A simple Fourier transform $(x,y)\rightarrow (k,\eta)$, together with the Plancherel identity yields that 
\begin{align}\label{ID}
\sum_{k\nq 0}\int|\psi(t,k,\eta)|^2d\eta\approx\sum_{k\nq 0}\int\frac{(\eta+kt)^4 |w(t,k,\eta)|^2}{(k^2+\eta^2)^2(\eta+kt)^4}d\eta\lesssim \frac{1}{k^4t^4}\|(\pa_y+t\pa_x) w\|_{H^2}^2\lesssim \frac{\|w_{\text{in}}\|_{H^2}^2}{k^4t^4}.
\end{align} 
Hence, one can obtain decay for the solution $\psi$ by paying regularity.
Such estimates have been justified in the nonlinear Euler equations \cite{BM13,HI20,JI20,MZ20} and the Navier-Stokes equations \cite{BMV14} uniformly-in-$\nu$ without boundaries. 
See \cite{Zillinger2017,WZZ19,WZZ20,Jia2020,J20,WZZ18,ISJ22,J23,beekie2024uniform,ChenWeiZhang23linear,GNRS20,CZEW20,chen2024enhanced,BH20,WZ19,CLWZ20,almog2021stability} and the references therein for some of the many works on the linearized Euler and Navier-Stokes problems. 

In general, if there are no extra assumptions regarding the location of the forcing, the inviscid damping estimate derived in \eqref{ID} is basically sharp. However, if the external forcing $w$ is compactly supported in space, and one is focusing on the solution $\psi$ in the regions away from the support of $w$, better decay estimates can be derived. This type of estimate, in a bounded domain setting in these particular Gevrey spaces we are using, constitutes the first main result of our elliptic results. 

Next, we introduce the set of differential operators that capture various commutator terms that arise when one commutes the vector fields $\Gamma$ \eqref{vc_fld} and the co-normal weights $q$. We call these differential operators the {\bf$ICC$-operators}, which have the following general form:
\begin{align}\label{SJ_Heu}
S^{(a,b,c)}f_k\sim\lf(\frac{m+n}{q}\rg)^a\pa_y^b |k|^c f_k.
\end{align}
The basic intuition is that when commuting one copy of the vector field $\Gamma$ and the weight $q^n$, one naturally introduces an extra factor of the form $\sim\frac{n}{q}$ in the expression. Hence, in the analysis of bounded domains, the boundary weight $\frac{m+n}{q}$ plays an equivalent role as a derivative in $y$. Hence, an $L^2$-norm control over the quantity \eqref{SJ_Heu} is essentially a bound on the $\dot H^{a+b+c}$ norm of the solutions. 

With these preliminaries, we are ready to state the informal elliptic estimates.  
\begin{theorem}[Informal Version of  Elliptic Estimates] Assume that the external forcing $w$ satisfies similar estimates as the solutions studied in Theorem \ref{thm:PSinformal}. 
Assume further that the coordinate system functionals satisfy appropriate bounds. Then the solution $\psi$ to \eqref{BS} satisfies Gevrey smoothing bounds: 
\begin{itemize}
\item[a)] On $\lf(-\frac12, \frac12\rg)^c$, $\psi(t, \cdot)$, is small in a high Gevrey regularity space. It consists of a component that undergoes inviscid damping $\lesssim \frac{1}{t^{N}}$ for any $N > 0$ and a component that is exponentially small in $\nu$, i.e., $\lesssim \exp\{-c\nu^{-1/9}\}$; 
\item[b)] On the transition zone: $[-\frac12, -\frac14) \cup (\frac14, \frac12]$, the solution is ``Pseudo-Gevrey regular" (in the sense that the regularity depends on the location within the infinitely small slices as defined by \eqref{x1} -- \eqref{chi}).
\end{itemize}
In all cases, these Gevrey bounds are degenerate, weighted in both space and time.  These estimates hold for all $t \lesssim \nu^{-1/3-\eta}$ for some sufficiently small $\eta > 0$.  
\end{theorem}

\section{Main Result: Parabolic Regularity}
Our main result will quantify the regularity gain for the solution, $\omega$, to \eqref{PS} by controlling several energy-dissipation-$\mathcal{CK}$ functionals of $\omega$ in terms of corresponding functionals on $f$, functionals on $U^x_0(t, y)$, and $\| \omega_{\rm in} \|_{L^2}$.  

\subsection{Preliminary Objects}
\noindent \underline{\textsc{Cutoff Functions:}} First, we fix a Gevrey index $s > 1$. We define the following associated parameters, which will arise in our analysis:
\begin{align} \label{s:prime}
s > 1, \qquad 0 <  \sigma < s - 1, \qquad \sigma_{\ast} := (s-1) - \sigma > 0. 
\end{align}
For technical purposes, we assume that $\sigma_\ast\geq 10\sigma$. Given the parameter $\sigma$ defined above, we design now a sequence of cut-off functions $\chi_n(y)$, $n \in \mathbb{N}$. We first choose the following parameters: 
\begin{align} \label{x1}
x_1 = &\frac38, \\ \label{xn}
x_{n+1} = &x_n + \frac{c_{\sigma}}{n^{1+\sigma}}, \qquad n \ge 1, \\ \label{yn}
y_n = & x_n + \frac{c_{\sigma}}{100 n^{1+\sigma}}.
\end{align}
Above, the constant $c_{\sigma}$ is chosen so that $c_{\sigma} \sum_{n \ge 1} \frac{1}{n^{1+\sigma}} < \frac18$. We now define cut-offs, $\chi_n(y)$, adapted to these scales:
\begin{align}\label{chi}
\chi_n(y) = \begin{cases} 0, \qquad - x_n < y < x_n \\ 1, \qquad \{- 1 < y < -y_n\} \cup \{y_n < y < 1\} \end{cases}, \qquad n \ge 1.
\end{align}
The following properties follow from our choice of sequences, $\{x_n\}, \{y_n\}$: 
\begin{subequations}
\begin{align} \label{chi:prop:1}
\cap_{n = 1}^{\infty} \{ \chi_n = 1 \} \supset & \lf(-1, -\frac12\rg) \cup \lf(\frac12, 1\rg), \\ \label{chi:prop:2}
\text{supp}(\nabla^k \chi_{n+1}) \subset & \{ \chi_n = 1 \} \qquad \text{ for all } k \in \mathbb{N}, n \in \mathbb{N}, \\ \label{chi:prop:3}
|\pa_y^{j} \chi_n| \lesssim &n^{j (1 + \sigma)} \chi_{n-1}.
\end{align}
\end{subequations}

\vspace{2 mm}

\noindent \underline{\textsc{Weight Functions:}} We now define our co-normal weight function 
\begin{align} \label{q:defn}
q(y) = & \begin{cases} 99(y +1) , \qquad -1 < y < -1+ \frac{1}{100}; \\ 1,\qquad -1+\frac{1}{50}<y<1-\frac{1}{50}; \\ 99(1 - y), \qquad 1 - \frac{1}{100} < y < 1, \end{cases}
\end{align}
where $q$ is taken to be monotone and smooth in the connecting regions.  
We now define our vorticity localization weight-function:
\begin{align} \label{defndW}
W(t, y)=\frac{(|y|-1/4-L\ep\arctan (t))_+^2}{K\nu(1+t)}.
\end{align}
Above, $K$ is a large but universal constant and the product $L\ep$ is a small but  universal constant. We highlight that the extra damping effect from the term $L\ep\arctan (t)$ is only applied in our companion paper \cite{BHIW24a}. Hence, one can safely drop this term and obtain all the estimates in this paper. 

We define the Gevrey radius as follows:
\begin{align} \label{Gev:la}
\lambda(t) = \lambda_0(1 + (1 + t)^{-\frac{1}{100}}).
\end{align} 
Here $\lambda_0\in(0,1]$ is an order one small number chosen depending on the size of the solution ($\ep$) and the Gevrey index $s$.  
Subsequently, we denote our Gevrey weight:
\begin{align} \label{Bweight}
B_{m,n}(t) :=\lf( \frac{\lambda^{m+n}}{(m+n)!} \rg)^{s}.   
\end{align}
We will also need the following decaying in time weight function 
\begin{align} \label{varphi}
\varphi(t) := \frac{1}{(1 + t^2)^{\frac12}}=\frac{1}{\lan t\ran }.
\end{align}
One simple consequence of this weight is the following: 
\begin{align} \label{ineq:varphi}
\sup_{t} \varphi(t) |t| \lesssim 1. 
\end{align}
However, the more delicate aspect of this weight is it does not decay \textit{too fast} so as to force our exterior norms to lose too many derivatives (measured by an appropriate Gevrey index) over the critical time-scale $t \in [0, \nu^{- \frac13-})$. This will proven rigorously in Lemma \ref{lem:ExtToInt}. Comparing the temporal functions $\lambda(t)$ and $\varphi(t)$, it is useful to record the following 
\begin{align} \label{varphi:yessir}
&\frac{\dot{\varphi}(t)}{\varphi(t)} \gtrsim \frac{1}{\langle t \rangle} \Rightarrow 1 \lesssim \langle t \rangle \frac{\dot{\varphi}(t)}{\varphi(t)}, \\ \label{lambda:yessir}
&\frac{\dot{\lambda}(t)}{\lambda(t)} \gtrsim \frac{1}{\langle t \rangle^2} \Rightarrow 1 \lesssim \langle t \rangle^2 \frac{\dot{\lambda}(t)}{\lambda(t)}. 
\end{align}
We introduce the notations 
\begin{subequations} \label{a:weight}
\begin{align} 
\label{a:weight:neq}
\bold{a}_{m,n}(t) := & B_{m,n}(t)  \varphi(t)^{1+n}, \\
	\bold{a}_{n}(t):  = &\bold{a}_{0, n}(t).
\end{align}
\end{subequations}
The final (minor) set of weights we will use are defined as follows: let $\{ \theta_n\}_{n = 0}^\infty$ be a sequence of positive numbers such that 
\begin{align}
\begin{aligned}
&\theta_{n + 1} \le \theta_n, \qquad n = 0, 1, \dots ,\\
&\frac{\theta_{n+1}}{\theta_n} = \delta_{\text{Drop}}, \qquad n < n_\ast, \label{dl_Drp}\\
&\theta_n = 1, \qquad n \ge n_\ast,
\end{aligned}
\end{align}
where the parameters $n_\ast$ and $\delta_{\text{Drop}}$ will be chosen (precisely in \eqref{shst:1} and \eqref{delta:drop:choice}) based on universal constants. 

\subsection{Parabolic Energy Functionals}

\vspace{2 mm}

\noindent \textsc{Vorticity Norms:} We will have three families of energy-dissipation-$\mathcal{CK}$ functionals to measure $\omega$, the solution to \eqref{PS}: the ``$\gamma$" version corresponding to $L^2$ level of regularity, the ``$\alpha$" version corresponding to $H^1_y$ level of regularity, and the ``$\mu$" version corresponding to $H^1_x$ level of regularity. We note that these $(\gamma, \alpha, \mu)$ functionals are then propagated at the Gevrey level (summing over $(m, n)$ below). We note that our Gevrey spaces \textit{do not} directly correspond to taking standard derivatives $\p_x, \p_y$. This is due to a variety of reasons, chief among which is the inclusion of a weakening weight $q^n$ as well as the sliding cut-offs $\chi_{m + n}$. Therefore,
it is important to note that simply controlling the $\mathcal{E}^{(\gamma)}$ functional yields essentially no useful direct $H^1$ information. 

For an abstract function $f = f(x, y)$ and its $x$-Fourier transform $f_k(y)$, we define the $f_{m,n}$-notation
\begin{align}\label{f_mn}
f_{m,n}:=&\pa_x^m\Gamma^n f,\quad f_{m,n;k}:=(ik)^m\Gamma_k^n f_k.
\end{align}Here, $\Gamma_k$ is defined as
\begin{align}
&\Gamma_k:=v_y^{-1}\pa_y +ikt. \label{defn:Gamma}
\end{align} Then introduce now the energy functionals:\begin{subequations}
\begin{align} \label{ef:a}
\mathcal{E}^{(\gamma)}[f] := & \sum_{m = 0}^\infty \sum_{n =0}^\infty \theta_n^2 \bold{a}_{m,n}^2 \| q^n f_{m,n} e^W \chi_{m + n}  \|_{L^2}^2, \\ \label{ef:b}
\mathcal{E}^{(\alpha)}[f] := & \sum_{m = 0}^\infty \sum_{n =0}^\infty \theta_n^2 \bold{a}_{m,n}^2 \| \sqrt{\nu} \p_y (q^nf_{m,n}) e^W \chi_{m + n}  \|_{L^2}^2, \\ \label{ef:c}
\mathcal{E}^{(\mu)}[f] := & \sum_{m = 0}^\infty \sum_{n =0}^\infty \theta_n^2 \bold{a}_{m,n}^2 \| \sqrt{\nu} \p_x (q^n f_{m,n}) e^W \chi_{m + n}  \|_{L^2}^2.
\end{align}
\end{subequations}\begin{subequations}
Correspondingly, we have the dissipation functionals:
\begin{align} \label{df:a}
\mathcal{D}^{(\gamma)}[f] := & \sum_{m = 0}^\infty \sum_{n =0}^\infty \theta_n^2 \bold{a}_{m,n}^2 \| \sqrt{\nu} \nabla (q^n f_{m,n}) e^W \chi_{m + n}  \|_{L^2}^2, \\ \label{df:b}
\mathcal{D}^{(\alpha)}[f] := &  \sum_{m = 0}^\infty \sum_{n =0}^\infty \theta_n^2 \bold{a}_{m,n}^2 \| \nu \p_y \nabla (q^n f_{m,n}) e^W \chi_{m + n}  \|_{L^2}^2, \\ \label{df:c}
\mathcal{D}^{(\mu)}[f] := & \sum_{m = 0}^\infty \sum_{n =0}^\infty  \theta_n^2 \bold{a}_{m,n}^2 \| \nu \p_x \nabla (q^n f_{m,n}) e^W \chi_{m + n}  \|_{L^2}^2. 
\end{align}
\end{subequations}
Next, we have the $\mathcal{CK}$ functionals, which are generated from a time derivative falling on a decreasing quantity. In our analysis, we have three decreasing quantities $\varphi(t)$ defined in \eqref{varphi}, the Gevrey radius $\lambda(t)$ defined in \eqref{Gev:la}, as well as our weight function $W(t, y)$ defined in \eqref{defndW}. Consequently, we have three families of $\mathcal{CK}$ terms which will be used in different ways:   
\begin{subequations}
\begin{align} \label{CKgammavarphi}
\mathcal{CK}^{(\gamma; \varphi)}[f] := &  \sum_{m = 0}^\infty \sum_{n =0}^\infty (1 + n) \theta_n^2 \bold{a}_{m,n}^2 \frac{\dot{\varphi}}{\varphi} \| q^nf_{m,n} e^W \chi_{m + n}  \|_{L^2}^2,  \\
\mathcal{CK}^{(\alpha; \varphi)}[f] := & \sum_{m = 0}^\infty \sum_{n =0}^\infty (1 + n) \theta_n^2 \bold{a}_{m,n}^2 \frac{\dot{\varphi}}{\varphi} \| \sqrt{\nu} \p_y (q^n f_{m,n}) e^W \chi_{m + n}  \|_{L^2}^2, \\
\mathcal{CK}^{(\mu; \varphi)}[f] := &  \sum_{m = 0}^\infty \sum_{n =0}^\infty (1 + n) \theta_n^2 \bold{a}_{m,n}^2 \frac{\dot{\varphi}}{\varphi} \| \sqrt{\nu} \p_x (q^nf_{m,n}) e^W \chi_{m + n}  \|_{L^2}^2, \\
\mathcal{CK}^{(\gamma; \lambda)}[f] := &  \sum_{m = 0}^\infty \sum_{n =0}^\infty (m + n) \theta_n^2 \bold{a}_{m,n}^2 \frac{\dot{\lambda}}{\lambda} \| q^n f_{m,n} e^W \chi_{m + n}  \|_{L^2}^2,  \\
\mathcal{CK}^{(\alpha; \lambda)}[f] := & \sum_{m = 0}^\infty \sum_{n =0}^\infty (m + n) \theta_n^2 \bold{a}_{m,n}^2 \frac{\dot{\lambda}}{\lambda} \| \sqrt{\nu} \p_y (q^n f_{m,n}) e^W \chi_{m + n}  \|_{L^2}^2, \\
\mathcal{CK}^{(\mu; \lambda)}[f] := &  \sum_{m = 0}^\infty \sum_{n =0}^\infty (m + n) \theta_n^2 \bold{a}_{m,n}^2 \frac{\dot{\lambda}}{\lambda} \| \sqrt{\nu} \p_x (q^n f_{m,n}) e^W \chi_{m + n}  \|_{L^2}^2, \\
\mathcal{CK}^{(\gamma; W)}[f]:=&\sum_{m = 0}^\infty \sum_{n = 0}^\infty \theta_n^2 \bold{a}_{m,n}^2\lf\|\sqrt{-\pa_t W}q^nf_{m,n}e^W\chi_{m+n}\rg\|_{L^2}^2, \label{CK_W_ga} \\
\mathcal{CK}^{(\al; W)}[f]:=& \sum_{m = 0}^\infty \sum_{n = 0}^\infty \theta_n^2 \bold{a}_{m,n}^2\left\|\sqrt{ -\pa_t W  } \sqrt{\nu}\pa_y (q^nf_{m,n})e^W\chi_{m+n}\right\|_{L^2}^2,\label{CK_W_al}\\  
\mathcal{CK}^{(\mu; W)}[f]:=& \sum_{m  = 0}^\infty \sum_{n = 0}^\infty \theta_n^2 \bold{a}_{m,n}^2 \lf\|\sqrt{-\pa_t W} \sqrt{\nu}\p_x(q^nf_{m,n})e^W\chi_{m+n}\rg\|_{L^2}^2.\label{CK_W_mu}
\end{align}
\end{subequations}
Roughly speaking, by comparing $\mathcal{CK}^{(\cdot; \varphi)}$ and $\mathcal{CK}^{(\cdot; \lambda)}$ we see a trade-off between regularity and decay: $(m+n)$ is in general stronger than $(1 + n)$, but $\frac{\dot{\varphi}}{\varphi} \sim \langle t \rangle^{-1}$ which is larger than $\frac{\dot{\lambda}}{\lambda}$. On the other hand, $\mathcal{CK}^{(\cdot, W)}$ encodes more subtle spatial localization information through the quotient $\frac{\p_t W}{W}$ which is absent from both $\mathcal{CK}^{(\cdot; \varphi)}$ and $\mathcal{CK}^{(\cdot; \lambda)}$. 

The aforementioned functionals all have their $k$-by-$k$ counterparts for $f_{m,n;k}$ \eqref{f_mn}. We postpone their definitions to Section \ref{sec:k_by_k_func}.

\vspace{2 mm}

\noindent \textsc{Exterior Coordinate System Functionals:} Next, for $\iota\in\{\alpha, \gamma\}$, 
we introduce the following coordinate system functionals:\begin{subequations}\label{E_coord}
\begin{align}
	\mathcal{E}_{\overline{H}}^{(\iota)} := & 
	\sum_{n\ge0} \nu^{i_{\iota}/2} \theta_{n}^2 (n+1)^{2\sss-2}  \bold{a}_{n+1}^2 \brak{t}^{3+2\ss} \|\partial_y^{i_{\iota}} \overline{H}_n e^{W/2} \chi_n \|_{L^2}^2, \\
	\mathcal{E}_{G}^{(\iota)} := & 
	\theta_{0}^2 \langle t \rangle^{4-2K\eps} \|\partial_y^{i_{\iota}} G \chi_0 \|_{L^2}^2 + 
	\sum_{n\ge1} \theta_{n}^2 
	\bold{a}_{n}^2 \brak{t}^{3+2\ss} 
	\| \partial_y^{i_{\iota}} G_n e^{W/2} \chi_{n-1+i_\iota}\|_{L^2}^2, \\
\mathcal{E}_{{H}}^{(\iota)} := & 
	\sum_{n\ge0} \theta_{n}^2 \nu^{i_{\iota}/2} \theta_{n}^2 \bold{a}_n^2 \| \partial_y^{i_{\iota}} H_n e^{W/2} \chi_n \|_{L^2}^2,
\end{align}
\end{subequations}
with 
\begin{align*}
	i_{\iota} = \begin{cases}
		1, & \ \ \iota = \alpha \\
		0, & \ \ \iota = \gamma
	\end{cases},
\end{align*}
and $K$ is a big constant.
The corresponding dissipation functionals are given as \begin{subequations} \label{D_coord}
\begin{align} 
\mathcal{D}_{\overline{H}}^{(\iota)} & = \sum_{n\ge0} \nu^{1+i_{\iota}/2} \theta_{n}^2 (n+1)^{2\sss-2}  \bold{a}_{n+1}^2 \brak{t}^{3+2\ss}
\|  \partial_y^{1+i_{\iota}} \overline{H}_n e^{W/2} \chi_n \|_{L^2}^2, \\ 
\mathcal{D}_{H}^{(\iota)} & =  \nu \theta_{0}^2 \langle t \rangle^{4-2K\eps}\enorm{\partial_y^{1+i_{\iota}} G\chi_0}^2  + \sum_{n\ge1} 
\nu \theta_{n}^2 \bold{a}_{n}^2 \brak{t}^{3+2\ss} 
\| \partial_y^{1+i_{\iota}} G_n e^{W/2} \chi_{n-1+i_\iota}\|_{L^2}^2,  \\
\mathcal{D}_{G}^{(\iota)} & = \sum_{n\ge0}\theta_{n}^2 \bold{a}_{n}^2 \nu^{1+i_{\iota}/2} \|  \partial_y^{1+i_{\iota}} {H}_n e^{W/2} \chi_n \|_{L^2}^2,
\end{align}
\end{subequations}\begin{subequations}
and the $\mathcal{CK}$ terms are written as
  \begin{align}\label{CK_coord}
	\mathcal{CK}_{\overline{H}}^{(\iota)} :=& \sum_{n\ge0} \theta_{n}^2 \mathcal{CK}_{\overline{H}, n}^{(\iota)}, \\
	\mathcal{CK}_{G}^{(\iota)} :=& \sum_{n\ge0} \theta_{n}^2 \mathcal{CK}_{G, n}^{(\iota)}, \\
	\mathcal{CK}_{H}^{(\iota)} :=& \sum_{n\ge0} \theta_{n}^2 \mathcal{CK}_{H, n}^{(\iota)},
\end{align}
where
\begin{align}
	\mathcal{CK}_{\overline{H},n}^{(\iota)} := & \sum_{j = 1}^2  \mathcal{CK}^{(\iota;j)}_{\overline{H}, n}(t), \\
	\mathcal{CK}_{\overline{H},n}^{(\iota;1)}:= &  \nu^{i_{\iota}/2}(n+1)^{2\sss-2} \bold{a}_{n+1}^2 \brak{t}^{3+2\ss} \|\partial_y^{i_{\iota}} \overline{H}_n 
	\sqrt{- \partial_{t}W} e^{W/2} \chi_n \|_{L^2}^2,
	\label{CK1:defi} \\
	\mathcal{CK}_{\overline{H},n}^{(\iota;2)} :=&  -  \nu^{i_{\iota}/2} (n+1)^{2\sss-2} \bold{a}_{n+1}   \dot{\bold{a}}_{n+1} \brak{t}^{3+2\ss}
	\| \partial_y^{i_{\iota}} \overline{H}_n e^{W/2} \chi_n \|_{L^2}^2, \label{CK2:defi} \\
	\mathcal{CK}_{G,n}^{(\iota)} := & \begin{cases}  \paren{4\langle t \rangle^{4-2K\eps_1}t^{-1} - (4-2K\eps_1)t\brak{t}^{2-2\eps}}
		\enorm{\partial_y^{i_{\iota}} G\myr{\chi_0}}^2 \qquad & n=0
		\\ \sum_{j=1}^{2} \mathcal{CK}_{G, n}^{(\iota; j)}(t)^2 \qquad & n\ge 1, 
	\end{cases}\\
	\mathcal{CK}_{G,n}^{(\iota;1)} :=& \bold{a}_{n}^2 \brak{t}^{3+2\ss} \|\partial_y^{i_{\iota}} G_n \sqrt{- \partial_{t}W} e^{W/2} \chi_{n-1+i_\iota} \|_{L^2}^2,\\
	\mathcal{CK}_{G,n}^{(\iota;2)} :=	&-  \bold{a}_{n} \dot{\bold{a}}_{n} 
	\brak{t}^{3+2\ss}
	\| \partial_y^{i_{\iota}} G_n  \chi_{n-1+i_\iota} \|_{L^2}^2,\\
	\mathcal{CK}_{H,n}^{(\iota)} :=& - \nu^{i_{\iota}/2}  \bold{a}_{n} \dot{\bold{a}}_{n}
	\| \partial_y^{i_{\iota}}{H}_n e^{W/2} \chi_n \|_{L^2}^2
	+ \nu^{i_{\iota}/2}  \bold{a}_{n}^2   \| \partial_y^{i_{\iota}}{H}_n \sqrt{- W_t} e^{W/2} \chi_n \|_{L^2}^2  .
\end{align}
\end{subequations}
We remind the readers that as in the functionals of $\omega$, we have three kinds of $\mathcal{CK}$ terms coming from the functions $\phi, \lambda,$ and $W$. However, to ease the notation, we do not differentiate them and denote the sum as $\mathcal{CK}$ for each coordinate functions $\overline H, H$, and $G$.

\subsection{Main Theorem: Parabolic Estimates} 
Next, we explain the assumptions on the external velocity field $U_0^x$ (mostly stated as assumptions on the coordinate change $y \mapsto v(t,y)$ and the auxiliary unknowns $G,H,\overline{H}$).
Fix $T> 0$ arbitrary. Then we assume that $\forall \lambda > 0$ and $\forall s > 1$, $\exists \eps > 0$ (necessarily uniform as $\lambda \to 0$) such that the following estimates hold:
\begin{subequations}\label{asmp}
\begin{itemize}
\item Interior coordinate system estimates: 
\begin{align}
\sup_{0 \leq t \leq T} \sum_{n=0}^\infty \lf(B_{0,n}^\text{low}\rg)^2\|\wt\chi_1^{\mf c}\pav^n(v_y^2-1)\|^2_{L^2} \leq  \ep^2,\quad  B ^\mathrm{low}_{m,n}:=\lf(\frac{2^{-(m+n)}}{(m+n)!}\rg)^{4},
\label{v_y_sob}
\end{align}
where $\wt \chi_1$ is a Gevrey-$s$ cutoff satisfying
\begin{align}\label{wt_chi_intro}
 \wt \chi_1(\xi)= \begin{cases}1,\qquad |\xi|\geq \frac{3}{8}-\frac{1}{80},\\ 
0,\qquad |\xi|\leq \frac{3}{8}-\frac{1}{40}, \\ \text{monotone}, \quad \text{others}, \end{cases}\qquad \wt \chi_1^\mathfrak{c}(\xi)=1-\wt \chi_1(\xi).
\end{align}
We highlight that the $\wt \chi_1$-cutoff function is defined in the new coordinate system $v$, so the variable $\xi$ takes values in $v(t,-1)$ and $v(t,1). $
\item Exterior coordinate system estimates: 
\begin{align} \label{boot:H}
\max_{\iota \in \set{H,\overline{H},G}} \sup_{0 \le t \le T} \mathcal{E}^{(\gamma)}_{\iota} +  \int_0^T \mathcal{D}^{(\gamma)}_{\iota}(t) dt \leq \eps^2 ,\\
\max_{\iota \in \set{H,\overline{H},G}} \sup_{0 \le t \le T} \mathcal{E}^{(\alpha)}_{\iota} +  \int_0^T \mathcal{D}^{(\alpha)}_{\iota}(t) dt \leq \eps^2.
\end{align}
\item Sobolev regularity estimates up to the boundary:
\begin{align}
\|\pa_yU_0^x\|_{L_t^\infty L_y^\infty}\leq \frac{1}{16},\quad \|v-y\|_{L_t^\infty L_y^\infty}\leq& \frac{1}{160};\label{close_vy}\\ 
\|v_y^{-1}\|_{L_t^\infty L_y^\infty}+\|v_y-1\|_{L_t^\infty H^{3}_{y}} \lesssim & 1.\label{v_y_asmp}
\end{align}
\end{itemize}
\end{subequations}

The main result of this paper for the parabolic equation \eqref{PS} is the following theorem.  
\begin{theorem} \label{thm:main:para}
Assume $\omega_{\mathrm{in}}$ satisfies \eqref{in:data:1} -- \eqref{in:data:2}.
Fix  $s > 1$  arbitrary, $\frac{1}{100} > \eta > 0$, and assume that the hypotheses \eqref{asmp} hold.
There exists a $\lambda_0>0$ and an $\eps_0>0$ such that if $0 < \lambda \leq \lambda_0$ and $0 \leq \eps \leq \eps_0$, the following energy inequalities are valid for all times $t \leq \nu^{-1/3-\eta}$ and $0 < \nu \leq 1$ (with all implicit constants are uniform in $\nu$ and $t$),
 \ifx
\begin{align} \n
&\frac{d}{d t} \mathcal{E}^{(\gamma)}[\omega]  + \mathcal{D}^{(\gamma)}[\omega] + \mathcal{CK}^{(\gamma; \varphi)}[\omega] + \mathcal{CK}^{(\gamma; W)}[\omega] 
 \\ \label{lights:on:1}
 & \qquad \lesssim     |\mathcal{I}^{(\gamma)}_{\mathrm{Source}}| + \eps \nu^{99}{\mathcal{E}^{(\gamma) }}
 + \eps \nu^{99} \sqrt{\mathcal{E}^{(\gamma) }}
 \paren{{\cd^{(\gamma)}} + \mathcal{CK}^{(\gamma)} + \cd_{H}^{(\gamma)} +\cd_{\overline H}^{(\gamma)}}, 
 \end{align}
 \begin{align}
  \n
 &\frac{d}{d t} \mathcal{E}^{(\alpha)}[\omega]  + \mathcal{D}^{(\alpha)}[\omega] + \mathcal{CK}^{(\alpha; \varphi)}[\omega] + \mathcal{CK}^{(\alpha; W)}[\omega] 
 \\ \label{lights:on:2}
 & \qquad \lesssim  \sum_{\iota \in (\gamma, \alpha, \mu)} |\mathcal{I}^{(\iota)}_{\mathrm{Source}}| + \eps\nu^{98} \sum_{\iota \in (\alpha, \gamma)} \mathcal{E}^{(\iota)} 
 \n \\ &\qquad\quad+ \eps \nu^{99}
\sum_{\iota \in (\alpha, \gamma)}  \sqrt{\mathcal{E}^{(\iota) }}
 \sum_{\iota \in (\alpha, \gamma)} \paren{{\cd^{(\iota)}} + \mathcal{CK}^{(\iota)} + \mathcal{D}^{(\iota)}_{H} +\mathcal{D}^{(\iota)}_{\overline{H}}};
 \end{align}
 \begin{align}
  \n
 &\frac{d}{d t} \mathcal{E}^{(\mu)}[\omega]  + \mathcal{D}^{(\mu)}[\omega] + \mathcal{CK}^{(\mu; \varphi)}[\omega] + \mathcal{CK}^{(\mu; W)}[\omega] 
 \\ \label{lights:on:3}
 & \qquad \lesssim  |\mathcal{I}^{(\mu)}_{\mathrm{Source}}| +\eps \nu^{99}
 {\mathcal{E}^{(\mu) }}  + \eps \nu^{99}
 \sqrt{\mathcal{E}^{(\mu) }}
 \paren{{\cd^{(\mu)}} + \mathcal{CK}^{(\mu)}+ \cd_{H}^{(\gamma)} +\cd_{\overline H}^{(\gamma)}}.
\end{align} \fi
 \begin{align} \n
& \mathcal{E}^{(\gamma)}[\omega(t)]  + \int_0^t\mathcal{D}^{(\gamma)}[\omega] + \mathcal{CK}^{(\gamma; \varphi)}[\omega] + \mathcal{CK}^{(\gamma; W)}[\omega]ds 
 \\ \label{lights:on:1}
 & \qquad \lesssim  \mathcal{E}^{(\gamma)}[\omega(0)]   |+\int_0^t|\mathcal{I}^{(\gamma)}_{\mathrm{Source}}|ds + \eps \nu^{99}\int_0^t{\mathcal{E}^{(\gamma) }}
 ds+ \eps \nu^{99} \int^t_0\sqrt{\mathcal{E}^{(\gamma) }}
 \paren{{\cd^{(\gamma)}} + \mathcal{CK}^{(\gamma)} + \cd_{H}^{(\gamma)} +\cd_{\overline H}^{(\gamma)}}ds, \\
  \n
 &\mathcal{E}^{(\alpha)}[\omega(t)]  + \int_0^t\lf(\mathcal{D}^{(\alpha)}[\omega] + \mathcal{CK}^{(\alpha; \varphi)}[\omega] + \mathcal{CK}^{(\alpha; W)}[\omega] \rg)ds
 \\ \label{lights:on:2}
 & \qquad \lesssim \sum_{\iota\in(\al,\gamma,\mu)}\mathcal{E}^{(\iota)}[\omega(0)]+\int_0^t \sum_{\iota \in (\gamma, \alpha, \mu)} |\mathcal{I}^{(\iota)}_{\mathrm{Source}}|ds + \eps\nu^{98}\int_0^t \sum_{\iota \in (\alpha, \gamma)} \mathcal{E}^{(\iota)} ds
 \n \\ &\qquad\quad+ \eps \nu^{99}\int_0^t
\sum_{\iota \in (\alpha, \gamma)}  \sqrt{\mathcal{E}^{(\iota) }}
 \sum_{\iota' \in (\alpha, \gamma)} \paren{{\cd^{(\iota')}} + \mathcal{CK}^{(\iota')} + \mathcal{D}^{(\iota')}_{H} +\mathcal{D}^{(\iota')}_{\overline{H}}}ds;\\
 \n
 &\mathcal{E}^{(\mu)}[\omega(t)]  + \int_0^t\mathcal{D}^{(\mu)}[\omega] + \mathcal{CK}^{(\mu; \varphi)}[\omega] + \mathcal{CK}^{(\mu; W)}[\omega]ds 
 \\ \label{lights:on:3}
 & \qquad \lesssim \mathcal{E}^{(\mu)}[\omega(0)]+ \int_0^t|\mathcal{I}^{(\mu)}_{\mathrm{Source}}| +\eps \nu^{99}\int_0^t
 {\mathcal{E}^{(\mu) }} ds + \eps \nu^{99}\int_0^t
 \sqrt{\mathcal{E}^{(\mu) }}
 \paren{{\cd^{(\mu)}} + \mathcal{CK}^{(\mu)}+ \cd_{H}^{(\gamma)} +\cd_{\overline H}^{(\gamma)}}ds,
 \end{align}
where the source terms are defined as follows:
\begin{align} \label{def:IS:g}
\mathcal{I}^{(\gamma)}_{\mathrm{Source}} := & \sum_{m = 0}^\infty \sum_{n = 0}^\infty \bold{a}_{m,n}^2 \theta_n^2 \mathrm{Re}\lf \langle q^n \p_x^m \Gamma^n f, \omega_{m,n} \chi_{m + n}^2 e^{2W} \rg\rangle, \\ \label{def:IS:a}
\mathcal{I}^{(\alpha)}_{\mathrm{Source}} := &\sum_{m = 0}^\infty \sum_{n = 0}^\infty \bold{a}_{m,n}^2 \theta_n^2 \nu \mathrm{Re} \lf\langle q^n \p_x^m \p_y \Gamma^n f, \p_y \omega_{m,n} \chi_{m + n}^2 e^{2W}\rg \rangle, \\ \label{def:IS:m}
\mathcal{I}^{(\mu)}_{\mathrm{Source}} := & \sum_{m = 0}^\infty \sum_{n = 0}^\infty \bold{a}_{m,n}^2 \theta_n^2 \nu \mathrm{Re} \lf\langle q^n \p_x^m \p_x \Gamma^n f, \p_x \omega_{m,n} \chi_{m + n}^2 e^{2W} \rg\rangle.
\end{align}
\end{theorem}
A few remarks are in order. 
\begin{remark}
The assumption that $\lambda$ is chosen small does not really lose generality as $s$ is arbitrary, however, it does simplify some technicalities in the proof.  
\end{remark}

\begin{remark}
With some additional technical work, we could likely extend the result to all $\eta < 2/3$. After the time $t \gtrsim \nu^{-1}$, the qualitative nature of the estimates should be changed significantly, as $W$ no longer controls any localization. Of course, after this time, the solution decays exponentially; if one wants to obtain sharper decay estimates, one should employ hypocoercivity-based methods that can capture the enhanced dissipation. 
\end{remark}

\begin{remark}
Formally at least, we expect that we could take $\varphi$ also $y$-dependent, namely fixing it $\equiv 1$ for (say) $\abs{y} < \frac{24}{25}$ and decaying like $t^{-1}$ only for $y$ close to the boundary. This is because the decaying $\varphi$ is only used for dealing with commutators that include $q'$. However, as this technical enhancement is not required for our companion work \cite{BHIW24a}, we did not pursue this more precise estimate. 
\end{remark}

Theorem \ref{thm:main:para} follows in two steps, which corresponds to Proposition \ref{pro:light:on} and Proposition \ref{pro:rhs:intro} below. First, we take the Fourier transform in $x$ of~\eqref{PS} to obtain
\begin{align}
\label{M1a}
&\pa_t \omega_k + (y + U^x_0(t, y))ik\omega_k  =\nu \de_k\omega_k+f_k,\\ \n
&\omega_k(t,y=\pm 1)=0,\qquad \
 \omega_k(t=0,y)=\omega_{\text{in};k}(y).
\end{align}
Here we used the notation $\de_k:=- |k|^2+\pa_y^2$. In order to derive higher regularity bounds, we derive the equation for the quantity $\mathring \omega_{m,n;k}$:
\begin{align}\label{f_mn_ring}
\mathring \omega_{m,n;k}:=|k|^m q^n\Gamma_k^n \omega_k.  
\end{align}
Similarly, we will use the notation $\mathring f_{m,n;k}$ for general functions. Here, we highlight that in this notation, we have the co-normal weight $q^n$ and we have $|k|^m$ as opposed to $(ik)^m$ in \eqref{f_mn}. 
The resulting equation reads as follows, 
\begin{align}
\label{om_mnk_intro}\pa_t \mathring{\omega}_{m,n;k} &+ ik(y + U^x_0(t, y))\mathring{\omega}_{m,n;k} - \nu \Delta_k \mathring{\omega}_{m,n;k} =\bold F_k^{(m,n)}+\bold C_{\mathrm{trans},k}^{(m,n)}+\bold C_{\mathrm{visc},k}^{(m,n)}+ \mathbf{C}^{(m,n)}_{q,k},
\end{align}
where the detailed expressions of the right hand side terms are in \eqref{T_Ctvq} in Lemma \ref{lem:eq:comp:a}. The terms $\mathbf{C}^{(m,n)}_{\cdots}$ encode all the commutators arising from various sources.
\begin{proposition}[Section \ref{sec:ext:lin}] \label{pro:light:on} Assume that the hypotheses \eqref{asmp} hold. Let $\iota \in \{\gamma, \alpha, \mu \}$. Then, if $\lambda$ \eqref{Gev:la} is small enough compared to universal constants, there exists universal constant $C_\#$ such that 
\begin{align}   \label{outout1}
&\frac{d}{d t} \mathcal{E}^{(\iota)}[\omega]  + \mathcal{D}^{(\iota)}[\omega] + \sum_{\mf D\in(\varphi,\lambda, W)}\mathcal{CK}^{(\iota; \mf D)}[\omega]   
 \leq C_\#\lf(   |\mathcal{I}^{(\iota)}_{q}| + |\mathcal{I}^{(\iota)}_{\mathrm{trans}} |+  |\mathcal{I}^{(\iota)}_{\mathrm{visc}}| +  |\mathcal{I}^{(\iota)}_{\mathrm{Source}}|\rg),\quad \iota\in\{\gamma,\mu\};\\ 
 &\frac{d}{d t} \lf(\sum_{\iota\in(\gamma,\mu)}\mathcal{E}^{(\iota)}[\omega]+c_\al\mathcal{E}^{(\al)}[\omega] \rg)+\frac{1}{2}\sum_{\iota\in(\gamma,\mu)}\lf(\mathcal{D}^{(\iota)}[\omega] + \sum_{\mathfrak{D}\in(\varphi,\lambda,W)}\mathcal{CK}^{(\iota; \mf D)}[\omega] \rg) \n \\
 &\qquad+ c_\al\lf(\mathcal{D}^{(\al)}[\omega] + \sum_{\mathfrak D\in(\varphi,\lambda, W)}\mathcal{CK}^{(\al; \mf D)}[\omega] \rg)\leq C_\# \sum_{\iota\in(\al,\gamma,\mu)}\lf(|\mathcal{I}^{(\iota)}_{q}| + |\mathcal{I}^{(\iota)}_{\mathrm{trans}} |+  |\mathcal{I}^{(\iota)}_{\mathrm{visc}}| +  |\mathcal{I}^{(\iota)}_{\mathrm{Source}}|\rg).\label{outout2}
\end{align}
Here $c_\al\in[\frac{1}{16}, \frac{1}{2}]$ is a small universal constant. Moreover, the error terms appearing above are defined precisely in \eqref{def:Inn:g} -- \eqref{def:Inn:u}, and can be described as follows:
\begin{subequations}
\begin{align} \n
\mathcal{I}^{(\iota)}_q := & \text{\normalfont commutators of the form $[q, \nu \Delta]$}, \\ \n
\mathcal{I}^{(\iota)}_{ \mathrm{trans}} := & \text{\normalfont commutators of the form $[\Gamma, \p_t + (y + U_0) \p_x]$}, \\ \n
\mathcal{I}^{(\iota)}_{\mathrm{visc}} := & \text{\normalfont commutators of the form $[\Gamma, \nu \Delta]$},
\end{align}
\end{subequations}
and $I_{\mathrm{Source}}^{(\iota)}$ has been defined above in \eqref{def:IS:g} -- \eqref{def:IS:m}. 
\end{proposition}

\vspace{2 mm}

The next proposition provides estimates on the first three error terms appearing in the above estimate. We have 
\begin{proposition}[Section \ref{sec:ext:comm}] \label{pro:rhs:intro} Assume that the hypotheses \eqref{asmp} hold. There exists a constant $c_\ast\leq \frac{1}{256C_\#}$ (with $C_\#$ being defined in \eqref{outout1},\eqref{outout2}) such that following estimates hold for the $\gamma$-level:\begin{subequations}\label{Im:qu:a-c}
\begin{align} \label{Im:qu:a}
|\mathcal{I}^{(\gamma)}_q| \leq &c_\ast  \mathcal{D}^{(\gamma)}[\omega],  \\  \label{Im:qu:b}
|\mathcal{I}^{(\gamma)}_{\mathrm{trans}}| \lesssim & \eps \nu^{99}
\sqrt{\mathcal{E}^{(\gamma) }}
\paren{\sqrt{\mathcal{E}^{(\gamma) }} + \sqrt{\cd^{(\gamma)}} + \sqrt{\mathcal{CK}^{(\gamma)}}}\paren{1 + \sqrt{\cd^{(\gamma)}_{ H}} + \sqrt{\cd^{(\gamma)}_{ \overline  H}}},  \\  \label{Im:qu:c}
|\mathcal{I}^{(\gamma)}_{\mathrm{visc}}| \lesssim & \eps  \nu^{100}
\sqrt{\mathcal{E}^{(\gamma)}} \paren{		\sqrt{\cd^{(\gamma)}} + \sqrt{\mathcal{CK}^{(\gamma)}}}  + \eps  \nu^{100}\mathcal{D}^{(\gamma)}.
\end{align}
\end{subequations}
The following estimates are valid for the $\alpha$-level:\begin{subequations}\label{Im:qu:d-f}
\begin{align}  \label{Im:qu:d}
|\mathcal{I}^{(\alpha)}_q| \le & c_\ast  ( \mathcal{D}^{(\alpha)}[\omega] + \mathcal{D}^{(\mu)}[\omega] + \mathcal{D}^{(\gamma)}[\omega]+  \mathcal{CK}^{(\al,W)}),  \\  \label{Im:qu:e}
|\mathcal{I}^{(\alpha)}_{\mathrm{trans}}| \lesssim &\eps  \nu^{98}
\sqrt{\mathcal{E}^{(\alpha) }} \sum_{\iota\in \{\alpha, \gamma \}} 
\paren{\sqrt{\mathcal{E}^{(\iota) }} + \sqrt{\cd^{(\iota)}} + \sqrt{\mathcal{CK}^{(\iota)}}} \sum_{\iota\in \{\alpha, \gamma \}} \paren{1 + \sqrt{\cd^{(\iota)}_{ H}} + \sqrt{\cd^{(\iota)}_{ \overline  H}}}, \\  \label{Im:qu:f}
|\mathcal{I}^{(\alpha)}_{\mathrm{visc}}| \lesssim & 
\eps \nu^{99}
\sqrt{\mathcal{E}^{(\alpha)}} \paren{\sqrt{\mathcal{E}^{(\alpha)}} + \sqrt{\mathcal{D}^{(\alpha)}}+\sqrt{\mathcal{D}^{(\gamma)}}
	+\sqrt{\mathcal{D}^{(\mu)}}+ \sqrt{\mathcal{CK}^{(\alpha)}}}	\paren{\sqrt{\cd_{H}^{(\gamma)}}+\sqrt{\cd_{H}^{(\alpha)}}}.
\end{align}
\end{subequations}
The following estimates are valid for the $\mu$-level:\begin{subequations}\label{Im:qu:g-i}
\begin{align}  \label{Im:qu:g}
|\mathcal{I}^{(\mu)}_q| \le & c_\ast  \mathcal{D}^{(\mu)}[\omega],  \\  \label{Im:qu:h}
|\mathcal{I}^{(\mu)}_{\mathrm{trans}}| \lesssim &\eps
\nu^{99}
\sqrt{\mathcal{E}^{(\mu) }}
\paren{\sqrt{\mathcal{E}^{(\mu) }} + \sqrt{\cd^{(\mu)}} + \sqrt{\mathcal{CK}^{(\mu)}}}\paren{1 + \sqrt{\cd^{(\mu)}_{ H}} + \sqrt{\cd^{(\mu)}_{ \overline  H}}}, \\  \label{Im:qu:i}
|\mathcal{I}^{(\mu)}_{\mathrm{visc}}| \lesssim & \eps  \nu^{100}
\sqrt{\mathcal{E}^{(\mu)}} \paren{		\sqrt{\cd^{(\mu)}} + \sqrt{\mathcal{CK}^{(\mu)}}}  + \eps  \nu^{100}\mathcal{D}^{(\mu)}.
\end{align}
\end{subequations}
\end{proposition}

We conclude the second section by a rigorous proof of Theorem \ref{thm:main:para} assuming Proposition \ref{pro:light:on}, \ref{pro:rhs:intro}.
\begin{proof}
For the $\mathcal{E}^{(\gamma)}$ and $\mathcal{E}^{(\mu)}$ estimates, one can combine the bounds \eqref{outout1} \eqref{Im:qu:a-c}, \eqref{Im:qu:g-i} and integrate in time  to obtain the bounds \eqref{lights:on:1}, \eqref{lights:on:3} directly. Thanks to the expression \eqref{outout2}, we have that  
\begin{align*}
&\frac{d}{dt}\mathcal{E}^{(\al)}[\omega]+   \mathcal{D}^{(\al)}[\omega] + \sum_{\mathfrak D\in(\varphi,\lambda, W)}\mathcal{CK}^{(\al; \mf D)}[\omega]  +\frac{1}{2c_\al}\sum_{\iota\in(\gamma,\mu)}\lf(  \mathcal{D}^{(\iota)}[\omega] + \sum_{\mathfrak{D}\in(\varphi,\lambda,W)} \mathcal{CK}^{(\iota; \mf D)}[\omega]\rg)   \n \\
 &\leq -\frac{1}{c_\al}\frac{d}{dt}\sum_{\iota\in(\gamma,\mu)}\mathcal{E}^{(\iota)}[\omega]+ \frac{C_\#}{c_\al} \sum_{\iota'\in(\al,\gamma,\mu)}\lf(|\mathcal{I}^{(\iota')}_{q}| + |\mathcal{I}^{(\iota')}_{\mathrm{trans}} |+  |\mathcal{I}^{(\iota')}_{\mathrm{visc}}| +  |\mathcal{I}^{(\iota')}_{\mathrm{Source}}|\rg).
\end{align*} Applying the bounds \eqref{Im:qu:a-c},  \eqref{Im:qu:d-f}, and \eqref{Im:qu:g-i} yields that
\begin{align*}
&\frac{d}{dt}\mathcal{E}^{(\al)}[\omega]+   \mathcal{D}^{(\al)}[\omega] + \sum_{\mathfrak D\in(\varphi,\lambda, W)}\mathcal{CK}^{(\al; \mf D)}[\omega]  +\frac{1}{2c_\al}\sum_{\iota\in(\gamma,\mu)}\lf(  \mathcal{D}^{(\iota)}[\omega] + \sum_{\mathfrak{D}\in(\varphi,\lambda,W)} \mathcal{CK}^{(\iota; \mf D)}[\omega]\rg)   \n \\
 &\leq -\frac{1}{c_\al}\frac{d}{dt}\sum_{\iota\in(\gamma,\mu)}\mathcal{E}^{(\iota)}[\omega]+ \frac{4c_\ast C_\#}{c_\al}\lf( \sum_{\iota\in(\al,\gamma,\mu)}\mathcal{D}^{(\iota)}[\omega] +  \mathcal{CK}^{(\al,W)}\rg)\\
 &\quad+C\sum_{\iota \in (\gamma, \alpha, \mu)} |\mathcal{I}^{(\iota)}_{\mathrm{Source}}| + \eps\nu^{98} \sum_{\iota \in (\alpha, \gamma)} \mathcal{E}^{(\iota)} 
 + C\eps \nu^{99}
\sum_{\iota \in (\alpha, \gamma)}  \sqrt{\mathcal{E}^{(\iota) }}
 \sum_{\iota' \in (\alpha, \gamma)} \paren{{\cd^{(\iota')}} + \mathcal{CK}^{(\iota')} + \mathcal{D}^{(\iota')}_{H} +\mathcal{D}^{(\iota')}_{\overline{H}}}.
\end{align*}
Since $c_\al\in[\frac{1}{16},\frac{1}{2}]$, and $c_\ast\leq \frac{1}{256C_\#}$, we have that the second term can be absorbed by the left hand side. 
Now we integrate in time and obtain the result \eqref{lights:on:2}.
\end{proof}
 \label{sec:OutlinePR}

\section{Main Result: Elliptic Regularity}

\subsection{Preparatory Objects \& Definitions}

\noindent \underline{\textsc{Cutoff Functions:}} To distinguish between the interior and exterior components of various solutions and quantities, we define the following two families of  cutoff functions: 
\\ \noindent
a) the interior/exterior cutoff functions of the first kind $\chi^I/\chi^E$, which are Gevrey-$s$ regular and satisfy
\begin{align} \label{chi:I:def}
\chi^I(\xi) =& \begin{cases} 1, \qquad \xi \in  \lf(-\frac12, \frac12\rg), \\ 0, \qquad \abs{\xi} > \frac{3}{4} ,\\
\text{monotone},\quad \text{others}. \end{cases}\ \qquad
\chi^E(\xi)= 1-\chi^{I}(\xi). 
\end{align}

\noindent
b) the interior/exterior cutoff functions of the second kind $\wt\chi_1/\wt\chi_1^{\mf c}$ defined above in \eqref{wt_chi_intro}, where 
the cutoff $\wt \chi_1$ can be viewed as a fattened version of $\chi_1$ in \eqref{chi}.

\noindent \underline{\textsc{Elliptic Equations:}}  The starting point for our elliptic estimates are the relation with a one-parameter family of forcing $\{w(t)\}_{t\geq0}$:
\begin{align}\label{BSlaw}
\de \psi=w,\quad \psi(t,x,y=\pm1)=0.
\end{align}
In the paper \cite{BHIW24a}, we decompose the solution $\psi(t,x,y)$ in two different ways, denoted by $\psi = \Psi^{(I)} + \Psi^{(E)}$ or $\psi = \phi^{(I)} + \phi^{(E)}$, based on the type of information we need. 
\vspace{2 mm}


\ifx
Recall the definition of the $(z,v)$ coordinates in \eqref{}; originally from \eqref{} in the cylinder case.  
Defining $f^I = \chi_I f$, we obtain the following PDE in the $(z,v)$ coordinates  
\begin{align*}
\partial_t f^I + g \partial_v f^I + v' \grad^\perp \phi_{k \neq 0} \cdot \grad f^I - \nu \widetilde{\Delta_t} f^I = \mathcal{C},  
\end{align*}
with the commutator given by 
\begin{align*}
\mathcal{C} =  g \chi_I' f - v' \chi_I'  \partial_x \phi_{k \neq 0} f - \nu (v')^2 \chi_I'' f - 2 \nu (v')^2 \chi_I' (\partial_v-t\partial_z) f. 
\end{align*}
The left-hand side has now been basically reduced to the profile equation treated in \cite{}, except for the fact that the Biot-Savart law includes contributions from both $\chi_I f$ and $(1-\chi_I)f$.

The goal is hence to view $f^I$ as a function on the cylinder and use the Fourier multiplier methods of \cite{BM13,HI20} to deal with the nonlinear term uniformly-in-$\nu$.
Enhanced dissipation is dealt with as well using these methods, adapting an idea from \cite{BGM15III}.
In order to apply these methods we need to further localize the PDEs for the coordinate system unknowns, $g,h,\overline{h}$, defining
\begin{align}\label{g_Ih_I}
g^I = \chi^I g, \quad h^I  = h \chi^I h, \quad \overline{H}^I = \overline{h} \chi^I,
\end{align}
for which the corresponding PDEs become
\begin{align*}
& \partial_t \overline{h}^I + \frac{2}{t} \overline{h}^I + g \partial_v \overline{h}^I + \left(\chi^I_c v' \grad^\perp \phi_{k \neq 0} \cdot \grad \overline{h}\right)_0 - \nu \widetilde{\Delta_t} \overline{h}^I = \mathcal{C}_{\overline{h}} \\
& \partial_t g^I + \frac{2}{t} g^I + g \partial_v g^I + \left(\chi^I_c v' \grad^\perp \phi_{k \neq 0} \cdot \grad u\right)_0 - \nu \widetilde{\Delta_t} g^I = \mathcal{C}_g \\
& \partial_t h^I + + g \partial_v h^I - \nu \widetilde{\Delta_t} h^I - \overline{h}^I = \mathcal{C}_{h}, 
\end{align*}
where the commutators above are of the form
\begin{align*}
\mathcal{C}_{\ast} = g \chi_I' (\ast) - \nu (v')^2 \chi_I'' (\ast) - 2 \nu (v')^2 \chi_I' \partial_v (\ast). 
\end{align*}
The commutators $\mathcal{C}, \mathcal{C}_{\ast}$ are supported only near the interval $\frac{3}{4} \leq \abs{v}  \leq \frac{31}{40}$, and here the design of the  strong weight $W$ and the higher Gevrey regularity in $\partial_x,\Gamma_t$ derivatives incorporated into $\mathcal{E}^{(\gamma)}$ and $\mathcal{E}^{(\gamma)}_{G,H,\overline{H}}$ will provide strong estimates capable of estimating these terms; see for example Lemma \ref{lem:ExtToInt} below.  
\fi
 
 \ifx
\noindent \textbf{$(\Psi^{(I)}, \Psi^{(E)})$ Decomposition:}  We define the interior cutoff function $\chi_I(v(t,y))$
\begin{align} \label{chi:I:def}
\chi_I(\siming{v(t,y)}) = \begin{cases} 1 \qquad v(t,y) \in  (-\frac12, \frac12), \\ 0 \qquad v(t,y) \in (-\frac34, \frac34)^c . \end{cases}
\end{align} 
We denote $f^I = \chi_I f$. 
In order to apply the ideas of \cite{BM13,HI20} as a nearly black-box, we need to reduce the PDEs to be in a form as close as possible as used therein.
Namely, we divide the effective \siming{stream function?} in the $(z,v)$ variable as the `truly interior' part
\begin{align}
& -\partial_{xx}\psi^I -(1+ h^I)^2 (\partial_v-t\partial_z)^2 \psi^I + (1+h^I)\partial_v h^I (\partial_v - t\partial_z) \psi^I = f^I,\label{Intr_psi_I} \\
& \psi^I(x,v=v(t,\pm 1)) = 0.\n 
\end{align}
and the `exterior' part 
\begin{align}
& -\Delta_t \psi^E = (1-\chi^I)f - \left( (1+h^I)^2 - (1 +h)^2 \right) (\partial_v-t\partial_z)^2 \psi^I \n\\
& \qquad \qquad -  \left((1+h)\partial_v h - (1+h^I)\partial_v h^I\right)(\partial_v - t\partial_z) \psi^I, \label{Intr_psi_E} \\ 
& \psi^E(x,v=v(t,\pm 1)) = 0.\n 
\end{align}
Here $\de_t$ is the Laplace ($\pa_{x}^2+\pa_{y}^2$) in the $(z,v)$ coordinate. 
The estimates on $\psi^I$ are performed as in \cite{BM13,HaoIonescu}; the treatment of $\psi^E$ is discussed in this paper. 

To derive the equation for $\Psi^{(E)}(t,x,y)=\psi^E(t,z(t,x,y),v(t,y))$ in the physical variables, we use the relation $\pa_x=\pa_z,\quad\pa_y=v_y\pav- tv_y\pa_z =v_y(\pav-t\pa_z)$ to derive 
\begin{align}
-\Delta \Psi^{(E)} = &(1-\chi^I)\omega - \left( (1+H^I)^2 - (1 +H)^2 \right) \pav^2 \Psi^{(I)}\n \\
& -  \left((1+H)\pav H - (1+H^I)\pav H^I\right)\pav \Psi^{(I)},\label{dmp_int2_intr}\\
\Psi^{(E)}(x,\pm 1) =& 0.\n
\end{align}
\siming{I am worried about the signs. For example, here it is written in $-\de \psi=\omega. $} 
Here we recall that 
\begin{align}
H^I(t,y)=h^I(t,v(y))=\chi^I (v(y))h(t,v(y))=\chi^I (v(y))H(t,y). 
\end{align} 
\fi

\noindent \textbf{The $(\Psi^{(I)}, \Psi^{(E)})$ Decomposition:}
\ifx \footnote{\siming{
In the $(z,v)$ variable as the `interior' part
\begin{align}
& \partial_{xx}\psi^I +(1+ h^I)^2 (\partial_v-t\partial_z)^2 \psi^I - (1+h^I)\partial_v h^I (\partial_v - t\partial_z) \psi^I = w^I,\label{Intr_psi_I} \\
& \psi^I(x,v=v(t,\pm 1)) = 0.\n 
\end{align}
and the `exterior' part 
\begin{align}
& \Delta_t \psi^E = (1-\chi^I)w + \left( (1+h^I)^2 - (1 +h)^2 \right) (\partial_v-t\partial_z)^2 \psi^I \n\\
& \qquad \qquad +  \left((1+h)\partial_v h - (1+h^I)\partial_v h^I\right)(\partial_v - t\partial_z) \psi^I, \label{Intr_psi_E} \\ 
& \psi^E(x,v=v(t,\pm 1)) = 0.\n 
\end{align}
Here $\de_t$ is the Laplace ($\pa_{x}^2+\pa_{y}^2$) in the $(z,v)$ coordinate.} \myr{Just a recording of the plan to treat the $\psi^{(E)}$-contribution in the interior. We can introduce another cut-off function $\chi_{II}$, which has support strictly contained in the support of $\chi_I$, in the equation for the $\psi^{(E)}$. Now we have that $\de_t (\chi_{II}\psi^{(E)})=\pm [\chi_{II},\de_t ]\psi^{(E)}$. Now the commutator is only present on the support where $\chi_{II}'\neq 0$. And we can choose it such that it is outside the transition region of $\chi_{m+n}$. And hence one can have the high regularity control over the $H.$ Once we derive the bound \eqref{FE_ell}, we can estimate the $\pm [\chi_{II},\de_t ]\psi^{(E)}$ in high Gevrey and hence we have that $\chi_{II}\psi^{(E)}$ is in high Gevrey. This $\chi_{II}$-trick happens before. }} 
\fi
The first decomposition is intended for obtaining estimates of the solution $\psi$ in the interior of the channel. Let us recall the ideas from the companion paper \cite{BHIW24a} without digging into technical details. One uses the cutoffs $\chi_I/\chi_E$ \eqref{chi:I:def} to localized the forcing $w$ and the $H$-quantity \eqref{defn:h:cap} that captures the deviation between $v$ and $y$. Thanks to technical considerations, one ends up with the following
\begin{align}\Delta \Psi^{(I)} = &\ww^I -   \left( (1+H^I)^2 - (1+H)^2 \right) \pav^2 \Psi^{(I)} -  \left((1+H)\pav H - (1+H^I)\pav H^I\right)\pav \Psi^{(I)},\label{dmp_int1_intr}\\
\Delta \Psi^{(E)} = &\ww^E + \left( (1+H^I)^2 - (1+H)^2 \right) \pav^2 \Psi^{(I)} +  \left((1+H)\pav H - (1+H^I)\pav H^I\right)\pav \Psi^{(I)},\label{dmp_int2_intr}\\
&\hspace{1.2 cm}\Psi^{(I)}(x,\pm 1) = 0,\qquad\Psi^{(E)}(x,\pm 1) = 0.\n
\end{align}
Here we used the notations
\begin{align}
F^I(t,y)=\chi^I (v(t,y))F(t,y),\quad F^E(t,y)=\chi^E (v(t,y))F(t,y),
\end{align}
where $F$ is a general function. However, we would highlight that the notations $\Psi^{(I)},\ \Psi^{(E)}$ do not represent truncated functions. The motivation in implementing this decomposition is that when we write the $\Psi^{(I)}$-equation in the special coordinate system (the nonlinear $(z,v)$-coordinate in \cite{BMV14}), the solutions $\Psi^{(I)}$ are easy to estimate. We refer the readers to the companion paper \cite{BHIW24a} and  \cite{BM13,HI20} for further details. The treatment of the $\Psi^{(E)}$-component is discussed in this paper.

\vspace{2  mm}

\noindent
\textbf{The $(\phi^{(I)}, \phi^{(E)})$ Decomposition:} We also consider another decomposition (based on the cutoff functions $\wt\chi_1/\wt\chi_1^\mathfrak c$ \eqref{wt_chi_intro}) to understand the key features of the solutions near the boundary. 
Now we decompose the solution as $\psi_k = \phi_k^{(E)}(t,y) + \phi_k^{(I)}(t,v(t,y))$, where these two quantities satisfy the following elliptic equations\begin{subequations}
\begin{align}   \label{d:phi:I:in}
(\pa_v^2-|k|^2)\ \phi^{(I)}_k(t,v) &=\widetilde{\chi}_1^\mathfrak{c}(v)\ \ww_k(t,y(t,v))+{\wt \chi_1^\mathfrak{c}(v)(\pa_v^2-(v_y \partial_v)^2)\phi_k^{(I)}(t,v)}, \\  \label{d:phi:E:in}
(\pa_y^2-|k|^2)\ \phi^{(E)}_k(t,y) &=\ \widetilde{\chi}_1\lf(v(t,y)\rg) \ww_k(t,y)+{\wt\chi_1(v(t,y))}(\pav^2-\pa_{y}^2)\phi^{(I)}_k(t,v(t,y)),\\
\phi^{(I)}_k(t,v)|_{v=v(t,\pm 1)} & = 0, \qquad
 \phi^{(E)}_k(t,y)|_{y = \pm 1}  =  0. 
\end{align} 
\end{subequations}

\noindent \underline{\textsc{Gevrey Coefficients:}} In order to measure the interior forcing, we define the following Gevrey coefficients 
\begin{align}\label{B_in_mn} B ^\mathrm{low}_{m,n}=\lf(\frac{2^{-(m+n)}}{(m+n)!}\rg)^{4}.
\end{align}
If the generating vorticity is supported in the interior region, we can utilize a smoothing mechanism of the Green's function to derive stronger estimates. Let us consider the new Gevrey index, 
\begin{align}
{ \widehat{\lambda}  :=} & 2\lambda(t), \ \
 \widehat{B}_{m,n} :=  \lf( \frac{ \widehat{\lambda}^{m+n}}{(m+n)!} \rg)^{s},\ \ \widehat{\bold{a}}_{m,n} :=  \widehat{B}_{m,n} \varphi^{n+1}. \label{hat_bf_a_intro}
\end{align}

\subsection{Presentation of Elliptic Energy Functionals}

\noindent \textsc{\underline{Interior Functionals:}} 
\ifx
Here we recall the following energy functional from \cite{BHIW24a}:
\begin{align}
\mathcal{E}_{\text{Int}}(t) := & \norm{\mathfrak{A}(t,\grad)  (\chi_I f) }_{L^2}^2,\\ 
\mathcal{E}_{\text{Int,Coord}}^{(h)} := & \norm{A_R (\chi_Ih)}_{L^2}^2.\label{E_IntCh} 
\end{align}
\fi
{We define the classical Gevrey norms
\begin{align}\label{E_Int}
\mathcal{E}_{\mathrm{Int}}^{\mathrm{low}}[w]:=&\sum_{k\in \mathbb{Z}}\sum_{m+n=0}^\infty \lf(B_{m,n}^{\mathrm{low}}\rg)^2\||k|^m(\pa_v+ikt)^n(\chi_I(t,v)w_k)\|_{L^2}^2=:\sum_{k\in \mathbb{Z}}\mathcal{E}_{\mathrm{Int}}^\mathrm{low}[w_k].
\end{align}


\ifx
We also need the CK terms
\begin{align*}
	\mathcal{CK}_{\overline{H}}^{(\iota)} :=& \sum_{n\ge0} \mathcal{CK}_{\overline{H}, n}^{(\iota)} \\
	\mathcal{CK}_{G}^{(\iota)} :=& \sum_{n\ge0} \mathcal{CK}_{G, n}^{(\iota)} \\
	\mathcal{CK}_{H}^{(\iota)} :=& \sum_{n\ge0} \mathcal{CK}_{H, n}^{(\iota)}
\end{align*}
where
\begin{align}
	\mathcal{CK}_{\overline{H},n}^{(\iota)} := & \sum_{j = 1}^2  \mathcal{CK}^{(\iota;j)}_{\overline{H}, n}(t), \\
	\mathcal{CK}_{\overline{H},n}^{(\iota;1)}:= &  \nu^{i_{\iota}/2}(n+1)^{2\sss-2} \bold{a}_{n+1}^2 \brak{t}^{5+2\ss} \|\partial_y^{i_{\iota}} \overline{H}_n 
	\sqrt{- \partial_{t}W} e^{W/2} \chi_n \|_{L^2}^2,
	\label{CK1:defi} \\
   \mathcal{CK}_{\overline{H},n}^{(\iota;2)} :=&  -  \nu^{i_{\iota}/2} (n+1)^{2\sss-2} \bold{a}_{n+1}   \frac{d}{dt}\paren{\brak{t}\bold{a}_{n+1}} \brak{t}^{4+2\ss}
	\| \partial_y^{i_{\iota}} \overline{H}_n e^{W/2} \chi_n \|_{L^2}^2, \label{CK2:defi} \\
	\mathcal{CK}_{G,n}^{(\iota)} := & \begin{cases}  \paren{4\langle t \rangle^{4-2\eps}t^{-1} - (4-2\eps)t\brak{t}^{2-2\eps}}
		\enorm{\partial_y^{i_{\iota}} G\chi_0}^2 \qquad & n=0
		\\ \sum_{j=1}^{2} \mathcal{CK}_{G, n}^{(\iota; j)}(t)^2 \qquad & n\ge 1.
	\end{cases}\\
	\mathcal{CK}_{G,n}^{(\iota;1)} :=& \bold{a}_{n+1}^2 \brak{t}^{5+2\ss} \|\partial_y^{i_{\iota}} G_n \sqrt{- \partial_{t}W} e^{W/2} \chi_{n-1+i_\iota} \|_{L^2}^2,\\
	\mathcal{CK}_{G,n}^{(\iota;2)} :=	&-  \bold{a}_{n+1} \dot{\bold{a}}_{n+1} 
	\brak{t}^{5+2\ss}
	\| \partial_y^{i_{\iota}} G_n  \chi_{n-1+i_\iota} \|_{L^2}^2,\\
	\mathcal{CK}_{H,n}^{(\iota)} :=& - \nu^{i_{\iota}/2} \brak{t}\bold{a}_{n} \frac{d}{dt}(\brak{t}\bold{a}_{n} )
	\| \partial_y^{i_{\iota}}{H}_n e^{W/2} \chi_n \|_{L^2}^2
	\\&+ \nu^{i_{\iota}/2} \brak{t}^2\bold{a}_{n}^2   \| \partial_y^{i_{\iota}}{H}_n \sqrt{- W_t} e^{W/2} \chi_n \|_{L^2}^2  .
\end{align}
\vspace{2 mm}

\noindent \textsc{\underline{Sobolev Boundary Functionals:}} Due to the degeneracy of the weight, $q$, near the boundaries, we will need to supplement our norms with the following functionals, which do not degenerate at the boundaries, but rather have extra powers of $\nu$. We introduce the following Sobolev energy functionals
\begin{align}
\mathcal{E}_{sob, 0}(t) := & \sum_{k \in \mathbb{Z}} \| \omega_{k,0} e^W\|_{L^2}^2 \\
\mathcal{D}_{sob, 0}(t) := &\sum_{k \in \mathbb{Z}} \|\nu^{\frac12} \nabla_k \omega_{k,0}e^W \|_{L^2}^2 \\
\mathcal{CK}_{sob,0}(t) := & \sum_{k \in \mathbb{Z}}\|   \omega_{k,0}\sqrt{-\pa_t W}    e^{W} \|_{L^2}^2 \\
\mathcal{E}_{sob, n}(t) := & \nu^{2n} \sum_{k \in \mathbb{Z}} \|\nabla \omega_{k,n-1} e^W\|_{L^2}^2, \qquad \qquad n \ge 1 \\
\mathcal{D}_{sob, n}(t) := & \nu^{2n} \sum_{k \in \mathbb{Z}} \|\nu^{\frac12} \nabla_k^2 \omega_{k,n-1}e^W \|_{L^2}^2, \qquad n \ge 1, \\
\mathcal{CK}_{sob, n}(t) := & \nu^{2n} \sum_{k \in \mathbb{Z}} \|\nabla \omega_{k,n-1} \sqrt{-\partial_t W} e^W\|_{L^2}^2, \qquad \qquad n \ge 1
\end{align}
and finally the complete bulk Sobolev norm we propagate: 
\begin{align}
\mathcal{E}_{sob}(t) := \sum_{i = 0}^4 \mathcal{E}_{sob,i}(t), \qquad \mathcal{D}_{sob}(t) := \sum_{i = 0}^4 \mathcal{D}_{sob,i}(t), \qquad \mathcal{CK}_{sob}(t) := \sum_{i = 0}^4 \mathcal{CK}_{sob,i}(t)
\end{align}
\fi
\ifx
It turns out that the bounds on these bulk Sobolev functionals are themselves subtle, and require a delicate interplay between the one-dimensional quantities $\alpha_j^{\pm}(t, x) := \p_y^j \omega(t, x, \pm 1)$. We introduce functionals to measure these one-dimensional quantities as follows:
\begin{align}
\mathcal{E}_{\text{Trace},0}(t) := &\| \nu^{\frac32} \alpha_1 e^W \|_{L^2_x}^2, \qquad \mathcal{D}_{\text{Trace},0}(t) :=  \| \nu^{2} \p_x \alpha_1 e^W\|_{L^2_x}^2, \\
\mathcal{E}_{\text{Trace},1}(t) := &\| \nu^{\frac52} \p_x \alpha_1 e^W \|_{L^2_x}^2, \qquad \mathcal{D}_{\text{Trace},1}(t) :=  \| \nu^2 \p_t \alpha_1 e^W\|_{L^2_x}^2, \\
\mathcal{E}_{\text{Trace},2}(t) := &\| \nu^{\frac72} \p_x^2 \alpha_1 e^W\|_{L^2_x}^2, \qquad \mathcal{D}_{\text{Trace},2}(t) := \| \nu^3 \p_t \p_x \alpha_1 e^W\|_{L^2_x}^2. 
\end{align}
Above, when we use $W$, we mean $W|_{y = \pm 1}(t, x)$. Consequently, we also gain CK functionals: 
\begin{align}
\mathcal{CK}_{\text{Trace},0}(t) := &\| \sqrt{-\p_t W} \nu^{\frac32} \alpha_1 e^W \|_{L^2_x}^2, \\
\mathcal{CK}_{\text{Trace},1}(t) := &\|\sqrt{-\p_t W} \nu^{\frac52} \p_x \alpha_1 e^W \|_{L^2_x}^2,  \\
\mathcal{CK}_{\text{Trace},2}(t) := &\|\sqrt{-\p_t W} \nu^{\frac72} \p_x^2 \alpha_1 e^W\|_{L^2_x}^2. 
\end{align}
We subsequently also need the following maximal regularity functionals: 
\begin{align}
\mathcal{D}_{\text{Trace, Max},1}(t) &:= \| \nu^3 \p_x^2 \alpha_1 e^W \|_{L^2_x}^2, \\
\mathcal{D}_{\text{Trace, Max},2}(t) &:= \| \nu^4 \p_x^3 \alpha_1 e^W \|_{L^2_x}^2. 
\end{align}
Finally, we will need to measure (in $L^2_t$) the quantities $\alpha_4$ and $\p_t \alpha_4$. Hence, we introduce 
\begin{align}
\mathcal{D}_{\text{Trace, Large},0}(t) &:= \| \nu^3 \alpha_4 e^W\|_{L^2_x}^2,\\
\mathcal{D}_{\text{Trace, Large},1}(t) &:= \| \nu^4 \p_x \alpha_4 e^W\|_{L^2_x}^2, \\
\mathcal{D}_{\text{Trace, Large},2}(t) &:= \| \nu^4 \p_t \alpha_4 e^W\|_{L^2_x}^2 + \| \nu^5 \p_{xx} \alpha_4 e^W\|_{L^2_x}^2. 
\end{align}
We will have our full trace norms: 
\begin{align}
\mathcal{E}_{\text{Trace}}(t) := & \sum_{i = 0}^2 \mathcal{E}_{\text{Trace},i}(t), \\
\mathcal{D}_{\text{Trace}}(t) := &\sum_{i = 0}^2 \mathcal{D}_{\text{Trace},i}(t) + \sum_{i = 1}^2 \mathcal{D}_{\text{Trace, Max},i}(t) + \sum_{i = 1}^3 \mathcal{D}_{\text{Trace, Large},i}(t),  \\
\mathcal{CK}_{\text{Trace}}(t) := &\sum_{i = 0}^2 \mathcal{CK}_{\text{Trace},i}(t).
\end{align}

\vspace{2 mm}

\noindent \textsc{Sobolev Cloud Functionals:} It turns out to close our nonlinear analysis, we need to introduce the so-called Sobolev ``Cloud" norms. These norms should be interpreted as an interpolant (even though they cannot be obtained by any interpolation) between interior vorticity and exterior vorticity in the following sense: they retain the weight $e^W$ from the exterior energy functionals, whereas they are localized, due to $\chi_{\text{cloud}}$, to the interior. In particular, we define: 
\begin{align}
\mathcal{E}_{\text{cloud}}(t) := &\sum_{m + n \le 20}  \| \varphi^{n} \p_x^m \Gamma^n \omega e^W \chi_{\text{cloud}} \|_{L^2}^2, \\
\mathcal{D}_{\text{cloud}}(t) := &\sum_{m + n \le 20} \|\varphi^{n} \sqrt{\nu} \nabla \p_x^m \Gamma^n \omega e^W \chi_{\text{cloud}} \|_{L^2}^2 \\
\mathcal{CK}_{\text{cloud}}^{(W)}(t) := &\sum_{m + n \le 20} \|\varphi^{n} \sqrt{-\dot{W}} \p_x^m \Gamma^n \omega e^W \chi_{\text{cloud}} \|_{L^2}^2 \\
\mathcal{CK}_{\text{cloud}}^{(\varphi)}(t) := &\sum_{m + n \le 20}  \|\varphi^{n} \sqrt{-\dot{\varphi}} \p_x^m \Gamma^n \omega e^W \chi_{\text{cloud}} \|_{L^2}^2
\end{align}
\fi
\vspace{2 mm}

\noindent \textsc{\underline{Elliptic Functionals:}}

We measure the two decompositions with different Gevrey norms. 
We consider the following norm for the exterior contribution $\Psi^{(E)}_k(t,x,y)$
\begin{align}\label{FE_ell}
\mathcal{F}_{ell}^{(E)}(t) := &\sum_{k\in \mathbb{Z}\backslash\{0\}}\sum_{m+n=0}^\infty \lf(\frac{(2\widetilde{\lambda})^{2(m+n)/r}}{(m+n)!^{2/r}}\rg)\|\chi_{m+n}|k|^m q^n\Gamma_k^n  \Psi^{(E)}_k\|_{L^2}^2.
\end{align}
{Here $\wt \lambda>0$ is an arbitrary parameter. For the application in our companion paper \cite{BHIW24a}, the Gevrey radius $2\wt \lambda$ will be chosen such that the functional $\mathcal{F}_{ell}^{(E)}$ dominates the $\|A_R (\cdots)\|_{L^2}$ norm.  } 
We now introduce a set of functionals to measure the solution $\phi^{(I)}$, defined in \eqref{d:phi:I:in}: 
\begin{align}
\mathcal{E}_{ell}^{(I, out)}(t) := & \sum_{k\in \mathbb{Z} {\backslash\{0\}}}\sum_{m+n=0}^\infty \widehat{\bold{a}}_{m,n}^2\|\chi_1 |k|^m(\pa_v+ikt)^n\pa_v^\ell \phi^{(I)}_k\|_{L^2_v}^2,\quad \ell\in\{0,1,2\}, \\
\mathcal{E}_{ell}^{(I, full)}(t) := & \sum_{k\in \mathbb{Z} {\backslash\{0\}}}\sum_{m+n=0}^{1000} \widehat{B}_{m,n}^2\| |k|^m(\pa_v+ikt)^n  \phi^{(I)}_k\|_{L^2_v}^2.
\end{align}


\noindent \textsc{\underline{The $ICC$ Pseudo-differential Operators:}}
If the solution is supported near the boundary, one needs to control various commutator terms involving $q$. It turns out that the following $ICC$-operators are crucial to express these commutators 
\begin{subequations}\label{SJ}
\begin{align}
\label{S_nq}
S^{(a,b,c)}_{m,n}f_k= &{\lf(\frac{m+n}{q}\rg)^a}\pa_y^b |k|^{c} (|k|^m q^n\Gamma_k^n f_k)\  \mathbbm{1}_{\mathbb {S}_n^\ell}(a,b,c),\quad a+b+c=\ell\leq 2 ; \\
\label{J_nq}
J^{(a,b,c)}_{m,n}f_k= &{\lf(\frac{m+n}{q}\rg)^a}\pav ^b |k|^{c} ( \chi_{m+n}|k|^m q^n\Gamma_k^n f_k)\  \mathbbm{1}_{\mathbb {S}_n^\ell}(a,b,c),\quad a+b+c=\ell\leq3 ; \\ \label{J_0}
J^{(a,b,0)}_{0,n}f_0= &{\lf(\frac{n}{q}\rg)^a}\pav^b  ( \chi_{n} q^n\pav^n f_0)\ \mathbbm{1}_{\mathring{\mathbb S}_n^{\ell}}(a,b), \quad a+b=\ell\leq 3. 
\end{align}%
\end{subequations}
Here, the parameters $a,b,c$ take values in $\mathbb{N}$ and $\pav$ is defined as follows\begin{align}\pav:=v_y^{-1}\pa_y.\end{align} 
The index sets $\mathbb{S}_n^{\ell}$ and $\mathring {\mathbb {S}}_n^{\ell  }$ are defined as follows
\begin{align}\label{Snell}
\mathbb {S}_n^{\ell}:=&\lf\{(a,b,c)\in\mathbb N^{3}\big|a+b+c=\ell\rg\}\cap\lf\{(a,b,c)\in\mathbb{N}^3\big|a+b+c\leq n\text{ or }a=0\rg\},\\
\n \mathring {\mathbb{S}}_n^{\ell}:=&\lf\{(a,b)\in\mathbb N^{2}\big| a+b=\ell \rg\}\cap\lf\{(a,b)\in\mathbb N^{2}\big|a+b\leq n \text{ or }a=0\rg\}.
\end{align}
Moreover, $S
_{m,n}^{(a,b,c)} f_k= J_{m,n}^{(a,b,c)}f_k\equiv 0$ for $n<0. $ We use the following notations to denote the vectors formed by the operations $S_{m,n}^{(\cdots)},\ J_{m,n}^{(\cdots)}$: \begin{subequations}\label{SJ_vect}
\begin{align}
S_{m,n}^{(\leq\zeta)}f_k:=\{S^{(a,b,c)}_{m,n}f_k\}_{a+b+c\leq \zeta},&\quad S_{m,n}^{(\zeta)}f_k:=\{S^{(a,b,c)}_{m,n}f_k\}_{a+b+c= \zeta};\label{S_vec}\\ 
J_{m,n}^{(\leq\zeta)}f_k:=\{J^{(a,b,c)}_{m,n}f_k\}_{a+b+c\leq \zeta},&\quad J_{m,n}^{(\zeta)}f_k:=\{J^{(a,b,c)}_{m,n}f_k\}_{a+b+c=\zeta};\label{J_vec}\\
J_{0,n}^{(\leq\zeta)}f_0:=\{J^{(a,b,0)}_{0,n}f_0\}_{a+b\leq \zeta},&\quad J_{0,n}^{(\zeta)}f_0:=\{J^{(a,b,0)}_{m,n}f_0\}_{a+b=\zeta}.\label{J0_vec}
\end{align}
\end{subequations}
Moreover, we measure these vectors with the classical $L^2$-norm of vectors:\begin{subequations}\label{SJ_norm}
\begin{align}\lf\|S_{m,n}^{(\leq\zeta)}f_k\rg\|_{L^2}^2:=&\sum_{a+b+c\leq\zeta}\lf\|S_{m,n}^{(a,b,c)}f_k\rg\|_{L^2}^2,\quad \lf\|S_{m,n}^{(\zeta)}f_k\rg\|_{L^2}^2:=\sum_{a+b+c=\zeta}\lf\|S_{m,n}^{(a,b,c)}f_k\rg\|_{L^2}^2;\label{S_vec_n}\\
\lf\|J_{m,n}^{(\leq\zeta)}f_k\rg\|_{L^2}^2:=&\sum_{a+b+c\leq\zeta}\lf\|J_{m,n}^{(a,b,c)}f_k\rg\|_{L^2}^2,\quad\lf\|J_{m,n}^{(\zeta)}f_k\rg\|_{L^2}^2:=\sum_{a+b+c=\zeta}\lf\|J_{m,n}^{(a,b,c)}f_k\rg\|_{L^2}^2,\label{J_vec_n}\\
\lf\|J_{0,n}^{(\leq\zeta)}f_0\rg\|_{L^2}^2:=&\sum_{a+b\leq\zeta}\lf\|J_{0,n}^{(a,b,0)}f_0\rg\|_{L^2}^2,\quad\lf\|J_{0,n}^{(\zeta)}f_0\rg\|_{L^2}^2:=\sum_{a+b=\zeta}\lf\|J_{0,n}^{(a,b,0)}f_0\rg\|_{L^2}^2.\label{J0_vec_n}
\end{align}
\end{subequations}
It is worth highlighting that there are two components to these norms. Part of the norm is the classical Sobolev norm which only involves $\pav=\vyn\pa_y
$ and $|k|$. The other part of the norm involves the boundary weight $\frac{m+n}{q}$. For $a+b+c> n$, we will not consider the boundary weight factor $\frac{m+n}{q}$ because the quantity might not be well-defined, but we can still consider the classical Sobolev part. 

We will derive the estimates for the following functional
\begin{align}
\mathcal{J}_{{ell}}^{(\ell)}(t):=\sum_{k\in\mathbb{Z}\backslash\{0\}}\ \sum_{m,n\geq 0}\ \sum_{a+b+c=\ell}\ {\bf a}_{m,n}^2 \|J_{m,n}^{(a,b,c)}\phe_k\|_{L_y^2}^2,\quad \ell\in\{1,2,3\}.\label{Gj_intro}
\end{align}
We highlight that these norms control the usual $H^{\ell}$ norm of the function. We further highlight that for the $\ell=1,2$ case, we will obtain $L_t^\infty \mathcal{J}_{ell}^{(\ell)}$-estimates whereas for the $\ell=3$ case, we will only have $L_t^2\mathcal{J}_{ell}^{(3)}$ bounds.  
\vspace{2 mm}

\ifx 
\subsection{Bootstraps and Main Propositions}


We will execute a standard bootstrap argument. Specifically, we assume the following estimates hold on some time interval $[0,T]$ and show that the same inequalities hold with `4' replaced with `2' for all $\eps$ sufficiently small, thus permitting the extension of the desired estimates on the coordinate system and vorticity for all time (see \cite{BHIW24a} and Section \ref{} for additional details).
\begin{itemize}
\item Interior profile estimates (see Section \ref{etc} for definition of the Fourier multipliers), 
\begin{align}  \label{boot:Intf}
&\sup_{0 \le t \le T} \mathcal{E}_{\text{Int}}(t) + \int_0^T [ \mathcal{CK}_{\lambda}(t) +  \mathcal{CK}_{w}(t) +  \mathcal{CK}_{M}(t) ] dt + \int_0^T \mathcal{D}_{\text{Int}}(t)  dt \leq 4 \eps^2 
\end{align}

\item Interior coordinate system estimates:
\begin{align} \label{boot:IntH}
\sup_{0 \le t \le T} \mathcal{E}_{\text{Int,Coord}}^{(h)} + \sum_{\iota \in \{\lambda, w \}}  \int_0^T \mathcal{CK}_{\text{Int,Coord}, \iota}^{(h)} dt +  \int_0^T \mathcal{D}_{\text{Int,Coord}}^{(h)} dt \leq 4 \eps^2 \\
\sup_{0 \le t \le T} \mathcal{E}_{\text{Int,Coord}}^{(g)} + \sum_{\iota \in \{\lambda, w \}}  \int_0^T \mathcal{CK}_{\text{Int,Coord}, \iota}^{(g)} dt +  \int_0^T \mathcal{D}_{\text{Int,Coord}}^{(g)} dt \leq 4 \eps^2 \\
\sup_{0 \le t \le T} \mathcal{E}_{\text{Int,Coord}}^{(\bar{h})} + \sum_{\iota \in \{\lambda, w \}}  \int_0^T \mathcal{CK}_{\text{Int,Coord}, \iota}^{(\bar{h})} dt +  \int_0^T \mathcal{D}_{\text{Int,Coord}}^{(\bar{h})} dt \leq 4 \eps^2 
\end{align}
\item Exterior vorticity estimates:
\begin{align} \label{boot:ExtVort}
\sup_{0 \le t \le T} \mathcal{E}^{(\gamma)} + \sum_{\iota \in \{\varphi, W \}} \int_0^T \mathcal{CK}^{(\gamma; \iota)}(t) dt +  \int_0^T \mathcal{D}^{(\gamma)}(t) dt \le 4 \eps^2 \\
\sup_{0 \le t \le T} \mathcal{E}^{(\alpha)} + \sum_{\iota \in \{\varphi, W \}} \int_0^T \mathcal{CK}^{(\alpha; \iota)}(t) dt +  \int_0^T \mathcal{D}^{(\alpha)}(t) dt \le 4 \eps^2\\
\sup_{0 \le t \le T} \mathcal{E}^{(\mu)} + \sum_{\iota \in \{\varphi, W \}} \int_0^T \mathcal{CK}^{(\mu; \iota)}(t) dt +  \int_0^T \mathcal{D}^{(\mu)}(t) dt \le 4 \eps^2.
\end{align}
\item Exterior coordinate system estimates: 
\begin{align} \label{boot:H}
\max_{\iota' \in \set{H,\overline{H},G}} \sup_{0 \le t \le T} \mathcal{E}^{(\gamma)}_{\iota'} + \sum_{\iota \in \{\varphi, W \}} \int_0^T \mathcal{CK}^{(\gamma; \iota)}_{\iota'}(t) dt +  \int_0^T \mathcal{D}^{(\gamma)}_{\iota'}(t) dt \le 2\eps^2 \\
\max_{\iota' \in \set{H,\overline{H},G}} \sup_{0 \le t \le T} \mathcal{E}^{(\alpha)}_{\iota'} + \sum_{\iota \in \{\varphi, W \}} \int_0^T \mathcal{CK}^{(\alpha; \iota)}_{\iota'}(t) dt +  \int_0^T \mathcal{D}^{(\alpha)}_{\iota'}(t) dt \le 2\eps^2.
\end{align}
\item Sobolev boundary estimates: 
\begin{align} \label{boot:sob}
&\sup_{0 \le t \le T} \mathcal{E}_{sob}(t) + \int_0^T \mathcal{CK}_{sob}(t) dt + \int_0^T \mathcal{D}_{sob}(t)  dt \leq 4 \eps^2. 
\end{align}
\item Cloud norm estimates:
\begin{align} \label{boot:cloud}
&\sup_{0 \le t \le T} \mathcal{E}_{\text{cloud}}(t) + \int_0^T \mathcal{CK}_{\text{cloud}}^{(\varphi)}(t) + \mathcal{CK}_{\text{cloud}}^{(W)}(t) dt + \int_0^T \mathcal{D}_{\text{cloud}}(t)  dt \leq 4 \eps^2. 
\end{align}
\end{itemize}

The main step in the proof of Theorem \ref{thm:main} is the following. 
\begin{proposition} \label{prop:boot}
For all $\zeta > 0$ sufficiently small, and all $\lambda_0 > 0$, $\exists \eps_0$ such that for all $\eps \in (0,\eps_0)$ and all $\nu \in (0,1)$, if the bootstrap hypotheses \eqref{boot:Intf},\eqref{boot:IntH},\eqref{boot:ExtVort}, and \eqref{boot:H} hold on $[0,T]$ for some $T < \nu^{-1/3-\zeta}$, then the same inequalities in fact hold with `4' replaced with `2'.
\end{proposition}

Proposition \ref{prop:boot} along with the standard well-posedness theory implies that the estimates \eqref{boot:Intf},\eqref{boot:IntH},\eqref{boot:ExtVort}, and \eqref{boot:H} hold on $[0,\nu^{-1/3-\zeta}]$. Denote $t_\ast = \nu^{-1/3-\zeta}$.  
At $t_\ast$ we then deduce the estimates
\begin{align*}
\norm{\omega_0(t_\ast)}_{H^4} \lesssim \eps \\
\norm{\omega_{\neq}(t_\ast)}_{H^4} \lesssim \nu^{10} e^{\delta' \nu^{1/3} t_\ast},  
\end{align*}
for some small $\delta' > 0$.
The results of \cite{BHIW23I} then imply the following global-in-time estimates for $t \geq t_\ast$ (possibly after adjusting $\delta'$) 
\begin{align*}
\norm{\omega_0(t)}_{L^2} & \lesssim e^{-\nu(t-t_\ast)}\eps \approx e^{-\nu t} \eps \\ 
\norm{\omega_{\neq}(t)}_{L^2} & \lesssim \nu^{10} e^{\delta' \nu^{1/3} t}; 
\end{align*}
Combined with the estimates proved in Proposition \ref{prop:boot}, these imply Theorem \ref{thm:main}. 

\subsection{Proof of Proposition \ref{prop:boot} \& Sub-Propositions}
In this section we briefly outline the strategy to proving Proposition \ref{prop:boot}.   

There are really four main kinds of estimates in the proof.
The first kind are those that are done in the $(z,v)$ coordinate system and use the Fourier multiplier methods of \cite{BM13,HaoIonescu}; namely, the inequalities \eqref{boot:Intf} and \eqref{boot:IntH}. 
The second kind are high regularity estimates on the exterior vorticity and coordinate system variables which simultaneously quantify the gain in Gevrey index away from the interior, the strong localization away from the boundary, and Gevrey regularity with respect to the almost-commuting vector field $\Gamma_t$ that is adapted to match the $\partial_v$ derivatives. However, these estimates are done with the co-normal weight $q$, and so regularity is not obtained all the way to the boundary.  
These are the family of inequalities \eqref{boot:ExtVort} and \eqref{boot:H}.
The next family are the Sobolev boundary estimates, which propagates strong localization away from the boundary and finite regularity all the way to the boundary.
Hence, these estimates must deal with the boundaries in a more fundamental way, but have the advantage that we only need a few derivatives.
These are \eqref{boot:sob}.
The last family of inequalities are \eqref{boot:cloud}, which are the Sobolev cloud norms. These estimates are used to obtain \red{What exactly?}

\jacob{I am not sure really how to proceed...}
\fi

\subsection{Main Theorem: Elliptic Estimates} \label{sec:elliptic}
Note that as a consequence of the assumptions in made in \eqref{asmp}, there holds 
\begin{align}
\sum_{n=0}^{\infty}B_{0,n}^2\varphi^{2n}\|J^{(\leq 1)}_{0,n}(v_y^2-1) \|_{L^2(\text{supp}\wt\chi_1)}^2\lesssim&\exp\{-\nu^{-1/6}\}.
\end{align}

\begin{theorem}[Section \ref{sec:Elliptic}]  \label{pro:ell:intro} Assume  that the hypotheses \eqref{asmp} hold and recall the definitions \eqref{E_Int}.
 The $\mathcal{J}_{ell}$ functionals satisfy the following bounds: 
\begin{align}
\mathcal{J}_{ell}^{(1)} \lesssim& e^{-\nu^{-1/9}}\lf( \mathcal{E}^{(\gamma)}[\ww] +\lf(\mathcal{E}_{H}^{(\gamma)}+\mathcal{E}_{H}^{(\al)}\rg)\mathcal{E}_{\mathrm{Int}}^{\mathrm{low}}[\ww]\rg), \label{jell:1}\\
\mathcal{J}_{ell}^{(2)} \lesssim& e^{-\nu^{-1/9}}\lf(\mathcal{E}^{(\gamma)}[\ww] +\lf(\mathcal{E}_{H}^{(\gamma)}+\mathcal{E}_{H}^{(\al)}\rg)\mathcal{E}_{\mathrm{Int}}^{\mathrm{low}}[\ww]\rg), \label{jell:2}\\
\mathcal{J}_{ell}^{(3)} \lesssim& e^{-\nu^{-1/9}}  \mathcal{D}^{(\gamma)}[\ww] +e^{-\nu^{-1/9}}\sum_{\iota\in\{\al,\gamma\}}\lf(\mathcal{E}_{H}^{(\iota)}+\mathcal{D}_{H}^{(\iota)}\rg)\lf(\mathcal{E}^{(\gamma)}[\ww]
+\mathcal{E}_{\mathrm{Int}}^{\mathrm{low}}[\ww]\rg). \label{jell:3}
\end{align}
The $\mathcal{E}_{ell}$ functionals satisfy the following bounds: 
\begin{align} 
\mathcal{E}_{ell}^{(I, out)}(t)  \lesssim &  _\mathfrak{n}\frac{1}{ \lan t\ran^2(|k|t)^{2\mathfrak{n}}}\mathcal{E}_{\mathrm{Int}}^{\mathrm{low}}[\ww];\label{eell:out} \\
\mathcal{E}_{ell}^{(I, full)}(t) \lesssim & \frac{1}{\langle t \rangle^4} \mathcal{E}_{\mathrm{Int}}^{\mathrm{low}}[\ww].  \label{eell:sob} 
\end{align}
Finally, the $\mathcal{F}_{ell}^{(E)}$ functional satisfies the following bounds: 
\begin{align} 
\mathcal{F}_{ell}^{(E)}(t) \lesssim &  e^{-\nu^{-1/8}}\lf(\mathcal{E}^{(\gamma)}[\ww] + \mathcal{E}_{\mathrm{Int}}^{\mathrm{low}}[\ww] \right)\lf(1+\lf(\mathcal{E}_H^{(\al)}\rg)^2 +\lf(\mathcal{E}^{(\gamma)}_H\rg)^2\rg). \label{fell:e}
\end{align}
\end{theorem}
\begin{remark} The bounds above are split into three groups. Informally, the first, \eqref{jell:1} -- \eqref{jell:3} measure information moving from exterior to exterior. The second, \eqref{eell:out} -- \eqref{eell:sob} captures interior to exterior. The third, \eqref{fell:e} captures exterior to interior.  
\end{remark}


 \label{sec:Outline}


\section{Parabolic Regularity Estimates} \label{sec:ext:lin}
In this section, we prove the $(\gamma, \alpha, \mu)$ energy inequalities appearing in Proposition \ref{pro:light:on}.


\subsection{Definitions of the Mode-by-mode Functionals}\label{sec:k_by_k_func}

\ifx\sameer{The ``bootstrap" in the hypocoercivity argument is a bit different: we induct on $n$. Not sure if we should be stating this in the beginning \dots}
\ifx We introduce the following notation: 
\begin{definition} Let $j = 0, 1$, $k \in \mathbb{Z}$, and $n \in \mathbb{N}$. The quantities $\mathcal{D}^{(j)}_{m,n;k}(t)$, $\mathcal{E}^{(j)}_{m,n;k}(t)$ refer to a quantity such that 
\begin{align} \label{defn:I}
\int_0^{T} \sum_{k \in \mathbb{Z}} |\mathcal{D}^{(j)}_{m,n; k}(t)|^2 dt \le& \| \omega \|_{X^j_{ m,n}}^2, \\
\sup_{0 \le t \le T} \sum_{k \in \mathbb{Z}} |\mathcal{E}^{(j)}_{m,n;k}(t)|^2 \le & \| \omega \|_{X_{m, n}^{j}}^2. 
\end{align}
\siming{Siming: I find the definition here a little confusing. Need to clarify the meaning here.}
\end{definition}
\fi
\fi

\ifx
This section will lay out the general hypocoercivity functional adapted to the exterior region and sketch the strategy to derive nonlinear enhanced dissipation. The task is to estimate terms that appear when one takes the time derivative of the hypocoercivity functional. We will only highlight the main estimates in this section to sketch the ideas. However, some estimates are technical and challenging, and we will postpone their proofs to Section \ref{...} and Section {...}.  

\siming{Before diving into the actual hypocoercivity functional \eqref{Hypo_Fnc}, we first use a simple example to motivation our definition. We consider the following passive scalar equation in the channel 
\begin{align} 
\pa_t f_\nq+ \overline{U}_0^x(t,y) \pa_x f_\nq=\nu\de f_\nq,\quad f(t,x,y=\pm 1)=0,\quad f_{\nq}(t=0)=f_{\text{in};\nq}.
\end{align}
Here $\overline{U}^x_0 (t,y)=y+U^x_0(t,y)$. 
The Dirichlet boundary  condition $f(t,x,y=\pm 1)\equiv 0$ yields boundary constraint
\begin{align}
\pa_{yy}f_k(t,y=\pm 1)=0.\label{ext_cnd}
\end{align}
The classical Hypocoercivity functional is defined as follows
\begin{align}\mathcal{F}[f_k]=\alpha\nu^{2/3}|k|^{-2/3}\|\pa_{y}f_k\|_2^2+2\beta\nu^{1/3}|k|^{-4/3}k\mathrm{Re}\langle i (1+\pa_y U_0^x) f_k,\partial_y f_k\rangle+\gamma \|f_k\|_2^2.
\end{align}
Here $\al, \beta,\gamma$ are three parameters to be chosen. The inner product $\lan \cdot, \cdot\ran$ is defined as $\lan f,g\ran=\int_{-1}^1f\overline{g}dy. $ The time-dependent shear flow $y+U_0^x$ is strictly monotone, and $\|\pa_y U_0^x\|_\infty<1/2.$ One choose $\al, \gamma$ large enough so that the following equivalence relation holds
\begin{align}
\mathcal{F}[f_k]\approx \|f_k\|_2^2+\nu^{2/3}|k|^{-2/3}\|\pa_{y}f_k\|_2^2.\label{equiv}
\end{align}
Now we take the time derivative of the above functional and use the boundary conditions $f_k(t,y=\pm 1)=\pa_{yy} f_k(t,y=\pm 1)=0$, \eqref{ext_cnd}, and end up with the following
\begin{align}
\frac{d}{dt}&\mathcal{F}[f_k]\\
\leq&-2 \al \nu^{5/3} |k|^{-2/3}\|\pa_{yy}f_k\|_2^2+2\al \nu^{2/3} |k|^{1/3}\|\pa_y\overline{U}^x_0\|_\infty\| f_k\|_2\|\pa_{y}f_k\|_2\\
&+2\beta \nu^{1/3}|k|^{-4/3}k\mathrm{Re}\lan i (\nu\pa_{yy} f_k-\overline{U}^x_0 ik f_k),\pa_y f_k\ran+2\beta \nu^{1/3}|k|^{-4/3}k\mathrm{Re}\lan i f_k,\pa_y (\nu\pa_{yy} f_k-\overline{U}^x_0ik f_k)\ran\\
&-2\gamma \nu\|\pa_y f_k\|_2^2\\
\leq&-\al 2 |k|^{-2/3}\nu^{5/3}\|\pa_{yy}f_k\|_2^2+2\al \nu^{2/3} \|1+\pa_y{U}_0^x\|_\infty^{1/2}|k|^{1/3}\| \sqrt{1+\pa_y U_0^x}f_k\|_2\|\pa_{y}f_k\|_2\\
&-2\beta \nu^{1/3}|k|^{2/3}\lf\|\sqrt{1+\pa_y{U}_0^x}f_k\rg\|_2^2+4\nu^{4/3}\beta |k|^{-1/3}\|\pa_{yy} f_k\|_2 \|\pa_y f_k\|_2  -2\gamma \nu \|\pa_y f_k\|_2^2\\
\leq&- \al |k|^{-2/3}\nu^{5/3}\|\pa_{yy} f_k\|_2^2- \beta \nu^{1/3}|k|^{2/3}\lf\|\sqrt{1+\pa_y{U}_0^x} f_k\rg\|_2^2-\lf(2\gamma-C\frac{\al^2}{\beta}\|1+\pa_y{U}_0^x\|_\infty-C\frac{\beta^2}{\al}\rg)\nu\|\pa_yf_k\|_2^2.
\end{align} 
Now we choose the $\gamma=\gamma(\|1+\pa_yU_0^x\|_\infty,\al, \beta)$ large enough, recall the relation \eqref{equiv}, and apply the estimate $\|\pa_yU^x_0\|_\infty\leq \frac{1}{2}$ to get decay estimate, i.e., 
\begin{align}
\frac{d}{dt}\mathcal{F}\leq -\frac{1}{C}\nu^{1/3}|k|^{2/3}(\nu^{2/3}|k|^{-2/3}\|\pa_y f_k\|_2^2+\|f_k\|_2^2)\leq- \frac{1}{C}\nu^{1/3}|k|^{2/3}\mathcal{F}.
\end{align}
This is the general Hypocoercivity argument for strictly monotone shear flow in the channel. The remaining part of the section is to apply this idea on a wide spectrum of regularity levels. 
}
\fi

We introduce some notations which will be in use throughout this section. First, compared to \eqref{ef:a} -- \eqref{ef:c}, \eqref{df:a} -- \eqref{df:c}, we want to introduce notations for each fixed index, $(m,n)$ and frequency $k$. Therefore, we supply the following: 
\begin{definition} Let $k \in \mathbb{Z}$, and $m, n \in \mathbb{N}$. Consider the quantity \begin{align}\mathring{\omega}_{m,n;k}:= |k|^m q^n \Gamma_k^n \omega_k,
\end{align} where $q$, $\Gamma_k$ are defined in  \eqref{q:defn} and \eqref{defn:Gamma}, respectively. Further recall the coefficients  $\bold{a}_{m,n}$ \eqref{a:weight}. We define an energy to capture the size of $\mathring{\omega}_{m,n;k}$,
\begin{align}  \n \mathcal{E}_{m,n;k} :=& \myr{c_\al}\mathcal{E}^{(\alpha)}_{m,n;k}+\mathcal{E}^{(\mu)}_{m,n;k} + \mathcal{E}^{(\gamma)}_{m,n;k}   \\ &\hspace{-1.7cm} :=  \myr{c_\al}\bold{a}_{m,n}^2\nu\|\pa_y\mathring{\omega}_{m,n;k} e^W\chi_{m+n}\|_{L^2}^2+\bold{a}_{m,n}^2\nu|k|^{2}\|\mathring{\omega}_{m,n;k} e^W\chi_{m+n}\|_{L^2}^2   
 +\bold{a}_{m,n}^2\|\mathring{\omega}_{m,n;k} e^W\chi_{m+n}\|_{L^2}^2\label{defn:F:1}.  
\end{align}
Here $\chi_{m+n}$ and $W$ are defined in \eqref{chi} and \eqref{defndW}, respectively. The parameter $c_\al\in[\frac{1}{16},\frac{1}{4}]$ is a universal constant chosen in Proposition \ref{pro:prf_pr22}. 
The following quantities naturally appear when one takes the time derivatives of the energy \eqref{defn:F:1} above:

\noindent
a) Diffusion terms $\mathcal{D}$:  
\begin{align}
\n  \mathcal{D}_{m,n;k} :=&   \mathcal{D}^{(\alpha)}_{m,n;k} +   \mathcal{D}^{(\mu)}_{m,n;k}+ \mathcal{D}^{(\gamma)}_{m,n;k} \\ \n
  :=&\bold{a}_{m,n}^2\nu^{2}\|\na_k\pa_y\mathring{\omega}_{m,n;k}e^W\chi_{m+n}\|_{L^2}^2  +\bold{a}_{m,n}^2\nu^{2}|k|^{2}\|\na_k\mathring{\omega}_{m,n;k}e^W\chi_{m+n}\|_{L^2}^2\\
 &+\bold{a}_{m,n}^2\nu\|\na_k \mathring{\omega}_{m,n;k}e^W\chi_{m+n}\|_{L^2}^2.\label{defn:D:1}
\end{align}
Here $\na_k:=(ik,\pa_y).$
\siming{\footnote{The $\mathcal{D}$ terms here and the ones used in the elliptic section are different. But up to lower order terms, they should be the same. (Checked.)}}

\noindent
b) Cauchy-Kovalevskaya terms $\mathcal{CK}$:\begin{subequations}\label{CK_H}
\begin{align}
\mathcal{CK}_{m,n;k}^{(\iota;\varphi)}:=&-\frac{(1+n)\dot \varphi}{\varphi}\ \mathcal{E}^{(\iota)}_{m,n;k},\quad \iota\in\{\al,\mu,\gamma\},\label{CK_phi_B_1}\\
\mathcal{CK}_{m,n;k}^{(\al; W)}:=& \bold{a}_{m,n}^2\nu\left\|\sqrt{ -\pa_t W  }\pa_y \mathring{\omega}_{m,n;k}e^W\chi_{m+n}\right\|_{L^2}^2,\label{CK_W_al_1}\\ 
\mathcal{CK}_{m,n;k}^{(\mu; W)}:=& \bold{a}_{m,n}^2\nu|k|^{2}\lf\|\sqrt{-\pa_t W}\mathring{\omega}_{m,n;k}e^W\chi_{m+n}\rg\|_{L^2}^2,\label{CK_W_mu_1}\\ 
\mathcal{CK}_{m,n;k}^{(\gamma; W)}:=&\bold{a}_{m,n}^2\lf\|\sqrt{-\pa_t W}\mathring{\omega}_{m,n;k}e^W\chi_{m+n}\rg\|_{L^2}^2.\label{CK_W_ga_1}
\end{align}
\end{subequations}
\ifx
\noindent
c) Hypocoercivity gain $\mathcal{G}$:
\begin{align}
\mathcal{G}_{m,n;k}:=\bold{a}_{m,n}^2\nu^{\frac{1}{3}}|k|^{\frac{2}{3}}\lf\|\sqrt{1+\pa_y U_0^x}\mathring{\omega}_{m,n;k}e^W\chi_{m+n}\rg\|_{L^2}^2. \label{G}
\end{align}
\fi
\end{definition} 

This concludes the subsection. The remaining part of the section is organized as follows: In subsection \ref{sec:comstr}, we derive the fundamental equation for the higher conormal derivatives of the solution $\mr\omega_{m,n;k}$; in subsection \ref{sec:amg}, we estimate the time derivative of each component of the energy \eqref{defn:F:1}; in subsection \ref{sec:con}, we conclude the proof of the key proposition; in the last subsection \ref{sec:techlem}, we provide some technical lemmas. 


\subsection{Commutator Structures}\label{sec:comstr}
We derive now the equation satisfied by $\mathring{\omega}_{m,n;k}$.
\ifx{\color{blue}\begin{lemma} The following identity is valid
\begin{align} \label{rep1}
(q(y) \Gamma_k)^n   = q^n \Gamma_k^n + \sum_{m = 1}^{n-2} b_{m}^{n}(y) q^{n-m} \Gamma_k^{n-m} + | q'(y) |^{n-1} q(y) \Gamma_k,
\end{align}
for functions $b_{m,n}(y)$ that are given by 
\begin{align} \label{b:exp}
b_{m}^{n}(y) := & \alpha^n_m |q'(y)|^{m} + \widetilde{b}_m^n(y).
\end{align}
The functions $\widetilde{b}_m^n(y)$ satisfy the bounds: 
\begin{align}
\text{bounds on $\widetilde{b}_m^n(y)$ \dots}
\end{align}  
and the coefficients $\alpha^n_m$ satisfy the asymptotic behavior 
\begin{align} \label{bdsalphanm}
\alpha^n_m \sim \begin{cases} n^{2m} \qquad 1 \le m \le n-2 \\ 1 \qquad m = n-1 \end{cases} 
\end{align}
\end{lemma}}
\begin{remark} The interpretation of \eqref{b:exp} should be that $\widetilde{b}_m^n(y)$ are the ``non-principal" parts of the operator, due to extra vanishing of $\pa_y^j q$, for $j \ge 2$.  
\end{remark}
\begin{proof} We proceed by induction on $n$. We will extract the following inductive formulae for $\alpha^n_m$:
\begin{align} \label{alphasys}
\alpha^{n+1}_m = \begin{cases} n+ \alpha^n_{1}, \qquad m = 1 \\ \alpha^n_m + (n+1 -m) \alpha^n_{m-1}, \qquad 2 \le m \le n-1 \\ \alpha^n_{n-1}, \qquad m = n  \\ \end{cases},
\end{align} 
with initialization 
\begin{align}
\alpha^2_1 = 1, \qquad (m, n) \in \{ n \ge 2, 1 \le m \le n-1 \ \}.
\end{align}
We first of all start the induction with $m = 1$, which upon reading \eqref{alphasys}, we get $\alpha^{n+1}_1 = n + \alpha^n_1$. \sameer{To finish \dots}
\end{proof} 
\fi

\begin{lemma} \label{lem:eq:comp:a} Let $\omega_k$ satisfy the equation \eqref{M1a}. Then $\mathring{\omega}_{m,n;k}=|k|^mq^n\Gamma_k^n \omega_k$ satisfies the following system,
\begin{align}
\label{aff1}\pa_t \mathring{\omega}_{m,n;k} &+ ik(y + U^x_0(t, y))\mathring{\omega}_{m,n;k} - \nu \Delta_k \mathring{\omega}_{m,n;k} =\bold F_k^{(m,n)}+\bold C_{\mathrm{trans},k}^{(m,n)}+\bold C_{\mathrm{visc},k}^{(m,n)}+ \mathbf{C}^{(m,n)}_{q,k},
\end{align}
where the above commutators are defined by \begin{subequations}\label{T_Ctvq}
\begin{align} 
\bold F_k^{(m,n)}:=& |k|^mq^n \Gamma_k^n f_k;\\ \label{ghtyU:1}
\bold C_{\mathrm{trans},k}^{(m,n)}:=&-|k|^mq^n\sum_{\ell = 1}^n \binom{n}{\ell} \pav^\ell G\ \   \Gamma_k^{n-\ell+1} \omega_k;\\ \label{ghtyU:2}
\bold C_{\mathrm{visc},k}^{(m,n)}:=&\nu|k|^mq^n \sum_{\ell = 1}^n \binom{n}{\ell}  \pav^\ell v_y^2 \ \ \pav^2 \Gamma_k^{n-\ell} \omega_k;\\ \n
 \mathbf{C}^{(m,n)}_{q, k} := & { -2 \nu|k|^m q' \frac{n}{q} \pa_y \mathring{\omega}_{m,n;k} +\nu|k|^m (q')^2 \frac{n(n+1)}{q^2} \mathring{\omega}_{m,n;k} -\nu|k|^m q'' \frac{n}{q} \mathring{ \omega}_{m,n;k}} \\  
 =: & \sum_{i = 1}^3 \mathbf{C}_{q,k,i}^{(m,n)}.\label{Cnq}
\end{align}\end{subequations}
\siming{The sign for $\mathbf{C}^{(m,n)}_{q, k}$ is flipped. And I think the second term does not have `$2$'. Double check the $n\leq 2$ case. 
}
\end{lemma}

In order to prove this lemma, we need to develop some general result about commutators which will be useful also in future sections. 
\begin{lemma}\label{lem:commutator_AB}
Let $\mathcal{A}$ and $\mathcal{B}$ be two linear operators. Define the operation 
\begin{align}
\mathrm{ad}_\mathcal{A}^n (\mathcal{B}):=\underbrace{[\mathcal{A},[\mathcal{A},[\mathcal{A},...[\mathcal{A}}_{n},\mathcal{B}]]...],\quad  n\in\{1,2, 3, ...\}.
\end{align} 
Then
\begin{align} \label{comm_rel}
[\mathcal{A}^n, \mathcal{B}]
=\sum_{\ell=0}^{n-1}\binom{n}{\ell} \mathrm{ad}_\mathcal{A}^{n-\ell} (\mathcal{B})\mathcal{A}^\ell. 
\end{align}
\end{lemma}
\begin{proof} We prove the relation through induction. First of all, we observe that the relation \eqref{comm_rel} holds for $n=1$ through definition. Next we assume that the relation \eqref{comm_rel} holds for $n-1\geq 1$, and show that the relation \eqref{comm_rel} holds for $n$. We accomplish this task by expressing the commutator $[\mathcal{A}^n, \mathcal{B}]$ as follows:
\begin{align*}
[\mathcal{A}^n, \mathcal{B}]=&\mathcal{A}^n\mathcal{B}-\mathcal{A}\mathcal{B}\mathcal{A}^{n-1}+\mathcal{A}\mathcal{B}\mathcal{A}^{n-1}-\mathcal{B}\mathcal{A}^n=\mathcal{A}[\mathcal{A}^{n-1},\mathcal{B}]+[\mathcal{A},\mathcal{B}]\mathcal{A}^{n-1}\\
=&\mathcal{A}\sum_{\ell=0}^{n-2}\binom{n-1}{\ell}\mathrm{ad}_\mathcal{A}^{n-1-\ell}(\mathcal{B})\mathcal{A}^\ell+[\mathcal{A},\mathcal{B}]\mathcal{A}^{n-1}\\
=&\sum_{\ell=0}^{n-2}\binom{n-1}{\ell}\mathrm{ad}_\mathcal{A}^{n-\ell}(\mathcal{B})\mathcal{A}^\ell+\sum_{\ell=0}^{n-2}\binom{n-1}{\ell}\mathrm{ad}_\mathcal{A}^{n-1-\ell}(\mathcal{B}) \mathcal{A}^{\ell+1}+[\mathcal{A},\mathcal{B}]\mathcal{A}^{n-1}\\
=&\sum_{\ell=0}^{n-2}\binom{n-1}{\ell}\mathrm{ad}_\mathcal{A}^{n-\ell}(\mathcal{B})\mathcal{A}^\ell+\sum_{\ell=1}^{n-1}\binom{n-1}{\ell-1}\mathrm{ad}_\mathcal{A}^{n-\ell}(\mathcal{B}) \mathcal{A}^{\ell}+[\mathcal{A},\mathcal{B}]\mathcal{A}^{n-1}.
\end{align*}Here we have used the induction hypothesis to rewrite $[\mathcal{A}^{n-1},\mathcal{B}]$. Now we note that 
\begin{align*}
\binom{n-1}{\ell}+\binom{n-1}{\ell-1}=\binom{n}{\ell}.
\end{align*} Application of this relation yields that the following relation holds
\begin{align*}
[\mathcal{A}^n, \mathcal{B}]=&\sum_{\ell=0}^{n-1}\binom{n}{\ell}\mathrm{ad}_\mathcal{A}^{n-\ell}(\mathcal{B})\mathcal{A}^\ell+\lf(-\binom{n}{n-1}+\binom{n-1}{n-2}\rg)\mathrm{ad}_\mathcal{A}(\mathcal{B}) \mathcal{A}^{n-1}+[\mathcal{A},\mathcal{B}]\mathcal{A}^{n-1}\\
=&\sum_{\ell=0}^{n-1}\binom{n}{\ell}\mathrm{ad}_\mathcal{A}^{n-\ell}(\mathcal{B})\mathcal{A}^\ell.
\end{align*} 

\ifx
{\bf Proof \# 2: }
\footnote{The second author would like to thank Yiyue Zhang for suggesting this version of the proof. }
Define the left multiplication $\mathcal{L}_A$ and $\mathcal{R}_A$. Moreover, we define the Bracket operator $\mathcal{F}_A(B)=[A,B]=(\mathcal{L}_A -\mathcal{R}_A)( B)$. Now $\mathcal{F}^n_A(B)=\underbrace{[A,[A,[A,...[A}_{n},B]]...]$. Now we observe that 
\begin{align}
\mathcal{L}_A\mathcal{R}_A B=ABA=\mathcal{R}_A\mathcal{L}_A B.
\end{align}
Hence we have the important relation 
\begin{align}
\mathcal{F}_A \mathcal{R}_A=\mathcal{R}_A\mathcal{F}_A.
\end{align}
Now we have that \begin{align}
[A^n, B]=&A^n B-BA^n=(\mathcal{F}_A+\mathcal{R}_A)^n B-\mathcal{R}_A^n B=\sum_{\ell=0}^{n}\left(\begin{array}{rr}n\\ \ell\end{array}\right) \mathcal{F}_A^{n-\ell}\mathcal{R}_A^{\ell } B-\mathcal{R}_A^n B=\sum_{\ell=0}^{n-1}\left(\begin{array}{rr}n\\ \ell\end{array}\right) \mathcal{R}_A^{\ell }\mathcal{F}_A^{n-\ell} B\\
=&\sum_{\ell=0}^{n-1}\left(\begin{array}{rr}n \\ \ell\end{array}\right)\underbrace{[A,[A,[A,...[A,}_{n-\ell}B]]....] A^\ell
\end{align}
\fi

\end{proof}

\begin{lemma} \label{lem:com:BA}The following commutator relations hold:
\begin{subequations}\label{commutators}
\begin{align}[\pa_{y},\Gamma_k^n ]=&-ikt\sum_{\ell=0}^{n-1}\binom{n}{\ell}\pav^{n-\ell}(v_y-1)\Gamma_k^{\ell}+\sum_{\ell=1}^{n}\binom{n}{\ell}\pav^{n-\ell+1}(v_y-1)\Gamma_k^{\ell}\label{cm_py_G_n}\\
[\pa_{yy},\Gamma_k^n ]=&-\sum_{\ell=0}^{n-1}\binom{n}{\ell} \left(2\pav^{n-\ell-1}v''\ \pav^2 +\pav^{n-\ell}v'' \ \pav\right)\Gamma_k^{\ell};\label{cm_pyy_G_n}\\
[\pa_y,q^{n}]f=&q'\lf(\frac{n}{q}q^n f\rg) ;\label{cm_py_qn}\\
[\pa_{yy},q^{n}]f=&-(q')^2\frac{n^2+n}{n^2}\lf(\frac{n^2}{q^2} q^n f\rg)+2q'\lf(\frac{n }{q}\pa_y(q^n f)\rg)+q''\lf(\frac{n}{q}q^n f\rg);\label{cm_pyy_qn}\\
\lf[\pav,q^n\rg] f=&\frac{q'}{v_y}\lf(\frac{n}{q}q^n f\rg);\label{cm_pv_qn}\\ \n
\lf[ \pav^2,q^n\rg]f=&2\frac{q'}{(v_y)^2}\lf(\frac{n}{q}\pa_y\lf(q^n f\rg)\rg)+\frac{q''}{(v_y)^2}\lf(\frac{n}{q}q^n f\rg)-\frac{q'}{(v_y)^2}\lf(\vyi\pa_y v_y\rg)\lf(\frac{n}{q}q^n f\rg)\\ 
&-\frac{(q')^2}{(v_y)^2}\frac{n^2+n}{n^2}\lf(\frac{n^2}{q^2}q^n f\rg);\label{cm_pvv_qn}\\
\lf[\pav^2,q\rg]f=&2\frac{q'}{v_y^2}  \pa_y  f  +\frac{q''}{v_y^2} f  -\frac{q'}{ v_y ^3}v _{yy}   f .\label{cm_pvv_q}
\end{align} 
\end{subequations}
Here we recall the notation $\pav=v_y^{-1}\pa_y.$
\ifx
\myb{Even though the $q'$ is not used here, we can still use the almost orthogonality between $\omega_k$ and $q', q''$, which is a natural consequence of the weighted estimate. We did not use this orthogonality in the previous work.  }
\fi
\end{lemma}
\begin{proof}
We recall Lemma \ref{lem:commutator_AB} and compute the $[\Gamma_k,\pa_y]$
\begin{align*}
[\Gamma_k,\pa_y]f=[\pv,\pa_y]f=\pa_y(\frac{1}{v_y}) v_y (\vyi\pa_y f)=-\frac{1}{v_y}\pa_y (v_y-1) \times \vyi\pa_y f.
\end{align*}
Hence the commutator is a product of the $\pav$-derivative of the $v_y-1, f $. This will help us in the future estimates.  Now all the higher commutators are simple:
\begin{align*}
\mathrm{ad}^n_{\Gamma_k} (\pa_y)f=-(\vyi\pa_y)^n(v_y-1)\times \pv f. 
\end{align*}
Now the commutator is
\begin{align*}
[\pa_y, \Gamma_k^n]=&-\sum_{\ell=0}^{n-1}\binom{n}{\ell} \mathrm{ad}_{\Gamma_k}^{n-\ell}(\pa_y)\Gamma_k^{\ell}=\sum_{\ell=0}^{n-1}\binom{n}{\ell}(\vyi\pa_y)^{n-\ell}(v_y-1)\times \pv\Gamma_k^{\ell}\\
=&-ikt\sum_{\ell=0}^{n-1}\binom{n}{\ell}(\vyi\pa_y)^{n-\ell}(v_y-1)\Gamma_k^{\ell}+\sum_{\ell=1}^{n}\binom{n}{\ell}(\vyi\pa_y)^{n-\ell+1}(v_y-1)\Gamma_k^{\ell}.
\end{align*}
Next we compute the $[\Gamma_k,\pa_{yy}]$ using the $\Gamma_0$-derivative of $v''=\pa_{yy}v=v_y\frac{1}{v_y}\pa_y v_y$:
\begin{align*}
[\pa_{yy},\Gamma_k]f=[\pa_{yy},\pv]f=\pa_{yy} (\vyi\pa_y f)-\vyi\pa_{yyy} f=-2v''(\vyi\pa_y)^2f-(\vyi\pa_y)v''\ (\vyi\pa_y)f.
\end{align*}
\begin{align*}&
\mathrm{ad}^n_{\Gamma_k} (\pa_{yy})f= 2(\vyi\pa_y)^{n-1}v''\cdot (\vyi\pa_y)^2f+(\vyi\pa_y)^nv''\cdot (\pv)f.
\end{align*}
\begin{align*}
[\pa_{yy},\Gamma_k^n ]=&-\sum_{\ell=0}^{n-1}\binom{n}{\ell} \mathrm{ad}_{\Gamma_k}^{n-\ell}(\pa_{yy})\Gamma_k^{\ell}\\
= &-\sum_{\ell=0}^{n-1}\binom{n}{\ell} \left(2(\pv)^{n-\ell-1}v''\ (\pv)^2 +(\pv)^{n-\ell}v'' \ \pv\right)\Gamma_k^{\ell}\\
=&-\sum_{\ell=0}^{n-1}\binom{n}{\ell} \left(2(\pv)^{n-\ell-1}v''\ (\Gamma-ikt)^2 +(\pv)^{n-\ell}v'' \ (\Gamma-ikt)\right)\Gamma_k^{\ell}.
\end{align*}
The proof of the last two commutator relation is more direct. First of all,
\begin{align*}
[\pa_{y},q^n] f=\pa_y(q^n f)-q^n\pa_y f=nq^{n-1} q' f=q'\frac{n}{q}q^n f. 
\end{align*}
Next we consider the commutator
\begin{align*}
[\pa_{yy}, q^n ]f=&\pa_{yy}(q^n f)-q^n\pa_{yy}f=2q'\left(\frac{n}{q} q^{n}\pa_y f\right)+(q')^2\left(\frac{n(n-1)}{q^2}q^{n}f\right)+ q'' \left(\frac{n}{q}q^nf\right)\\
=&2q' \left(\frac{n}{q}[q^n,\pa_y]f\right)+2q'\left(\frac{n}{q}\pa_y(q^n f)\right)+(q')^2\left(\frac{n(n-1)}{q^2}q^n f\right)+q''\left(\frac{n}{q}q^n f\right)\\
=&-2(q')^2\frac{n^2}{q^2} q^nf+(q')^2\frac{n^2-n}{q^2}q^n f +2q'\left(\frac{n}{q}\pa_y(q^n f)\right)+q''\left(\frac{n}{q}q^n f\right)\\
=&-(q')^2\frac{n^2+n}{q^2} q^n f+2q'\lf(\frac{n }{q}\pa_y(q^n f)\rg)+q''\lf(\frac{n}{q}q^n f\rg). 
\end{align*}
The proof of \eqref{cm_pv_qn} is direct. Next we focus on the expression \eqref{cm_pvv_qn},
\begin{align*}
\lf[\lf(\vyi \pa_y\rg)^2, q^n \rg]f=&\lf(\vyi\pa_y\rg)^2(q^n f)-q^n \lf(\vyi\pa_y \rg)^2 f\\ =&2\frac{q'}{ v_y }\frac{n}{q} q^n\frac{1}{v_y}\pa_y f+\frac{1}{v_y}\pa_y\lf(\frac{q'}{v_y}\frac{n}{q}q^n\rg)f\\
=& 2\frac{q'}{(v_y)^2}\lf(\frac{n}{q}q^n\pa_y f\rg)+\frac{1}{v_y}\pa_y\lf(\frac{q'}{v_y}\rg)\lf(\frac{n}{q}q^n f\rg)+\frac{(q')^2}{(v_y)^2}\frac{n(n-1)}{q^2}q^n f\\
=&2\frac{q'}{(v_y)^2}\lf(\frac{n}{q}\pa_y\lf(q^n f\rg)\rg)-2\frac{(q')^2}{(v_y)^2}\lf(\frac{n^2}{q^2}q^n f\rg)+\frac{q''}{(v_y)^2}\lf(\frac{n}{q}q^n f\rg)\\
&-\frac{q'}{(v_y)^2}\lf(\vyi\pa_y v_y\rg)\lf(\frac{n}{q}q^n f\rg)+\frac{(q')^2}{(v_y)^2}\lf(\frac{n(n-1)}{q^2}q^n f\rg)\\
=&2\frac{q'}{(v_y)^2}\lf(\frac{n}{q}\pa_y\lf(q^n f\rg)\rg)+\frac{q''}{(v_y)^2}\lf(\frac{n}{q}q^n f\rg)\\
&-\frac{q'}{(v_y)^2}\lf(\vyi\pa_y v_y\rg)\lf(\frac{n}{q}q^n f\rg)-\frac{(q')^2}{(v_y)^2}\frac{n^2+n}{n^2}\lf(\frac{n^2}{q^2}q^n f\rg).
\end{align*}
\end{proof}

With these commutator identities in hand, we can now prove Lemma \ref{lem:eq:comp:a}.
\begin{proof}[Proof of Lemma \ref{lem:eq:comp:a}] We organize the proof in two steps. In the first step, we commute the vector field $\Gamma_k^n$ with the equation, and then we commute the boundary weight $q(y)^n$ with the resulting equation. 

\vspace{2 mm}

\noindent
\textbf{Step \# 1: Commute $\Gamma_k^n$-vector field.} 
We apply the vector field $\Gamma_k^n$ on the equation \eqref{M1a}, and rewrite the equation in the form:
\begin{align}\n
\pa_t \Gamma_k^n\omega_k + ik(y + U^x_0(t, y)) \Gamma_k^n \omega_k - \nu \Delta_k \Gamma_k^n \omega_k  
  =  & \Gamma_k^n f_k \\ \label{EQ:pat_G_om1}
  &-[\Gamma_k^n, \pa_t+ ik(y + U^x_0(t, y))-\nu\pa_{yy} ]\omega_k .
\end{align} 
The goal is to rewrite the last commutator term. First of all, we recall the equation 
\eqref{eq:v} and compute the commutator as follows
\begin{align*}
\mathrm{ad}_{\Gamma_k}^m (\pa_t) f 
=&\mathrm{ad}_{\Gamma_k}^{m-1}\lf(\lf[v_y^{-1}\pa_y+ikt, \frac{1}{(v_y)^2}( v_{ty})\pa_y-ik\rg]\rg) f\\
=&\mathrm{ad}_{\Gamma_k}^{m-1}\lf(\lf[v_y^{-1}\pa_y+ikt, \frac{1}{(v_y)^2}(\nu v_{yyy}+G_y )\pa_y-ik\rg]\rg) f\\
=&\nu\pav^m(v_{yy})\ \pav f+\pav^m G\ \pav f-ik\mathbbm{1}_{m=1}.
\end{align*}
Next we invoke the relation 
\eqref{comm_rel} to derive the commutator expression 
\begin{align} \n
[\Gamma_k^n,\pa_t]\omega_k=&\sum_{\ell=0}^n\binom{ n}{\ell}\lf(\nu\pav^{n-\ell}(v_{yy})\ \pav\Gamma_k^\ell \omega_k+\pav^{n-\ell} G\ \pav\Gamma_k^\ell \omega_k\rg)\\
&-ik n\Gamma_k^{n-1}\omega_k. \label{Ga_n,pa_t}
\end{align}
Next we compute the commutator $[\Gamma_k^n,  ik(y+U_0^x)]$. To this end, we apply the relation \eqref{comm_rel}, the definition of $G$ \eqref{defn:g:cap} and  calculate that 
\begin{align*}
[\Gamma_k^n,& ik(y+U_0^x)]\omega_k\\
=&\sum_{\ell=0}^{n-1}\binom {n}{ \ell} \mathrm{ad}_{\Gamma_k}^{n-\ell}(ik(y+U_0^x))\Gamma_k^\ell\omega_k
=\sum_{\ell=0}^{n-1}\binom{n}{\ell} \mathrm{ad}_{\Gamma_k}^{n-\ell-1}\lf(\lf[v_y^{-1}\pa_y+ikt,ik(y+U_0^x)\rg]\rg)\Gamma_k^\ell\omega_k\\
=&\sum_{\ell=0}^{n-1}\binom {n}{\ell} \mathrm{ad}_{\Gamma_k}^{n-\ell-1}\lf(\frac{ikt}{v_y}\lf(\pa_y G+\frac{v_y}{t}\rg)\rg)\Gamma_k^\ell\omega_k
=\sum_{\ell=0}^{n-1}\binom {n}{\ell} \lf(\pv\rg)^{n-\ell}G\  ikt \Gamma_k^\ell \omega_k+nik \Gamma_k^{n-1}\omega_k. 
\end{align*}
Combining these two commutator representations and the commutator expression \eqref{cm_py_G_n} and the relation \eqref{cm_pyy_G_n}, we have that
\begin{align*}
[\Gamma_k^n, \pa_t+ik(y+U_0^x)-\nu\pa_{yy}]\omega_k=&\sum_{\ell=0}^{n-1}\binom{n}{ \ell} \lf(\pv\rg)^{n-\ell}G\ \Gamma^{\ell+1}\omega_k\\
&-\nu\sum_{\ell=0}^{n-1}\binom{n}{ \ell}  \lf(2\lf(\pv\rg)^{n-\ell-1}v_{yy}\rg)\ \lf(\pv\rg)^2\Gamma_k^\ell\omega_k.
\end{align*}
When combined with \eqref{EQ:pat_G_om1}, this yields the following
\begin{align}\n
\pa_t \Gamma_k^n& \omega_k + ik(y + U^x_0(t, y)) \Gamma_k^n \omega_k - \nu \Delta_k \Gamma_k^n \omega_k  \\
  =& \Gamma_k^n f_k 
  -  \sum_{\ell = 1}^n \binom{n}{\ell} \pav^\ell G\ \   \Gamma_k^{n-\ell+1} \omega_k   \siming{+ } \nu \sum_{\ell = 1}^n \binom{n}{\ell}  \pav^\ell v_y^2 \ \ \pav^2 \Gamma_k^{n-\ell} \omega_k.\label{EQ:G_om}
\end{align} 

\noindent
\textbf{Step \# 2: Derivation of \eqref{aff1}.}
Next, we insert our co-normal weights, $q(y)^n$. The computation is straight forward. We observe that from the equation \eqref{EQ:G_om}, one derives the following relation
\begin{align*}
\pa_t& q^n\Gamma_k^n\omega_k+ik(y+U_0^x)q^n\Gamma_k^n\omega_k -\nu\de_k(q^n\Gamma^n_k\omega_k)\\
=&q^n\Gamma_k^n f_k 
  -  q^n\sum_{\ell = 1}^n \binom{n}{\ell} \pav^\ell G\ \   \Gamma_k^{n-\ell+1} \omega_k     {+ } \nu q^n\sum_{\ell = 1}^n \binom{n}{\ell}  \pav^\ell v_y^2 \ \ \pav^2 \Gamma_k^{n-\ell} \omega_k+\nu [q^n,\pa_{yy}]\Gamma_k^n\omega_k.
\end{align*}
Multiplying this equation by $|k|^m$ and invoking the commutator relation \eqref{cm_pyy_qn} yields the result \eqref{aff1}.
\end{proof}
\ifx{
\myb{
As a final step in our derivation, we need to take the appropriate linear combination of \eqref{aff1} as indicated by \eqref{rep1}. This produces commutators that are linear combinations of \eqref{Cnt}, \eqref{Cnv}, \eqref{Cnq}.
\begin{lemma} \label{lemma:eq:t}
\begin{align} \n
&\pa_t \mathring{\omega}_{m,n;k} + ik(y + U^x_0(t, y)) \mathring{\omega}_{m,n;k} - \nu \Delta_k \mathring{\omega}_{m,n;k} \\ \label{aff2}
& \qquad = -\bold{T}^{(m,n)}_{main,k} - \bold{T}^{(m,n)}_{b,k} - \bold{C}^{(m,n)}_{trans, k} - \bold{C}^{(m,n)}_{visc, k}  - \bold{C}^{(m,n)}_{q, k}  - \bold{C}^{(m,n)}_{b,k},
\end{align}
where the commutators appearing above are defined as follows:
\begin{subequations}
\begin{align}
\bold{C}^{(m,n)}_{trans, k} := & q^n \mathcal{C}^{(m,n)}_{trans,k} + \sum_{m = 1}^{n-2} b^{(n)}_m(y) q^{n-m} \mathcal{C}^{(n-m)}_{trans,k} + |q'|^{n-1} q\mathcal{C}^{(1)}_{trans, k}  \\
\bold{C}^{(m,n)}_{visc, k} := &q^n \mathcal{C}^{(m,n)}_{visc,k} + \sum_{m = 1}^{n-2} b^{(n)}_m(y) q^{n-m} \mathcal{C}^{(n-m)}_{visc,k} + |q'|^{n-1} q\mathcal{C}^{(1)}_{visc, k} \\ \label{boldCnqk}
 \bold{C}^{(m,n)}_{q, k} := & \mathcal{C}^{(m,n)}_{q,k} + \sum_{m = 1}^{n-2} b^{(n)}_m(y) \mathcal{C}^{(n-m)}_{q,k} + |q'|^{n-1} \mathcal{C}^{(1)}_{q, k}  \\ 
  \bold{C}^{(m,n)}_{b,k} := &\nu \sum_{m = 1}^{n} \pa_y^2 b^n_m(y) \widetilde{\omega}_{m,n;k} + 2 \nu \sum_{m = 1}^{n} \pa_y b^n_m(y) \pa_y \widetilde{\omega}_{m,n;k} ,
\end{align}
\end{subequations}
and the nonlinear terms appearing above are defined as follows:
\begin{subequations}
\begin{align} \label{Tmain}
\bold{T}^{(m,n)}_{main,k} :=  &q^n \Gamma_k^n(\nabla^\perp \psi_{\neq 0} \cdot \nabla  \omega)_k, \\
\bold{T}^{(m,n)}_{b,k} := & \sum_{m=1}^{n-1} b^n_m(y) q^{n-m} \Gamma_k^{n-m} (\nabla^\perp \psi_{\neq 0} \cdot \nabla  \omega)_k.
\end{align}
\end{subequations}
\end{lemma}}
\begin{proof} \sameer{The proof comes from using the linear combination \eqref{rep1} and the equations \eqref{aff1}.}
\end{proof}
\fi

We will now need to differentiate \eqref{aff1} once in $\pa_y$ in order to derive the appropriate system for $\pa_y \mathring{\omega}_{m,n;k}$. Indeed, doing so results in 
\begin{align} \n
\pa_t \pa_y&  \mathring{\omega}_{m,n;k} + ik(y + U^x_0(t, y)) \pa_y  \mathring{\omega}_{m,n;k} - \nu \Delta_k \pa_y  \mathring{\omega}_{m,n;k} \\
 = &  \pa_y \mathbf{T}^{(m,n)}_{k} + \pa_y \mathbf{C}^{(m,n)}_{\text{trans}, k} +\pa_y \mathbf{C}^{(m,n)}_{\text{visc}, k}  +\pa_y \mathbf{C}^{(m,n)}_{q, k} - ik (1 + \pa_y U^x_0)  \mathring{\omega}_{m,n;k}. \label{hghg1}
\end{align}
To estimate the commutator term $\mathbf{C}_{q}^{(m,n)}$, we further decompose it in the following lemma. 
\begin{lemma} The spacial derivative of the $\mathbf{C}^{(m,n)}_{q,k}$ term reads as follows,
\begin{align} 
\pa_y \bold{C}^{(m,n)}_{q, k}= &\nu \Upsilon_n^{(1)}  \frac{n}{q} \pa_y^2 {\omega}_{m,n;k} +\nu \Upsilon_n^{(2)} \frac{n^2}{q^2} \pa_y {\omega}_{m,n;k} +\nu \Upsilon_n^{(3)} \frac{n^3}{q^3} {\omega}_{m,n;k}, \label{IDL12}
\end{align}
where the coefficients are defined as follows:
\begin{subequations}
\begin{align} \label{Ups:1}
\Upsilon_n^{(1)}  := & \myr{-2 q'},\\ \label{Ups:2}
\Upsilon_n^{(2)}  := &\myr{\frac{n+3}{n}(q')^2-\frac{3}{n}q'' q},\\ 
\label{Ups:3}\Upsilon_n^{(3)}  := &  \frac{2n+3}{n^2}q'' q' q -2 \frac{n+1}{n^2}(q')^3-  \frac{1}{n^2}q'''q^2.   
\end{align}
\end{subequations}
\end{lemma}
\begin{proof} The proof is a straightforward computation. 
\ifx From \eqref{boldCnqk}, we have the linear combination
\begin{align*}
\pa_y  \bold{C}^{(m,n)}_{q, k} = \pa_y \mathcal{C}^{(m,n)}_{q,k} .
\end{align*}
Thus, the result will follow upon obtaining the identity \eqref{Cnq} and the following identity for $\pa_y \mathcal{C}^{(m,n)}_{q,k}$:
\begin{align}
\pa_y \mathcal{C}^{(m,n)}_{q,k} = \nu \Upsilon_n^{(1)}  \frac{n}{q} \pa_y^2 \widetilde{\omega}_{m,n;k} +\nu \Upsilon_n^{(2)} \frac{n^2}{q^2} \pa_y \widetilde{\omega}_{m,n;k} +\nu \Upsilon_n^{(3)} \frac{n^3}{q^3} \widetilde{\omega}_{m,n;k}.
\end{align} 
A direct computation using \eqref{Cnq} results in the identity
\begin{align} \n
-\pa_y \mathcal{C}^{(m,n)}_{q, k} := & - 3 \nu n (n-1) q(y)^{n-2} |q'(y)|^2 \Gamma_k^n \pa_y |k|^m\omega_k - 3 \nu n q(y)^{n-1} q''(y) \Gamma_k^n \pa_y|k|^m \omega_k \\ \label{exp:L1}
& - 3 \nu n (n-1) q(y)^{n-2} q'(y) q''(y) \Gamma_k^n|k|^m \omega_k - \nu n (n-1) (n-2) q(y)^{n-3} q'(y)^3 \Gamma_k^n |k|^m\omega_k \\ \n
& - \nu n q(y)^{n-1} q'''(y) \Gamma_k^n |k|^m\omega_k - 2 \nu n q(y)^{n-1} q'(y) \Gamma_k^n \pa_y^2|k|^m \omega_k = \sum_{i = 1}^6 \mathcal{L}^{(i)}_1.
\end{align}
We proceed term by term through \eqref{exp:L1}. First, 
\begin{align*}
\mathcal{L}_1^{(1)} := & - 3 \nu n(n-1) |q'(y)|^2 q^{n-2} \pa_y \Gamma_k^n \omega_k \\
= & - 3 \nu n(n-1) |q'(y)|^2 \frac{1}{q^2} \pa_y \mathring{\omega}_{m,n;k} + 3 \nu n^2 (n-1) |q'(y)|^3 \frac{1}{q^3} \mathring{\omega}_{m,n;k}.  
\end{align*}
Next,
\begin{align*}
\mathcal{L}_1^{(2)} := &- 3 \nu n q(y)^{n-1} q''(y) \Gamma_k^n \pa_y \omega_k = - 3 \nu q''(y) \frac{n}{q} \pa_y \mathring{\omega}_{m,n;k} + 3 \nu q''(y) \frac{n^2}{q^2} \mathring{\omega}_{m,n;k}, \\
\mathcal{L}_1^{(3)} := &- 3 \nu q''(y) q'(y) \frac{n(n-1)}{q(y)^2} \mathring{\omega}_{m,n;k}, \\
\mathcal{L}_1^{(4)} := & - \nu|q'(y)|^3 \frac{ n (n-1)(n-2) }{q^3} \mathring{\omega}_{m,n;k}, \\
\mathcal{L}_1^{(5)} := & - \nu q'''(y) \frac{n}{q} \mathring{\omega}_{m,n;k}.
\end{align*}
Next, we have the following long identity:
\begin{align} \n
2 \nu n q'(y) \frac{1}{q} \pa_y^2 \{ q(y)^n \Gamma_k^n \omega_k \} = & 2 \nu n q'(y) \frac{1}{q} n(n-1) q(y)^{n-2} |q'(y)|^2 \Gamma_k^n \omega_k + 2 \nu n q'(y) \frac{1}{q} n q^{n-1} q''(y) \Gamma_k^n \omega_k \\ \n
& + 4 \nu n q'(y) \frac{1}{q} n q(y)^{n-1} q'(y) \pa_y \Gamma_k^n \omega_k + 2 \nu n q' \frac{1}{q} q^n \pa_y^2 \Gamma_k^n \omega_k \\ \n
= & 2 \nu  (q')^3  \frac{ n^2 (n-1)}{q(y)^3} \mathring{\omega}_{m,n;k}+ 2 \nu n q'(y) \frac{1}{q} n q^{n-1} q''(y) \Gamma_k^n \omega_k + 4 \nu n^2 |q'|^2 \frac{1}{q^2} \pa_y \mathring{\omega}_{m,n;k} \\ \n
& - 4 \nu n^3 (q')^3 \frac{1}{q^3} \mathring{\omega}_{m,n;k}+ 2 \nu n q' \frac{1}{q} q^n \pa_y^2 \Gamma_k^n \omega_k.
\end{align}
Rearranging this, we obtain 
\begin{align*}
\mathcal{L}_1^{(6)} := &2 \nu n q(y)^{n-1} q'(y) \Gamma_k^n \pa_y^2 \omega_k \\
= & 2 \nu n q'(y) \frac{1}{q} \pa_y^2\mathring{\omega}_{m,n;k} -  4 \nu n^2 |q'|^2 \frac{1}{q^2} \pa_y \mathring{\omega}_{m,n;k} +  4 \nu n^3 (q')^3 \frac{1}{q^3}\mathring{\omega}_{m,n;k} -  2 \nu  (q')^3  \frac{ n^2 (n-1)}{q(y)^3} \mathring{\omega}_{m,n;k} \\
& - 2 \nu  q'(y) \frac{n^2}{q^2} q''(y) \mathring{\omega}_{m,n;k}.
\end{align*}
Summing together the above identities gives the desired identity, \eqref{IDL12}. Specifically, there is only one contribution to $\Upsilon^{(1)}_n$. Next, we have 
\begin{align*}
\Upsilon^{(2)}_n = - 3 \frac{n-1}{n} |q'|^2 - 3 q'' \frac{q}{n} - 4|q'|^2 = - (7 - \frac3n) |q'|^2 - 3 q'' \frac{q}{n}.
\end{align*}
Third, we have by collecting terms and simplifying
\begin{align*}
\Upsilon^{(3)}_n = &3 \frac{n-1}{n} |q'|^3 + 3 q'' \frac{q}{n} - 3 q'' q' \frac{q}{n} \frac{n-1}{n} - |q'|^3 \frac{(n-1)(n-2)}{n^2} - q''' \frac{q^2}{n^2} \\
& + 4 |q'|^2 - 2 |q'|^3 \frac{n-1}{n}  - 2 q' q'' \frac{q}{n} \\
= & (4 + \frac2n - \frac{2}{n^2}) |q'|^3 + (3q'' - (5 - \frac3n)q'' q') \frac{q}{n} - q''' (\frac{q}{n})^2.
\end{align*}
This concludes the proof of the lemma. 

\fi
\end{proof}

\subsection{The $\alpha/\mu/ \gamma$ Estimates}\label{sec:amg}
In this section we obtain energy bounds for frozen $(m, n;k)$. We start with the $L^2$-component of the functional. Throughout the lemmas of this section, $\mathcal{I}_{m,n;k}^{(\cdot),(\cdot)}$ will be local variables: for the proof of each lemma, we use these notations to represent various terms appeared in the decomposition. However, after the proof is accomplished, these variables will be refreshed to represent other things.

\begin{lemma}[The $\gamma$-term Estimates] The time evolution of the $\gamma$-component of the functional $\mathcal{E}_{m,n;k}$ \eqref{defn:F:1} can be estimated  as follows:
\begin{align}  \n
\frac{d}{dt}&\lf(\bold{a}_{m,n}^2\| \mathring{\omega}_{m,n;k} e^W  \chi_{m+n} \|_{L^2}^2\rg)  +{\mathcal{D}^{(\gamma)}_{m,n;k}+2\mathcal{CK}_{m,n;k}^{(\gamma;\varphi)}+\mathcal{CK}_{m,n;k}^{(\gamma;W)}}\\
\n &\myr{\lesssim \frac{\lambda^{2s}}{(m+n)^{2\sigma_{\ast}}}  \mathcal{D}^{(\gamma)}_{m,n-1;k}\mathbbm{1}_{n\geq 1}+ \frac{\lambda^{2s}}{m^{2\sigma_{\ast}}} \mathcal{D}^{(\gamma)}_{m-1,0;k}\mathbbm{1}_{n=0,m+n\geq 1} }\\
&\quad +\bold{a}_{m,n}^2 \mathrm{Re}\lf\langle \mathring{\omega}_{m,n;k}, \lf(\bold{T}^{(m,n)}_{k} +\bold{C}^{(m,n)}_{\mathrm{trans};k} + \bold{C}^{(m,n)}_{\mathrm{visc};k}  +\bold{C}^{(m,n)}_{q,k} \rg)e^{2W} \chi_{m+n}^2 \rg\rangle . \label{gam:est:le1}
\end{align}
\siming{Double check the $m+n=0$ case, I felt that the linear term can be absorbed by the $\mathcal{D}$ and $\mathcal{CK}$. So we don't need to worry? }The dissipation term $\mathcal{D}^{(\gamma)}$ is defined in \eqref{defn:D:1} and the $\mathcal{CK}^{(\gamma)}$ terms are defined in \eqref{CK_H}. 
Moreover, the parameter $\sigma_\ast$ is defined in \eqref{s:prime} and the terms $\bold{T}^{(m,n)}_{k},\ \bold{C}^{(m,n)}_{\mathrm{trans};k}, \ \bold{C}^{(m,n)}_{\mathrm{visc};k} , \ \bold{C}^{(m,n)}_{q,k}$ are defined in \eqref{aff1}. 
\end{lemma}
\begin{proof} 
 We freeze $(m,n;k)$ and assume that $m+n\geq 1$. Thanks to the definition of $\mathcal{CK}^{(\gamma)}$ \eqref{CK_phi_B_1}, \eqref{CK_W_ga_1}, we observe that 
\begin{align}
\frac{d}{dt}\lf(\bold{a}_{m,n}^2\| \mathring{\omega}_{m,n;k} e^W  \chi_{m+n} \|_{L^2}^2\rg)=\bold{a}_{m,n}^2\frac{d}{dt}\|\mathring{\omega}_{m,n;k}e^W\chi_{m+n}\|_{L^2}^2-2\mathcal{CK}^{(\gamma;\varphi)}_{m,n;k}-2\mathcal{CK}^{(\gamma;B)}_{m,n;k}.\label{dt_ga}
\end{align} Recalling the weight \eqref{Bweight} and the equation \eqref{aff1}, we compute the identity (letting $\overline{U}^x_0(t, y) := y + U^x_0(t, y)$), 
\begin{align}\n
&\hspace{-1cm}\bold{a}_{m,n}^2 \frac{ 1}{2} \frac{d}{dt}\|  \mathring{\omega}_{m,n;k} e^W  \chi_{m+n} \|_{L^2}^2\\ \n
=&  \bold{a}_{m,n}^2\mathrm{Re}\lf\langle \mathring{\omega}_{m,n;k}, \pa_t \mathring{\omega}_{m,n;k} e^{2W}  \chi_{m+n}^2  \rg\rangle  + \bold{a}_{m,n}^2 \mathrm{Re}\lf\langle \mathring{\omega}_{m,n;k},  \mathring{\omega}_{m,n;k}  \pa_t( e^{2W} )\chi_{m+n}^2 \rg \rangle\\ \n
= & - \bold{a}_{m,n}^2\mathrm{Re} \lf\langle \mathring{\omega}_{m,n;k}, i k \overline{U}^x_0 \mathring{\omega}_{m,n;k} e^{2W} \chi_{m+n}^2\rg \rangle +\bold{a}_{m,n}^2 \mathrm{Re}\lf\langle \mathring{\omega}_{m,n;k}, \nu \Delta_k \mathring{\omega}_{m,n;k} e^{2W}  \chi_{m+n}^2\rg\rangle  \\ \n
&+ \bold{a}_{m,n}^2\mathrm{Re}\lf\langle \mathring{\omega}_{m,n;k},  \mathring{\omega}_{m,n;k} \pa_t \lf( e^{2W} \rg) \chi_{m+n}^2\rg\rangle   \\ \n
& + \bold{a}_{m,n}^2 \mathrm{Re}\lf\langle \mathring{\omega}_{m,n;k}, \lf(\bold{T}^{(m,n)}_{k} +\bold{C}^{(m,n)}_{\mathrm{trans};k} + \bold{C}^{(m,n)}_{\mathrm{visc};k}  +\bold{C}^{(m,n)}_{q,k} \rg)e^{2W} \chi_{m+n}^2 \rg\rangle  \\ 
 =: & \sum_{i = 1}^{4} \mathcal{I}_{m,n;k}^{(i)}.\label{I_i}
\end{align}
First of all, we note that integration by parts together with the boundary condition $\mathring{\omega}_{m,n;k}(t,y=\pm 1)=0$ yields that $  \mathcal{I}^{(1)}_{m,n;k} = 0$. We thus turn to the remaining terms above.  

Next we estimate the terms $\mathcal{I}^{(2)}_{m,n;k} + \mathcal{I}^{(3)}_{m,n;k}$ in \eqref{I_i}. We recall the boundary condition $\mathring{\omega}_{m,n;k}(t,y=\pm 1)=0$ and integrate by parts to get 
\begin{align}\n
\mathcal{I}^{(2)}_{m,n;k}& +\mathcal{I}^{(3)}_{m,n;k}  \\ \n
= &\bold{a}_{m,n}^2\mathrm{Re}\lf\langle \mathring{\omega}_{m,n;k}, \nu \pa_y^2 \mathring{\omega}_{m,n;k} e^{2W} \chi_{m+n}^2 \rg \rangle -  \nu  \bold{a}_{m,n}^2  |k|^2 \lf\| \mathring{\omega}_{m,n;k} e^W \chi_{m+n} \rg \|_{L^2}^2 \\ \n 
&+   \bold{a}_{m,n}^2 \mathrm{Re}\lf\langle \mathring{\omega}_{m,n;k},  \mathring{\omega}_{m,n;k} \pa_t (e^{2W} )\chi_{m+n}^2\rg\rangle \\
\n = &  - \bold{a}_{m,n}^2\nu \| \nabla_k \mathring{\omega}_{m,n;k} e^W \chi_{m+n} \|_{L^2}^2-\sqrt{2}\bold{a}_{m,n}^2\lf\|\mathring{\omega}_{m,n;k}\sqrt{-\pa_t W}    e^{W} \chi_{m+n}\rg\|_{L^2}^2  \\ \n
& -\bold{a}_{m,n}^2\mathrm{Re}\lf \langle   \mathring{\omega}_{m,n;k},    {\nu} \pa_y \mathring{\omega}_{m,n;k}  \pa_y ( e^{2W} )\chi_{m+n}^2\rg\ran -   \bold{a}_{m,n}^2\nu\Re \lf\langle \mathring{\omega}_{m,n;k},\pa_y \mathring{\omega}_{m,n;k} e^{2W} \pa_y  (\chi_{m+n}^2) \rg\rangle  \\
 =:&- \bold{a}_{m,n}^2\nu \| \nabla_k \mathring{\omega}_{m,n;k} e^W \chi_{m+n} \|_{L^2}^2-\sqrt{2}\bold{a}_{m,n}^2\lf\|   \mathring{\omega}_{m,n;k}  \sqrt{-\pa_t W}    e^{W} \chi_{m+n}\rg\|_{L^2}^2+   \sum_{i = 1}^2 \mathcal{I}^{(2,i)}_{m,n;k}.\label{I_2i} 
\end{align}  The term $\mathcal{I}^{(2,1)}_{m,n;k}$ can be estimated with \eqref{W_prop} as follows
\begin{align}\n
\mathcal{I}^{(2,1)}_{m,n;k}\leq &\bold{a}_{m,n}^2\nu\|| \pa_y W |\ \mathring{\omega}_{m,n;k} e^W \chi_{m+n}\|_2\|\pa_y\mathring{\omega}_{m,n;k}e^W\chi_{m+n}\|_2\\
\leq&\frac{1}{4} \bold{a}_{m,n}^2 \nu \|\pa_y\mathring{\omega}_{m,n;k}e^W\chi_{m+n}\|_2^2+\frac{1}{8}\bold{a}_{m,n}^2\lf\|   \mathring{\omega}_{m,n;k}  \sqrt{-\pa_t W}    e^{W} \chi_{m+n}\rg\|_2^2.\label{I_21'}
\end{align}
Here we have chosen $K$ large enough in \eqref{W_prop}. 


Next, we recall the definitions of $\bold{a}_{m,n}, \, B_{m,n}$ \eqref{a:weight}, \eqref{Bweight}, the parameter choice \eqref{s:prime} and the property of the cutoff $\chi_{m+n}$ \eqref{chi:prop:3}
\begin{align*}
|\mathcal{I}_{m,n;k}^{(2, 2)}| \lesssim & (m+n)^{1+ \sigma } \bold{a}_{m,n}^2  \| \sqrt{\nu} \pa_y\mathring{\omega}_{m,n;k} e^W \chi_{m+n}\|_{L_y^2}\| \sqrt{\nu} \mathring{\omega}_{m,n;k} e^W \chi_{m+n-1} \|_{L_y^2}\myr{\mathbbm{1}_{n\geq 1}\mathbbm{1}_{m+n\geq 1}} \\
&+  m^{1+ \sigma } \bold{a}_{m,0}^2  \| \sqrt{\nu} \pa_y\mathring{\omega}_{m,0;k} e^W \chi_{m}\|_{L_y^2}\| \sqrt{\nu} \mathring{\omega}_{m,0;k} e^W \chi_{m-1} \|_{L_y^2}\myr{\mathbbm{1}_{m\geq 1}}\\
\lesssim& (m+n)^{\sigma } \bold{a}_{m,n}^2   \| \sqrt{\nu} \pa_y\omega_{m,n;k} e^W \chi_{m+n}\|_{L_y^2}\lf\| \sqrt{\nu}\lf(\frac{m+n}{q} \mathring{\omega}_{m,n;k}\rg) e^W \chi_{m+n-1} \rg\|_{L_y^2}\myr{\mathbbm{1}_{n\geq 1}}\\
&+\bold{a}_{m,0}  \| \sqrt{\nu} \pa_y\mathring{\omega}_{m,0;k} e^W \chi_{m}\|_{L_y^2}\ \bold{a}_{m-1,0} \lambda^{s}m^{-\sigma_\ast } \| \sqrt{\nu}|k| \omega_{m-1,0;k} e^W \chi_{m-1} \|_{L_y^2}\myr{\mathbbm{1}_{m\geq 1}}.
\end{align*} 
In the next section, we will develop several technical estimates concerning the quantity $S_{m,n}^{(1,0,0)}\omega_k:=\frac{m+n}{q}\mathring{\omega}_{m,n;k}\mathbbm{1}_{\mathbb{S}_n^1}(1,0,0)$ \eqref{S_nq}. Without digging into the details, we directly invoke \eqref{S:est:1} to obtain that 
 \begin{align} \n
|\mathcal{I}_{m,n;k}^{(2, 2)}| \lesssim &(m+n)^{\sigma } \bold{a}_{m,n }   \| \sqrt{\nu} \pa_y\mathring{\omega}_{m,n;k} e^W \chi_{m+n}\|_{L_y^2} \lambda^{s}(m+n)^{-\sigma-\sigma_\ast}\sqrt{\mathcal{D}_{n-1, m;k}^{(\gamma)}}\\
\leq &  \frac{1}{8} \bold{a}_{m ,n}^2   \| \sqrt{\nu} \pa_y\mathring{\omega}_{m,n;k} e^W \chi_{m+n}\|_{L_y^2}^2  +C \frac{\lambda^{2s}}{(m+n)^{2\sigma_{\ast}}} \mathcal{D}^{(\gamma)}_{n-1,m;k}.\label{I_22} 
\end{align}
Combining \eqref{dt_ga}, \eqref{I_i}, \eqref{I_2i}, \eqref{I_21'},  and \eqref{I_22} yields \eqref{gam:est:le1}. 

\siming{(Check!)} Finally, we observe that the estimates above applied to the $m+n=0$ case with mild modification.   
This concludes the proof of the lemma. 
\end{proof}

We now want to commute our $\gamma$ estimate by a factor of $\nu|k|$.
\begin{lemma}[The $\mu$-term Estimate]
The time evolution of the $\mu$-component of the functional $\mathcal{E}_{m,n;k}$ \eqref{defn:F:1} can be estimated as follows:
\begin{align} \n
\frac{d}{dt}&\lf(\bold{a}_{m,n}^2 \nu|k|^2\| \mathring{\omega}_{m,n;k} e^W  \chi_{m+n} \|_{L^2}^2\rg) +  \mathcal{D}_{m,n;k}^{(\mu)}+{2} \mathcal{CK}_{m,n;k}^{(\mu;\varphi)}+\mathcal{CK}_{m,n;k}^{(\mu;W)}\\
&\lesssim \frac{\lambda^{2s}}{(m+n)^{2\sigma_{\ast}}}  \mathcal{D}^{(\mu)}_{m,n-1;k}\mathbbm{1}_{n\geq 1} + \frac{\lambda^{2s}}{m^{2\sigma_{\ast}}}  \mathcal{D}^{(\mu)}_{m-1,0;k}\myr{\mathbbm{1}_{n=0}}\mathbbm{1}_{m\geq 1}\n \\
&\quad+\bold{a}_{m,n}^2 \nu|k|^2\mathrm{Re}\lf\langle \mathring{\omega}_{m,n;k}, \lf(\bold{T}^{(m,n)}_{k} +\bold{C}^{(m,n)}_{\mathrm{trans};k} + \bold{C}^{(m,n)}_{\mathrm{visc};k}  +\bold{C}^{(m,n)}_{q,k} \rg)e^{2W} \chi_{m+n}^2 \rg\rangle . \label{mu:est:le}
\end{align}
The dissipation term $\mathcal{D}^{(\mu)}$ is defined in \eqref{defn:D:1} and the Cauchy-Kovalevskaya terms $\mathcal{CK}^{(\mu)}$ are defined in \eqref{CK_phi_B_1} and \eqref{CK_W_mu_1}. The parameter $\sigma_\ast$ is defined in \eqref{s:prime}. The terms $\bold{T}^{(m,n)}_{k},\ \bold{C}^{(m,n)}_{\mathrm{trans};k}, \ \bold{C}^{(m,n)}_{\mathrm{visc};k} , \ \bold{C}^{(m,n)}_{q,k}$ are defined in \eqref{aff1}. 
\end{lemma}
\begin{proof} This follows by scaling \eqref{gam:est:le1} by $\nu|k|^2 $, and using the bound 
\begin{align*}
\nu |k|^2 \mathcal{D}^{(\gamma)}_{m,n-1;k}(t) \lesssim \mathcal{D}^{(\mu)}_{m,n-1;k}(t).
\end{align*}
\end{proof}

\begin{lemma}[The $\alpha$-term Estimate]
The time evolution of the $\al$-component of the functional $\mathcal{E}_{m,n;k}$ \eqref{defn:F:1} can be estimated as follows:
\begin{align} \n
\frac{d}{dt}&\lf( c_\al \bold{a}_{m,n}^2\nu   \|  \pa_y \mathring{\omega}_{m,n;k} e^{W} \chi_{m+n} \|_{L^2}^2\rg) +c_\al\mathcal{D}_{m,n;k}^{(\al)}+2c_\al\mathcal{CK}_{m,n;k}^{(\al;\varphi)} +c_\al\mathcal{CK}_{m,n;k}^{(\al;W)} \\
 \leq & C\frac{\lambda^{2s}}{(m+n)^{2\sigma_{\ast}}}\mathcal{D}_{m,n-1;k}\mathbbm{1}_{n\geq 1}   +C\frac{\lambda^{2s}}{m^{2\sigma_{\ast}}}\mathcal{D}_{m-1,0;k}\mathbbm{1}_{m\geq 1,n=0} + 2 c_\al \|1+\pa_y U^x_0 \|_{L^\infty_{t,y}}\mathcal{D}^{(\gamma)}_{m,n;k}\n \\
 &+ \mathbbm{1}_{n=1} \frac{C\lambda^{2s}}{(m+1)^{2\sigma_\ast}}\lf(\mathcal{D}_{m,0;k}^{(\al)}+\mathcal{CK}_{m,0;k}^{(\al,W)}\rg)  \n \\ 
 &+  C\bold{a}_{m,n}^2 \nu\mathrm{Re} \lf\langle  \pa_y \mathring{\omega}_{m,n;k}, \lf(\pa_y \bold{T}^{(m,n)}_{k}+\pa_y \bold{C}^{(m,n)}_{\mathrm{trans},k} +\pa_y \bold{C}^{(m,n)}_{\mathrm{visc},k}+\pa_y \bold{C}^{(m,n)}_{q,k}\rg)  e^{2W} \chi_{m+n}^2\rg\rangle .\label{alpha:est:le}
\end{align}
Here the $\mathcal{D}^{(\al)},\, \mathcal{CK}^{(\al)}$ terms are defined in \eqref{defn:D:1}, \eqref{CK_phi_B_1}, \eqref{CK_W_al_1}.
Here we also recall the definitions in \eqref{hghg1}. 
\end{lemma}
\begin{proof} We freeze $( m,n;k)$ and compute upon invoking \eqref{CK_H}, \eqref{hghg1} (below we define $\wt{U}^x_0 := y + U^x_0$),
\begin{align}
\n\frac{c_\al}{2} & \frac{d}{dt} \lf(\bold{a}_{m,n}^2 \nu\|  \pa_y \mathring{\omega}_{m,n;k} e^{W} \chi_{m+n} \|_{L^2}^2\right)  \\
\n = & \frac{c_\al}{2}   \nu\mathrm{Re} \lf\langle   \pa_y \mathring{\omega}_{m,n;k} ,   \pa_y \mathring{\omega}_{m,n;k}  \pa_t(\bold{a}_{m,n}^2 e^{2W})  \chi_{m+n}^2\rg\rangle  + c_\al\bold{a}_{m,n}^2 \nu \mathrm{Re}\lf\langle  \pa_t \pa_y \mathring{\omega}_{m,n;k} ,   \pa_y \mathring{\omega}_{m,n;k}   e^{2W} \chi_{m+n}^2 \rg \rangle \\
\n \leq & -c_\al\sum_{\iota\in\{\varphi,W\}}\mathcal{CK}_{m,n;k}^{(\al;\iota)}-    c_\al\bold{a}_{m,n}^2\nu\mathrm{Re}\lf \langle   \pa_y \mathring{\omega}_{m,n;k},  i k \wt{U}^x_0 \pa_y \mathring{\omega}_{m,n;k} e^{2W}\chi_{m+n}^2 \rg\rangle \\ 
\n &+ c_\al\bold{a}_{m,n}^2 \nu\mathrm{Re} \lf\langle \pa_y \mathring{\omega}_{m,n;k},  \nu \Delta_k \pa_y \mathring{\omega}_{m,n;k}  e^{2W} \chi_{m+n}^2\rg\rangle \\
\n &  -   c_\al\bold{a}_{m,n}^2 \nu\mathrm{Re}\lf \langle   \pa_y \mathring{\omega}_{m,n;k},  i k \pa_y \wt U^x_0 \mathring{\omega}_{m,n;k} e^{2W}\chi_{m+n}^2\rg \rangle\\
\n &+  c_\al\bold{a}_{m,n}^2 \nu\mathrm{Re} \lf\langle  \pa_y \mathring{\omega}_{m,n;k}, \lf(\pa_y \bold{T}^{(m,n)}_{k}+\pa_y \bold{C}^{(m,n)}_{\mathrm{trans},k} +\pa_y \bold{C}^{(m,n)}_{\mathrm{visc},k}+\pa_y \bold{C}^{(m,n)}_{q,k}\rg)  e^{2W} \chi_{m+n}^2\rg\rangle  \\
=: & -c_\al\sum_{\iota\in\{\varphi,W\}}\mathcal{CK}_{m,n;k}^{(\al;\iota)}+ \sum_{i = 1}^{4} \mathcal{I}_{m,n;k}^{(i)}.\label{al_I_i}
\end{align}
We first note that the term $\mathcal{I}_{m,n;k}^{(1)}$ in \eqref{al_I_i} is zero, 
\begin{align} \mathcal{I}_{m,n;k}^{(1)} = - c_\al\bold{a}_{m,n}^2 \nu\mathrm{Re} \int_{-1}^1  | \pa_y \mathring{\omega}_{m,n;k}|^2  i k (y + U^x_0(t, y)) e^{2W}\chi_{m+n}^2 dy = 0.\label{al_I_1}
\end{align}  
Next we address the term $\mathcal{I}_{m,n;k}^{(2)}$ in \eqref{al_I_i}. Integration by parts yields the following decomposition
\begin{align} \n
\mathcal{I}_{m,n;k}^{(2)}:=  & c_\al\bold{a}_{m,n}^2\nu\mathrm{Re}\lf\langle \pa_y \mathring{\omega}_{m,n;k}, \nu \pa_y^2 (\pa_y \mathring{\omega}_{m,n;k})  e^{2W}   \chi_{m+n}^2 \rg\rangle-c_\al\bold{a}_{m,n}^2\nu^{2}|k|^{2}  \lf\|  \pa_y \mathring{\omega}_{m,n;k}e^{W}\chi_{m+n}\rg\|_{L_y^2}^2   \\ \n
= & -  c_\al\bold{a}_{m,n} ^2\nu^{2} \|  \pa_y^2 \mathring{\omega}_{m,n;k} e^W \chi_{m+n} \|_{L_y^2}^2  -  c_\al \bold{a}_{m,n} ^2\nu^{2}|k|^{2} \| \pa_y \mathring{\omega}_{m,n;k}e^W \chi_{m+n} \|_{L_y^2}^2 \\ \n
&- 2c_\al \bold{a}_{m,n}^2 \nu^{2} \mathrm{Re} \lf\langle \pa_y \mathring{\omega}_{m,n;k},   \pa_y^2 \mathring{\omega}_{m,n;k} ( \pa_yW) e^{2W} \chi_{m+n}^2 \rg\rangle\\
\n  &+ c_\al\bold{a}_{m,n}^2\nu^{2} \mathrm{Re}\lf(  \pa_y \mathring{\omega}_{m,n;k} \overline{\pa_y^2 \mathring{\omega}_{m,n;k}}\rg) e^{2W} \chi_{m+n}^2\bigg|_{y = \pm 1} \\ \n
& - c_\al\bold{a}_{m,n}^2 \nu^{2} \mathrm{Re} \lf\langle  \pa_y \mathring{\omega}_{m,n;k}, \pa_y^2 \mathring{\omega}_{m,n;k}  e^{2W} \pa_y (\chi_{m+n}^2)\rg \rangle \\ 
= :& - c_\al \bold{a}_{m,n}^2 \nu^{2} \|  \pa_y \nabla_k \mathring{\omega}_{m,n;k} e^W  \chi_{m+n} \|_{L_y^2}^2+  \sum_{i = 1}^{3}  \mathcal{I}_{m,n;k}^{(2,i)}. \label{al_I_2i}
\end{align}
We recall the negative $\mathcal{CK}$ term in \eqref{al_I_i} and the dissipative term in \eqref{al_I_2i}, and apply the property of the $W$-weight \eqref{W_prop} to estimate the $\mathcal{I}_{m,n;k}^{(2,1)}$-term as follows
\begin{align}\n 
|\mathcal{I}_{m,n;k}^{(2,1)}| \leq &2\bold{a}_{m,n}^2  \nu^{2} \|  |\pa_yW|   \pa_y \mathring{\omega}_{m,n;k}e^W \chi_{m+n} \|_{L^2_y} \|  \pa_y^2 \mathring{\omega}_{m,n;k}e^W \chi_{m+n}  \|_{L^2_y} \\
\n \leq & \bold{a}_{m,n}^2   \frac{8}{\sqrt{K}}  \lf \|  \sqrt{-\pa_t W}   \nu^{1/2}\pa_y \mathring{\omega}_{m,n;k} e^W \chi_{m+n}\rg\|_{L^2_y} \| \nu  \pa_y^2 \mathring{\omega}_{m,n;k} e^W \chi_{m+n}\|_{L^2_y} \\
 \leq & \frac{C}{K^{\frac12}}\mathcal{CK}_{m,n;k}^{(\al;W)}+\frac{C}{K^{\frac{1}{2}}} \bold{a}_{m,n}^2\nu^{2}\|  \nabla_k \pa_y \mathring{\omega}_{m,n;k} e^W  \chi_{m+n} \|_{L_y^2}^2. \label{al_I_21'}
\end{align}
If $K$ is chosen larger than some universal constants, these two terms can be absorbed by the negative terms in \eqref{al_I_i} and \eqref{al_I_2i}.  

Next we address the important boundary contribution from above, $\mathcal{I}_{m,n;k}^{(2,2)}$, for which we invoke Lemma \ref{lemma:BC}, identity \eqref{id:BC}, which shows that 
\begin{align} \label{al_I_22}
|\mathcal{I}^{(2, 2)}_{m,n;k} |\lesssim \mathbbm{1}_{n=1} \frac{\lambda^{2s}}{(m+1)^{2\sigma_\ast}}\lf(\mathcal{D}_{m,0;k}^{(\al)}+\mathcal{CK}_{m,0;k}^{(\al)}\rg). 
\end{align}

Next, we treat the term $\mathcal{I}_{m,n;k}^{(2,3)}$ in \eqref{al_I_2i}. The idea is similar to the one applied in \eqref{I_22}. However, this time, we need to invoke a \myr{technical estimate \eqref{S:est:2}} on the function $S_{m,n}^{(1,1,0)}f_k:=\frac{m+n}{q}\pa_y (|k|^mq^n\Gamma_k^n\omega_{k})\mathbbm{1}_{\mathbb S_n^2}(1,1,0)$ \eqref{S_nq}. Combining it with the bounds on the cut-off functions \eqref{chi:prop:3},  we have that
\begin{align}
|\mathcal{I}_{m,n;k}^{( 2,3)}| \leq & Cc_\al\Big(\bold{a}_{m,n} \| \nu\pa_y^2 \mathring{\omega}_{m,n;k} e^W \chi_{m+n} \|_{L^2_y} \Big)  \Big(\bold{a}_{m,n}  (m+n)^{1+\sigma} \| \nu\pa_y \mathring{\omega}_{m,n;k} e^W \chi_{m+n-1} \|_{L^2_y} \Big) \n \\
\leq& \frac{c_\al}{8} \bold{a}_{m,n}^2\nu^{2 }\| \pa_y^2 \mathring{\omega}_{m,n;k} e^W \chi_{m+n} \|_{L^2_y} ^2\n \\
&+C\bold{a}_{m-1,0;k}^2 \nu^{2} \frac{\lambda^{2s}}{m^{2\sigma}} \lf\| \pa_y |k|\omega_{m-1,0;k} e^W \chi_{m-1} \rg\|_{L^2_y}^2\mathbbm{1}_{m\geq 1, n=0} \n \\
&+C\bold{a}_{m,0;k}^2 \nu^{2 } \frac{\lambda^{2s}\varphi^2}{(m+1)^{2\sigma_\ast}} \lf\| \pa_y (q(\vyn\pa_y+ikt)\mathring{\omega}_{m,0;k}) e^W \chi_{m} \rg\|_{L^2_y}^2\mathbbm{1}_{m\geq 0,n=1} \n \\
&+C\bold{a}_{m,n}^2 \nu^{2 }(m+n)^{2\sigma} \lf\| \frac{m+n}{q}\pa_y \mathring{\omega}_{m,n;k} e^W \chi_{m+n-1} \rg\|_{L^2_y}^2\mathbbm{1}_{n\geq 2} \n \\
\leq & \frac{c_\al}{8} \bold{a}_{m,n}^2\nu^{2} \| \pa_y^2 \mathring\omega_{\myr{m, n ;k}} e^W \chi_{m+n} \|_{L^2_y} ^2+ \frac{C\lambda^{2s}}{(m+n)^{2 \sigma_{\ast}}} \lf(\mathcal{D}^{(\alpha)}_{m,n-1;k}\mathbbm{1}_{n\geq 1} + \mathcal{D}^{(\mu)}_{m,n-1;k}\mathbbm{1}_{n\geq 1} \rg)\n\\
&+\frac{C\lambda^{2s}}{m^{2 \sigma_{\ast}}}\mathcal{D}^{(\alpha)}_{m-1,0;k}\mathbbm{1}_{m\geq 1}\mathbbm{1}_{n=0}. \label{al_I_23}
\end{align}

Finally, we estimate the term $\mathcal{I}^{(3)}_{m,n;k}$ in \eqref{al_I_i} as follows, 
\begin{align}\n
|\mathcal{I}^{(3)}_{m,n;k}| \leq & c_\al\bold{a}_{m,n}^2\|  {1 + \pa_y U^x_0 }\|_{L^\infty_{t,y}} \|\nu^{1/2}  |k|  \mathring{\omega}_{m,n;k} e^W \chi_{m+n} \|_{L^2} \|  \nu^{\frac12} \pa_y \mathring{\omega}_{m,n;k} e^W \chi_{m+n}\|_{L^2} \\
\leq & c_\al\bold{a}_{m,n}^2\| {1 + \pa_y U^x_0 }\|_{L_{t,y}^\infty}\mathcal{D}^{(\gamma)}_{m,n;k}. \label{al_I_3}
\end{align}Combining the decomposition \eqref{al_I_i}, \eqref{al_I_2i}, and the estimates  \eqref{al_I_1}, \eqref{al_I_21'}, \eqref{al_I_23}, \eqref{al_I_3}, we have obtained \eqref{alpha:est:le}.  
This concludes the proof of the lemma. 
\end{proof}
\ifx
\begin{lemma}[The $\beta$-term Estimate] 
The time evolution of the $\beta$-component of the hypocoercivity functional $\bold{E}_{m,n;k}$ \eqref{defn:E:1} can be estimated as follows:
\begin{align} \n
\frac{d}{dt} &\lf(\bold{a}_{m,n}^2\nu^{\frac13}|k|^{-\frac43}  \mathrm{Re} \lf\langle i k \mathring{\omega}_{m,n;k} , \pa_y \mathring{\omega}_{m,n;k} e^{2W} \chi_{m+n}^2  \rg\rangle\rg) +\mathcal G_{m,n;k} \\ \n 
\leq &  \siming{\frac{c_\al}{8c_\beta}}\sum_{\iota\in \{\varphi,B,W\}}\mathcal{CK}^{(\al;\iota)}_{m,n;k}+\siming{\frac{C c_\al}{ c_\beta}}\sum_{\iota\in \{\varphi,B,W\}}\mathcal{CK}^{(\gamma;\iota)}_{m,n;k}\\
\n &+ \siming{\frac{c_\al}{8c_\beta}}\mathcal{D}^{(\alpha)}_{m,n;k}+C\siming{\frac{ c_\beta}{c_\al}}\mathcal{D}^{(\gamma)}_{m,n;k}+ \frac{C\lambda^{2s}}{ (m+n)^{2\sigma_\ast}} \sum_{\iota \in \{ \alpha, \gamma,  \mu\}} \mathcal{D}^{(\iota)}_{m,n-1;k} \mathbbm{1}_{n\geq 1}\myr{+\frac{\lambda^{2s}}{m^{2\sigma_\ast}} {\mathcal{D}^{(\gamma)}_{m-1,0;k}}\mathbbm{1}_{m\geq1,n= 0}}\\
&+ \bold{a}_{m,n}^2\nu^{\frac13}|k|^{-\frac43}\mathrm{Re} \lf \langle i k\lf( \bold{T}^{(m,n)}_{ k}+ \bold{C}^{(m,n)}_{\mathrm{trans},k}+\bold{C}^{(m,n)}_{\mathrm{visc},k}+\bold{C}^{(m,n)}_{q,k}\rg), \pa_y \mathring{\omega}_{m,n;k} e^{2W} \chi_{m+n}^2\rg \rangle \n \\
&+\bold{a}_{m,n}^2\nu^{\frac13}|k|^{-\frac43}  \mathrm{Re} \lf\langle i k \mathring{\omega}_{m,n;k}, \lf ( \pa_y \bold{T}^{(m,n)}_{k}+ \pa_y\bold{C}^{(m,n)}_{\mathrm{trans},k}+\pa_y\bold{C}^{(m,n)}_{\mathrm{visc},k}+\pa_y \bold{C}^{(m,n)}_{q,k} \rg) e^{2W} \chi_{m+n}^2\rg\rangle. \label{bgty}
\end{align}
Here we recall the definitions in \eqref{defn:D:1}, \eqref{CK_H}, \eqref{G}, \eqref{aff1} and \eqref{hghg1}. 
\end{lemma}
\begin{proof} We compute for frozen $(m,n, k)$, the identity
\begin{align} \n
 \frac{d}{dt}\lf(\bold{a}_{m,n}^2\nu^{\frac13}|k|^{-\frac43}   \mathrm{Re} \lf\langle i k \mathring{\omega}_{m,n;k} , \pa_y \mathring{\omega}_{m,n;k} e^{2W} \chi_{m+n}^2  \rg \rangle\rg) 
= &  \nu^{\frac13} |k|^{-\frac43}  \mathrm{Re} \lf\langle i k  \mathring{\omega}_{m,n;k}, \pa_y \mathring{\omega}_{m,n;k} \pa_t(  \bold{a}_{m,n}^2 e^{2W} )\chi_{m+n}^2 \rg \rangle \\ &+\bold{a}_{m,n}^2 \nu^{\frac13} |k|^{-\frac43}  \mathrm{Re}\lf \langle i k \pa_t \mathring{\omega}_{m,n;k}, \pa_y \mathring{\omega}_{m,n;k} e^{2W} \chi_{m+n}^2\rg\rangle \n \\ \n
& +  \bold{a}_{m,n}^2\nu^{\frac13}|k|^{-\frac43}  \mathrm{Re}\lf \langle i k \mathring{\omega}_{m,n;k}, \pa_y \pa_t \mathring{\omega}_{m,n;k} e^{2W} \chi_{m+n}^2\rg \rangle \\ 
 =: &\sum_{i=1}^3\mathcal{I}_{m,n;k}^{(i)}.\label{beta_I_i}
\end{align}
We first bound the $\mathcal{I}_{m,n;k}^{(1)}$ from above by a simple application of Cauchy-Schwartz:
\begin{align}
|\mathcal{I}_{m,n;k}^{(1)}&| =2\bold{a}_{m,n}^2\nu^{\frac13}|k|^{-\frac13}\lf|\mathrm{Re}\lf\lan i\mathring{\omega}_{m,n;k},\pa_{y}\omega_{k, n}\lf(\frac{(m+n)\dot \varphi}{\varphi}+\frac{\dot B_{m,n}}{B_{m,n}}+\dot W\rg)e^{2W}\chi_{m+n}^2\rg\ran\rg| \\
\leq&C\lf(\bold{a}_{m,n}\sqrt{(m+n)\frac{|\dot \varphi|}{\varphi}}\lf\| \mathring{\omega}_{m,n;k}  e^W \chi_{m+n}\rg \|_{L^2_y} \rg)\lf(\bold{a}_{m,n}\sqrt{(m+n)\frac{|\dot \varphi|}{\varphi}}\lf\| \nu^{\frac13} |k|^{-\frac13} \pa_y \mathring{\omega}_{m,n;k} e^W \chi_{m+n} \rg\|_{L^2_y}\rg)\\
&+C\lf(\bold{a}_{m,n}\sqrt{\frac{|\dot B_{m,n}|}{B_{m,n}}}\lf\| \mathring{\omega}_{m,n;k}  e^W \chi_{m+n}\rg \|_{L^2_y} \rg)\lf(\bold{a}_{m,n}\sqrt{\frac{|\dot B_ {m,n}|}{B_{m,n}}}\lf\| \nu^{\frac13} |k|^{-\frac13} \pa_y \mathring{\omega}_{m,n;k} e^W \chi_{m+n} \rg\|_{L^2_y}\rg)\\
&+ C\lf(\bold{a}_{m,n}\lf\| \mathring{\omega}_{m,n;k} \sqrt{-\pa_t W} e^W \chi_{m+n}\rg \|_{L^2_y}\rg)\lf(\bold{a}_{m,n} \lf\| \nu^{\frac13} |k|^{-\frac13} \pa_y \mathring{\omega}_{m,n;k} \sqrt{-\pa_t W} e^W \chi_{m+n} \rg\|_{L^2_y} \rg) \\
\leq& \siming{\frac{c_\al}{16c_\beta}}\sum_{\iota\in \{\varphi,B,W\}}\mathcal{CK}^{(\al;\iota)}_{m,n;k}+\siming{\frac{Cc_\beta}{c_\al}}\sum_{\iota\in \{\varphi,B,W\}}\mathcal{CK}^{(\gamma;\iota)}_{m,n;k} .\label{beta_I_1}
\end{align}
Here we recall the definitions \eqref{CK_H}.

We now handle the $\mathcal{I}^{(2)}_{m,n;k}$ term, for which we recall equation \eqref{aff1} to produce the identity 
\begin{align} 
\n \mathcal{I}_{m,n;k}^{(2)} = & \bold{a}_{m,n}^2 \nu^{\frac13}  |k|^{\frac23}  \mathrm{Re} \lf \langle (y + U^x_0) \mathring{\omega}_{m,n;k}, \pa_y \mathring{\omega}_{m,n;k} e^{2W} \chi_{m+n}^2 \rg \rangle  \\ \n
&+   \bold{a}_{m,n}^2\nu^{\frac43}  |k|^{ -\frac43}  \mathrm{Re} \lf \langle i k  \Delta_k \mathring{\omega}_{m,n;k}, \pa_y \mathring{\omega}_{m,n;k} e^{2W} \chi_{m+n}^2\rg\rangle \\
&+  \bold{a}_{m,n}^2\nu^{\frac13}|k|^{-\frac43}\mathrm{Re} \lf \langle i k\lf( \bold{T}^{(m,n)}_{ k}+ \bold{C}^{(m,n)}_{\mathrm{trans},k}+\bold{C}^{(m,n)}_{\mathrm{visc},k}+\bold{C}^{(m,n)}_{q,k}\rg), \pa_y \mathring{\omega}_{m,n;k} e^{2W} \chi_{m+n}^2\rg \rangle \n \\
  =: &  \sum_{i = 1}^{3}\mathcal{I}^{(2,i)}_{m,n;k} .\label{beta_I_2i}
\end{align}
We bound each of these contributions now. First, we will see that $\mathcal{I}_{m,n;k}^{(2,1)}$ cancels exactly with $\mathcal{I}_{m,n;k}^{(3,1)}$ below, so we do not need to estimate them at this stage. Turning to the diffusive term, we control it from above via 
\begin{align}
|\mathcal{I}_{m,n;k}^{(2,2)} | \lesssim &\bold{a}_{m,n}^2\nu^{\frac43} |k|^{ -\frac43} \lf |\mathrm{Re} \lf \langle i k (\pa_y^2 - |k|^2)\mathring{\omega}_{m,n;k}, \pa_y \mathring{\omega}_{m,n;k} e^{2W} \chi_{m+n}^2\rg \rangle\rg |\n  \\
 \lesssim & \bold{a}_{m,n}^2 \| \nu^{\frac56}  |k|^{-\frac13}\pa_y^2 \mathring{\omega}_{m,n;k} e^W \chi_{m+n} \|_{L_y^2} \|  \nu^{\frac12}\pa_y \mathring{\omega}_{m,n;k} e^W \chi_{m+n} \|_{L^2_y} \n \\
 & +  \bold{a}_{m,n}^2\| \nu^{\frac56}  |k|^{\frac53}\pa_y \mathring{\omega}_{m,n;k} e^W \|_{L_y^2} \| \nu^{\frac12} |k|\mathring{\omega}_{m,n;k} e^W\chi_{m+n} \|_{L^2_y} \n \\
 \lesssim & \sqrt{\mathcal D^{(\alpha)}_{m,n;k}\mathcal{D}_{m,n;k}^{(\gamma)}}.\label{beta_I_22}
\end{align} 
Here $ \mathcal D^{(\alpha)}_{m,n;k}, \ \mathcal{D}_{m,n;k}^{(\gamma)}$ are defined in \eqref{defn:D:1}.

Upon invoking \eqref{hghg1}, we have that
\begin{align}\n 
&\mathcal{I}_{m,n;k}^{(3)}\\
  &= -\bold{a}_{m,n}^2 \nu^{\frac13}|k|^{ \frac23}  \mathrm{Re} \lf \langle   \mathring{\omega}_{m,n;k},  (y + U^x_0) \pa_y \mathring{\omega}_{m,n;k} e^{2W} \chi_{m+n}^2\rg \rangle\n  \\
&  \quad-\bold{a}_{m,n}^2 \nu^{\frac13}|k|^{\frac23} \mathrm{Re} \lf\langle    \mathring{\omega}_{m,n;k},    (1 + \pa_y U^x_0) \mathring{\omega}_{m,n;k} e^{2W} \chi_{m+n}^2\rg\rangle + \bold{a}_{m,n}^2\nu^{\frac13}|k|^{-\frac43} \mathrm{Re} \lf\langle i k \mathring{\omega}_{m,n;k}, \nu \Delta_k \pa_y \mathring{\omega}_{m,n;k}  e^{2W} \chi_{m+n}^2\rg\rangle  \n \\ 
& \quad + \bold{a}_{m,n}^2\nu^{\frac13}|k|^{-\frac43}  \mathrm{Re} \lf\langle i k \mathring{\omega}_{m,n;k}, \lf ( \pa_y \bold{T}^{(m,n)}_{k}+ \pa_y\bold{C}^{(m,n)}_{\mathrm{trans},k}+\pa_y\bold{C}^{(m,n)}_{\mathrm{visc},k}+\pa_y \bold{C}^{(m,n)}_{q,k} \rg) e^{2W} \chi_{m+n}^2\rg\rangle \n \\
&=:   \sum_{i = 1}^{4} \mathcal{I}_{m,n;k}^{(3,i)}.\label{beta_I_3i}
\end{align}
Clearly, we have 
\begin{align}
\mathcal{I}^{(2,1)}_{ m,n;k} + \mathcal{I}^{(3,1)}_{ m,n;k} = 0.\label{beta_I_23_1}
\end{align} We now extract the main favorable contribution from this estimate, which is coming from $\mathcal{I}^{(3,2 )}_{m, n;k}$:
\begin{align}
\mathcal{I}^{(3,2)}_{ m,n;k} = - \bold{a}_{m,n}^2\nu^{\frac13}|k|^{\frac23}  \lf \| \sqrt{1 + \pa_y U^x_0} \mathring{\omega}_{m,n;k} e^W \chi_{m+n} \rg\|_{L^2_y}^2.\label{beta_I_32}
\end{align}
We now move to the diffusive terms, which can be split into 
\begin{align}
\mathcal{I}^{(3,3)}_{ m,n;k} = &  \bold{a}_{m,n}^2 \nu^{\frac13}|k|^{-\frac43} \mathrm{Re} \lf\langle i k \mathring{\omega}_{m,n;k}, \nu \pa_y^3 \mathring{\omega}_{m,n;k}  e^{2W} \chi_{m+n}^2\rg\rangle  -  \bold{a}_{m,n}^2 \nu^{\frac13}|k|^{\frac23} \mathrm{Re}\lf \langle i k \mathring{\omega}_{m,n;k}, \nu  \pa_y \mathring{\omega}_{m,n;k}  e^{2W} \chi_{m+n}^2\rg \rangle \\
=:&  \mathcal{I}^{(3, 3, 1)}_{m,n;k} + \mathcal{I}^{(3,3, 2)}_{m,n;k}.\label{beta_I33i}
\end{align}
We can immediately estimate the latter term via 
\begin{align}
| \mathcal{I}^{(3,3, 2)}_{ m,n;k}| \lesssim &  \bold{a}_{m,n}^2 \| \nu^{\frac12}|k| \mathring{\omega}_{m,n;k} e^W \chi_{m+n} \|_{L^2_y}  \| \nu^{\frac56}|k|^{\frac23} \pa_y \mathring{\omega}_{m,n;k} e^W \chi_{m+n} \|_{L^2_y} \\
\lesssim & \sqrt{\mathcal{D}_{m,n;k}^{(\alpha)}\mathcal{D}_{m,n;k}^{(\gamma)} }.\label{beta_I332}
\end{align}
Here $ \mathcal D^{(\alpha)}_{m,n;k}, \ \mathcal{D}_{m,n;k}^{(\gamma)}$ are defined in \eqref{defn:D:1}.

Next, we recall the $ \mathcal{I}^{(3,3,1)}_{m,n;k}$-term in \eqref{beta_I33i}, and integrate by parts to obtain 
\begin{align}
 \mathcal{I}^{(3,3, 1)}_{ m,n;k} = & \bold{a}_{m,n}^2 \nu^{\frac13}|k|^{-\frac43}   \mathrm{Re} \lf\langle i k \mathring{\omega}_{m,n;k}, \nu \pa_y^3 \mathring{\omega}_{m,n;k}  e^{2W} \chi_{m+n}^2\rg\rangle \n \\
 = & - \bold{a}_{m,n}^2  \nu^{\frac43}|k|^{-\frac43} \mathrm{Re}\lf \langle i k \pa_y \mathring{\omega}_{m,n;k},  \pa_y^2 \mathring{\omega}_{m,n;k}  e^{2W} \chi_{m+n}^2\rg\rangle\n \\
 &- 2 \bold{a}_{m,n}^2 \nu^{\frac43}|k|^{-\frac43} \mathrm{Re} \lf\langle i k \mathring{\omega}_{m,n;k},  \pa_y^2 \mathring{\omega}_{m,n;k}  (\pa_y W)  e^{2W} \chi_{m+n}^2\rg\rangle \n \\
 & -  \bold{a}_{m,n}^2 \nu^{\frac43}|k|^{-\frac43} \mathrm{Re}\lf \langle i k \mathring{\omega}_{m,n;k},  \pa_y^2 \mathring{\omega}_{m,n;k}   e^{2W} \pa_y( \chi_{m+n}^2)\rg\rangle =: \sum_{i = 1}^3  \mathcal{I}^{(3,3, 1, i)}_{m,n;k}.\label{beta_I331i} 
\end{align}
We bound the $\mathcal{I}_{m,n;k}^{(3,3,1,1)}$ term  contributions via 
\begin{align}
| \mathcal{I}^{(3,3, 1, 1)}_{n, m;k} | \lesssim &  \bold{a}_{m,n}^2 \| \nu^{\frac12} \pa_y \mathring{\omega}_{m,n;k} e^W \chi_{m+n} \|_{L^2_y} \| \nu^{\frac56}|k|^{-\frac13} \pa_y^2 \mathring{\omega}_{m,n;k} e^W \chi_{m+n} \|_{L^2_y} \\
\lesssim & \sqrt{\mathcal{D}^{(\alpha)}_{m,n;k}(t) \mathcal{D}^{(\gamma)}_{m,n;k}(t)}.\label{beta_I3311} 
\end{align}Next, we estimate the $\mathcal{I}_{m,n;k}^{(3,3,1,2)}$ term with the property of $W$ \eqref{W_prop} as follows,
\begin{align}
| \mathcal{I}^{(3,3, 1, 2)}_{m,n;k} | \leq & C\bold{a}_{m,n}^2 \lf\| \sqrt{\nu}\mathring{\omega}_{m,n;k}  (\pa_y W) e^W \chi_{m+n} \rg\|_{L^2_y} \lf\| \nu^{\frac56}|k|^{-\frac13}\pa_y^2 \mathring{\omega}_{m,n;k} e^W \chi_{m+n}\rg \|_{L^2_y} \n \\
\leq &\frac{C}{K^{\frac12}} \bold{a}_{m,n}^2 \lf\| \mathring{\omega}_{m,n;k} \sqrt{-\pa_t W} e^W \chi_{m+n} \rg\|_{L^2_y} \lf\| \nu^{\frac56}|k|^{-\frac13}\pa_y^2 \omega_{m,n;k} e^W \chi_{m+n} \rg\|_{L^2_y} \n \\
\leq & \frac{C}{K^{\frac12}}\siming{\frac{c_\beta}{c_\al}}\bold{a}_{m,n}^2 \lf\| \mathring{\omega}_{m,n;k} \sqrt{-\pa_t W} e^W \chi_{m+n} \rg\|_{L^2_y}^2+  \siming{\frac{c_\al}{16c_\beta}}\mathcal{D}^{(\al)}_{m,n;k}. \label{beta_I3312}
\end{align}
 Lastly, we bound the $\mathcal{I}_{m,n;k}^{(3,3,1,3)}$-term in \eqref{beta_I331i} with the bounds on the cut-off functions \eqref{chi:prop:3}, \myr{the fact that one needs $m+n\geq 1$ for the term to be nonzero}, and a technical estimate \eqref{S:est:1} on the function $S_{m,n;k}:=\frac{m+n}{q} \mathring{\omega}_{m,n;k}\mathbbm{1}_{n\geq 1}$ as follows, 
\begin{align}\n
| &\mathcal{I}^{(3,3, 1,3)}_{n, m;k}| \\
\n \lesssim & \mathbf{a}_{m,n;k}^2\lf\| \nu^{\frac56} |k|^{-\frac13} \pa_y^2 \mathring{\omega}_{m,n;k}e^W \chi_{m+n} \rg\|_{L^2_y} \bigg( (m+n)^{\sigma} \lf\| \nu^{\frac12} \frac{m+n}{q} \mathring{\omega}_{m,n;k} e^W\chi_{m+n-1} \rg\|_{L^2_y}\mathbbm{1}_{n\geq 1} \\
\n&\quad\qquad+m^{1+\sigma}\|\nu^{\frac12}|k| \omega_{m-1,0;k} e^W\chi_{m-1} \|_{L_y^2}\mathbbm{1}_{m\geq 1,n=0} \bigg)\\ \n
\lesssim&\mathbf{a}_{m,n;k}    \lf\| \nu^{\frac56} |k|^{-\frac13} \pa_y^2 \mathring{\omega}_{m,n;k} e^W\chi_{m+n} \rg\|_{L^2_y}\\
\n&\times \lf( \mathbf{a}_{m,n-1;k}\frac{\lambda^{1/s'}}{(m+n)^{\sigma_\ast}}\lf\| \nu^{\frac12} S_{m,n;k} e^W\chi_{m+n-1} \rg\|_{L^2_y}+ \mathbf{a}_{m-1,0;k}\frac{\lambda^{1/s'}}{m^{\sigma_\ast}} \lf\| \nu^{1/2} |k| \omega_{m-1,0;k} e^W\chi_{m-1} \rg\|_{L^2_y}\mathbbm{1}_{m\geq 1,n=0}\rg)\\
\lesssim & \sqrt{\mathcal{D}^{(\alpha)}_{m,n;k}}\lf(\frac{\lambda^{2s}}{(m+n)^{2\sigma_\ast}} {\mathcal{D}^{(\gamma)}_{m,n-1;k}}\mathbbm{1}_{n\geq 1}+\frac{\lambda^{2s}}{m^{2\sigma_\ast}} {\mathcal{D}^{(\gamma)}_{m-1,0;k}}\mathbbm{1}_{m\geq1,n= 0}\rg)^{1/2} .\label{beta_I3313}
\end{align} 
Combining the decomposition \eqref{beta_I_i}, \eqref{beta_I_2i}, \eqref{beta_I_3i}, \eqref{beta_I33i}, \eqref{beta_I331i}, and the estimates \eqref{beta_I_1}, \eqref{beta_I_22}, \eqref{beta_I_23_1}, \eqref{beta_I_32}, \eqref{beta_I332}, \eqref{beta_I3311}, \eqref{beta_I3312}, \eqref{beta_I3313} yields \eqref{bgty}. 
 The proof is complete. 
\end{proof}
\fi

\subsection{Proof of Proposition \ref{pro:light:on}}\label{sec:con}

\ifx\myr{I have changed some lower order terms in the previous argument, so the computation here needs to be adjusted.}
\siming{
\begin{proposition} If the constants $c_{\gamma}, c_{\mu}, c_{\beta}$, and $c_{\alpha}$ are chosen as follows
\begin{align}
c_\al=c_\mu=\frac12,\quad c_\beta=1, \quad c_\gamma\geq c_{\gamma;0}(\|\sqrt{1+\pa_y U_0^x}\|_{L^\infty_{t,y}}),
\end{align}
then there exist constants $\delta \in (0,1),\quad  C>1$  such that 
\begin{align} \label{en:in:b}
\frac{d}{ dt}  \bold{E}_{m,n;k} + \delta \nu^{\frac13}|k|^{\frac23} \bold{E}_{m,n;k} + \mathcal{D}_{m,n;k} +\sum_{\zeta\in \{\al ,\mu,\gamma\}}\sum_{\iota\in \{\varphi,B,W\}}\mathcal{CK}^{(\zeta;\iota)}_{m,n;k} \leq C n^{-2\sigma_\ast} \sum_{\iota \in \{\gamma, \mu, \alpha \}} (\mathcal{D}^{(\iota)}_{ m,n-1;k})+... .
\end{align}
\end{proposition}
}\fi
 Proposition \ref{pro:light:on} is formalized as follows:
\begin{proposition}\label{pro:prf_pr22} 
If $\lambda$ \eqref{Gev:la} is small enough compared to universal constants and $c_\al=\frac{1}{8}$ in \eqref{defn:F:1}, the following estimate holds,
\begin{align}   \n
&\frac{d}{d t} \sum_{m+n=0}^\infty\mathcal{E}_{m,n;k}   +\frac{1}{9}\sum_{m+n=0}^\infty\sum_{\iota\in\{\al,\mu,\gamma\}}\mathcal{D}_{m,n;k}^{(\iota)} + \frac{1}{9}\sum_{m+n=0}^\infty\sum_{\iota\in\{\al,\mu,\gamma\}}\sum_{\zeta\in\{\varphi,W\}}\mathcal{CK}_{m,n;k}^{(\iota,\zeta)}  \\ \label{en:in:b}
& \lesssim   \sum_{m+n=0}^\infty\sum_{\iota \in \{ \alpha,  \mu, \gamma\}} \mathcal{I}^{(\iota)}_{m,n; k},
\end{align}
where we define the following inner products 
\begin{subequations}
\begin{align} \label{def:Inn:g}
\mathcal{I}^{(\gamma)}_{m,n, k}:= & \bold{a}_{m,n}^2 \mathrm{Re} \lf\langle \lf(\bold{T}^{(m,n)}_{k} +\bold{C}^{(m,n)}_{\mathrm{trans};k} + \bold{C}^{(m,n)}_{\mathrm{visc};k}  +\bold{C}^{(m,n)}_{q,k} \rg), \mathring\omega_{m,n;k}  e^{2W} \chi_{m + n}^2\rg \rangle, \\  \label{def:Inn:a}
\mathcal{I}^{(\alpha)}_{m,n, k} := & \bold{a}_{m,n}^2 \nu\mathrm{Re}\lf \langle \p_y \lf(\bold{T}^{(m,n)}_{k} +\bold{C}^{(m,n)}_{\mathrm{trans};k} + \bold{C}^{(m,n)}_{\mathrm{visc};k}  +\bold{C}^{(m,n)}_{q,k} \rg), \p_y \mathring\omega_{m,n;k}  e^{2W} \chi_{m + n}^2\rg \rangle, \\  \label{def:Inn:u}
\mathcal{I}^{(\mu)}_{m,n, k} := & { \bold{a}_{m,n}^2 \nu\mathrm{Re} \lf\langle |k|\lf(\bold{T}^{(m,n)}_{k} +\bold{C}^{(m,n)}_{\mathrm{trans};k} + \bold{C}^{(m,n)}_{\mathrm{visc};k}  +\bold{C}^{(m,n)}_{q,k} \rg), |k|\mathring \omega_{m,n;k}   e^{2W} \chi_{m + n}^2 \rg\rangle.}
\end{align}
\end{subequations}
\end{proposition}
\begin{proof} 
To obtain the estimate \eqref{en:in:b}, we sum over \eqref{gam:est:le1}, \eqref{mu:est:le} and \eqref{alpha:est:le}
\begin{align*}
\frac{d}{dt}&\sum_{m+n=0}^\infty\mathcal{E}_{m,n;k}+c_\al\sum_{m+n=0}^\infty\mathcal{D}_{m,n;k}^{(\al)} + c_\al\sum_{\zeta\in\{\varphi,W\}}\mathcal{CK}_{m,n;k}^{(\al,\zeta)} \\
&+\sum_{m+n=0}^\infty\sum_{\iota\in\{\mu,\gamma\}}\mathcal{D}_{m,n;k}^{(\iota)} + \sum_{\iota\in\{\mu,\gamma\}}\sum_{\zeta\in\{\varphi,W\}}\mathcal{CK}_{m,n;k}^{(\iota,\zeta)} \\
\leq &2c_\al\|1+\pa_yU_0^x\|_{L_{t,y}^\infty}\mathcal D_{m,n;k}^{(\gamma
)}+C\lambda^{2s}\sum_{m+n=0}^\infty\sum_{\iota\in\{\al,\mu,\gamma\}}\lf(\mathcal{D}_{m,n;k}^{(\iota)} + C\sum_{\zeta\in\{\varphi,W\}}\mathcal{CK}_{m,n;k}^{(\iota,\zeta)}\rg)\\
&+C\sum_{m+n=0}^\infty \sum_{\iota\in\{\al,\mu,\gamma\}}\mathcal{I}_{m,n;k}^{(\iota)}.
\end{align*}
Then we recall the assumption \eqref{close_vy},  choose $c_\al=\frac{1}{8}$ and $\lambda$ \eqref{Gev:la} small enough so that the first two terms on the right hand side are absorbed by the corresponding terms on the left hand side. This concludes the proof.  
\ifx 
We record:
\begin{align} \n
&\frac{\pa_t}{2}c_{\gamma} \varphi^{2(n+m)}|B_{m,n}|^2  e^{2\delta_E \nu^{\frac13}t}  \| |k|^m \mathring{\omega}_{m,n;k} e^W  \chi_{m+n} \|_{L^2}^2 \\ \n
= & c_{\gamma} (n+m) \varphi^{2n+2m-1} \dot{\varphi}(t) |B_{m,n}|^2 |k|^{2m} e^{\delta_E \nu^{\frac13}t} \| \mathring{\omega}_{m,n;k} e^W \chi_{m+n} \|_{L^2}^2 \\ \n
& + c_{\gamma} \varphi^{2(n+m)}B_{m,n} \dot{B}_{m,n}  e^{2\delta_E \nu^{\frac13}t}  \| |k|^m \mathring{\omega}_{m,n;k} e^W  \chi_{m+n} \|_{L^2}^2 \\ \n
& + c_{\gamma} \varphi^{2(n+m)}|B_{m,n}|^2 \delta_E \nu^{\frac13}e^{2\delta_E \nu^{\frac13}t}  \| |k|^m \mathring{\omega}_{m,n;k} e^W  \chi_{m+n} \|_{L^2}^2 \\ \n
& + c_{\gamma} \varphi^{2(n+m)}|B_{m,n}|^2  e^{2\delta_E \nu^{\frac13}t}  \frac{\pa_t}{2} \| |k|^m \mathring{\omega}_{m,n;k} e^W  \chi_{m+n} \|_{L^2}^2 \\ \n
\le & c_{\gamma} (n+m) \varphi^{2n+2m-1} \dot{\varphi}(t) |B_{m,n}|^2 |k|^{2m} e^{\delta_E \nu^{\frac13}t} \| \mathring{\omega}_{m,n;k} e^W \chi_{m+n} \|_{L^2}^2 \\ \n
 & + c_{\gamma} \varphi^{2(n+m)}B_{m,n} \dot{B}_{m,n}  e^{2\delta_E \nu^{\frac13}t}  \| |k|^m \mathring{\omega}_{m,n;k} e^W  \chi_{m+n} \|_{L^2}^2 \\ \n
& + c_{\gamma} \varphi^{2(n+m)}|B_{m,n}|^2 \delta_E \nu^{\frac13}e^{2\delta_E \nu^{\frac13}t}  \| |k|^m \mathring{\omega}_{m,n;k} e^W  \chi_{m+n} \|_{L^2}^2 \\ \n
&-c_{\gamma} \nu \varphi^{2(n+m)}  |B_{m,n}|^2  |k|^{2m} e^{2\delta_E \nu^{\frac13}t} \| \pa_y \mathring{\omega}_{m,n;k} e^W \chi_{m+n} \|_{L^2}^2 \\ \label{gam:est:le}
&-c_{\gamma}\nu  \varphi^{2(n+m)} |B_{m,n}|^2 |k|^{2m}e^{2\delta_E \nu^{\frac13}t} |k|^2 \| \mathring{\omega}_{m,n;k} e^W \chi_{m+n} \|_{L^2}^2 + \frac{C c_{\gamma}}{n^{2\sigma}}  (\mathcal{D}^{(\gamma)}_{n-1,m;k}(t))^2,
\end{align}
where we have invoked our $\gamma$-estimate, \eqref{gam:est:le}. To compactify the notations in \eqref{gam:est:le}, define the energy, CK, and the diffusion/ damping terms via 
\begin{align}  \n
CK^{(\gamma)}_{m,n;k} := & -  (n+m) \varphi^{2n+2m-1} \dot{\varphi}(t) |B_{m,n}|^2 |k|^{2m} e^{\delta_E \nu^{\frac13}t} \| \mathring{\omega}_{m,n;k} e^W \chi_{m+n} \|_{L^2}^2\\ \label{def:CK:gam}
& -   \varphi^{2(n+m)}B_{m,n} \dot{B}_{m,n}  e^{2\delta_E \nu^{\frac13}t}  \| |k|^m \mathring{\omega}_{m,n;k} e^W  \chi_{m+n} \|_{L^2}^2, \\ \n
\bold{D}^{(\gamma)}_{m,n;k} := & -  \varphi^{2(n+m)}|B_{m,n}|^2 \delta_E \nu^{\frac13}e^{2\delta_E \nu^{\frac13}t}  \| |k|^m \mathring{\omega}_{m,n;k} e^W  \chi_{m+n} \|_{L^2}^2 \\ \n
&+\nu \varphi^{2(n+m)}  |B_{m,n}|^2  |k|^{2m} e^{2\delta_E \nu^{\frac13}t} \| \pa_y \mathring{\omega}_{m,n;k} e^W \chi_{m+n} \|_{L^2}^2 \\ \label{def:D:gam:bld}
&+\nu  \varphi^{2(n+m)} |B_{m,n}|^2 |k|^{2m}e^{2\delta_E \nu^{\frac13}t} |k|^2 \| \mathring{\omega}_{m,n;k} e^W \chi_{m+n} \|_{L^2}^2 
\end{align}
We may therefore express \eqref{gam:est:le} by
\begin{align} \label{gam:sum}
&\frac{\pa_t}{2}c_{\gamma} \mathcal{E}^{(\gamma)}_{m,n;k} + c_{\gamma} CK^{(\gamma)}_{m,n;k} + c_{\gamma} \bold{D}^{(\gamma)}_{m,n;k} \le \frac{C c_{\gamma}}{(m+n)^{2\sigma}} \mathcal{D}^{(\gamma)}_{n-1,m;k}.
\end{align}
A rescaling of \eqref{gam:est:le} by a factor of $\frac{c_{\mu}}{c_{\gamma}} \nu^{\frac23} |k|^{\frac43}$ produces the estimate
\begin{align} \label{mu:sum}
&\frac{\pa_t}{2}c_{\mu} \mathcal{E}_{m,n;k}^{(\mu)}+ c_{\mu} CK^{(\mu)}_{m,n;k} + c_{\mu} \bold{D}^{(\mu)}_{m,n;k} \le \frac{ C c_{\mu}}{(m+n)^{2\sigma}} \mathcal{D}^{(\mu)}_{n-1,m;k},
\end{align}
where we have defined 
\begin{align}
\mathcal{E}_{m,n;k}^{(\mu)} = \nu^{\frac23} |k|^{\frac43} \mathcal{E}_{m,n;k}^{(\gamma)}, \qquad CK^{(\mu)}_{m,n;k} :=\nu^{\frac23} |k|^{\frac43} CK^{(\gamma)}_{m,n;k}, \qquad  \bold{D}^{(\mu)}_{m,n;k} =  \nu^{\frac23} |k|^{\frac43}\bold{D}^{(\gamma)}_{m,n;k}.
\end{align}
Next, we compute the contribution of the $\alpha$ estimate:  
\begin{align} \n
&\frac{\p_t}{2}  c_{\alpha}  \nu^{\frac23} |B_{m,n}|^2 |k|^{2m-\frac23}  \varphi^{2(n+m)} e^{2\delta_E \nu^{\frac 13}t}  \|  \pa_y \mathring{\omega}_{m,n;k} e^{W} \chi_{m+n}^2 \|_{L^2}^2 \\ \n
= & c_{\alpha} \nu^{\frac23} B_{m,n} \dot{B}_{m,n} |k|^{2m - \frac23} \varphi^{2(m+n)} e^{2\delta_E \nu^{\frac13}t} e^{2\delta_E \nu^{\frac13}t} \| \pa_y \mathring{\omega}_{m,n;k} e^W \chi_{m+n} \|_{L^2}^2 \\ \n
& + c_{\alpha} \nu^{\frac23} B_{m,n}^2 |k|^{2m- \frac23} (n+m) \varphi^{2n+2m-1} \dot{\varphi} e^{2\delta_E \nu^{\frac13}t} \| \pa_y \mathring{\omega}_{m,n;k} e^W \chi_{m+n} \|_{L^2}^2 \\ \n
& + c_{\alpha} \nu^{\frac23}B_{m,n}^2 |k|^{2m- \frac23} \varphi^{2(n+m)} \delta_E \nu^{\frac13} e^{2\delta_E \nu^{\frac13}t} \| \pa_y \mathring{\omega}_{m,n;k} e^W \chi_{m+n} \|_{L^2}^2 \\ \n
& + c_{\alpha} \nu^{\frac23}B_{m,n}^2 |k|^{2m- \frac23} \varphi^{2(n+m)}e^{2\delta_E \nu^{\frac13}t} \frac{\pa_t}{2} \| \pa_y \mathring{\omega}_{m,n;k} e^W \chi_{m+n} \|_{L^2}^2 \\ \n
\le & c_{\alpha} \nu^{\frac23} B_{m,n} \dot{B}_{m,n} |k|^{2m - \frac23} \varphi^{2(m+n)} e^{2\delta_E \nu^{\frac13}t} e^{2\delta_E \nu^{\frac13}t} \| \pa_y \mathring{\omega}_{m,n;k} e^W \chi_{m+n} \|_{L^2}^2 \\ \n
& + c_{\alpha} \nu^{\frac23} B_{m,n}^2 |k|^{2m- \frac23} (n+m) \varphi^{2n+2m-1} \dot{\varphi} e^{2\delta_E \nu^{\frac13}t} \| \pa_y \mathring{\omega}_{m,n;k} e^W \chi_{m+n} \|_{L^2}^2 \\ \n
& + c_{\alpha} \nu^{\frac23}B_{m,n}^2 |k|^{2m- \frac23} \varphi^{2(n+m)} \delta_E \nu^{\frac13} e^{2\delta_E \nu^{\frac13}t} \| \pa_y \mathring{\omega}_{m,n;k} e^W \chi_{m+n} \|_{L^2}^2 \\ \n
& - c_{\alpha} B_{m,n}^2 |k|^{2m- \frac23} \varphi^{2(n+m)}e^{2\delta_E \nu^{\frac13}t} \| \nu^{\frac56} \nabla_k \pa_y \mathring{\omega}_{m,n;k}  e^W \chi_{m+n} \|_{L^2}^2 + \frac{C c_{\alpha}}{(m+n)^{2\sigma}} \sum_{\iota \in \{ \gamma, \mu, \alpha \}} \mathcal{D}_{m,n-1;k}^{(\iota)} \\ \label{dalpha}
& + C  c_{\alpha}\nu^{\frac16}|k|^{\frac13}\sqrt{ \mathcal{E}^{(\gamma)}_{m,n;k} \mathcal{D}^{(\gamma)}_{m,n;k}},
\end{align}
where we have invoked our $\alpha$-estimate, \eqref{alpha:est:le}. We again introduce: 
\begin{align} \n
CK^{(\alpha)}_{m,n;k} := & - \nu^{\frac23} B_{m,n} \dot{B}_{m,n} |k|^{2m - \frac23} \varphi^{2(m+n)} e^{2\delta_E \nu^{\frac13}t} e^{2\delta_E \nu^{\frac13}t} \| \pa_y \mathring{\omega}_{m,n;k} e^W \chi_{m+n} \|_{L^2}^2 \\ 
& -  \nu^{\frac23} B_{m,n}^2 |k|^{2m- \frac23} (n+m) \varphi^{2n+2m-1} \dot{\varphi} e^{2\delta_E \nu^{\frac13}t} \| \pa_y \mathring{\omega}_{m,n;k} e^W \chi_{m+n} \|_{L^2}^2, \\ \n
\bold{D}^{(\alpha)}_{m,n;k} := &   B_{m,n}^2 |k|^{2m- \frac23} \varphi^{2(n+m)}e^{2\delta_E \nu^{\frac13}t} \| \nu^{\frac56} \nabla_k \pa_y \mathring{\omega}_{m,n;k}  e^W \chi_{m+n} \|_{L^2}^2\\
& - \nu^{\frac23}B_{m,n}^2 |k|^{2m- \frac23} \varphi^{2(n+m)} \delta_E \nu^{\frac13} e^{2\delta_E \nu^{\frac13}t} \| \pa_y \mathring{\omega}_{m,n;k} e^W \chi_{m+n} \|_{L^2}^2,
\end{align}
after which \eqref{dalpha} can be rewritten as 
\begin{align}  \label{alph:sum}
&\frac{\pa_t}{2} c_{\alpha}  \mathcal{E}^{(\alpha)}_{m,n;k}+ c_{\alpha} CK^{(\alpha)}_{m,n;k} +c_{\alpha}\bold{D}^{(\alpha)}_{m,n;k} \le C   c_{\alpha} \nu^{\frac16}|k|^{\frac13}\sqrt{ \mathcal{E}^{(\gamma)}_{m,n;k} \mathcal{D}^{(\gamma)}_{m,n;k}} + \frac{C c_{\alpha}}{(m+n)^{2\sigma}}  \sum_{\iota \in \{ \gamma, \mu, \alpha \}} \mathcal{D}_{n-1,m;k}^{(\iota)}.
\end{align}
Finally, we move to our $\beta$-estimate. We compute 
\begin{align} \n
&\frac{\pa_t}{2} c_{\beta} \nu^{\frac13}|k|^{2m-\frac43} \varphi^{2(n+m)} e^{2\delta_E \nu^{\frac13}t} |B_{m,n}|^2 \mathrm{Re} \langle i k \mathring{\omega}_{m,n;k} e^W, \pa_y \mathring{\omega}_{m,n;k} e^W \chi_{m+n}^2  \rangle \\ \n
= & c_{\beta} (n+m) \nu^{\frac13} |k|^{2m - \frac 43} \varphi^{2n+2m-1} \dot{\varphi} e^{2\delta_E \nu^{\frac13}t} |B_{m,n}|^2 \mathrm{Re} \langle i k \mathring{\omega}_{m,n;k} e^W, \pa_y \mathring{\omega}_{m,n;k} e^W \chi_{m+n}^2  \rangle \\ \n
 & + c_{\beta} \nu^{\frac13} |k|^{2m - \frac 43} \varphi^{2n+2m} e^{2\delta_E \nu^{\frac13}t} B_{m,n} \dot{B}_{m,n} \mathrm{Re} \langle i k \mathring{\omega}_{m,n;k} e^W, \pa_y \mathring{\omega}_{m,n;k} e^W \chi_{m+n}^2  \rangle \\ \n
 & +  c_{\beta}  \nu^{\frac13} |k|^{2m - \frac 43} \varphi^{2n+2m} \delta_E \nu^{\frac13} e^{2\delta_E \nu^{\frac13}t} |B_{m,n}|^2 \mathrm{Re} \langle i k \mathring{\omega}_{m,n;k} e^W, \pa_y \mathring{\omega}_{m,n;k} e^W \chi_{m+n}^2  \rangle \\ \n
 & +  c_{\beta} \nu^{\frac13} |k|^{2m - \frac 43} \varphi^{2n+2m}  e^{2\delta_E \nu^{\frac13}t} |B_{m,n}|^2 \frac{ \pa_t}{2} \mathrm{Re} \langle i k \mathring{\omega}_{m,n;k} e^W, \pa_y \mathring{\omega}_{m,n;k} e^W \chi_{m+n}^2  \rangle \\ \n
 \le &c_{\beta} (n+m) \nu^{\frac13} |k|^{2m - \frac 43} \varphi^{2n+2m-1} \dot{\varphi} e^{2\delta_E \nu^{\frac13}t} |B_{m,n}|^2 \mathrm{Re} \langle i k \mathring{\omega}_{m,n;k} e^W, \pa_y \mathring{\omega}_{m,n;k} e^W \chi_{m+n}^2  \rangle \\ \n
 & + c_{\beta} \nu^{\frac13} |k|^{2m - \frac 43} \varphi^{2n+2m} e^{2\delta_E \nu^{\frac13}t} B_{m,n} \dot{B}_{m,n} \mathrm{Re} \langle i k \mathring{\omega}_{m,n;k} e^W, \pa_y \mathring{\omega}_{m,n;k} e^W \chi_{m+n}^2  \rangle \\ \n
 & +  c_{\beta}  \nu^{\frac13} |k|^{2m - \frac 43} \varphi^{2n+2m} \delta_E \nu^{\frac13} e^{2\delta_E \nu^{\frac13}t} |B_{m,n}|^2 \mathrm{Re} \langle i k \mathring{\omega}_{m,n;k} e^W, \pa_y \mathring{\omega}_{m,n;k} e^W \chi_{m+n}^2  \rangle  \\ \label{bhg}
 &- c_{\beta} \nu^{\frac13}|k|^{2m+\frac23} e^{2\delta_E \nu^{\frac13}t} |B_{m,n}|^2 \| \mathring{\omega}_{m,n;k} e^W \chi_{m+n} \|_{L^2_y}^2 + C c_{\beta}\sqrt{ \mathcal{D}^{(\gamma)}_{m,n;k} \mathcal{D}^{(\alpha)}_{m,n;k}} + \frac{C c_{\beta}}{(m+n)^{2\sigma}} \sum_{\iota  \in \{ \gamma, \mu, \alpha \}} \mathcal{D}^{(\iota)}_{n-1,m;k},
\end{align}
where we have invoked \eqref{bgty}. We now define a few quantities: 
\begin{align*}
\mathcal{I}^{(\beta)}_{m,n;k} :=& \nu^{\frac13}|k|^{2m-\frac43} \varphi^{2(n+m)} e^{2\delta_E \nu^{\frac13}t} |B_{m,n}|^2 \mathrm{Re} \langle i k \mathring{\omega}_{m,n;k} e^W, \pa_y \mathring{\omega}_{m,n;k} e^W \chi_{m+n}^2  \rangle, \\ \n
CK^{(\beta)}_{m,n;k}(t) := & -  \nu^{\frac13} |k|^{2m - \frac 43} \varphi^{2n+2m-1} \dot{\varphi} e^{2\delta_E \nu^{\frac13}t} |B_{m,n}|^2 \mathrm{Re} \langle i k \mathring{\omega}_{m,n;k} e^W, \pa_y \mathring{\omega}_{m,n;k} e^W \chi_{m+n}^2  \rangle \\ 
 & -  (n+m) \nu^{\frac13} |k|^{2m - \frac 43} \varphi^{2n+2m} e^{2\delta_E \nu^{\frac13}t} B_{m,n} \dot{B}_{m,n} \mathrm{Re} \langle i k \mathring{\omega}_{m,n;k} e^W, \pa_y \mathring{\omega}_{m,n;k} e^W \chi_{m+n}^2  \rangle, \\ \n
\bold{D}^{(\beta)}_{n, m;k} := & \nu^{\frac13}|k|^{2m+\frac23} e^{2\delta_E \nu^{\frac13}t} |B_{m,n}|^2 \| \mathring{\omega}_{m,n;k} e^W \chi_{m+n} \|_{L^2_y}^2 \\
& - \nu^{\frac13} |k|^{2m - \frac 43} \varphi^{2n+2m} \delta_E \nu^{\frac13} e^{2\delta_E \nu^{\frac13}t} |B_{m,n}|^2 \mathrm{Re} \langle i k \mathring{\omega}_{m,n;k} e^W, \pa_y \mathring{\omega}_{m,n;k} e^W \chi_{m+n}^2  \rangle.
\end{align*}
Using these definitions, the estimate \eqref{bhg} reads 
\begin{align} \label{sum:beta}
\frac{\pa_t}{2} c_{\beta} \mathcal{I}^{(\beta)}_{n, m;k}  +  c_{\beta} CK^{(\beta)}_{m,n;k} +c_{\beta}\bold{D}^{(\beta)}_{m,n;k} \le C c_{\beta} \sqrt{\mathcal{D}^{(\gamma)}_{m,n;k}\mathcal{D}^{(\alpha)}_{m,n;k}} +  \frac{C c_{\beta}}{(m+n)^{2\sigma}}  \sum_{\iota \in \{ \gamma, \mu, \alpha \}} \mathcal{D}_{n-1,m;k}^{(\iota)}.
\end{align}
We now sum up the bounds \eqref{gam:sum}, \eqref{mu:sum}, \eqref{alph:sum}, \eqref{sum:beta} to get 
\begin{align} \n
\frac{\pa_t}{2} \bold{E}_{m,n;k} + CK_{m,n;k} + \bold{D}_{m,n;k} \le & C_{\ast} c_{\alpha} \nu^{\frac16}|k|^{\frac13} \sqrt{\mathcal{E}^{(\gamma)}_{m,n;k} \mathcal{D}^{(\gamma)}_{m,n;k} }+  C_{\ast} c_{\beta} \sqrt{\mathcal{D}^{(\gamma)}_{m,n;k}\mathcal{D}^{(\alpha)}_{m,n;k}} \\
&+ \frac{C_{\ast}}{(m+n)^{2\sigma}}\sum_{\iota \in \{ \gamma, \mu, \alpha \}} \mathcal{D}_{n-1,m;k}^{(\iota)},
\end{align}
for a universal constant $C_{\ast} < \infty$, and where we define 
\begin{align}
\bold{E}_{m,n;k} := &c_{\gamma}  \mathcal{E}^{(\gamma)}_{m,n;k}+  c_{\mu}  \mathcal{E}^{(\mu)}_{m,n;k} + c_{\alpha} \mathcal{E}^{(\alpha)}_{m,n;k} +  c_{\beta}\mathcal{I}^{(\beta)}_{m,n;k}\\
\bold{D}_{m,n;k} := &c_{\gamma} \bold{D}^{(\gamma)}_{m,n;k}+  c_{\mu} \bold{D}^{(\mu)}_{m,n;k} + c_{\alpha} \bold{D}^{(\alpha)}_{m,n;k} +  c_{\beta}\bold{D}^{(\beta)}_{m,n;k}, \\
CK_{m,n;k} := &c_{\gamma} CK^{(\gamma)}_{m,n;k}+  c_{\mu} CK^{(\mu)}_{m,n;k} + c_{\alpha} CK^{(\alpha)}_{m,n;k} +  c_{\beta}CK^{(\beta)}_{m,n;k}.
\end{align}
We will now choose the weights $c_{\alpha}, c_{\gamma}, c_{\beta}$ so as to satisfy the following constraints: 
\begin{align}
&C_{\ast} c_{\alpha} c_{\gamma}^{\frac12} \le \frac{c_{\gamma}}{100}, \qquad \frac{C_{\ast} c_{\alpha}}{4 c_{\gamma}^{\frac12}} \le \frac{c_{\beta}}{100}, \qquad C_{\ast} c_\beta c_\gamma^{\frac12} \le \frac{c_{\gamma}}{100}, \qquad \frac{C_{\ast} c_{\beta}}{4 c_{\gamma}^{\frac12}} \le \frac{c_{\alpha}}{100}.
\end{align}
It is clear that such a choice exists, for instance by choosing $c_{\alpha} = c_{\beta} = 1$, and $c_{\gamma}$ large relative only to $C_{\ast}$. An application of the standard Young's inequality for products then produces
\begin{align} \n
\frac{\pa_t}{2} \bold{E}_{m,n;k} + CK_{m,n;k} + \bold{D}_{m,n;k} \le & C_{\ast} c_{\alpha} ( \frac{1}{4 c_{\gamma}^{\frac12}} \nu^{\frac13} |k|^{\frac23} \mathcal{E}_{m,n;k}^{(\gamma)} + c_{\gamma}^{\frac12}  \mathcal{D}_{m,n;k}^{(\gamma)}  ) \\ \n
& + C_{\ast} c_{\beta} \Big( c_{\gamma}^{\frac12} \mathcal{D}_{m,n;k}^{(\gamma)} + \frac{1}{4 c_{\gamma}^{\frac12}} \mathcal{D}^{(\alpha)}_{m,n;k} \Big) +  \frac{C_{\ast}}{(m+n)^{2\sigma}}\sum_{\iota \in \{ \gamma, \mu, \alpha \}} \mathcal{D}_{n-1,m;k}^{(\iota)} \\ \n
\le & \frac{c_{\beta}}{100} \nu^{\frac13} |k|^{\frac23} \mathcal{E}_{m,n;k}^{(\gamma)} + \frac{c_{\gamma}}{50} \mathcal{D}^{(\gamma)}_{m,n;k} + \frac{c_{\alpha}}{100} \mathcal{D}^{(\alpha)}_{m,n;k} \\ \label{mhg1}
& + \frac{C_{\ast}}{(m+n)^{2\sigma}}\sum_{\iota \in \{ \gamma, \mu, \alpha \}} \mathcal{D}_{n-1,m;k}^{(\iota)}.
\end{align}
We now want to prove several bounds on the quantities appearing on the left-hand side of \eqref{mhg1}. The most involved bounds are on $\bold{D}_{m,n;k}$, where we need to split out terms we assign for ``damping" and terms we assign for ``diffusion". In particular, we establish
\begin{align}
\bold{D}_{m,n;k} \ge c_{\beta} \nu^{\frac13} |k|^{\frac23} \mathcal{E}_{m,n;k} + \frac{c_{\gamma}}{2} \mathcal{D}^{(\gamma)}_{m,n;k} + c_{\alpha} \mathcal{D}^{(\alpha)}_{m,n;k} + c_{\mu}  \mathcal{D}^{(\mu)}_{m,n;k}.
\end{align}
Inserting this inequality into \eqref{mhg1}, absorbing the first three terms on the right-hand side of \eqref{mhg1} into the left-hand side, and letting 
\begin{align}
c_{\ast} := \frac{1}{500} \min \{c_{\gamma}, c_{\beta}, c_{\mu}, c_{\alpha} \},
\end{align}
we obtain the inequality 
\begin{align} \n
\frac{\pa_t}{2} \bold{E}_{m,n;k} + CK_{m,n;k} + c_{\ast} \nu^{\frac13} |k|^{\frac23} \mathcal{E}_{m,n;k} + c_{\ast} \mathcal{D}_{m,n;k} \lesssim & \frac{1}{(m+n)^{2\sigma}}\sum_{\iota \in \{ \gamma, \mu, \alpha \}} \mathcal{D}_{n-1,m;k}^{(\iota)}.
\end{align}
We can now use the definitions \eqref{defn:E:1}, \eqref{defn:F:1} directly to see that there exists a $c^{\ast}$ such that $\mathcal{E}_{m,n;k} \ge c^{\ast} \bold{E}_{m,n;k}$, which then gives
\begin{align} \label{ult:en}
\frac{\pa_t}{2} \bold{E}_{m,n;k} + CK_{m,n;k} + c_{\ast} c^{\ast} \nu^{\frac13} |k|^{\frac23} \bold{E}_{m,n;k} + c_{\ast} \mathcal{D}_{m,n;k} \lesssim & \frac{\lambda^{2s}}{(m+n)^{2\sigma}}\sum_{\iota \in \{ \gamma, \mu, \alpha \}} \mathcal{D}_{m,n-1;k}^{(\iota)}\myr{\mathbbm{1}_{n\geq1}+\mathcal{D}_{m-1,0;k}^{(\iota)}\mathbbm{1}_{m\geq1,n=0}?}.
\end{align}  
Upon noting that $CK_{m,n;k} \ge 0$, we have proven \eqref{en:in:b}. 
\fi
\end{proof}
\subsection{Technical Lemmas: Auxiliary Bounds for Energy Estimates}\label{sec:techlem}

We first record the following important inequality relating the losses we incur from our Gevrey cut-off scale, $\chi_{m+n}$ \eqref{chi}, with the Gevrey weights $B_{m,n}$:
\begin{lemma}  Fix $\ell\in\{1, 2\}$.  The weights $\{B_{m,n} \}_{ m+n\geq \ell}$ \eqref{Bweight} satisfy the following bound, 
\begin{align} \label{gevbd1212}
(m+n)^{\ell(1 + \sigma)}B_{m,n}  \leq  & C(s) \lambda^{  \ell s}(m+n)^{- \ell \sigma_{\ast}} B_{m-\ell_1,n-\ell_2},\quad \ell_1+\ell_2=\ell.
\end{align}
Here the parameters $\sigma,\, \sigma_\ast$ are defined in \eqref{s:prime} and the constant $C(s)\leq 4$ for $s\in [1,2]$. Moreover, the $(\ell_1,\ell_2)$ is chosen such that $m-\ell_1\geq 0$, and $\, n-\ell_2\geq 0. $ 
\end{lemma}\siming{HS: Double check the $m+n$ small case. Check: Here one will not be able to select the $\ell_1,\ell_2$ as in the lemma. But in this case, the coefficients are simple. }
\begin{remark}
For $m+n<\ell$, the coefficients $B_{m,n}$ are of order $1$ and one can estimate the coefficient directly.
\end{remark} 
\begin{proof} Recalling the $B_{m,n}$-definition \eqref{Bweight} and the relation $ \sigma+ \sigma_\ast= s-1$, which is a consequence of the definition \eqref{s:prime}, we have that
\begin{align*} \n
(m+n)^{1 +  \sigma} B_{m,n}  = & (m+n)^{1 + \sigma} \lf( \frac{\lambda^{m+n}}{(m+n)!} \rg)^{s} =(m+n)^{{s}-\sigma_{\ast}} \lf( \frac{\lambda^{m+n}}{(m+n)!} \rg)^{s} \\ 
= & (m+n)^{-\sigma_{\ast}}\lambda^{s} \lf( \frac{\lambda^{m+n-1}}{(m+n-1)!} \rg)^{s} =(m+n)^{-\sigma_{\ast}} \lambda^{ {s}} B_{m-\ell_1,n-\ell_2}.
\end{align*}
Here $\ell_1+\ell_2=1.$ 
By invoking the estimate obtained again, we have that for $m+n\geq 2$,
\begin{align*}\n 
(m+n)^{2+2\sigma}B_{m,n} &\leq \lambda^{s}(m+n)^{1+\sigma}(m+n)^{-\sigma_\ast}B_{m-\ell_1',n-\ell_2'}\mathbbm{1}_{\ell_1'+\ell_2'=1}\\ 
&\leq \lambda^{2s}\lf(\frac{m+n}{m+n-1}\rg)^{1+\sigma+\sigma_\ast} (m+n)^{-2\sigma_\ast}B_{m-\ell_1,n-\ell_2}\mathbbm{1}_{\ell_1+\ell_2=2}\\
&\leq \frac{C(s)\lambda^{2s}}{(m+n)^{2\sigma_\ast}}B_{m-\ell_1,n-\ell_2}\mathbbm{1}_{\ell_1+\ell_2=2}.
\end{align*}
This concludes the proof of the lemma. 
\end{proof}

\begin{lemma}
Recall the function $W$ from \eqref{defndW}, 
\begin{align}
W(t,y)=\frac{(|y|-1/4-L\ep\arctan(t))_+^2}{K\nu(1+ t)}.
\end{align}  The function $W$ satisfies the following estimate:
\begin{align}\label{W_prop}
    \pa_t W+\frac{1}{8}K\nu|\pa_y W|^2  \leq& -\frac{(|y|-1/4-L\ep\arctan(t))_+^2}{2K\nu(1+t)^2}  -\frac{2L\epsilon(|y|-1/4-L\ep\arctan(t))_+}{K\nu(1+t)(1+t^2)} \leq  0.
\end{align}
Here $K,L\geq 1$ are large constants.  Alternatively, there exists a constant $C$, independent of the parameters $K, L$ such that 
\begin{subequations}
\begin{align} \label{wdot:est:a}
\nu |\p_y W|^2 \le & - \frac{C}{K} \pa_t{W}; \\  \label{wdot:est:b}
\frac{\eps}{(1 + t)^2}|\p_y W| \le & - \frac{C}{L}\pa_t{W}.
\end{align}
\end{subequations} 
\end{lemma}
\begin{proof}
We take the time derivative and spatial derivative to obtain that
\begin{align*}
\pa_t   W +\frac{K}{8}\nu |\pa_y W|^2 =&-\frac{ (|y|-1/4-L\ep\arctan(t))_+^2}{K\nu(1+t)^2} -\frac{2L\epsilon(|y|-1/4-L\ep\arctan(t))_+}{K\nu(1+t)(1+t^2)} \\
&+ K\nu \frac{(|y|-1/4-L\ep\arctan(t))_+^2}{2K^2\nu^2(1+t)^2} \\
\leq &-\frac{(|y|-1/4-L\ep\arctan(t))_+^2}{2K\nu(1+t)^2} -\frac{2L\epsilon(|y|-1/4-L\ep\arctan(t))_+}{K\nu(1+t)(1+t^2)} \leq 0.
\end{align*}
\end{proof}

\begin{lemma}The following two estimates hold for arbitrary $M\in \mathbb{N}$,
\begin{align}
\sum_{m+n=0}^M{\bf a}_{m,n}^2\nu\|\pa_y(\chi_{m+n}\mathring{\omega}_{m,n;k})e^W\|_{L^2}^2\lesssim& \sum_{m+n=0}^M\mathcal{D}_{m,n;k}^{(\gamma)};\label{Dga_switch}\\
\sum_{m+n=0}^M{\bf a}_{m,n}^2\nu^2\|\na_k\pa_y(\chi_{m+n}\mathring{\omega}_{m,n;k})e^W\|_{L^2}^2\lesssim &\sum_{m+n=0}^M\mathcal{D}_{m,n;k}^{(\al)}+\sum_{m+n=0}^M\mathcal{D}_{m,n;k}^{(\mu)}.\label{Dal_mu_switch}
\end{align}
\end{lemma}
\begin{proof}We start by proving \eqref{Dga_switch}. We
 recall that $\chi_0\equiv1$, ${\bf a}_{m,n}=B_{m,n}\varphi^{n+1}$ and the $\chi_{m+n}'$ estimate \eqref{chi:prop:3} to obtain 
\begin{align*}
\sum_{m+n=0}^M&{\bf a}_{m,n}^2\nu\|\pa_y(\chi_{m+n}\mathring{\omega}_{m,n;k})e^W\|_{L^2}^2\\
\lesssim&\sum_{m+n=0}^M{\bf a}_{m,n}^2\nu\|\chi_{m+n}(\pa_y\mathring{\omega}_{m,n;k})e^W \|_{L^2}^2+\sum_{m+n=0}^M{\bf a}_{m,n}^2\nu\| [\pa_y,\chi_{m+n}]\mathring{\omega}_{m,n;k}e^W\|_{L^2}^2\\
\lesssim&\sum_{m+n=0}^M\mathcal{D}_{m,n;k}^{(\gamma)}+\sum_{m+n=0}^M\mathbbm{1}_{n\geq 1}{\bf a}_{m,n}^2(m+n)^{2+2\sigma}\nu\| \chi_{m+n-1}(\pa_v+ikt)\omega_{m,n-1;k}e^W\|_{L^2}^2\\
&+\sum_{m=1}^M {\bf a}_{m,0;k}^2m^{2+2\sigma}\nu\| \chi_{m-1}|k|\omega_{m-1,0;k}e^W\|_{L^2}^2.
\end{align*}
Here we apply the definition $\Gamma=\pa_v+ikt$ in the last inequality. Now, we invoke a technical relation of $B_{m,n}$ \eqref{gevbd1212} and end up with the result: 
\begin{align*}
\sum_{m+n=0}^M {\bf a}_{m,n}^2\nu\|\pa_y(\chi_{m+n}\mathring{\omega}_{m,n;k})e^W\|_{L^2}^2\lesssim \sum_{m+n=0}^M\mathcal{D}_{m,n;k}^{(\gamma)}+\lambda^{2s}\sum_{m+n=0}^{M-1}(m+n)^{-2\sigma_\ast}\mathcal{D}_{m,n;k}^{(\gamma)}
\end{align*} proving \eqref{Dga_switch}. 

Next we focus on \eqref{Dal_mu_switch}. First of all, we observe that the $\sum {\bf a}_{m,n}^2\nu\|k\pa_y (\chi_{m+n}\mathring{\omega}_{m,n;k})e^W\|_{L^2}^2$ can be estimated in the same fashion as in \eqref{Dga_switch}. \siming{(Check!)} Now we focus on the following 
\begin{align}\n
\sum_{m+n=0}^M&{\bf a}_{m,n}^2\nu^2\|\pa_{yy}(\chi_{m+n}\mathring{\omega}_{m,n;k})e^W\|_{L^2}^2\\ \n 
\lesssim&\sum_{m+n=0}^M{\bf a}_{m,n}^2\nu^2\|\chi_{m+n}(\pa_{yy}\mathring{\omega}_{m,n;k})e^W \|_{L^2}^2+\sum_{m+n=0}^M{\bf a}_{m,n}^2\nu^2\|[ \chi_{m+n},\pa_{yy}]\mathring{\omega}_{m,n;k}e^W\|_{L^2}^2\\ \n
\lesssim&\sum_{m+n=0}^M\mathcal{D}_{m,n;k}^{(\al)}+\sum_{m+n=2}^M{\bf a}_{m,n}^2(m+n)^{4+4\sigma}\nu^2\| \chi_{m+n-2}\mathring{\omega}_{m,n;k}e^W\|_{L^2}^2\\
&+\sum_{m+n=2}^M{\bf a}_{m,n}^2(m+n)^{2+2\sigma}\nu^2\| \chi_{m+n-1}\pa_y\mathring{\omega}_{m,n;k}e^W\|_{L^2}^2+\sum_{m+n=1}{\bf a}_{m,n}^2\nu^2\|[\chi_{m+n},\pa_{yy}]\mathring{\omega}_{m,n;k}e^W\|_{L^2}^2 \n \\
=:&\sum_{j=1}^4T_j.\label{switch_T} 
\end{align}
The $T_1$ term is consistent with the result. We focus on the $T_2$-term. By invoking the definition ${\bf a}_{m,n}=B_{m,n}\varphi^{n+1}$, the property \eqref{ineq:varphi}, the bound $\|\vyn\|_{L_{t,y}^\infty}+\|v_{y}\|_{L_t^\infty W_y^{1,\infty}}\leq C,$ and the technical relation \eqref{gevbd1212}, we obtain that there exists an $\ell\in\{0,1,2\}$ such that the following bound holds 
\begin{align*}
T_2\lesssim&\sum_{m+n=0}^M{\bf a}_{m,n}^2(m+n)^{4+4\sigma}\mathbbm{1}_{\{m-\ell\geq0,\, n-2+\ell\geq 0\}}\nu^2\| \chi_{m+n-2}(\pa_v+ikt)^{2-\ell}|k|^\ell\omega_{m-\ell,n-2+\ell;k}e^W\|_{L^2}^2\\
\lesssim&\lambda^{4s}\sum_{m+n=0}^M{\bf a}_{m-\ell,n-2+\ell;k}^2\mathbbm{1}_{\{m-\ell\geq0,\, n-2+\ell\geq 0\}}(m+n)^{-4\sigma_\ast}\varphi^{4-2\ell}\nu^2\\
& \times \| \chi_{m+n-2}(\pa_v+ikt)^{2-\ell}|k|^\ell\omega_{m-\ell,n-2+\ell;k}e^W\|_{L^2}^2\\
\lesssim&\lambda^{4s}\sum_{m+n=0}^{M-2}{\bf a}_{m,n}^2   \nu^2\| \chi_{m+n}\na_k^2\mathring{\omega}_{m,n;k}e^W\|_{L^2}^2\lesssim\lambda^{4s}\sum_{m+n=0}^{M-2}(\mathcal{D}_{m,n;k}^{(\al)}+\mathcal{D}_{m,n;k}^{(\mu)}).
\end{align*} 
The estimate for the $T_3$-term is similar and we obtain that \siming{(Check!)} 
\begin{align*}
T_3\lesssim\lambda^{2s}\sum_{m+n=0}^{M-1}(\mathcal{D}_{m,n;k}^{(\al)}+\mathcal{D}_{m,n;k}^{(\mu)}).
\end{align*}
For the $T_4$-term, we have that 
\begin{align*}
T_4\lesssim& \sum_{m+n=1}{\bf a}_{m,n}^2\nu^2\|[\chi_1,\pa_{yy}]\mathring{\omega}_{m,n;k}e^W\|_2^2=\sum_{m+n=1}{\bf a}_{m,n}^2\nu^2(\|\pa_y\mathring{\omega}_{m,n;k}e^W\|_2^2+\|\mathring{\omega}_{m,n;k}e^W\|_{2}^2)\\
\lesssim& \lambda^{2s} (\mathcal{D}_{0,0;k}^{(\al)}+\mathcal{D}_{0,0;k}^{(\mu)}).
\end{align*}
Combining the estimates above with the decomposition \eqref{switch_T}, we obtain \eqref{Dal_mu_switch}. \ifx
\begin{align*}
\lesssim&\sum_{m+n=0}^M\mathcal{D}_{m,n;k}^{(\al)}+\sum_{m+n=0}^M{\bf a}_{m,n}^2(m+n)^{4+4\sigma}\nu^2\| \chi_{m+n-2}(\pa_v+ikt)^2\omega_{m,n-2;k}e^W\|_{L^2}^2\\
&+\sum_{m+n=0}^M{\bf a}_{m,n}^2(m+n)^{2+2\sigma}\nu^2\| \chi_{m+n-1}\pa_y(\pa_v+ikt)\omega_{m,n-1;k}e^W\|_{L^2}^2\\
&+\sum_{m=0}^M {\bf a}_{m,0;k}^2m^{4+4\sigma}\nu^2\| \chi_{m-2}|k|^2\omega_{m-2,0;k}e^W\|_{L^2}^2\\\lesssim&\sum_{m+n=0}^M\mathcal{D}_{m,n;k}^{(\al)}+\sum_{m+n=0}^{M-1}\mathcal{D}_{m,n;k}^{(\mu)}. 
\end{align*}\fi
\end{proof}

\ifx
\begin{proof}Here we prove the following combinatorial relation
$\binom{a+b}{c+d}\geq\binom{a}{c}\binom{b}{d}$.
We consider the expression
\begin{align*}
(1+x)^{a+b}=(1+x)^a(1+x)^b.\label{rlt}
\end{align*}
Now we expand the left hand side with the binomial formula
\begin{align}
(1+x)^{a+b}=\sum_{m=0}^{a+b}\binom{a+b}{m}x^m.
\end{align}
Now we expand the right hand side of \eqref{rlt}:
\begin{align}
(1+x)^a(1+x)^b=\lf(\sum_{c=0}^a\binom{a}{c}x^c\rg)\lf(\sum_{d=0}^b\binom{b}{d} x^d\rg)=\sum_{m=0}^{a+b}x^{m}\sum_{c=0}^m\binom{a}{c}\binom{b}{m-c} . 
\end{align}
By comparing the coefficients of $x^m$, we have that 
\begin{align}
\binom{a+b}{m}=\sum_{c=0}^m\binom{a}{c}\binom{b}{m-c}. 
\end{align}
Now choose $m=c+d$ and fix an arbitrary $c\in\{0,...,a+b\}$, we have that 
\begin{align}
\binom{a+b}{c+d}\geq\binom{a}{c}\binom{b}{d}.
\end{align}
\end{proof}
\fi
We record here a lemma which establishes that the top order boundary contribution from our $\alpha$ estimate, vanishes. 
\begin{lemma} \label{lemma:BC} For all $n \in \mathbb{N}$, $k \in \mathbb{Z}$, we have  \begin{align} 
\mathrm{Re} \pa_y \mathring{\omega}_{m,n;k} \overline{\pa_y^2 \mathring{\omega}_{m,n;k}}|_{y = \pm 1} = 0,\quad n\neq 1.\label{id:BC}
\end{align}
For $n=1$, the boundary contribution can be estimated as follows
\begin{align}\n
\bold{a}_{m,1}^2&\nu^{2}\big|\mathrm{Re} \pa_y(\mathring{\omega}_{m,1;k})\overline{\pa_{y}^2(\mathring{\omega}_{m,1;k}} )e^{2W}\chi_{m+1}^2\big|_{y=\pm 1}\big|\lesssim 
\bold{a}_{m,1}^2\nu^{2}|\pa_y\mathring\omega_{m,0;k} |^2 e^{2 W}\bigg|_{y=\pm 1}\\
\lesssim &\frac{\lambda^{2s}}{(m+1)^{2\sigma_\ast}}\mathcal{D}_{m,0;k}^{(\al)}+\frac{\lambda^{2s}}{(m+1)^{2\sigma_\ast}}\mathcal{CK}_{m,0;k}^{(\al, W)}.\label{BC_n=1}
\end{align}
Here,  the $\mathcal{D}$-term and the $\mathcal{CK}$-term are defined in \eqref{defn:D:1}, \eqref{CK_W_al_1}.
\end{lemma}\begin{proof}
For $n>1$, the result is a direct consequence of the fact that 
\begin{align}
\pa_y \mathring{\omega}_{m,n;k}\bigg|_{y=\pm 1}=\pa_y (q^n\Gamma_k^n \omega_k)\bigg|_{y=\pm 1}=\lf((n-1)q^{n-1}q'\Gamma_k^n\omega_k+q^n\pa_y\Gamma^{n}\omega_k\rg) \bigg|_{y=\pm 1}=0. 
\end{align}
For $n=0$, we recall that the boundary condition $\omega_k(t, y=\pm1)=0$ implies $\pa_{yy}\omega_k(t,y=\pm 1)=0$ and guarantees that the left hand side of \eqref{id:BC} is zero.
The $n=1$ case is the most delicate one, we invoke the properties $q(\pm 1)=q''(\pm 1)=0$, $\omega_k(t, y=\pm1)=0$ and $\pa_{yy}\omega_k(t,y=\pm 1)=0$ to obtain that
\begin{align*}
\mathrm{Re} &\pa_y\mathring{\omega}_{m,1;k}\overline{\pa_{yy}\mathring{\omega}_{m,1;k}}e^{2W}\chi_{m+1}^2\bigg|_{y=\pm1}\\ 
=&2\mathrm{Re} q'\lf(v_y^{-1}\pa_y+ikt\rg)(|k|^m\omega_k)\overline{q'\pa_y\lf(v_y^{-1}\pa_y+ikt\rg)(|k|^m\omega_k)}e^{2W}\chi_{m+1}^2\bigg|_{y=\pm1}\\ 
=&2\mathrm{Re} q'v_y^{-1}\pa_y(|k|^m\omega_k)\overline{q'\lf(\frac{-v_{yy}}{v_y^2}\pa_y(|k|^m\omega_k)+v_y^{-1}\pa_{yy}(|k|^m\omega_k)+ikt\pa_y(|k|^m\omega_k)\rg)}e^{2W}\chi_{m+1}^2\bigg|_{y=\pm1}\\
=&-2\frac{v_{yy}}{v_y^3}|q'|^2|\pa_y(|k|^m\omega_k)|^2e^{2W}\chi_{m+1}^2\bigg|_{y=\pm1}=-2\frac{v_{yy}}{v_y^3}|q'|^2|\pa_y\mathring\omega_{m,0;k}|^2e^{2W}\chi_{m+1}^2\bigg|_{y=\pm1}.
\end{align*}  
Here we have used the fact that $\mathrm{Re}\lf(ikt |\pa_y(|k|^m\omega_k)|^2\rg)=0$ in the last line. Combining this relation and the fact that $\|\vyi\|_\infty+\|v_{yy}\|_\infty\leq C$ yields the first inequality in  \eqref{BC_n=1}. 

Now we develop the second inequality in \eqref{BC_n=1}. 
We use the following trace theorem, which is a natural consequence of integration by parts:
\begin{align*}
\bold{a}_{m,1}^2&\nu^{2} |\pa_y\mathring{\omega}_{m,0;k}|^2 e^{2 W}\chi_{m+1}^2\bigg|_{y=\pm 1}\\
\lesssim & \bold{a}_{m,1}^2\nu^{2} \|\pa_y\mathring{\omega}_{m,0;k}e^{W}\chi_{m+1}\|_{L^2}^2+\bold{a}_{m,1}^2\nu^{2} \|\pa_y(\pa_y\mathring{\omega}_{m,0;k}e^{W}\chi_{m+1})\|_{L^2}^2. 
\end{align*}
We further expand the second term with the property \eqref{chi:prop:3} as follows
\begin{align*}
\bold{a}_{m,1}^2&\nu^{2} |\pa_y\mathring{\omega}_{m,0;k}|^2 e^{2 W}\chi_{m+1}^2\bigg|_{y=\pm 1}\\
\lesssim &\bold{a}_{m,1}^2\nu^{2} \| \pa_y^2\mathring{\omega}_{m,0;k}e^{W}\chi_{m+1}\|_{L^2}^2+\bold{a}_{m,1}^2\nu^{2}(m+1)^{2+2\sigma}\| \pa_y\mathring{\omega}_{m,0;k}e^{W}\|_{L^2(\text{support}\chi_{m+1})}^2\\
&+\bold{a}_{m,1}^2\nu\|\pa_y \mathring{\omega}_{m,0;k} |\nu^{1/2}\pa_y W| e^W\chi_m\|_2^2=:T_1+T_2+T_3.
\end{align*}
Now we distinguish between two cases: $k\neq 0$ and $k=0$. 
If $k\neq 0$, the first two terms can be bounded by $\mathcal{D}^{(\al)}$. Then the relation \eqref{W_prop} yields that
\begin{align*}
\bold{a}_{m,1}^2\nu^{2} |\pa_y\mathring{\omega}_{m,0;k}|^2 e^{2 W}\chi_{m+1}^2\bigg|_{y=\pm 1}\lesssim&\frac{\lambda^{2s}}{(m+1)^{2\sigma_\ast}}\mathcal{D}_{m,0;k}^{(\al)}+\bold{a}_{m,1}^2\nu\|\pa_y\mathring{\omega}_{m,0;k} \sqrt{-\pa_t W} e^{W}\chi_m\|_2^2\\
\lesssim&{\frac{\lambda^{2s}}{(m+1)^{2\sigma_\ast}}\mathcal{D}_{m,0;k}^{(\al)}+\frac{\lambda^{2s}}{(m+1)^{2\sigma+2\sigma_\ast}}\mathcal{CK}_{m,0;k}^{(\al, W)}}.
\end{align*} 
On the other hand, if $k=0$, then $m=0$ and we observe that by \eqref{W_prop},
\begin{align*}
\sqrt{-\pa_t W}\geq \frac{1}{CK\nu(1+t)^2}\geq \frac{1}{CK\nu^{1/3-2\eta}},\quad y\in \text{support}\chi_1,\quad t\leq \nu^{-1/3-\eta}.
\end{align*}
Hence,
\begin{align*}
T_2\lesssim \bold{a}_{0,1}^2\nu^{2}\| \pa_y\mathring{\omega}_{0,0;0}\sqrt{-\pa_t W}e^{W}\|_{L^2(\text{support}\chi_{1})}^2\lesssim \nu\lambda^{2s}\mathcal{CK}_{0,0;0}^{(\al,W)}.
\end{align*}
The $T_1,\ T_3$ terms are estimated in an identical fashion as before. Combining the estimates in the  cases $k\neq 0$ and $k=0$, we obtain the second inequality in  \eqref{BC_n=1}. Hence the proof is concluded. 
\end{proof}
\ifx
\sameer{
\begin{proof} We first treat the case $n = 0$. In this case, the identity \eqref{id:BC} will follow upon noting that $\pa_y^2 \omega_{k,0}|_{y = \pm 1} = 0$. To see this, we evaluate \eqref{M1a} at $y = \pm 1$ which gives 
\begin{align} \n
\nu \pa_y^2 \omega_k|_{y = \pm 1} = & \pa_t \omega_k|_{y = \pm 1} + i k(\pm 1 + U^x_0(t, \pm 1)) \omega_k|_{y= \pm 1} + |k|^2 \omega_k|_{y = \pm 1} \\ \n
&+ (U^x_{\neq 0}\pa_x \omega)_k|_{y = 0} + (U^y \pa_y \omega)_k|_{y = \pm 1}  = 0,
\end{align}
upon invoking $\omega_k|_{y = \pm 1} = 0$, \eqref{M1c}, and $U^y|_{y = \pm 1} = 0$, \eqref{}. \sameer{We need to write this condition somewhere \dots} 
We now treat by induction the cases when $n \ge 1$. Fix $n \ge 1$. Our inductive hypothesis will be 
\begin{align} \label{ind:BC:1}
\mathrm{Re} \pa_y \omega_{k,m} \pa_y^2 \omega_{k,m}|_{y = \pm 1} = 0 \text{ for } 0 \le m \le n-1. 
\end{align}
The following identity follows from the definitions of our vector-fields,
\begin{align*}
\pa_y \mathring{\omega}_{m,n;k} \pa_y^2 \mathring{\omega}_{m,n;k} = & \pa_y ( q(y) \Gamma \omega_{k,n-1}) \pa_y^2 (q(y) \Gamma \omega_{k,n-1}) \\
= & 2|q'(y)|^2 \Gamma_k \omega_{k,n-1} \pa_y \Gamma_k \omega_{k,n-1} + q'(y) q''(y) \Gamma_k \omega_{k,n-1} \Gamma_k \omega_{k,n-1} \\ 
&+ q(y) q'(y) \Gamma_k \omega_{k,n-1} \pa_y^2 \Gamma_k \omega_{k,n-1}  + q(y) q''(y) \pa_y \Gamma_k \omega_{k,n-1} \Gamma_k \omega_{k,n-1} \\
&+ 2 q(y) q'(y) \pa_y \Gamma_k \omega_{k,n-1} \pa_y \Gamma_k \omega_{k,n-1} + q(y)^2 \pa_y \Gamma_k \omega_{k,n-1} \pa_y^2 \Gamma_k \omega_{k,n-1} 
\end{align*}
Evaluating the above identity at $y = \pm 1$ and using $q(\pm 1) = q''(\pm 1) = 0$, we see only the first term survives. Invoking now the definition of our vector-field $\Gamma_k$, we obtain 
\begin{align*}
\frac{1}{2|q'(\pm 1)|^2}\pa_y \mathring{\omega}_{m,n;k} \pa_y^2 \mathring{\omega}_{m,n;k}|_{y = \pm 1} = & \Gamma_k \omega_{k,n-1} \pa_y \Gamma_k \omega_{k,n-1}|_{y = \pm 1} \\
= &  (\frac{\pa_y}{v_y} + i k t) \omega_{k,n-1} \pa_y (\frac{\pa_y}{v_y} + i k t) \omega_{k,n-1}|_{y = \pm 1} \\
= & - \frac{v''}{v_y^3} |\pa_y \omega_{k,n-1}|^2 |_{y = \pm 1}+ \frac{1}{v_y^2} \pa_y \omega_{k,n-1} \pa_y^2 \omega_{k,n-1}|_{y = \pm 1} + \frac{ikt}{v_y} |\pa_y \omega_{k,n-1}|^2|_{y = \pm 1} \\
&- \frac{v''}{v_y^2} i kt \omega_{k,n-1} \pa_y \omega_{k,n-1}|_{y = \pm 1} + \frac{ikt}{v_y} \omega_{k,n-1}\pa_y^2 \omega_{k,n-1} |_{y = \pm 1}\\
&- |kt|^2 \omega_{k,n-1} \pa_y \omega_{k,n-1}|_{y = \pm 1}.
\end{align*}
We now take the real part which eliminates the third term above. Subsequently, we use that $v''(\pm 1) = 0$, \eqref{}, (\sameer{how do we enforce this? localization of $H$?}), which eliminates the first and fourth terms. We then use that $\mathring{\omega}_{m,n;k} = 0$ for all $n \in \mathbb{N}$ which eliminates the fifth and sixth terms above. This then leaves only the second term, for which we invoke the inductive hypothesis, \eqref{ind:BC:1}. The lemma is proven. 
\end{proof} }
{\color{blue} We note that $\pa_{yy}\mathring{\omega}_{m,n;k}=\pa_y^2(q^n)\lf(v_y^{-1}\pa_y+ikt\rg)\omega_k+ 2(\pa_y q^n)\pa_y\lf(v_y^{-1}\pa_y+ikt\rg)\omega_k+2q^n\pa_{yy}\lf(v_y^{-1}\pa_y+ikt\rg)\omega_k.$ The first level is nontrivial.
\begin{align}
\mathrm{Re} \pa_y\omega_{k,1}\overline{\pa_{yy}\omega_{k,1}}\big|_{y=\pm1}=&2\mathrm{Re} q'\lf(v_y^{-1}\pa_y+ikt\rg)\omega_k\overline{q'\pa_y\lf(v_y^{-1}\pa_y+ikt\rg)\omega_k}\big|_{y=\pm1}\\
=&2\mathrm{Re} q'v_y^{-1}\pa_y\omega_k\overline{q'\lf(\frac{-v_{yy}}{v_y^2}\pa_y\omega_k+v_y^{-1}\pa_{yy}\omega_k+ikt\pa_y\omega_k\rg)}\bigg|_{y=\pm1}=-2\frac{v_{yy}}{v_y^3}|q'|^2|\pa_y\omega_k|^2\bigg|_{y=\pm1}.
\end{align}
It seems to me that $\eta_{k,1}=q(\pa_y+iv_y kt)\omega_k$ is better. 
\begin{align}
\mathrm{Re} \pa_y\eta_{k,1}\overline{\pa_{yy}\eta_{k,1}}\big|_{y=\pm1}=&2\mathrm{Re} q'\lf(\pa_y+ikt{v_y}\rg)\omega_k\overline{q'\pa_y\lf(\pa_y+ikt{v_y}\rg)\omega_k}\big|_{y=\pm1}\\
=&2\mathrm{Re} q'\pa_y\omega_k\overline{q'\lf(v_y^{-1}\pa_{yy}\omega_k+iktv_y\pa_y\omega_k +iktv_{yy}\omega_k\rg)}=0.
\end{align}  
Hence at least for $n=1$, we might want to switch to this new definition. 

}
\fi

\ifx
\begin{lemma}The Navier boundary condition $\omega(y=\pm 1)=0$ naturally gives rise to the following conditions
\begin{align}\label{BC_n}
\pa_{y}^{2n}\omega\big|_{y=\pm 1}=0,\quad \forall n\in \mathbb{N}. 
\end{align}
\end{lemma}
\begin{proof}
Assume that for $m\leq n$, the result \eqref{BC_n} is proven, i.e.,
\begin{align}\label{Induction_hyp}
\pa_y^{2m}\omega\big|_{y=\pm 1}=0, \quad 0\leq m\leq n. 
\end{align}
Now we take the $\pa_{y}^{2n}$-derivative of the equation and evaluate the terms on the boundary and end up with
\begin{align*}
\pa_t \pa_y^{2n}\omega+\sum_{k=0}^{2n}\lf(\begin{array}{cc}2n\\ k\end{array}\rg)\pa_{y}^{2n-k}u^1\pa_x\pa_y^{k}\omega+\sum_{k=0}^{2n}\lf(\begin{array}{cc}2n\\ k\end{array}\rg)\pa_{y}^{2n-k}u^2 \pa_y^{k+1}\omega=\nu\lf(\pa_{y}^{2n+2}+\pa_{xx}\pa_y^{2n}\rg)\omega, \quad y=\pm1. 
\end{align*}
Now we use the fact that $\pa_x^m\omega(y=\pm 1)=0, \ \pa_x^m u^2(y=\pm1)=0$ and the hypothesis \eqref{Induction_hyp} to filter out the obvious zero terms:
\begin{align*}T_1+T_2:=\sum_{k=0}^{n-1}\lf(\begin{array}{cc}2n\\ 2k+1\end{array}\rg)\pa_{y}^{2n-2k-1}u^1\ \pa_x\pa_y^{2k+1}\omega+\sum_{k=0}^{n-1}\lf(\begin{array}{cc}2n\\ 2k\end{array}\rg)\pa_{y}^{2n-2k}u^2\ \pa_y^{2k+1}\omega=\nu\pa_{y}^{2n+2}\omega, \quad y=\pm1. \label{T12}
\end{align*}
Now we use the divergence free condition and the vorticity definition to get
\begin{align*}
\pa_{y} u^1=\pa_x u^2-\omega, \quad \pa_{yy} u^2=-\pa_{xy} u^1=\pa_x (\omega-\pa_x u^2)=\pa_x\omega-\pa_{xx}u^2. 
\end{align*}
Now we use this to calculate each term in the $T_1$ term of \eqref{T12}
\begin{align}
\pa_y^{2n-2k-2}\pa_y u^1=&\pa_x \pa_y^{2n-2k-2}u^2-\pa_y^{2n-2k-2}\omega=\pa_x \pa_y^{2n-2k-4}\pa_{yy}u^2\\
=&\pa_x(-\pa_{xx}) (\pa_{yy})^{n-k-2} u^2=...=\pa_x(-\pa_{xx})^{n-k-1}u^2=0,\quad y=\pm 1. 
\end{align}
Hence $T_1=0.$ By the same argument, we saw that each term in $T_2$ vanishes
\begin{align}
\pa_y^{2n-2k}u^2 \big|_{y=\pm 1}=\pa_y^{2n-2k-2}\lf(\pa_x\omega-\pa_{xx}u^2\rg)\big|_{y=\pm 1}=-\pa_{xx}\pa_{y}^{2n-2k-4}\lf(\pa_x\omega-\pa_{xx}u^2\rg)\big|_{y=\pm 1}=\lf(-\pa_{xx}\rg)^{n-k}u^2\big|_{y=\pm 1}=0. 
\end{align}
Since every term in the sum is zero, $T_2=0. $ Hence, combining the computations and \eqref{T12}, we obtain
\begin{align}
(\pa_{yy})^{n+1}\omega\big|_{y=\pm1}=0. 
\end{align}
\end{proof}
\fi

\section{Commutator Estimates} \label{sec:ext:comm}

In this section, we collect all of the commutator bounds we need in order to close our energy scheme. In particular, the main result of this section will be the proof of Proposition \ref{pro:rhs:intro}.

\subsection{Bounds on the Building Blocks $S_{m,n}^{(a,b,c)}\omega_k$}
\ifx
Given the decompositions \eqref{Cnq} and \eqref{IDL12}, we need to estimate the following crucial quantities: 
\begin{align} \label{SnLnJn}
S_{m,n;k} := &\frac{m+n}{q} \mathring{\omega}_{m,n;k},\\ 
\n  J_{m,n;k} :=& {\frac{m+n}{q}} \pa_y \mathring{\omega}_{m,n;k},\quad K_{m,n;k}:=\frac{m+n}{q}|k|\mathring{\omega}_{m,n;k},\qquad L_{m,n;k} := \frac{(m+n)^2}{q^2}\mathring{\omega}_{m,n;k}, \quad n\geq 2;\\
S_{m,1;k}=&(m+1)|k|^m\Gamma_k\omega_k,\quad \underline{J_{m,1;k}=K_{m,1;k}=L_{m,1;k}=0},\quad S_{m,0;k}=J_{m,0;k}=K_{m,0;k}=L_{m,0;k}=0.\n
\end{align}Here $\mathring{\omega}_{m,n;k}=|k|^mq^n\Gamma^n\omega_k$.
\siming{
We see that using the definition \eqref{S_vec}, the above are 
\begin{align}
S^{(1,0,0)}_{m,n}\omega_k=S_{m,n;k},\quad S^{(2,0,0)}_{m,n}\omega_k=L_{m,n;k},\quad S^{(1,1,0)}_{m,n}\omega_k=J_{m,n;k},\quad S^{(1,0,1)}_{m,n}\omega_k=K_{m,n;k}.
\end{align}
\footnote{Note that the definition is different from the $J^{(a,b,c)}$ in the elliptic section in two aspects. First, we have $\pa_y^b$ here instead of $\pav^b$. Second, we always assume that $a\geq 1$ here. 
Moreover, we have that the quantities that we are interested in are the $a\geq 1,\, a+b+c\leq 2$ case. }}

\fi
This subsection is devoted to the proof of the following Lemma \ref{lem:SLJ}. Along the way, we need a technical combinatorial Lemma \ref{lem:comb_sqr}.   
\begin{lemma} \label{lem:SLJ}If $1>\lambda>0$ is smaller than a universal constant, the quantities $S_{m,n}\omega_k, $ defined in \eqref{S_vec} satisfies the following bounds:
\begin{subequations}\label{S:est:all}
\begin{align} \label{S:est:1}
\bold{a}_{m,n}^2 \nu  \lf\| S_{m,n}^{(1,0,0)} \omega_k e^W  \chi_{m+n-1}\rg\|_{L^2} ^2 \lesssim & \sum_{\ell=0}^{n-1}(\mathfrak C\lambda^s)^{2(n-\ell)}\lf(\frac{(m+n)!}{(m+\ell)!}\rg)^{-2(\sigma+\sigma_\ast)}  \mathcal{D}^{(\gamma)}_{m,\ell; k},\\ \label{SnBd1sum} 
\quad \sum_{m+n=0}^M\bold{a}_{m,n}^2 \nu  \lf\| S_{m,n}^{(1,0,0)}\omega_k e^W  \chi_{m+n -1}\rg\|_{L^2} ^2 \lesssim & \lambda^{2s} \sum_{m+n=0}^{M-1} \mathcal{D}^{(\gamma)}_{m,\ell;k}.
\end{align}
For the $j=2$ level,
\begin{align}
\bold{a}_{m,n}^2 \nu^2 \lf \| S_{m,n}^{(2)} \omega_k e^W \chi_{m+n -1}\rg \|_{L^2} ^2\lesssim &\sum_{\ell=0}^{n-1}(\mathfrak{C}\lambda^s)^{2n-2\ell}\lf(\frac{(m+n)!}{(m+\ell)!}\rg)^{-2\sigma - 2\sigma_{\ast}} (\mathcal{D}^{(\alpha)}_{m,\ell;k} + \mathcal{D}^{(\mu)}_{m,\ell;k}), 
\label{S:est:2}\\ 
\sum_{m+n=0}^M\bold{a}_{m,n}^2 \nu^2 \lf \| S_{m,n}^{(2)}\omega_k e^W \chi_{m+n -1}\rg \|_{L^2}^2 \lesssim & \lambda^{2s}\sum_{m+n=0}^{M-1} \lf(\mathcal{D}^{(\alpha)}_{m,n;k} + \mathcal{D}^{(\mu)}_{m,n;k} \rg).
\label{SnBd2sum} 
\end{align}
\ifx
\siming{Previous:
\begin{align} 
\label{LnBound}
\bold{a}_{m,n}^2  \nu^2 \lf \| L_{m,n;k} e^W \chi_{m+n-1} \rg\|_{L^2}^2  \lesssim &  \sum_{\ell=0}^{n-1}(\mathfrak{C}\lambda^s)^{2n-2\ell}\lf(\frac{(m+n)!}{(m+\ell)!}\rg)^{-2\sigma - 2\sigma_{\ast}} ( \mathcal{D}^{(\alpha)}_{m,\ell;k}+   \mathcal{D}^{(\mu)}_{m,\ell;k} ),   \\ 
\label{LnBoundsum}
\sum_{m+n=0}^M\bold{a}_{m,n} ^2 \nu^2 \lf \| L_{m,n;k} e^W \chi_{m+n -1} \rg\|_{L^2}^2  \lesssim & \lambda^{2s} \sum_{m+n=0}^{M-1}\lf( \mathcal{D}^{(\alpha)}_{m,n;k} + \mathcal{D}^{(\mu)}_{m,n;k}\rg) ,   \\ 
\label{JnBound}
\bold{a}_{m,n}^2 \nu^2 \lf \| J_{m,n;k} e^W \chi_{m+n -1}\rg \|_{L^2} ^2\lesssim &\sum_{\ell=0}^{n-1}(\mathfrak{C}\lambda^s)^{2n-2\ell}\lf(\frac{(m+n)!}{(m+\ell)!}\rg)^{-2\sigma - 2\sigma_{\ast}} (\mathcal{D}^{(\alpha)}_{m,\ell;k} + \mathcal{D}^{(\mu)}_{m,\ell;k}),\\
\label{JnBoundsum}\sum_{m+n=0}^M\bold{a}_{m,n}^2 \nu^2 \lf \| J_{m,n;k} e^W \chi_{m+n -1}\rg \|_{L^2}^2 \lesssim & \lambda^{2s}\sum_{m+n=0}^{M-1} \lf(\mathcal{D}^{(\alpha)}_{m,n;k} + \mathcal{D}^{(\mu)}_{m,n;k} \rg).
\end{align}}
\fi
\end{subequations}
Here $\mathcal{D}$-terms are defined in \eqref{defn:D:1}. The implicit constants and the constant $\mathfrak{C}$'s depend only on $s\in(1,2]$. 
\end{lemma}

\begin{proof} We decompose the proof into several components. 

\noindent
{\bf Step \# 1: Proof of \eqref{S:est:1}, \eqref{SnBd1sum}.}   We proceed by induction. For the base case, $n = 0$, we trivially have $S_{m,0}^{(1,0,0)}\omega_k = 0$. We now record the inductive identity, valid for all $n \ge 1$,  
\begin{align} \n
S_{m,n}^{(1,0,0)}\omega_k := &\frac{m+n}{q} \mathring{\omega}_{m,n;k} = (m+n) q^{n-1}|k|^m \Gamma_k^n \omega_k  = (m+n) q^{n-1} \lf({\vyn\pa_y} + i k t\rg) \Gamma_k^{n-1} |k|^m\omega_k \\ \n
= & (m+n) v_y^{-1}{\pa_y} \lf( q^{n-1} \Gamma_k^{n-1}|k|^m \omega_k \rg) - {(m+n)}(n-1){\vyn}{q'}  q^{n-2} \Gamma_k^{n-1}|k|^m \omega_k \\ \n
&+ (ikt) (m+n) q^{n-1} \Gamma_k^{n-1}|k|^{m} \omega_k \\ \label{start1}
= &{(m+n)}{\vyn} \pa_y \omega_{m,n-1;k} - {\frac{(m+n)}{m+n-1} }(n-1) {\vyn}q'  S_{m,n-1}^{(1,0,0)}\omega_k +  (ikt) (m+n) \omega_{m,n-1;k}.
\end{align}
We recall the definition of $\bold{a}_{m,n}$ \eqref{a:weight}, the bound $t\varphi\leq C$ \eqref{ineq:varphi}, with the bound \eqref{gevbd1212} and estimate the right hand side above as follows,
\begin{align} \n
 \bold{a}_{m,n} &\nu^{\frac12} \|  S_{m,n}^{(1,0,0)}\omega_k\ e^W \chi_{m+n-1} \|_{L^2}  \\ \n
 \lesssim &  (m+n) \bold{a}_{m,n}\lf\|v_y^{-1} \rg\|_{L^\infty} \lf\| \nu^{\frac12}\pa_y\mathring{ \omega}_{m,n-1;k} e^W \chi_{m+n-1}\rg \|_{L^2}  \\ \n
 & + (m+n) \bold{a}_{m,n} \lf\|{\vyn}{q'} \rg\|_{L^\infty} \lf\| \nu^{\frac12}S_{m,n-1;k} ^{(1,0,0)}e^W \chi_{m+n-1} \rg\|_{L^2} \\ \n
  & +(m+ n) B_{m,n}  \varphi^{m+n - 1}   (\varphi  t)  \lf\| \nu^{\frac12} |k| \omega_{m,n-1;k} e^W \chi_{m+n-1} \rg \|_{L^2} \\
\label{ye1}  \lesssim & \lambda^s (m+n)^{-\sigma-\sigma_\ast}\sqrt{\mathcal{D}^{(\gamma)}_{ m,n-1; k}}+ \lambda^s(m+ n)^{-\sigma-\sigma_\ast}\bold{a}_{m,n - 1;k}  \lf\| \nu^{\frac12}S_{m,n-1}^{(1,0,0)}\omega_k\ e^W \chi_{m+n-2}\rg \|_{L^2}.
\end{align}
Here we have invoked the bound \eqref{v_y_asmp} to control $\displaystyle\lf\|{}v_y^{-1}\rg \|_{L^\infty}$. Now we have that there exists a constant $\mathfrak C$, independent of $m,n$, such that iteration of  the above inequality yields that for $\lambda$ small compared to universal constants,
\begin{align*}
 \bold{a}_{m,n}^2& \nu  \|  S_{m,n}^{(1,0,0)}\omega_k\  e^W \chi_{m+n-1} \|_{L^2}^2
\leq  \sum_{\ell=0}^{n-1}(\mathfrak C\lambda^s)^{2(n-\ell)}\lf(\frac{(m+n)!}{(m+\ell)!}\rg)^{-2(\sigma+\sigma_\ast)} \mathcal{D}^{(\gamma)}_{m,\ell; k}.
\end{align*}This is \eqref{S:est:1}. Next, we have that by Lemma \ref{lem:comb_sqr},
\begin{align}
\sum_{m+n=0}^M \bold{a}_{m,n}^2& \nu  \|  S_{m,n}^{(1,0,0)}\omega_k\ e^W \chi_{m+n-1} \|_{L^2}^2\lesssim \sum_{m=0}^M\sum_{n=0}^{M-m}\sum_{\ell=0}^{n-1}(\mathfrak C\lambda^s)^{2(n-\ell)}\lf(\frac{(m+n)!}{(m+\ell)!}\rg)^{-2(\sigma+\sigma_\ast)} \mathcal{D}^{(\gamma)}_{m,\ell; k}\n \\
\lesssim&\sum_{m=0}^M\sum_{\ell=0}^{M-m-1}\sum_{n=\ell+1}^{M-m}(\mathfrak C\lambda^s)^{2(n-\ell)} \mathcal{D}^{(\gamma)}_{m,\ell; k} 
\lesssim\lambda^{2s}\sum_{m+n=0}^{M-1} \mathcal{D}^{(\gamma)}_{m,n; k}.\label{sqr_upgrade}
\end{align} 
This is \eqref{SnBd1sum}. Hence the first step is concluded. 
\siming{(Double check the $\ell=0$ case!) For checking purpose only 
\begin{align*}
{\bf a}_{m,1;k}^2&\nu\|S_{m,1}^{(1,0,0)}\omega_k\ e^W\chi_{m}\|_{L^2}={\bf a}_{m,1;k}^2\nu\||k|^m (\pav +ikt)\omega_ke^W\chi_{m}\|_{L^2}^2\\
\leq& C\frac{\lambda^{2s}}{(m+1)^{2s}}{\bf a}_{m,0;k}^2\nu(\|\vyn\|_\infty^2\|\pa_y \mathring\omega_{m,0;k}e^W\chi_{m}\|_2^2+(\varphi^2t^2)\||k|\omega_{m,0;k}e^W\chi_{m}\|_2^2)\leq C\frac{\lambda^{2s}}{(m+1)^{2s}}\mathcal{D}_{m,0;k}^{(\gamma)}. 
\end{align*}
Hence the $\ell=0,1$ is consistent with the estimate above. 
}

\noindent
{\bf Step \# 2: Estimate of $S_{m,n}^{(1,0,1)}\omega_k$.}
These estimates are natural consequences of multiplying the bounds \eqref{S:est:1}, \eqref{SnBd1sum} by  $\nu|k|$, and applying the definitions \eqref{defn:D:1} $\nu|k|\mathcal{D}^{(\gamma)}_{m,n;k}= \mathcal{D}^{(\mu)}_{m,n;k}$. The explicit estimate is as follows
\ifx
\begin{align*} \n
 \bold{a}_{m,n} &\nu  \|  K_{m,n;k} e^W \chi_{m+n-1} \|_{L^2}  \\ \n
 \lesssim &  (m+n) \bold{a}_{m,n} \nu \lf\||k| \pa_y \omega_{m,n-1;k} e^W \chi_{m+n-1}\rg \|_{L^2}  \\ \n
 & + (m+n) \bold{a}_{m,n}\nu   \lf\|  K_{m,n-1;k} e^W \chi_{m+n-1} \rg\|_{L^2} \\ \n
  & +(m+n) B_{m,n}   \varphi^{m+n - 1}   (\varphi  t) \nu \lf\|   |k|^2 \omega_{m,n-1;k} e^W \chi_{m+n-1} \rg \|_{L^2} \\
  \lesssim & \lambda^s (m+n)^{-\sigma-\sigma_\ast}\sqrt{\mathcal{D}^{(\mu)}_{ m,n-1; k}}+ \lambda^s(m+ n)^{-\sigma-\sigma_\ast}\bold{a}_{m,n - 1;k} \nu   \lf\| K_{m,n-1;k} e^W \chi_{m+n-2}\rg \|_{L^2}.
\end{align*}
Hence, by summing up in $m+n$ and taking $\lambda $ to be small, we have that 
\begin{align}
\sum_{m+n=0}^M \bold{a}_{m,n}^2\nu^2 \|K_{m,n;k}e^W\chi_{m+n-1}\|_{L^2}^2\lesssim \lambda^{2s}\sum_{m+n=0}^{M-1}\mathcal{D}_{m,n;k}^{(\mu)}. 
\end{align}\fi
\begin{align}
\label{Kest}\bold{a}_{m,n}^2  \nu^2 \lf \| S_{m,n}^{(1,0,1)}\omega_k\  e^W \chi_{m+n-1} \rg\|_{L^2}^2  \lesssim &  \sum_{\ell=0}^{n-1}(\mathfrak C\lambda^s)^{2(n-\ell)}\lf(\frac{(m+n)!}{(m+\ell)!}\rg)^{-2(\sigma+\sigma_\ast)}  \mathcal{D}^{(\mu)}_{m,\ell; k},\\ 
\label{Kest_sum}
\sum_{m+n=0}^M\bold{a}_{m,n}^2 \nu^2 \lf \| S_{m,n}^{(1,0,1)}\omega_k\  e^W  \chi_{m+n -1} \rg\|_{L^2}^2  \lesssim & \lambda^{2s} \sum_{m+n=0}^{M-1}\mathcal{D}^{(\mu)}_{m,\ell;k}.
\end{align}
This concludes Step \# 2.

\noindent
{\bf Step \# 3: Estimate of $S^{(2,0,0)}_{m,n} \omega_k$. }
We proceed inductively. The base case is $n=2$, we estimate the term as follows:
\begin{align*}
 S_{m,2}^{(2,0,0)} \omega_k=&(m+2)^2\lf(\pav^2+2ikt\pav-k^2t^2\rg)|k|^m\omega_k. 
\end{align*}
Now we have that by the definition of ${\bf a}_{m,2}$ \eqref{a:weight}, the Gevrey coefficient bound \eqref{gevbd1212}, the $ {v_y}$-bound  \eqref{v_y_asmp},  $t\varphi\leq C$ \eqref{ineq:varphi}, and the definition of $\mathcal{D}_{m,n;k}^{(\al)},\ \mathcal{D}_{m,n;k}^{(\mu)}$ \eqref{defn:D:1}, the following estimate holds
\begin{align}\n
{\bf a}_{m,2}&\nu\lf\|S_{m,2}^{(2,0,0)} \omega_k\ e^W\chi_{m+1}\rg\|_2\\ \n 
  \lesssim& {\bf a}_{m,2}(m+2)^2 \nu \lf\|\na_k\pa_{y}|k|^{m}\omega_k e^W\chi_{m+1}\rg\|_2 
 +B_{m,2}(m+2)^2\varphi^2   \nu (t\varphi)\||k|\pa_y|k|^m\omega_k e^W\chi_{m+1}\|_2\\ 
&+B_{m,2}(m+2)^2 \varphi\nu(t\varphi )^2\||k|^{2+m}\omega_k e^W\chi_{m+1}\|_2
 \lesssim \lf(\mathcal{D}^{(\al)}_{m,0;k}+\mathcal{D}^{(\mu)}_{m,0;k}\rg)^{1/2}.\label{L_est_2}
\end{align}
For $n > 2$, we open up two $\Gamma_k$-derivatives to obtain the expression
 \begin{align}\n
 S_{m,n}^{(2,0,0)} \omega_k =&(m+n)^2q^{n-2}\lf( \pav^2+2ikt\pav-k^2t^2\rg)\Gamma_k^{n-2}|k|^m\omega_k\\
\n =&(m+n)^2\pav^2(q^{n-2}\Gamma_k^{n-2}|k|^m\omega_k)+(m+n)^2\lf[q^{n-2}, \pav^2\rg]\Gamma_k^{n-2}|k|^m\omega_k\\ \n
\n &+2{(m+n)^2}ikt\pav\lf(q^{n-2}\Gamma_k^{n-2}|k|^m\omega_k\rg)+2(m+n)^2ikt\lf[q^{n-2},\pav\rg]\Gamma_k^{n-2}|k|^m\omega_k \\
& -(m+n)^2|k|^2t^2\lf(q^{n-2}\Gamma_k^{n-2}|k|^m\omega_k\rg)\n \\
 =:&\sum_{j=1}^5\mathcal{T}_{3;j}.\label{LI1-5}
 \end{align}
Here, `$3$' stands for Step \# 3 and `$j$' denotes the `j'-th term.  
We start by estimating the $\mathcal{T}_{3;1}$ term in \eqref{LI1-5}. By further expanding the expression, we have that
\begin{align*}
\mathbf{a}_{m,n}&\nu \|\mathcal{T}_{3;1}e^W\chi_{m+n-1}\|_2
=\mathbf{a}_{m,n}\nu \lf\|(m+n)^2\lf(\frac{1}{v_y^2}\pa_{yy}\omega_{m,n-2;k}-\frac{v_{yy}}{v_y^3}\pa_y\mathring\omega_{m,n-2;k}\rg)e^W\chi_{m+n-1}\rg\|_2\\
\lesssim&(m+n)^2\bold{a}_{m,n}\nu \lf(\|\pa_{yy}\mathring\omega_{m,n-2;k}e^W\chi_{m+n-2}\|_2+ \|\pa_{y}\mathring\omega_{m,n-2;k}e^W\chi_{m+n-2}\|_2\rg).
\end{align*}Note that the definition of $\bold{a}_{m,n}$ \eqref{a:weight} and the estimate \eqref{gevbd1212} yields that 
\begin{align*}(m+n)^2\bold{a}_{m,n}=&(m+n)^2B_{m,n} \varphi^{m+n} \leq \frac{C\lambda^{2s}B_{m,n-2}\varphi^{m+n-2}}{(m+n)^{2(\sigma+\sigma_\ast)}}  
\leq  \frac{C\lambda^{2s}}{(m+n)^{2(\sigma+\sigma_\ast)}} {\bf a}_{m,n-2;k}.
\end{align*}
Combining this estimate and the ${v_y}$-bound \eqref{v_y_asmp}, and the definition of the dissipation $\mathcal{D}^{(\al)}_{m,n;k}$ \eqref{defn:D:1} yields that 
\begin{align}
\mathbf{a}_{m,n}\nu \|\mathcal{T}_{3;1}e^W\chi_{m+n-1}\|_2\leq &  \frac{C \lambda^{2s}\mathbf{a}_{m,n-2} }{(m+n)^{2(\sigma+\sigma_\ast)}}\nu \|\na_k\pa_{y}\mathring\omega_{m,n-2;k}e^W\chi_{m+n-2}\|_2
\leq  \frac{C\lambda^s\mathcal{D}^{(\al)}_{m,n-2;k}}{(m+n)^{2(\sigma+\sigma_\ast)}}. \label{LI1_est}
\end{align}

To estimate the $\mathcal{T}_{3;2}$ term, we apply the commutator relation \eqref{cm_pvv_qn} $(n-2\geq 2)$, \eqref{cm_pvv_q} $(n-2=1)$, and the definition of $\bold{a}_{m,n}$ \eqref{a:weight}, the bound \eqref{gevbd1212},  the bound of $\|v_y^{-1}\|_\infty,\, \|v_{yy}\|_\infty$, and the bound $\varphi t\leq C$ \eqref{ineq:varphi}, to derive that,
\ifx \begin{align}\n
-\frac{\mathcal{I}_{L;2}\mathbbm{1}_{n\neq 3}}{n(n-1)} =&2\frac{q'}{(v_y)^2}\lf(\frac{n-2}{q}\pa_y\omega_{k,n-2}\rg)+\frac{q''}{(v_y)^2}\lf(\frac{n-2}{q}\omega_{k,n-2}\rg)\\
\n &-\frac{q'}{(v_y)^2}\lf(\vyi\pa_y v_y\rg)\lf(\frac{n-2}{q}\omega_{k,n-2} \rg)-\frac{(q')^2}{(v_y)^2}\frac{(n-2)^2+n-2}{(n-2)(n-3)}\lf(\frac{(n-2)(n-3)}{q^2}\omega_{k,n-2}\rg)\\
\n =&\frac{2q'}{v_y^2}J_{k,n-2}+\frac{q''}{v_y^2}S_{k,n-2}-\frac{q'}{v_y^3}v_{yy}S_{k,n-2}-\frac{(q')^2}{v_y^2}\frac{n-1}{n-3}L_{k,n-2},\quad n\nq 3.
 \end{align}
 For $n=3$, the $\mathcal{I}_{L;2}$ term has a simpler form \siming{Double the proof of \eqref{cm_pvv_qn} for $n=1,2$! } 
 \begin{align}\n
-\frac{1}{6} \mathcal{I}_{L;2}\mathbbm{1}_{n=3}=&\frac{q''}{v_y^2}\Gamma_k \omega_k-\frac{q'v_{yy}}{v_y^3}\Gamma_k \omega_k+\frac{2q'}{v_y^2}\pa_y\Gamma_k \omega_k\\
 \n =&\lf(\frac{q''}{v_y^3}-\frac{q'v_{yy}}{v_y^4}-\frac{2q'v_{yy}}{v_y^4}\rg)\pa_y\omega_k +\lf(\frac{q''}{v_y^2}-\frac{q'v_{yy}}{v_y^3}\rg)ikt\omega_k+\frac{2q'}{v_y^2}ikt\pa_y\omega_k+\frac{2q'}{v_y^3}\pa_{yy}\omega_k.
 \end{align}\fi 
 \begin{align}
 \n \bold{a}_{m,n} &\nu \|\mathcal{T}_{3;2}e^W\chi_{m+n-1}\|_2\\ \n
  \lesssim& \mathbbm{1}_{n\geq 4}(m+n)^2\bold{a}_{m,n}\nu \sum_{\substack{a+b+c=2,\\ a\nq 0}}\|S_{m,n-2}^{(a,b,c)}\omega_{k} e^W\chi_{m+n-1}\|_2+\mathbbm{1}_{n=3}(m+3)^2\bold{a}_{m,3}\nu \|\Gamma_k\omega_{k} e^W\chi_{m+2}\|_{H^1} \\
 \lesssim & \mathbbm{1}_{n\geq 4}\frac{\lambda^{2s}{\bf a}_{m,n-2}\nu}{(m+n)^{2(\sigma+\sigma_\ast)}}  \sum_{\substack{a+b+c=2,\\ a\nq 0}}\|S_{m,n-2}^{(a,b,c)}\omega_{k} e^W\chi_{m+n-1}\|_2+\frac{\lambda^{2s}\mathbbm{1}_{n=3}}{(m+3)^{2\sigma+2\sigma_\ast}}\sqrt{\mathcal{D}^{(\al)}_{m,0;k}+\mathcal{D}^{(\mu)}_{m,0;k}}.  \label{LI2est}
 \end{align} 
\ifx For $n=3$, we have that by the definition of $\bold{a}_{m,n}$ \eqref{a:weight}, the bound of $\|v_y^{-1}\|_\infty,\, \|v_{yy}\|_\infty$, and the bound $\varphi t\leq C$ \eqref{ineq:varphi}, the following estimate holds
 \begin{align}
\n \bold{a}_{m,n}\mathbbm{1}_{n=3}&\nu \|\mathcal{I}_{L;2}e^W\chi_{m+n-1}\|_2\\
 \n \lesssim&{\bf a}_{m,3}\nu \|\pa_y\omega_k e^W\chi_{m+2} \|_2+B_{m,3}\varphi^2  \nu |k|^{m }(t\varphi)\||k|\omega_k e^W\chi_{m+2}\|_2\\
 \n &+B_{m,3}\varphi^2  \nu|k|^{m }(t\varphi)\||k|\pa_y\omega_k e^W\chi_{m+2}\|_2+{\bf a}_{m,3}\nu \|\pa_{yy}\omega_k e^W\chi_{m+2} \|_2\\
 \lesssim &\mathcal{D}^{(\al)}_{m,0;k}+\mathcal{D}^{(\mu)}_{m,0;k}.\label{LI2estn=3}
 \end{align}\fi
Next we recall the definition of $\bold{a}_{m,n}$ \eqref{a:weight}, the bound \eqref{gevbd1212},  the bound of $\|v_y^{-1}\|_\infty$ \eqref{v_y_asmp}, and the bound $\varphi t\leq C$ \eqref{ineq:varphi}, and then estimate the $\mathcal{T}_{3;3}$ as follows
\begin{align}\n 
& \bold{a}_{m,n} \nu \|\mathcal{T}_{3;3}e^W\chi_{m+n-1}\|_2 
    \lesssim (m+n)^2{ B}_{m,n}\varphi^{m+n}\nu  (\varphi t)\||k|\pa_{y}\mathring\omega_{m, n-2;k}e^W\chi_{m+n-2}\| _2\\ 
   & \lesssim \frac{\lambda^{2s}B_{m,n-2}\varphi^{m+n}}{(m+n)^{2(\sigma+\sigma_\ast) }}\nu \||k|\pa_{y}\mathring\omega_{m, n-2;k}e^W\chi_{m+n-2}\|_2 
     \lesssim \frac{\lambda^{2s}}{(m+n)^{2(\sigma+\sigma_\ast) }}\sqrt{\mathcal{D}^{(\al)}_{m,n-2;k}}.\label{LI3_est} 
\end{align} 
To estimate the $\mathcal{T}_{3;4}$ term, we invoke the commutator relation \eqref{cm_pv_qn} to get that
\begin{align*}
\mathcal{T}_{3;4}= &-2(m+n)^2ikt\frac{q'}{v_y}\lf(\frac{n-2}{q}|k|^m q^{n-2}\Gamma ^{n-2}\omega_k\rg)\\
=&-2\frac{(m+n)^2(n-2)}{m+n-2}i\frac{k}{|k|}t\frac{q'}{v_y}\ S_{m,n-2}^{(1,0,1)}\omega_k\ \mathbbm{1}_{n\geq 4}-2(m+3)^2i k  t \frac{q'}{v_y}(\pav +ikt )|k|^m\omega_k\ \mathbbm1_{n=3}.
\end{align*}
Then we have that by the bound of $\|{v_y^{-1}}\|_\infty, \, \|q'\|_\infty$, the estimate \eqref{Kest}, the definition of $\bold{a}_{m,n}$ \eqref{a:weight}, the bound \eqref{gevbd1212},  and the bound $\varphi t\leq C$ \eqref{ineq:varphi}, the following estimate holds
\begin{align}\n
 \bold{a}_{m,n}&\nu \|\mathcal{T}_{3;4}e^W\chi_{m+n-1}\|_2\\ \n
 \lesssim& (m+n)^2\bold{a}_{m,n}\nu  t\|S_{m,n-2}^{(1,0,1)}\omega_k e^W\chi_{m+n-1}\|_2\mathbbm{1}_{n\geq 4}+(m+3)^2\bold{a}_{m,3}\nu\|k\Gamma_k |k|^m \omega_k e^W \chi_{m}\|_{2}\mathbbm{1}_{n=3}\\
    &\lesssim\frac{\lambda^{2s}{\bf a}_{m,n-2}}{(m+n)^{2(\sigma+\sigma_\ast)}}\nu \|S_{m,n-2}^{(1,0,1)}\omega_k e^W\chi_{m+n-1}\|_2\mathbbm{1}_{n\geq 4}+ \frac{\lambda^{2s}}{(m+3)^{2(\sigma+\sigma_\ast)}}\sqrt{\mathcal{D}_{m,0;k}^{(\al)}+\mathcal{D}^{(\mu)}_{m,0;k}}\mathbbm{1}_{n=3} .\label{LI4_est}
\end{align}
\siming{Maybe this is the reason why the $\mu$ is needed.} The estimate of the $\mathcal{T}_{3;5}$ contribution is similar
\begin{align}
\n \bold{a}_{m,n}&\nu \|\mathcal{T}_{3;5}e^W\chi_{m+n-1}\|_2\lesssim (m+n)^2\bold{a}_{m,n}\nu  t^2\||k|^2\omega_{m,n-2;k}e^W\chi_{m+n-1}\|_2\\
   \lesssim&\frac{\lambda^{2s}{\bf a}_{m,n-2}}{(m+n)^{2(\sigma+\sigma_\ast)}}\nu \||k|^2\omega_{m,n-2;k}e^W\chi_{m+n-2}\|_2\lesssim \frac{\lambda^{2s}{\bf a}_{m,n-2}}{(m+n)^{2(\sigma+\sigma_\ast)}}(\mathcal{D}_{m,n-2;k}^{(\mu)})^{1/2}.\label{LI5_est}
\end{align}
Summarizing the decomposition \eqref{LI1-5}, the  estimates \eqref{L_est_2}, \eqref{LI1_est}, \eqref{LI2est}, \eqref{LI3_est}, \eqref{LI4_est} and \eqref{LI5_est}, we obtain that 
\begin{align}\n &
{\bf a}_{m,2}\nu \lf\|S_{m,2}^{(2,0,0)}\omega_k\ e^W\chi_{m+1}\rg\|_2 \lesssim \lambda^{2s} \sqrt{\mathcal{D}^{(\al)}_{m,0;k}+\mathcal{D}_{m,0;k}^{(\mu)}}; \\ 
&\n \bold{a}_{m,n}\nu \lf\|S_{m,n}^{(2,0,0)}\omega_k\ e^W\chi_{m+n-1}\rg\|_2\mathbbm{1}_{n>2}\\
\n &\lesssim \frac{\lambda^{2s}}{(m+n)^{2\sigma+2\sigma_\ast}}\lf(\mathcal{D}_{m,n-2;k}^{(\al)}+\mathcal{D}_{m,n-2;k}^{(\mu)}\rg)^{1/2}\mathbbm{1}_{n>2}+ \frac{\lambda^{2s} }{(m+3)^{2\sigma+2\sigma_\ast}}\lf(\mathcal{D}_{m,0;k}^{(\al)}+\mathcal{D}_{m,0;k}^{(\mu)}\rg)^{1/2}\mathbbm{1}_{n=3}\\
&\quad+\frac{\lambda^{2s}{\bf a}_{m,n-2}}{(m+n)^{2\sigma+2\sigma_\ast}}\nu \sum_{\substack{a+b+c=2\\ a\nq 0}}\|S_{m,n-2}^{(a,b,c)}\omega_k e^W\chi_{m+n-3}\|_2 \mathbbm{1}_{n>2}.\label{Itrtn}
\end{align}
 \ifx
{\bf Previous:} We proceed inductively. The base case of $n = 1$ is trivial as $L_{k,1} = 0$. We fix $n \ge 2$, and assume that \eqref{LnBound} is true up to $n - 1$. A long computation then gives the following commutator identity:
\begin{align} \nonumber
L_{n, k} := & n(n-1) q^{n-2} \Gamma_k^{n} \omega_k \\  \nonumber
= & n(n-1)  q^{n-2} (\pa_y + i k t)^2 \Gamma_k^{n-2} \omega_k \\  \nonumber
= & n(n-1) \pa_y^2 \omega_{k,n-2} + 2 n(n-1) (ikt) \pa_y \omega_{k,n-2} + (ikt)^2 n(n-1) \omega_{k,n-2} \\  \nonumber
& - q' n(n-1) (n-2) (ikt) q^{n-3} \Gamma_k^{n-2} \omega_k - n(n-1)(n-2) q'' q^{n-3} \Gamma_k^{n-2} \omega_k \\  \nonumber
& - n(n-1)(n-2) q' \pa_y \{ q^{n-3} \Gamma_k^{n-2} \omega_k \} - n(n-1)(n-2) q' q^{n-3} \Gamma_k^{n-1} \omega_k \\  \nonumber
= & n(n-1) \pa_y^2 \omega_{k,n-2} + 2 n(n-1) (ikt) \pa_y \omega_{k,n-2} + (ikt)^2 n(n-1) \omega_{k,n-2} \\  \nonumber
& - q' n(n-1) (n-2) (ikt) q^{n-3} \Gamma_k^{n-2} \omega_k - n(n-1)(n-2) q'' q^{n-3} \Gamma_k^{n-2} \omega_k \\  \nonumber
& - n(n-1)(n-2)(n-3) |q'|^2 q^{n-4} \Gamma_k^{n-2} \omega_k - n(n-1) (n-2) q' q^{n-3} \Gamma_k^{n-1} \omega_k \\  \nonumber
&+ (ikt) n(n-1) (n-2) q' q^{n-3} \Gamma_k^{n-2} \omega_k - n q'L_{n-1, k} \\  \nonumber
= & n(n-1) \pa_y^2 \omega_{k,n-2} + 2 n(n-1) (ikt) \pa_y \omega_{k,n-2} + (ikt)^2 n(n-1) \omega_{k,n-2} \\  \nonumber
& - q' n(n-1) (n-2) (ikt) q^{n-3} \Gamma_k^{n-2} \omega_k - n(n-1)(n-2) q'' q^{n-3} \Gamma_k^{n-2} \omega_k \\ \n
& - n(n-1) |q'|^2 L_{n-2} - 2n q' L_{n-1} + (ikt) n(n-1) (n-2) q' q^{n-3} \Gamma_k^{n-2} \omega_k \\ \n
=:&  \sum_{i = 1}^8 L_{n, k}^{(i)}.
\end{align}
The primary term above is the first term, $L^{(1)}_{n, k}$, which we estimate here
\begin{align*}
 B_{m,n} \varphi^{m+n} |k|^{m - \frac13} \nu^{\frac56} \| L^{(1)}_{k,n} e^W e^{\delta_E \nu^{\frac13}t}\chi_n \|_{L^2}   \lesssim & B_{n,  m} n^2  \varphi^{m+n} |k|^m \nu^{\frac56} \|    \pa_y^2 \omega_{k,n-2}  e^W e^{\delta_E \nu^{\frac13}t}\chi_n \|_{L^2} \\
 \lesssim & B_{n-2, m}  \varphi^{m+n} |k|^m \nu^{\frac56} \|  \pa_y^2 \omega_{k,n-2}  e^W e^{\delta_E \nu^{\frac13}t}\chi_n \|_{L^2} \\
 \lesssim & \mathcal{D}^{(\alpha)}_{n-2,m, k},
\end{align*} 
where we have invoked the bound \eqref{gevbd1212}. The bound on $L_n^{(5)}$ is trivial. The terms with $(ikt)$ factors, $L_n^{(i)}, i = 2,3,4,8$ are estimated now. First,
\begin{align*}
 &B_{n, m} \varphi^{m+n} |k|^{m - \frac13} \nu^{\frac56} \| L^{(2)}_{k,n} e^W e^{\delta_E \nu^{\frac13}t}\chi_n \|_{L^2}  \\
 \lesssim &  B_{m,n} n^2 \varphi^{m+n} |k|^{m+1 - \frac13} |t| \nu^{\frac56} \| \p_y \omega_{k,n-2}  e^W e^{\delta_E \nu^{\frac13}t}\chi_n \|_{L^2}  \\
 \lesssim & B_{n-2, m} \varphi^{m+n - 1} |k|^m \| \nu^{\frac56} |k|^{\frac23} \pa_y \omega_{k,n-2} e^W e^{\delta_E \nu^{\frac13}t} \chi_n \|_{L^2}  \lesssim \mathcal{D}^{(\alpha)}_{n-2, m, k}(t). 
\end{align*}
Next, we estimate the $L_{n,k}^{(3)}$ term, 
\begin{align*}
&B_{m,n} \varphi^{m+n} |k|^{m - \frac13} \nu^{\frac56} e^{\delta_E \nu^{\frac13}t} \| L_{n,k}^{(3)} e^W \chi_n \|_{L^2} \\
 \lesssim & B_{m,n} n^2 \varphi^{m+n-2} (\varphi^2 t^2) |k|^{m} e^{\delta_E \nu^{\frac13}t} \| \nu^{\frac56}  |k|^{\frac53}  \omega_{n-2,k} e^W \chi_n \|_{L^2} \\
  \lesssim & B_{n-2, m} \varphi^{m+n-2} |k|^{m} e^{\delta_E \nu^{\frac13}t} \| \nu^{\frac56}  |k|^{\frac53}  \omega_{n-2,k} e^W \chi_n \|_{L^2} \lesssim \mathcal{D}_{n-2,m,k}^{(\mu)}(t)
\end{align*}
Next, we estimate the $L_{n,k}^{(4)}$ term, for which we will need to invoke our already established bounds on $S_{m,n;k}$,
\begin{align*}
&B_{m,n} \varphi^{m+n} |k|^{m - \frac13} \nu^{\frac56} e^{\delta_E \nu^{\frac13}t} \| L_{n,k}^{(4)} e^W \chi_n \|_{L^2} \\
 \lesssim & B_{m,n} n^2 \varphi^{m+n} t |k|^{m + \frac23} \nu^{\frac56} e^{\delta_E \nu^{\frac13}t} \| \frac{n-2}{q} \omega_{k,n-2} e^W \chi_n \|_{L^2} \\
  \lesssim & B_{n-2, m} \varphi^{m+n-2}  |k|^{m + \frac23} \nu^{\frac56} e^{\delta_E \nu^{\frac13}t} \| S_{k,n-2} e^W \chi_n \|_{L^2} \lesssim \mathcal{D}^{(\mu)}_{n-2, m, k}(t). 
\end{align*}
Next, we estimate the $L_{n,k}^{(6)}$ term, for which we use the inductive hypothesis on $L_{m,n;k}$:
\begin{align*}
&B_{m,n} \varphi^{m+n} |k|^{m - \frac13} \nu^{\frac56} e^{\delta_E \nu^{\frac13}t} \| L_{n,k}^{(6)} e^W \chi_n \|_{L^2} \\
 \lesssim &B_{m,n} n^2 \varphi^{m+n} |k|^{m - \frac13} \nu^{\frac56} e^{\delta_E \nu^{\frac13}t} \|  L_{k,n-2} e^W \chi_n \|_{L^2} \\
\lesssim & B_{n-2, m} \varphi^{m+n} |k|^{m - \frac13} \nu^{\frac56} e^{\delta_E \nu^{\frac13}t} \|  L_{k,n-2} e^W \chi_n \|_{L^2} \lesssim \mathcal{D}^{(\alpha)}_{n-2,m,k} + \mathcal{D}^{(\mu)}_{n-2,m,k} + \mathcal{D}^{(\gamma)}_{n-2,m,k}.
\end{align*}
The remaining terms are estimated in a nearly identical fashion to what has already been done.
\fi

\noindent
{\bf Step 4: Iteration relation of $S_{m,n}^{(1,1,0)}\omega_k$.} First of all, we recall that $S_{m,1}^{(1,1,0)}\omega_k=S_{m,0}^{(1,1,0)}\omega_k=0$. For $n=2$, we recall the commutator relation \eqref{cm_pv_qn} and derive the following 
\begin{align}\n
S_{m,2}^{(1,1,0)}\omega_k=& \frac{(m+2)}{q}\pa_y\lf( q^2\Gamma_k^2|k|^m\omega_k\rg)=(m+2){q}^{-1}\pa_y\lf( q[q,\Gamma_k]\Gamma_k|k|^m\omega_k+q(\pav+ikt)(q\Gamma_k|k|^m\omega_{k})\rg)\\
\n =&(m+2){q}^{-1}\pa_y\lf(-(q'\vyn q)\Gamma_k\omega_k+ (q\vyn)\pa_{y}(q\Gamma_k|k|^m\omega_k)+(m+2)qikt (q\Gamma_k|k|^m\omega_k)\rg)\\
\n=&(m+2)\lf(-q^{-1}\pa_y(q'\vyn q)\Gamma_k|k|^m\omega_k-q'\vyn \pa_{y}(\pav+ikt)|k|^m\omega_k\rg)\\
\n &+(m+2)\lf(\vyn\pa_{yy}(\mathring{\omega}_{m,1;k})+q^{-1}\pa_y(q\vyn)(q'\Gamma_k|k|^m\omega_k+q\pa_y(\pav +ikt)|k|^m\omega_{k})\rg) \\
&+(m+2)\lf( q'ikt (q^{-1}\mathring{\omega}_{m,1;k})+ ikt\pa_y\mathring{\omega}_{m,1;k}\rg)=:\sum_{j=1}^3\mathcal{T}_{4;j}.\label{IJ_crc1-3}
\end{align}
We invoke the bound of $\|{v_y^{-1}}\|_\infty,\, \| v_{yy}\|_\infty,  \, \|q'\|_\infty$, the estimate \eqref{Kest}, the definitions of $\bold{a}_{m,n}$ \eqref{a:weight} and $\mathcal{D}$ \eqref{defn:D:1}, and the bounds \eqref{ineq:varphi},  \eqref{gevbd1212}, \eqref{S:est:1}, to obtain the following estimates
\begin{align*}\n 
{\bf a}_{m,2}&\nu \|\mathcal{T}_{4;1} e^W\chi_{m+1}\|_2\\
\n \lesssim & 
(m+2){\bf a}_{m,2}\nu \bigg(\|\pa_y(q'\vyn q)\|_\infty \|S_{m,1}^{(1,0,0)}\omega_k e^W\chi_{m+1}\|_2+\|q'v_y^{-3}v_{yy}\|_\infty\|\pa_y |k|^m\omega_ke^W\chi_{m+1}\|_2\\
&\quad\quad\quad\quad\quad\quad\quad +\|q'v_y^{-2}\|_\infty\|\pa_{yy}|k|^m\omega_ke^W\chi_{m+1}\|_2+t \|q'v_y^{-1}\|_\infty\||k|\pa_y|k|^m\omega_k e^W\chi_{m+1}\|_2\bigg) \\
\lesssim &\frac{\lambda^{s}{\bf a}_{m,1}}{(m+2)^{\sigma+\sigma_\ast}}\nu \|S_{m,1}^{(1,0,0)}\omega_k e^W\chi_{m+1}\|_2+\myr{\frac{\lambda^{2s}\sqrt{\mathcal{D}^{(\al)}_{m,0;k}}}{(m+2)^{2\sigma+2\sigma_\ast}}}+\frac{\lambda^{2s}B_{m,0}\varphi^2  (t\varphi) }{(m+2)^{2\sigma+2\sigma_\ast}}\nu \||k|\pa_y\mathring\omega_{m,0;k}e^W\chi_{m+1}\|_2\\
\lesssim &\myr{\frac{\lambda^s}{(m+2)^{\sigma+\sigma_\ast}}}(\mathcal{D}^{(\al)}_{m,0;k})^{1/2}.
\end{align*}Application of similar arguments yields the estimates for the other two terms in \eqref{IJ_crc1-3}, i.e., 
\begin{align*}
&\n {\bf a}_{m,2}\nu \|\mathcal{T}_{4;2} e^W\chi_{m+1}\|_2\\ 
\n &  \lesssim (m+2)
{\bf a}_{m,2}\nu \bigg(\|\pa_{yy}\mathring{\omega}_{m,1;k}e^W\chi_{m+1}\|_2+\|S_{m,1}^{(1,0,0)}\omega_k \ e^W\chi_{m+1}\|_2 + \|\pa_y|k|^m\omega_k e^W\chi_{m+1}\|_2\\
&\qquad\qquad\qquad\qquad+\|\pa_y(q\vyn)\|_\infty\|\vyn\|_\infty\|\pa_{yy}|k|^m\omega_k e^W\chi_{m+1}\|_2+ t\||k|\pa_{y}|k|^m\omega_k e^W\chi_{m+1}\|_2\bigg) \\ 
\n &\lesssim(\mathcal{D}^{(\al)}_{m,1;k}+\mathcal{D}^{(\al)}_{m,0;k}+\mathcal{D}^{(\mu)}_{m,0;k})^{1/2};
\end{align*}\begin{align*}
\n {\bf a}_{m,2}\nu \|\mathcal{T}_{4;3} e^W\chi_{m+1}\|_2 \lesssim&
(m+2){\bf a}_{m,2}\nu \bigg(t \||k|S_{m,1}^{(1,0,0)}\omega_k\ e^W\chi_{m+1}\|_2+t\||k|\pa_y\mathring{\omega}_{m,1;k}e^W\chi_{m+1}\|_2\bigg) \\
\n &\lesssim(\mathcal{D}^{(\mu)}_{m,1;k}+\mathcal{D}^{(\al)}_{m,0;k})^{1/2}.
\end{align*}
Combining the estimates we developed and the decomposition \eqref{IJ_crc1-3}, we obtain that 
\begin{align}
{\bf a}_{m,2}^2\nu^{2}\|S_{m,2}^{(1,1,0)}\omega_k\ e^W\chi_{m+1}\|_2^2\lesssim\sum_{\ell=0}^1\sum_{\iota\in\{\al,\mu\}}\mathcal D^{(\iota)}_{m,\ell;k}.\label{J_2} 
\end{align}
For $n>2$, we have that    
\begin{align}
S_{m,n}^{(1,1,0)}\omega_k=&\frac{m+n}{q}\pa_y \lf(q^n\Gamma_k^n|k|^m\omega_k\rg)=\frac{m+n}{q}\pa_y\lf( q[q^{n-1},\Gamma_k]\Gamma_k^{n-1}|k|^m\omega_k+q\Gamma_k(q^{n-1}\Gamma_k^{n-1}|k|^m\omega_{k})\rg).\n
\end{align}
Hence we recall the commutator relation \eqref{cm_pv_qn} and obtain that 
\begin{align}
\n S_{m,n;k}^{(1,1,0)}\omega_k=&\frac{m+n}{q}\pa_y\lf(-\frac{q'}{v_y}{(n-1)} q^{n-1}\Gamma_k^{n-1}|k|^m\omega_k+q\lf(\pav+ikt\rg)(q^{n-1}\Gamma_k^{n-1}|k|^m\omega_k)\rg)\\
\n =&-\frac{(m+n)(n-1)}{(m+n-1)^2}(q \pa_y\lf( {q'}\vyn\rg)) \ \underbrace{\frac{(m+n-1)^2}{q^2}\mathring\omega_{m,n-1;k}}_{=S_{m,n-1}^{(2,0,0)}\omega_k}\\ \n
&-\frac{(m+n)(n-1)q'}{(m+n-1)v_y} \ \underbrace{\frac{(m+n-1)}{q}\pa_y\mathring\omega_{m,n-1;k}}_{=S_{m,n-1}^{(1,1,0)}\omega_k}\\
\n &+\frac{m+n}{m+n-1}\pa_y\lf(q\vyn\rg)\underbrace{\frac{m+n-1}{q}\pa_y\mathring\omega_{m,n-1;k}}_{=S_{m,n-1}^{(1,1,0)}\omega_k}+(m+n)v_y^{-1}\pa_{yy}\mathring\omega_{m,n-1;k}\\
 & +({m+n}) ikt\pa_{y}\mathring\omega_{m,n-1;k}+\frac{m+n}{m+n-1}q'i\frac{k}{|k|}t\underbrace{\frac{(m+n-1)}{q}|k|\mathring\omega_{m,n-1;k}}_{=S_{m,n-1}^{(1,0,1)}\omega_k}.\label{IJ_1-6} 
\end{align} The treatment of these terms are similar to the treatment in Step $1$. We invoke the bound of $\|{v_y^{-1}}\|_\infty,\, \|q v_{yy}\|_\infty,  \, \|q'\|_\infty$, the estimate \eqref{Kest}, the definition of $\bold{a}_{m,n}$ \eqref{a:weight}, and the bound \eqref{gevbd1212}, to obtain the following estimates for $n\geq 3$, 
\ifx
\begin{align}
\n \bold{a}_{m,n}\nu \|\mathcal{I}_{J;1}e^W\chi_{n-1}\|_2\lesssim &
\bold{a}_{m,n}\nu \|q\pa_y(q'\vyn)\|_\infty\|L_{k,n-1}e^W\chi_{n-1}\|_2\\
\lesssim & n^{-\sigma-\sigma_\ast}{\bf a}_{m,n-1;k}\nu \|L_{k,n-1}e^W\chi_{n-2}\|_2;\label{IJ_1}\\
\n \bold{a}_{m,n}\nu \|\mathcal{I}_{J;2}e^W\chi_{n-1}\|_2\lesssim &
 n \bold{a}_{m,n}\nu \|q'\vyn\|_\infty\|J_{k,n-1}e^W\chi_{n-1}\|_2\\
\lesssim & n^{-\sigma-\sigma_\ast}{\bf a}_{m,n-1;k}\nu \|L_{k,n-1}e^W\chi_{n-2}\|_2;\label{IJ_2}\\
\n \bold{a}_{m,n}\nu \|\mathcal{I}_{J;3}e^W\chi_{n-1}\|_2\lesssim &
\bold{a}_{m,n}\nu \|\pa_y(q'\vyn)\|_\infty\|J_{k,n-1}e^W\chi_{n-1}\|_2\\
\lesssim & n^{-\sigma-\sigma_\ast}{\bf a}_{m,n-1;k}\nu \|J_{k,n-1}e^W\chi_{n-2}\|_2\label{IJ_3}. 
\end{align}
To estimate the remaining three terms in \eqref{IJ_1-6}, we further recall the the bound $\varphi t\leq C$ \eqref{ineq:varphi} and the definition of the dissipation terms $\mathcal{D}$ \eqref{defn:D:1}, to obtain the following
\begin{align}
\n \bold{a}_{m,n}\nu \|\mathcal{I}_{J;4}e^W\chi_{n-1}\|_2\lesssim &
n\bold{a}_{m,n}\nu \|\vyn\|_\infty\|\pa_{yy}\omega_{k,n-1}e^W\chi_{n-1}\|_2\\
\lesssim & n^{-\sigma-\sigma_\ast}{\bf a}_{m,n-1;k}\mathcal{D}_{m,n-1;k}^{(\al)};\label{IJ_4}\\
\n \bold{a}_{m,n}\nu \|\mathcal{I}_{J;5}e^W\chi_{n-1}\|_2\lesssim &
n B_{m,n}|k|^{m} \varphi^{m+n-1} \nu (t \varphi)\||k|\pa_y\omega_{k,n-1}e^W\chi_{n-1}\|_2\\
\lesssim & n^{-\sigma-\sigma_\ast}\nu \||k|\pa_y\omega_{k,n-1}e^W\chi_{n-2}\|_2\lesssim n^{-\sigma-\sigma_\ast}\mathcal{D}_{m,n-1;k}^{(\al)};\label{IJ_5}\\
\n \bold{a}_{m,n}\nu \|\mathcal{I}_{J;6}e^W\chi_{n-1}\|_2\lesssim &
{B}_{m,n;k}\varphi^{m+n-1} \nu^{\frac{5}{6}}|k|^{m+\frac{2}{3}}(t\varphi)\|q'\|_\infty\|S_{k,n-1}e^W\chi_{n-1}\|_2\\
\lesssim & n^{-\sigma-\sigma_\ast}{\bf a}_{m,n-1;k}\nu \||k|S_{k,n-1}e^W\chi_{n-2}\|_2\label{IJ_6}. 
\end{align}
Combining the estimates \eqref{J_2}, \eqref{IJ_1}, \eqref{IJ_2}, \eqref{IJ_3}, \eqref{IJ_4}, \eqref{IJ_5}, \eqref{IJ_6}, and the decomposition \eqref{IJ_1-6},  we obtain the following recurrence relation for $n\geq 3$:
$\n {\bf a}_{m,2}&\nu \|J_{k,2}e^W\chi_{m+1}\|_2\lesssim\sum_{\ell=0}^1\sum_{\iota\in\{\al,\mu\}}\sqrt{\mathcal D^{(\iota)}_{m,\ell;k}};\\ $
\fi
\begin{align}
\n &\bold{a}_{m,n}\nu \|S_{m,n}^{(1,1,0)}\omega_k\ e^W\chi_{m+n-1}\|_2\\
& \lesssim \frac{\lambda^{s}}{(m+n)^{\sigma+\sigma_\ast}}(\mathcal{D}^{(\al)}_{m,n-1;k}+\mathcal{D}^{(\mu)}_{m,n-1;k})^{1/2}+ \frac{\lambda^{s}{\bf a}_{m,n-1}}{(m+n)^{\sigma+\sigma_\ast}}\nu \sum_{a+b=1}\|S_{m,n-1}^{(1+a,b,0)}\omega_k e^W\chi_{m+n-2}\|_2.\label{Itrtn2}
\end{align}

\noindent 
{\bf Step 5: Iteration. }Application of \eqref{Kest}, \eqref{Kest_sum}, \eqref{Itrtn}, \eqref{J_2} and \eqref{Itrtn2} gives us \eqref{S:est:2}. Now an argument as in \eqref{sqr_upgrade} yields \eqref{SnBd2sum}. \siming{Check!?}
\ifx {\bf Previous: }We compute the following identity: 
\begin{align} \nonumber
J_{n,k} := & n \frac{1}{q} \pa_y \mathring{\omega}_{m,n;k} \\ \nonumber
= & n \frac{1}{q} \pa_y \{ q^n \Gamma_k^n \omega_k \} = n \frac{1}{q} \pa_y \{ q q^{n-1} \Gamma_k^n \omega_k \} \\ \nonumber
= & n \pa_y \{ q^{n-1} \Gamma_k^n \omega_k \} + n  q' q^{n-2} \Gamma_k^n \omega_k \\ \nonumber
= & n \pa_y \{ q^{n-1} (\pa_y + i k t) \Gamma_k^{n-1} \omega_k \} + n  q' q^{n-2} \Gamma_k^n \omega_k \\ \nonumber
= & n \pa_y^2\{ q^{n-1} \Gamma_k^{n-1} \omega_k \} - n(n-1) \pa_y \{ q' q^{n-2} \Gamma_k^{n-1} \omega_k \} + i k t n \pa_y \omega_{k,n-1} + n q' q^{n-2} \Gamma_k^n \omega_k \\ \nonumber
= & n \pa_y^2 \omega_{k,n-1} - n(n-1) q'' q^{n-2} \Gamma_k^{n-1} \omega_k - n(n-1)(n-2) |q'|^2 q^{n-3} \Gamma_k^{n-1} \omega_k \\ \nonumber
& - n(n-1) q' q^{n-2} \Gamma_k^n \omega_k + n(n-1) q' q^{n-2} (ikt) \Gamma_k^{n-1} \omega_k + i k t n \pa_y \omega_{k,n-1} + n q' q^{n-2} \Gamma_k^n \omega_k \\ \nonumber
= &  n \pa_y^2 \omega_{k,n-1} - n|q'|^2 L_{n-1} - q' L_n + n(n-1) q' q^{n-2} (ikt) \Gamma_k^{n-1} \omega_k \\ \n
&   + i k t n \pa_y \omega_{k,n-1} + \frac{1}{n-1} q' L_{n}- n(n-1) q'' q^{n-2} \Gamma_k^{n-1} \omega_k \\ \n
= & n \pa_y^2 \omega_{k,n-1}- n |q'|^2 L_{n-1,k}  + \Big( \frac{2-n}{n-1} \Big) |q'|^2 L_{n,k} + q' (ikt) \frac{n(n-1)}{q} \omega_{k,n-1} \\ \label{Jn:exp}
& + (ikt) n \pa_y \omega_{k,n-1}-\frac{ n(n-1)}{q} q''  \omega_{k,n-1} = \sum_{i = 1}^6 J_{n,k}^{(i)}.
\end{align}
We estimate each of these terms individually. The primary term is the first, which we bound as follows:
\begin{align*}
&B_{m,n} \varphi^{m+n} |k|^{m-\frac13} \nu^{\frac56} \| J_{n,k}^{(1)} e^W e^{\delta_E \nu^{\frac13}t} \chi_n\|_{L^2} \\
 \lesssim & B_{n-1, m} \varphi^{m+n} |k|^{m-\frac13} \nu^{\frac56} \|  \pa_y^2 \omega_{k,n-1} e^W e^{\delta_E \nu^{\frac13}t} \chi_n\|_{L^2} \lesssim \mathcal{D}^{(\alpha)}_{n-1, m, k}(t). 
\end{align*}
The $L_{n-1,k}, L_{n,k}$ contributions are easily bounded upon invoking \eqref{LnBound}. For the contribution $L_{n,k}^{(4)}$, we need to invoke our $S_{n,k}$ bound, \eqref{SnBd2} via, 
\begin{align*}
&B_{m,n} \varphi^{m+n} |k|^{m-\frac13} \nu^{\frac56} e^{\delta_E \nu^{\frac13}t}  \| J^{(4)}_{k,n} e^W \chi_n\|_{L^2} \\
\lesssim & B_{m,n} n \varphi^{m+n-1} (\varphi t) |k|^{m+\frac13} \nu^{\frac56} e^{\delta_E \nu^{\frac13}t} \|   \frac{n-1}{q} \omega_{k,n-1} e^W  \chi_n\|_{L^2} \\
\lesssim & B_{n-1, m}  \varphi^{m+n-1}  |k|^{m+\frac13} \nu^{\frac56}  e^{\delta_E \nu^{\frac13}t} \|  S_{n-1,k} e^W \chi_n\|_{L^2} \lesssim  \mathcal{D}^{(\mu)}_{n-1, m, k}(t).
\end{align*} 
For the contribution of $L^{(5)}_{n,k}$, we have 
\begin{align*}
&B_{m,n} \varphi^{m+n} |k|^{m-\frac13} \nu^{\frac56} e^{\delta_E \nu^{\frac13}t}  \| L^{(5)}_{k,n} e^W \chi_n\|_{L^2} \\
\lesssim & B_{n-1, m} \varphi^{m+n-1} |k|^{m+\frac23} \nu^{\frac56} e^{\delta_E \nu^{\frac13}t}  \| \nu^{\frac56} \pa_y \omega_{k,n-1} e^W \chi_n\|_{L^2} \lesssim \mathcal{D}^{(\alpha)}_{m,n-1;k}.
\end{align*}
The final term $L^{(6)}_{n,k}$ follows in a straightforward manner: it is lower order and can be controlled easily using \eqref{S:est:1}. The proof is complete. \fi
\end{proof}

\subsection{Commutator Estimates, $\bold{C}^{(m,n)}_{q,k}$}

We now obtain commutator bounds on the various contributions of $\bold{C}^{(m,n)}_{q,k}$ which appear in our $\gamma/\mu/\alpha$ energy scheme. First, we have 
\begin{lemma} \label{lem:tm:1} The following bound is valid:
\begin{align} \label{CommCnqbd}
\bold{a}_{m,n}^2 \lf|\mathrm{Re} \lf\langle \mathring{\omega}_{m,n;k}, \bold{C}^{(m,n)}_{q,k} e^{2W}  \chi_{m+n}^2\rg \rangle\rg| \leq \frac{1}{B}\mathcal{D}^{(\gamma)}_{m,n;k}+CB\sum_{\ell=0}^{n-1}(\mathfrak C\lambda^{s})^{2n-2\ell} \lf(\frac{(m+n)!}{(m+\ell)!}\rg)^{-2\sigma-2\sigma_\ast}  \mathcal{D}^{(\gamma)}_{m,\ell;k}.
\end{align}
Here $B>1$ is an arbitrary constant and $C$ is universal constant. 
\ifx 
Moreover, for arbitrary $1\leq M\in \mathbb{N},$ and $\lambda$ chosen small enough compared to constants that only depend on $s$, the following estimate holds
\begin{align}\label{CommCnqbdsum}
 \sum_{m+n=0}^M\bold{a}_{m,n}^2   \lf|\mathrm{Re} \lf \langle   \mathring{\omega}_{m,n;k},  \bold{C}^{(m,n)}_{q,k} e^{2W} \chi_{m+n}^2 \rg\rangle\rg|  \leq \lambda\sum_{m+n=0}^{M}  \mathcal{D}^{(\gamma)}_{m,n;k}.
\end{align}
\fi 
\end{lemma}
\begin{proof} We distinguish between the $n=1$ case and the $n\geq 2$ case. 

For the $n=1$ case, we invoke the definition of $S_{m,1}^{(1,0,0)}\omega_k$ \eqref{S_nq} and have that 
\begin{align}
{\bf C}_{q,k}^{(m,1)}=\nu[q,\pa_{yy}]\Gamma_k|k|^m\omega_k=&-2\nu q'\pa_y \Gamma_k|k|^m\omega_k-\nu q''\Gamma_k|k|^m\omega_k
=-2\nu q'\pa_y\Gamma_k|k|^m\omega_k-\nu q'' S_{m,1}^{(1,0,0)}\omega_k. \label{Cq1}
\end{align}
Next we invoke the definition $\mathcal{D}$ \eqref{defn:D:1}, together with the bound of $S_{m,1}^{(1,0,0)}\omega_k$ \eqref{S:est:1}, to estimate the left hand side of \eqref{CommCnqbd} as follows
\begin{align*}
\bold{a}_{m,1}^2& \lf|\mathrm{Re} \lf\langle \mathring{\omega}_{m,1;k}, \bold{C}^{(m,1)}_{q,k} e^{2W}  \chi_{m+1}^2\rg \rangle \rg|\\
 =& {\bf a}_{m,1;k}^2\nu \lf|\mathrm{Re} \lf\lan \mathring{\omega}_{m,1;k}, \lf(-2\nu q'\pa_y\lf(q^{-1}\mathring{\omega}_{m,1;k}\rg)-\nu q'' S_{m,1}^{(1,0,0)}\omega_k\rg) e^{2W}\chi_{m+1}^2\rg\ran\rg| \\
\lesssim&{\bf a}_{m,1}^2\nu\|q^{-1}\mathring{\omega}_{m,1;k}e^W\chi_{m+1}\|_2 \|\pa_{y}\mathring{\omega}_{m,1;k}e^W\chi_{m+1}\|_2+{\bf a}_{m,1;k}^2\nu\|q^{-1}\mathring{\omega}_{m,1;k}e^W\chi_{m+1}\|_2^2 \\ 
\lesssim&\mathcal{D}_{m,0;k}^{(\gamma)}+{\bf a}_{m,1;k}^2\nu \|S_{m,1}^{(1,0,0)}\omega_k\ e^{W}\chi_{m+1}\|_2^2 \lesssim \mathcal{D}_{m,0;k}^{(\gamma)}. 
\end{align*}
This is estimate \eqref{CommCnqbd} for $n=1.$

We invoke $\eqref{Cnq}_{n\geq 2}$ to produce the identity 
\begin{align}\n
 & \bold{a}_{m,n}^2  \mathrm{Re} \lf\langle \mathring{\omega}_{m,n;k}, \bold{C}^{(m,n)}_{q,k} e^{2W}  \chi_{m+n}^2 \rg\rangle   \\ \n
  &= \bold{a}_{m,n}^2\mathrm{Re}\lf\langle \mathring{\omega}_{m,n;k},\lf(-2 \nu q' \frac{n}{q} \pa_y {\omega}_{m,n;k} +\nu (q')^2 \frac{n(n+1)}{q^2} {\omega}_{m,n;k} -\nu q'' \frac{n}{q} { \omega}_{m,n;k}\rg) e^{2W}  \chi_{m+n}^2\rg \rangle \\
&= :  \sum_{j = 1}^3 \mathcal{T}_{j}.\label{I_Cq} 
\end{align}
We now estimate each of the terms appearing above. First,  thanks to the definition \eqref{S_nq}, the estimate \eqref{S:est:1} and the combinatorial estimate \eqref{sum_comb}, we observe that
\begin{align*}
|\mathcal{T}_{1}| \lesssim & \bold{a}_{m,n}^2 \| \nu^{\frac12} \pa_y \mathring{\omega}_{m,n;k} e^W \chi_{m+n} \|_{L^2}  \| \nu^{\frac12} {S}_{m,n;k} e^W \chi_{m+n} \|_{L^2} \\
\lesssim & \sqrt{\mathcal{D}_{m,n;k}^{(\gamma)}}\bold{a}_{m,n} \| \nu^{\frac12} S_{m,n}^{(1,0,0)}\omega_k e^W \chi_{m+n} \|_{L^2}  \\
  \lesssim &  \sqrt{\mathcal{D}^{(\gamma)}_{m,n;k}} \lf(\sum_{\ell=0}^{n-1}( \mathfrak C\lambda^{s})^{2n-2\ell}\lf(\frac{(m+n)!}{(m+\ell)!}\rg)^{-2\sigma-2\sigma_\ast}  \mathcal{D}^{(\gamma)}_{m,\ell;k}\rg)^{\frac{1}{2}}.
\end{align*}
An application of the Young's inequality yields that the first term can be estimated as in \eqref{CommCnqbd}. 
The second and last terms are estimated in a similar fashion, 
\begin{align*}
|\mathcal{T}_{2}| \lesssim &\bold{a}_{m,n}^2 \| \nu^{\frac12} S_{m,n;k} e^W \chi_{m+n} \|_{L^2}^2\lesssim \sum_{\ell=0}^{n-1}( \mathfrak C\lambda^{s})^{2n-2\ell} \lf(\frac{(m+n)!}{(m+\ell)!}\rg)^{-2\sigma-2\sigma_\ast} \mathcal{D}^{(\gamma)}_{m,\ell;k};\\
|\mathcal{T}_{3}| \lesssim & \bold{a}_{m,n}^2 \| \nu^{\frac12} \mathring{\omega}_{m,n;k} e^W \chi_{m+n} \|_{L^2}  \| \nu^{\frac12} {S}_{m,n}^{(1,0,0)}\omega_k e^W \chi_{m+n} \|_{L^2} \\
\lesssim & \bold{a}_{m,n}^2 \| \nu^{\frac12} S_{m,n}^{(1,0,0)}\omega_k e^W \chi_{m+n} \|_{L^2}^2    \lesssim \sum_{\ell=0}^{n-1}( \mathfrak C\lambda^{s})^{2n-2\ell} \lf(\frac{(m+n)!}{(m+\ell)!}\rg)^{-2\sigma-2\sigma_\ast}  \mathcal{D}^{(\gamma)}_{m,\ell;k}.
\end{align*}
Combining the estimates above and the decomposition \eqref{I_Cq} yields  \eqref{CommCnqbd} for $n\geq 2$.
\ifx 
Finally, we derive the estimate \eqref{CommCnqbdsum} with a similar argument as in \eqref{sqr_upgrade}. By setting $B=\lambda^{-s},\, s>1$ in the expression \eqref{CommCnqbd}, and sum in the index $m,n$, we have that
\begin{align*}
&\sum_{m+n=0}^M\bold{a}_{m,n}^2 \lf|\mathrm{Re} \lf\langle \mathring{\omega}_{m,n;k}, \bold{C}^{(m,n)}_{q,k} e^{2W}  \chi_{m+n}^2\rg \rangle\rg| \\
&\leq \lambda^s\mathcal{D}^{(\gamma)}_{m,n;k}+C\mathfrak C^2\lambda^{s}\sum_{m=0}^M\sum_{n=1}^{M-m} \sum_{\ell=0}^{n-1}(\mathfrak C\lambda^{s})^{2(n- \ell-1)} \lf(\frac{(m+n)!}{(m+\ell)!}\rg)^{-2\sigma-2\sigma_\ast}  \mathcal{D}^{(\gamma)}_{m,\ell;k}\\
&\leq \lambda^s\mathcal{D}^{(\gamma)}_{m,n;k}+C\mathfrak C^2\lambda^{s}\sum_{m=0}^M\sum_{\ell=0}^{M-m-1} \mathcal{D}^{(\gamma)}_{m,\ell;k}\lf(\sum_{n=\ell+1}^{M-m} (\mathfrak C\lambda^{s})^{2(n- \ell-1)}\rg) .
\end{align*}
Now we can choose the $\mathfrak C\lambda^s\leq \frac{1}{2}$, and $\lambda^{s-1}$ to be small enough depending on $\mathfrak{C},\ s$ and constants to get
\begin{align}
\sum_{m+n=0}^M\bold{a}_{m,n}^2 \lf|\mathrm{Re} \lf\langle \mathring{\omega}_{m,n;k}, \bold{C}^{(m,n)}_{q,k} e^{2W}  \chi_{m+n}^2\rg \rangle\rg|\leq \lambda^s \sum_{m=0}^M\sum_{\ell=0}^{M-m} \mathcal{D}^{(\gamma)}_{m,\ell;k}\leq \lambda^s \sum_{m+n=0}^M \mathcal{D}^{(\gamma)}_{m,n;k}.\label{sqr_upgrd}
\end{align}
\fi 
As a result, the proof of the lemma is completed. 

\end{proof}

\begin{lemma} The following bound is valid:  
\begin{align} \n
   \bold{a}_{m,n}^2&\nu\lf |\mathrm{Re} \lf \langle \pa_y \mathring{\omega}_{m,n;k},  \pa_y \bold{C}^{(m,n)}_{q,k} e^{2W} \chi_{m+n}^2 \rg\rangle\rg| \\
   \leq&\frac{1}{B}\mathcal{D}_{m,n;k}^{(\al)} +CB \sum_{\ell=0}^{n-1}\sum_{\iota \in \{ \alpha, \mu \}}  ( \mathfrak C\lambda^{s})^{2n-2\ell}\lf(\frac{(m+n)!}{(m+\ell)!}\rg)^{-\sigma_\ast} \mathcal{D}^{(\iota)}_{m,\ell;k}+C\lambda^{2s}\mathcal{CK}_{m,0;k}^{(\al,W)}\myr{\mathbbm{1}_{n=1}}.\label{CommCnqBd1} 
\end{align}
Here $B>1$ is an arbitrary constant and $C$ is universal constant. 
\ifx 
 Moreover, for arbitrary $M\in \mathbb{N},$ and $\lambda$ chosen small enough compared to constants that only depend on $s$, the following estimate holds
\begin{align}\label{CommCnqBd1sum}
 \sum_{m+n=0}^M\bold{a}_{m,n}^2 \nu \lf|\mathrm{Re} \lf \langle \pa_y \mathring{\omega}_{m,n;k},  \pa_y \bold{C}^{(m,n)}_{q,k} e^{2W} \chi_{m+n}^2 \rg\rangle\rg|  \leq \lambda\sum_{m+n=0}^{M}\sum_{\iota \in \{ \alpha, \mu \}}  \mathcal{D}^{(\iota)}_{m,n;k}+\lambda\sum_{m=0}^M\mathcal{CK}_{m,0;k}^{(\al,W)}.
\end{align}\fi 
\end{lemma}
\begin{proof} We distinguish between the $n=1$ case and the $n\geq 2$ case. 

To estimate the $n=1$ case, we expand the $\pa_y {\bf C}_{q,k}^{(1)}$ with \eqref{Cq1} as follows
\begin{align}
\pa_y {\bf C}_{q,k}^{(1)}=-2\nu q'\pa_{yy}\Gamma_k\omega_k-3\nu q''\pa_y\Gamma_k\omega_k-\nu q'''\Gamma_k\omega_k.\label{pyCq1}
\end{align}
Based on this, we decompose the left hand side as follows
\begin{align}
{\bf a}_{m,1;k}^2&\nu^2\lf|\mathrm{Re} \lf \lan \pa_y\mathring{\omega}_{m,1;k},\pa_y{\bf C}_{q,k}^{(m,1)}\chi_{m+1}^2\rg\ran \rg| \n \\ \n
\leq &
{\bf a}_{m,1;k}^2\nu^2\lf|\mathrm{Re} \lf \lan \pa_y\mathring{\omega}_{m,1;k},2\nu q'\pa_{yy}\Gamma_k|k|^m\omega_k e^{2W}\chi_{m+1}^2\rg\ran\rg|\\ 
&+{\bf a}_{m,1;k}^2\nu^2\lf|\mathrm{Re} \lf \lan \pa_y\mathring{\omega}_{m,1;k},3\nu q''\pa_y\Gamma_k|k|^m\omega_ke^{2W}\chi_{m+1}^2\rg\ran\rg|\n \\ 
&+ {\bf a}_{m,1;k}^2\nu^2\lf|\mathrm{Re} \lf \lan \pa_y\mathring{\omega}_{m,1;k},\nu q'''\Gamma_k|k|^m\omega_k e^{2W}\chi_{m+1}^2\rg\ran\rg| 
=: \sum_{j=1}^3\mathcal{T}_{1;j}.\label{al_pyCq}
\end{align}
To estimate the first term, we implement an integration by parts  to estimate the $\mathcal{T}_{1;1}$-term in \eqref{al_pyCq} as follows
\begin{align}\n
\mathcal{T}_{1;1}\leq& C{\bf a}_{m,1}^2\nu^2 \\ \n
&\times \bigg(\lf|\Re\lf\lan -q'\pa_{yy}\mathring{\omega}_{m,1;k}-q''\pa_y|k|^m\omega_k,\pa_y\Gamma_k|k|^m\omega_k e^{2W}\chi_{m+1}^2 \rg\ran\rg|\\
\n &+\lf|\mathrm{Re} \lf(\pa_y\mathring{\omega}_{m,1;k} q'\overline{\pa_y\Gamma_k|k|^m\omega_k}e^{2W}\chi_{m+1}^2\rg)\bigg|_{y=-1}^{y=1}\rg|+\lf|\mathrm{Re} \lf\lan  q'\pa_y|k|^m\omega_k,\pa_y\Gamma_k|k|^m\omega_k( \pa_y W)e^{2W}\chi_{m+1}^2 \rg\ran\rg|\n\\
\n &+\lf|\mathrm{Re} \lf\lan  q'\pa_y|k|^m\omega_k,\pa_y\Gamma_k|k|^m\omega_k( \chi' _{m+1})e^{2W}\chi_{m+1} \rg\ran\rg|\bigg)\\
=:&\sum_{j=1}^4\mathcal{T}_{1;1j}.\label{al_pyCq1}
\end{align}
The estimate of the first term is a direct consequence of the definitions \eqref{defn:D:1}, \eqref{a:weight} and the bounds that $t\varphi\leq C$ \eqref{ineq:varphi}, $\|\vyn\|_\infty+\|v_{yy}\|_\infty\leq C$ \eqref{v_y_asmp}, 
\begin{align}
\n \mathcal{T}_{1;11}\leq& C{\bf a}_{m,1}^2\nu^2\\
\n &\times\lf|\mathrm{Re} \lf\lan q'\pa_{yy}\mathring{\omega}_{m,1;k}+q''\pa_y\mathring{\omega}_{m,1;k},(-v_y^{-2}v_{yy}\pa_y|k|^m\omega_k+\vyn \pa_{yy}|k|^m\omega_k+ikt\pa_y|k|^m\omega_k) e^{2W}\chi_{m+1}^2\rg\ran\rg|\\
\leq&\frac{1}{B} \mathcal{D}_{m,1;k}^{(\al)}+CB\lambda^{2s} \mathcal{D}_{m,0;k}^{(\al)}.\label{al_pyCq11} 
\end{align}
Here $B>0$ is an arbitrary constant. To estimate the second term in \eqref{al_pyCq1}, we apply 
 the boundary conditions $\omega_k(t,y=\pm 1)=\pa_{yy}\omega_k(t, y=\pm 1)=0$ and the estimate \eqref{BC_n=1}, $\|\vyn\|_\infty+\|v_{yy}\|_\infty\leq C$, to obtain that
\begin{align} 
\n \mathcal{T}_{1;12}\lesssim& {\bf a}_{m,1}^2\nu^2|k|^{2m}\lf|\mathrm{Re}  |q'|^2\lf(\vyn\pa_{y}\omega_k+ikt\omega_k\rg)\overline{\lf(\vyn\pa_{yy}\omega_k-v_y^{-2}v_{yy}\pa_y \omega_k +ikt\pa_y\omega_k\rg)}e^{2W}\chi_{m+1}^2\bigg|_{y=-1}^{y=1}\rg|\\
\lesssim& {\bf a}_{m,1}^2\nu^2|k|^{2m}\lf|\mathrm{Re}  \lf(\vyn\pa_{y}\omega_k\rg)\lf(\overline{v_y^{-2}v_{yy}\pa_y \omega_k} +ikt\overline{\pa_y\omega_k}\rg)e^{2W}\chi_{m+1}^2\bigg|_{y=-1}^{y=1}\rg| \n\\ 
\lesssim&{\bf a}_{m,1}^2\nu^2\bigg|v_y^{-3} v_{yy} |\pa_y\mathring\omega_{m,0;k}|^2e^{2W}\chi_{m+1}^2\bigg|_{y=-1}^{y=1}\bigg| \lesssim \lambda^{2s}\mathcal{D}_{m,0;k}^{(\al)}+\lambda^{2s}\mathcal{CK}_{m,0;k}^{(\al,W)}.
\end{align}
Upon invoking the estimates \eqref{W_prop}, \eqref{ineq:varphi}, and the definitions \eqref{a:weight}, \eqref{defn:D:1}, we have that the third term in \eqref{al_pyCq1} can be estimated as follows,
\begin{align}
\n \mathcal{T}_{1;13}\lesssim& \lambda^{2s}\mathcal{D}_{m,0;k}^{(\al)}+\lambda^{2s}\mathcal{CK}_{m,0;k}^{(\al,W)}. 
\end{align} 
For the fourth term in \eqref{pyCq1}, we invoke the relations \eqref{s:prime}, \eqref{chi:prop:3} and \eqref{gevbd1212} to obtain
\begin{align}
\n\mathcal{T}_{1;14}\lesssim& {\bf a}_{m,1}^2\nu^{}\| \pa_y \mathring\omega_{m,0;k}e^{W}\chi_m\|_2({m+1})^{1+\sigma}\|\pa_y(\pav+ikt)\mathring\omega_{m,0;k}e^{W}\chi_{m}\|_2\\ \n
\lesssim& \lambda^{s}{\bf a}_{m,0} \nu^{1/2}\| \pa_y \mathring\omega_{m,0;k}e^{W}\chi_m\|_2\frac{\lambda^{s}{\bf a}_{m,0} \nu^{1/2}}{({m+1})^{\sigma_\ast}}(\|\pa_y \pav \mathring\omega_{m,0;k}e^{W}\chi_{m}\|_2+\|\pa_y|k|\mathring\omega_{m,0;k}e^{W}\chi_{m}\|_2)\\
\lesssim& \frac{\lambda^{2s}}{({m+1})^{\sigma_\ast}}\mathcal{D}_{m,0;k}^{(\al)}. 
\end{align}
This concludes the proof of the estimate \eqref{CommCnqBd1} in the $n=1$ case.

In the $n\geq 2$ case, we decompose the integrand on the left-hand side of \eqref{CommCnqBd1} according to the identity \eqref{IDL12}, which gives 
\begin{align} \n
  \bold{a}_{m,n}^2& \nu\lf| \mathrm{Re} \lf\langle \pa_y \mathring{\omega}_{m,n;k},  \pa_y \bold{C}^{(m,n)}_{q,k} e^{2W} \chi_{m+n}^2 \rg\rangle\rg| \\ \n
\leq &  \bold{a}_{m,n} ^2 \nu^2 \lf|\mathrm{Re} \lf\langle \pa_y \mathring{\omega}_{m,n;k}, \frac{n}{q} \Upsilon^{(1)}_n \pa_y^2  {\omega}_{m,n;k} e^{2W} \chi_{m+n}^2 \rg\rangle \rg|\\ \n
&+ \bold{a}_{m,n}^2 \nu^2 \lf|\mathrm{Re} \lf\langle \pa_y \mathring{\omega}_{m,n;k}, \frac{n^2}{q^2} \Upsilon^{(2)}_n \pa_y  {\omega}_{m,n;k} e^{2W} \chi_{m+n}^2 \rg\rangle\rg |\\ 
&+ \bold{a}_{m,n} ^2 \nu^2 \lf| \mathrm{Re} \lf\langle \pa_y \mathring{\omega}_{m,n;k}, \frac{n^3}{q^3} \Upsilon^{(3)}_n  {\omega}_{m,n;k} e^{2W} \chi_{m+n}^2\rg \rangle\rg| 
=: \sum_{j = 1}^{3} \mathcal{T}_{2;j}.\label{I_paC}
\end{align}
Here, the $\{\Upsilon^{(i)}\}_{i=1}^3$ are defined in \eqref{IDL12} and \eqref{Ups:1}--\eqref{Ups:3}. We bound each of these contributions individually, starting with $\mathcal{T}_{2;1}$. By invoking  the definition of $\mathcal{D}^{(\al)}$ \eqref{defn:D:1} and \eqref{SJ}, the $S_{m,n}^{(1,1,0)}\omega_k$-estimate \eqref{S:est:2} and the combinatorial estimate \eqref{sum_comb}, we obtain that
\begin{align*}
\mathcal{T}_{2;1} \lesssim &  \bold{a}_{m,n}^2 \| \Upsilon^{(1)}_n \|_{L^\infty} \lf \|  \nu \frac{n}{q} \pa_y \mathring{\omega}_{m,n;k}e^W \chi_{m+n} \rg\|_{L^2}\| \nu  \pa_y^2 \mathring{\omega}_{m,n;k} e^W \chi_{m+n} \|_{L^2} \\
\lesssim &  \lf(\bold{a}_{m,n} \|  \nu S_{m,n}^{(1,1,0)}\omega_k e^W \chi_{m+n} \|_{L^2}\rg)\sqrt{ \mathcal{D}_{m,n;k}^{(\alpha)}} \\
\lesssim &  \lf(\sum_{\ell=0}^{n-1}(\mathfrak C \lambda^{s})^{2n-2\ell}\lf(\frac{(m+n)!}{(m+\ell)!}\rg)^{-2\sigma-2\sigma_\ast} \lf(\mathcal{D}^{(\alpha)}_{m,\ell;k} +\mathcal{D}^{(\mu)}_{m,\ell;k}\rg)\rg)^{\frac{1}{2}} \sqrt{\mathcal{D}_{m,n;k}^{(\alpha)}}.
\end{align*}
Next, we apply similar argument to obtain that
\begin{align*}
  \mathcal{T}_{2;2}  \lesssim &  \bold{a}_{m,n}^2 \| \Upsilon^{(2)}_n \|_{L^\infty} \lf \|  \nu \frac{n}{q} \pa_y \mathring{\omega}_{m,n;k}e^W \chi_{m+n}\rg \|_{L^2} ^2
\lesssim  \bold{a}_{m,n} ^2 \| \Upsilon^{(2)}_n \|_{L^\infty}  \|  \nu S_{m,n}^{(1,1,0)}\omega_k  e^W \chi_{m+n} \|_{L^2} ^2 \\
\lesssim & \sum_{\ell=0}^{n-1}(\mathfrak C \lambda^{s})^{2n-2\ell}\lf(\frac{(m+n)!}{(m+\ell)!}\rg)^{-2\sigma-2\sigma_\ast} \lf(\mathcal{D}^{(\alpha)}_{m,\ell;k} + \mathcal{D}^{(\mu)}_{m,\ell;k}\rg). 
\end{align*}
Finally, by invoking  the definition of $\mathcal{D}^{(\al)}$ \eqref{defn:D:1} and $S_{m,n}^{(a,b,c)}\omega_k$ \eqref{S_nq}, the estimates  \eqref{S:est:1}, \eqref{S:est:2} and the combinatorial estimate \eqref{sum_comb}, we obtain that 
\begin{align*}
 \mathcal{T}_{2;3} \lesssim & \bold{a}_{m,n}^2 \| \Upsilon^{(3)}_n \|_{L^\infty}\lf \| \nu  \frac{n}{q} \pa_y \mathring{\omega}_{m,n;k}e^W \chi_{m+n} \rg\|_{L^2} \lf \| \nu  \frac{n^2}{q^2}   {\omega}_{m,n;k} e^W \chi_{m+n}\rg \|_{L^2} \\
\lesssim & \bold{a}_{m,n}^2 \| \Upsilon^{(3)}_n \|_{L^\infty} \| \nu  S_{m,n}^{(1,1,0)}\omega_k e^W \chi_{m+n} \|_{L^2} \| \nu  {S}_{m,n}^{(2,0,0)}\omega_k e^W \chi_{m+n} \|_{L^2} \\
\lesssim & \sum_{\ell=0}^{n-1}(\mathfrak C \lambda^{s})^{2n-2\ell}\lf(\frac{(m+n)!}{(m+\ell)!}\rg)^{-2\sigma-2\sigma_\ast} \lf(\mathcal{D}^{(\alpha)}_{m,\ell;k} + \mathcal{D}^{(\mu)}_{m,\ell;k}\rg).
\end{align*}
Combining all the estimates developed thus far, we obtain the result \eqref{CommCnqBd1}. 
\end{proof}

We now need the $\mu$ version, which follows largely in an analogous manner to the $\gamma$ term. We have 
\begin{lemma} The following bound is valid:
\begin{align}\label{beed:mu}
 \bold{a}_{m,n}^2 \nu \lf|\mathrm{Re} \lf\langle |k| \mathring{\omega}_{m,n;k}, |k| \bold{C}^{(m,n)}_{q,k} e^{2W}  \chi_{m+n}^2\rg \rangle\rg|  \leq \frac{1}{B}\mathcal{D}^{(\mu)}_{m,n;k}+CB\sum_{\ell=0}^{n-1}(\mathfrak C\lambda^s)^{2n-2\ell} \lf(\frac{(m+\ell)!}{(m+n)!}\rg)^{2\sigma_\ast}  \mathcal{D}^{(\mu)}_{m,\ell;k}.
\end{align}
Here $B>1$ is an arbitrary constant and $C,\ \mf C$ are universal constants.\ifx 
 Moreover, for arbitrary $M\in \mathbb{N},$ and $\lambda$ chosen small enough compared to constants that only depend on $s$, the following estimate holds
\begin{align}\label{beedsum:mu}
 \sum_{m+n=0}^M\bold{a}_{m,n}^2 \nu \lf|\mathrm{Re} \lf\langle |k| \mathring{\omega}_{m,n;k}, |k| \bold{C}^{(m,n)}_{q,k} e^{2W}  \chi_{m+n}^2\rg \rangle\rg|  \leq \lambda \sum_{m+n=0}^M\mathcal{D}^{(\mu)}_{m,n;k}.
\end{align}\fi 
\end{lemma}
\begin{proof} This follows in an identical manner to Lemma \ref{lem:tm:1}, with the bounds on $S_{m,n}^{(1,0,0)}\omega_k$ being replaced by the bound \eqref{Kest}.
\end{proof}

\ifx 
To conclude this section, we observe that the estimates \eqref{CommCnqbdsum}, \eqref{CommCnqBd1sum}, \eqref{beedsum:mu} are the desired bounds \eqref{Im:qu:a}, \eqref{Im:qu:d} and \eqref{Im:qu:g}.\fi 
\subsection{Commutator Estimates, $\bold{C}^{(m,n)}_{\mathrm{visc}, k}$}
\label{visc:comm:sect}
\ifx\siming{\begin{align*}
(\frac{\pa_y}{v_y})\Rightarrow&\pav,\quad {\bf a}_{n,m}\Rightarrow {\bf a}_{m,n},\quad\bold{C}^{(n)}_{visc, k}\Rightarrow\bold{C}^{(m,n)}_{\mathrm{visc}, k},\quad \bold{C}^{(n)}_{trans, k}\Rightarrow\bold{C}^{(m,n)}_{\mathrm{trans}, k}\\
\omega_{m,n;k}\Rightarrow& \omnk=q^n|k|^m\Gamma_k^n\omega_k,\quad n,m\Rightarrow m,n,\quad \tilde{\chi}_1\Rightarrow\widetilde{\chi}_1,\quad D_k\Rightarrow q\Gamma_k,\quad D_0^l\Rightarrow q^l\pav^l.
\end{align*}I hope there will be no terms like $\omega_{n,m-\ell+1}$ be changed to $\omega_{m,n-\ell+1}$.  
You can search for the $\mr{\cdots}$ symbol to find the new $\omnk. $} \siming{Change: $m\rightarrow \ell$}\fi
Recall that the commutator coming from the viscous term \eqref{ghtyU:2} reads as follows  
\begin{align*} 
	 \bold{C}^{(m,n)}_{{\rm visc}, k} :=   \nu {q^n}\sum_{\ell = 1}^n \binom{n}{\ell}  \pav^\ell v_y^2 \pav^2 \Gamma_k^{n-\ell} |k|^m\omega_k 
	=   \nu {q^n}\sum_{\ell = 1}^n \binom{n}{\ell}  \pav^\ell \paren{v_y^2-1} \pav^2 \Gamma_k^{n-\ell}|k|^m \omega_k.
\end{align*}
Then we have the following lemma for the $\gamma$ level estimate. 
\begin{lemma}[Viscous commutator for $\gamma$ estimate] 
	Let $\bold{C}^{(m,n)}_{\mathrm{visc}, k}$ be defined as above. It holds   
	\label{visc:comm}
	\begin{align}
		\label{visc:comm:esti}
			& \sum_{m,n\ge 0}\sum_{k\neq 0} \abs{\brk{ |\bold{a}_{m,n}|^2    \omnk, \bold{C}^{(m,n)}_{\mathrm{visc}, k} e^{2W}\chi_{m+n}^2}}
			\nn&\qquad
			\lesssim 	\eps  \nu^{100}
			\sqrt{\mathcal{E}^{(\gamma)}} \paren{		\sqrt{\cd^{(\gamma)}} + \sqrt{\mathcal{CK}^{(\gamma)}}}  + \eps  \nu^{100}
			\mathcal{D}^{(\gamma)}.
	\end{align}
\end{lemma}
\begin{proof}
For each fixed $ m,n\in \mathbb{N}$, we integrate by parts to obtain
\begin{align}
	\label{visc:comm:eq01}
	& {\nu^{-1}}\brk{ |\bold{a}_{m,n}|^2    \omnk,   \bold{C}^{(m,n)}_{\mathrm{visc}, k} e^{2W}\chi_{m+n}^2}
	\nn&\qquad\qquad\qquad=
	\sum_{l=1}^{n} \binom{n}{l}
	\brk{	  |\bold{a}_{m,n}|^2  \omnk,  |k|^mq^{n}\pav^l  (v_y^2 - 1) \pav^2 \Gamma_k^{n-l} \omega_k e^{2W}\chi_{m+n}^2}
	\nn&\qquad\qquad\qquad=
	-\sum_{l=1}^{n} \binom{n}{l}
	\brk{|\bold{a}_{m,n}|^2  \pav\omnk,   |k|^mq^{n}\pav^l  (v_y^2 - 1) \pav \Gamma_k^{n-l} \omega_k e^{2W}\chi_{m+n}^2  }
	\nn&\qquad\qquad\qquad\quad
	-\sum_{l=1}^{n} \binom{n}{l}
	\brk{|\bold{a}_{m,n}|^2  \omnk,   |k|^mq^{n}\pav^{l}  (v_y^2 - 1) \frac{\pa_y}{v_y} \Gamma_k^{n-l} \omega_k \frac{2\pa_yW}{v_y}e^{2W}\chi_{m+n}^2  }
	\nn&\qquad\qquad\qquad\quad
	-\sum_{l=1}^{n} \binom{n}{l}
	\brk{|\bold{a}_{m,n}|^2  \omnk,   |k|^m \pav\paren{q^{l}\pav^{l}  (v_y^2 - 1) }q^{n-l}\pav \Gamma_k^{n-l} \omega_k e^{2W}\chi_{m+n}^2  }
	\nn&\qquad\qquad\qquad\quad
	-\sum_{l=1}^{n} \binom{n}{l}
	\brk{|\bold{a}_{m,n}|^2  \omnk,   |k|^m q^{l}\pav^{l}  (v_y^2 - 1)  \pav q^{n-l}\pav \Gamma_k^{n-l} \omega_k e^{2W}\chi_{m+n}^2  }
	\nn&\qquad\qquad\qquad\quad
	-\sum_{l=1}^{n} \binom{n}{l}
	\brk{|\bold{a}_{m,n}|^2  \omnk,   |k|^m q^{n}\pav^{l}  (v_y^2 - 1) \pav \Gamma_k^{n-l} \omega_k e^{2W}
		\frac{2\pa_y\chi_{m+n}}{v_y}\chi_{m+n}  }
	\nn&\qquad\qquad\qquad\quad
	-\sum_{l=1}^{n} \binom{n}{l}
	\brk{|\bold{a}_{m,n}|^2  \omnk,  |k|^mq^{n} \partial_y\paren{\frac{1}{v_y}}\pav^{l}  (v_y^2 - 1) \pav \Gamma_k^{n-l} \omega_k e^{2W}\chi_{m+n}^2  }
	\nn&\qquad\qquad\qquad= 	V_1+V_2+V_3+V_4+ V_5 + V_6.
\end{align}
\textbf{The term $V_1$:}
	We divide the sum into two cases: 
	\begin{align*}
		V_1&=-\sum_{l=1}^{n} \binom{n}{l}
		\brk{|\bold{a}_{m,n}|^2  \pav\omnk,   |k|^mq^{n} \pav^l  (v_y^2 - 1) \pav \Gamma_k^{n-l} \omega_k e^{2W}\chi_{m+n}^2  }
		\nn&=-
		\paren{\sum_{n/2\le l \le n} + 	\sum_{0<l <n/2}} \binom{n}{l}
		\brk{|\bold{a}_{m,n}|^2  \pav\omnk,  |k|^mq^{n} \pav^l  (v_y^2 - 1) \pav \Gamma_k^{n-l} \omega_k e^{2W}\chi_{m+n}^2  }
	\end{align*}
	For the first piece, we use H\"older's inequality to arrive at 
	\begin{align*}
		V_{11}: & = \sum_{n/2\le l \le n} \binom{n}{l}
		\brk{|\bold{a}_{m,n}|^2  \pav\omnk,  |k|^mq^{n} \pav^l  (v_y^2 - 1) \pav \Gamma_k^{n-l} \omega_k e^{2W}\chi_{m+n}^2  } \\&
		 \le 
		\sum_{n/2\le l \le n} \binom{n}{l} 
		|\bold{a}_{m,n}|^2 \enorm{ \pav\omnk e^W\chi_{m+n}} \enorm{ q^{l-1}\pav^l  (v_y^2 - 1)   \chi_{l-1}\widetilde{\chi}_1 } 
		\\&\quad\times
		\norm{  |k|^mq^{n-l+1} \pav \Gamma_k^{n-l} \omega_k e^{W} \chi_{m+n}}_{L^\infty}.
	\end{align*}
	By the Sobolev inequality and $|v_y| \gtrsim 1$, it follows
	\begin{align*}
		&\norm{ |k|^mq^{n-l+1} \pav \Gamma_k^{n-l} \omega_k  e^{W} \chi_{m+n}}_{L^\infty}
		\\&\qquad\lesssim 
		\enorm{ |k|^mq^{n-l+1} \pav \Gamma_k^{n-l} \omega_k e^{W}  \chi_{m+n-l+1}} 
				\\&\qquad\quad+ 	\enorm{ |k|^m \partial_y \paren{q^{n-l+1} \pav \Gamma_k^{n-l} \omega_k e^{W}  \chi_{m+n-l+1}}}
		\\&\qquad\lesssim
		\enorm{ |k|^mq^{n-l+1} \pav \Gamma_k^{n-l} \omega_k e^{W}  \chi_{m+n-l+1}} 
						\\&\qquad\quad+ 	
		\enorm{ |k|^m \partial_y \paren{q^{n-l+1} \pav \Gamma_k^{n-l} \omega_k  } e^{W} \chi_{m+n-l+1}}
		\\&\qquad\quad 
		+
		\enorm{ |k|^m q^{n-l+1} \pav \Gamma_k^{n-l} \omega_k e^{W} \partial_y \chi_{m+n-l+1}}
							\\&\qquad\quad+ 	
		\enorm{  |k|^m q^{n-l+1} \pav \Gamma_k^{n-l} \omega_k  \partial_y W e^{W} \chi_{m+n-l+1}}
		\\&\qquad
		= \rom{1}_{1}+\rom{1}_{2}+\rom{1}_{3}+\rom{1}_{4}.
	\end{align*}
	The $\rom{1}_{1}$ term is estimated by
	\begin{align*}
		\rom{1}_{1} 
		&\lesssim
		\enorm{ |k|^m q^{n-l+1}\Gamma_k^{n-l+1} \omega_k e^{W}   \chi_{m+n-l+1}} 
						\\&\qquad+ 	\enorm{ |k|^m |kt| q^{n-l+1}\Gamma_k^{n-l} \omega_k e^{W}   \chi_{m+n-l+1}}
		\\&\lesssim
		\enorm{ |k|^m q^{n-l+1}\Gamma_k^{n-l+1} \omega_k e^{W}   \chi_{m+n-l+1}} 
		\\&\qquad+ 	t\enorm{ |k|^{m+1}  q^{n-l}\Gamma_k^{n-l} \omega_k e^{W}   \chi_{m+n-l+1}}
		\\&\lesssim
		\enorm{ |k|^m q^{n-l+1}\Gamma_k^{n-l+1} \omega_ke^{W}   \chi_{m+n-l+1}} 
		\\&\qquad+ 	\nu^{-1/3-\eta} \enorm{ |k|^{m+1}  q^{n-l}\Gamma_k^{n-l} \omega_k  e^W \chi_{m+n-l+1}}
	\end{align*}
	where we used the definition of $\Gamma$, the boundedness of $q$, and $t\le \nu^{-1/3-\eta}$. 
	The treatment of $\rom{1}_{2}$ is slightly different:
	\begin{align}
		\rom{1}_{2} &=
		\enorm{ |k|^m \partial_y \paren{q^{n-l+1} (\Gamma_k-ikt) \Gamma_k^{n-l} \omega_k } e^{W}  \chi_{m+n-l+1}}
		\nn&\le
		\enorm{ |k|^m \partial_y \paren{q^{n-l+1} \Gamma_k^{n-l+1} \omega_k}  e^{W}  \chi_{m+n-l+1}}
		\n \\ &\qquad + 		\enorm{ |k|^m \partial_y \paren{q^{n-l+1} kt\Gamma_k^{n-l} \omega_k}  e^{W}  \chi_{m+n-l+1}}
		\nn&\le
		\enorm{ |k|^m \partial_y \paren{q^{n-l+1} \Gamma_k^{n-l+1} \omega_k}  e^{W}  \chi_{m+n-l+1}}
		\n \\  &\qquad+ 		\enorm{ |k|^m \partial_yq q^{n-l} kt\Gamma_k^{n-l} \omega_k  e^{W}  \chi_{m+n-l+1}}
		\nn&\qquad
		+ 		t\enorm{ |k|^{m+1} q \partial_y\paren{q^{n-l} \Gamma_k^{n-l} \omega_k } e^{W}  \chi_{m+n-l+1}}
		\nn&\le
	    \enorm{ |k|^m \partial_y \paren{q^{n-l+1} \Gamma_k^{n-l+1} \omega_k}  e^W  \chi_{m+n-l+1}}
		\n \\  &\qquad+ 	\nu^{-1/3-\eta}	\enorm{ |k|^{m+1} q^{n-l} \Gamma_k^{n-l} \omega_k    \chi_{m+n-l+1}}
		\nn&\qquad
		+ 	\nu^{-1/3-\eta}\enorm{ |k|^{m+1} q \partial_y\paren{q^{n-l} \Gamma_k^{n-l} \omega_k }e^W  \chi_{m+n-l+1}}
		\label{V:1:2}
	\end{align}
	where we used the boundedness of $\partial_yq$ as well as \eqref{nu:to:W} in the last step.
	For the term $\rom{1}_{3}$, from the definition of $\chi_{m+n}$ with $\Gamma$, $v_y \gtrsim 1$, and~\eqref{nu:to:W}, we have
	\begin{align*}
		\rom{1}_{3} &=	\enorm{ |k|^m q^{n-l+1} \pav \Gamma_k^{n-l} \omega_k  e^W \partial_y\chi_{m+n-l+1}}
		\\&\lesssim (m+n-l)^{1+\sigma} \enorm{ |k|^m q^{n-l+1} \pa_y\Gamma_k^{n-l} \omega_k e^W  \chi_{m+n-l} \widetilde{\chi}_1}
		\\&\lesssim
		(m+n-l)^{1+\sigma} \enorm{ |k|^m \partial_{y} (q^{n-l} \Gamma_k^{n-l} \omega_k ) e^W \chi_{m+n-l}\widetilde{\chi}_1 }
		\\&\qquad+
		 (m+n-l)^{1+\sigma}(n-l) \enorm{ |k|^{m+1} q^{n-l} \Gamma_k^{n-l} 
		 	\omega_k e^W\chi_{m+n-l}\widetilde{\chi}_1 }
	\end{align*}
where $ {\widetilde{\chi}_1}$ is a fattened version of $\chi_1$ while $\supp {\widetilde{\chi}_1} \subset [-1, -1/2]\cup[1/2, 1]$.  	Next we turn to the last term $\rom{1}_{4}$. 
From the definition of 
 the weight $W$, it is straightforward to check that
 \begin{align*}
 	\partial_{y} W \lesssim \frac{1}{\nu},
 \end{align*}
 from where using the bound for $\rom{1}_1$ we arrive at 
\begin{align*}
		\abs{\rom{1}_4} \lesssim& \frac{1}{\nu} \enorm{  |k|^m q^{n-l+1} \pav \Gamma_k^{n-l} \omega_k e^{W} \chi_{m+n-l+1}}
		\\ \lesssim & 
		\nu^{-1}\enorm{ |k|^m q^{n-l+1}\Gamma_k^{n-l+1} \omega_ke^{W}   \chi_{m+n-l+1}} 
		\\&\quad+ 	\nu^{-4/3-\delta} \enorm{ |k|^{m+1}  q^{n-l}\Gamma_k^{n-l} \omega_k  e^W \chi_{m+n-l+1}}.
\end{align*}
Therefore, we obtain
\begin{align*}
	\nu \sum_{m,n\ge 0}\sum_{k\neq 0} V_{11} &\lesssim
	\nu \sum_{m,n\ge 0}\sum_{k\neq 0}\sum_{n/2\le l \le n} \binom{n}{l} 
	\bold{a}_{m,n}\bold{a}_{0,l}^{-1}\bold{a}_{m,n-l+1}^{-1}|\bold{a}_{m,n}| \enorm{ \pav\omnk e^W\chi_{m+n}} 
	\\&\quad\times
	\bold{a}_{0,l}\enorm{ q^{l-1}\pav^l  (v_y^2 - 1)  \chi_{l}} \bold{a}_{m,n-l+1}
	\\&\quad \times
	\paren{
		\nu^{-1}\enorm{ \mr\omega_{m, n-l+1; k} e^{W}   \chi_{m+n-l+1}} 
		+ \nu^{-4/3-\delta}  \enorm{\mr \omega_{m+1, n-l; k}  e^W \chi_{m+n-l+1}} \right.
		\\&\quad\quad\quad\left.
		+
		\enorm{  \partial_y \mr\omega_{m, n-l+1; k}  e^W \chi_{m+n-l+1}}
		+ 	\nu^{-1/3-\eta}\enorm{  q \partial_y\mr\omega_{m+1, n-l; k}e^W  \chi_{m+n-l+1}}
			\right.	\nn&\quad \quad\quad
		\left.
+(m+n-l)^{1+\sigma} \enorm{  \partial_{y} \mr\omega_{m, n-l; k} e^{W} \chi_{m+n-l}\widetilde{\chi}_1} \right. 
		\\&\quad\quad\quad\left.
+
(m+n-l)^{1+\sigma}(n-l) \enorm{ \mr\omega_{m+1, n-l; k} e^{W} \chi_{m+n-l}\widetilde{\chi}_1 } }.
\end{align*}
Next we use Corollary~\ref{comb:boun}, \eqref{comb:boun:rem}, Young's inequality, and H\"older's inequality to obtain
\begin{align*}
	\nu \sum_{m,n\ge 0}\sum_{k\neq 0} V_{11} &\lesssim 	\sqrt{\mathcal{D}^{(\gamma)}} \paren{		\sqrt{\mathcal{E}^{(\gamma)}}
		+ \sqrt{\mathcal{D}^{(\gamma)}}} 
		\paren{	\sum_{n\ge1} \bold{a}_{0,n}^2 \enorm{ q^{n-1}\pav^n  (v_y^2 - 1) \chi_{m+n}}^2}^{1/2}
\end{align*}
In view of that
	\begin{equation}
		\label{nu:to:W}
		\nu^{-100} \le C e^{W/2}\ \ \ \ \mbox{for}\ \ \ \ t\le \nu^{-1/3-\eta}, 
	\end{equation}
by an argument similar to~\eqref{S:est:1}, we further get
\begin{align}
	\label{est:V:11}
\nu \sum_{m,n\ge 0}\sum_{k\neq 0} V_{11} &\lesssim  \nu^{100}
\sqrt{\mathcal{D}^{(\gamma)}} \paren{		\sqrt{\mathcal{E}^{(\gamma)}} + \sqrt{\mathcal{D}^{(\gamma)}}}
\paren{ \sum_{n\ge1}\sum_{\ell\le n-1} \mathfrak C^{n-\ell} \paren{\frac{n!}{\ell!}}^{-2\sigma}
\bold{a}_{0,\ell}^2 \right.
\nn&\qquad \left.
\times\enorm{ \partial_y \paren{q^{\ell-1}\pav^{\ell-1}  (v_y^2 - 1)}   e^{W/2}\chi_{\ell}}^2}^{1/2}
\nn&\lesssim \nu^{100}
	\sqrt{\mathcal{D}^{(\gamma)}} \paren{		\sqrt{\mathcal{E}^{(\gamma)}} + \sqrt{\mathcal{D}^{(\gamma)}}}
 \paren{\sum_{n\ge0} \bold{a}_n^2 \| \partial_y(q^n\pav^n
(v_y^2-1)) e^{W/2} \chi_n \|_{L^2}^2}^{1/2},
\end{align} 
where we have used Fubini in the last step. Using the Product Lemma~\ref{pro:1} 
leads to
\begin{align*}
	\nu \sum_{m,n\ge 0}\sum_{k\neq 0} V_{11}  &\lesssim 
	\nu^{99}
	\sqrt{\mathcal{D}^{(\gamma)}} \paren{		\sqrt{\mathcal{E}^{(\gamma)}} + \sqrt{\mathcal{D}^{(\gamma)}}}
	 \paren{\sqrt{\mathcal{E}_{H}^{(\alpha)}} + \mathcal{E}_{H}^{(\alpha)} }
	\\&\lesssim
	\eps \nu^{99}
	\sqrt{\mathcal{D}^{(\gamma)}} \paren{		\sqrt{\mathcal{E}^{(\gamma)}} + \sqrt{\mathcal{D}^{(\gamma)}}}.
\end{align*}
	Now we consider the case $0<l<n/2$:
	\begin{align*}
		V_{1,2}:&=\sum_{0<l <n/2} \binom{n}{l}
		\brk{|\bold{a}_{m,n}|^2  \pav\omnk,  |k|^mq^{n} \pav^l  (v_y^2 - 1) \pav \Gamma_k^{n-l} \omega_k e^{2W}\chi_{m+n}^2  }
		\\&\lesssim
		\sum_{0<l <n/2} \binom{n}{l}  |\bold{a}_{m,n}|^2
		\enorm{ \pav\omnk\chi_{m+n}}  \norm{q^{l}\pav^l  (v_y^2 - 1) \chi_{m+n}}_{L^\infty} 
		\\&\quad \times	\enorm{ |k|^mq^{n-l}  \pav \Gamma_k^{n-l} \omega_k e^{W}\chi_{m+n-l}\widetilde{\chi}_1}.
	\end{align*}
	For the last term in the above expression, 
	noting 
	that 
	\begin{equation}
		\label{comm:q:deri}
		[q^\ell, \partial_y] = - \ell q^{\ell-1}\partial_yq
	\end{equation}
	and 
	\begin{equation}
		\label{q:deri}
		q^\ell\partial_y f= \partial_y(q^\ell f) +  [q^\ell, \partial_y]f
	\end{equation} 
	for any $\ell\in\mathbb{N}$,
	we get
	\begin{align*}
		&\enorm{ |k|^mq^{n-l}  \pav \Gamma_k^{n-l} \omega_k e^{W}\chi_{m+n}} 
		\nn&\qquad\lesssim
		\enorm{ |k|^m  \pa_y \paren{q^{n-l}\Gamma_k^{n-l}\omega_k}  e^{W}\chi_{m+n}} 
		\\&\qquad\qquad+
		(n-l) \enorm{ |k|^mq^{n-l-1}\partial_yq  \Gamma_k^{n-l} \omega_k e^{W}\chi_{m+n}} 
		\nn&\qquad\lesssim
		\enorm{ |k|^m  \pa_y \paren{q^{n-l}\Gamma_k^{n-l}\omega_k}  e^{W}\chi_{m+n}} 
		+
		\enorm{ |k|^m S_{m,n-l}^{(1,0,0)}\omega_k e^{W}\chi_{m+n}}
	\end{align*} 
	where we used the boundedness of $\partial_yq$ and~\eqref{S_nq} 
	   in the last step.
	Similar as before, the main difficulty is to estimate
	\begin{align}
		\label{l:inft:coor}
		&\norm{q^{l}\pav^l  (v_y^2 - 1) \chi_{m+n}}_{L^\infty} \lesssim 	
		\norm{q^{l}\pav^l  (v_y^2 - 1) \chi_{ l+1}}_{L^2}
		+ \norm{\partial_y\paren{q^{l}\pav^l  (v_y^2 - 1) \chi_{l+1}}}_{L^2}
		\nn&\qquad
		\lesssim
		\norm{q^{l}\pav^l  (v_y^2 - 1) \chi_{ l+1}}_{L^2}
		+ \norm{\partial_y(q^{l}\pav^l  (v_y^2 - 1)) \chi_{ l+1}}_{L^2}
		\nn&\qquad\qquad+
		( l+1)^{1+\sigma} \norm{q^{l}\pav^l  (v_y^2 - 1) \chi_{l}}_{L^2}
		\nn&\qquad
		\lesssim
		\paren{ l^{1+\sigma} + 1}\norm{q^{l}\pav^l  (v_y^2 - 1) \chi_{l}}_{L^2}
		+ \norm{\partial_y(q^{l}\pav^l  (v_y^2 - 1)) \chi_{ l}}_{L^2}.
	\end{align}
Collecting the estimates about $V_{12}$, by Corollary~\ref{comb:boun}, \eqref{comb:boun:rem}, Young's inequality, and H\"older's inequality, similarly as~\eqref{est:V:11}, we have
\begin{align*}
	\nu \sum_{n,m\ge 0}\sum_{k\neq 0} V_{12} &\lesssim \nu^{100}
	\sqrt{\mathcal{D}^{(\gamma)}} \paren{		\sqrt{\mathcal{E}^{(\gamma)}} + \sqrt{\mathcal{D}^{(\gamma)}}}
	\norm{v_y^2-1}_{Y_{1,0}}
	\lesssim \eps  \nu^{100}
	\sqrt{\mathcal{D}^{(\gamma)}} \paren{		\sqrt{\mathcal{E}^{(\gamma)}} + \sqrt{\mathcal{D}^{(\gamma)}}}.
\end{align*}
\textbf{The term $V_2$:}
	Like in the previous situation, we consider two cases:
	\begin{align}
		\label{V:2:1}
		V_2&=-\sum_{l=1}^{n} \binom{n}{l}
		\brk{|\bold{a}_{m,n}|^2  \omnk,   |k|^mq^{n}\pav^{l}  (v_y^2 - 1) \pav \Gamma_k^{n-l} \omega_k \frac{2\pa_yW}{v_y}e^{2W}\chi_{m+n}^2  }
		\nn&= 
		-\paren{\sum_{n/2\le l \le n} + 	\sum_{0<l <n/2}} \binom{n}{l} \lf(\cdots\rg) =: V_{21} + V_{22}.
	\end{align}	
	For the first piece, we use Cauchy-Schwarz inequality to obtain
	\begin{align*}
		V_{21}:&=\sum_{n/2\le l \le n} \binom{n}{l} 
		\brk{|\bold{a}_{m,n}|^2  \omnk,   |k|^mq^{n}\pav^{l}  (v_y^2 - 1) \pav \Gamma_k^{n-l} \omega_k \frac{2\pa_yW}{v_y}e^{2W}\chi_{m+n}^2  }
		\nn& \lesssim
		\sum_{n/2\le l \le n} \binom{n}{l} 
		|\bold{a}_{m,n}|^2 \enorm{ \omnk e^W\chi_{m+n}}   \norm{q^{l-1}\pav^{l}  (v_y^2 - 1) e^{W/2} \chi_{l-1}\widetilde{\chi}_1}_{ L^2}
		\nn&\quad\times
		\norm{  |k|^m q^{n-l+1}\pav \Gamma_k^{n-l} \omega_k \frac{2\pa_yW}{v_y} e^{W/2} \chi_{m+n-l+1}}_{ L^\infty}.
	\end{align*}
For the second term in the above bound, we invoke~\eqref{S:est:1}.
	For the third term, by Sobolev embedding, we arrive at
	\begin{align*}
		&\norm{  |k|^m q^{n-l+1}\pav \Gamma_k^{n-l} \omega_k \frac{2\pa_yW}{v_y}e^{W/2} \chi_{m+n-l+1}    }_{ L^\infty}
		\\&\qquad\lesssim	
		\norm{  |k|^m q^{n-l+1}\pav \Gamma_k^{n-l} \omega_k \pa_yW e^{W/2} \chi_{m+n-l+1}    }_{ L^\infty}
		\\&\qquad\lesssim
		\norm{  |k|^m q^{n-l+1}\pav \Gamma_k^{n-l} \omega_k \pa_yW e^{W/2} \chi_{m+n-l+1}    }_{ L^2}
		\\&\qquad\qquad+	\norm{  |k|^m\partial_y\paren{ q^{n-l+1}\pav \Gamma_k^{n-l} \omega_k \pa_yW  e^{W/2} \chi_{m+n-l+1}   } }_{ L^2}
		\\&\qquad\lesssim	
		\norm{  |k|^m q^{n-l+1}\pav \Gamma_k^{n-l} \omega_k \pa_yW e^{W/2} \chi_{m+n-l+1}    }_{ L^2}
		\\&\qquad\qquad+		\norm{  |k|^m\partial_y\paren{q^{n-l+1}\pav \Gamma_k^{n-l}} \omega_k \pa_yW e^{W/2} \chi_{m+n-l+1}   }_{ L^2}		
		\\&\qquad\qquad
		+	\norm{  |k|^mq^{n-l+1}\pav \Gamma_k^{n-l} \omega_k \partial_y \paren{\pa_yW e^{W/2}} \chi_{m+n-l+1}    }_{ L^2}		
		\\&\qquad\qquad+		\norm{  |k|^m   q^{n-l+1}\pav \Gamma_k^{n-l} \omega_k \pa_yW e^{W/2} \partial_y\paren{\chi_{m+n-l+1} } }_{ L^2}
		\\&\qquad
		=\rom{2}_{1} + \rom{2}_{2} + \rom{2}_{3} + \rom{2}_{4}.
	\end{align*}
	For the $\rom{2}_{1}$ term, we use 
	\begin{equation}
		\label{prop:W}
	\nu^{-100}	\partial_yW \lesssim e^{W/2}, \ \ \ \ \ \ \ 		\nu^{-100} \partial_y^2W \lesssim e^{W/2},\quad \forall t\leq \nu^{-1/3-\eta},
	\end{equation}
	and \eqref{nu:to:W}
	to obtain
	\begin{align*}
		\rom{2}_{1} &\lesssim
		\norm{  |k|^m q^{n-l+1} \Gamma_k^{n-l+1} \omega_k \pa_yW e^{W/2} \chi_{m+n-l+1}    }_{ L^2}
		\\&\qquad   
		+
		t \norm{  |k|^{m+1} q^{n-l+1} \Gamma_k^{n-l} \omega_k \pa_yW e^{W/2} \chi_{m+n-l+1}    }_{ L^2}
		\\&\lesssim 	\nu^{100}
		\norm{ \mr \omega_{m, n-l+1; k}  e^{W}  \chi_{m+n-l+1}    }_{ L^2}
		+ \nu^{99}
          \norm{ \mr \omega_{m+1, n-l; k}e^{W}  \chi_{m+n-l+1}    }_{ L^2}
	\end{align*}
	for $t\le\nu^{-1/2}$. The second term $\rom{2}_{2}$ is treated in a similar fashion
	\begin{align*}
		\rom{2}_{2} 
		&\lesssim
		\norm{  |k|^m\partial_y\paren{q^{n-l+1} \Gamma_k^{n-l+1}} \omega_k \pa_yW e^{W/2} \chi_{m+n-l+1}   }_{ L^2}		
		\\&\qquad
		+
		t\norm{  |k|^{m+1}\partial_y\paren{q^{n-l+1} \Gamma_k^{n-l}} \omega_k \pa_yW e^{W/2} \chi_{m+n-l+1}   }_{ L^2}		
		\\&\lesssim \nu^{100}
		\norm{  \partial_y\mathring\omega_{m, n-l+1, k} e^{W} \chi_{m+n-l+1}   }_{ L^2}		
		+
\nu^{99} \norm{  \partial_y\mr\omega_{m+1, n-l; k} e^{W} \chi_{m+n-l+1}   }_{ L^2}
		\\&\qquad
		+ \nu^{99} \norm{  \mr\omega_{m+1, n-l; k} e^{W} \chi_{m+n-l+1}   }_{ L^2}.
	\end{align*}
	By a similar argument as that of $\rom{2}_{2} $, 
	we arrive at
	\begin{align*}       
		\rom{2}_{3} \lesssim
		\nu^{100}\norm{\mr  \omega_{m, n-l+1; k} e^W \chi_{m+n-l+1}    }_{ L^2}
		+
		\nu^{99} \norm{\mr  \omega_{m+1, n-l; k} e^W \chi_{m+n-l+1}    }_{ L^2}.
	\end{align*}
	The last piece $\rom{2}_{4} $ is estimated in an analogous way
	\begin{align*}
		\rom{2}_{4} 
		&\lesssim
		(m+n-l)^{ {1+\sigma}}\norm{  |k|^m   q^{n-l+1}\pav \Gamma_k^{n-l} \omega_k \partial_yW e^{W/2}\chi_{m+n-l}  }_{ L^2}
		\\&\lesssim \nu^{100}
		(m+n-l)^{1+\sigma}\norm{  \partial_{y}\mr \omega_{m, n-l; k} e^W \chi_{m+n-l} }_{ L^2}
		\\&\qquad
		+ \nu^{100}(m+n-l)^{1+\sigma}(n-l)  \norm{ \mr \omega_{m, n-l; k} e^{W}\chi_{m+n-l}  }_{ L^2}.
	\end{align*}
As in the treatment of $V_1$, gathering the estimates, by Corollary~\ref{comb:boun}, \eqref{comb:boun:rem}, Young's inequality, and H\"older's inequality, we again arrive at
\begin{align*}
	\nu \sum_{m,n\ge 0}\sum_{k\neq 0} V_{21} \lesssim 	\eps  \nu^{100}
	\sqrt{\mathcal{E}^{(\gamma)}} \paren{		\sqrt{\cd^{(\gamma)}} + \sqrt{\mathcal{CK}^{(\gamma)}}}.
\end{align*}
The second term in~\eqref{V:2:1} may be estimated by
\begin{align*}
	V_{22}:&=\sum_{0< l < n/2} \binom{n}{l} 
	\brk{|\bold{a}_{m,n}|^2  \omnk,   |k|^mq^{n}\pav^{l}  (v_y^2 - 1) \pav \Gamma_k^{n-l} \omega_k \frac{2\pa_yW}{v_y}e^{2W}\chi_{m+n}^2  }
	\nn& \lesssim
	\sum_{0< l < n/2} \binom{n}{l} 
	|\bold{a}_{m,n}|^2 \enorm{ \omnk e^W\chi_{m+n}}   \norm{q^{l}\pav^{l}  (v_y^2 - 1) \frac{2\pa_yW}{v_y} \chi_{m+n}}_{ L^\infty}
	\nn&\quad\times
	\norm{  |k|^m q^{n-l}\pav \Gamma_k^{ n-l} \omega_k  e^W\chi_{m+n-l} \widetilde{\chi}_1 }_{ L^2}.
\end{align*}
Following the arguments in the treatment of $V_{12}$ step by step, we may arrive at
\begin{align*}
	\nu \sum_{m,n\ge 0}\sum_{k\neq 0} V_{22} \lesssim 
	\eps  \nu^{100}
	\sqrt{\mathcal{E}^{(\gamma)}} \paren{		\sqrt{\cd^{(\gamma)}} + \sqrt{\mathcal{CK}^{(\gamma)}}}.
\end{align*}
\textbf{The term $V_3$:}
	We recall
	\begin{align*}
		V_3
		=   		-\sum_{l=1}^{n} \binom{n}{l}
		\brk{|\bold{a}_{m,n}|^2  \omnk,   |k|^m \pav\paren{q^{l}\pav^{l}  (v_y^2 - 1) }q^{n-l}\pav \Gamma_k^{n-l} \omega_k e^{2W}\chi_{m+n}^2  }
	\end{align*}
	and get by H\"older's inequality 
	\begin{align*}
		\abs{V_3}
		&\lesssim
		\sum_{l=1}^{n} \binom{n}{l}
		|\bold{a}_{m,n}|^2 \norm{ \omnk  e^{W/2}   \chi_{m+n}}_{L^\infty}   \enorm{\pav\paren{q^{l}\pav^{l}  (v_y^2 - 1) } e^{W/2} \chi_{l}} 
		\\&\qquad
		\times \enorm{ |k|^m q^{n-l}\pav \Gamma_k^{n-l} \omega_k e^{W}\chi_{m+n-l} \widetilde{\chi}_1}.
	\end{align*}
	We use Sobolev's inequality to obtain
	\begin{align*}
		&\norm{ \omnk e^{W/2}\chi_{m+n}}_{L^\infty}
		\lesssim
		\enorm{ \omnk  e^{W/2} \chi_{m+n}} + 	\enorm{ \partial_y\paren{\omnk  e^{W/2} \chi_{m+n}}}
		\\	&\qquad\lesssim
		\enorm{ \omnk  e^{W/2} \chi_{m+n}} + 	\enorm{ \partial_y\omnk  e^{W/2} \chi_{m+n}}
						\\	&\qquad\quad
		+ 	\enorm{ \omnk  \partial_y e^{W/2} \chi_{m+n}}
		+ 	(m+n)^{1+\sigma}\enorm{ \omnk   e^{W/2}
			\chi_{m+n-1} \widetilde{\chi}_1 }.
	\end{align*}
	For the last term in the above expression, we have
	\begin{align*}
		(m+n)^{1+\sigma}&\enorm{ \omnk  	 \chi_{m+n-1} \widetilde{\chi}_1 }
		\lesssim
		(m+n)^{1+\sigma}\enorm{  |k|^mq^n  \pav \Gamma^{n-1}\omega_{k}e^{W/2}	 \chi_{m+n-1} \widetilde{\chi}_1 }
		\\&\qquad +
	t   (m+n)^{1+\sigma}\enorm{  |k|^{m+1}q^{n-1}  \Gamma^{n-1}\omega_{k}e^{W/2}	 \chi_{m+n-1} \widetilde{\chi}_1 }
		\\&\quad 		\lesssim
		(m+n)^{1+\sigma}\enorm{   \pa_y \mr\omega_{m, n-1;k} e^{W/2}	 \chi_{m+n-1} \widetilde{\chi}_1 }
				\\&\qquad+
		(m+n)^{1+\sigma}n \enorm{ \mr \omega_{m, n-1;k} e^{W/2}	 \chi_{m+n-1} \widetilde{\chi}_1 }
		\\&\qquad +
        t (m+n)^{1+\sigma}\enorm{ \mr \omega_{m+1, n-1;k} e^{W/2}	 \chi_{m+n-1} \widetilde{\chi}_1 }
		\\&\quad 		\lesssim \nu^{100}
		(m+n)^{1+\sigma}\enorm{   \pa_y \mr\omega_{m, n-1;k} e^{W}	 \chi_{m+n-1} \widetilde{\chi}_1 }
				\\&\qquad+
		\frac{\nu^{99}}{t}(m+n)^{1+\sigma}n \enorm{ \mr \omega_{m, n-1;k} e^{W}	 \chi_{m+n-1} \widetilde{\chi}_1 }
		\\&\qquad +
		 \nu^{99} (m+n)^{1+\sigma}\enorm{  \mr\omega_{m+1, n-1; k} e^{W}	 \chi_{m+n-1} \widetilde{\chi}_1 }.
	\end{align*}
Then again by Corollary~\ref{comb:boun}, \eqref{comb:boun:rem}, Young's inequality, and H\"older's inequality, similar as~\eqref{est:V:11}, we arrive at
\begin{align*}
	\nu \sum_{m,n\ge 0}\sum_{k\neq 0} V_{3} \lesssim \eps  \nu^{100}
	\sqrt{\mathcal{D}^{(\gamma)}} \paren{		\sqrt{\mathcal{E}^{(\gamma)}} + \sqrt{\mathcal{D}^{(\gamma)}}}.
\end{align*}
\textbf{The term $V_4$:}
	Like in the treatment of $V_3$ term, we have
	\begin{align*}
		\abs{V_4} &\lesssim \sum_{l=1}^{n} \binom{n}{l}
		|\bold{a}_{m,n}|^2 
		\norm{	 \omnk e^{W/2}\chi_{m+n} }_{L^\infty}   
		\enorm{q^{l-1}\pav^{l}  (v_y^2 - 1)e^{W/2}\chi_l } 	
		\\&\qquad \times(n-l)\enorm{	 |k|^m q^{n-l}\pav \Gamma_k^{n-l} \omega_k e^{W}  \chi_{m+n-l} \widetilde{\chi}_1  },
	\end{align*}
	from where we follow the treatment of $V_3$ term line by line.\\
\textbf{The term $V_5$:}
	We recall that
	\begin{align*}
		V_5 = 	-\sum_{l=1}^{n} \binom{n}{l}
		\brk{|\bold{a}_{m,n}|^2  \omnk,   |k|^m q^{n}\pav^{l}  (v_y^2 - 1) \pav \Gamma_k^{n-l} \omega_k e^{2W}
			\frac{2\pa_y\chi_{m+n}}{v_y}\chi_{m+n}  }
	\end{align*}
from where applying H\"older's inequality gives
\begin{align*}
\abs{V_5} & \le 
	\sum_{l=1}^{n} \binom{n}{l}
	|\bold{a}_{m,n}|^2 (m+n)^{1+\sigma	}
	\norm{	 \omnk\chi_{m+n} }_{L^\infty}   
	\enorm{q^{l}\pav^{l}  (v_y^2 - 1)e^{W}\chi_l } 	
	\\&\qquad \times\enorm{	 |k|^m q^{n-l}\pav \Gamma_k^{n-l} \omega_k e^{W} \chi_{m+n-l} \widetilde{\chi}_1   } 
\end{align*}
\textbf{The term $V_6$:}
	We recall that
	\begin{align*}
		V_6 = 		-\sum_{l=1}^{n} \binom{n}{l}
		\brk{|\bold{a}_{m,n}|^2  \omnk,  |k|^mq^{n} \partial_y\paren{\frac{1}{v_y}}\pav^{l}  (v_y^2 - 1) \pav \Gamma_k^{n-l} \omega_k e^{2W}\chi_{m+n}^2  }
	\end{align*}
	from where applying H\"older's inequality gives
	\begin{align*}
		\abs{V_6} & \le 
		\sum_{l=1}^{n} \binom{n}{l}
		|\bold{a}_{m,n}|^2 
		\norm{	 \omnk\chi_{m+n} }_{L^\infty}   
		\norm{\partial_y\paren{\frac{1}{v_y}}}_{L^\infty}
		\enorm{q^{l}\pav^{l}  (v_y^2 - 1)e^{W}\chi_l } 	
		\\&\qquad \times\enorm{	 |k|^m q^{n-l}\pav \Gamma_k^{n-l} \omega_k e^{W} \chi_{m+n-l} \widetilde{\chi}_1 } .
		\end{align*}
Hence the proof is finished by summing up the estimates of $V_1$--$V_6$ in \eqref{visc:comm:eq01} and we omit further details.
\end{proof}
For the $\alpha$ and $\mu$ level estimate, we get the following two lemmas.
\begin{lemma}[Viscous commutator for $\alpha$ estimate] 
  It holds
	\label{visc:comm:alph}
	\begin{align*}
		& {\nu} \sum_{m,n\ge 0}\sum_{k\neq 0} \abs{\brk{ |\bold{a}_{m,n}|^2    \partial_y \omnk,  \partial_y \paren{ \bold{C}^{(m,n)}_{\mathrm{visc}, k}} e^{2W}\chi_{m+n}^2}}
		\nn&\qquad
		\lesssim
		\eps \nu^{99}
		\sqrt{\mathcal{E}^{(\alpha)}} \paren{\sqrt{\mathcal{E}^{(\alpha)}} + \sqrt{\mathcal{D}^{(\alpha)}}+\sqrt{\mathcal{D}^{(\gamma)}}
			+\sqrt{\mathcal{D}^{(\mu)}}+ \sqrt{\mathcal{CK}^{(\alpha)}}}	{\paren{\sqrt{\cd_{H}^{(\gamma)}}+\sqrt{\cd_{H}^{(\alpha)}}}}.
	\end{align*}
\end{lemma}
\begin{proof}
	For fixed $m,n\ge0$, note
	\begin{align*}
	&{\nu }\brk{ |\bold{a}_{m,n}|^2    \partial_y \omnk,  \partial_y \paren{\bold{C}^{(m,n)}_{\mathrm{visc}, k}} e^{2W}\chi_{m+n}^2}
	\nn&\quad
	=
    \nu^{2} 
	\sum_{l=1}^{n} \binom{n}{l}
	\brk{	  |\bold{a}_{m,n}|^2  \partial_y\omnk,  |k|^m
		\partial_y \paren{q^{n}\pav^l  (v_y^2 - 1) \pav^2 \Gamma_k^{n-l} \omega_k} e^{2W}\chi_{m+n}^2}
		\nn&\quad
	=
	\nu^{2} 
	\sum_{l=1}^{n} \binom{n}{l}\bigg( 
	\brk{	  |\bold{a}_{m,n}|^2  \partial_y\omnk,  |k|^m
	n\partial_yq q^{n-1}\pav^l  (v_y^2 - 1) \pav^2 \Gamma_k^{n-l} \omega_k e^{2W}\chi_{m+n}^2}
		\nn&\qquad
	+
	\brk{	  |\bold{a}_{m,n}|^2  \partial_y\omnk,  |k|^m
   q^{n}\partial_y \pav^l  	(v_y^2 - 1) 
   \pav^2 \Gamma_k^{n-l} \omega_k e^{2W}\chi_{m+n}^2}
		\nn&\qquad
	+
	\brk{	  |\bold{a}_{m,n}|^2  \partial_y\omnk,  |k|^m
		q^{n}\pav^l  (v_y^2 - 1) \partial_y\paren{\pav^2 \Gamma_k^{n-l} \omega_k} e^{2W}\chi_{m+n}^2}\bigg)
		\nn&\quad
      = V^\alpha_1 + V^\alpha_2 + V^\alpha_3.
    \end{align*}
%
\textbf{The term $V^\alpha_1$:}
By the boundedness of $\partial_yq$, it is easy to obtain
\begin{align*}
	\abs{V^\alpha_1} &\lesssim
		\nu^{2} 
	\sum_{l\le n/2} \binom{n}{l}
		  |\bold{a}_{m,n}|^2 \enorm{ \partial_y\omnk e^{W}\chi_{m+n}}  
		   n \norm{ q^{l-1}\pav^l  (v_y^2 - 1) \chi_{n-1}}_{L^\infty}
		\\&\quad  
		\times\enorm{|k|^m q^{n-l}\pav^2 \Gamma_k^{n-l} \omega_k e^{W}\chi_{m+n}}
		\\&\quad     +
     	\nu^{2} \sum_{n\ge l> n/2} \binom{n}{l}
     |\bold{a}_{m,n}|^2 \enorm{ \partial_y\omnk e^{W}\chi_{m+n}}  
     \enorm{n q^{l-2}\pav^l  (v_y^2 - 1) e^{W/2}\chi_{l-1} \widetilde{\chi}_1 }
     		\\&\quad
     \times\norm{|k|^m q^{n-l+1}\pav^2 \Gamma_k^{n-l} \omega_k e^{W/2} \chi_{m+n-l+1}}_{L^\infty}
     \\& = V^\alpha_{1,1} + V^\alpha_{1,2}.
\end{align*}
By Sobolev embedding, we have
\begin{align}
	\label{v:term}
	&\norm{ q^{l-1}\pav^l  (v_y^2 - 1) \chi_{n-1}}_{L^\infty}
	\lesssim 
	\enorm{ q^{l-1}\pav^l  (v_y^2 - 1) \chi_{l}}^{1/2}
	\enorm{ \partial_y\paren{q^{l-1}\pav^l  (v_y^2 - 1) \chi_{l}}}^{1/2}
	\nn&\qquad
	\lesssim 
	\enorm{ q^{l-1}\pav^l  (v_y^2 - 1) \chi_{l}}^{1/2}
	\bigg((l-1)\enorm{ q^{l-2}\pav^l  (v_y^2 - 1) \chi_{l}}
	\nn &\qquad\quad
	+
	\enorm{ q^{l-1}\partial_y\paren{\pav^l  (v_y^2 - 1) }\chi_{l}}
	+
	l^{1+\sigma} \enorm{ q^{l-1}\pav^l  (v_y^2 - 1) \chi_{l-1}
}\bigg)^{1/2}.
\end{align}
For a term of the following type
\begin{align*}
	q^{l-1}\pav^l  (v_y^2 - 1)
\end{align*}
we invoke estimate~\eqref{S:est:1}. While for terms of type
\begin{equation*}
q^{l-2}\pa_y\pav^{l-1}  (v_y^2 - 1) 
= 
q^{l-2}\frac{\pa_y^2}{v_y}\pav^{l-2}  (v_y^2 - 1) + 
q^{l-2}\pa_y\paren{\frac{1}{v_y}}\pa_y\pav^{l-2}  (v_y^2 - 1),
\end{equation*}
and \begin{align*}
	\enorm{|k|^m q^{n-l}\pav^2 \Gamma_k^{n-l} \omega_k e^{W}\chi_{m+n}}
\end{align*}
we perform an argument similar to~ \eqref{S:est:1} and \eqref{S:est:2}.  
Collecting the estimates about $V^\alpha_{1,1}$, by Corollary~\ref{comb:boun}, \eqref{comb:boun:rem}, Young's inequality, H\"older's inequality, and Fubini's theorem, we have
\begin{align*}
	\sum_{m,n\ge 0}\sum_{k\neq 0} V^\alpha_{1,1} &\lesssim \nu^{100}
	\sqrt{\mathcal{E}^{(\alpha)}} \paren{		\sqrt{\mathcal{E}^{(\alpha)}} + \sqrt{\mathcal{D}^{(\alpha)}}+\sqrt{\mathcal{D}^{(\gamma)}}+\sqrt{\mathcal{D}^{(\mu)}}}
	 \brak{t} \paren{\sum_{n\ge0} \bold{a}_n^2 \|{S^{(0,1,0)}_{0,n}(v_y^2-1)}e^{W/2} \chi_n \|_{L^2}^2}^{1/4}
	\\&\qquad\qquad\times
	\paren{\sum_{n\ge0} \bold{a}_n^2 \| {S^{(0,2,0)}_{0,n}(v_y^2-1)} e^{W/2} \chi_n \|_{L^2}^2}^{1/4}
	\\&\lesssim
	\eps \nu^{99}
	\sqrt{\mathcal{E}^{(\alpha)}} \paren{\sqrt{\mathcal{E}^{(\alpha)}} + \sqrt{\mathcal{D}^{(\alpha)}}+\sqrt{\mathcal{D}^{(\gamma)}}+\sqrt{\mathcal{D}^{(\mu)}}}
	\paren{\sqrt{\cd_{H}^{(\gamma)}}+\sqrt{\cd_{H}^{(\alpha)}}}.
\end{align*}
where we also used assumptions \eqref{asmp} and Product Lemma~\ref{pro:1}.
For the term $V^\alpha_{1,2}$, we note by Sobolev inequality
\begin{align*}
	&\norm{|k|^m q^{n-l+1}\pav^2 \Gamma_k^{n-l} \omega_k e^{W/2}\chi_{m+n-l+1}}_{L^\infty}
	\\&\qquad
	\lesssim
	\enorm{|k|^m q^{n-l+1}\pav^2 \Gamma_k^{n-l+1} \omega_k e^{W/2} \chi_{m+n-l+1}}^{1/2}
		\enorm{|k|^m \partial_y\paren{q^{n-l+1}\pav^2 \Gamma_k^{n-l} \omega_k e^{W/2} \chi_{m+n-l+1}}}^{1/2}
		\\&\qquad
	\lesssim
	\enorm{|k|^m q^{n-l+1}\pav^2 \Gamma_k^{n-l} \omega_k e^{W/2} \chi_{m+n-l+1}}^{1/2}
	\Bigg(\enorm{|k|^m \partial_y\paren{q^{n-l+1}\pav^2 \Gamma_k^{n-l} \omega_k} e^{W/2}
	 \chi_{m+n-l+1}}
	\\&\qquad\quad
	+ \enorm{|k|^m q^{n-l+1}\pav^2 \Gamma_k^{n-l} \omega_k \partial_{y} e^{W/2}
		\chi_{m+n-l+1}}
	\\&\qquad\quad
	+(m+n-l+1)\enorm{|k|^{m} q^{n-l+1}\pav^2 \Gamma_k^{n-l} \omega_k \chi_{m+n-l}}^{1/2}\Bigg)
	^{1/2}.
\end{align*}
By the definition of $\Gamma$, we obtain
\begin{align*}
	&\enorm{|k|^m q^{n-l+1}\pav^2 \Gamma_k^{n-l} \omega_k  e^{W/2} \chi_{m+n-l+1} }
	\\&\qquad
		\le
	\enorm{|k|^m q^{n-l+1}\pav \Gamma_k^{n-l+1} \omega_k  e^{W/2} \chi_{m+n-l+1} }
	+ 
	t\enorm{|k|^m k q^{n-l+1}\pav \Gamma_k^{n-l} \omega_k  e^{W/2} \chi_{m+n-l+1} }
	\\&\qquad
	\lesssim
	\enorm{|k|^m q^{n-l+1}\pav \Gamma_k^{n-l+1} \omega_k  e^{W/2} \chi_{m+n-l+1} }
	+ 
	t\enorm{|k|^m k q^{n-l+1}\pav \Gamma_k^{n-l} \omega_k    e^{W/2} \chi_{m+n-l+1} }
		\\&\qquad
	\lesssim 
    \enorm{|k|^m q^{n-l+1}\pav \Gamma_k^{n-l+1} \omega_k  e^{W/2} \chi_{m+n-l+1} }
    + 
	t \enorm{|k|^m k \pa_y \paren{q^{n-l+1} \Gamma_k^{n-l} \omega_k}    e^{W/2} \chi_{m+n-l+1} }
    	\\&\qquad\quad + 
    (n-l+1) t\enorm{|k|^m k \pa_yq q^{n-l} \Gamma_k^{n-l} \omega_k    e^{W/2} \chi_{m+n-l+1} }
		\\&\qquad
\lesssim 
\enorm{|k|^m q^{n-l+1}\pav \Gamma_k^{n-l+1} \omega_k  e^{W/2} \chi_{m+n-l+1} }
+ 
	\nu^{-1/3-\eta}\enorm{ \pa_y \mr\omega_{m+1, n-l;k}    e^{W/2} \chi_{m+n-l+1} }
\\&\qquad\quad + 
(n-l+1)	\nu^{-1/3-\eta} \enorm{  \mr \omega_{m+1, n-l;k}   e^{W/2} \chi_{m+n-l+1} }
\end{align*}
where we used~\eqref{nu:to:W}.
Similarly, we have
\begin{align*}
	&\enorm{|k|^m \partial_y\paren{q^{n-l+1}\pav^2 \Gamma_k^{n-l} \omega_k}
		 e^{W/2} \chi_{m+n-l+1} }
	\\&\qquad
	\le
	\enorm{|k|^m \partial_y\paren{q^{n-l+1}\pav \Gamma_k^{n-l+1} \omega_k}  e^{W/2} \chi_{m+n-l+1} }
	\\&\qquad \quad + 
	t\enorm{|k|^m k \partial_y\paren{q^{n-l+1}\pav \Gamma_k^{n-l} \omega_k}  e^{W/2} \chi_{m+n-l+1} }
	\\&\qquad
	\lesssim 
	\enorm{ \pa_y^2		\mr 
	\omega_{m, n-l+1;k} e^{W/2} \chi_{m+n-l+1}  } 
		+ 
	(n-l+1)\enorm{  \pa_y\paren{\frac{\partial_yq}{q}\mr \omega_{m, n-l+1;k}} e^{W/2} \chi_{m+n-l+1} } 
	\\&\qquad\quad+ 
	\enorm{|k|^m q^{n-l+1}\pa_y \Gamma_k^{n-l+1} \omega_k  e^{W/2} \chi_{m+n-l+1} }
	+
	t\enorm{ \partial_y\mr\omega_{m+1, n-l+1;k}  e^{W/2} \chi_{m+n-l+1} }
		\\&\qquad\quad+
	t^2\enorm{ \partial_y\paren{q\mr\omega_{m+2, n-l;k}}  e^{W/2} \chi_{m+n-l+1} }
	\\&\qquad
	\lesssim 
	\enorm{ \pa_y^2	\mr	 
		\omega_{m, n-l+1;k} e^{W/2} \chi_{m+n-l+1}  } 
			\\&\qquad\quad+ 
	(n-l+1)\enorm{ |k|^m \pa_y\paren{q^{n-l}\pav \Gamma_k^{n-l} \omega_{k}} e^{W/2} \chi_{m+n-l+1} } 
	\\&\qquad\quad+ 
	(n-l+1) t \enorm{  \pa_y\paren{ \mr\omega_{m, n-l;k}} e^{W/2} \chi_{m+n-l+1} } 
	+
	\enorm{|k|^m q^{n-l+1}\pa_y \Gamma_k^{n-l+1} \omega_k  e^{W/2} \chi_{m+n-l+1} }
	\\&\qquad\quad	
	+
	t\enorm{ \partial_y\mr\omega_{m+1, n-l+1;k}  e^{W/2} \chi_{m+n-l+1} }
	+
	t^2\enorm{ \partial_y\paren{q\mr\omega_{m+2, n-l;k}}  e^{W/2} \chi_{m+n-l+1} }
\end{align*}
where we used~\eqref{nu:to:W} in the last step.
Therefore, as for $V^\alpha_{1,2}$, we deduce
\begin{align*}
	\sum_{m,n\ge 0}\sum_{k\neq 0} V^\alpha_{1,2}
		\lesssim
	\eps \nu^{99}
	\sqrt{\mathcal{E}^{(\alpha)}} \paren{\sqrt{\mathcal{E}^{(\alpha)}} + \sqrt{\mathcal{D}^{(\alpha)}}+\sqrt{\mathcal{D}^{(\gamma)}}+\sqrt{\mathcal{D}^{(\mu)}}}
	\paren{\sqrt{\cd_{H}^{(\gamma)}}+\sqrt{\cd_{H}^{(\alpha)}}}.
\end{align*}
\\
\textbf{The term $V^\alpha_2$:}
We recall
\begin{align*}
	V^\alpha_2=
		\nu^{2} 
	\sum_{l=1}^{n} \binom{n}{l}
	\brk{	  |\bold{a}_{m,n}|^2  \partial_y\omnk,  |k|^m
		q^{n}\partial_y \pav^l  (v_y^2 - 1) 
		\pav^2 \Gamma_k^{n-l} \omega_k e^{2W}\chi_{m+n}^2}.
\end{align*}
Upon dividing the sum into two parts: $l\le n/2$ and $l>n/2$, we follow the treatment of 
$V^\alpha_1$ to get
\begin{align*}
	\abs{V^\alpha_2} \lesssim
	V^\alpha_{2,1}+V^\alpha_{2,2}
\end{align*}
where $V^\alpha_{2,i}$ are given as
\begin{align*}
	V^\alpha_{2,1} &= 
	\nu^{2} 
	\sum_{l\le n/2} \binom{n}{l}
	  |\bold{a}_{m,n}|^2 \enorm{ \partial_y\omnk e^{W}\chi_{m+n}}
	    \norm{|k|^m q^{l}\partial_y \pav^l  (v_y^2 - 1) \chi_{l+1}}_{L^\infty}
	    \\&\qquad\times
		\enorm{    |k|^m q^{n-l} \pav^2 \Gamma_k^{n-l} \omega_k e^{W}\chi_{m+n-l}}
\end{align*}
and
\begin{align*}
	V^\alpha_{2,2} 
	&= 
	\nu^{2} 
	\sum_{l> n/2} \binom{n}{l}
	|\bold{a}_{m,n}|^2 \enorm{ \partial_y\omnk e^{W}\chi_{m+n}}
	\norm{|k|^m q^{l-1}\partial_y \pav^l  (v_y^2 - 1) e^{W/2}\chi_{l-1}\widetilde{\chi}_1}_{L^2}
	\\&\qquad\times
	\norm{  q^{n-l+1} \pav^2 \Gamma_k^{n-l} \omega_k e^{W/2} \chi_{m+n-l+1}}_{L^\infty}.
\end{align*}
To estimate the $V^\alpha_{2,1}$ term, we note by Sobolev inequality and Minkowski inequality,
\begin{align*}
	&\norm{|k|^m q^{l}\partial_y \pav^l  (v_y^2 - 1) \chi_{l+1}}_{L^\infty}
	\\&\qquad\lesssim
	\enorm{|k|^m q^{l}\partial_y \pav^l  (v_y^2 - 1) \chi_{l+1}}^{1/2}
	\enorm{|k|^m \partial_y \paren{q^{l}\partial_y \pav^l  (v_y^2 - 1)\chi_{l+1}} }^{1/2}
	\\&\qquad\lesssim
	\enorm{|k|^m q^{l}\partial_y \pav^l  (v_y^2 - 1) \chi_{l+1}}^{1/2}
	\Bigg(\enorm{|k|^m \partial_y \paren{q^{l}\partial_y \pav^l  (v_y^2 - 1)}\chi_{l+1} }
	\\&\qquad\quad+
	\enorm{|k|^m q^{l}\partial_y \pav^l  (v_y^2 - 1)\partial_y \chi_{l+1} }\Bigg)^{1/2}
\end{align*}
The treatment of $\norm{  q^{n-l+1} \pav^2 \Gamma_k^{n-l} \omega_k \chi_{m+n}}_{L^\infty}$ in the term
	$V^\alpha_{2,2}$ is the same as in $V^\alpha_{1,2}$ and we omit further details. 
Using  Corollary~\ref{comb:boun}, \eqref{comb:boun:rem}, Young's inequality, and H\"older's inequality, similarly as~\eqref{est:V:11}, we conclude the estimate of $V^\alpha_2$. \\
\textbf{The term  $V^\alpha_3$:}
Using $\Gamma_k = \Gamma_0 + ikt$$=\pav+ikt$, $V^\alpha_3$ could be further rewritten as 
\begin{align*}
	V^\alpha_3&=
		\nu^{2} 
	\sum_{l=1}^{n} \binom{n}{l}
	\brk{	  |\bold{a}_{m,n}|^2  \partial_y\omnk,  |k|^m
		q^{n}\pav^l  (v_y^2 - 1) \partial_y\paren{\pav \Gamma_k^{n-l+1} \omega_k} e^{2W}\chi_{m+n}^2}
		\\&\quad + 
		\nu^{2} 
		\sum_{l=1}^{n} \binom{n}{l}
		\brk{	  |\bold{a}_{m,n}|^2  \partial_y\omnk,  |k|^m
			q^{n}\pav^l  (v_y^2 - 1) \partial_y\paren{\pav \Gamma_k^{n-l} ikt \omega_k} e^{2W}\chi_{m+n}^2},
\end{align*}
then this term could be treated almost the same as the $V_2$ term in Lemma~\ref{visc:comm}, except the $\gamma$ norms are replaced by the correspond $\alpha$ norms as well as some commutator estimates elaborated in Section~\ref{sec:Elliptic}, and  we conclude the proof without giving more details.
\end{proof}
\begin{lemma}[Viscous commutator for $\mu$ estimate] 
	Let $\bold{C}^{(m,n)}_{\mathrm{visc}, k}$ be defined as above. It holds   
	\label{visc:comm:mu}
	\begin{align}
		& {\nu} \sum_{m,n\ge 0}\sum_{k\neq 0} \abs{\brk{ |\bold{a}_{m,n}|^2    ik\omnk, ik\bold{C}^{(m,n)}_{\mathrm{visc}, k} e^{2W}\chi_{m+n}^2}}
		\nn&\qquad \lesssim
		\eps  \nu^{100}
		\sqrt{\mathcal{E}^{(\mu)}} \paren{		\sqrt{\cd^{(\mu)}} + \sqrt{\mathcal{CK}^{(\mu)}}}  + \eps  \nu^{100}
		\mathcal{D}^{(\mu)}.
	\end{align}
\end{lemma}
\begin{proof}
	Noting
	\begin{align*}
		ik \bold{C}^{(m,n)}_{{\rm visc}, k} :=   \nu q^n|k|^m\sum_{\ell = 1}^n \binom{n}{\ell}  \pav^\ell v_y^2\ \pav^2 \Gamma_k^{n-\ell} ik \omega_k 
		=   \nu \sum_{\ell = 1}^n \binom{n}{\ell} q^{\ell} \pav^\ell \paren{v_y^2-1} q^{n-\ell} \pav^2 ik \Gamma_k^{n-\ell} |k|^m\omega_k,
	\end{align*}
	the proof follows almost identically as that of Lemma~\ref{visc:comm}.
\end{proof}

\subsection{Commutator Estimates, $\bold{C}^{(m,n)}_{\mathrm{trans}, k}$}
We recall the transport commutator is defined as 
\begin{align*}
	\bold{C}^{(m,n)}_{\mathrm{trans}, k} 
	:=  q^n \sum_{l = 1}^n \binom{n}{l} \pav^l G \ |k|^m\Gamma_k^{n-l+1} \omega_k.
\end{align*}
Then we have the following estimates.
\begin{lemma}[Transport commutator for $\gamma$ estimate] 
	Let $\bold{C}^{(m,n)}_{\mathrm{trans}, k}$ be defined as above. It holds
	\label{tran:comm}
	\begin{align}
		\label{tran:comm:esti}
		\sum_{m,n\ge 0}\sum_{k\neq 0} & \theta_{n}^2 |\bold{a}_{m,n}|^2 \abs{\brk{    \omnk,   \bold{C}^{(m,n)}_{\mathrm{trans}, k} e^{2W}\chi_{m+n}^2}}
		\nn & \lesssim 
\eps
\nu^{99}
\sqrt{\mathcal{E}^{(\gamma) }}
\paren{\sqrt{\mathcal{E}^{(\gamma) }} + \sqrt{\cd^{(\gamma)}} + \sqrt{\mathcal{CK}^{(\gamma)}}}\paren{1 + \sqrt{\cd^{(\gamma)}_{ H}} + \sqrt{\cd^{(\gamma)}_{ \overline  H}}}.
	\end{align}
\end{lemma}
\begin{proof}
	The transport commutator term could be written as
	\begin{align}
		\label{comm:tran}
		\bold{C}^{(m,n)}_{\mathrm{trans}, k} 
		&= 
		q^n \sum_{l = 1}^n \binom{n}{l} \pav^{l-1}
		\paren{\frac{\overline{H}}{v_y}} \Gamma_k^{n-l+1} |k|^m\omega_k.
	\end{align}
	Hence, we have
	\begin{align*}
		&\brk{ |\bold{a}_{m,n}|^2    \omnk, 
			\bold{C}^{(m,n)}_{ {\rm trans},k} e^{2W}\chi_{m+n}^2}
		\\&\qquad\qquad\qquad=
		\sum_{l=0}^{n-1} \binom{n}{l}
		\brk{	  |\bold{a}_{m,n}|^2  \omnk, q^{n}\pav^{l} \frac{\overline{H}}{v_y}\mr \omega_{m,n-l;k} e^{2W}\chi_{m+n}^2 }.
	\end{align*}
	Similar as the treatment of the viscous commutator term in section~\ref{visc:comm:sect}, we split the sum into two parts:
	\begin{align*}
		\paren{\sum_{l\le n/2} + \sum_{n/2<l\le n-1} }
		\binom{n}{l}
		\brk{	  |\bold{a}_{m,n}|^2  \omnk, 
			|k|^m q^l\pav^{l} \frac{\overline{H}}{v_y} q^{n-l}\Gamma_k^{n-l}\omega_{k} e^{2W}\chi_{m+n}^2 }
		=
		T_{\gamma,1}^{m,n} +T_{\gamma,2}^{m,n}.
	\end{align*}
	The first term is estimated by
	\begin{align*}
		\abs{T_{\gamma,1}^{m,n}}&  \le 	
		\sum_{l\le n/2} |\bold{a}_{m,n}|^2
		\binom{n}{l}
		\nnorm{    \omnk e^{W}\chi_{m+n}}_{L^2}
		\nnorm{ q^l\pav^{l} \frac{\overline{H}}{v_y}\tilde \chi_{l+1}}_{L^\infty_y}
		\enorm{  |k|^m  q^{n-l}\Gamma_k^{n-l}\omega_{k} e^{W}\chi_{m+n-l}}
	\end{align*}
where we denote
\begin{align*}
	\widetilde{\chi}_{l+1} = 
	\begin{cases}
		\chi_{l+1}, & l\ge 1 \\
		\widetilde{\chi}_{1}, & l =0
	\end{cases}
\end{align*}
with $	\widetilde{\chi}_{1}$ being a fattened version of $\chi_1$ while $\supp 	\widetilde{\chi}_{1} \subset [-1, -7/16]\cup[7/16, 1]$. 
	The trouble term $\nnorm{ q^l\pav^{l} \frac{\overline{H}}{v_y}\widetilde{\chi}_{l+1}}_{L^\infty_y}$ is treated as in~\eqref{l:inft:coor}
	\begin{align*}
		\nnorm{ q^l\pav^{l} \frac{\overline{H}}{v_y}\widetilde{\chi}_{l+1}}_{L^\infty_y}
		&\lesssim
		\paren{ l^{1+\sigma} + 1}\norm{q^l\pav^{l} \frac{\overline{H}}{v_y}\tilde{\widetilde{\chi}}_{l}}_{L^2}
		+ \norm{\partial_y( q^l\pav^{l} \frac{\overline{H}}{v_y}) \tilde{\widetilde{\chi}}_{l}}_{L^2}
	\end{align*}
with
\begin{align*}
\widetilde{\widetilde{\chi}}_{l} = 
	\begin{cases}
		\chi_{l}, & l\ge 1 \\
		\widetilde{\widetilde{\chi}}_{1}, & l =0
	\end{cases}
\end{align*}
where $\widetilde{\widetilde{\chi}}_{l}$ is a fattened version of $\widetilde{\chi}_{l}$ while $\supp 	\widetilde{\widetilde{\chi}}_{l} \subset [-1, -1/2]\cup[1/2, 1]$.
	Therefore, we get
	\begin{align}
		\label{T:gamm:1}
		\sum_{m,n\ge0}\sum_{k\in\ZZ}\theta_{n}^2 T_{\gamma,1}^{m,n}
		&\lesssim
		\sum_{m,n\ge0}\sum_{k\in\ZZ}
		\sum_{l\le n/2} |\bold{a}_{m,n}|^2
		\binom{n}{l}
		\nnorm{    \omnk e^{W}\chi_{m+n}}_{L^2}\bold{a}_{m,n-l}^{-1}\bb_{l+1}^{-1}(l+1)^{2-2/s}
		\nn&\quad
		\times
		\bold{a}_{m,n-l}\enorm{  |k|^m  q^{n-l}\Gamma_k^{n-l}\omega_{k} e^{W}\chi_{m+n}}
		\bb_{l+1}(l+1)^{2/s-2}\paren{ l^{1+\sigma} + 1}\norm{q^l\pav^{l} \frac{\overline{H}}{v_y}\widetilde{\widetilde{\chi}}_{l}}_{L^2}
		\nn&\quad+
		\sum_{m,n\ge0}\sum_{k\in\ZZ}
		\sum_{l\le n/2} |\bold{a}_{m,n}|^2
		\binom{n}{l}
		\nnorm{    \omnk e^{W}\chi_{m+n}}_{L^2}\bold{a}_{m,n-l}^{-1}\bb_{l+1}^{-1}(l+1)^{2-2/s}
		\nn&\quad
		\times
		\bold{a}_{m,n-l}\enorm{  |k|^m  q^{n-l}\Gamma_k^{n-l}\omega_{k} e^{W}\chi_{m+n}}
		\bb_{l+1}(l+1)^{2/s-2}
		 \norm{\partial_y\paren{ q^l\pav^{l} \frac{\overline{H}}{v_y}} \widetilde{\widetilde{\chi}}_{l} }_{L^2}.
	\end{align}
	By Corollary~\ref{comb:boun}, using H\"older's and Young's inequalities, we deduce 
	\begin{align*}
		\sum_{m,n\ge0}\sum_{k\in\ZZ} \theta_{n}^2 T_{\gamma,1}^{m,n}
		&\lesssim
		\theta_{n}^2 \mathcal{E}^{(\gamma)}
		\nu^{100}
\left( \paren{\sum_{n\ge0}(n+1)^{2\sss-2}  \bold{a}_{n+1}^2  \left\|  \partial_{y} \paren{\frac{\overline{H}}{v_y}}_n e^{W/2} \chi_n \right\|_{L^2}^2}^{1/2}
\right. \\&\quad
\left.+  \paren{\sum_{n\ge0}(n+1)^{2\sss-2}  \bold{a}_{n+1}^2  \left\| \paren{\frac{\overline{H}}{v_y}}_n e^{W/2} \chi_n \right\|_{L^2}^2}^{1/2} \right).
	\end{align*}
Noting that
\begin{align}
	\label{1onvy}
	\frac{1}{v_y} = 	\frac{1}{1 + H} = \sum_{l=0}^{\infty} (-1)^l H^l,
\end{align}
we have by Lemma~\ref{pro:-s} and Lemma~\ref{pro:1:-s} that
\begin{align*}
	\paren{\sum_{n\ge0}(n+1)^{2\sss-2}  \bold{a}_{n+1}^2  \left\|  \partial_{y} \paren{\frac{\overline{H}}{v_y}}_n e^{W/2} \chi_n \right\|_{L^2}^2}^{1/2} \le &\nu^{-1/2}
	\sqrt{\cd_{\overline H}^{(\gamma)}}\paren{\sum_{i=0}^{\infty} \paren{\mathcal{E}_{ H}^{(\alpha)}}^{i/2}}
	\\ \lesssim &\nu^{-1/2}\sqrt{\cd_{\overline H}^{(\gamma)}},
\end{align*}
and 
\begin{align*}
	\paren{\sum_{n\ge0}(n+1)^{2\sss-2}  \bold{a}_{n+1}^2  \left\|  \paren{\frac{\overline{H}}{v_y}}_n e^{W/2} \chi_n \right\|_{L^2}^2}^{1/2} \le &\sqrt{\mathcal{E}^{(\gamma)}_{\overline H}} +  \nu^{-1/2}
	\paren{  \sqrt{\mathcal{E}^{(\gamma)}_{ H}} 
		\sqrt{\cd^{(\gamma)}_{\overline H}} +  \sqrt{\mathcal{E}^{(\gamma)}_{ \overline  H}}	\sqrt{\cd^{(\gamma)}_{ H}}} .
\end{align*}
Hence, we get
	\begin{align*}
	\sum_{m,n\ge0}\sum_{k\in\ZZ} \theta_{n}^2 T_{\gamma,1}^{m,n}
	\lesssim
  \eps  \nu^{99} \mathcal{E}^{(\gamma)} \paren{1 + \sqrt{\cd^{(\gamma)}_{ H}} + \sqrt{\cd^{(\gamma)}_{ \overline  H}}}
\end{align*}
by the bootstrap assumption and the boundedness of the sequence $\{\theta_{n}\}$.
	While for $T_{\gamma,2}^{m,n}$, we have
	\begin{align*}
		\abs{T_{\gamma,2}^{m,n}} 
		\lesssim&
		\sum_{n/2< l\le n-1} |\bold{a}_{m, n}|^2
		\binom{n}{l}
		\nnorm{    \omnk e^{W}\chi_{m+n}}_{L^2}
		\nnorm{q^l\pav^{l} \frac{\overline{H}}{v_y}e^{W/2}\chi_{n-1}}_{L^2}
		\\&\quad\times
		\norm{  |k|^m q^{n-l}\Gamma_k^{n-l}\omega_{k} e^{W/2} \chi_{m+n}}_{L^\infty}.
	\end{align*}
	By Sobolev embedding, we get
	\begin{align*}
		&\norm{   |k|^m q^{n-l}\Gamma_k^{n-l}\omega_{k} e^{W/2}\chi_{m+n}}_{L^\infty}
		\\&\qquad\lesssim
		\norm{  |k|^m  q^{n-l}\Gamma_k^{n-l}\omega_{k}e^{W/2} \chi_{m+n-l+1}}_{L^2}^{1/2}
		\norm{   |k|^m \partial_y \paren{ q^{n-l}\Gamma_k^{n-l}\omega_{k}e^{W/2} \chi_{m+n-l+1}}}_{L^2}^{1/2}	
		\\&\qquad\lesssim
		\norm{    \partial_y \omega_{m,n-l;k} e^{W/2} \chi_{n-l}}_{L^2}
		+ (m+n-l)^{1+\sigma}
		\norm{   \omega_{m,n-l;k} e^{W/2} \chi_{n-l}}_{L^2}
		\\&\qquad\qquad
		+ \nu^{-1}\norm{   \omega_{m,n-l;k} e^{W/2} \chi_{n-l}}_{L^2}
	\end{align*}
	A similar argument as~\eqref{T:gamm:1}, we arrive at
	\begin{align*}
		\sum_{m,n\ge0}\sum_{k\in\ZZ} \theta_{n}^2 T_{\gamma,2}^{m,n}
		&\lesssim \theta_{n}^2
		\nu^{100} \sqrt{\mathcal{E}^{(\gamma)}}\paren{\sqrt{\cd^{(\gamma)}}+\sqrt{\mathcal{CK}^{(\gamma)}}}
   \paren{\sum_{n\ge0}(n+1)^{2\sss-2}  \bold{a}_{n+1}^2  \left\|  \paren{\frac{\overline{H}}{v_y}}_n e^{W/2} \chi_n \right\|_{L^2}^2}^{1/2}
 \\&\lesssim \theta_{n}^2 \nu^{99}
  \sqrt{\mathcal{E}^{(\gamma) }}
 \paren{\sqrt{\cd^{(\gamma)}} + \sqrt{\mathcal{CK}^{(\gamma)}}} \paren{ \sqrt{\mathcal{E}^{(\gamma)}_{\overline H}} +  \sqrt{\mathcal{E}^{(\gamma)}_{ H}} 
 	\sqrt{\cd^{(\gamma)}_{\overline H}} +  \sqrt{\mathcal{E}^{(\gamma)}_{ \overline  H}} 
 	\sqrt{\cd^{(\gamma)}_{H}}}
 	\\& 
		\lesssim \eps
		\nu^{99}
		\sqrt{\mathcal{E}^{(\gamma) }}
		\paren{\sqrt{\cd^{(\gamma)}} + \sqrt{\mathcal{CK}^{(\gamma)}}}\paren{1 + \sqrt{\cd^{(\gamma)}_{ H}} + \sqrt{\cd^{(\gamma)}_{ \overline  H}}}
	\end{align*}
	concluding the proof, where we again used the boundedness of the sequence $\{\theta_{n}\}$.
\end{proof}
\begin{lemma}[Transport commutator for $\alpha$ estimate] 
	Let the assumptions in Lemma~\ref{tran:comm} hold. Then it follows 
	\label{tran:comm:alph}
	\begin{align}
		\label{tran:comm:esti:alph}
		&\sum_{m,n\ge 0}\sum_{k\neq 0} 
		\nu	\theta_{n}^2  |\bold{a}_{m,n}|^2   \abs{\brk{  
			\partial_y\omnk,  
			\partial_y \bold{C}^{(m,n)}_{\mathrm{trans}, k} e^{2W}\chi_{m+n}^2}}
				\nn & \lesssim
	\eps  \nu^{98}
	\sqrt{\mathcal{E}^{(\alpha) }} \sum_{\iota\in \{\alpha, \gamma \}} 
	\paren{\sqrt{\mathcal{E}^{(\iota) }} + \sqrt{\cd^{(\iota)}} + \sqrt{\mathcal{CK}^{(\iota)}}} \sum_{\iota\in \{\alpha, \gamma \}}
	 \paren{1 + \sqrt{\cd^{(\iota)}_{ H}} + \sqrt{\cd^{(\iota)}_{ \overline  H}}}.
	\end{align}
\end{lemma}
\begin{proof}
	From~\eqref{comm:tran}, we get
	\begin{align*}
		\partial_y \bold{C}^{(m,n)}_{\mathrm{trans}, k} 
		&= |k|^m 
		\sum_{l = 1}^n \binom{n}{l} \partial_y\paren{q^{l-1}\pav^{l-1}
			\frac{\overline{H}}{v_y}} 
		q^{n-l+1}\Gamma_k^{n-l+1} \omega_k
		+|k|^m 
		\sum_{l = 1}^n \binom{n}{l} q^{l-1}\pav^{l-1}
		\frac{\overline{H}}{v_y}
		\partial_y\paren{q^{n-l+1}\Gamma_k^{n-l+1} \omega_k}.
	\end{align*}
	Following the estimates in Lemma~\ref{tran:comm} line by line for the above two terms, upgrading the corresponding regularity from $\gamma$ to $\alpha$, we get~\eqref{tran:comm:esti:alph}.
\end{proof}
\begin{lemma}[Transport commutator for $\mu$ estimate] 
	Let $\bold{C}^{(m,n)}_{\mathrm{trans}, k}$ be defined as above. It holds
	\label{tran:comm:mu}
	\begin{align}
		\sum_{m,n\ge 0}\sum_{k\neq 0} & \theta_{n}^2 |\bold{a}_{m,n}|^2 \abs{\brk{ \omnk,  \bold{C}^{(m,n)}_{\mathrm{trans}, k} e^{2W}\chi_{m+n}^2}}
				\nn & \lesssim
\eps
\nu^{99}
\sqrt{\mathcal{E}^{(\mu) }}
\paren{\sqrt{\mathcal{E}^{(\mu) }} + \sqrt{\cd^{(\mu)}} + \sqrt{\mathcal{CK}^{(\mu)}}}\paren{1 + \sqrt{\cd^{(\mu)}_{ H}} + \sqrt{\cd^{(\mu)}_{ \overline  H}}}.
	\end{align}
\end{lemma}
\begin{proof}
	By the same argument as in Lemma~\ref{visc:comm:mu}, we obtain the desired bound in the statement. 
\end{proof}

\subsection{Proof of Proposition \ref{pro:rhs:intro}}

We now consolidate all the bounds from this section in order to prove Proposition \ref{pro:rhs:intro}. 
\begin{proof}[Proof of Proposition \ref{pro:rhs:intro}] We subdivide the proof into three steps.

\begin{proof}[Proof of \eqref{Im:qu:a}, \eqref{Im:qu:d}, and \eqref{Im:qu:g}]
We first invoke the bound \eqref{CommCnqbd} to obtain the following estimate for $|\mathcal{I}_q^{(\gamma)}|$ 
\begin{align*}
|\mathcal{I}_{q}^{(\gamma)}| \le & \frac{1}{B}\mathcal{D}^{(\gamma)}+CB\sum_{m = 0}^\infty \myr{\sum_{n = 1}^\infty}  \sum_{k \in \mathbb{Z}}\sum_{\ell=0}^{n-1}(\mathfrak C\lambda^{s})^{2n-2\ell} \lf(\frac{(m+n)!}{(m+\ell)!}\rg)^{-2\sigma-2\sigma_\ast}  \mathcal{D}^{(\gamma)}_{m,\ell;k}. 
\end{align*}
Now we invoke the technical inequality $\eqref{tht_app_1}_{\mf G_{\cdots}=\mathcal{D}^{(\gamma)}_{\cdots}}$ and select the $\delta_{\text{Drop}},\ n_\ast$ in Lemma \ref{lem:tht_app} appropriately (to guarantee that the parameter $\mf B$ is large enough) to obtain the result. \ifx
\siming{Previous:
\begin{align} \n
|\mathcal{I}_q^{(\gamma)}| = &\lf| \sum_{m = 0}^\infty \sum_{n = 0}^\infty \sum_{k \in \mathbb{Z}} \theta_{n}^2 \bold{a}_{m,n}^2 \mathrm{Re} \lf\langle  \bold{C}^{(m,n)}_{q,k} , \omega_{m,n;k}  e^{2W} \chi_{m + n}^2\rg \rangle\rg| \\ \n
\le & \lf| \sum_{m = 0}^\infty \sum_{n = 0}^\infty \mathbbm{1}_{n < n_{\ast}} \sum_{k \in \mathbb{Z}} \theta_{n}^2 \bold{a}_{m,n}^2 \mathrm{Re} \lf\langle  \bold{C}^{(m,n)}_{q,k} , \omega_{m,n;k}  e^{2W} \chi_{m + n}^2\rg \rangle\rg| \\ \n
& +  \lf| \sum_{m = 0}^\infty \sum_{n = 0}^\infty \mathbbm{1}_{n \ge n_{\ast}} \sum_{k \in \mathbb{Z}}  \bold{a}_{m,n}^2 \mathrm{Re} \lf\langle  \bold{C}^{(m,n)}_{q,k} , \omega_{m,n;k}  e^{2W} \chi_{m + n}^2\rg \rangle\rg| := |\mathcal{I}_{q, low}^{(\gamma)}| + |\mathcal{I}_{q,hi}^{(\gamma)}|. 
\end{align}
For $\mathcal{I}_{q,hi}^{(\gamma)}$, we invoke the bound \eqref{CommCnqbd} to obtain 
\begin{align*}
|\mathcal{I}_{q,hi}^{(\gamma)}| \le & \sum_{m = 0}^\infty \sum_{n = n_\ast}^\infty  \sum_{k \in \mathbb{Z}} \lf(\frac{1}{B}\mathcal{D}^{(\gamma)}_{m,n;k}+CB\sum_{\ell=0}^{n-1}(\mathfrak C\lambda^{s})^{2n-2\ell} \lf(\frac{(m+n)!}{(m+\ell)!}\rg)^{-2\sigma-2\sigma_\ast}  \mathcal{D}^{(\gamma)}_{m,\ell;k}\rg) \\
\le & \frac{1}{B} \mathcal{D}^{(\gamma)} +CB  \sum_{m = 0}^\infty \sum_{n = n_\ast}^\infty  \sum_{k \in \mathbb{Z}}\sum_{\ell=0}^{n-1}(\mathfrak C\lambda^{s})^{2n-2\ell} \lf(\frac{(m+n)!}{(m+\ell)!}\rg)^{-2\sigma-2\sigma_\ast}  \mathcal{D}^{(\gamma)}_{m,\ell;k}.
\end{align*}
For frozen $(m, k)$, we study the sum 
\begin{align*}
S_{m,k} := \sum_{n = n_\ast}^\infty \sum_{\ell = 0}^{n-1} (\mathfrak C\lambda^{s})^{2n-2\ell} \lf(\frac{(m+n)!}{(m+\ell)!}\rg)^{-2\sigma-2\sigma_\ast}  \mathcal{D}^{(\gamma)}_{m,\ell;k}.
\end{align*}
By interchanging the order of summation, we obtain 
\begin{align*}
S_{m,k} =& \sum_{\ell = 0}^{n_\ast - 1} \sum_{n = n_\ast}^\infty (\mathfrak C\lambda^{s})^{2n-2\ell} \lf(\frac{(m+n)!}{(m+\ell)!}\rg)^{-2\sigma-2\sigma_\ast}  \mathcal{D}^{(\gamma)}_{m,\ell;k} \\
&+ \sum_{\ell = n_\ast}^\infty \sum_{n = \ell + 1}^\infty (\mathfrak C\lambda^{s})^{2n-2\ell} \lf(\frac{(m+n)!}{(m+\ell)!}\rg)^{-2\sigma-2\sigma_\ast}  \mathcal{D}^{(\gamma)}_{m,\ell;k} = S^{(\text{rectangle})}_{m,k} + S^{(\text{triangle})}_{m,k}
\end{align*}
To ease the notation going forward, we define 
\begin{align*}
A := (\mathfrak C\lambda^{s})^2, \qquad b := 2(\sigma + \sigma_\ast), \qquad \mathfrak{K}: \mathbb{N} \rightarrow \mathbb{R}, \qquad \mathfrak{K}(z) = \frac{A^z}{[(m+z)!]^b}.
\end{align*}

\vspace{2 mm}

\noindent \underline{Estimation of $S^{(\text{triangle})}_{m,k}$:} This is the more challenging of these bounds due to the double infinite sum in $\ell, n$.  To handle such a sum, we issue the following 
\begin{claim} \label{cl:oop} Fix any $C > 1$. There exists $L_\ast$ sufficiently large relative to $A, b, C$ and universal constants such that the function $\mathfrak{K}$ satisfies the following inequality
\begin{align}
\sum_{n = L_\ast}^\infty \mathfrak{K}(n) \le C  \mathfrak{K}(L_\ast).
\end{align}
\end{claim}
\begin{proof}[Proof of Claim \ref{cl:oop}] This follows almost immediately by comparing $\mathfrak{K}(n)$ to a geometric function. More precisely, define 
\begin{align*}
R > 1, \qquad C_0 = C_0(L_\ast) = R^{L_\ast} \frac{A^{L^\ast}}{[(m+L_\ast)!]^b}, \qquad \tau(n) = C_0 \frac{1}{R^n}.
\end{align*}
Such a choice of $C_0$ ensures that 
\begin{align}
\tau(L_\ast) = \mathfrak{K}(L_\ast). 
\end{align}
We first note that, due to the standard geometric sum properties, 
\begin{align*}
\sum_{n = L_\ast}^\infty \tau(n) = C_0 \sum_{n = L_\ast}^\infty \frac{1}{R^n} = C_0 \frac{1}{R^{L_\ast}} \frac{1}{1-\frac{1}{R}} =  \frac{1}{1-\frac{1}{R}} \tau(L_\ast).
\end{align*}
We secondly notice that for $L_\ast$ sufficiently large, we have the order property 
\begin{align} \label{order:r:1}
\mathfrak{K}(n) \le \tau(n), \qquad n \ge L_\ast.
\end{align}
To establish, \eqref{order:r:1}, we notice that $K(L_\ast) = \tau(L_\ast)$. Therefore, it suffices to check 
\begin{align*}
\frac{\mathfrak{K}(n+1)}{\mathfrak{K}(n)} \le \frac{\tau(n+1)}{\tau(n)}, \qquad n \ge L_\ast, 
\end{align*}
which is equivalent to 
\begin{align*}
\frac{A}{(m+n+1)^b} \le \frac{1}{R}, \qquad n \ge L_\ast, 
\end{align*}
which can clearly be achieved by choosing $L_\ast$ sufficiently large. Therefore, we have 
\begin{align*}
\sum_{n = L_\ast}^\infty \mathfrak{K}(n) \le \sum_{n = L_\ast}^\infty \tau(n) \le \frac{1}{1-\frac{1}{R}} \tau(L_\ast) = \frac{1}{1 - \frac{1}{R}} \mathfrak{K}(L_\ast). 
\end{align*}
\end{proof}

Returning now to the bound, we have by choosing $n_\ast$ sufficiently large
\begin{align} \n
S^{(\text{triangle})}_{m,k} = & \sum_{\ell = n_\ast}^\infty  \frac{1}{\mathfrak{K}(\ell)}  \mathcal{D}^{(\gamma)}_{m,\ell;k} ( \sum_{n = \ell + 1}^\infty \mathfrak{K}(n)  ) \le \sum_{\ell = n_\ast}^\infty  \frac{\mathfrak{K}(\ell + 1)}{\mathfrak{K}(\ell)}  \mathcal{D}^{(\gamma)}_{m,\ell;k} \\ \label{shst:1}
= &C \sum_{\ell = n_\ast}^\infty \frac{A^{\ell + 1}}{[(m + \ell + 1)!]^b} \frac{[(m+\ell)!]^b}{A^\ell}   \mathcal{D}^{(\gamma)}_{m,\ell;k} \le  \frac{CA}{(m+n_\ast + 1)^b} \sum_{\ell = n_\ast}^\infty   \mathcal{D}^{(\gamma)}_{m,\ell;k}.
\end{align}
By taking $n_\ast$ large and summing over $m, k$, we obtain 
\begin{align*}
\sum S^{(\text{triangle})}_{m,k} \le o(1) \mathcal{D}^{(\gamma)}.
\end{align*}
\vspace{2 mm}

\noindent \underline{Estimation of $S^{(\text{rectangle})}_{m,k}$:} This bound works in a nearly identical manner to $\mathcal{S}^{(\text{triangle})}_{m,k}$. 

\vspace{2 mm}

\noindent \underline{Bound of $\mathcal{I}^{(\gamma)}_{q,low}$:} We have upon again invoking \eqref{CommCnqbd}
\begin{align} \n
|\mathcal{I}^{(\gamma)}_{q,low}| \le  &\frac{1}{B} \mathcal{D}^{(\gamma)} + C_B C_{n_\ast} \sum_{m = 0}^\infty \sum_{n = 0}^{n_\ast} \sum_{k \in \mathbb{Z}} \theta_n^2 \sum_{\ell = 0}^{n-1} \mathcal{D}_{m,\ell;k} \\ \n
\le & \frac{1}{B} \mathcal{D}^{(\gamma)} + C_B C_{n_\ast} \sum_{m = 0}^\infty \sum_{n = 0}^{n_\ast} \sum_{k \in \mathbb{Z}}  \sum_{\ell = 0}^{n-1} \Big( \frac{\theta_n}{\theta_{\ell}} \Big)^2 \theta_\ell^2 \mathcal{D}_{m,\ell;k} \\ \n
\le & \frac{1}{B} \mathcal{D}^{(\gamma)} + C_B C_{n_\ast} \sum_{m = 0}^\infty \sum_{n = 0}^{n_\ast} \sum_{k \in \mathbb{Z}}  \sum_{\ell = 0}^{n-1} \delta_{\text{Drop}}^2  \theta_\ell^2 \mathcal{D}_{m,\ell;k} \\ \label{delta:drop:choice}
= &  \frac{1}{B} \mathcal{D}^{(\gamma)} + \delta_{\text{Drop}}^2 C_B C_{n_\ast} \sum_{m = 0}^\infty  \sum_{k \in \mathbb{Z}}  \sum_{\ell = 0}^{\infty}   \theta_\ell^2 \mathcal{D}_{m,\ell;k},
\end{align}
which upon bringing $\delta_{\text{Drop}}$ small relative to $n_\ast$, we obtain the right-hand side is bounded by $o(1) \mathcal{D}^{(\gamma)}$. 
}\fi

This concludes the proof of \eqref{Im:qu:a}, whereas the proofs of \eqref{Im:qu:d}, \eqref{Im:qu:g} are nearly identical. 
\end{proof}

\begin{proof}[Proof of \eqref{Im:qu:b}, \eqref{Im:qu:e}, and \eqref{Im:qu:h}]
It is a direct consequence of Lemma~\ref{tran:comm}, Lemma~\ref{tran:comm:alph}, and Lemma~\ref{tran:comm:mu}. In fact, noting the sequence $\{\theta_{n}\}$ is bounded, 
we have 
\begin{align*}
	|\mathcal{I}^{(\gamma)}_{\mathrm{trans}}| \le \sum_{n,m\ge 0}\sum_{k\neq 0} & |\bold{a}_{ m,n}|^2 \theta_{n}^2 \abs{\brk{   \omega_{m, n; k},  |k|^m \bold{C}^{(n)}_{trans, k} e^{2W}\chi_{m+n}^2}}
		\\&\lesssim
	\sum_{n,m\ge 0}\sum_{k\neq 0} \abs{\brk{ |\bold{a}_{ m,n}|^2    \omega_{m, n; k},  |k|^m \bold{C}^{(n)}_{trans, k} e^{2W}\chi_{m+n}^2} }
	\nn & \lesssim \sqrt{	\mathcal{E}^{(\gamma)}}  \paren{ \cd^{(\gamma)} + \mathcal{CK}^{(\gamma)} 
		+ \cd_{\overline H}^{(\gamma)}} + \frac{\paren{\mathcal{E}^{(\gamma)}}^{3/2}+ \paren{\mathcal{E}^{(\gamma)}_{\overline H}}^{3/2}}{\brak{t}^2}
	\nn &
	\lesssim \eps \paren{ \cd^{(\gamma)} + \mathcal{CK}^{(\gamma)} 
		+   \cd_{\overline H}^{(\gamma)}} + \frac{\eps^3}{\brak{t}^2}.
\end{align*}
And the proof of \eqref{Im:qu:e} and \eqref{Im:qu:h} follows similarly.
\end{proof}

\begin{proof}[Proof of \eqref{Im:qu:c}, \eqref{Im:qu:f}, and \eqref{Im:qu:i}]
	According to the definition of $\mathcal{I}^{(\gamma)}_{\mathrm{visc}}$, these inequalities are exactly the statement of Lemma~\ref{visc:comm}, Lemma~\ref{visc:comm:alph}, and Lemma~\ref{visc:comm:mu},
	concluding the proof.
\end{proof}

\end{proof}
\subsection{Technical Lemmas: Gevrey Norm Estimates on Physical Side}

\subsubsection{Binomial coefficients lemmas}
We need the following lemmas about the binomial coefficients bound from our companion paper~\cite{BHIW24a}.

\begin{lemma}
	\label{comb:boun:vari:2}
	For $l\le n/2$ and $n\ge 5$, we have
	\begin{align}
		\label{comb:boun:vari:2:est}
	\brak{t}^{-2}	\bold{a}_{m,n+1} \binom{n}{l}\bold{a}_{m+1,l}^{-1}\binom{n-l}{j}\bold{a}_{0,j}^{-1}\bold{a}_{0,n-l-j}^{-1}
		\le \paren{\frac{1}{2}}^{l\paren{\sss-1}}.
	\end{align}
	Moreover, if $m\le c_0n$ for some $c_0>0$, then it holds
	\begin{align}
		\label{comb:boun:est:vari:2:refi}
	\brak{t}^{-2}	\bold{a}_{m,n+1} \binom{n}{l}\bold{a}_{m+1,l}^{-1}\binom{n-l}{j}\bold{a}_{0,j}^{-1}\bold{a}_{0,n-l-j}^{-1}
		\le \paren{\frac{1}{2}}^{l\paren{\sss-1}}
		\paren{\frac{c_0+1/2}{c_0+1}}^{m\sss}.
	\end{align}
	If $m\ge C_0n$, then we have
	\begin{align}
		\label{comb:boun:est:vari:2:refi1}
		\brak{t}^{-2} \bold{a}_{m,n+1}& \binom{n}{l}\bold{a}_{m+1,l}^{-1}\binom{n-l}{j}\bold{a}_{0,j}^{-1}\bold{a}_{0,n-l-j}^{-1} \le
		\paren{\frac{1}{2^l}}^{\sss-1}\paren{\frac{1}{C_0+1}}^{(n-l)\sss}.
	\end{align}
\end{lemma}
\begin{remark}
	Note that in the above theorem, for $\bold{a}_{\alpha,\beta}$, it is essentially the sum $\alpha+\beta$ that matters. Hence, we are able to generalize it to a more general version:
	\begin{align}
		\label{comb:boun:rem}
		\brak{t}^{-2} \bold{a}_{m,n+1} \binom{n}{l}\bold{a}_{\alpha_1, \beta_1}^{-1}\binom{n-l}{j}\bold{a}_{\alpha_2, \beta_2}^{-1}\bold{a}_{\alpha_3, \beta_3}^{-1}
		\le \paren{\frac{1}{2}}^{l\paren{\sss-1}}.
	\end{align}
	with 
	\begin{align*}
		\alpha_1+\beta_1=m+l+1,\    \alpha_2+\beta_2=j,\     \alpha_3+\beta_3=n-l-j.
	\end{align*}
	Of course, the corresponding inequalities for \eqref{comb:boun:est:vari:2:refi} and \eqref{comb:boun:est:vari:2:refi1} also holds.
\end{remark}

\begin{remark}
	From Theorem~\ref{comb:boun:vari:2}, one may also obtain
	\begin{align}
		\label{comb:boun:rem:1}
		\brak{t}^{-2} \bold{a}_{m,n+1} \binom{n}{l}\bold{a}_{m+1,l}^{-1}\binom{n-l}{j}\bold{a}_{0,j}^{-1}\bold{a}_{0,n-l-j}^{-1} m^a
		\le C \paren{\frac{1}{2}}^{l\paren{\sss-1}} 
	\end{align}
	for a parameter $a\le n/4$, where the constant $C$ depends on $a$.
\end{remark}

Two direct corollaries are the following.
\begin{corollary}
	\label{comb:boun}
	For $l\le n/2$ and $n\ge 5$, we have
	\begin{align}
		\label{comb:boun:est}
	\brak{t}^{-1}	\bold{a}_{m,n} \binom{n}{l}\bold{a}_{m,l}^{-1}\bold{a}_{0,n-l}^{-1} \le \paren{\frac{1}{2}}^{l\paren{\sss-1}}.
	\end{align}
	Moreover, if $m\le c_0n$ for some $c_0>0$, then it holds
	\begin{align}
		\label{comb:boun:est:refi}
	\brak{t}^{-1}	\bold{a}_{m,n} \binom{n}{l}\bold{a}_{m,l}^{-1}\bold{a}_{0,n-l}^{-1} \le \paren{\frac{1}{2}}^{l\paren{\sss-1}}
		\paren{\frac{c_0+1/2}{c_0+1}}^{m\sss}.
	\end{align}
	If $m\ge C_0n$, then we have
	\begin{align}
		\label{comb:boun:est:refi1}
		\brak{t}^{-1} \bold{a}_{m,n} \binom{n}{l}\bold{a}_{m,l}^{-1}\bold{a}_{0,n-l}^{-1} \le
		\paren{\frac{1}{2^l}}^{\sss-1}\paren{\frac{1}{C_0+1}}^{(n-l)\sss}.
	\end{align}
\end{corollary}
\begin{corollary}
	\label{comb:boun:vari:1}
	For $l\le n/2$ and $n\ge 5$, we have
	\begin{align}
		\label{comb:boun:vari:1:est}
		\brak{t}^{-1}\bold{a}_{m,n+1} \binom{n}{l}\bold{a}_{m+1,l}^{-1}\bold{a}_{0,n-l}^{-1} \le \paren{\frac{1}{2}}^{l\paren{\sss-1}}.
	\end{align}
	Moreover, if $m\le c_0n$ for some $c_0>0$, then it holds
	\begin{align}
		\label{comb:boun:est:vari:1:refi}
		\brak{t}^{-1}\bold{a}_{m,n+1} \binom{n}{l}\bold{a}_{m+1,l}^{-1}\bold{a}_{0,n-l}^{-1} \le \paren{\frac{1}{2}}^{l\paren{\sss-1}}
		\paren{\frac{c_0+1/2}{c_0+1}}^{m\sss}.
	\end{align}
	If $m\ge C_0n$, then we have
	\begin{align}
		\label{comb:boun:est:vari:1:refi1}
		\brak{t}^{-1}\bold{a}_{m,n+1} \binom{n}{l}\bold{a}_{m+1,l}^{-1}\bold{a}_{0,n-l}^{-1}
		\le
		\paren{\frac{1}{2^l}}^{\sss-1}\paren{\frac{1}{C_0+1}}^{(n-l)\sss}.
	\end{align}
\end{corollary}

\subsubsection{Product rule}
We need the following product lemmas from our companion paper~\cite{BHIW24a}.

\begin{lemma}
	\label{pro:-s}
	For $f|_{y=\pm1}=g|_{y=\pm1}=0$, we have
	\begin{align*}
		\norm{fg}_{Y_{0,-s}} \lesssim 
		\norm{ f}_{Y_{0,-s}}
		\norm{g}_{\overline Y_{1,0}} + \norm{f}_{Y_{1,-s}}
		\norm{g}_{\overline Y_{0,0}}.
	\end{align*}
	Without the homogeneous boundary condition, it holds
	\begin{align*}
		\norm{fg}_{Y_{0,-s}} \lesssim 
		\norm{ f}_{Y_{0,-s}}
		\paren{\norm{g}_{\overline Y_{1,0}} + \enorm{g}} + \paren{\norm{f}_{Y_{1,-s}} + \enorm{f}}
		\norm{g}_{\overline Y_{0,0}}.
	\end{align*}
\end{lemma}

\begin{lemma}
	\label{pro:1}
	Assume $f|_{y=\pm1}=g|_{y=\pm1}=0$, we have
	\begin{align*}
		\norm{fg}_{Y_{1,0}} \lesssim 
		\norm{f}_{Y_{1,0}}
		\norm{g}_{\overline Y_{1,0}}.
	\end{align*}
	Without the boundary condition requirement, it holds
	\begin{align*}
		\norm{fg}_{Y_{1,0}} \lesssim 
		\paren{\norm{f}_{Y_{1,0}} + \enorm{f}}
		\paren{\norm{g}_{\overline Y_{1,0}} + \enorm{g}}.
	\end{align*}
\end{lemma}

\begin{lemma}
	\label{pro:1:-s}
	Suppose $f|_{y=\pm1}=g|_{y=\pm1}=0$, we have
	\begin{align*}
		\norm{fg}_{Y_{1,-s}} \lesssim 
		\norm{f}_{Y_{1,-s}}
		\norm{g}_{\overline Y_{1,0}}.
	\end{align*}
	Without boundary value requirement, it holds
	\begin{align*}
		\norm{fg}_{Y_{1,-s}} \lesssim 
		\paren{\norm{f}_{Y_{1,-s}} + \enorm{f}}
		\paren{\norm{g}_{\overline Y_{1,0}} + \enorm{g}}.
	\end{align*}
\end{lemma}

\subsubsection{Convolution type lemmas}
We also need the following versions of  convolution type lemma, proof of which may be found in our companion paper~\cite{BHIW24a}.
\begin{lemma}
	\label{con:no:k}
	It holds
	\begin{align*}
		\sum_{n\ge 0}\sum_{0\le m\le n}\sum_{k\in\mathbb{Z}} f_{n-m,k}g_{m,k}h_n
		\le
		\sum_{n\ge 0} \paren{\sum_{k\in\mathbb{Z}}|f_{n,k}|^2}^{1/2} 
		\paren{\sum_{n\ge 0}\sum_{k\in\mathbb{Z}}|g_{n,k}|^2}^{1/2}
		\paren{\sum_{n\ge 0}|h_{n}|^2}^{1/2}.
	\end{align*}
\end{lemma}
\begin{lemma}
	\label{con:k}
	It holds
	\begin{align*}
		\sum_{n\ge 0}\sum_{0\le m\le n}\sum_{k\in\mathbb{Z}}\sum_{l\in\mathbb{Z}} f_{n-m,k-l}g_{m,l}h_{n,k}
		\le
		\paren{\sum_{n\ge 0}\sum_{k\in\mathbb{Z}}|f_{n,k}|}
		\paren{\sum_{n\ge 0}\sum_{k\in\mathbb{Z}}|g_{n,k}|^2}^{1/2}
		\paren{\sum_{n\ge 0}\sum_{k\in\mathbb{Z}}|h_{n,k}|^2}^{1/2}.
	\end{align*}
\end{lemma}

\subsection{Combinatorial lemma}
To derive the estimates \eqref{SnBd1sum} and \eqref{SnBd2sum}, we need the following combinatorial lemmas. 
\begin{lemma}\label{lem:comb_sqr}
The following inequality holds
\begin{align}\label{sum_comb}
\sum_{\ell=0}^{n-1}\mathfrak{C}^{n-\ell} \lf(\frac{n!}{\ell!}\rg)^{-\sigma-\sigma_\ast} \leq C,\quad \forall n\in \mathbb{N}.
\end{align}
Here the bound $C$ only depends on $\mathfrak C\geq 1, \, \sigma,\, \sigma_\ast.$
\end{lemma}
\begin{proof}
We define a threshold $\mathfrak L=\max\lf\{(2\mathfrak{C})^{\frac{1}{\sigma+\sigma_\ast}},10\rg\}.$ We observe that if $n<2\mathfrak L^2+2$, then the sum in \eqref{sum_comb} is bounded in terms of $\mathfrak C, \, \sigma,\, \sigma_\ast.$ Hence we assume that $n\geq 2\mathfrak{L}^2+2$ without loss of generality. Now we further decompose the sum into two parts 
\begin{align}
\sum_{\ell=0}^{n-1}\mathfrak{C}^{n-\ell} \lf(\frac{n!}{\ell!}\rg)^{-\sigma-\sigma_\ast} =\lf(\sum_{\ell=0}^\mathfrak {L} +\sum_{\ell=\mathfrak L +1}^{n-1}\rg)=\mathcal{I}_1+\mathcal{I}_2.\label{I_cm}
\end{align}
Now we focus on the $\mathcal{I}_{2}$ term. In this regime, we can compute the ratio between consecutive terms $\ell$, $\ell+1$:
\begin{align*}
\frac{\mathfrak C^{n-\ell}\lf(\frac{n!}{\ell!}\rg)^{-\sigma-\sigma_\ast}}{\mathfrak C^{n-\ell-1}\lf(\frac{n!}{(\ell+1)!}\rg)^{-\sigma-\sigma_\ast}}=\frac{\mathfrak C}{(\ell+1)^{\sigma+\sigma_\ast}}. 
\end{align*}Hence, if $\ell\geq \mathfrak L$, then the ratio is less than $\frac{1}{2}$ and we can estimate the $\mathcal{I}_2$ as follows
\begin{align}
\mathcal{I}_2\leq \sum_{\ell=\mathfrak L+1}^{n-1}\lf(\frac{1}{2}\rg)^{n-1-\ell}\frac{\mathfrak {C}}{n^{\sigma+\sigma_\ast}}\leq 1.\label{I_cm_2} 
\end{align}Now we estimate the term $\mathcal{I}_{1}$ in \eqref{I_cm},
\begin{align}
\mathcal{I}_1\leq (1+\mathfrak{L})\max_{\ell\in\{0, ...,\mathfrak{L}\} }\mathfrak{C}^{2\frac{n}{2}}\lf(n(n-1)\cdot\cdot\cdot(\ell+1)\rg)^{-\sigma-\sigma_\ast}\leq (1+\mathfrak{L})\lf(\frac{\mathfrak C^2}{\lfloor n/2\rfloor^{\sigma+\sigma_\ast}}\rg)^{\lfloor\frac{n}{2}\rfloor}\mathfrak C \leq (1+\mathfrak{L})\mathfrak C. \label{I_cm_1}
\end{align} 
Combining the discussion above and the decomposition \eqref{I_cm} and the estimates  \eqref{I_cm_2}, \eqref{I_cm_1}, we have obtained \eqref{sum_comb}. 
\end{proof}
\ifx

\begin{corollary}
We have that 
\begin{subequations}
\begin{align}
\sum_{n=0}^\infty\bold{a}_{m,n}^2\nu\|S_{m,n;k} e^W\chi_{n-1}\|_{L^2}^2\lesssim& \sum_{n=0}^\infty\lf(\mathcal{D}_{m,n;k}^{(\gamma)}\rg)^2;\\
\sum_{n=0}^\infty\bold{a}_{m,n} ^2\nu^{\frac53}|k|^{\frac43} \lf \| S_{m,n;k} e^W  \chi_{n -1} \rg\|_{L^2} ^2 \lesssim & \sum_{n=0}^{\infty}\lf(\mathcal{D}^{(\mu)}_{m,n;k}\rg)^2,  \\  
\end{align}\end{subequations}
\end{corollary}
\begin{proof}
The estimate \eqref{sum_comb}, when combined with \eqref{S:est:1} type estimates, yields  expression
\begin{align}
\bold{a}_{m,n}^2\nu\|S_{m,n;k}e^W\chi_{n-1}\|_{L^2}^2\lesssim \sum_{\ell=0}^{n-1}\lf(\frac{n!}{\ell!}\rg)^{-\sigma-\sigma_\ast}(\mathcal{D}_{m,\ell;k}^{(\gamma)})^2. 
\end{align}
Hence 
\begin{align}
\sum_{n=0}^\infty\bold{a}_{m,n}^2\nu\|S_{m,n;k} e^W\chi_{n-1}\|_{L^2}^2\lesssim\sum_{n=0}^\infty\sum_{\ell=0}^{n-1}\lf(\frac{n!}{\ell!}\rg)^{-\sigma-\sigma_\ast}(\mathcal{D}_{m,\ell;k}^{(\gamma)})^2.
\end{align}
Now by the Fubini-Tonelli,
\begin{align}
\sum_{n=0}^\infty&\bold{a}_{m,n}^2\nu\|S_{m,n;k} e^W\chi_{n-1}\|_{L^2}^2\\
\lesssim&\sum_{\ell=0}^\infty\sum_{n=\ell+1}^{\infty}\lf(\frac{n!}{\ell!}\rg)^{-\sigma-\sigma_\ast}(\mathcal{D}_{m,\ell;k}^{(\gamma)})^2
\lesssim\sum_{\ell=0}^\infty\sum_{n=\ell+1}^{\infty}\lf( {n(n-1)\cdot\cdot\cdot(\ell+1)}\rg)^{-\sigma-\sigma_\ast}(\mathcal{D}_{m,\ell;k}^{(\gamma)})^2\\\lesssim &\sum_{\ell=0}^\infty\sum_{n-\ell=1}^{\infty}\lf( ({n-\ell+1)(n-\ell)\cdot\cdot\cdot 1}\rg)^{-\sigma-\sigma_\ast}(\mathcal{D}_{m,\ell;k}^{(\gamma)})^2\\
\lesssim &\sum_{\ell=0}^\infty(\mathcal{D}_{m,\ell;k}^{(\gamma)})^2 \sum_{n=1}^{\infty}\lf( (n+1)!\rg)^{-\sigma-\sigma_\ast}\\
\lesssim &\sum_{\ell=0}^\infty(\mathcal{D}_{m,\ell;k}^{(\gamma)})^2. 
 \end{align}
\end{proof}
\fi


{
\begin{lemma}\label{lem:tht_app}
Given a constant $\mf B>1$, there exists $\delta_{\text{Drop}}(\sigma,\mf B), n_\star(\sigma,\mf B)$ (defined in \eqref{dl_Drp} such that the following estimate holds
\begin{align}
\sum_{m = 0}^\infty \sum_{n =1}^\infty \sum_{\ell=0}^{n-1}\theta_n^2 (\mathfrak C\lambda^{s})^{2n-2\ell} \lf(\frac{(m+n)!}{(m+\ell)!}\rg)^{-2\sigma}  \mathfrak{G}_{m,\ell;k}  \leq \frac{1}{\mf B}\sum_{m=0}^\infty \sum_{n =0}^\infty \theta_n^2 \mathfrak{G}_{m,n;k}.\label{tht_app}
\end{align}
Here $\mathfrak{G}_{m,\ell;k}$ are positive quantities.
\end{lemma}
\begin{proof} We denote the left hand side of \eqref{tht_app} as follows
\begin{align} \n \eqref{tht_app}_{{L.H.S}}= &  \sum_{m = 0}^\infty \sum_{n = 1}^\infty\sum_{\ell=0}^{n-1} \mathbbm{1}_{n < n_{\ast}} \theta_{n}^2  (\mathfrak C\lambda^{s})^{2n-2\ell} \lf(\frac{(m+n)!}{(m+\ell)!}\rg)^{-2\sigma}  \mathfrak{G}_{m,\ell;k}    \\ \n
& +    \sum_{m = 0}^\infty \sum_{n = 1}^\infty \sum_{\ell=0}^{n-1}\mathbbm{1}_{n \ge n_{\ast}}  (\mathfrak C\lambda^{s})^{2n-2\ell} \lf(\frac{(m+n)!}{(m+\ell)!}\rg)^{-2\sigma}  \mathfrak{G}_{m,\ell;k}   := T_1+T_2.\label{tht_app_1} 
\end{align}
For $T_2$, we froze $(m, k)$, and study the sum 
\begin{align*}
S_{m,k} := \sum_{n = n_\ast}^\infty \sum_{\ell = 0}^{n-1} (\mathfrak C\lambda^{s})^{2n-2\ell} \lf(\frac{(m+n)!}{(m+\ell)!}\rg)^{-2\sigma}  \mathfrak{G}_{m,\ell;k}.
\end{align*}
By interchanging the order of summation, we obtain 
\begin{align*}
S_{m,k} =& \sum_{\ell = 0}^{n_\ast - 1} \sum_{n = n_\ast}^\infty (\mathfrak C\lambda^{s})^{2n-2\ell} \lf(\frac{(m+n)!}{(m+\ell)!}\rg)^{-2\sigma}  \mathfrak{G}_{m,\ell;k} \\
&+ \sum_{\ell = n_\ast}^\infty \sum_{n = \ell + 1}^\infty (\mathfrak C\lambda^{s})^{2n-2\ell} \lf(\frac{(m+n)!}{(m+\ell)!}\rg)^{-2\sigma}  \mathfrak{G}_{m,\ell;k} = S^{(\text{rectangle})}_{m,k} + S^{(\text{triangle})}_{m,k}.
\end{align*}
To ease the notation going forward, we define 
\begin{align*}
A := (\mathfrak C\lambda^{s})^2, \qquad b := 2\sigma , \qquad \mathfrak{K}: \mathbb{N} \rightarrow \mathbb{R}, \qquad \mathfrak{K}(z) = \frac{A^z}{[(m+z)!]^b}.
\end{align*}

\vspace{2 mm}

\noindent \underline{Estimation of $S^{(\text{triangle})}_{m,k}$:} This is the more challenging of these bounds due to the double infinite sum in $\ell, n$.  To handle such a sum, we issue the following 
\begin{claim} \label{cl:oop} Fix any $C > 1$. There exists $L_\ast$ sufficiently large relative to $A, b, C$ and universal constants such that the function $\mathfrak{K}$ satisfies the following inequality
\begin{align}
\sum_{n = L_\ast}^\infty \mathfrak{K}(n) \le C  \mathfrak{K}(L_\ast).
\end{align}
\end{claim}
\begin{proof}[Proof of Claim \ref{cl:oop}] This follows almost immediately by comparing $\mathfrak{K}(n)$ to a geometric function. More precisely, define 
\begin{align*}
R > 1, \qquad C_0 = C_0(L_\ast) = R^{L_\ast} \frac{A^{L^\ast}}{[(m+L_\ast)!]^b}, \qquad \tau(n) = C_0 \frac{1}{R^n}.
\end{align*}
Such a choice of $C_0$ ensures that 
\begin{align}
\tau(L_\ast) =C_0 \frac{1}{R^{L_\star}}= \mathfrak{K}(L_\ast). 
\end{align}
We first note that, due to the standard geometric sum properties, 
\begin{align*}
\sum_{n = L_\ast}^\infty \tau(n) = C_0 \sum_{n = L_\ast}^\infty \frac{1}{R^n} = C_0 \frac{1}{R^{L_\ast}} \frac{1}{1-\frac{1}{R}} =  \frac{1}{1-\frac{1}{R}} \tau(L_\ast).
\end{align*}
We secondly notice that for $L_\ast$ sufficiently large, we have the order property 
\begin{align} \label{order:r:1}
\mathfrak{K}(n)=\frac{A^n}{[(m+n)!]^b} \le \tau(n)=\frac{R^{L_\star}A^{L_\star}}{[(m+L_\star)!]^b}, \qquad n \ge L_\ast.
\end{align}
To establish, \eqref{order:r:1}, we notice that $K(L_\ast) = \tau(L_\ast)$. Therefore, it suffices to check 
\begin{align*}
\frac{\mathfrak{K}(n+1)}{\mathfrak{K}(n)} \le \frac{\tau(n+1)}{\tau(n)}, \qquad n \ge L_\ast, 
\end{align*}
which is equivalent to 
\begin{align*}
\frac{A}{(m+n+1)^b} \le \frac{1}{R}, \qquad n \ge L_\ast, 
\end{align*}
which can clearly be achieved by choosing $L_\ast$ sufficiently large. Therefore, we have 
\begin{align*}
\sum_{n = L_\ast}^\infty \mathfrak{K}(n) \le \sum_{n = L_\ast}^\infty \tau(n) \le \frac{1}{1-\frac{1}{R}} \tau(L_\ast) = \frac{1}{1 - \frac{1}{R}} \mathfrak{K}(L_\ast). 
\end{align*}
\end{proof}

Returning now to the 
$S^{(\text{triangle})}_{m,k}$-bound, we have by choosing $n_\ast$ sufficiently large
\begin{align} \n
S^{(\text{triangle})}_{m,k} = & \sum_{\ell = n_\ast}^\infty  \frac{1}{\mathfrak{K}(\ell)}  \mathfrak{G}_{m,\ell;k} \lf( \sum_{n = \ell + 1}^\infty \mathfrak{K}(n)  \rg) \le \sum_{\ell = n_\ast}^\infty  \frac{\mathfrak{K}(\ell + 1)}{\mathfrak{K}(\ell)}  \mathfrak{G}_{m,\ell;k} \\ \label{shst:1}
= &C \sum_{\ell = n_\ast}^\infty \frac{A^{\ell + 1}}{[(m + \ell + 1)!]^b} \frac{[(m+\ell)!]^b}{A^\ell}   \mathfrak{G}_{m,\ell;k} \le  \frac{CA}{(m+n_\ast + 1)^b} \sum_{\ell = n_\ast}^\infty   \mathfrak{G}_{m,\ell;k}.
\end{align}
By taking $n_\ast$ large and summing over $m, k$, we obtain 
\begin{align*}
\sum S^{(\text{triangle})}_{m,k} \le \frac{1}{3\mf B} \sum_{m=0}^\infty \sum_{n =0}^\infty \theta_n^2 \mathfrak{G}_{m,n;k}.
\end{align*}
\vspace{2 mm}

\noindent \underline{Estimation of $S^{(\text{rectangle})}_{m,k}$:} This bound works in a nearly identical manner to $\mathcal{S}^{(\text{triangle})}_{m,k}$. 

\vspace{2 mm}

\noindent \underline{Bound of $T_1$:} We have upon again invoking \eqref{CommCnqbd}
\begin{align} \n
T_1 \le  & C_{n_\ast} \sum_{m = 0}^\infty \sum_{n = 0}^{n_\ast}   \theta_n^2 \sum_{\ell = 0}^{n-1} \mathfrak{G}_{m,\ell;k} \\ \n
\le &  C_{n_\ast} \sum_{m = 0}^\infty \sum_{n = 0}^{n_\ast} \sum_{k \in \mathbb{Z}}  \sum_{\ell = 0}^{n-1} \Big( \frac{\theta_n}{\theta_{\ell}} \Big)^2 \theta_\ell^2  \mathfrak{G}_{m,\ell;k}\\ \n
\le &   C_{n_\ast} \sum_{m = 0}^\infty \sum_{n = 0}^{n_\ast}   \sum_{\ell = 0}^{n-1} \delta_{\text{Drop}}^2  \theta_\ell^2  \mathfrak{G}_{m,\ell;k} \\ \label{delta:drop:choice}
= &   \delta_{\text{Drop}}^2  C_{n_\ast} \sum_{m = 0}^\infty   \sum_{\ell = 0}^{\infty}   \theta_\ell^2  \mathfrak{G}_{m,\ell;k},
\end{align}
which upon bringing $\delta_{\text{Drop}}$ small relative to $n_\ast$, we obtain the right-hand side is bounded by one third of the right hand side of \eqref{tht_app}. 

By summing all the conclusions above, we  conclude the proof of \eqref{tht_app}. 
\end{proof}
}


\section{Elliptic Regularity Estimates} \label{sec:Elliptic}
\subsection{Setup for Elliptic Bounds}\label{sec:dcmp_ellp}
In this section, we study the elliptic equation: 
\begin{align}
\Delta_k \psi_k = \ww_k, \qquad \psi_k|_{y = \pm 1} = 0.  
\end{align}
When we implements analysis in different regions, we need different set of elliptic estimates. We categorize them as interior elliptic estimates and exterior elliptic estimates. 

\noindent  
{\bf  Elliptic Estimates in the Interior Analysis: }
First of all, we recall the elliptic equation for $\Psi^{(E)}$
\begin{align}
\Delta \Psi^{(E)} = &\ww^E + \left( (1+H^I)^2 - (1 +H)^2 \right) \pav^2 \Psi^{(I)}\n \\
&  + \left((1+H)\pav H - (1+H^I)\pav H^I\right)\pav \Psi^{(I)},\label{dmp_int2}\\
\Psi^{(E)}(x,\pm 1) =& 0.\n
\end{align}
The first main goal in this section is to prove the following proposition.
\begin{proposition}[Estimate of \eqref{dmp_int2}]
\label{pro:ellptc_int}
Consider the solutions to the equation \eqref{dmp_int2}. Under the assumptions \eqref{asmp}, the following estimate holds:
\begin{align}
\sum_{m+n=0}^\infty& \lf(\frac{(2\widetilde{\lambda})^{m+n}}{(m+n)!}\rg)^{2/r}\|\chi_{m+n}|k|^mq^n\Gamma_k^n \Psi_k^{(E)}\|_{H_y^1}^2
\n \\
&\lesssim\lf(\mathcal{E}^{(\gamma)}[\ww_k]+\mathcal{E}_{\mathrm{Int}}^{\mathrm{low}}[\ww_k]\rg)\lf(1+\mathcal{E}_{H}^{(\al)}+\mathcal{E}_{H}^{(\gamma)}\rg)^2e^{-\nu^{-1/8}}.\label{elp_int}
\end{align}
This estimate implies \eqref{fell:e} in Theorem \ref{pro:ell:intro}.
\end{proposition}

\noindent
{\bf Elliptic Estimates in the Exterior Analysis: }

For the elliptic estimates applied in the exterior region, we apply another decomposition of the solution \eqref{d:phi:I:in}, \eqref{d:phi:E:in}. 
\ifx
Before providing the cutoff function, we recall some properties of $\chi_1$ cut-off \eqref{chi} and the coordinate system. 
Next we use the bootstrap hypothesis of the coordinate system to derive that the there exists a constant $C_\dagger$ such that variable $v$ and $y$ are close in the following sense,
\begin{align*}
\|v-y\|_{L_t^\infty L_y^\infty}\leq C_\dagger\ep. 
\end{align*}
\siming{(We also need a statement to cite here.)} 
Now we define the following cut-off function $\wt \chi_1(v)$ \siming{Double check the range!}
\begin{align}\label{wt_chi}
 \wt \chi_1^\mathfrak{c}(v)=1-\wt \chi_1(v),\quad \wt \chi_1(v)=\lf\{\begin{array}{cc}1,&\quad |v|\geq \frac{3}{8}-\max\{2C_\dagger \ep,\frac{1}{32}\}, \\ \text{smooth monotone}, &\quad |v|\in [\frac{3}{8}-\max\{3C_\dagger \ep,\frac{1}{16}\},\frac{3}{8}-\max\{2C_\dagger \ep,\frac{1}{32}\}], \\ 
0,&\quad |v|\leq \frac{3}{8}-\max\{3C_\dagger \ep,\frac{1}{16}\}. \end{array}\right.
\end{align} The cutoff $\wt \chi_1$ is a fatten version of $\chi_1$ written in the $v$-coordinate. We have the following property 
\begin{align}
\wt\chi_1(v)\equiv 1, \quad y(v)\in \text{support}\{\chi_n\},\quad 1\leq n\in \mathbb{N}. 
\end{align} We choose to write the cutoff in the $v$-variable so that the computation of $\pav^n\wt \chi_1$ is easier. 

Now we decompose the stream function as $\phi_k = \phi_k^{(E)} + \phi_k^{(I)}$, where these two quantities satisfy the following elliptic equations
\begin{subequations}
\begin{align} \label{ell:I}
(\pav^2-|k|^2)\ \phi^{(I)}_k(t,v) &=\widetilde{\chi}_1^\mathfrak{c}(v)\ \ww_k(t,v)+{\wt \chi_1^\mathfrak{c}(v)(\pav^2-\pa_y^2)\phi_k^{(I)}(t,v)}, \\
\n\qquad &\phi^{(I)}_k(t,v)|_{v=v(t,\pm 1)} \ =\ 0, \\  \label{ell:E}
(\pa_y^2-|k|^2)\ \phi^{(E)}_k(t,y) &=\ \widetilde{\chi}_1\lf(v(y)\rg) \ww_k(t,y)+{\wt\chi_1(v(y))}(\pav^2-\pa_{y}^2)\phi^{(I)}_k(t,v(y)),\\
\n \qquad &\phi^{(E)}_k(t,y)|_{y = \pm 1} = 0. 
\end{align}
\end{subequations}
Here $\pav:=v_y^{-1}\pa_y$ and we will use a simplified notation $\Delta_{k,v}:=\pav^2-|k|^2$. 
\fi
We observe that 
\begin{align*}
\de_k&\lf(\phi_k^{(I)}(v(y))+\phi_k^{(E)}(y)\rg)=(|k|^2-\pav^2)\lf(\phi_k^{(I)}(v(y))\rg)+\widetilde{\chi}_1(v(y))\ww_k(y)\\
&=\lf(\widetilde{\chi}_1^\mathfrak{c}(v(y))+\widetilde{\chi}_1(v(y))\rg)\ww_k(y) =\ww_k(y),\quad \lf(\phi_k^{(I)}(v(y))+\phi_k^{(E)}(y)\rg)\bigg|_{y=\pm 1}=0.
\end{align*}
Here in the equation, we omit the argument $t$ in the $v(t,y),\ \wt\chi_1(t,v(t,y)),\ \wt\chi_1^{\mf c}(t,v(t,y))$ expressions for the sake of notation simplicity. Furthermore, we define the following quantities
\begin{align}\label{dfn_omRIE}
\wt \ww_k^{(I)}(t,v(t,y)):=&(1-\wt \chi_1(t,v(t,y)))\ww_k(t,y),\quad \wt \ww_k^{(E)}(t,y)=\wt\chi_1(t,v(t,y))\ww_k(t,y),\\
\n\quad R_k^{(I)}:=&\text{R.H.S. of }\eqref{d:phi:I:in},\quad R_k^{(E)}:=\text{R.H.S. of }\eqref{d:phi:E:in}.
\end{align}
Furthermore, we define an intermediate cut-off function $\chi_\ast(t,\cdot)\in C^\infty([v(t,-1),v(t,1)])$ such that for a positive constant $c>0$ and any $t$, the following conditions hold
\begin{align}\label{chi_ast}
\chi_\ast(t,v)=&\lf\{\begin{array}{cc} 1,&\quad v\in \text{support }\chi_1(y(t,\cdot));\\
0,&\quad v\in \text{support }\wt\chi_1^\mathfrak{c}(t,\cdot);\\
\text{smooth},&\quad \text{others}.
\end{array}\rg. \\
&\hspace{-0.5cm}\text{distance}(\text{support }\chi_\ast(t,\cdot),\text{support }\wt \chi_1^\mathfrak{c}(t,\cdot))\geq c>0.\n
\end{align} Here $y(t,v)$ is the inverse function of $v(t,y)$ at fixed time $t$. Here the cutoff function $\chi_\ast$ has a mild dependence on time because the boundary of the domain is $v(t,\pm 1)$. Finally, we have the relation $\chi_{n}(y)\chi_{\ast}(t,v(t,y))=\chi_n(y)$ for $n\geq 1$.  

Now we present a more precise version of Theorem \ref{pro:ell:intro}.
\begin{proposition}\label{pro:IE_phi_ext} Assume that the conditions in  \eqref{asmp} hold. Further assume that $k\neq 0$. 

\noindent
{\bf a) Interior estimates: }Recall the definition of $\widehat{\bf a}_{m,n}$ \eqref{hat_bf_a_intro} and consider the solution to \eqref{d:phi:I:in}. 
The following estimate holds 
\begin{align}\label{phi_I_est}
\sum_{m+n=0}^\infty \widehat{\bf a}_{m,n}^2\|{\chi_\ast}|k|^m(\pa_v+ikt)^n{|k|^{\mathfrak{l}_1}}\pa_v^{\mathfrak{l}_2}\phn\|_{L_v^\infty}^2\lesssim \frac{1}{\lan t\ran^2}\|\widetilde{\ww}_k^{(I)}\|_{L_v^2}^2\lesssim \frac{1}{\lan t\ran^2}\mathcal{E}_{\mathrm{Int}}^{\mathrm{low}}[w_k],\quad \mathfrak{l}_1,\mathfrak l_2\in\{0,1,2,3,4\}.
\end{align}
\siming{\footnote{Double Check the higher $\mathfrak{l}_1,\mathfrak{l}_2$ case. I just added the $3,4$. Check once.}}
Here the implicit constant is universal and $\mathcal{E}_{\mathrm{Int}}^{\mathrm{low}}$ is defined in \eqref{E_Int}.  
The smooth cutoff function $\chi_\ast$ is defined in \eqref{chi_ast}. Moreover, for any $m\in \mathbb{N}$, the following Sobolev norm estimate holds 
\begin{align}
\|\chi_\ast |k|^m\pa_v^\mathfrak{l}\phi_k^{(I)}\|_{L_v^\infty}\leq C(m) \|\wt\ww_k^{(I)}\|_{L_v^2},\quad \mathfrak{l}\in\{0,1,2,3,4\}.\label{phi_I_est_2}
\end{align}

Next, the following inviscid damping estimate holds for $\mathfrak{l}\in\{0,1,2\}$,
\begin{align}\n
\sum_{m+n=0}^{\infty}&\widehat{\bf a}_{m,n}^2\|\chi_\ast |k|^m(\pa_v+ikt)^n\pa_v^\mathfrak{l}\phi_k^{(I)}\|_{L_v^\infty}^2\lesssim \frac{\varphi^2}{ (|k|t)^{2\mathfrak{n}}}\sum_{\ell=0}^{\mathfrak{n}}\| (\pa_v+ikt)^{\ell} R_k^{(I)}\|_{L_v^2}^2\\
\lesssim &\frac{\varphi^2}{ (|k|t)^{2\mathfrak{n}}}\sum_{\ell=0}^{\mathfrak{n}}\| (\pa_v+ikt)^{\ell}\wt \ww_k^{(I)}\|_{L_v^2}^2\lesssim\frac{\varphi^2}{ (|k|t)^{2\mathfrak{n}}}\mathcal{E}_{\mathrm{Int}}^{\mathrm{low}}[w_k].\label{phi_I_est_3}
\end{align}  
Here $R_k^{(I)}$ is defined in \eqref{dfn_omRIE}. Furthermore, for finite Sobolev norm, the following inviscid damping estimate holds
\begin{align}\label{phI_est_1_1}
\sum_{m+n=0}^{1000}\widehat B_{m,n}\||k|^m(\pa_v+ikt)^n\phi_k^{(I)}\|_{L_v^2}\lesssim\frac{1}{\lan t\ran^2}\sum_{m+n=0}^{1002}\widehat B_{m,n}\||k|^m(\pa_v+ikt)^n\ww_k^{(I)}\|_{L_v^2}.
\end{align}
\noindent
{\bf b) Exterior estimates:} Under the assumption \eqref{asmp}, we have the following estimate
\begin{align}\n
\sum_{m+n=0}^\infty \sum_{a+b+c=0}^2{\bf a}_{m,n}^2\|J^{(a,b,c)}_{m,n}\phe_k\|_{L^2}^2\lesssim& \sum_{m+n=0}^\infty {\bf a}_{m,n}^2\| J_{m,n}^{(0)}\wwe_k\|_{L^2}^2+e^{-\nu^{-1/9}}(\mathcal{E}_{H}^{(\gamma)}+\mathcal{E}_{H}^{(\al)})\|\wt\ww_k^{(I)}\|_{L^2}^2\\
\lesssim& e^{-\nu^{-1/9}}\lf(\mathcal{E}_k^{(\gamma)}+(\mathcal{E}_{H}^{(\gamma)}+\mathcal{E}_{H}^{(\al)})\mathcal{E}_{\mathrm{Int};k}^{\rm low}\rg).\label{phi_E_est_1}
\end{align}
Here, the left hand side contains the usual $H^2$-based Gevrey bound of $\phe_k$.  
Moreover, if $a+b+c=3$, 
\begin{align}\n
\sum_{m+n=0}^\infty&
 \sum_{a+b+c=3}{\bf a}_{m,n}^2\|J^{(a,b,c)}_{m,n}\phe_k\|_2^2\\
\n \lesssim& \sum_{m+n=0}^\infty\sum_{a+b+c=0}^1 {\bf a}_{m,n}^2\|J^{(a,b,c)}_{m,n}\wwe_k\|_2^2\\
\n &+e^{-\nu^{-1/8}}\lf(\sum_{m+n=0}^\infty{\bf a}_{m,n}^2\|J_{m,n}^{(0)}\wwe_{k}\|_2^2+\|\wt \ww_k^{(I)}\|_{H_v^1}^2\rg)\sum
_{\iota\in\{\al,\gamma\}} \lf(\mathcal{E}_{H}^{(\iota)}+\mathcal{D}_H^{(\iota)}\rg)\\
\lesssim&e^{-\nu^{-1/9}}\lf(\mathcal{D}_k^{(\gamma)}+(\mathcal{E}_k^{(\gamma)}+\mathcal{E}_{\mathrm{Int};k}^{\rm low})\sum
_{\iota\in\{\al,\gamma\}} \lf(\mathcal{E}_{H}^{(\iota)}+\mathcal{D}_H^{(\iota)}\rg) \rg).\label{phi_E_est_2}
\end{align}
Here the left hand side contains the usual $H^3$-based Gevrey bound of $\phe_k$.

\end{proposition}
\begin{remark}
The above estimates in {\bf b)} hold with ``$\infty$'' replaced by ``$M\in \mathbb{N}$''.
\end{remark}


\ifx
\begin{remark}
Thanks to the estimates \eqref{crd_implJ1}, \eqref{crd_implJ2}, we have that the contributions from $\wt\chi_1(v_y^2-1)$ are controlled by the local behavior of $H$ near the boundary. Hence the influence from the interior vorticity is small $\lesssim \nu^{10}. $ 
\end{remark}\fi

Our strategy for proving the estimates \eqref{phi_I_est_2}, \eqref{phi_I_est_3} will be to use elliptic regularity estimates and a delicate inductive scheme in Gevrey spaces to get the $\phi_k^{(E)}$ bounds, and carry out detailed analysis of the Green's function to obtain the $\phi_k^{(I)}$ bounds. 

Since the exterior estimate is the most delicate one, we detail our strategy:

\noindent
{\bf Step \# 1: } First of all, by the assumption \eqref{asmp} and the technical Lemma \ref{Lem:H&vy2-1}, we have that
 \begin{align}
\sum_{n=0}^\infty\sum_{a+b=0}^1 B_{0,n}^2\varphi^{2n}\|J^{(a,b,0)}_{0,n}{(\wt\chi_1(v_y^2-1))}\|_{L^2}^2\leq& C\exp\{-\nu^{-1/8}\}(\mathcal{E}_{H}^{(\gamma)}+\mathcal{E}_{H}^{(\al)}),\label{sml_vy2-1}\\
\sum_{n=0}^\infty\sum_{a+b=2} B_{0,n}^2\varphi^{2n}\|J^{(a,b,0)}_{0,n}{(\wt\chi_1(v_y^2-1))}\|_{L^2}^2\leq& C\exp\{-\nu^{-1/8}\}(\mathcal{D}_{H}^{(\gamma)}+\mathcal{D}_{H}^{(\al)}),\label{sml_vy2-1_2}
\end{align}where $B_{0,n}$ is defined in \eqref{Bweight}.

\noindent
{\bf Step \# 2: }
Through detailed analysis of the elliptic equation, we derive the  following pointwise-in-time estimate given that the right hand side of  \eqref{sml_vy2-1} is small enough, 
\begin{align}
\sum_{m+n=0}^\infty \sum_{a+b+c=0}^2{\bf a}_{m,n}^2\|J^{(a,b,c)}_{m,n}\phe_k\|_{L^2}^2\lesssim \sum_{m+n=0}^\infty {\bf a}_{m,n}^2\| J_{m,n}^{(0)}\wwe_k\|_{L^2}^2+\exp\{-\nu^{-1/9}\}\|\wt\ww_k^{(I)}\|_{L^2}^2.   \label{J_ph_est1}
\end{align}
Then we prove the following estimate
\begin{align}\n
\sum_{m+n=0}^\infty&
 \sum_{a+b+c=3}{\bf a}_{m,n}^2\|J^{(a,b,c)}_{m,n}\phe_k\|_2^2\\ \n &\lesssim \sum_{m+n=0}^\infty\sum_{a+b+c=0}^1 {\bf a}_{m,n}^2\|J^{(a,b,c)}_{m,n}\wwe_k\|_2^2\\ 
& \quad+\lf(\sum_{n=0}^\infty\sum_{a+b=0}^{2}  {B_{0,n}^2\varphi^{2n}}\|J^{(a,b,0)}_{0,n}(\wt\chi_1(v_y^2-1))\|_2^2\rg)\lf(\sum_{m+n=0}^\infty {\bf a}_{m,n}^2\| J_{m,n}^{(0)}\wwe_{k}\|_2^2+\|\wt\ww^{(I)}_{k}\|_{H_k^1}^2\rg).  \label{J_phe_est}
\end{align}
\siming{\footnote{(Double check the last line!?)}} 
We recall from \eqref{J_vec} that $J_{m,n}^{(0)}\wwe_{k}=\chi_{m+n}|k|^m q^n\Gamma_k^n\wwe_k$. 
The proof of these two claims \eqref{J_ph_est1}, \eqref{J_phe_est} is the main technical part of this section. The proof of the claims are complicated thanks to the non-traditional definition of the operator $J_{m,n}^{(a,b,c)}f_k$, which only assigns $\displaystyle \frac{m+n}{q}$ weights when $a+b+c\leq n$. Hence we distinguish between the small-$n$ case 
and the large-$n$ case separately. In the first case, we mainly focus on deriving the Sobolev norm controls and keep track of various boundary contributions arising when one implements integration by parts. This is the content of Subsection \ref{sub2sct:lw_rg}. In the second case, boundary contributions are zero and we mainly focus on the control over the combinatorial coefficients in the definition of the Gevrey norms. This is the content of Subsection \ref{sub2sct:h_rg}. 
 
\noindent
{\bf Step \# 3: } Finally, we use the technical Lemma \ref{lem:JtoD} to translate the right hand side of \eqref{J_ph_est1} and \eqref{J_phe_est} to  quantities $\mathcal{E}^{(\gamma)}, \mathcal{D}^{(\gamma)},\mathcal{E}^{(\al)}_H,\mathcal{E}_H^{(\gamma)},\mathcal{E}_{\mathrm{Int}}^{\rm low}$. 

\noindent
{\bf Notation:} In this section, the $T_{i_1 i_2 i_3}$ will be used to denote various terms in the estimates and will be local variables for the proof of each lemma. If the proof has several steps, the $i_1$ will represent the step number. Moreover, if we decompose the term $T_{i_1i_2}$ into several sub-terms, they will be denoted as $T_{i_1i_2i_3}$. 

\subsection{Bounds on $\mathcal{E}_{ell}^{(I, \cdot)}$}
We are interested here in the elliptic problem \eqref{d:phi:I:in}, which we copy as follows: 
\begin{align} \label{ell:I:2}
(\pa_v^2-|k|^2)\phi^{(I)}_k(v) &=  \wt \ww_k^{(I)}(v)+\wt \chi_1^\mathfrak{c}(v)(\pa_v^2-([v_y]\pa_v)^2)\phi_k^{(I)}(v)=:R^{(I)}_k(v), \qquad \phi^{(I)}_k|_{v(t,\pm 1)} = 0. 
\end{align}
Here $[v_y]$ is the function $v_y(t,y)$ expressed in the $v$-coordinate and we drop the $t$-dependence. 
We have explicit solution formula:
\begin{align}
\label{greensRep} \phi^{(I)}_k(t, v) =& \int_{v(t,-1)}^{v(t,1)} \mathfrak  G_k(v, v')  R_k^{(I)}(t, v') dv',\\
 \mathfrak G_k(v,v'):=&-\frac{1}{k\sinh(k|v(t,1)-v(t,-1)|)}\left\{\begin{array}{cc}\sinh(k(v(t,1)-v'))\sinh(k(-v(t,-1)+v)),&\quad v\leq v';\\
\sinh(k(v(t,1)-v))\sinh(k(-v(t,-1)+v')),&\quad v\geq v'.\end{array}\right.\n \\
 =&\myr{-\frac{ \cosh(k(|v-v'|-|v(t,1)-v(t,-1)|)) }{2k\sinh(k|v(t,1)-v(t,-1)|)}+\frac{ \cosh(k(v+v'-v(t,1)-v(t,-1))) }{2k\sinh(k|v(t,1)-v(t,-1)|)}} \n \\
 =: & H_k(t,v-v')+S_k(t, v+v').\n
\end{align}

With the expression \eqref{greensRep}, we can estimate the higher derivatives of the solution $\phi_k^{(I)}$. Thanks to the definitions of the cutoff $\chi_1$ \eqref{chi} and $\wt\chi_1^{\mf c}$ \eqref{wt_chi_intro}, we have the separation of support \begin{align}\text{distance}_y\lf(\text{support}\{\wt \chi_1^{\mathfrak{c}}(v(t,\cdot))\},\text{support}\{ \chi_1(\cdot)\}\right)\geq \frac{1}{320}>0.\label{supp_sep}
\end{align}Since the support of the forcing $R_k^{(I)}$ \eqref{ell:I:2} is contained in $\text{support}\{\wt \chi_1^{\mathfrak{c}}(v(t,\cdot))\}$, we expect that $\phi^{(I)}_k$ is very smooth in terms of Gevrey index and radius of analyticity. 
Thanks to the Lebesgue Dominated Convergence theorem, it can be justified that the $\pa_v$ derivatives of $\phi^{(I)}$ has the following expression given that the distance between $v$ and $\text{support}\ R_k^{(I)}$ is strictly positive,
\begin{align}
\pa_v^n\phi^{(I)}_k(t, v) = \int_{v(t,-1)}^{v(t,1)} \pa_v^n  \mathfrak G_k(v, v') R_k^{(I)}(t, v') dv',\qquad \text{distance}\lf(\text{support}\{\wt \chi_1^{\mathfrak{c}}(\cdot)\},v\right)>0.
\end{align}\siming{\footnote{ %
(?? Derivative w.r.t the $v(t,1), \, v(t,-1)$??) 
It seems that the solution will be extremely smooth. Now we plug this to the $\phi^{(E)}_k$ part. The new term will contribute $\|J^{(1)}(v_{y}^2-1)\|_{Gevrey}$, but it is fine.? }}


 Moreover, we would like to derive the inviscid damping estimate for the interior. To this end, we will derive finite $\wt\Gamma_k$-vector field estimates of $\phi_k^{(I)}$. 
\begin{lemma}\label{lem:fnt_G_phiI}Assume that   \eqref{v_y_sob} holds and define
\begin{align}
\wt\Gamma_k=\pa_v+ikt. \label{wtGa}
\end{align}  If the threshold $\ep$ is smaller than a constant depending on a fixed $M\in\mathbb{N}$, the following estimates hold
\begin{align}\label{phi_I_Sbv}
\sum_{n=0}^M\sum_{b+c=2}\|\pa_v^b|k|^c \wt\Gamma_k^n \phi_k^{(I)}\|_{L_{v}^2}^2+\sum_{n=0}^M\|\wt\Gamma_k^n R_k^{(I)}\|_{L_v^2}^2\leq & C\lf(M,\|[v_y]^2-1\|_{H_v^{M+2}(\mathrm{supp}\wt\chi_1^{\mf c})}\rg)\sum_{n=0}^M\|\wt\Gamma_k^n\wt\ww_k^{(I)}\|_{L^2_v}^2;\\
\sum_{n=0}^{M+2}\|\pa_v^n\phi_k^{(I)}\|_{L_{v}^{2}}^2\leq & C\lf(M,\|[v_y]^2-1\|_{H_v^{M+2}(\mathrm{supp}\wt\chi_1^{\mf c})}\rg) \sum_{n=0}^M\|\pa_v^n\wt\ww_k^{(I)}\|_{L_{v}^{2}}^2.\label{phi_I_Sbv_2}
\end{align}
\siming{\footnote{We need to explicitly derive this! Check!}} 
Here $R_k^{(I)}$ is defined in \eqref{dfn_omRIE} and $\|f_k\|_{H_{v}^2}$ is the standard Sobolev norm in the $v$-coordinate, i.e., $\displaystyle\|f_k\|_{H_{v}^M}^2:=\sum_{n=0}^M\|\pa_v^nf_k\|_{L_v^2}^2$.
\end{lemma}
\begin{remark}
We remark that in the sequel, we only use the finite regularity estimate (e.g., $M=1000$) of the functions $R_k^{(I)}$ \eqref{dfn_omRIE} and $\phi_k^{(I)}$ in this lemma. Hence the explicit expression of the coefficients in \eqref{phi_I_Sbv} and \eqref{phi_I_Sbv_2} are not needed. 
\end{remark}
\begin{proof}

We divide the proof in two steps.

\noindent
{\bf Step \# 1: Proof of \eqref{phi_I_Sbv}.}
First of all, we study the structure of the right hand side $R_k^{(I)}$ in \eqref{ell:I:2}. For the sake of notation simplicity, we denote 
\begin{align}\label{Z}
\mathbbm{z}(t,v):=Z(t,y(v)),\quad Z(t,y)=v_y(t,y)^2-1.
\end{align} By applying the $\wt\Gamma_k^n,\, n\in \mathbb{N}$, on both side of the equation \eqref{d:phi:E:in}, we obtain that
\begin{align}\n
(\pa_v^2 -|k|^2)\wt\Gamma_k^n\phi_k^{(I)}(t,v) 
=&\wt\Gamma_k^n\wt\ww_k^{(I)}(t,v)+\wt\Gamma_k^n \lf(\wt{\chi}_1^\mathfrak{c}\lf(-([v_y]^2-1)\pa_v^2\phi_k^{(I)}+\frac{1}{2}\pa_v([v_y]^2-1)\pa_v\phi_k^{(I)}\rg)\rg)\\
=&\wt\Gamma_k^n\wt\ww_k^{(I)}(t,v)-\wt\Gamma_k^n \lf(\wt{\chi}_1^\mathfrak{c}\lf(\mathbbm{z}\ \pa_v^2\phi_k^{(I)}- \frac{1}{2}\pa_v \mathbbm{z} \ \pa_v\phi_k^{(I)}\rg)\rg).\label{eq:GnphiI}
\end{align}

Next, we apply the induction argument to prove the estimate \eqref{phi_I_Sbv}. For $\eqref{eq:GnphiI}_{n=0}$, by a simple variant of the elliptic regularity lemma \ref{lem:max_reg} and the Sobolev estimate \eqref{v_y_sob}, we have obtained that 
\begin{align*}
\sum_{b+c=2}\|\pa_v^b|k|^c \phi_k^{(I)}\|_{L^2_v}\lesssim& \|\wt \ww_k^{(I)}\|_{L_v^2}+\|\wt\chi_1^\mf{c} \mathbbm{z}\pa_v^2\phi^{(I)}_k\|_{L_v^2}+\|\wt\chi_1^\mf{c} \pa_v\mathbbm{z}\pa_v\phi^{(I)}_k\|_{L_v^2} 
\lesssim  \|\wt \ww_k^{(I)}\|_{L_v^2}+\ep \sum_{b+c=2}\|\pa_v^b|k|^c \phi^{(I)}_k\|_{L^2_v}.
\end{align*} 
Hence if the $\ep$ is small enough compared to uniform constants, we have that 
\begin{align*}
\sum_{b+c=2}\|\pa_v^b|k|^c \phi_k^{(I)}\|_{L^2_v}\leq C\|\wt \ww_k^{(I)}\|_{L_v^2}.
\end{align*} This concludes the starting level of the induction.  

Assume that the following estimate holds for $1\leq n-1\leq N$:
\begin{align}
\||k|^2\wt\Gamma_k^{\ell}\phi_k^{(I)}\|_{L_v^2}+\||k|\pa_v\wt\Gamma_k^{\ell}\phi_k^{(I)}\|_{L_v^2}+\|\pa_v^2\wt\Gamma_k^{\ell}\phi_k^{(I)}\|_{L_v^2}\lesssim \sum_{\ell'=0}^{\ell}\|\wt\Gamma_k^{\ell'}\wt\ww_k^{(I)}\|_{L^2_v},\quad \forall\ell\leq n-1.\label{indct_hyp}
\end{align}
We would like to show that the same estimate holds with $``n-1"$ replaced by $``n"$. We test the equation \eqref{eq:GnphiI} with the $\wt\Gamma_k^n\phi_k^{(I)}$ and apply integration by parts on $v\in[v(t,-1), v(t,1)]$:
\begin{align}\n
 \| \pa_v\wt\Gamma_k^n\phi_k^{(I)}\|_{L^2_v}^2+ |k|^2\|\wt\Gamma_k^n\phi_k^{(I)}\|_{L^2_v}^2 =&\mathrm{Re}\pa_v\wt\Gamma_k^n\phi_k^{(I)}\overline{\wt\Gamma_k^n\phi_k^{(I)}}\bigg|_{v=v(t,-1)}^{v=v({t,1})}- \mathrm{Re}\int_{v(t,-1)}^{v(t,1)} \wt\Gamma_k^n  \wt\ww_k^{(I)}\overline{\wt\Gamma_k^n\phi_k^{(I)}}dv\\
 &+\mathrm{Re}\int_{v(t,-1)}^{v{(t,1)}}\wt\Gamma_k^n\lf(\wt \chi_1^\mathfrak{c} \mathbbm{z}\  \pa_v^2\phi_k^{(I)}-\frac{1}{2}\wt \chi_1^\mathfrak{c}\pa_v \mathbbm{z}\ \pa_v\phi_k^{(I)}\rg)\overline{\wt\Gamma_k^n\phi_k^{(I)}}dv\n \\
 =:&T_1+T_2+T_3.\label{enrgy_rl}
\end{align}
Now we treat the first term on the right hand side of \eqref{enrgy_rl}. We observe that since $\wt\chi_1^\mathfrak{c}(v)$ is identically zero in a neighborhood of the boundary, we have that the \eqref{d:phi:I:in} yields that
\begin{align}\label{fnt_G_phiI_bc}
\pa_v^{2m}\phi_k^{(I)}(t,v(t,\pm1))=0,\quad \pa_v^{2m+1}\phi_k^{(I)}(t,v(t,\pm 1))=|k|^{2m}\pa_v\phi_k^{(I)}(t,v(t,\pm 1)), \quad m\in\mathbb{N}.
\end{align}
Now we apply these relations to obtain
\begin{align*}
\mathrm{Re}\pa_v\wt\Gamma_k^n\phi_k^{(I)}\overline{\wt\Gamma_k^n\phi_k^{(I)}}&\bigg|_{v=v(t,-1)}^{v=v({t,1})}=\mathrm{Re}\lf(\sum_{\ell=0}^n\binom{n}{\ell}\pa_v^{\ell+1}(ikt)^{n-\ell}\phi_k^{(I)}\rg)\overline{\lf(\sum_{\ell=0}^n\binom{n}{\ell}\pa_v^{\ell }(ikt)^{n-\ell}\phi_k^{(I)}\rg)}\bigg|_{v=v(t,\pm1)}\\
=&
\mathrm{Re}\lf(\sum_{\ell=0}^n\binom{n}{\ell}\mathbbm{1}_{\ell\, \text{even}}\pa_v^{\ell+1}(ikt)^{n-\ell}\phi_k^{(I)}\rg)\overline{\lf(\sum_{\ell=0}^n\binom{n}{\ell}\mathbbm{1}_{\ell\, \text{odd}}\pa_v^{\ell }(ikt)^{n-\ell}\phi_k^{(I)}\rg)}\bigg|_{v=v(t,\pm1)}.
\end{align*}
For every pairing in this expression, we have that 
\begin{align*}
\mathbbm{1}_{\ell_1\, \text{even}}&\mathbbm{1}_{\ell_2\, \text{odd}}\mathrm{Re}(ikt)^{n-\ell_1}\pa_v^{\ell_1+1}\phi_k^{(I)}\overline{(ikt)^{n-\ell_2}\pa_v^{\ell_2 }\phi_k^{(I)}}\bigg|_{v=v(t,\pm1)}\\
=&
\mathbbm{1}_{\ell_1\, \text{even}}\mathbbm{1}_{\ell_2\, \text{odd}}|k|^{\ell_1+\ell_2-1}\mathrm{Re}\lf((ikt)^{2n-\ell_1-\ell_2} (-1)^{n-\ell_2} \pa_v \phi_k^{(I)}\overline{\pa_v\phi_k^{(I)}}\rg)\bigg|_{v=v(t,\pm1)}=0.
\end{align*}
Therefore, the boundary term add up to zero in \eqref{enrgy_rl}. Now the second term in \eqref{enrgy_rl} can be estimated as follows:
\begin{align*}
|T_2|\leq &C\|\wt\Gamma_k^n\wt \ww_k^{(I)}\|_{L_v^2}\|\wt\Gamma_k^n\phi_k^{(I)}\|_{L_v^2}.
\end{align*}
Since $\wt \chi_1^\mathfrak{c}$ vanishes near the boundary, we can apply integration by parts to rewrite the last term and estimate it as follows
\begin{align*}
|T_3&|= \lf|\sum_{\ell=0}^n\binom{n}{\ell}\int \lf(\pa_v^{n-\ell}(\wt\chi_1^\mathfrak{c} \mathbbm{z} )\pa_v^2\wt\Gamma_k^{\ell}\phi_k^{(I)}-\frac{1}{2}\pa_v^{n-\ell}(\wt\chi_1^\mathfrak{c} \pa_v \mathbbm{z} )\pa_v\wt\Gamma_k^{\ell}\phi_k^{(I)}\rg)\overline{\wt\Gamma_k^n\phi_k^{(I)}} dv\rg|\\
\leq &\lf| \int \lf( \wt\chi_1^\mathfrak{c} \mathbbm{z} \pa_v\wt\Gamma_k^{n}\phi_k^{(I)}\rg)\overline{\pa_v\wt\Gamma_k^n\phi_k^{(I)}} dv+\int \lf(\pa_v(\wt\chi_1^\mathfrak{c} \mathbbm{z} )\pa_v\wt\Gamma_k^{n}\phi_k^{(I)}\rg)\overline{\wt\Gamma_k^n\phi_k^{(I)}} dv\rg|\\
&+C(n)\|\wt \chi_1^\mathfrak{c}   \mathbbm{z}\|_{W^{n,\infty}_v}\|  \wt\Gamma_k^n  \phi_k^{(I)}\|_{L_v^2}\sum_{\ell=0}^{n-1}\|\pa_{v}^2\wt\Gamma_k^\ell\phi_k^{(I)}\|_{L_v^2} 
+C(n)\|\wt \chi_1^\mathfrak{c} \pa_v \mathbbm{z}\|_{W^{n,\infty}_v}\|\wt\Gamma_k^n\phi_k^{(I)}\|_{L_v^2} \sum_{\ell=0}^n \|\pa_v \wt\Gamma_k^\ell\phi_k^{(I)}\|_{L_v^2}.
\end{align*}
Now we invoke the induction hypothesis \eqref{indct_hyp} to obtain that 
\begin{align*}
|T_3|\leq&C(n)\|\mathbbm{z}\|_{H^{2}_v(\text{supp}(\wt \chi_1^{\mf c}))}\lf(\|\pa_v\wt\Gamma_k^n\phi_k^{(I)}\|_{L_v^2}^2+\|k\wt\Gamma_k^n\phi_k^{(I)}\|_{L_v^2}^2\rg)+C(n) \|\mathbbm{z}\|_{H^{n+2}_v(\text{supp}(\wt \chi_1^{\mf c}))} \|\wt\Gamma_k^n\phi_k^{(I)}\|_{L_v^2} \sum_{\ell=0}^{n-1} \| \wt\Gamma_k^\ell\wt \ww_k^{(I)}\|_{L_v^2}.
\end{align*}
Combining the estimates above and \eqref{enrgy_rl}, we apply the assumption \eqref{v_y_sob} to obtain that
\begin{align}
\|\pa_v\wt\Gamma_k^n\phi_k^{(I)}\|_{L_v^2}^2+\|\wt\Gamma_k^n\phi_k^{(I)}\|_{L_v^2}\|k^2\wt\Gamma_k^n\phi_k^{(I)}\|_{L_v^2}\lesssim& \sum_{\ell=0}^{n}\|\wt\Gamma_k^\ell\wt \ww_k^{(I)}\|_{L_v^2}\|\wt\Gamma_k^n\phi_k^{(I)}\|_{L_v^2}.  \label{engy_rel_1}
\end{align}Hence, we have that 
\begin{align*}
\||k|^2\wt\Gamma_k^n\phi_k^{(I)}\|_{L_v^2}\lesssim \sum_{\ell=0}^{n}\|\wt\Gamma_k^\ell\wt \ww_k^{(I)}\|_{L_v^2}. 
\end{align*}
Now we multiply the equation \eqref{engy_rel_1} by $|k|^2$ and invoke the above relation to get
\begin{align*}
\||k|\pa_v\wt\Gamma_k^n\phi_k^{(I)}\|_{L_v^2}\lesssim \sum_{\ell=0}^{n}\|\wt\Gamma_k^\ell\wt \ww_k^{(I)}\|_{L_v^2}.
\end{align*}
Finally, we test the equation \eqref{eq:GnphiI} by $\pa_v^2\wt\Gamma_k^n\phi_k^{(I)}$, and apply the estimates derived before to obtain that 
\begin{align*}
\|\pa_v^2\wt\Gamma_k^n\phi_k^{(I)}\|_{L_v^2}^2\leq& C\lf(n,\|\mathbbm z\|_{H^{n+2}_v(\text{supp}(\wt\chi_1^\mf{c}))}\rg)\sum_{\ell\leq n}\|\wt\Gamma_k^\ell{\ww^{(I)}_k}\|_{L_v^2}\|\pa_v^2\wt\Gamma_k^n\phi_k^{(I)}\|_{L_v^2}+\|\pa_v^2\wt\Gamma_k^n\phi_k^{(I)}\|_{L_v^2}\||k|^2 \wt\Gamma_k^n\phi_k^{(I)}\|_{L_v^2}\\
&+C\|\mathbbm{z}\|_{H^{2}_v(\text{supp}(\wt\chi_1^\mf{c}))}\|\pa_v^2\wt\Gamma_k^n\phi_k^{(I)}\|_{L_v^2}^2. 
\end{align*} 
Hence, by the estimate \eqref{v_y_sob} and the smallness of $\ep$, we have that
\begin{align*}
\|\pa_v^2\wt\Gamma_k^n\phi_k^{(I)}\|_{L_v^2}\leq  C\lf(n,\|\mathbbm z\|_{H^{n+2}_v(\text{supp}(\wt\chi_1^\mf{c}))}\rg)\sum_{\ell\leq n}\|\wt\Gamma_k^\ell{\ww^{(I)}_k}\|_{L_v^2}.
\end{align*}
Hence, we observe that 
\begin{align*}
\sum_{b+c=2}\|\pa_v^b|k|^c\wt\Gamma_k^{n}\phi_k^{(I)}\|_{L_v^2}\leq C\lf(n,\|\mathbbm z\|_{H^{n+2}_v(\text{supp}(\wt\chi_1^\mf{c}))}\rg) \sum_{\ell=0}^{n}\|\wt\Gamma_k^{\ell}\wt\ww_k^{(I)}\|_{L^2_v}. 
\end{align*} 
This concludes the proof of \eqref{indct_hyp} for $n$. To estimate the $R_k^{(I)}$ in \eqref{d:phi:I:in}, we can invoke the elliptic estimate of $\phi_k^{(I)}$ and combine it with the assumption $\sum_{n=0}^{M+2}\|\pa_v^n\mathbbm{z}\|_{H^{n}(\text{supp}\chi_1^\mf{c})}\lesssim \ep.$ We omit further details for the sake of brevity.

\noindent
{\bf Step \# 2: Proof of \eqref{phi_I_Sbv_2}.} The proof is identical to the one in step \# 1. One needs to replace $\wt\Gamma_k$ by $\pa_v$. Thanks to the condition \eqref{fnt_G_phiI_bc}, the boundary contributions 
\begin{align*}
\mathrm{Re}  \pa_v^{n+1}\phi_k^{(I)}\overline{\pa_v^n\phi_{k}^{(I)}}\bigg|_{v=v(t,-1)}^{v=v(t,1)}=0.
\end{align*} We omit further details. This completes the proof.
\end{proof}

Now we can prove Proposition \ref{pro:IE_phi_ext} a). 

\begin{proof}[Proof of Proposition \ref{pro:IE_phi_ext} a)] We divide the proof into several steps.

\noindent
{\bf Step \# 1: Proof of \eqref{phi_I_est},  \eqref{phi_I_est_2} \eqref{phi_I_est_3}:}
We apply the property that $t\varphi\lesssim 1$, and the fact $\sum_{\ell=0}^n\binom{n}{\ell}=2^n$ to estimate the term as follows:
\begin{align}\n
&\sum_{m+n=0}^\infty\widehat{\bf a}_{m,n}^2\|\chi_\ast\wt\Gamma_k^n|k|^{m+\mathfrak l_1}\pa_v^{\mathfrak{l}_2}\phn\|_{L^\infty_v}^2\\
\n &=\sum_{m+n=0}^\infty \widehat{B}_{m,n}^2\varphi^{2(1+n)}\lf\|\chi_\ast( \pa_v+ikt)^n\int_{v(-1)}^{v(1)}|k|^{m+\mathfrak{l}_1}\pa_v^{\mathfrak{l}_2} \GG_k(v,v') R^{(I)}_k(v')dv'\rg\|_{L^\infty}^2\\ \n
 &\leq\sum_{m+n=0}^\infty\sum_{n'=0}^n \widehat{B}_{m,n}^2\binom{n}{n'}^{2}\varphi^{2(1+n')}\lf\| \chi_\ast(y(v))\int_{v(-1)}^{v(1)}  |k|^{m+n-n'+\mf l_1}\pa_v^{n'+\mathfrak{l}_2}\GG_k(v,v') R^{(I)}_k(v')dv'\rg\|_{L_v^\infty}^2\\ \n
 &\leq\sum_{m+n=0}^\infty\max_{n'\{0,\cdots,n\}} \widehat{B}_{m,n}^2 2^{2n}\varphi^{2(1+n')}\lf\|\chi_\ast\int_{v(-1)}^{v(1)}  |k|^{m+n-n'+\mf l_1}\pa_v^{n'+\mathfrak{l}_2}\GG_k(v,v') R^{(I)}_k(v')dv'\rg\|_{L_v^\infty}^2\\ \n 
 &\leq\varphi^2\|R^{(I)}_k\|_{L_v^2}^2 \lf(\sum_{m+n=0}^\infty\max_{n'\in\{0,\cdots,n\}} \frac{\widehat\lambda^{2(m+n)s}{2^{2n}}}{((m+n)!)^{2s}} \sup_{v\in[v(t,-1),v(t,1)]}\lf\|\mathbbm{1}_{|v-v'|\geq c>0} |k|^{m+n-n'+\mf l_1}\pa_v^{n'+\mathfrak{l}_2}\GG_k(v,\cdot)\rg\|_{L^2_{v'}}^2\rg)\\ \label{Int_phi_T_1}
&=:\varphi^2\|R^{(I)}_k\|_{L_v^2}^2 T_1.
\end{align}
To estimate the last line, we recall the explicit form of the Green's function \eqref{greensRep}. Without loss of generality, we consider the $v\leq v'$ case \siming{(Check the other case!)}. Hence,
\begin{align*}
T_1 &\mathbbm{1}_{v\leq v'}\leq\sum_{m+n=0}^\infty\max_{n'\in\{0,\cdots,n\}}\frac{(2\widehat{\lambda})^{2(m+n)s}}{(m+n)!^{2s}}\frac{|k|^{2(m+n+\mathfrak{l}_1+\mathfrak{l}_2)}}{k^2(\sinh (k|v(1)-v(-1)|))^2}\\
&\qquad\qquad\times\sup_v\lf\| \mathbbm{1}_{|v-v'|\geq c}\sinh\lf(k(v(t,1)-v')\rg)\sinh^{(n')}\lf(k(-v(t,-1)+v)\rg)\rg\|_{L_{v'}^2}^2\\
\lesssim&\sum_{m+n=0}^\infty\frac{(2\widehat{\lambda})^{2(m+n)s}}{(m+n)!^{2s}}\frac{|k|^{2(m+n+\mf l_1+\mathfrak{l}_2)}}{k^2(\sinh(k|v(1)-v(-1)|))^2 }\sup_v\int_{v'-v\geq c}e^{2k(v(1)-v')}e^{2k(v-v(t,-1))}dv'  \\
\lesssim&\sum_{m+n=0}^\infty\frac{(2\widehat{\lambda})^{2(m+n)s}}{(m+n)!^{2s}}\frac{|k|^{2(m+n+\mf l_1+\mathfrak{l}_2)}}{k^2}\sup_v\int_{v'-v\geq c} e^{-2k(v'-v)}dv' \\
\lesssim&\sum_{m+n=0}^\infty\frac{(2\widehat{\lambda})^{2(m+n)s}}{(m+n)!^{2s}}\frac{|k|^{2(m+n+\mf l_1+\mathfrak{l}_2)}}{e^{2kc}}\lesssim\sum_{m+n=0}^\infty\frac{(2\widehat{\lambda})^{2(m+n)s}}{(m+n)!^{2s}}\frac{( m+n)!^2}{(c/2)^{2m+2n}}.
\end{align*}
Now we use the fact that $\sum_{M=0}^\infty \frac{C^M}{M!^{2s-2}}\lesssim 1$ for any fixed bounded $C<\infty$ to obtain
\begin{align*}T_1\mathbbm{1}_{v\leq v'}\lesssim\sum_{M=0}^\infty\sum_{m+n=M}\frac{(4\widehat{\lambda}/c)^{2Ms}}{M!^{2s-2}}\lesssim\sum_{M=0}^\infty\myr{M}2^{-2Ms}\frac{(8\widehat{\lambda}/c)^{2Ms}}{M!^{2s-2}} \leq C(\widehat{\lambda},c,s).
\end{align*}
\siming{When we write $\sum_{m+n=0}^\infty f(m+n)$, we should always remember that we are dealing with a double sum and one CANNOT change it directly to $\sum_{M=0}^\infty f(M)$! There will be an extra factor of \myr{$M$} in the change. I will need to double check my writing!}  Now we  invoke the estimate \eqref{phi_I_Sbv} and bound the expression \eqref{Int_phi_T_1} as follows:
\begin{align}\label{Int_phi_2}
\sum_{m+n=0}^\infty \widehat{\bf a}_{m,n}^2\|\chi_\ast\wt\Gamma_k^n|k|^{m+\mf l_1}\pa_v^{\mathfrak{l}_2}\phn\|_{L^\infty_y}^2\lesssim&\varphi^2\|\wt\ww^{(I)}_k\|_{L_v^2}^2.
\end{align}
This concludes the proof of \eqref{phi_I_est}. 

The derivation for the estimate \eqref{phi_I_est_2} is similar, we have that 
\begin{align*}
\|\chi_\ast |k|^m \pa_v^{\mathfrak{l}}\phi_{k}^{(I)}\|_{L_v^\infty}\leq& C \lf\|\chi_\ast \int_{v(-1)}^{v(1)}|k|^m \pa_v^\mathfrak{l}\GG_k(v,v') R^{(I)}_k(v')dv'\rg\|_{L_v^\infty} 
\leq C(m)\|R^{(I)}_k\|_{L_v^2}\leq C(m)\|\wt\ww_k^{(I)}\|_{L_v^2}.
\end{align*}

To derive the estimate \eqref{phi_I_est_3},
 we modify the previous computation as follows 
\begin{align*}
\sum_{m+n=0}^\infty&\widehat{\bf a}_{m,n}^2\|\chi_\ast\wt\Gamma_k^n|k|^m\pa_v^\mathfrak{l}\phn\|_{L^\infty_y}^2\\
\leq &\sum_{m+n=0}^\infty\max_{n'\in\{0,\cdots,n\}} \widehat{B}_{m,n}^2 2^{2n}\varphi^{2(1+n')}\lf\|\chi_\ast\int_{v(-1)}^{v(1)}  |k|^{m+n-n'}\pa_v^{n'+\mathfrak{l}}\GG_k(v,v') R^{(I)}_k(v')dv'\rg\|_{L_v^\infty}^2\\
=&\sum_{m+n=0}^\infty\max_{n'\in\{0,\cdots,n\}} \widehat{B}_{m,n}^2 2^{2n}\varphi^{2(1+n')}\lf\|\chi_\ast\int_{v(-1)}^{v(1)}  |k|^{m+n-n'}\pa_v^{n'+\mathfrak{l}}\GG_k(v,v') \frac{\pa_{v'}^\mathfrak{n} e^{-iktv'}}{(-ikt)^\mathfrak{n}} e^{iktv'} R^{(I)}_k(v')dv'\rg\|_{L_v^\infty}^2\\
\lesssim&\frac{\varphi^2}{|kt|^{2\mathfrak{n}}}\sum_{\ell=0}^{\mathfrak{n}}\binom{\mathfrak{n}}{\ell}\sum_{m+n=0}^\infty \max_{n'\in\{0,\cdots,n\}} \widehat{B}_{m,n}^2 2^{2n}\varphi^{2n'}\\
&\qquad\times\lf\|\chi_\ast\int_{v(-1)}^{v(1)}  |k|^{m+n-n'}\pa_v^{n'+\mathfrak{l}}\pa_{v'}^{\mathfrak{n}-\ell}\GG_k(v,v') e^{-iktv'}  e^{iktv'} \wt\Gamma_k^{\ell}R^{(I)}_k(v')dv'\rg\|_{L_v^\infty}^2\\
\lesssim&\frac{\varphi^2}{|kt|^{2\mathfrak{n}}} \sum_{\ell=0}^\mathfrak{n}\|\wt\Gamma_k^{\ell}\wt\ww_k^{(I)}\|_{L^2_v}^2.
\end{align*}
Here in the last line, we have invoked \eqref{phi_I_Sbv} and applied a similar argument as in derivation of \eqref{Int_phi_2}. 
 \siming{We should double check the proof one more time!! Especially the last line (about the bound of the big factor.) I didn't check.}
 
\noindent
{\bf Step \# 2: Proof of \eqref{phI_est_1_1}.} We apply the following argument:
\begin{align*}
|k|^m \wt\Gamma_k^n\phn =  \int_{v(-1)}^{v(1)}  |k|^{m}(\pa_v+ikt)^n\GG_k(v,v') R^{(I)}_k(v')dv'.
\end{align*}
Now we invoke the structure of the Green's function \eqref{greensRep}
\begin{align}
\GG_k(v,v')=H_k(v-v')+S_k(v+v').
\end{align}
Here $H_k(v-v')$ encodes the jumps of derivatives and the $S_k(v+v')$ satisfies the following bound ($k>0$)
\begin{align}
\n\lf||k|^m\pa_v^n S_k(v+v')\rg|=&\lf||k|^m\frac{\pa_v^n\cosh(k(v+v'-v(t,1)-v(t,-1))) }{2k\sinh(k|v(t,1)-v(t,-1)|)} \rg|\\
\lesssim& |k|^{m+n}\max\lf\{e^{k(v+v'-2v(t,1))},e^{-k(v+v'-2v(t,-1))}\rg\} \n \\
\lesssim&|k|^{m+n}e^{-\mathfrak c |k|}\lesssim C(\mf c,{m+n}),\quad \text{dist}(v',\{v(t,1),v(t,-1)\})\geq\mathfrak{c}>0.\label{S_estimate}
\end{align} Hence we have that
\begin{align*}
|k|^m \wt\Gamma_k^n\phn=& (\pa_v+ikt)^{n-1}\int_{v(-1)}^{v(1)}  |k|^{m}\lf(-{\pa_{v'}}+ikt\rg)\GG_k(v,v')R^{(I)}_k(v')dv'\\
&+(\pa_v+ikt)^{n-1}\int_{v(-1)}^{v(1)}  |k|^{m}\lf(\pa_v+{\pa_{v'}}\rg)\GG_k(v,v')R^{(I)}_k(v')dv'.
\end{align*}
Since $(\pa_v+{\pa_{v'}})H_k(v-v')=0$ and $\GG_k(v,v'=\pm 1)=0$, we have that 
\begin{align*}
|k|^m \wt\Gamma_k^n\phn=& (\pa_v+ikt)^{n-1}\int_{v(-1)}^{v(1)}  |k|^{m}\GG_k(v,v') \lf({\pa_{v'}}+ikt\rg)R^{(I)}_k(v')dv'\\
&+(\pa_v+ikt)^{n-1}\int_{v(-1)}^{v(1)}  |k|^{m}2\pa_v S_k(v+v')R^{(I)}_k(v')dv'\\
=&(\pa_v+ikt)^{n-2}\int_{v(-1)}^{v(1)}  |k|^{m}\GG_k(v,v') \lf({\pa_{v'}}+ikt\rg)^2R^{(I)}_k(v')dv'\\
&+(\pa_v+ikt)^{n-2}\int_{v(-1)}^{v(1)}  |k|^{m}2\pa_v S_k(v+v')\lf({\pa_{v'}}+ikt\rg) R^{(I)}_k(v')dv'\\
&+(\pa_v+ikt)^{n-2}\int_{v(-1)}^{v(1)}  |k|^{m}2\pa_v^2 S_k(v+v')R^{(I)}_k(v')dv'\\
&+(\pa_v+ikt)^{n-2}\int_{v(-1)}^{v(1)}  |k|^{m}2\pa_v S_k(v+v')({\pa_{v'}}+ikt)R^{(I)}_k(v')dv'\\
&-(\pa_v+ikt)^{n-2}\int_{v(-1)}^{v(1)}  |k|^{m}2\pa_v S_k(v+v'){\pa_{v'}}R^{(I)}_k(v')dv'.
\end{align*}
Since $R_k^{(I)}$ and all its higher derivatives vanish on the boundary, we implement the integration by parts in the last term,
\begin{align*}
|k|^m \wt\Gamma_k^n\phn
=&(\pa_v+ikt)^{n-2}\int_{v(-1)}^{v(1)}  |k|^{m}\GG_k(v,v') \lf({\pa_{v'}}+ikt\rg)^2R^{(I)}_k(v')dv'\\
&+4(\pa_v+ikt)^{n-2}\int_{v(-1)}^{v(1)}  |k|^{m}\pa_v S_k(v+v')\lf({\pa_{v'}}+ikt\rg) R^{(I)}_k(v')dv'\\
&+4(\pa_v+ikt)^{n-2}\int_{v(-1)}^{v(1)}  |k|^{m}\pa_v^2 S_k(v+v')R^{(I)}_k(v')dv'.
\end{align*}
Hence after $n$ iteration steps, we have that 
\begin{align*}
|k|^m\wt\Gamma_k^n \phi_k^{(I)}=\int_{v(-1)}^{v(1)}|k|^m\GG_k(v,v')\wt\Gamma_k^n R^{(I)}_k(v')dv'+\sum_{m'+n'=0}^{n}C_{m',n'}\int_{v(-1)}^{v(1)}  |k|^{m+m'}\pa_v^{m'} S_k(v+v')\wt\Gamma_k^{n'}R^{(I)}_k(v')dv'.
\end{align*}
Now we follow the same strategy as in the proof of \eqref{phi_I_est_3}. By applying the fact that $\pa_{v'}(e^{ikt v'}f_k(v'))=e^{ikt v'}(\pa_{v'}+ikt)f_k(v')=e^{ikt v'}\wt \Gamma_kf_k$, we obtain that  
\begin{align*}
|k|^m\wt\Gamma_k^n \phi_k^{(I)}=&\int_{v(-1)}^{v(1)}|k|^m\GG_k(v,v')\frac{{\pa_{v'}}^2 e^{-iktv'}}{(-ikt)^2} e^{iktv'}\wt\Gamma_k^n R^{(I)}_k(v')dv'\\
&+\sum_{m'+n'=0}^{n}C_{m',n'}\int_{v(-1)}^{v(1)}  |k|^{m+m'}\pa_v^{m'} S_k(v+v')\frac{{\pa_{v'}}^2 e^{-iktv'}}{(-ikt)^2} e^{iktv'}\wt\Gamma_k^{n'}R^{(I)}_k(v')dv'\\
=&\frac{1}{(-ikt)^2}\int_{v(-1)}^{v(1)}|k|^m\GG_k(v,v')\wt\Gamma_k^{n+2} R^{(I)}_k(v')dv'\\
&+\frac{2}{(-ikt)^2}\int_{v(-1)}^{v(1)}|k|^m\pa_v\GG_k(v,v') \wt\Gamma_k^{n+1} R^{(I)}_k(v')dv' 
+\frac{1}{(-ikt)^2}|k|^m \wt\Gamma_k^{n} R^{(I)}_k(v')\\
&+\frac{1}{k^2t^2}\sum_{m'+n'=0}^{n+2}\wt C_{m',n'}\int_{v(-1)}^{v(1)}  |k|^{m+m'}\pa_v^{m'} S_k(v+v')\wt\Gamma_k^{n'}R^{(I)}_k(v')dv'.
\end{align*}
Here we recall the $S_k$-estimate for $v'$ supported away from the boundary \eqref{S_estimate}. 
By taking the $L^2$ norm and invoking \eqref{phi_I_Sbv}, we have the following
\begin{align*}
\sum_{m+n=0}^{1000}\widehat{B}_{m,n}^2\||k|^m\wt\Gamma_k^n \phi_k^{(I)}\|_{L_y^2}^2\lesssim \frac{1}{k^4t^4}\sum_{m'+n'=0}^{1002}\widehat{B}_{m',n'}^2\||k|^{m'}\wt\Gamma_k^{n'}R_k^{(I)}\|_{L_v^2}^2\lesssim\frac{1}{k^4t^4}\sum_{m'+n'=0}^{1002}\widehat{B}_{m',n'}^2\||k|^{m'}\wt\Gamma_k^{n'}\ww _k^{(I)}\|_{L_v^2}^2.
\end{align*}
This concludes the proof of \eqref{phI_est_1_1}.
\end{proof}

{\subsection{Bounds on $\mathcal{J}_{ell}^{(\ell)}$}


We recall the equation \eqref{d:phi:E:in}
\begin{align}
 \de_k \phi_k^{(E)}=\wt{\ww}_k^{(E)}+\wt\chi_1(\pav^2-\pa_y^2)\phi_k^{(I)},\qquad \phe_k|_{y=\pm1}=0.
\end{align}
Recalling the vector field $\Gamma_kf_k= \left( {v_y^{-1}}\pa_y+ikt\right)f_k=(\pav+ikt)\ f_k$, the $\chi_{m+n}$ cutoff \eqref{chi} and the $q^n$ weight \eqref{q:defn}, then the equation for $\chi_{m+n}|k|^mq^n \Gamma_k^n \phi_k$ can be rewritten as follows:
\begin{align}\n
\de_k( \chi_{m+n}&|k|^m q^n\Gamma_k^n \phi_k^{(E)})\\ \n
=&\chi_{m+n} |k|^mq^n\Gamma_k^n \wt\ww_k^{(E)}+{\chi_{m+n} |k|^mq^n\Gamma_k^n(\wt\chi_1(\pav^2-\pa_y^2)\phi_k^{(I)})}+\chi_{m+n}|k|^m q^n[\pa_{yy}, \Gamma_k^n]\phi_k^{(E)}\\
\n &+\chi_{m+n}[\pa_{yy},q^n](|k|^m\Gamma_k^n \phi_k^{(E)})+[\pa_{yy},\chi_{m+n}](|k|^mq^n \Gamma_k^n\phi_k^{(E)})\\
=:&\chi_{m+n} |k|^m q^n\Gamma_k^n \wt\ww_k^{(E)}+{\mathbb{C}_{m,n}^{(I)}}+\mathbb{T}_{1;m,n}+\mathbb{T}_{2;m,n}+\mathbb{T}_{3;m,n}.\label{T123}
\end{align}
Here $\mathbb{C}_{m,n}^{(I)}$ denotes the contribution from the interior and $\{\mathbb{T}_{j;m,n}\}_{j=1}^3$ represents the commutator terms. We recall the simplified notations
\begin{align}
\phmn:=&\chi_{m+n}|k|^m q^n\Gamma_k^n \phi_k^{(E)},\quad \omn:=\chi_{m+n}|k|^m q^n\Gamma_k^n \wt\ww_k^{(E)}.\label{phe_ome_mnk}
\end{align}
We expand the expression 
\begin{align}\n
(\pav^2-\pa_y^2)\phi_k^{(I)}=&\pav^2 \phi_k^{(I)}-v_y\pav(v_y\pav\phi_k^{(I)})=(1-v_y^2)\pav^2\phi_k^{(I)}+v_{yy}\pav\phi_k^{(I)}\\ =
&- (v_y^2-1)  \pav^2\phi_k^{(I)}+\frac{1}{2}\pav(v_{y}^2-1)\pav\phi_k^{(I)}.\label{pphiI_dcmp}
\end{align}
The interior contribution can be rewritten as follows
\begin{align}\mathbb{C}_{m,n}^{(I)}:=\chi_{m+n}|k|^mq^n\Gamma_k^n\underbrace{\lf(-\wt\chi_1Z\  \pav^2\phi_k^{(I)}+\frac{1}{2}\wt\chi_1\pav Z\ \pav\phi_k^{(I)}\rg)}_{\displaystyle =\wt\chi_1(\pav^2-\pa_y^2)\phi_k^{(I)}},\quad Z:=  v_y^2- 1.\label{CImn}
\end{align}
Here the commutator terms are the following \siming{(Check the $\pm$-sign!)}
\begin{subequations}\label{T_i} 
\begin{align}
\mathbb{T}_{1;m,n}=&\chi_{m+n}|k|^m q^n \left(\sum_{\ell=0}^{n-1}\binom{n}{\ell} \left(2\pav^{n-\ell-1}v''\ \pav^2 +\pav^{n-\ell}v'' \ \pav\right)\Gamma_k^{\ell}\right)\phe_k;\label{T_1}\\
\mathbb{T}_{2;m,n}=&\chi_{m+n}|k|^m(n(n-1)q^{n-2}(q')^2+nq^{n-1} q''+2nq^{n-1}q'\pa_y)\Gamma_k^n\phe_k;\label{T_2}\\
\mathbb{T}_{3;m,n}=&(2\chi_{m+n}'\pa_y+\chi_{m+n}'')|k|^m q^n\Gamma_k^n \phe_k.\label{T_3}
\end{align}
\end{subequations}
We organize our analysis as follows: in Subsection \ref{sub2sc:Int}, we derive estimate on the interior contribution $\mathbb{C}_{m,n}^{(I)}$; in Subsection \ref{sub2sct:lw_rg}, we derive estimate for small $n$; in Subsection \ref{sub2sct:h_rg}, we derive estimate for large $n$; in Subsection \ref{sub2sct:con}, we conclude. 
\subsubsection{Interior Contributions}\label{sub2sc:Int}
We present the estimate of the second term in \eqref{T123}.
\begin{lemma}\label{lem:R_I_est} For any $M\in \mathbb{N}$, the following estimate of the $\mathbb{C}_{m,n}^{(I)}$ \eqref{CImn} holds
\begin{align}
\sum_{m+n=0}^{M}{\bf a}_{m,n}^2\|{\mathbb{C}_{m,n}^{(I)}}\|_{L_y^2}^2
\lesssim&\|\wt\ww_k^{(I)}\|_{L^2}^2\sum_{n=0}^{M}\sum_{a+b\leq 1}B_{0,n}^2\varphi^{2n{+2}}\|\mathbbm{1}_{\mathrm{supp}\wt\chi_1} J_{0,n}^{(a,b,0)}Z\|_{L^2}^2.\label{pv2-py2phI}
\end{align}Moreover, 
\siming{\footnote{\myr{Siming: We should write the right hand side in terms of $\|J^{(a,b,0)}_{0,n}Z\|_{L^2}$. This is expected because for $n\geq 1$, we have that $\chi_nq^n\pav^n(\wt \chi_1\pav Z)=\chi_nq^n\pav^{n+1} Z$, now we quote Lemma \ref{lem:vy2-1est} and its variants. The $n=0$ case can be estimated. \siming{We need $\|\wt \chi_1 Z\|_{L^2}+\|\wt \chi_1\pav Z\|_{L^2}\lesssim \exp\{-\nu^{-1/7}\}.$}}}}
\begin{align}
\sum_{m+n=0}^{M}{\bf a}_{m,n}^2\|\pa_y{\mathbb{C}_{m,n}^{(I)}}\|_{L_y^2}^2
\lesssim&\|\wt\ww_k^{(I)}\|_{\myr{H^1}}^2\sum_{n=0}^{M}\sum_{a+b\leq 2}B_{0,n}^2\varphi^{2n{+2}}\|\mathbbm{1}_{\mathrm{supp}\wt\chi_1} J_{0,n}^{(a,b,0)}Z\|_{L^2}^2.
\label{ppv2-py2phI}
\end{align}Here, we use the notation ${Z:=(v_y^2-1) }$. Moreover, the implicit constant does not depend on $M$ and we can take the limit as $M\rightarrow \infty. $ \siming{The extra $\varphi^2$ in the estimate will not cause problem because we have smallness of the $Z$ on the support of $\wt \chi_1$ $(\lesssim \exp\{-\nu^{-1/7}\}).$} 
\end{lemma}
\siming{We note that as long as $n\neq 0$, the $\wt\chi_1$ in the $\Gamma_k^n$ is effectively $1$. So we don't really need to worry too much about $\wt\chi_1.$ If $n=0$, then we have $\wt\chi_1\pav(v_y^2-1)$ localization control. }
\begin{proof} We divide the proof into two steps.

\noindent{\bf Step \# 1: Proof of \eqref{pv2-py2phI}.}

Thanks to Proposition \ref{pro:IE_phi_ext}, we observe that $\pav\phi_k^{(I)},\, \pav^2\phi_k^{(I)}$ have very high regularity. 

Now we are ready to estimate the expression \eqref{pv2-py2phI}. We apply the product rule
\begin{align}\n
\sum_{m+n=0}^{M}&{\bf a}_{m,n}^2\|\chi_{m+n}|k|^mq^n\Gamma_k^n(\wt\chi_1 (\pav^2-\pa_y^2)\phi_k^{(I)})\|_{L_y^2}^2\\ \n
=&\sum_{m+n=0}^{M}{\bf a}_{m,n}^2\lf(\lf\|\chi_{m+n}|k|^mq^n\Gamma_k^n \lf(\wt\chi_1Z\pav^2\phi_k^{(I)} -\frac{1}{2}\wt\chi_1\pav Z\pav\phi_k^{(I)}\rg)\rg\|_{L_y^2}\rg)^2\\ \n
\lesssim&\sum_{m+n=0}^{M}{\bf a}_{m,n}^2\lf(\sum_{\ell=0}^n\binom{n}{\ell}\lf\|\chi_{m+n}q^n \lf(\pav^{n-\ell}(\wt\chi_1Z)\ |k|^m\Gamma_k^{\ell}\pav^2\phi_k^{(I)} -\frac{1}{2}\pav^{n-\ell}(\wt\chi_1\pav Z)\ |k|^{m}\Gamma_k^\ell\pav\phi_k^{(I)}\rg)\rg\|_{L_y^2}\rg)^2\\ \n
\lesssim&\sum_{m+n=0}^{M}{\bf a}_{m,n}^2\lf(\sum_{\ell=0}^n\binom{n}{\ell}\lf\|\chi_{m+n}q^n \pav^{n-\ell}(\wt\chi_1Z)\ |k|^m\Gamma_k^{\ell}\pav^2\phi_k^{(I)} \rg\|_{L_y^2}\rg)^2\\ 
&+\sum_{m+n=0}^{M}{\bf a}_{m,n}^2\lf(\sum_{\ell=0}^n\binom{n}{\ell}\lf\|\chi_{m+n}q^n \pav^{n-\ell}(\wt\chi_1\pav Z)\ |k|^{m}\Gamma_k^\ell\pav\phi_k^{(I)}\rg\|_{L_y^2}\rg)^2=:T_1+T_2.\label{py-pvphi_T12}
\end{align}
Recalling the definitions \eqref{Bweight}, \eqref{Gj_intro}, \eqref{hat_bf_a_intro}, we have the following relation 
\begin{align}\n
\frac{{\bf a}_{m,n}}{B_{0,n-\ell} \widehat {\bf a}_{m,\ell}}= &\frac{\frac{\lambda^{(m+n)s}\varphi^{(n+1)}}{((m+n)!)^{s}}}{\frac{\lambda^{(n-\ell)s}} {((n-\ell)!)^{s}}\times \frac{\widehat\lambda(t)^{(\ell+m)s}\varphi^{(\ell+1)}}{((\ell+m)!)^{{s}}}}=\varphi^{n-\ell}\left(\frac{(n+m-(m+\ell))! (\ell+m)!}{{(n+m)!}}\right)^{s}\\
=&\varphi^{n-\ell}\binom{n+m}{m+\ell}^{-s}\leq\varphi^{n-\ell}\binom{n}{\ell}^{-s}. \label{py-pvphi_rel}
\end{align} 

Now we apply the H\"older inequality to obtain that 
\begin{align}\n
 {T}_{1}
 \lesssim &\sum_{m+n=0}^M\lf( \sum_{\ell=0}^{n-1}\binom{n}{\ell}\frac{{\bf a}_{m,n}}{B_{0,n-\ell}\widehat{\bf a}_{m,\ell}}B_{0,n-\ell} \lf\|\chi_{n-\ell}q^{n-\ell}\pav^{n-\ell}(\wt\chi_1 Z)\rg\|_{L_y^2} \widehat{\bf a}_{m,\ell}\| |k|^m\Gamma_{k}^\ell \pav^2\phi_k^{(I)}\|_{L_v^\infty}\rg)^2 \\
\n &\lesssim \sum_{m+n=0}^M\lf(\sum_{\ell=0}^{n}\binom{n}{ \ell}^{1-s}\varphi^{n-\ell}B_{0,n-\ell} \lf\|\chi_{n-\ell}q^{n-\ell}\pav^{n-\ell}(\wt\chi_1Z)\rg\|_{L_y^2}\widehat{ {\bf a}}_{m,\ell}\||k|^m\Gamma_{k}^\ell \pav^2\phi_k^{(I)}\|_{L_v^\infty}\rg)^2.\n
\end{align}
Now we have that by the relation \eqref{s:prime}, the combinatorial fact \eqref{prod}, and interchanging the order of summation, 
\begin{align*}\n
T_{1 } \lesssim&\sum_{m+n=0}^M \lf(\sum_{\ell=0}^{n-1}\binom{n}{\ell}^{-2\sigma-2\sigma_\ast}\rg) 
 \lf(\sum_{\ell=0}^{n-1}\varphi^{2n-2\ell}B_{0,n-\ell;0}^2\lf\|\chi_{n-\ell}q^{n-\ell}\pav^{n-\ell}(\wt\chi_1Z)\rg\|_{L^2}^2 \widehat{\bf a}_{m,\ell}^2\| |k|^m\Gamma_{k}^\ell\pav^2\phi^{(I)}_k\|_{L_v^\infty}^2\rg)\n \\
\lesssim&\sum_{m=0}^M\sum_{n=0}^{M-m} \sum_{\ell=0}^{n-1}\varphi^{2n-2\ell}B_{0,n-\ell} ^2\lf\|\chi_{n-\ell}q^{n-\ell}\pav^{n-\ell}(\wt\chi_1Z)\rg\|_{L^2}^2\widehat{\bf a}_{m,\ell}^2\| |k|^m\Gamma_{k}^\ell\pav^2\phi^{(I)}_k\|_{L_v^\infty}^2 \\
=&\sum_{m=0}^M\sum_{\ell=0}^{M-m-1} \sum_{n=\ell+1}^{M-m}\varphi^{2n-2\ell}B_{0,n-\ell} ^2\lf\|\chi_{n-\ell}q^{n-\ell}\pav^{n-\ell}(\wt\chi_1Z)\rg\|_{L^2}^2\widehat{\bf a}_{m,\ell}^2\| |k|^m\Gamma_{k}^\ell\pav^2\phi^{(I)}_k\|_{L_v^\infty}^2\\
\lesssim&\lf(\sum_{n=0}^{M}B_{0,n}^2\varphi^{2n}\|\chi_nq^n\pav^n (\wt\chi_1Z)\|_{L^2}^2\rg)\lf(\sum_{m+\ell=0}^M\widehat{\bf a}_{m,\ell}^2\| |k|^m\Gamma_{k}^\ell\pav^2\phi^{(I)}_k\|_{L_v^\infty}^2\rg). 
\end{align*}
Thanks to the Proposition \ref{pro:IE_phi_ext}, we have that 
\begin{align}
T_1\lesssim \lf(\sum_{n=0}^{M}B_{0,n}^2\varphi^{2n\myr{+2}}\|\chi_nq^n\pav^n (\wt\chi_1Z)\|_{L^2}^2\rg)\|\wt\ww_k^{(I)}\|_{L_v^2}^2.
\end{align}
\siming{The estimate for $\phi_k^{(I)}$ has $\varphi^2.$} 
A similar argument yields that 
\begin{align}
T_2\lesssim\lf(\sum_{n=0}^{M}B_{0,n}^2\varphi^{2n\myr{+2}}\|\chi_nq^n\pav^{n}(\wt\chi_1\pav Z)\|_{L^2}^2\rg)\|\wt\ww_k^{(I)}\|_{L_v^2}^2.
\end{align}
Combining the bounds and \eqref{py-pvphi_T12} yields that
\begin{align*}
\sum_{m+n=0}^{M}{\bf a}_{m,n}^2\|&{\mathbb{C}_{m,n}^{(I)}}\|_{L_y^2}^2
\lesssim\sum_{n=0}^{M}B_{0,n}^2\varphi^{2n {+2}}\lf(\|\chi_nq^n\pav^{n} (\wt\chi_1 Z)\|_{L^2}^2+\|\chi_nq^n\pav^{n}(\wt\chi_1\pav Z)\|_{L^2}^2\rg)\|\wt\ww_k^{(I)}\|_{L_v^2}^2\\
\lesssim&\|\wt\ww_k^{(I)}\|_{L_v^2}^2\lf(\sum_{n=1}^{M}B_{0,n}^2\varphi^{2n {+2}}\lf(\|J^{(0,0,0)}_{0,n}Z\|_{L^2}^2+\|\chi_nq^n\pav^{n+1} Z\|_{L^2}^2\rg)+\sum_{\ell=0}^1\myr{\varphi^2}\|\mathbbm{1}_{\text{supp}\wt\chi_1} \pav^\ell Z\|_{L^2}^2\rg).
\end{align*} 
Now we invoke Lemma \ref{lem:vy2-1est} to obtain  \eqref{pv2-py2phI}. This concludes {\bf Step \# 1}.

\noindent {\bf Step \# 2: Proof of \eqref{ppv2-py2phI}.}
We expand the expression on the left hand side as follows
\begin{align}\n
\sum_{m+n=0}^{M}&{\bf a}_{m,n}^2\lf\|\pa_y\lf(\chi_{m+n}|k|^mq^n\Gamma_k^n \lf(\wt \chi_1(\pav^2-\pa_y^2)\phi_k^{(I)}\rg)\rg)\rg\|_{L_y^2}^2\\ \n
\lesssim& \sum_{m+n=0}^{M}{\bf a}_{m,n}^2\lf\| \chi_{m+n}|k|^mq^n   \pav\lf(\Gamma_k^n\lf(\wt \chi_1 (\pav^2-\pa_y^2)\phi_k^{(I)}\rg)\rg)\rg\|_{L_y^2}^2\\
&+\sum_{m+n=0}^{M}{\bf a}_{m,n}^2\lf\|\pav\lf(\chi_{m+n}q^n\rg)|k|^m\Gamma_k^n \lf(\wt \chi_1 (\pav^2-\pa_y^2)\phi_k^{(I)} \rg)\rg\|_{L_y^2}^2
=:T_3+T_4.\label{pv2py2T345}
\end{align} 
We first estimate the $T_3$, which can be further decomposed into two sub-terms ($m+n\neq 0$ and $m+n=0$) with the expression \eqref{pphiI_dcmp}:
\begin{align}\n
T_3\leq &\sum_{m+n=0}^{M}{\bf a}_{m,n}^2\lf\| \chi_{m+n}|k|^mq^n \pav\lf(\Gamma_k^n\lf(\wt \chi_1 Z\ \pav^2\phi_k^{(I)}\rg)\rg)\rg\|_{L_y^2}^2\\ \n
&+\sum_{m+n=0}^{M}{\bf a}_{m,n}^2\lf\| \chi_{m+n}|k|^mq^n \pav\lf(\Gamma_k^n\lf(\wt \chi_1\pav Z\ \pav\phi_k^{(I)}\rg)\rg)\rg\|_{L_y^2}^2\\ 
\lesssim&\sum_{m+n=1}^{M}\sum_{\mf l=0}^2{\bf a}_{m,n}^2\lf\| \chi_{m+n}|k|^mq^n \Gamma_k^n\lf(\pav^{\mf l}(\wt \chi_1 Z)\ \pav^{3-\mf l}\phi_k^{(I)}\rg)\rg\|_{L_y^2}^2\\
&+\mathbbm{1}_{m+n=0}\varphi^2\sum_{\ell_1=0}^2\sum_{\ell_2=0}^3\|\mathbbm{1}_{\text{supp}\wt\chi_1}\pav^{\ell_1} Z\|_{L^2}^2\|\pav^{\ell_2}\phi^{(I)}_k\|_{L^2}^2=:T_{31}+T_{32}.\label{pv2py2T3}
\end{align}
Thanks to the Sobolev estimate \eqref{phi_I_Sbv_2}, we have that the second term is bounded
\begin{align}\label{pv2py2T32}
T_{32}\lesssim \mathbbm{1}_{m+n=0}\sum_{b=0}^2B_{0,0}^2\myr{\varphi^2}\|\mathbbm{1}_{\text{supp}\wt\chi_1}J^{(0,b,0)}_{0,0} Z\|_{L^2}^2\|\ww^{(I)}_k\|_{H^1}^2.
\end{align}
\siming{Here it seems that if we use the Sobolev estimate, we will need the $\|v_y^2-1\|_{H^3}$ and this is fine.}
Next, we focus on the estimate of the first term $T_{31}$. We recall the property of the cutoff $\chi_\ast$ \eqref{chi_ast} that $\chi_\ast\equiv 1$ on the support of $\chi_{m+n}, \ m+n\neq 0$, and apply the product rule to derive that 
\begin{align*}
&T_{31}\lesssim\sum_{m+n=1}^{M}\sum_{\mf l=0}^2{\bf a}_{m,n}^2\lf(\sum_{\ell=0}^n \binom{n}{\ell}\lf\|\chi_{m+n}\chi_{\ast} \lf( q^{n-\ell}\pav^{n-\ell+\mf l}(\wt\chi_1Z)\rg)\lf(|k|^m q^\ell\Gamma_k^{\ell}\pav^{3-\mf l}\phi_k^{(I)} \rg)\rg\|_{L_y^2}\rg)^2\\
&\lesssim\sum_{m+n=1}^{M}\sum_{\mf l=0}^2 \lf(\sum_{\ell=0}^n \binom{n}{\ell}^{-\sigma-\sigma_\ast}B_{0,n-\ell}\varphi^{n-\ell}\lf\|\chi_{n-\ell}q^{n-\ell} \pav^{n-\ell+\mf l}(\wt\chi_1Z)\rg\|_{L_y^2}{\bf \widehat{a}}_{m,\ell}\lf\|\chi_{\ast}|k|^m q^\ell\Gamma_k^{\ell}\pav^{3-\mf l}\phi_k^{(I)}\rg\|_{L_y^{\infty}}\rg)^2.
\end{align*}
Here in the last line, we have used the relation \eqref{py-pvphi_rel} and the H\"older inequality. Then we apply the bound of $\phi_k^{(I)}$ \eqref{phi_I_est}, the estimate of $Z=v_y^2-1$ \eqref{vy2_1est4} \siming{(HS: If $\ell=n$, then we can directly read the bound.)}, and the combinatorial fact \eqref{prod}    to obtain that 
\begin{align*}
&T_{31}\lesssim\sum_{\mf l=0}^2\sum_{m+n=1}^M\lf(\sum_{\ell=0}^{n}\binom{n}{\ell}^{-2\sigma}\rg)\lf(\sum_{\ell=0}^{n}\sum_{a+b\leq 2}B_{0,n-\ell}^2\varphi^{2n-2\ell}\lf\|\mathbbm{1}_{\text{supp}\wt\chi_{1}} J^{(a,b,0)}_{0,n-\ell}Z\rg\|_{L_y^2}^2\ \widehat{\bf a}_{m,\ell}^2\lf\|\chi_\ast |k|^m q^\ell\Gamma_k^{\ell}\pav^{3-\mf l}\phi_k^{(I)}\rg\|_{L_y^\infty}^2\rg) \\
&\lesssim\sum_{\mf l=0}^2\sum_{m=0}^M\sum_{n=0}^{M-m}\sum_{\ell=0}^{n}\lf(\sum_{a+b\leq 2}B_{0,n-\ell}^2\varphi^{2n-2\ell}\lf\|\mathbbm{1}_{\text{supp}\wt\chi_{1}} J^{(a,b,0)}_{0,n-\ell}Z\rg\|_{L_y^2}^2\rg) \widehat{\bf a}_{m,\ell}^2\lf\|\chi_\ast |k|^m q^\ell\Gamma_k^{\ell}\pav^{3-\mf l}\phi_k^{(I)}\rg\|_{L_y^\infty}^2\\
&\lesssim \sum_{\mf l=0}^2\sum_{m=0}^M\sum_{\ell=0}^{M-m} \widehat{\bf a}_{m,\ell}^2\lf\|\chi_\ast |k|^m q^\ell\Gamma_k^{\ell}\pav^{3-\mf l}\phi_k^{(I)}\rg\|_{L_y^\infty}^2\lf(\sum_{n=\ell}^{M-m}\sum_{a+b\leq 2}B_{0,n-\ell}^2\varphi^{2n-2\ell}\lf\|\mathbbm{1}_{\text{supp}\wt\chi_{1}} J^{(a,b,0)}_{0,n-\ell}Z\rg\|_{L_y^2}^2\rg)\\
&\lesssim \|\wt \ww_k^{(I)}\|_{L^2}^2\sum_{n=0}^M\sum_{a+b\leq 2} B_{0,n}^2\varphi^{2n\myr{+2}}\|\mathbbm{1}_{\mathrm{supp}\wt \chi_1}J^{(a,b,0)}_{0,n}Z\|_{L_y^2}^2.
\end{align*}
Combining this bound with \eqref{pv2py2T32} yields that the $T_3$ term is consistent with \eqref{ppv2-py2phI}. 

To estimate the $T_4$-term in \eqref{pv2py2T345}, we invoke the expansion \eqref{CImn}, the bound \eqref{chi:prop:3}, the bound \eqref{s:prime} and the observation that the term with $m+n=0$ is zero. Then we apply the argument as in \eqref{pv2py2T3} to obtain that 
\begin{align*}
T_4\lesssim&\sum_{m+n=1}^{M}\sum_{\mf l=0}^1\mathbbm{1}_{n\geq 1}{\bf a}_{m,n}^2\lf\| (m+n)^{1+\sigma}\chi_\ast\chi_{m+n-1}q^{n-1}|k|^m(\pav+ikt)\Gamma_k^{n-1}\lf(\pav ^{\mf l}Z \ \pav^{2-\mf l}\phi_k^{(I)}\rg)\rg\|_{L_y^2}^2\\
&+\sum_{m=1}^{M}\sum_{\mf l=0}^1{\bf a}_{m,0}^2\lf\| m^{1+\sigma}\chi_\ast\wt \chi_1\chi_{m-1}|k|^m \lf(\pav ^{\mf l}Z\ \pav^{2-\mf l}\phi_k^{(I)}\rg)\rg\|_{L_y^2}^2\\
\lesssim&\myr{\sum_{\mf l=0}^{2}}\sum_{m+n=1}^{M}\mathbbm{1}_{n\geq 1}\sum_{\ell=0}^{n}\frac{\lambda^{2s}{\bf a}_{m,n-1}^2\varphi^2}{(m+n)^{2\sigma_\ast}}\binom{n-1}{\ell}\lf\| \mathbbm{1}_{\text{supp}\wt \chi_1} \chi_{n-1-\ell}q^{n-1-\ell}\pav^{n-1-\ell+\mf l}Z\rg\|_{L^2}\lf\|\chi_\ast |k|^m\Gamma_k^{\ell}\pav^{\myr{3}-\mf l }\phi_k^{(I)}\rg\|_{L^\infty}^2\\
&+\myr{\sum_{\mf l=0}^{1}}\sum_{m+n=1}^{M}\sum_{\ell=0}^{n}\frac{\lambda^{2s}{\bf a}_{m,n-1}^2\myr{\varphi^2t^2}}{(m+n)^{2\sigma_\ast}}\binom{n-1}{\ell}\lf\| \mathbbm{1}_{\text{supp}\wt \chi_1} \chi_{n-1-\ell}q^{n-1-\ell}\pav^{n-1-\ell+\mf l}Z\rg\|_{L^2}\lf\|\chi_\ast |k|^m\Gamma_k^{\ell}\myr{|k|}\pav^{2-\mf l} \phi_k^{(I)}\rg\|_{L^\infty}^2\\
&+\lambda^{2s}\sum_{b=0}^1 \|\mathbbm{1}_{\text{supp}\wt \chi_1} J^{(0,b,0)}_{0,0}Z\|_{L^2}^2\sum_{m=1}^{M}{\bf a}_{m-1,0}^2m^{-2\sigma_\ast}\|\chi_{\ast} |k|\pav^{2-\mf l}|k|^{m-1}\phi_k^{(I)}\|_{L_y^\infty}^2.
\end{align*}
Now we invoke the bound \eqref{phi_I_est}, the combinatorial fact \eqref{prod}, the estimates \eqref{vy2_1est3}, \eqref{vy2_1est4} to obtain that  
\begin{align*}
T_4\lesssim &\sum_{\mf l=0}^{2}\sum_{m+n=1}^{M}\sum_{\ell=0}^{n}\frac{\lambda^{2s}}{(m+n)^{2\sigma_\ast}}B_{n-1-\ell}^2\varphi^{2n-2-2\ell}\lf\| \mathbbm{1}_{\text{supp}\wt \chi_1} \chi_{n-1-\ell}q^{n-1-\ell}\pav^{n-1-\ell+\mf l}Z\rg\|_{L^2}\\
&\qquad\qquad\qquad\qquad\times\widehat{\bf a}_{m,\ell}^2\lf(\lf\|\chi_\ast |k|^m\Gamma_k^{\ell}\pav^{\myr{3}-\mf l }\phi_k^{(I)}\rg\|_{L^\infty}^2+\lf\|\chi_\ast |k|^m\Gamma_k^{\ell}|k|\pav^{\myr{2}-\mf l }\phi_k^{(I)}\rg\|_{L^\infty}^2\rg)\\
&+\lambda^{2s}\sum_{b=0}^1 \varphi^2\|\mathbbm{1}_{\text{supp}\wt \chi_1} J^{(0,b,0)}_{0,0}Z\|_{L^2}^2\|\wt\ww_k^{(I)}\|_{L_v^2}^2\\
\lesssim&\|\wt\ww_k^{(I)}\|_{L_v^2}^2\sum_{n=0}^{M}\sum_{a+b\leq2}B_{0,n}^2\varphi^{2n+2} \|\mathbbm{1}_{\text{supp}\wt\chi_1} J_{0,n}^{(a,b,0)}Z\|_{L^2}^2.
\end{align*} 

\end{proof}

\subsubsection{Lower regularity bounds}
\label{sub2sct:lw_rg}
In this section, we derive estimates with small-$n$ level. To simplify the notation, we introduce the norm/semi-norm 
\begin{align}\|f_k\|_{\dot H_k^j}^2:=\sum_{b+c=j}\|\pa_y^b|k|^c f_k\|_{L_y^2}^2,\quad  
\|f_{k}\|_{H_k^j}^2:=\sum_{j'\leq j}\|f_k\|_{\dot H_{k}^{j'}}^2. \label{H_k_y}
\end{align}
The first main lemma is the following.
\begin{lemma}\label{lem:n<=1} Assume the bootstrap assumption. 
The following estimates hold for $n\leq 1$ and arbitrary $M\in \mathbb{N}$,\begin{subequations}\label{est_phen<=1}
\begin{align}
\sum_{m=0}^{M}\sum_{a+b+c\leq 2}{\bf a}_{m,0}^2\|J^{(a,b,c)}_{m,0}\phe_k\|_{L^2}^2\lesssim&\sum_{m=0}^{M} {\bf a}_{m,0}^2\lf(\| \mr\wwe_{m,0;k} \|_{L^2}^2+\|\ci_{m,0}\|_{L_y^2}^2\rg),\label{est_n=0_J2}\\
\sum_{m=0}^{M}\sum_{a+b+c= 3}{\bf a}_{m,0}^2\|J^{(a,b,c)}_{m,0}\phe_k\|_{L^2}^2\lesssim&\sum_{m=0}^{M} {\bf a}_{m,0}^2\lf(\| \mr\wwe_{m,0;k} \|_{\dot H_k^1}^2+\|\ci_{m,0}\|_{H_k^1}^2\rg)+|v_{yy}|\big|_{y=\pm 1}\|\wt\ww_k^{(I)}\|_{L_v^2}^2, \label{est_n=0_J3} \\ 
\sum_{m=0}^{M-1}\sum_{a+b+c\leq 2}{\bf a}_{m,1}^2\|J^{(a,b,c)}_{m,1}\phe_k\|_{L^2}^2\lesssim&\sum_{m+n=0}^{M} \mathbbm{1}_{n\in\{0,1\}}{\bf a}_{m,n}^2\lf(\| J_{m,n}^{(0)}\wwe_{k} \|_{L^2}^2+\|\ci_{m,n}\|_{L_y^2}^2\rg),
\label{est_n=1_J2}\\ 
\sum_{m=0}^{M-1}\sum_{a+b+c= 3}{\bf a}_{m,1}^2\|J^{(a,b,c)}_{m,1}\phe_k\|_{L^2}^2\lesssim&\sum_{m+n=0}^{M} \mathbbm{1}_{n\in\{0,1\}}{\bf a}_{m,n}^2\lf(\|J_{m,n}^{(0)} \wwe_{k} \|_{H^1_k}^2+\|\ci_{m,n}\|_{H_k^1}^2\rg)+|v_{yy}|\big|_{y=\pm 1}\|\wt\ww_k^{(I)}\|_{L_v^2}^2.\label{est_n=1_J3}
\end{align}\end{subequations}
Here the implicit constant is independent of $M$ and $\mathbb S_n^1$-set is defined in \eqref{Snell}.
\end{lemma}
\begin{proof}
\noindent
\textbf{Step \# 0: Preliminaries and $0$-th level estimates.}
First of all, we observe the following equivalence relation for $k\neq 0,\ \ell\in\{2,3\}$,
\begin{align}\sum_{b+c=\ell}\|J_{m,n}^{(0,b,c)}f_k\|_{L^2}
\approx \sum_{b+c=\ell}\|\pa_y^b |k|^c(\chi_{m+n}|k|^mq^n\Gamma_k^nf_{k})\|_{L^2}. 
\label{Jel_equiv}
\end{align}  
Here the implicit constant depends on $ \|v_y\|_{L_t^\infty W_y^{\ell-1,\infty}}$ and $\|\vyn\|_{L_{t,y}^\infty}$. 

Thanks to the elliptic estimate \eqref{mx_rg_v},  the relations $\eqref{Jel_equiv}_{\ell=2}$, \eqref{CImn} and the bound \eqref{pv2-py2phI}, we have that
\begin{align}
\sum_{b+c=2}&\|J_{0,0}^{(0,b,c)}\phe_k\|_{L_y^2}\lesssim\|\phe_k\|_{\dot H_{k}^2}\lesssim\|\wt \ww_{k}^{(E)}\|_{L_y^2}+\lf\|\mathbb{C}_{0,0}^{(I)}\rg\|_{L_y^2}.\label{J2_lv0} 
\end{align}
This is consistent with $\eqref{est_n=0_J2}_{m=0}$. 

Next we prove the $m=0$ case of the estimate \eqref{est_n=0_J3}. From the relations 
\begin{align}\n
\pa_{yy}\phe_k-|k|^2\phe_k=&\wwe_k+\wt \chi_1(\pav^2-\pa_y^2)\phi_k^{(I)},\quad (\pav^2 -|k|^2)\phi^{(I)}=\wt\ww_k^{(I)}+\wt \chi_1^\mathfrak{c}(\pav^2-\pa_y^2)\phi_k^{(I)},\\
\n \quad \phe_k(y=\pm 1)=&\wwe(y=\pm 1)=\phi_k^{(I)}(v=v(\pm 1))=\pav^2 \phi_k^{(I)}(v=v(\pm 1))=0,
\end{align} one obtain that 
\begin{align}\n\pa_{yy}\phe_k\bigg|_{ y=\pm 1}=-\pa_{yy}\phi^{(I)}_k\bigg|_{y=\pm 1}=\lf(-v_y^2\pav^2\phi_k^{(I)}- v_{yy}\pav\phi_k^{(I)}\rg)\bigg|_{y=\pm 1}=- v_{yy}\pav\phi_k^{(I)}\bigg|_{y=\pm 1}.
\end{align}
Thanks to the solution expression \eqref{greensRep} and the fact that $v(\pm 1)$ and the support of $R_k^{(I)}$ have order $1$ distance, we have that by \eqref{phi_I_est_2} \siming{(Check!)}
\begin{align}\label{bc_phI}
|\pa_{v}\phi_k^{(I)}(v=v(\pm1))|\leq \frac{C}{|k| ^2} \|\wt\ww_k^{(I)}\|_{L_v^2}.
\end{align}  Hence the following relations hold
\begin{align*}
\de_k|k|\phe_k=&|k|\wwe_k+|k|\wt \chi_1(\pav^2-\pa_y^2)\phi_k^{(I)},\quad |k|\phe_k\big|_{y=\pm 1}=0 ;\\
\de_k\pa_{y}\phe_k=&\pa_y\wwe_k+ \pa_y(\wt \chi_1(\pav^2-\pa_y^2)\phi_k^{(I)}),\quad \pa_{yy}\phe_k\big|_{y=\pm 1}=- v_{yy}\pav\phi_k^{(I)}\big|_{y=\pm 1} .
\end{align*}
Thanks to the elliptic bounds \eqref{mx_rg},  \eqref{mx_rg_b}, the definition of $\ci_{0,0}$ \eqref{CImn}, and the estimate \eqref{phi_I_Sbv_2} the following estimates hold
\begin{align*}\n
\||k|\phe_k\|_{\dot H_k^2}&+\|\pa_y \phe_k\|_{\dot H_{k}^2} 
\lesssim \||k|\wwe_k\|_{L^2}+\|\pa_y\wwe_k\|_{L^2}+\|\ci_{0,0}\|_{\dot H_{k}^1}+|k|^{1/2}|v_{yy}\pav\phi_k^{(I)}|\bigg|_{v=v(\pm 1)}\\
\lesssim& \|\wwe_k\|_{\dot H_k^1}+\|\ci_{0,0}\|_{\dot H_{k}^1}+ |v_{yy}|\big|_{y=\pm1}\|\wt\ww_k^{(I)}\|_{L_v^2} .
\end{align*}
Combining it with \eqref{Jel_equiv} yields that
\begin{align}
 \sum_{b+c=3}\|J_{0,0}^{(0,b,c)}\phe_k\|_{L^2}\lesssim \|\wwe_k\|_{\dot H_k^1}+\|\ci_{0,0}\|_{\dot H_k^1} +|v_{yy}|\big|_{y=\pm 1}\|\wt\ww_k^{(I)}\|_{L_v^2}. \label{J3_lv0}
\end{align}The implicit constant depends on $\|\vyn\|_{L^\infty}, \|v_y\|_{W^{2,\infty}}$. This is consistent with $\eqref{est_n=0_J3}_{m=0}$. 
This concludes the proof at the zeroth level. 

\noindent
{\bf Step \# 1: Proof of \eqref{est_n=0_J2}.}
\begin{align}\de_k \lf(\chi_m |k|^m \phe_k\rg)=&\chi_m|k|^m\wt \ww_k^{(E)} +[\pa_{yy},\chi_m]\lf( |k|^m\phe_k\rg)+\ci_{m,0},\quad \chi_m|k|^m\phe_k(y=\pm1)=0.\label{eq_lv2_x}
\end{align}
Thanks to the elliptic estimate \eqref{mx_rg}, we have that 
\begin{align}\n
\sum_{b+c=2}{\bf a}_{m,0}^2\|J^{(0,b,c)}_{m,0}\phe_k\|_2^2 
\lesssim & {\bf a}_{m,0}^2\|\chi_m|k|^m\wwe_k\|_{L^2}^2+{\bf a}_{m,0}^2\|[\pa_{yy},\chi_m](|k|^{m}\phe_k)\|_{L^2}^2+{\bf a}_{m,0}^2\|\ci_{m,0}\|_{L^2}^2\\
=:&T_{11}+T_{12}+T_{13}.\label{n=0_J2}
\end{align}
The $T_{11}$ and $T_{13}$ are already consistent with \eqref{est_n=0_J2}, so we focus on the $T_{12}$-term. If $m=0$, the commutator is vacuous. If $m=1$, the term can be bounded by applying \eqref{J2_lv0} 
\begin{align*}
T_{12}\lesssim {\bf{a}}_{1,0}^2\|\pa_y |k|\phe_k\|_{L_y^2}^2+{\bf{a}}_{1,0}^2\| |k|\phe_k\|_{L_y^2}^2\lesssim \|\wwe_k\|_{L_y^2}^2+\lf\|\ci_{0,0}\rg\|_{L_y^2}^2.
\end{align*}
Finally, if $m\geq 2$, we estimate the $T_{12}$ term with \eqref{s:prime} and the equivalence relation \eqref{Jel_equiv}, as follows:\begin{align*}
T_{12}\lesssim& {\bf a}_{m,0}^2\|\chi_m''|k|^2(\chi_{m-2}|k|^{m-2}\phe_k)\|_{L^2}^2+{\bf a}_{m,0}^2\|\chi_m'|k|\pa_y(\chi_{m-1}|k|^{m-1}\phe_k)\|_{L^2}^2\\
\lesssim& \lambda^{4s}{\bf a}_{m-2,0}^2\||k|^2(\chi_{m-2}|k|^{m-2}\phe_k)\|_{L^2}^2+\lambda^{2s}{\bf a}_{m-1,0}^2\||k|\pa_y(\chi_{m-1}|k|^{m-1}\phe_k)\|_{L^2}^2\\
\lesssim&\lambda^{2s} \sum_{m'=m-2}^{m-1}\sum_{b+c=2}{\bf a}_{m',0}^2\|J^{(0,b,c)}_{m',0}\phe_k\|_{L^2}^2.
\end{align*}
Now we sum from $m=0$ to $M$ in \eqref{n=0_J2} to obtain that 
\begin{align*}
\sum_{m=0}^M&\sum_{b+c=2} {\bf a}_{m,0}^2\|J_{m,0}^{(0,b,c)}\phe_k\|_{L_y^2}^2\\
\leq& C\sum_{m=0}^{M}{\bf{a}}_{m,n}^2\|\wwe_{m,0}\|_{L_y^2}^2+C\sum_{m=0}^M{\bf a}_{m,0}^2\lf\|\ci_{m,0}\rg\|_{L_y^2}^2+C\lambda^{2s}\sum_{m=0}^{M-1}\sum_{b+c=2} {\bf a}_{m,0}^2\|J_{m,0}^{(0,b,c)}\phe_k\|_{L_y^2}^2.
\end{align*} By taking the $\lambda>0$ small enough compared to the constant appeared in the expression, we absorb the last term by the left hand side and obtain \eqref{est_n=0_J2}.

\noindent
{\bf Step \# 2: Proof of \eqref{est_n=0_J3}.}
We consider the $m\geq 1$ case. Since $(a,b,c)\in S_{0}^3$, the only non-vacuous terms on the left hand side come from $\|\pa_y^b|k|^c\mr\phe_{m,0;k}\|_{L^2}, \, b+c=3$.    
First, we observe that the estimate \eqref{est_n=0_J2} yields that 
\begin{align}\label{est_n=0_J3_k}
\sum_{m=0}^M\sum_{b+c=2}{\bf a}_{m,0}^2\lf\|\pa_y^b|k|^{c+1}\phe_k\rg\|_{L^2}^2\lesssim \sum_{m=0}^M{\bf a}_{m,0}^2\lf\||k|\wwe_k\rg\|_{L^2}^2+\sum_{m=0}^M{\bf a}_{m,0}^2\lf\||k|\ci_{m,0}\rg\|_{L_y^2}^2.
\end{align} Hence the only unknown component is for $\pa_y^3\mr\phe_{m,0;k}$. To obtain this, we consider the term $\pa_y(\mr\phe_{m,0;k})=\pa_y(\chi_{m} |k|^m\phe_k)$. We observe that it solves the following equation with Neumann boundary condition,
\begin{align*}
\de_k\pa_y  \mr\phe_{m,0;k}=\pa_y \mr\wwe_{m,0;k}+\pa_y[\pa_{yy},\chi_m]\lf(|k|^m\phe_k\rg)+\pa_y \ci_{m,0},\quad \pa_{yy}\mr \phe_{m,0;k}\big|_{y=\pm 1}=-|k|^mv_{yy}\pav\phi_k^{(I)}\big|_{v=v(\pm 1)}.
\end{align*} 
The main term to estimate is the second term on the right hand side, we rewrite it in the following fashion,
\begin{align*}
&\pa_y[\pa_{yy},\chi_m]\lf(|k|^m\phi^{(E)}_k\rg)=\chi_m'''|k|^m\phe_k+2\chi_m''\pa_y\lf(|k|^{m}\phe_k\rg)+2\chi_m'\pa_{yy}\lf(|k|^m\phe_k\rg)\\ 
&=\lf\{\begin{array}{rr}\begin{aligned}\chi_{m}'''\ |k|^3\mr\phe_{{m-3},0;k}\ +\ 2\chi_{m}''\ |k|^2\pa_y\mr\phe_{m-2,0;k}\ +\ 2\chi_m'\ |k|\pa_{yy}\mr\phe_{m-1,0;k},&\quad {m\geq 3};\\
 \chi_2'''\ |k|^2\phe_k\ +\ 2\chi_2''\ |k|^{2}\pa_y\phe_k\ +\ 2\chi_2'|k|\pa_{yy}\lf(\chi_{1}|k|\phe_k\rg),&\quad  m= 2;\\
 \chi_1'''\ |k|\phe_k\ +\ 2\chi_1''\ |k|\pa_y \phe_k\  +\ 2\chi_1'\ |k|\pa_{yy}\phe_k ,&\quad  m= 1.\end{aligned}\end{array}\rg.
\end{align*}
Here $\mr\phe_{m,0;k}=\chi_{m}|k|^{m}\phe_k$. 
Thanks to the elliptic estimate \eqref{mx_rg_b}, $\eqref{est_J2}_{m=0},\ \eqref{est_J3}_{m=0}$,  \eqref{phi_I_est}, and \eqref{bc_phI}, we have that for $m=1$,
\begin{align*}
{\bf a}_{1,0}\|\pa_y \mr\phe_{1,0;k}\|_{\dot H^2_k}
\lesssim&{\bf a}_{1,0} \|\mr\wwe_{1,0;k}\|_{\dot H_k^1}+{\bf a}_{1,0}\||k|\pa_{yy}\phe_k\|_{L^2}+{\bf a}_{1,0}\||k|\pa_{y}\phe_k\|_{L^2}\\
&+{\bf a}_{1,0}\|\pa_y \ci_{1,0}\|_{L^2}+{\bf a}_{1,0}|k|^{1/2}||k| v_{yy}\pav\phi_k^{(I)}|\bigg|_{v=v(\pm1)}\\
\lesssim &\sum_{m=0}^1{\bf a}_{m,0}\|\mr\wwe_{m,0;k}\|_{\dot H_k^1}+\sum_{m=0}^1{\bf a}_{m,0}\|\ci_{m,0}\|_{H_k^1}+|v_{yy}|\big|_{y=\pm 1}\|\oi_k\|_{L_v^2}. 
\end{align*}
This is consistent with \eqref{est_n=0_J3}. For the $m=2$ case, we estimate it using the bounds   \eqref{phi_I_est}, \eqref{bc_phI}, and  \eqref{est_n=0_J3_k}, 
\begin{align*}
{\bf a}_{2,0}\|\pa_y \mr\phe_{2,0;k}\|_{\dot H^2_k}
\lesssim&{\bf a}_{2,0} \|\mr\wwe_{2,0;k}\|_{\dot H_k^1}+\sum_{m=0}^1\sum_{b+c=2}{\bf a}_{m,0}\|\pa_{y}^b|k|^{c+1}\mr\phe_{m,0;k}\|_{L^2}\\
&+{\bf a}_{2,0}\|\pa_y \ci_{2,0}\|_{L^2}+{\bf a}_{2,0}|k|^{1/2}\lf||k|^2 v_{yy}\pav\phi_k^{(I)}\rg|\bigg|_{v=v(\pm1)}\\
\lesssim &\sum_{m=0}^2{\bf a}_{m,0}\|\mr\wwe_{m,0;k}\|_{\dot H_k^1}+\sum_{m=0}^2{\bf a}_{m,0}\|\ci_{m,0}\|_{H_k^1}+|v_{yy}|\big|_{y=\pm 1}\sum_{m=0}^2{\bf a}_{m,0}\lf\|\chi_{\ast}\pav\lf( |k|^{m+1/2}\phi_k^{(I)}\rg)\rg\|_\infty\\
\lesssim& \sum_{m=0}^2{\bf a}_{m,0}\|\mr\wwe_{m,0;k}\|_{\dot H_k^1}+\sum_{m=0}^2{\bf a}_{m,0}\|\ci_{m,0}\|_{H_k^1}+|v_{yy}|\big|_{y=\pm 1}\lf\|\oi_k\rg\|_{L_v^2}.
\end{align*}
Here $\chi_\ast$ is defined in \eqref{chi_ast}. 
Finally, we consider the case where $m\geq 3.$ In this case, the Gevrey coefficients starts to play major roles. We recall the fundamental relations discussed in the previous sections \eqref{gevbd1212}, and the estimates of the cutoff $\chi_{m+n}$ \eqref{chi:prop:3},  \begin{align*}
{\bf a}_{m,0}&\|\pa_y \mr\phe_{m,0;k}\|_{\dot H^2_k}
\\
\lesssim&{\bf a}_{m,0} \|\mr\wwe_{m,0;k}\|_{\dot H_k^1}+{\bf a}_{m,0}m^{3+3\sigma}\||k|^3\mr\phe_{m-3,0;k}\|_{L^2}+{\bf a}_{m,0}m^{2+2\sigma}\||k|^2\pa_{y}\mr\phe_{m-2,0;k}\|_{L^2}\\
&+{\bf a}_{m,0}m^{1+\sigma}\||k|\pa_{y}^2\mr\phe_{m-1,0;k}\|_{L^2}+{\bf a}_{m,0}\|\pa_y \ci_{m,0}\|_{L^2}+{\bf a}_{m,0}|k|^{1/2}\lf||k|^m v_{yy}\pav\phi_k^{(I)}\rg|\bigg|_{v=v(\pm1)}\\
\lesssim&{\bf a}_{m,0} \|\mr\wwe_{m,0;k}\|_{\dot H_k^1}+\lambda^{3s}{\bf a}_{m-3,0}m^{-3\sigma_\ast}\||k|^3\mr\phe_{m-3,0;k}\|_{L^2}+\lambda^{2s}{\bf a}_{m-2,0}m^{-2\sigma_\ast}\||k|^2\pa_{y}\mr\phe_{m-2,0;k}\|_{L^2}\\
&+\lambda^{s} {\bf a}_{m-1,0}m^{-\sigma_\ast}\||k|\pa_{y}^2\mr\phe_{m-1,0;k}\|_{L^2}+{\bf a}_{m,0}\|\pa_y \ci_{m,0}\|_{L^2}+{\bf a}_{m,0}|k|^{1/2}\lf||k|^m v_{yy}\pav\phi_k^{(I)}\rg|\bigg|_{v=v(\pm1)}\\
\lesssim &{\bf a}_{m,0}\|\mr\wwe_{m,0;k}\|_{\dot H_k^1}+\lambda^{s}\sum_{m'=m-3}^{m-1}\sum_{b+c=2}\|\pa_y^b|k|^{c+1}\mr\phe_{m',0;k}\|_{L^2}+{\bf a}_{m,0}\|\ci_{m,0}\|_{H_k^1}+|v_{yy}|\big|_{y=\pm 1}\|\oi_k\|_{L_v^2}.
\end{align*}
Now we sum over the estimates we obtained so far and invoke the bound \eqref{est_n=0_J3_k} to obtain
\begin{align*}
\sum_{m=0}^M\sum_{b+c=2}{\bf a}_{m,0}^2\|\pa_y^{b+1}|k|^c \mr\phe_{m,0;k}\|_{L^2}^2\lesssim \sum_{m=0}^M{\bf a}_{m,0}^2\|\mr\ww_{m,0;k}\|_{\dot H_k^1}^2+\sum_{m=0}^M{\bf a}_{m,0}^2\|\ci_{m,0}\|_{\dot H_k^1}^2+|v_{yy}|\big|_{y=\pm 1}\|\oi_k\|_{L_v^2}.
\end{align*}
Combining this estimates and \eqref{est_n=0_J3_k} yields \eqref{est_n=0_J3}.

\noindent
{\bf Step \# 3: Proof of \eqref{est_n=1_J2}.} Here we highlight that on the left hand sides of \eqref{est_n=1_J2}, there is an extra term $J_{m,1}^{(1,0,0)}\phe_k$ to be estimated. Hence, we further decompose the argument in two substeps: 

\noindent
{\bf Step \# 3a: }In this substep, we only consider the Sobolev component $\|J_{m,1}^{(0,b,c)}\phe_k\|_{L^2}^2$ on the left hand side of \eqref{est_n=1_J2}. First of all, we have the commutator relation
\begin{align}\n
[A,BCD]f=&ABCD f-BCDA f\\
\n=&ABCD f-BACDf+BACD f-BCAD f+BCADf-BCDA f\\
=&[A,B](CDf)+ B[A,C]Df+BC[A,D]f.\label{ABCD_rel}
\end{align}
By setting $A=\pa_{yy},\, B=\chi_{m+1},\,C=q,\, D=\Gamma_k$, we can derive the following equation for $\mr\phe_{m,1;k}$:
\begin{align}\ \n
\de_k\mr\phe_{m,1;k}=& \de_k(\chi_{m+1} |k|^{m}q\Gamma_k \phe_k)\\
\n =&\chi_{m+1} |k|^{m}q\Gamma_k\wt \ww_k^{(E)}+[\pa_{yy},\chi_{m+1} q\Gamma_k]|k|^{m}\phe_k+\chi_{m+1}|k|^{m}q\Gamma_k(\wch(\pav^2-\pa_y^2)\phi_k^{(I)}) \\
\n =&\chi_{m+1}|k|^{m} q\Gamma_k\wt \ww_k^{(E)}+[\pa_{yy},\chi_{m+1}](|k|^{m}q\Gamma_k\phe_k)+\chi_{m+1}[\pa_{yy},q](|k|^{m}\Gamma_k\phe_k)\\
\n &+\chi_{m+1}q[\pa_{yy},\Gamma_k] |k|^{m}\phe_k+\ci_{m,1},\\
\mr\phe_{m,1;k}(y=\pm1)=&\chi_{m+1}|k|^{m}q\Gamma_k\phe_k(y=\pm1)=0.\label{eq_lv2_y}
\end{align}
Thanks to the equivalence \eqref{Jel_equiv} and the elliptic estimate \eqref{mx_rg}, we have that 
\begin{align}\n
\sum_{b+c\leq 2}&{\bf a}_{m,1}^2\|J^{(0,b,c)}_{m,1}\phe_k\|_2^2\lesssim \sum_{b+c=2}{\bf a}_{m,1}^2\|\pa_y^b|k|^c\mr\phe_{m, 1;k}\|_{L^2}^2\\
\n \lesssim &{\bf a}_{m,1}^2\|\chi_{m+1}|k|^{m} q\Gamma_k\wt \ww_k^{(E)}\|_{L^2}^2+{\bf a}_{m,1}^2\|[\pa_{yy},\chi_{m+1}](|k|^{m}q\Gamma_k\phe_k)\|_{L^2}^2+{\bf a}_{m,1}^2\|\chi_{m+1}[\pa_{yy},q](|k|^{m}\Gamma_k\phe_k)\|_{L^2}\\
\n &\qquad+{\bf a}_{m,1}^2\|\chi_{m+1}q[\pa_{yy},\Gamma_k] |k|^{m}\phe_k\|_{L^2}^2+{\bf a}_{m,1}^2\|\ci_{m,1}\|_{L^2}^2\\
=:&\sum_{j=1}^5 T_{3j}.\label{n=1_J2}
\end{align}
Since the $T_{31},\ T_{35}$ terms are consistent with \eqref{est_n=1_J2}, we focus on the remaining terms and start from $T_{32}$. Thanks to a similar argument as in the estimate of $T_{12}$-term in \eqref{n=0_J2}, we distinguish between $m=0$ and $m\geq 1$ cases. If $m=0$, we apply the facts that ${\bf a}_{m,n}=B_{m,n}\varphi^{n+1}$ \eqref{a:weight:neq} and $t\varphi\leq C$ \eqref{ineq:varphi}, \begin{align*}
T_{32}\lesssim &{B}_{0,1}^2\varphi^{4}\|\pa_y (q(\pav+ikt)\phe_k)\|_{L_y^2}^2+B_{0,1}^2\varphi^4\| (\pav+ikt)\phe_k\|_{L_y^2}^2\\
\lesssim&\sum_{b+c=2}B_{0,0}^2\varphi^2\|\pa_y^b|k|\phe_k\|_{L^2}^2\lesssim \|\wwe_k\|_{L_y^2}^2+\lf\|\ci_{0,0}\rg\|_{L_y^2}^2.
\end{align*}
Here in the last line, we have invoked \eqref{J2_lv0}. On the other hand, if $m\geq 1$, we estimate the $T_{32}$ term with \eqref{s:prime}, the $\chi_{m+n}$-estimate \eqref{chi:prop:3}, the relation \eqref{gevbd1212} and the equivalence relation \eqref{Jel_equiv}, as follows:\begin{align*}
T_{32}\lesssim& {\bf a}_{m,1}^2\|\chi_{m+1}''|k|q(\pav+ikt)(\chi_{m-1}|k|^{m-1}\phe_k)\|_{L^2}^2+{\bf a}_{m,1}^2\|\chi_{m+1}'\pa_y( q(\pav+ikt)(\chi_{m}|k|^{m}\phe_k))\|_{L^2}^2\\
\lesssim& \lambda^{4s}{\bf a}_{m-1,0}^2 (t\varphi )^2\sum_{b+c=2}\|\pa_y^b|k|^c(\chi_{m-1}|k|^{m-1}\phe_k)\|_{L^2}^2+\lambda^{2s}{\bf a}_{m,0}^2(t\varphi )^2\sum_{b+c=2}\|\pa_y^b|k|^c(\chi_{m}|k|^{m}\phe_k)\|_{L^2}^2\\
\lesssim&\lambda^{2s} \sum_{m'=m-1}^{m}\sum_{b+c=2}{\bf a}_{m',0}^2\|J^{(0,b,c)}_{m',0}\phe_k\|_{L^2}^2.
\end{align*}
The $T_{33}$ term in \eqref{n=1_J2} can be estimated as follows
\begin{align*}T_{33}\lesssim &{\bf a}_{m,1}^2\|\chi_{m+1}\pav(\pav+ikt)(\chi_m|k|^m\phe_k)\|_{L^2}^2  
\lesssim \lambda^{2s}{\bf a}_{m,0}^2\varphi^2(\|J^{(0,2,0)}_{m,0}\phe_{k}\|_{L^2}^2+t^2\|J^{(0,1,1)}_{m,0}\phe_k\|_{L^2}^2)\\
\lesssim &\lambda^{2s}{\bf a}_{m,0}^2\sum_{b+c=2}\|J^{(0,b,c)}_{m,0}\phe_k\|_{L^2}^2.
\end{align*}
Thanks to the commutator relation \eqref{cm_pyy_G_n}, we have that the $T_{34}$ term can be estimated as follows:
\begin{align*}
T_{34}\lesssim & {\bf a}_{m,1}^2\lf(\|\chi_{m+1}q v_{yy}\pav^2|k|^m\phe_k\|_{L^2}^2+\|\chi_{m+1}q\vyn v_{yyy}\pav|k|^m\phe_k\|_{L^2}^2\rg)\\
\lesssim&\lambda^{2s}{\bf a}_{m,0}^2(m+1)^{-2-2\sigma-2\sigma_\ast}\lf(\| \pav^2(\chi_m|k|^m\phe_k)\|_{L^2}^2+\|\pav(\chi_m|k|^m\phe_k)\|_{L^2}^2\rg)\\
\lesssim&\lambda^{2s}{\bf a}_{m,0}^2\lf(\| J^{(0,2,0)}_{m,0} \phe_k\|_{L^2}^2+\|J^{(0,1,1)}_{m,0}\phe_k\|_{L^2}^2\rg).
\end{align*}
Now we sum from $m=0$ to $M-1$ and invoke \eqref{est_n=0_J2} to obtain \eqref{est_n=0_J2}.

\noindent
{\bf Step \# 3b: }\siming{This step can be omitted if we make sure that \eqref{JtoJ_2} is true.} In this substep, we consider the estimate of the quantity $J_{m,1}^{(1,0,0)}\phe_k$. We observe that 
\begin{align*}
J_{m,1}^{(1,0,0)}\phe_k=\frac{m+1}{q} \chi_{m+1}|k|^mq(\pav+ikt)\phe_k=(m+1)\chi_{m+1}(\pav+ikt)(\chi_{m}|k|^m\phe_k).
\end{align*}
Hence the estimate is a direct consequence of \eqref{est_n=0_J2},
\begin{align*}
\sum_{m=0}^{M-1} &{\bf a}_{m,1}^2\lf\|J_{m,1}^{(1,0,0)}\phe_k\rg\|_{L^2}^2\lesssim \sum_{m=0}^{M-1}{\bf a}_{m,1}^2(m+1)^2 \lan t\ran^2\sum_{b+c=1}\|\pav^b|k|^c\mr\phe_{m,0;k}\|_{L^2}^2\\
\lesssim& \lambda^{2s}\sum_{m=0}^{M-1}{\bf a}_{m,0}^2\sum_{b+c=1}\|\pav^b|k|^c\mr\phe_{m,0;k}\|_{L^2}^2\lesssim\sum_{m=0}^{M-1} {\bf a}_{m,0}^2\| \mr\wwe_{m,0;k} \|_{L^2}^2+\sum_{m=0}^{M-1}{\bf a}_{m,0}^2\|\ci_{m,0}\|_{L_y^2}^2. 
\end{align*}

This concludes the proof of \eqref{est_n=1_J2}.

\noindent
{\bf Step \# 4: Proof of \eqref{est_n=1_J3}.}
Since $a+b+c=3\geq 1$ in this case, we only needs to consider classical derivative estimates. 

From \eqref{est_n=1_J2}, we have that
\begin{align}
\sum_{m=0}^{M-1}\sum_{b+c=2}{\bf a}_{m,1}^2\|\pav^b|k|^{c+1}J_{m,1}^{(0)}\phe_{k}\|_{L^2}^2\lesssim&\sum_{m+n=0}^{M}\mathbbm{1}_{n\in\{0,1\}}{\bf a}_{m,n}^2\| |k|J_{m,n}^{(0)}\wwe_{k} \|_{L^2}^2+\sum_{m+n=0}^{M} \mathbbm{1}_{n\in\{0,1\}}{\bf a}_{m,n}^2\|\ci_{m,n}\|_{H_k^1}^2.\label{est_n=1_J3_k}
\end{align}
 We observe that the $\pa_y(\mr\phe_{m,1;k})=\pa_y(\chi_{m+1} |k|^mq( \pav +ikt )\phe_k)$ solves the following equation:
\begin{align}
\de_k\pa_y\mr\phe_{m,1;k}=&\pa_y \mr\wwe_{m,1;k}+\pa_y[\pa_{yy},\chi_{m+1} q\Gamma_k](\chi_m|k|^m\phe_k)+\pa_y\ci_{m,1}.
\end{align} 
Furthermore, we have that 
\begin{align*}
 \pa_y\mr\phe_{m,1;k} \big|_{y=\pm 1}
=&  {v_y^{-1}}\pa_y|k|^m\phe_k \big|_{y=\pm1}.
\end{align*}
Hence the elliptic estimate \eqref{mx_rg_b2} yields that
\begin{align}\n
{\bf a}_{m,1}^2 \sum_{b+c= 2}\|\pa_y^{b+1}|k|^c\mr\phe_{m,1;k}\|_{L^2}^2\lesssim&{\bf a}_{m,1}^2\|\pa_y\mr\wwe_{m,1;k}\|_{L^2}^2+{\bf a}_{m,1}^2\|\pa_y[\pa_{yy}, \chi_{m+1}q\Gamma](\chi_m|k|^m\phe_k)\|_{L^2}^2\\
 \n &+{\bf a}_{m,1}^2\|\pa_y \ci_{m,1}\|_{L^2}^2+{\bf a}_{m,1}^2|k|^{3}|\pav(|k|^m\phe_k)(y=\pm 1)|^2\\
=:&\sum_{j=1}^4 T_{4j}.\label{n=1_J3}
\end{align}
Since the $T_{41}$ and $T_{43}$ are consistent with \eqref{est_n=1_J3}, we focus on $T_{42}$ and $T_{44}$. We distinguish between the $m=0$, $m=1$ and the $m\geq2$ case. 

\noindent
{\bf Step \# 4a: the $m=0$ case.}

For the $T_{42}$ term in \eqref{n=1_J3}, we explicitly write it out as follows, 
\begin{align*}\n
T_{42}\lesssim&{\bf a}_{0,1}^2\lf\|\pa_y\lf(\pa_{yy}(\chi_1q\vyn)\pa_y\phe_k+2\pa_y (\chi_1 q\vyn)\pa_{yy}\phe_k+\pa_{yy}(q\chi_1)ikt\phe_k+2\pa_y(q\chi_1)ikt\pa_y\phe_k\rg)\rg\|_{L^2}^2\\
\lesssim&{\bf a}_{0,1}^2 \lan t\ran^2\sum_{b+c=3}\||k|^b\pa_y^c\phe_k\|_{L^2}^2
\lesssim {\bf a}_{0,0}^2\|\wwe_k\|_{\dot H_k^1}^2+|v_{yy}|^2\big|_{y=\pm 1}\|\wt\ww_k^{(I)}\|_{L_v^2}^2+{\bf a}_{0,0}^2\|\ci_{0,0}\|_{H_k^1}^2.
\end{align*} 
This is consistent with \eqref{est_n=1_J3}. 
To estimate the $T_{44}$ term, we invoke the Gagliardo-Nirenberg interpolation inequality to obtain that 
\begin{align}
|\pa_y\phe_k(y=\pm 1)|^2\lesssim \|\pa_y \phe_k\|_{L^2}(\|\pa_{yy}\phe_k\|_{L^2}+\|\pa_y\phe_k\|_{L^2}).\label{bc_bd}
\end{align}
Hence, we invoke \eqref{est_n=0_J3_k} to obtain that 
\begin{align}\n
T_{44}\lesssim &{\bf a}_{0,1}^2|k|^{3}\|\pa_y \phe_k\|_2(\|\pa_{yy}\phe_k\|_2+\|\pa_y\phe_k\|_2)\\ \n
 \lesssim&{\bf a}_{0,1}^2\||k|^2\pa_y\phe_k\|_2(\||k|\pa_{yy}\phe_k\|_2+\||k|\pa_{y}\phe_k\|_2)\\
\lesssim&{\bf a}_{0,0}^2\|\wt\ww_k^{(E)}\|_{\dot H_k^1}^2+|v_{yy}|\big|_{y=\pm 1}\|\wt\ww_k^{(I)}\|_{L_v^2}^2+{\bf a}_{0,0}^2\|\ci_{0,0}\|_{H^1}^2.
\end{align}
This is consistent with \eqref{est_n=1_J3}.  Combining the estimates above, we conclude {\bf Step \# 4a}. 

\noindent
{\bf Step \# 4b: the $m=1$ case.} The treatment is similar to {\bf Step \# 4a} and we omit the details. \siming{(Check?!)}

\noindent
{\bf Step \# 4c: the $m\geq2$ case.}
 For the $T_{42}$ term, we apply the commutator relation \eqref{ABCD_rel} to decompose it as follows
\begin{align}\n
T_{42}=& {\bf a}_{m,1}^2\|\pa_y[\pa_{yy}, \chi_{m+1}q\Gamma](\chi_m|k|^m\phe_k)\|_{L^2}^2\\ \n
\lesssim&{\bf a}_{m,1}^2\bigg(\|\pa_y[\pa_{yy},\chi_{m+1}](\chi_m|k|^{m}q\Gamma_k\phe_k)\|_{L^2}^2+\|\pa_y(\chi_{m+1}[\pa_{yy},q](\chi_m|k|^{m}\Gamma_k\phe_k))\|_{L^2}\\
& +\|\pa_y(\chi_{m+1}q[\pa_{yy},\Gamma_k] (\chi_m|k|^{m}\phe_k)\|_{L^2}^2\bigg)=:\sum_{j=1}^3T_{42j}\label{T42_dcmp}.
\end{align}
To estimate the $T_{21}^{(4)}$-term in \eqref{T42_dcmp}, we invoke the relation \eqref{s:prime} and the estimate \eqref{chi:prop:3}, together with the definition of $J^{(a,b,c)}_{m,n}\phe_k$ \eqref{J_nq} as follows  
\begin{align*}\n
T_{421}\lesssim&{\bf a}_{m,1}^2(m+1)^{6+6\sigma}\|\chi_{m}|k|^{m}q\Gamma_k\phe_k\|_{L^2}^2+{\bf a}_{m,1}^2 (m+1)^{4+4\sigma}\|\pa_y(\chi_{m}|k|^{m}q\Gamma_k\phe_k)\|_{L^2}^2\\
& +{\bf a}_{m,1}^2 (m+1)^{2+2\sigma}\|\pa_{yy}(\chi_{m}|k|^{m}q\Gamma_k\phe_k)\|_{L^2}^2\\
\lesssim&{\bf a}_{m-2,1}^2\lambda^{4s}(m+1)^{-4\sigma}\lf\|\frac{m+1}{q}|k|^2\lf(\chi_{m-2}|k|^{m-2}q\Gamma_k\phe_k\rg)\rg\|_{L^2}^2\\
&+{\bf a}_{m-1,1}^2 \lambda^{2s}(m+1)^{-4\sigma}\lf\|\frac{m+1}{q}\pav|k|(\chi_{m-1}|k|^{m-1}q\Gamma_k\phe_k)\rg\|_{L^2}^2\\
& +{\bf a}_{m-1,1}^2 \lambda^{2s}(m+1)^{-2\sigma}\||k|\pa_{yy}(\chi_{m}|k|^{m-1}q\Gamma_k\phe_k)\|_{L^2}^2\\
\lesssim&{\bf a}_{m-2,1}^2\lambda^{4s} \lf\|J^{(1,0,2)}_{m-2,1}\phe_k\rg\|_{L^2}^2+{\bf a}_{m-1,1}^2 \lambda^{2s}\lf\|J^{(1,1,1)}_{m-1,1}\phe_k\rg\|_{L^2}^2 +{\bf a}_{m-1,1}^2 \lambda^{2s}\|J^{(0,2,1)}_{m-1,1}\phe_k\|_{L^2}^2.
\end{align*}
Next we recall that on the support of $\chi_{m+1}$, $\chi_{m}=\chi_{m-1}\equiv1$ and estimate the $T_{422}$-term as follows
\begin{align*}
T_{422}\lesssim& {\bf a}_{m,1}^2\lf(\|\pa_y(\chi_{m+1}q''(\chi_m|k|^{m}\Gamma_k\phe_k))\|_{L^2}^2+\|\pa_y(\chi_{m+1}q'\pa_y(\chi_m|k|^{m}\Gamma_k\phe_k))\|_{L^2}^2\rg)\\
\lesssim&{\bf a}_{m,1}^2\bigg(\|\pa_y(\chi_{m+1}q'')(\chi_m|k|^{m}\Gamma_k\phe_k)\|_{L^2}^2+\|\chi_{m+1}q''\pa_y(\chi_m|k|^{m}\Gamma_k\phe_k)\|_{L^2}^2\\
&\qquad+\|\pa_y(\chi_{m+1}q')\pa_y(\chi_m|k|^{m}\Gamma_k\phe_k)\|_{L^2}^2+ \| \chi_{m+1}q'\pa_y^2(\chi_m|k|^{m}\Gamma_k\phe_k)\|_{L^2}^2\bigg)\\
\lesssim&{\bf a}_{m,1}^2\bigg((m+1)^{2+2\sigma}\|\Gamma_k(\chi_m|k|^{m}\phe_k)\|_{L^2}^2+\|\pav\Gamma_k(\chi_{m}|k|^{m }\phe_k)\|_{L^2}^2\\
&\qquad+(m+1)^{2+2\sigma} \|\pav\Gamma_k(\chi_{m}|k|^{m}\phe_k)\|_{L^2}^2+ \|\pa_y^2\Gamma_k(\chi_{m}|k|^{m}\phe_k)\|_{L^2}^2\bigg)\\
\lesssim&{\bf a}_{m,0}^2\lambda^{2s}(m+1)^{-2\sigma_\ast}\bigg(\sum_{b+c=2}\|J^{(0,b,c)}_{m,0}\phe_k\|_{L^2}^2+\sum_{b+c=3}\|J^{(0,b,c)}_{m,0} \phe_k\|_{L^2}^2\bigg).
\end{align*}
An application of the commutator relation  \eqref{cm_pyy_G_n} and the estimates  yields that the $T_{423}$ term is bounded as follows
\begin{align*}
T_{423}\lesssim&{\bf a}_{m,1}^2\big(\|\pa_y(\chi_{m+1}q \pa_{yy}(\vyn) (\chi_m|k|^{m}\phe_k))\|_{L^2}^2+\|\pa_y(\chi_{m+1}q v_y \pa_{y}\vyn \pav(\chi_m|k|^{m}\phe_k))\|_{L^2}^2\big)\\
\lesssim&{\bf a}_{m,1}^2\big(\|\pa_y(\chi_{m+1}q \pa_{yy}(\vyn))\ \chi_m|k|^{m}\phe_k\|_{L^2}^2+\|\chi_{m+1}q \pa_{yy}(\vyn)\pa_y( \chi_m|k|^{m}\phe_k)\|_{L^2}^2\\
&\quad+\|\pa_y(\chi_{m+1}q v_y \pa_{y}\vyn )\ \pav(\chi_m|k|^{m}\phe_k)\|_{L^2}^2+\| \chi_{m+1}q v_y^2 \pa_{y}\vyn \pav^2(\chi_m|k|^{m}\phe_k)\|_{L^2}^2\big)\\
\lesssim&{\bf a}_{m,1}^2\big((m+1)^{2+2\sigma}\|\chi_m|k|^{m}\phe_k\|_{L^2}^2+\|\pav(\chi_m|k|^{m}\phe_k)\|_{L^2}^2\\
&\qquad\qquad+(m+1)^{2+2\sigma}\|\pav(\chi_m|k|^{m}\phe_k)\|_{L^2}^2+\|\pav^2(\chi_m|k|^{m}\phe_k)\|_{L^2}^2\big)\\
\lesssim&{\bf a}_{m,0}^2(m+1)^{-2\sigma_\ast}\lambda^{2s}\sum_{b+c=2}\|J^{(0,b,c)}_{m,0}\phe_k\|_{L^2}^2.
\end{align*}
Combining the estimates of $T_{421},\, T_{422},\, T_{423}$ together with the decomposition \eqref{T42_dcmp}, we obtain that
\begin{align}
T_{42}\lesssim \lambda^{2s}{\bf a}_{m-2,1}^2\|J^{(1,0,2)}_{m-2,1}\phe_k\|_{L_y^2}^2+\lambda^{2s}\sum_{b+c=3}{\bf a}_{m-1,1}^2\|J^{(0,b,c)}_{m-1,1}\phe_k\|_2^2+\lambda^{2s}\sum_{b+c=3}{\bf a}_{m,0}^2\|J^{(0,b,c)}_{m,0}\phe_k\|_2^2.\label{T42_est}
\end{align}
Next we observe that the $T_{43}$-term is consistent with the final bound. 
We focus on the $T_{44}$-term. By \eqref{bc_bd},
\begin{align}
|\pav(|k|^m\phe_k)(y=\pm 1)|^2\lesssim \|\pa_y (\chi_m|k|^m\phe_k)\|_{L^2}(\|\pa_{yy}(\chi_m|k|^m\phe_k)\|_{L^2}+\|\pa_y(\chi_{m}|k|^m\phe_k)\|_{L^2}).
\end{align}
Hence,
\begin{align}\n
T_{44}\lesssim& {\bf a}_{m,1}^2\||k|^{2}\pa_y (\chi_m|k|^m\phe_k)\|_{L^2}(\||k|\pa_{yy}(\chi_m|k|^m\phe_k)\|_{L^2}+\||k|\pa_y(\chi_{m}|k|^m\phe_k)\|_{L^2})\\
\lesssim&\lambda^{2s} {\bf a}_{m,0}^2\sum_{b+c\leq 3}\|J^{(0,b,c)}_{m,0}\phe_k\|_{L^2}^2.\label{T44_est}
\end{align}
Combining the estimates \eqref{T42_est}, \eqref{T44_est} we developed thus far and the decomposition \eqref{n=1_J3},  summing up in $m$ and taking $\lambda$ to be small yields the result $\eqref{est_n=1_J3}$.
\siming{\myr{??}} 

\end{proof}
We also need the following lemma.
\begin{lemma}
Assume the bootstrap assumptions. Consider the solution to \eqref{d:phi:E:in}. The following estimates hold,
\begin{align}\sum_{a+b+c\leq 2}{\bf a}_{0,2}^2\|J^{(a,b,c)}_{0,2}\phe_k\|_{L_y^2}^2  
\lesssim &\sum_{m+n=0}^{2}{\bf a}_{m,n}^2\lf(\|J_{m,n}^{(0)}\wt \ww ^{(E)}_{k}\|_{L_y^2}^2 +\|\mathbb{C}_{m,n}^{(I)}\|_{L_y^2}^2\rg) ; \label{est_J2}\\ 
\n\sum_{a+b+c=3}{\bf a}_{0,2}^2\|J^{(a,b,c)}_{0,2}\phe_k\|_{L_y^2}^2 
\lesssim&\sum_{m+n=0}^{2}{\bf a}_{m,n}^2\lf(\|J_{m,n}^{(0)}\wt \ww ^{(E)}_{k}\|_{H_{k}^1}^2+\|\mathbb{C}_{m,n}^{(I)}\|_{H^1_{k}}^2\rg) + |v_{yy}|\big|_{y=\pm 1} \|\wt\ww^{(I)}_k\|_{L_{v}^2}^2\\
&+ \lf(\sum_{n=0}^2\sum_{b=1}^2\varphi^n\|J^{(0,b,0)}_{0,n}(v_y^2-1)\|_{L^2}\rg) {\bf a}_{0,0}^2\lf(\| \mr\wwe_{0,0;k} \|_{L^2}^2+\|\ci_{0,0}\|_{L_y^2}^2\rg).  \label{est_J3}%
\end{align}
\siming{\footnote{\myr{The estimate of $J_{0,3}^{(3)}$ is not okay because the implicit constant might depend on $\|v_{yyyyy}\|_{L^2}$. So it is better to use the method in the next section to prove it. }}}
Here the $\mathbb{C}_{m,n}^{(I)}$ are defined in \eqref{CImn}. 
\ifx 
Last but not least, the following $H^4$-estimate holds 
\begin{align} \sum_{ b+c= 4} \|\pa_y^b|k|^c\phe_k\|_{L_y^2}^2\lesssim \|\wt \ww ^{(E)}_{k}\|_{H_{k}^2}^2+\siming{\lf(|v_{yy}|^2\big|_{y=\pm1}+\|\wt\chi_1Z\|_{H_y^2}^2+\|\wt\chi_1\pav Z\|_{H_y^2}^2\rg) }\|\wt\ww^{(I)}_k\|_{L_{v}^2}^2. \label{est_J4}
\end{align}
\fi  
Here the implicit constant depends on the norms $\|v_{y}\|_{L_t^\infty H_y^{3}},\,\|\vyn\|_{L_{t,y}^\infty}$. \siming{From this estimate we observe that one needs $\|v_{yyyy}\|_{L^2}.$}
\end{lemma}
\begin{proof}
We organize the proof in steps.

\noindent
\textbf{Step \# 1: Proof of \eqref{est_J2}.}  
\siming{\bf This part can be omitted once we make sure that \eqref{JtoJ_2} is true.} \textbf{Step \# 1a: $a\neq 0$ case.}
This case can only happen with $n=2$ thanks to the definition of $J_{m,n}^{(a,b,c)}$. There are three sub-cases to consider, i.e., $J_{0,2}^{(2,0,0)}\phe_k,\, J_{0,2}^{(1,1,0)}\phe_k,\, J_{0,2}^{(1,0,1)}\phe_k$. We observe that $J_{0,2}^{(2,0,0)}\phe_k$ is fairly direct to estimate with  $\eqref{est_J2}$ and ${\bf a}_{0,2}\lesssim \varphi^3$, i.e.,
\begin{align}\n 
{\bf a}_{0,2}&\|J_{0,2}^{(2,0,0)}\phe_k\|_2\lesssim\varphi^3\|\chi_2 \Gamma_k^2\phe_k\|_2
\lesssim \varphi^3\lf(\|\pa_{yy}\phe_k\|_2+\|\pa_y \phe_k\|_2+t\||k|\pa_y\phe_k\|_2+t^2\||k|^2\phe_k\|_2\rg)\\
\lesssim &{\bf a}_{0,0}\|\wwe_k\|_{L_y^2}+{\bf a}_{0,0}\|\ci_{0,0}\|_{L^2}. \label{J_200_02}
\end{align}
Here we first apply the bound $\|\vyn\|_\infty+\|v_{yy}\|_\infty\leq C$ and then the bound $t\varphi\lesssim 1$ \eqref{ineq:varphi}. 

For the $J_{0,2}^{(1,1,0)}\phe_k$ term, we estimate it as follows, 
\begin{align}\n
\n {\bf a}_{0,2}\|&J_{0,2}^{(1,1,0)}\phe_k\|_2\lesssim\varphi^3\|q^{-1}\pav (\chi_2q^2\Gamma_k^2\phe_k)\|_2\\
\lesssim&\varphi^3\lf(\|\Gamma_k^2 \phe_k\chi_2\|_2+\|q\Gamma_k^2\phe_k\chi_2'\|_2+\|q\pav(\Gamma_k^2\phe_k)\chi_2\|_2\rg)
=:T_{11}+T_{12}+T_{13}. \label{J110_02}
\end{align}
The $T_{11},\ T_{12}$ terms can be estimated as in \eqref{J_200_02}. By invoking $t\varphi\lesssim1$, we end up with the result 
\begin{align*}
T_{11}+T_{12}\lesssim{\bf a}_{0,0}(\|\wwe_k\|_{L_y^2}+\|\ci_{0,0}\|_{L^2}).
\end{align*}
Hence we focus on the last term $T_{13}$ in \eqref{J110_02}. By applying $\eqref{est_n=0_J2}_{M=0}$, $\eqref{est_n=1_J2}_{M=1}$ and the commutator estimates \eqref{cm_qn_pavj}, \eqref{wtq}, we have that
\begin{align*}\n 
T_{13}\lesssim& \varphi^3\|\chi_2[q,\pav\Gamma_k](\Gamma_k\phe_k)\|_2+\varphi^3\|\chi_2\pav\Gamma_k(\chi_1 q\Gamma_k\phe_k) \|_2 \\
\lesssim& {\bf a}_{0,0}\varphi^2\lf(\|[q,\pav^2 ]\Gamma_k\phe_k\|_2+|k|t\|[q,\pav]\Gamma_k\phe_k\|_2\rg)+\varphi {\bf a}_{0,1}\lf(\|\pav^2 \mr \phe_{0,1;k} \|_2+ t\|k\pav\mr \phe_{0,1;k} \|_{2}\rg)\\
\lesssim&{\bf a}_{0,0}\varphi^2\sum_{\ell=1}^2\|\pav^{2-\ell}\Gamma_k\phe_k\|_{2}+{\bf a}_{0,0}\varphi \||k|\Gamma_k\phe_k\|_2+\sum_{m+n\leq 1}{\bf a}_{m,n}\lf(\|J_{m,n}^{(0)}\wwe_{k}\|_2+\|\ci_{m,n}\|_{L_y^2}\rg)\\
\lesssim&\sum_{m+n\leq 1}{\bf a}_{m,n}\lf(\|J_{m,n}^{(0)}\wwe_{k}\|_2+\|\ci_{m,n}\|_{L_y^2}\rg). 
\end{align*}
To conclude, we have that 
\begin{align}
 {\bf a}_{0,2}\|J_{0,2}^{(1,1,0)}\phe_k\|_{L_y^2}\lesssim\sum_{m+n\leq 1}{\bf a}_{m,n}\lf(\|J_{m,n}^{(0)}\wwe_{k}\|_2+\|\ci_{m,n}\|_{L_y^2}\rg). \label{J110_02est}
\end{align}

We estimate the $J_{0,2}^{(1,0,1)}\phe_k$-term in a similar fashion. We observe that
\begin{align}\n
{\bf a}_{0,2}\|J_{0,2}^{(1,0,1)}\phe_k\|_{L_y^2}\lesssim&\varphi^3\|\chi_2 q |k|\Gamma_k^2\phe_k\|_{L_y^2}\\
\n \lesssim & \varphi^3\lf(\| |k|\Gamma_k(\chi_{1}q\Gamma_k\phe_k)\|_{L_y^2}+\||k| [q,\pav](\chi_1 \Gamma_k\phe_k)\|_{L_y^2}\rg)\\ \n
\lesssim&\varphi^3 \||k|\pav\mr\phe_{0,1;k}\|_{L_y^2}+\varphi^2 \||k|^2\mr\phe_{0,1;k}\|_{L_y^2}+\varphi^2\||k|\pav\phe_k\|_{L_y^2}+\varphi\||k|^2\phe_k\|_{L_y^2}\\
\lesssim &{\bf a}_{0,1}  \|\mr\wwe_{0,1;k}\|_{L_y^2}+{\bf a}_{0,0} \|\wwe_{k}\|_{L_y^2}+\sum_{m+n\leq 1}{\bf a}_{m,n}\|\ci_{m,n}\|_{L_y^2}. \label{J_101_02} 
\end{align}
This is consistent with \eqref{est_J2} and concludes the treatment. 


\noindent 
{\bf Step \# 1b: $a=0$ case.} Thanks to the definition, it is enough to consider the terms $J^{(0,b,c)}_{0,2}\phe_k,\,\, b+c=2$. Recall the equation $\eqref{T123}_{m=0,n=2\}}$ with the Dirichlet boundary condition 
\begin{align*}
\chi_2  q^2\Gamma_k^2\phe_k\big|_{y=\pm 1}=0.
\end{align*} Thanks to the relation \eqref{Jel_equiv}, it is enough to estimate $\|\mr\phe_{0,2;k}\|_{\dot H_k^2}$. 
The following estimate is a direct consequence of Lemma \ref{lem:max_reg}, 
\begin{align}
{\bf a}_{0,2} \|\chi_2q^2\Gamma_k^2 \phe_k\|_{\dot H_k^2}
\lesssim& \varphi^3 \|\mr\wwe_{0,2;k}\|_{L^2}+\varphi^3\|\ci_{0,2}\|_{L^2}+\sum_{j=1}^3\varphi^3 \|\mathbb{T}_{j,2}\|_{L^2} 
=: \sum_{\ell=1}^5\mathcal{T}_{1\ell}.\label{mxrg_lv2_I}
\end{align}
Here the $\mathbb{T}$-terms are defined in \eqref{T_i}. The $\mathcal{T}_{11}, \ \mathcal{T}_{12}$ terms are consistent with the estimate \eqref{est_J2}. Hence we focus on the rest of the terms.
To begin with, we expand the $\mathcal{T}_{13}$ term with \eqref{T_1} to obtain the following estimate
\begin{align}
\mathcal{T}_{13} \lesssim& \varphi^{3} \sum_{\ell=0}^{1}\lf(\|\chi_2 q^2 \pav^{2-\ell} (v_{y}^2-1)\ \pav^2\Gamma^\ell\phe_k\|_{L^2}+\|\chi_2 q^2\pav^{3-\ell} (v_{y}^2-1)\ \pav\Gamma_k^\ell\phe_k\|_{L^2}\rg).
\end{align}
Application of the estimates $t\varphi\leq C$ \eqref{ineq:varphi}, $\|\vyn\|_{L^\infty}+\|v_{y}\|_{W^{2,\infty}}+\|v_y\|_{H^{3}} \leq C$,  \eqref{wtq} and \eqref{cm_qn_pavj} yields that 
\begin{align*}
\mathcal{T}_{13} \lesssim &\varphi^{3} \sum_{\ell=0}^{1}\lf(\|\chi_2 q^{2-\ell} \pav^{2-\ell} (v_{y}^2-1)\ \pav^2(q^{\ell}\Gamma^\ell\phe_k)\|_{L^2}+\|\chi_2 q^2\pav^{3-\ell} (v_{y}^2-1)\ \pav\Gamma_k^\ell\phe_k\|_{L^2}\rg)\\
\lesssim&\varphi^{3} \sum_{\ell=0}^{1}\lf(\|\chi_2 q^{2-\ell} \pav^{2-\ell} (v_{y}^2-1)\ \pav^2(q^{\ell}\Gamma^\ell\phe_k)\|_{L^2}+\|\chi_2 q^{2-\ell} \pav^{2-\ell} (v_{y}^2-1)\ [\pav^2,q^{\ell}]\Gamma^\ell\phe_k\|_{L^2}\rg.\\
&\qquad\qquad\lf.+\|\chi_2 q^{2-\ell}\pav^{3-\ell} (v_{y}^2-1)\ \pav(q^\ell\Gamma_k^\ell\phe_k)\|_{L^2}+\|\chi_2 q^{2-\ell}\pav^{3-\ell} (v_{y}^2-1)\ [\pav,q^\ell]\Gamma_k^\ell\phe_k\|_{L^2}\rg)\\
\lesssim&\varphi^{3-\ell} \sum_{\ell=0}^{1}\|\pav^{2-\ell}(v_y^2-1)\|_{L^\infty}\ \varphi^\ell\| \pav^2(\chi_{\ell}q^{\ell}\Gamma^\ell\phe_k)\|_{L^2}+\varphi^{2}\sum_{\ell=0}^1\|\pav^{2-\ell}(v_y^2-1)\|_{L^\infty}\varphi\|\pav^\ell\Gamma\phe_k\|_{L^2}\\
&+\varphi^2\|\pav^{2} (v_{y}^2-1)\|_{L^\infty}\ \varphi\| \pav(\chi_1 q\Gamma_k\phe_k)\|_{L^2}+\varphi^2\|\pav^{2} (v_{y}^2-1)\|_{L^\infty}\ \varphi \|\Gamma_k\phe_k\|_{L^2}\\
&+\varphi^3\|\pav^{3} (v_{y}^2-1)\|_{L^2}\| \pav\phe_k\|_{L^\infty}.
\end{align*} 
Now we apply the Gagliardo-Nirenberg interpolation, the bound \eqref{est_n=0_J2} to obtain 
\begin{align}
\mathcal{T}_{13}\lesssim \sum_{m+n=0}^1{\bf a}_{m,n}( \| \wwe _{m,n}\|_{L_y^2}+\|\ci_{m,n}\|_{L_y^2}).\label{mxlv2I2}
\end{align} 
For the $ \mathcal{T}_{14}$ and $ \mathcal{T}_{15}$ terms in \eqref{mxrg_lv2_I}, we  implement similar arguments to get the result, 
\begin{align} 
\n \mathcal T_{ 14}=&\varphi^3\|\mathbb{T}_{2,2}\|_2=\varphi^3\|\chi_2(2 (q')^2+2q  q''+4q q'\pa_y)\Gamma_k^2\phe_k\|_{L^2}\\
\n \lesssim&
\varphi^3\lf\| \chi_2  \Gamma_k^2\phe_k\rg\|_2+
\varphi^3 \lf\|\chi_2  q \Gamma_k^2\phe_k \rg\|_2\\
\n&+ \varphi^3 \lf\|   \pav\Gamma_k(\chi_{ 1}q\Gamma_k \phe_k)\rg\|_2+\varphi^3 \lf\|  [q, \pav\Gamma_k](\chi_{ 1}\Gamma_k \phe_k)\rg\|_2\\ 
\lesssim& \varphi^2 \|\chi_{1}q \Gamma _k\phe_k\|_{\dot H_k^2}+ \varphi\|\phe_k\|_{\dot H_k^2}\lesssim \sum_{n=0}^1{\bf a}_{0,\ell} \|\mr\wwe_{0,\ell;k}\|_{L_y^2}+\sum_{m+n\leq 1}{\bf a}_{m,n}\|\ci_{m,n}\|_{L^2_y};\label{mxlv2I3}\\
\n \mathcal{T}_{15}=&\varphi^3\|\mathbb{T}_{3,2}\|_2=\varphi^3\|(2\chi_2'\pa_y+\chi_2'') q^2\Gamma_k^2 \phe_k\|_2\\
 \n\lesssim&\varphi^3\| \chi_2'\pa_y ( q \Gamma_k(\chi_1 q\Gamma_k  \phi_k))\|_2+\varphi^3\| \chi_2'\pa_y (q [ q,\Gamma_k](\chi_1 \Gamma_k  \phe_k))\|_2\\
 \n&+\varphi^3\| \chi_2''( q \Gamma_k(\chi_1 q\Gamma_k  \phi_k))\|_2+\varphi^3\| \chi_2''( [q, \Gamma_k](\chi_1 q\Gamma_k  \phe_k))\|_2\\
\lesssim&\varphi^2\|\chi_{1}q \Gamma _k\phe_k\|_{\dot H_k^2}\lesssim \sum_{\ell=0}^1{\bf a}_{0,\ell} \|\mr\wwe_{0,\ell;k}\|_{L_y^2}+\sum_{m+n\leq 1}{\bf a}_{m,n}\|\ci_{m,n}\|_{L_y^2}.\label{mxlv2I4}
\end{align} 
Combining the estimates above yields the estimate 
\begin{align}
{\bf a}_{0,2}^2\|\chi_2q^2\Gamma_k^2\phe_k\|_{\dot H_k^2}^2\lesssim \sum_{m+n\leq 2}{\bf a}_{m,n}^2\|\wwe_{m,n} \|_{L_y^2}^2+\sum_{m+n\leq 2}{\bf a}_{m,n}^2\|\ci_{m,n}\|_{L_y^2}^2.
\end{align}
This concludes the proof of $\eqref{est_J2}$. 

\siming{It seems that we skip various terms below. }

\noindent
\textbf{Step \# 2: Proof of the estimate \eqref{est_J3}. }
\ifx
{\color{blue} {\bf This step can be omitted if we make sure that \eqref{JtoJ_2} is true.}
\noindent 
\textbf{Step \# 2a: $a\neq 0$ case.}  We start by considering the quantity 
\begin{align}
J^{(3,0,0)}_{0,3}\phe_k =\frac{27}{q^3}\chi_3 q^3\Gamma_k^3 \phe_k =27\chi_3\Gamma_k^3\phe_k.
\end{align} \ifx
To estimate this, we recall the elliptic equation $\de_k\phe_k=\wwe_k+(\pav^2-\pa_y^2)\phi_k^{(I)}$, and obtain that
\begin{align}
\de_k\pa_y\phe_k=\pa_y\wwe_k+\pa_y(\pav^2-\pa_y^2)\phi_k^{(I)},\quad \de_k ik\phi_k=ik\ww_k.
\end{align}
Since $\phi_k(\pm 1)=0,\ \ \pa_{yy}\phi_k(\pm1)=\ww_k(\pm 1)+|k|^2\phi_k(\pm 1)=0$,   the estimate \eqref{mx_rg} yields that the following estimate is satisfied 
\begin{align}\label{Gaphi_H2}
\|\pa_y\phi_k\|_{\dot H_k^2}\lesssim \|\pa_y\ww_k\|_{L^2},\quad \|k\phi_k\|_{\dot H^2_k}\lesssim \|k\ww_k\|_{L^2}.
\end{align}\fi
Application of $\eqref{est_J3}_{j=0}$ yields that 
\begin{align*}
{\bf a}_{0,3}\|J^{(3,0,0)}_{0,3}\phe_k\|_{L^2}\lesssim&\varphi^4\|\chi_3\Gamma_k^3 \phe_k\|_{L_y^2} 
\lesssim  \sum_{\ell=0}^3\varphi^4t^{3-\ell}\|\pav^\ell|k|^{3-\ell} \phe_k\|_{L^2}\\
\lesssim& {\bf a}_{0,0}\|\wwe_k\|_{L_y^2}+|v_{yy}|\big|_{y=\pm1}\|\ww_k^{(I)}\|_{L_v^2}+\sum_{m+n\leq 3}{\bf a}_{m,n}\|\ci_{m,n}\|_{H^1_k}.
\end{align*}
For the other $a>0$ cases, we can use Lemma \ref{lem:induction} to derive the estimates. Hence we omit the details. }
\fi
We derive the estimate for the $J_{0,2}^{(0,b,c)}\phe_k$ terms with $b+c=3$. \siming{If \eqref{JtoJ_2} holds, then we might not need to worry about the others. }

First of all, thanks to \eqref{est_J2}, we have that 
\begin{align}
\sum_{b+c= 2}{\bf a}_{0,2}^2\|J^{(0,b,c+1)}_{0,2}\phe_k\|_{L_y^2}^2  
\lesssim &\sum_{m+n=0}^{2}{\bf a}_{m,n}^2\lf(\||k|J_{m,n}^{(0)}\wt \ww ^{(E)}_{k}\|_{L_y^2}^2 +\||k|\mathbb{C}_{m,n}^{(I)}\|_{L_y^2}^2\rg).\label{J_3-k}
\end{align} 
To obtain the remaining components of $J_{0,2}^{(3)} \phe_k$, we focus on the quantity $\pa_y (\mr\phe_{0,2;k})$. Since there is $q^2$ factors in the expression, the boundary condition $\pa_y (\mr\phe_{0,2;k})(y=\pm 1)=0$ holds, and by the estimate \eqref{mx_rg} and the equation \eqref{T123}, we obtained the following:
\begin{align}
\n{\bf a}_{0,2}\|\pa_y(\mr\phe_{0,2;k})\|_{\dot H_k^2}\lesssim &\varphi^3\|\pa_y \mr\wwe_{0,2;k}\|_{L^2}+\varphi^3\|\pa_y\ci_{0,2}\|_{L^2}+\varphi^3\|\pa_y \mathbb{T}_{1,2}\|_{L^2}+\varphi^3\|\pa_y \mathbb{T}_{2,2}\|_{L^2}
+\varphi^3\|\pa_y \mathbb{T}_{3,2}\|_{L^2}\\
=:&\sum_{j=1}^5 T_{2j}.\label{J3_T2}
\end{align}
The $T_{21},\, T_{22}$ terms are consistent with \eqref{est_J3}. We estimate the $T_{23}$-term by applying \eqref{est_phen<=1} as follows,
\begin{align}\n
&T_{23} 
\lesssim\varphi^3\lf\|v_y\pav\lf(\chi_2 q^2 \left(\sum_{\ell=0}^{1}\binom{2}{\ell} \left(2\pav^{1-\ell}v_{yy}\ \pav^2 +\pav^{2-\ell}v_{yy} \ \pav\right)\Gamma_k^{\ell}\right)\phe_k\rg) \rg\|_{L^2} \\
\n&\lesssim \varphi^3\sum_{\ell=0}^{1} \lf\|\pav\lf(\chi_2 q^2 \pav^{1-\ell}v_{yy}\ \pav^2(\chi_\ell\Gamma_k^{\ell}\phe_k)\rg) \rg\|_{L^2}+\varphi^3\sum_{\ell=0}^1\lf\|\pav\lf(\chi_2 q^2\pav^{2-\ell}v_{yy} \ \pav(\chi_\ell\Gamma_k^{\ell}\phe_k)\rg) \rg\|_{L^2}\\
\n&\lesssim\varphi^3 \lf\|\pav\lf(\chi_2  v_{yy}\ q[q,\pav^2]\Gamma_k\phe_k\rg) \rg\|_{L^2}+\varphi^3 \sum_{\ell=0}^{1} \lf\|\pav\lf(\chi_2   \pav^{1-\ell}v_{yy}\ q^{2-\ell} \pav^2(\chi_\ell q^{\ell}\Gamma_k^{\ell}\phe_k)\rg) \rg\|_{L^2}\\
\n &\quad+\varphi^3\lf\|\pav\lf(\chi_2 \pav v_{yy} \ q[q,\pav]\Gamma_k\phe_k\rg) \rg\|_{L^2}+\varphi^3\sum_{\ell=0}^1\lf\|\pav\lf(\chi_2\pav^{2-\ell}v_{yy} \  q^{2-\ell}\pav(\chi_\ell q^\ell \Gamma_k^{\ell}\phe_k)\rg) \rg\|_{L^2}\\
&=:\sum_{j=1}^4T_{23j}.\label{T23}
\end{align}
Here the last term with $\ell=0$ costs the most regularity on $v_y$. We recall the definition $Z=v_y^2-1$, and estimate it as follows:
\begin{align*}
T_{234}&\mathbbm{1}_{\ell=0}\lesssim\varphi^3\lf\|\pav\lf(\chi_2q^{2}\pav^{3}Z\rg) \  \pav \phe_k\rg\|_{L^2}+\varphi^3\lf\|\pav^{3}Z \  \chi_2q^{2}\pav^2 \phe_k\rg\|_{L^2}.
\end{align*}
Now rewrite the first term:
\begin{align*}
\pav&\lf(\chi_2 q^2\pav^3 Z\rg)=\pav^2\lf(\chi_2 q^2\pav^2Z\rg)+\pav\lf([\chi_2 q^2,\pav]\pav^2Z\rg)\\
&=J^{(0,2,0)}_{0,2}Z-\pav\lf(\pav \chi_2\  q^2 \pav^2Z+2\chi_2\pav q\ q\pav^2Z\rg)\\
&=J^{(0,2,0)}_{0,2}Z-\pav\lf((\pav\chi_2\   q+2\chi_2\ \pav q) \lf( \pav\lf(\chi_1q\pav Z\rg)+ [\pav,q]\pav Z\rg)\rg)\\
&=J^{(0,2,0)}_{0,2}Z-(\pav\chi_2\   q+2\chi_2\ \pav q)J^{(0,2,0)}_{0,1}Z-\pav(\pav\chi_2\ q+2\chi_2\ \pav q)J^{(0,1,0)}_{0,1}Z\\
&\quad-\pav(\pav \chi_2 q\pav q+2\chi_2(\pav q)^2)J^{(0,1,0)}_{0,0}Z- (\pav \chi_2 q\pav q+2\chi_2(\pav q)^2)J^{(0,2,0)}_{0,0}Z.
\end{align*}
Hence the term $T_{234}\mathbbm{1}_{\ell=0}$ can be estimated through Gagliardo Nirenberg interpolation and the estimate \eqref{est_n=0_J2},
\begin{align*}
T_{234}\mathbbm{1}_{\ell=0}\lesssim& \lf(\sum_{n=0}^2\sum_{b=1}^2\varphi^n\|J^{(0,b,0)}_{0,n}Z\|_{L^2}\rg)\lf(\sum_{b+c=2}{\bf a}_{0,0}\|J^{(0,b,c)}_{0,0}\phe_k\|_{L^2}\rg)+{\bf a}_{0,0}\varphi^2\|J^{(0,2,0)}_{0,0}\phe_k\|_{L^2}\\
\lesssim& \lf(\sum_{n=0}^2\sum_{b=1}^2\varphi^n\|J^{(0,b,0)}_{0,n}Z\|_{L^2}+1\rg) {\bf a}_{0,0}^2\lf(\| \mr\wwe_{0,0;k} \|_{L^2}^2+\|\ci_{0,0}\|_{L_y^2}^2\rg).
\end{align*}
Here the implicit constant depends on $\|\vyn\|_{L_{t,y}\infty},\ \|v_y\|_{L_t^\infty H_y^3}$. 
Other terms in \eqref{T23} can be estimated as follows
\begin{align*}
\sum_{j=1}^3T_{23j}+\mathbbm{1}_{\ell=0}T_{234}&\lesssim \sum_{m+n=0}^2{\bf a}_{m,n}\lf(\|J_{m,n}^{(0)}\wwe_{k}\|_{H_k^1}+\|\ci_{m,n}\|_{H_k^1}\rg) +|v_{yy}|\big|_{y=\pm 1}\|\oi\|_{L_v^2}.
\end{align*}
Combining the decomposition \eqref{T23} and the estimate we obtained, we have that $T_{23}$ has a bound which is consistent with \eqref{est_J3}
The estimate of the $T_{24}$-term in \eqref{J3_T2} is as follows,
\begin{align*}
T_{24} 
\lesssim&\varphi^3\|\pav(\chi_2(2(q')^2+2q q''+4q q'\pa_y)\Gamma_k^2\phe_k)\|_{L^2}\\
\lesssim&\varphi^3\|\pav\Gamma_k^2\phe_k\|_{L^2}+\varphi^3 \| \chi_2 q  \pav\Gamma_k^2\phe_k\|_{L^2}+\varphi^3\|\pav(qq'')\Gamma_k^2\phe_k\|_{L^2}\\
&+\varphi^3\|\chi_2v_yq q'\pav\Gamma_k^2\phe_k\|_{L^2}+\varphi^3\|\pav(v_yq q')\Gamma_k^2\phe_k\|_{L^2}\\
\lesssim& \varphi^3\sum_{b+c\leq 2}|kt|^c\|\pa_y^{1+b}\phe_k\|_{L^2}+\varphi^3\|\chi_2\pav\Gamma_k(\chi_1 q\Gamma_k\phe_k)\|_{L^2}+\varphi^3\|\chi_2[q,\pav\Gamma_k](  \Gamma_k\phe_k)\|_{L^2}\\
\lesssim& \varphi^3\sum_{b+c\leq 2}|kt|^c\|\pa_y^{1+b}\phe_k\|_{L^2}+\varphi^3\sum_{b+c\leq1}(|k|t)^c\|\pa_y^b\mr\phe_{0,1;k}\|_{L^2}\\
\lesssim& {\bf a}_{0,0}\|\wwe _k\|_{H_k^1}+{\bf a}_{0,1} \|\mr\wwe _{0,1;k}\|_{H_k^1}+|v_{yy}|\big|_{y=\pm1}\|\oi_k\|_{L_v^2}+\sum
_{m+n=0}^{ 2}{\bf a}_{m,n}\|\ci_{m,n}\|_{H_k^1}.
\end{align*}
The estimate of the $T_{25}$-term in \eqref{J3_T2} is as follows,
\begin{align*}
T_{25}\lesssim&\varphi^3\|\pa_y(2\chi_2'\pa_y+\chi_2'') q^2\Gamma_k^2 \phe_k\|_{L^2}\\
\lesssim&\varphi^3\|\pav(\chi_2'v_y\pav (q^2\Gamma_k(\chi_1\Gamma_k \phe_k))\|_{L^2}+\varphi^3\|\pav( \chi_2''q^2\Gamma_k(\chi_1\Gamma_k \phe_k))\|_{L^2}\\
\lesssim&\varphi^3\|\pav(\chi_2'v_y\pav \Gamma_k(\chi_1q\Gamma_k \phe_k))\|_{L^2}+\varphi^3\|\pav(\chi_2'v_y\pav (q [q,\pav](\chi_1\Gamma_k \phe_k))\|_{L^2}\\
&+\varphi^3\|\pav( \chi_2''q \Gamma_k(\chi_1q\Gamma_k \phe_k))\|_{L^2}+\varphi^3\|\pav( \chi_2''q[q,\pav](\chi_1\Gamma_k \phe_k))\|_{L^2}\\
\lesssim&{\bf a}_{0,1}\|\mr\wwe_{0,1;k}\|_{H^1_k}+|v_{yy}|\big|_{y=\pm 1}\|\oi_k\|_{L_v^2}+\sum_{m+n=0}^{ 2}{\bf a}_{m,n}\|\ci_{m,n}\|_{H_k^1}.
\end{align*}
Combining the estimates above and the decomposition \eqref{J3_T2}, we have that
\begin{align*}
{\bf a}_{0,2}\|\pa_y\mr\phe_{0,2;k}\|_{H_k^2}\lesssim &\sum_{m+n=0}^2{\bf a}_{m,n}\lf(\|\pa_y \wwe_{m,n}\|_{L_y^2}+\|\ci_{m,n}\|_{H_k^1}\rg)+|v_{yy}|\big|_{y=\pm1}\|\oi_k\|_{L_v^2}\\ &+ \lf(\sum_{n=0}^2\sum_{b=1}^2\varphi^n\|J^{(0,b,0)}_{0,n}Z\|_{L^2}\rg) {\bf a}_{0,0}^2\lf(\| \mr\wwe_{0,0;k} \|_{L^2}^2+\|\ci_{0,0}\|_{L_y^2}^2\rg).
\end{align*}
Combining this and \eqref{J_3-k} yields \eqref{est_J3}.
\end{proof}
\subsubsection{Higher Regularity Bounds in Gevrey Spaces}
\label{sub2sct:h_rg}
\begin{lemma}\label{lem:ellp>3}
Let $M$ be an arbitrary natural number greater than $3$. The following estimate holds
\begin{align}
\n
\sum_{m+n=3}^{M}&\mathbbm{1}_{n\geq 2}{\bf a}_{m,n}^2\sum_{b+c=2} \|J^{(0,b,c)}_{m,n} \phe_k\|_{L^2} ^2 \\ \n
\lesssim&  \sum_{m+n=0}^M  {\bf a}_{m,n}^2\|J_{m,n}^{(0)} \wwe_{k} \|_{L^2}^2+\sum_{m+n=0}^M  {\bf a}_{m,n}^2\| \ci_{m,n} \|_{L^2}^2 +\lambda^{2s} \sum_{m+n=0}^{M-1}{\bf a}_{m,n}^2\sum_{a+b+c=2}\|J^{(a,b,c)}_{m,n}\phe_k\|_{L^2}^2\\
 &+   \lf(\sum_{n_1=0}^{M}{B}_{0,n_1}^2\varphi^{2n_1} \lf\|J^{(\leq 1)}_{0,n_1}Z\rg\|_{L^2}^2\rg)\lf(\sum_{m+n_2=0}^{M} {\bf a}_{m,n_2}^2\lf\| J^{(\leq 2)}_{m,n_2}\phe_k\rg\|_{L^2}^2\rg). \label{J2_est} 
\end{align}
Here ${\bf a}_{m,n}$ is defined in \eqref{a:weight:neq}, $B_{m,n}$ is defined in \eqref{Bweight}, $J^{(\leq j)}_{m,n}f_k$ is defined in \eqref{J_vec}, and $Z$ is the short hand notation for $Z=v_y^2-1$. Here the implicit constant is independent of $M$ and hence we can take the limit as $M$ approaches $\infty.$ 
\end{lemma}
\ifx
 \chi_n q^n\pav^n(\siming{And $B_{n-\ell}\varphi^{n-\ell}$ should be the $B_{0,n-\ell}$ in Fei's section.})
There are corresponding estimates for $b+c=3:$
\begin{align}\n
\sum_{m+n=3}^{m+n=M}&{\bf a}_{m,n}\sum_{b+c=3} \|J^{(0,b,c)}_{m,n} \phe_k\|_{L^2}  \\ \n
\lesssim&  \sum_{m+n=0}^M\sum_{b+c=3} {\bf a}_{m,n}\|J^{(0,b,c)}_{m,n}\wwe_k\|_{L^2}\\ \n 
&+ \sum_{m+n=0}^{M} \sum_{a+b=3}M^{-\sigma }\sum_{\ell=1}^{n}B_{0,n-\ell}\varphi^{n-\ell}\myr{\lf(\lf\|\chi_{n-\ell}q^{n-\ell}\pav^{n-\ell} v_y^2\rg\|_{L^\infty}+ \lf\|\chi_{n-\ell}q^{n-\ell}\pav^{n-\ell+1} v_y^2\rg\|_{L^2}\rg)}\\ \n
&\quad\quad \times  {\bf a}_{m,\ell}\lf\| J^{(a,b,0)}_{m,\ell}\phe_k\rg\|_{L^2}\\
&+ \sum_{m+n=0}^{M} \sum_{a+b=3}\varphi^n\myr{\|\chi_{n}q^n\pav^{n+1}v_y^2\|_{L^2}}{\bf a}_{m,\ell}\|\chi_m|k|^m\phe_k\|_{L^2}.\label{J3_bg}
\end{align}
\begin{align}
\sum_{m+n=3}^{m+n=M}&{\bf a}_{m,n}\sum_{b+c=4} \|J^{(0,b,c)}_{m,n} \phe_k\|_{L^2}  \\
\lesssim&  \sum_{m+n=0}^M\sum_{b+c=4} {\bf a}_{m,n}\|J^{(0,b,c)}_{m,n}\wwe_k\|_{L^2}\\
&+ \sum_{m+n=0}^{M} \sum_{a+b=2}M^{-\sigma }\sum_{\ell=1}^{n}B_{0,n-\ell}\varphi^{n-\ell}\myr{\lf(\lf\|\chi_{n-\ell}q^{n-\ell}\pav^{n-\ell} v_y^2\rg\|_{L^\infty}+ \lf\|\chi_{n-\ell}q^{n-\ell}\pav^{n-\ell+1} v_y^2\rg\|_{L^2}\rg)}\\
&\quad\quad \times  {\bf a}_{m,\ell}\lf\| J^{(a,b,0)}_{m,\ell}\phe_k\rg\|_{L^2}\\
&+ \sum_{m+n=0}^{M} \sum_{a+b=4}\varphi^n\myr{\|\chi_{n}q^n\pav^{n+1}v_y^2\|_{L^2}}{\bf a}_{m,\ell}\|\chi_m|k|^m\phe_k\|_{L^2}.
\end{align}
\fi
\begin{proof}
Thanks to the equation \eqref{T123} and the maximal regularity estimate \eqref{mx_rg_v}, we have that, 
\begin{align}\n
&\sum_{m+n=3}^M\mathbbm{1}_{n\geq 2}\sum_{b+c=2}{\bf a}_{m,n}^2\|J^{(0,b,c)}_{m,n}\phe_k\|_{L^2}^2\\
&\lesssim \sum_{m+n=3}^{M}\mathbbm{1}_{n\geq 2}{\bf a}_{m,n}^2\lf(\|J_{m,n}^{(0)}\wwe_{k}\|_{L^2}^2+\|\ci_{m,n}\|_{L^2}^2+\|\mathbb{T}_{1;m,n}\|_{L^2}^2+\|\mathbb{T}_{2;m,n}\|_{L^2}^2+\|\mathbb{T}_{3;m,n}\|_{L^2}^2\rg).\label{J2_>3_T}
\end{align}
The first two terms on the right hand side are consistent with the bound \eqref{J2_est}. Hence, we focus on estimating the $\mathbb{T}_{i;m,n}$-terms. 

\noindent 
{\bf Step \# 1: }We start by considering the $\mathbb{T}_{1;m,n}$-term,
\begin{align}\n
{\bf a}_{m,n}^2\|\mathbb{T}_{1;m,n}\|_{L^2}^2  
\n \lesssim& \lf({\bf a}_{m,n}\sum_{\ell=0}^{n-1}\binom{n}{\ell} \lf\| q^{n-\ell}\pav^{n-\ell}Z \ \chi_{m+n}  \pav^2\left(q^{\ell}\Gamma_k^\ell |k|^m\phe_k \right)\rg\|_2\rg)^{2}\\
\n&+\lf( {\bf a}_{m,n}\sum_{\ell=0}^{n-1}\binom{n}{\ell} \lf\| q^{n-\ell}\pav^{n-\ell}Z \ \chi_{m+n}  \lf[\pav^2,q^{\ell}\rg]\left(\Gamma_k^\ell |k|^m\phe_k \right)\rg\|_2\rg)^{2}\\ \n
&+\lf({\bf a}_{m,n}\sum_{\ell=0}^{n-1}\binom{n}{\ell} \lf\| q^{n-\ell }\pav^{n-\ell+1} Z\ \chi_{m+n} \pav\lf(q^{\ell}\Gamma_k^\ell |k|^m\phe_k\rg) \rg\|_2\rg)^{2}\\ 
&+\lf({\bf a}_{m,n}\sum_{\ell=0}^{n-1}\binom{n}{\ell} \lf\| q^{n-\ell}\pav^{n-\ell+1} Z \ \chi_{m+n}\lf[  \pav ,q^{\ell}\rg] \lf(\Gamma_k^\ell |k|^m\phe_k \rg) \rg\|_2\rg)^{2}
=:\sum_{i=1}^4(\mathbb{T}_{1;i})^2.\label{Step3_T_114}
\end{align}

Before diving into details, we make the following observation concerning the combinatoric numbers. Recalling the definitions \eqref{Bweight}, \eqref{a:weight:neq}, we have the following relation 
\begin{align}\n
\frac{{\bf a}_{m,n}}{B_{0,n-\ell}{\bf a}_{m,\ell}}=&\frac{\frac{\lambda^{(m+n)s}\varphi^{(m+n)}}{((m+n)!)^{s}}}{\frac{\lambda^{(n-\ell)s}} {((n-\ell)!)^{s}}\times \frac{\lambda(t)^{(m+\ell)s}\varphi^{(m+\ell)}}{((m+\ell)!)^{s}}}=\varphi^{n-\ell}\left(\frac{(m+n-(m+\ell))! (m+\ell)!}{{(m+n)!}}\right)^{s}\\
=&\varphi^{n-\ell}\binom{m+n}{m+\ell}^{-s}\leq\varphi^{n-\ell}\binom{n}{\ell}^{-s}.\label{b_com} 
\end{align} 
Now we apply the H\"older inequality and the estimates of $\vyn, \, v_{yy}$ to obtain that 
\begin{align}\n
\mathbb{T}_{1;1}
 \lesssim & \sum_{\ell=0}^{n-1}\binom{n}{\ell}\frac{{\bf a}_{m,n}}{B_{0,n-\ell}{\bf a}_{m,\ell}}B_{0,n-\ell} \lf\|\chi_{n-\ell}q^{n-\ell}\pav^{n-\ell}Z\rg\|_{L^\infty} {\bf a}_{m,\ell}\|J^{(0,2,0)}_{m,\ell}\phe_k\|_{L^2} \\
\n &\lesssim \sum_{\ell=0}^{n-1}\binom{m+n}{m+\ell}^{1-s}\varphi^{n-\ell}B_{0,n-\ell} \lf\|\chi_{n-\ell}q^{n-\ell}\pav^{n-\ell}Z\rg\|_{\dot H^1}{\bf a}_{m,\ell}\|J^{(0,2,0)}_{m,\ell}\phe_k\|_{L^2}.\n
\end{align}
Now we have that by the relation \eqref{s:prime} and the  combinatorial fact \eqref{prod}, 
\begin{align}
(\mathbb T_{1;1})^2\lesssim& \lf(\sum_{\ell=0}^{n-1}\binom{n}{\ell}^{-2\sigma-2\sigma_\ast}\rg)\lf(\sum_{\ell=0}^{n-1}\sum_{b\leq 1}\varphi^{2n-2\ell}B_{0,n-\ell} ^2\lf\|J^{(0,b,0)}_{0,n-\ell}Z\rg\|_{L^2}^2 {\bf a}_{m,\ell}^2\|J^{(0,2,0)}_{m,\ell}\phe_k\|_{L^2}^2\rg)\n \\
\lesssim&\sum_{\ell=0}^{n-1}\sum_{b\leq 1}\varphi^{2n-2\ell}B_{0,n-\ell} ^2\lf\|J^{(0,b,0)}_{0,n-\ell}Z\rg\|_{L^2}^2 {\bf a}_{m,\ell}^2\|J^{(0,2,0)}_{m,\ell} \phe_k\|_{L^2}^2 .\label{Step3_T11}
\end{align}
For the second term $\mathbb{T}_{1;2}$, we invoke the commutator relation \eqref{cm_J_est} to obtain
\begin{align*}
( \mathbb{T}_{1;2})^2 \lesssim&  \lf({\bf a}_{m,n}\sum_{\ell=1}^{n-1}\binom{n}{\ell} \lf\| \chi_{n-\ell}q^{n-\ell}\pav^{n-\ell} Z \ \chi_{m+n} \lf[\pav^2,q^{\ell}\rg]\left(\chi_{m+\ell}\Gamma_k^\ell|k|^m \phe_k \right)\rg\|_{L^2}\rg)^2\\
\lesssim &  \lf({\bf a}_{m,n}\sum_{\ell=1}^{n-1}\ \sum_{b_1=0}^1\binom{n}{\ell}   \lf\| J_{0,n-\ell}^{(0,b_1,0)} Z\rg\|_{L^2}\ \sum_{i_2=0}^2 \sum_{a_2+b_2\leq 2}\|J^{(a_2,b_2,0)}_{m,(\ell-i_2)_+}\phe_k\|_{L^2}\rg)^2.
\end{align*}
Now we are in a situation which is almost identical to the term $\mathbb{T}_{1;1}$, so we apply the same argument as before to obtain that  
\begin{align}
\lf( \mathbb{T}_{1;2}\rg)^2 \lesssim  \sum_{\ell=1}^{n-1}\ \sum_{b_1\leq 1}  B_{0,n-\ell}^2  \varphi^{2n-2\ell} \lf\|  J^{(0,b_1,0)}_{0,n-\ell} Z\rg\|_{L^2}^2 \ \sum_{i_2=0}^2\sum_{a_2+b_2\leq 2} {\bf a}_{m,(\ell-i_2)_+}^2\|J^{(a_2,b_2,0)}_{m,(\ell-i_2)_+}\phe_k\|_{L^2}^2. \label{Step3_T12}
\end{align} 
The $\mathbb{T}_{1;3}$-term can be estimated as follows with the help of \eqref{b_com} and the bound \eqref{vy2_1est3},
\begin{align*}\n
 &\mathbb{T}_{1;3} \lesssim {\bf a}_{m,n}\sum_{\ell=0}^{n-1}\binom{n}{\ell} \lf\|{\chi_{m+n}}q^{n- \ell }\pav^{n- \ell+1 } Z\rg\|_{L^2}\ \lf\| \pav \lf(\chi_{m+\ell }q^{\ell }\Gamma_k^{\ell }|k|^m \phe_k\rg) \rg\|_{L^\infty}\\
\n &\lesssim \sum_{\ell=0}^{n-1}\binom{m+n}{m+\ell}^{-\sigma-\sigma_\ast}\  B_{0,n-\ell}\varphi^{n-\ell}\lf\|\chi_{n-\ell}q^{n-\ell }\pav^{n-\ell+1}Z\rg\|_{L^2} \ \sum_{b_2\leq 2}{\bf a}_{m,\ell}\|J^{(0,b_2,0)}_{m,\ell }\phe_k \|_{L^2}\\ \n
 &\lesssim\sum_{\ell=0}^{n-1} \binom{m+n}{m+\ell}^{-\sigma} \sum_{i_1=0}^1B_{0,(n-\ell-i_1)_+}\varphi^{(n-\ell-i_1)_+}\lf\| J_{0,(n-\ell-i_1)_+}^{(\leq 1)} Z\rg\|_{L^2(\text{supp}\wt \chi_1)}{\bf a}_{m,\ell}\|J^{(0,2,0)}_{m,\ell }\phe_k \|_{L^2}.
\end{align*} 
By Lemma \ref{lem:vy2-1est}, the $\mathbb{T}_{1;3}$ term can be estimated as follows:
\begin{align}
(\mathbb{T}_{1;3})^2\lesssim& \sum_{\ell=0}^{n-1}\ \sum_{i_1=0}^1B_{0,(n-\ell-i_1)_+}\varphi^{(n-\ell-i_1)_+}\lf\| J_{0,(n-\ell-i_1)_+}^{(\leq 1)} Z\rg\|_{L^2(\text{supp}\wt \chi_1)}\ {\bf a}_{m,\ell}\|J^{(0,2,0)}_{m,\ell }\phe_k \|_{L^2}.\label{Step3_T13}
\end{align} %
For the $\mathbb{T}_{1;4}$-term in \eqref{Step3_T_114}, we estimate it using similar method as before,
\begin{align*}
\n \mathbb{T}_{1;4}&\lesssim \sum_{\ell=1}^{n-1}\binom{m+n}{m+\ell}^{-\sigma} B_{0,n-\ell}\varphi^{n-\ell}\lf\| \chi_{n-\ell} q^{n-\ell} \pav^{n-\ell+1} Z\rg\|_{L^2} {\bf a}_{m,\ell}\lf\| [\pav,q^\ell]( \chi_{m+\ell}|k|^m  \Gamma_k^{\ell}\phe_k)\rg\|_{L^\infty}\\
\lesssim&\sum_{\ell=1}^{n-1}\binom{m+n}{m+\ell}^{-\sigma}B_{0,n-\ell }\varphi^{n-\ell}\lf\|  \chi_{n-\ell}q^{n-\ell } \pav^{n-\ell+1} Z\rg\|_{L^2}\ \sum_{b_2=0}^1{\bf a}_{m,\ell}\lf\|\pav^{b_2}[\pav,q^{\ell } ](\chi_{m+\ell} |k|^m\Gamma_k^{\ell }\phe_k)\rg\|_{L^2}. 
\end{align*}
Application of the bounds \eqref{cm_pv_J_est}, and \eqref{vy2_1est3} yields that 
\begin{align}  (\mathbb{T}_{1;4})^2\lesssim&   \sum_{\ell=1}^{n-1}\ \sum_{i_1=0}^2B_{0,(n-\ell-i_1)_+}^2\varphi^{2(n-\ell-i_1)_+}\lf\| J^{(\leq 1)}_{0,(n-\ell-i_1)_+}Z\rg\|_{L^2}^2\sum_{i_2=0}^2{\bf a}_{m,(\ell-i_2)_+}^2  \|J_{m,(\ell-i_2)_+}^{(\leq 2)}\phe_k\|_{L^2}^2.\label{Step3_T14}
\end{align}

Combining the decomposition \eqref{J2_>3_T},  \eqref{Step3_T_114} and the estimates \eqref{Step3_T11}, \eqref{Step3_T12}, \eqref{Step3_T13}, \eqref{Step3_T14} yields that 
\begin{align}\n
&\sum_{m+n=3}^{M}\mathbbm{1}_{n\geq 2}{\bf a}_{m,n}^2\|\mathbb{T}_{1;m,n}\|_{L^2}^2\\
\n &\lesssim\sum_{i_1=0}^2\sum_{i_2=0}^2\lf(\sum_{m=0}^{M}\sum_{n=0}^{M-m} \sum_{\ell=0}^{n-1}\ B_{0,(n-\ell-i_1)_+}^2\varphi^{2(n-\ell-i_1)_+}\lf\| J^{(\leq 1)}_{0,(n-\ell-i_1)_+}Z\rg\|_{L^2}^2\ {\bf a}_{m,(\ell-i_2)_+}^2  \|J_{m,(\ell-i_2)_+}^{(\leq 2)}\phe_k\|_{L^2}^2\rg)\\
\n &\lesssim \sum_{i_1=0}^2\sum_{i_2=0}^2 \lf(\sum_{m=0}^{M} \sum_{\ell=0}^{M-m}\ {\bf a}_{m,(\ell-i_2)_+}^2  \|J_{m,(\ell-i_2)_+}^{(\leq 2)}\phe_k\|_{L^2}^2  \sum_{n=\ell}^{M-m}B_{0,(n-\ell-i_1)_+}^2\varphi^{2(n-\ell-i_1)_+}\lf\| J^{(\leq 1)}_{0,(n-\ell-i_1)_+}Z\rg\|_{L^2}^2\rg)\\
&\lesssim  \lf(\sum_{n_1=0}^{M}{B}_{0,n_1}^2\varphi^{2n_1} \lf\|J^{(\leq 1 )}_{0,n_1}Z\rg\|_{L^2}^2\rg)\lf(\sum_{m+n_2=0}^{M}{\bf a}_{m,n_2}^2\lf\| J^{(\leq 2)}_{m,n_2}\phe_k\rg\|_{L^2}^2\rg).\label{J2_T1_h}
\end{align} This is consistent with \eqref{J2_est} and concludes {\bf Step \# 1.} 


\noindent 
{\bf Step \# 2: }
We can estimate the $\mathbb{T}_ {2;m,n}( =\chi_{m+n}(n(n-1)q^{n-2}(q')^2+nq^{n-1} q''+2nq^{n-1}q'\pa_y)|k|^m\Gamma_k^n\phe_k)$ in \eqref{T123} as follows 
\begin{align}
\n {\bf a}_{m,n}^2 \|\mathbb{T}_{2;m,n}\|_{L^2}^2 
  \lesssim&
\bigg({\bf a}_{m,n}\lf\|\chi_{m+n} \frac{n^2}{q }|k|^m q^{n-1}\Gamma_k(\chi_{m+n-1}\Gamma_k^{n-1}\phe_k)\rg\|_{L^2} \\ \n
&+{\bf a}_{m,n}n\lf\|\chi_{m+n} q^{n-1} \Gamma_k(\chi_{m+n-1}|k|^m\Gamma_k^{n-1}\phe_k)\rg\|_{L^2}\\
\n&+
{\bf a}_{m,n}n\lf\|\chi_{m+n}q^{n-1}(\pav\Gamma_k)(\chi_{m+n-1}|k|^m\Gamma_k^{n-1}\phe_k)\rg\|_{L^2}\bigg)^2\\
\n\lesssim&\lambda^{2s}\bigg(\frac{{\bf a}_{m,n-1}}{(m+n)^{\sigma+\sigma_\ast}}\varphi\lf\| \frac{m+n}{q} \Gamma_k(\chi_{m+n-1}|k|^mq^{n-1}\Gamma_k^{n-1}\phe_k)\rg\|_{L^2}\\
\n &+\frac{{\bf a}_{m,n-1}}{(m+n)^{\sigma+\sigma_\ast}}\lf\|\frac{(m+n)^2}{q^2}(\chi_{m+n-1}|k|^mq^{n-1}\Gamma_k^{n-1}\phe_k)\rg\|_{L^2}\\
\n &+\frac{{\bf a}_{m,n-1}}{(m+n)^{\sigma+\sigma_\ast}}\varphi\lf\|\chi_{m+n}  \Gamma_k(\chi_{m+n-1}|k|^mq^{n-1}\Gamma_k^{n-1}\phe_k)\rg\|_{L^2} \\
\n&+\frac{{\bf a}_{m,n-1}}{(m+n)^{\sigma+\sigma_\ast}}\varphi\lf\|\chi_{m+n}  [q^{n-1},\Gamma_k](\chi_{m+n-1}|k|^m\Gamma_k^{n-1}\phe_k)\rg\|_{L^2}\\
\n&+\frac{{\bf a}_{m,n-1}}{(m+n)^{\sigma+\sigma_\ast}}\varphi\lf\|\chi_{m+n}  [q^{n-1}, \pav\Gamma_k](\chi_{m+n-1}|k|^m\Gamma_k^{n-1}\phe_k)\rg\|_{L^2}\\ \n
&+
\frac{{\bf a}_{m,n-1}}{(m+n)^{\sigma+\sigma_\ast}}\varphi\lf\|\chi_{m+n}(\pav\Gamma_k)(\chi_{m+n-1}|k|^mq^{n-1}\Gamma_k^{n-1}\phe_k)\rg\|_{L^2}\bigg)^2\\ 
\lesssim&\frac{\lambda^{2s}{\bf a}_{m,n-1}^2}{(m+n)^{2\sigma+2\sigma_\ast}}\sum_{a+b+c=2}\|J_{m,n-1}^{(a,b,c)}\phe_k\|_{L^2}^2. \label{St_5_T_2}
\end{align} 
Hence the sum $\sum_{m+n=3}^M\mathbbm{1}_{n\geq 2} {\bf a}_{m,n}^2\|\mathbb{T}_{2;m,n}\|_{L^2}^2$ is bounded by the right hand side of \eqref{J2_est}. Hence, the {\bf Step \# 2} is concluded. 

\noindent 
{\bf Step \# 3: }
The estimate of the $\mathbb{T}_{3;m,n}$ term is similar to the estimate before. Here we recall the properties of the cut-off function
\begin{align}
\|\pa_y \chi_{m+n}\|_\infty\leq C(m+n)^{1+\sigma}, \quad \|\pa_{yy}\chi_{m+n}\|_\infty\leq C(m+n)^{2(1+\sigma)}. 
\end{align} 
Similar to the estimate of $\mathbb{T}_{1;m,n}$, the main problem in estimating the $\chi_{m+n}$ commutator terms that results in the loss of $(m+n)^2$ powers. As a result, we have to open up two $\Gamma_k$-derivatives to extract regularity information from the $m+n-2$ level. We estimate the $\mathbb{T}_{3;m,n}$ as follows
\begin{align}\n
{\bf a}_{m,n}^2\|\mathbb{T}_{3;m,n}\|_{L^2}^2\leq &{\bf a}_{m,n}^2\|(2\chi_{m+n}'\pa_y+\chi_{m+n}'') q^n\Gamma_k^n |k|^m\phe_k\|_{L^2}^2\\   
\n \lesssim&{\bf a}_{m,n}^2(m+n)^{2+2\sigma}\lf\|\pav \lf(  q^n\lf(\pav+ikt\rg)(\chi_{m+n-1}|k|^m\Gamma_k^{n-1} \phe_k)\rg)\rg\|_2 ^2\\
\n &+{\bf a}_{m,n}^2(m+n)^{4(1+\sigma)}\lf\|q^n\lf(\pav+ikt\rg)^2(\chi_{m+n-2}|k|^{m}\Gamma_k^{n-2}\phe_k) \rg\|_2^2\\
=:&\mathbb{T}_{3;1}+\mathbb{T}_{3;2}.\label{Step_3_T_3} 
\end{align}
We first treat the term $\mathbb{T}_{3;1}$.  We recall the notation $\phe_{m,n}=\chi_{m+n}|k|^m q^n\Gamma_k^n\phe_k$. Applying the combinatoric bound \eqref{gevbd1212} and the  commutator relations \eqref{cm_J}, \siming{\eqref{cm_qn_pv_1}} yields that
\begin{align}\n
\mathbb{T}_{3;1}\lesssim&\frac{\lambda^{2s}{\bf a}_{m,n-1}^2}{(m+n)^{2\sigma_\star}}\varphi^2\lf(\lf\|\pav \lf(  q^n\pav(\chi_{m+n-1}|k|^m\Gamma_k^{n-1} \phe_k)\rg)\rg\|_2+|k|t\lf\|\pav \lf(  q^n\chi_{m+n-1}|k|^m\Gamma_k^{n-1} \phe_k\rg)\rg\|_2\rg)^2\\ \n
\lesssim& \frac{\lambda^{2s}{\bf a}_{m,n-1}^2}{(m+n)^{2\sigma_\star}}\bigg(\lf\|\pav\lf(q\pav\lf(\chi_{m+n-1}|k|^m q^{n-1}\Gamma_k^{n-1} \phe_k\rg)\rg)\rg\|_2 \\
&\qquad\qquad +\lf\|\pav \lf(q [\pav,q^{n-1}]\lf(\chi_{m+n-1}|k|^m\Gamma_k^{n-1} \phe_k\rg)\rg)\rg\|_2\n \\ 
\n &\quad\quad\qquad +\lf\|q |k|\pav \lf(  \chi_{m+n-1}|k|^mq^{n-1}\Gamma_k^{n-1} \phe_k\rg)\rg\|_2+\lf\| |k|  \lf(\chi_{m+n-1}|k|^mq^{n-1}\Gamma_k^{n-1} \phe_k\rg)\rg\|_2\bigg)^2\\
\n\lesssim &  \frac{\lambda^{2s}{\bf a}_{m,n-1}^2}{(m+n)^{2\sigma_\star}}\bigg(\lf\|\pav^2\mr\phe_{m,n-1;k}\rg\|_2^2+\lf\|\pav \mr\phe_{m,n-1;k}\rg\|_2^2   +\lf\|  (n-1)\pav\mr \phe_{m,n-1;k}\rg\|_2^2\\ \n&\qquad\qquad\qquad+\lf\|  \frac{n-1}{q } \mr\phe_{m,n-1;k} \rg\|_2^2+\lf\||k|\pav\mr\phe_{m,n-1;k}\rg\|_2^2+\lf\|\frac{m+n-1}{q}|k|\mr\phe_{m,n-1;k}\rg\|_2^2\bigg)\\
\lesssim &  \frac{\lambda^{2s}{\bf a}_{m,n-1}^2}{(m+n)^{2\sigma_\star}}\sum_{a+b+c=2}\lf\|J^{(a,b,c)}_{m,n-1}\phe_k\rg\|_2^2.\label{Step_4_T31} 
\end{align} 
Next we treat the term $\mathbb{T}_{3;2}$ in \eqref{Step_3_T_3}. We further decompose the term as follows, 
\begin{align}\n
\mathbb{T}_{3;2}\lesssim&\frac{\lambda^{4s}{\bf a}_{m,n-2}^2}{(m+n)^{4\sigma_\star} }\varphi^4 \lf\|q^n\lf( \pav^2+2ikt \pav  -k^2t^2\rg) (\chi_{m+n-2}\Gamma_k^{n-2}|k|^m\phe_k) \rg\|_2^2\\ \n
\lesssim & \frac{\lambda^{4s}{\bf a}_{m,n-2}^2}{(m+n)^{4\sigma_\star } }\bigg(\lf\|q^n\pav^2(\chi_{m+n-2}|k|^m\Gamma_k^{n-2}\phe_k)\rg\|_2^2+\lf\||k| q^n\pav( \chi_{m+n-2}|k|^m\Gamma_k^{n-2}\phe_k)\rg\|_2^2\\
\n &\qquad\qquad\qquad+ \lf\|\chi_{m+n-2}q^n|k|^2  \Gamma_k^{n-2}|k|^m\phe_k \rg\|_2^2\bigg)\\
=:&\mathbb{T}_{3;21}+\mathbb{T}_{3;22}+\mathbb{T}_{3;23}. \label{Step_4_T32} 
\end{align}
For the $\mathbb{T}_{3;21}$ term, we recall the commutator relations \eqref{cm_J}, \eqref{cm_qn_pv_1} and estimate it as follows:
\begin{align}\n
\mathbb{T}_{3;21}\lesssim & \frac{\lambda^{4s}{\bf a}_{m,n-2}^2}{(m+n)^{4\sigma_\star } }\lf(\lf\|q^2\pav^2(\chi_{m+n-2}q^{n-2}\Gamma_k^{n-2}|k|^m\phe_k)\rg\|_2^2+\lf\|q^2\lf[\pav^2,q^{n-2}\rg](\chi_{m+n-2}\Gamma_k^{n-2}|k|^m\phe_k)\rg\|_2^2\rg)\\
\lesssim&\frac{\lambda^{4s}{\bf a}_{m,n-2}^2}{(m+n)^{4\sigma_\star } }\sum_{a+b=2}\|J_{m,n-2}^{(a,b,0)}\phe_k\|_2^2.\label{Step_4_T321} 
\end{align} 
This is consistent with the result. 
Next we estimate the $\mathbb{T}_{3;22}$-term, we recall the commutator relations \eqref{cm_J}, \eqref{cm_qn_pv_1} to estimate the term as follows
\begin{align}\n
\mathbb{T}_{3;22}\lesssim&\frac{\lambda^{4s}{\bf a}_{m,n-2}^2}{(m+n)^{4\sigma_\star } }\lf(\lf\| |k|q^2 \pav\mr\phe_{m,n-2;k} \rg\|_2^2+\lf\| |k|q^2  [\pav ,q^{n-2}](\chi_{m+n-2}|k|^m\Gamma_k^{n-2}\phe_k)\rg\|_2^2\rg)\\
\lesssim&\frac{\lambda^{4s}{\bf a}_{m,n-2}^2}{(m+n)^{4\sigma_\star } }\sum_{a+b+c=2}\|J_{m,n-2}^{(a,b,c)}\phe_k\|_2^2.\label{Step_4_T322}
\end{align} 

Finally, we estimate the $\mathbb{T}_{3;23}$ term,
\begin{align}
\mathbb{T}_{3;23} \lesssim&\frac{\lambda^{4s}{\bf a}_{m,n-2}^2}{(m+n)^{4\sigma_\star } } \|J_{m,n-2}^{(0,0,2)}\phe_k\|_2 .\label{Step_4_T323}
\end{align}
Hence the sum $\sum_{m+n=3}^M\mathbbm{1}_{n\geq2} {\bf a}_{m,n}^2\|\mathbb{T}_{3;m,n}\|_{L^2}^2$ is bounded by the right hand side of \eqref{J2_est}. Hence, the {\bf Step \# 3} is concluded. Combining the estimates yields the result \eqref{J2_est}.  



\end{proof}
\ifx
\noindent
\textcolor{blue}{\bf The $J^3$ term:} 

For higher order Sobolev norms, we need a family of new quantities. Since one quickly runs out of letters, we define
\begin{align}
J^{(a,b,c)}_{m,n}=\lf(\frac{m+n}{q}\rg)^a(v_y^{-1}\pa_y)^b |k|^c ( \chi_{m+n} q^n\Gamma_k^n|k|^m \phe_k)\mathbbm{1}_{a+b+c\leq n,\text{ or }a=0}. 
\end{align}\fi

To derive the Sobolev $H^1$-based Gevrey estimates $\lf(\sum_{a+b+c=3}J^{(a,b,c)}_{m,n}\phe_k\rg)$, we plan to invoke the elliptic maximal regularity type estimate and the estimates of the terms $\pa_y \mathbb{T}_{1;m,n}, \, \pa_y \mathbb{T}_{2;m,n},\, \pa_y \mathbb{T}_{3;m,n}$. We summarize the estimates of the $\pa_y\mathbb{T}$-terms in the following lemma.
\begin{lemma} Assume that $\sigma_\ast\geq 8\sigma$, \myr{$m+n\geq 3, n\geq 2 $. \siming{(We have treated the $n\leq 1$ cases in the previous subsection.)}} Then the following estimates hold \begin{subequations}
\begin{align} 
\sum_{m+n=3}^M \mathbbm{1}_{n\geq 2}{\bf a}_{m,n}^2\|\pa_y\mathbb{T}_{1;m,n}\|_{L^2}^2
\lesssim&\sum_{\mf i=0}^1\lf(\sum_{n_1=0}^MB_{0,n_1}^2\varphi^{2n_1}\lf\|J^{(\leq 1+\mf i)}_{0,n_1}Z\rg\|_{L^2}^2\rg)\lf(\sum_{m+n_2=0}^M{\bf a}_{m,n_2}^2\lf\|J^{(\leq 3-\mf i)}_{m,n_2}\phe_k\rg\|_{L^2}^2\rg); \label{payT1} 
\end{align}
\siming{We can in fact derive the general estimate for general strictly monotone $v$. We can tune the $\myr{\eta=\lambda_{\phe}/\lambda_Z}$ small enough to gain a control over the norm of $Z$. But one still needs that $\|Z\|\leq 1/2$ to control the case where all $\Gamma$ land on $\phe$. }
\begin{align}
\sum_{m+n=3}^M\mathbbm{1}_{n\geq 2}{\bf a}_{m,n}^2\|\pa_y\mathbb{T}_{2;m,n}\|_{L^2}^2\lesssim& \lambda^{2s}\sum_{m+n=2}^{M-1}{\bf a}_{m,n}^2\lf\|J^{( \leq 3)}_{m,n}\phe_k\rg\|_{L^2}^2;\label{payT2}
\\
\sum_{m+n=3}^M\mathbbm{1}_{n\geq 2}{\bf a}_{m,n}^2\|\pa_y\mathbb{T}_{3;m,n}\|_{L^2}^2\lesssim&\lambda^{2s}\sum_{m+n=1}^{M-1}{\bf a}_{m,n}^2\lf\|J^{(\leq 3)}_{m,n}\phe_k\rg\|_{L^2}^2.\label{payT3}
\end{align}
\siming{Double check the sum in $\mathbb{T}_2,\mathbb{T}_3. $} 
Here, we use the notation $J_{m,n}^{(\leq\zeta)}f_k$ (or $J_{0,n}^{(\leq\zeta)}f_0$) to denote the vectors  $\{J^{(a,b,c)}_{m,n}f_k\}_{a+b+c\leq \zeta}$ (or $\{J^{(a,b,0)}_{0,n}f_0\}_{a+b\leq \zeta}$), and measure them with the classical $L^2$-norm of vectors:
\begin{align*}
\lf\|J_{m,n}^{(\leq\zeta)}f_k\rg\|_{L^2}^2:=\sum_{a+b+c\leq\zeta}\lf\|J_{m,n}^{(a,b,c)}f_k\rg\|_{L^2}^2,\quad\lf\|J_{0,n}^{(\leq\zeta)}f_0\rg\|_{L^2}^2:=\sum_{a+b\leq\zeta}\lf\|J_{0,n}^{(a,b,0)}f_0\rg\|_{L^2}^2.
\end{align*}
\end{subequations}

\end{lemma}
\begin{proof}
We divide the proof into three steps.

\noindent
{\bf Step \# 1: Proof of \eqref{payT1}.}
By recalling the definition \eqref{T_1} and $Z=v_y^2-1$, we have that 
\begin{align}\n
{\bf a}_{m,n}^2\|\pa_y\mathbb{T}_{1;m,n}\|_{L^2}^2
& \myr{\lesssim {\bf a}_{m,n}^2\lf(\sum_{\ell=0}^{n-1}\sum_{\mf i=0}^1\binom{n}{\ell}\lf\| \pa_y(\chi_{m+n}  q^n)  \pav^{n-\ell+\mf i}Z\,  \pav^{2-\mf i}  \lf(\Gamma_k^{\ell}|k|^m\phe_k\rg)\rg\|_{L^2}\rg)^2} \n \\ 
&\quad\myr{+{\bf a}_{m,n}^2\lf(\sum_{\ell=0}^{n-1}\sum_{\mf i=0}^2\binom{n}{\ell}\lf\| \chi_{m+n}  q^n  \pav^{n-\ell+\mf i}Z\,  \pav^{3-\mf i}  \lf(\Gamma_k^{\ell}|k|^m\phe_k\rg)\rg\|_{L^2}\rg)^2}
=:\sum_{j=1}^2\mathcal{T}_{1j}^2.\label{payT1_T_j}
\end{align}
We decompose the proof into substeps. In each step, we estimate one term in \eqref{payT1_T_j}.

\noindent
{\bf Step \# 1a: Estimation of the $\mathcal{T}_{11}^2$ term. }
We start by considering the $\mathcal{T}_1$-term in \eqref{payT1_T_j}. To this end, we invoke the assumption \eqref{chi:prop:3} to obtain that for $m+n\geq 3, n\geq 2$ 
\begin{align*}\n 
\mathcal{T}_{11} 
\n \lesssim&\frac{\lambda^s}{(m+n)^{\sigma_\ast}}{\sum_{\mf i=0}^1} \sum_{\ell=0}^{n-1}\binom{n}{\ell}{\bf a}_{m,n-1}\varphi\|\mathbbm{1}_{\text{supp}\chi_{m+n}} q^{n-1}\pav^{\mf i}\lf( \pav^{n-\ell}Z\rg)\, \pav^{2-\mf i}\lf( \Gamma_k^{\ell}|k|^m\phe_k\rg)\|_{L^2}. 
\end{align*}
We distinguish between the $\ell=0$ and $\ell\geq 1$ cases as follows
\begin{align}\n
\mathcal T_{11}\lesssim&\frac{\lambda^s}{(m+n)^{\sigma_\ast}}{\sum_{\mf i=0}^1}\sum_{\ell=\lceil n/2\rceil}^{n-1}\binom{n}{\ell} {\bf a}_{m,n-1}\varphi \lf \|q^{n-\ell}\pav^{\mf i}(\chi_{n-\ell}\pav^{n-\ell }Z)\   q^{\ell-1}\Gamma_k\pav^{2-\mf i}\lf(\chi_{m+\ell-1}\Gamma_k^{\ell-1}|k|^m\phe_k\rg)\rg\|_{L^{2}}\\ \n
&+\frac{\lambda^s}{(m+n)^{\sigma_\ast}}{\sum_{\mf i=0}^1}\sum_{\ell=1}^{\lfloor n/2\rfloor}\binom{n}{\ell} {\bf a}_{m,n-1}\varphi \lf \|\chi_{n-\ell-1}q^{n-\ell-1}\pav^{\mf i+1}(\pav^{n-\ell-1 }Z)\   q^\ell\pav^{2-\mf i}\lf(\chi_{m+\ell}\Gamma_k^{\ell}|k|^m\phe_k\rg)\rg\|_{L^{2}}\\ \n
&+\frac{\lambda^s}{(m+n)^{\sigma_\ast}}{\sum_{\mf i=0}^1} \frac{{\bf a}_{m,n-1}\ \varphi}{B_{0,n-1}{\bf a}_{m,0}}\, B_{0,n-1} \|q^{n-1}\pav^{\mf i+1}(\chi_{n-1}\pav^{n-1}Z)\|_{L^2}\ {\bf a}_{m,0} \lf \|\pav^{2-\mf i} \lf(\chi_{m}|k|^m\phe_k\rg)\rg\|_{L^\infty}\n \\ 
=:&\sum_{j=1}^3\mathcal T_{11j}.\label{J3_T_11}
\end{align}
To estimate the first term, we observe that 
\begin{align}
\binom{n}{\ell}=& \frac{n(n-1)!}{\ell (n-\ell)!(\ell-1)!} =\binom{n-1}{\ell-1} \frac{n}{\ell}\leq 2 \binom{n-1}{\ell-1}\leq 2 \binom{m+n-1}{m+\ell-1},\quad {\ell\geq \lceil n/2\rceil}.\label{comb_1}
\end{align}
For the first term $\mathcal T_{11j}$ in \eqref{J3_T_11}, we apply the commutator relations  \eqref{cm_J_est}, \eqref{cm_pv_J_est}, the relation \eqref{s:prime}, and the relation \eqref{comb_1} to estimate it as follows 
\begin{align*}
\n \mathcal T_{111}
&\lesssim \frac{\lambda^s}{(m+n)^{\sigma_\ast }}{\sum_{\mf i=0}^1}\sum_{\ell=\lceil n/2\rceil}^{n-1}\binom{m+n-1}{m+\ell-1}^{-\sigma-\sigma_\ast} \varphi^{n-\ell}\, B_{0,n-\ell}\|q^{n-\ell}\pav^{\mf i}(\chi_{n-\ell}\pav^{n-\ell}Z)\|_{H^1_y}\\ \n
&\qquad\times {\bf a}_{m,\ell-1}\bigg( \|q^{\ell-1}\pav^{3-\mf i}(\chi_{m+\ell-1}|k|^m\Gamma_k^{\ell-1}\phe_{k})\|_{L^2}+|k|\|q^{\ell-1}\pav^{2-\mf i}(\chi_{m+\ell-1}|k|^m\Gamma_k^{\ell-1}\phe_k)\|_{L^2}\bigg)\\
&\lesssim\frac{\lambda^s}{(m+n)^{\sigma_\ast }}{\sum_{\mf i=0}^1}\sum_{\ell=\lceil n/2\rceil}^{n-1}\binom{m+n-1}{m+\ell-1}^{-\sigma-\sigma_\ast} \lf(\sum_{i_1=0}^{1}\varphi^{n-\ell-i_1} B_{0,n-\ell-i_1}\|J^{(\leq 1+\mf i)}_{0,n-\ell-i_1} Z\|_{L^2}\rg) \n \\
&\qquad\qquad\qquad\qquad\times\lf(\sum_{i_2=0}^2{\bf a}_{m,\ell-i_2}  \|J^{(\leq 3-\mf i)}_{m,\ell-i_2} \phe_k\|_{L^2}\rg).
\end{align*}
Hence, we can invoke the estimate \eqref{prod} to obtain that $\mathcal T_{111}$ is bounded as follows
\begin{align}
\mathcal T_{111}^2\lesssim \lambda^{2s}{\sum_{\mf i=0}^1}\sum_{\ell=0}^{n-1}\  \sum_{i_1=0}^{1}\varphi^{2(n-\ell-i_1)_+} B_{0,(n-\ell-i_1)_+}^2\lf\|J^{(\leq 1+\mf i)}_{0,(n-\ell-i_1)_+} Z\rg\|_{L^2}^2\ \sum_{i_2=0}^2{\bf a}_{m,(\ell-i_2)_+}  \lf\|J^{(\leq 3-\mf i)}_{m,(\ell-i_2)_+} \phe_k\rg\|_{L^2}^2.\label{J3_T111}  
\end{align}

For the second term $\mathcal T_{112}$ in \eqref{J3_T_11}, we recall the relation \eqref{comb_1} and obtain that in this parameter regime, 
\begin{align*}
\binom{n}{\ell}=&\siming{ \frac{n(n-1)!}{(n-\ell) (n-\ell-1)!\ell!} =}\binom{n-1}{\ell} \frac{n}{n-\ell}\leq 2 \binom{n-1}{\ell}\leq 2\binom{m+n-1}{m+\ell},\quad {\ell \leq\lfloor n/2\rfloor}.
\end{align*}
Hence, 
\begin{align*}
 \binom{n}{\ell} \frac{{\bf a}_{m,n-1}}{B_{n-\ell-1}{\bf a}_{m,\ell}}\leq 2\frac{(m+n-1)!}{(n-\ell-1)! (m+\ell)!}\frac{ (n-\ell-1)!^{s}(m+\ell)!^{s}}{(m+n-1)!^{s}}\varphi^{n-\ell-1}=2\binom{m+n-1}{m+\ell}^{-\sigma-\sigma_\ast}\varphi^{n-\ell-1}.
\end{align*}
Now we estimate the $\mathcal T_{112}$-term with \eqref{cm_J}, \eqref{cm_qn_pv_3} \siming{(check!)} and \eqref{prod}, as follows,
\begin{align}\n
\mathcal T_{112}^2& \lesssim\lambda^{2s}{\sum_{\mf i=0}^1}\bigg(\sum_{\ell=1}^{\lfloor n/2\rfloor}\binom{m+n-1}{m+\ell}^{-\sigma_\ast} B_{0,n-\ell-1}\varphi^{n-\ell-1}\lf \|q^{n-\ell-1}\pav^{\mf i+1}(\chi_{n-\ell-1}\pav^{n-\ell-1}Z)\rg\|_{L^2}\\ \n
&\qquad\qquad\qquad\qquad\times{\bf a}_{m,\ell}\lf\|q^\ell\pav^{2-\mf i}\lf(\chi_{m+\ell}\Gamma_k^{\ell}|k|^m\phe_k\rg)\rg\|_{H^{1}}\bigg)^2\\
\lesssim&\lambda^{2s}{\sum_{\mf i=0}^1}\sum_{\ell=1}^{\lfloor n/2\rfloor}\lf(\sum_{i_1=0}^{1}\varphi^{2n-2\ell-2-2i_1}\, B_{0,n-\ell-1-i_1}^2\|J^{(\leq 1+\mf i)}_{0,n-\ell-1-i_1} Z\|_{L^2}^2\rg)
 \lf(\sum_{i_2=0}^2 {\bf a}_{m,\ell-i_2}^2  \|J^{( \leq 3-\mf i)}_{m,\ell-i_2} \phe_k\|_{L^2}^2\rg).\label{J3_T112}
\end{align}

Finally, we estimate the $\mathcal T_{113}$ term in \eqref{J3_T_11} with \eqref{cm_J_est}, 
\begin{align}\n
\mathcal T_{113}^2 \lesssim&\frac{\lambda^{2s}}{(m+n)^{\sigma_\ast}}{\sum_{\mf i=0}^1} B_{0,n-1}^2\varphi^{2n-2} \|q^{n-1}\pav^{1+\mf i}(\chi_{n-1}\pav^{n-1}Z)\|_{L^2}^2\ {\bf a}_{0,0} \lf \|\pav^{2-\mf i} \phe_k \rg\|_{H^1}^2\\ 
\lesssim&\frac{\lambda^{2s}}{(m+n)^{\sigma_\ast}}{\sum_{\mf i=0}^1}\sum_{i_1=0}^2 B_{0,(n-1-i_1)_+}^2\varphi^{2(n-1-i_1)_+} \|J^{(\leq 1+\mf i)}_{0,n-1-i_1}Z\|_{L^2}^2\ {\bf a}_{0,0} \lf \|J^{(\leq 3-\mf i)}_{0,0} \phe_k \rg\|_{L^2}^2. \label{J3_T113}
\end{align}
This is concludes the {\bf Step \# 1a}.

\noindent
{\bf Step \# 1b: Estimation of the $\mathcal{T}_{2}^2$ term. }
For the $\mathcal{T}_2^2$ term in \eqref{payT1_T_j}, we estimate it with the relation \eqref{b_com} as follows:
\begin{align}\n
\mathcal{T}_{12}^2\lesssim&\sum_{\mf i=1}^2\lf(\sum_{\ell=0}^{n-1}\binom{m+n}{m+\ell}^{-\sigma}\lf\| \chi_{m+n}  q^{n-\ell}  \pav^{n-\ell+\mf i}Z\rg\|_{L^2}\  \lf\| q^\ell \pav^{3-\mf i}  (\chi_{m+\ell} \Gamma_k^{\ell}|k|^m\phe_k)\rg\|_{L^\infty}\rg)^2\\
&+\lf(\sum_{\ell=0}^{n-1}\binom{m+n}{m+\ell}^{-\sigma}\lf\| \chi_{m+n}  q^{n-\ell}  \pav^{n-\ell}Z\rg\|_{L^\infty}\  \lf\| q^\ell \pav^3  (\chi_{m+\ell} \Gamma_k^{\ell}|k|^m\phe_k)\rg\|_{L^2}\rg)^2. 
\end{align}
Now we invoke the combinatorial fact \eqref{prod}, the commutator relations \eqref{cm_J}, \eqref{cm_qn_pv_3},  \eqref{cm_qn_pv},  and Lemma \ref{lem:vy2-1est}, 
\begin{align}
\n \mathcal{T}_{12}^2\lesssim&\sum_{\mf i=1}^2 \sum_{\ell=0}^{n-1}B_{0,n-\ell}^2\varphi^{2n-2\ell}\lf\|  \chi_{n-\ell}q^{n-\ell}  \pav^{n-\ell+\mf i}Z\rg\|_{L^2}^2\\ \n
&\quad\times{\bf a}_{m,\ell}^2  \lf(\lf\|J^{(0,3-\mf i,0)}_{m,\ell}\phe_k\rg\|_{H^1}+\lf\| [q^\ell, \pav^{3-\mf i}]  \lf(\chi_{m+\ell} \Gamma_k^{\ell}|k|^m\phe_k\rg)\rg\|_{H^1}\rg)\\ \n
&+ \sum_{\ell=0}^{n-1}B_{0,n-\ell}^2\varphi^{2n-2\ell}\lf\| J_{0,n-\ell}^{(\leq 1)}Z\rg\|_{L^2}^2{\bf a}_{m,\ell}^2  \lf(\lf\|J^{(0,3,0)}_{m,\ell}\phe_k\rg\|_{H^1}+\lf\| [q^\ell, \pav^{3}]  \lf(\chi_{m+\ell} \Gamma_k^{\ell}|k|^m\phe_k\rg)\rg\|_{L^2}\rg)\\ 
\lesssim&\sum_{\mf i=0}^1\sum_{\ell=0}^{n-1} \sum_{i_1=0}^2 B_{0,(n-\ell-i_1)_+}^2\varphi^{2(n-\ell-i_1)_+}\lf\| J^{(\leq 1+\mf i)}_{0,n-\ell-i_1} Z\rg\|_{L^2}^2  \ \sum_{i_2=0}^2{\bf a}_{m,(\ell-i_2)_+}^2\|J^{(\leq3-\mf i)}_{m,\ell-i_2}\phe_k\|_{L^2}^2.\label{payT1_T12}
\end{align}\siming{Check?}
\myr{Now summing all the estimates in {\bf Step \#1}, and applying a similar argument as in \eqref{J2_T1_h} yield \eqref{payT1}.}
\ifx
\noindent
{\bf Step \# 1c: Estimation of the $\mathcal{T}_{3}^2$ term. }
Next we consider the $\mathcal{T}_3$ term in \eqref{payT1_T_j}. 
We decompose the term as follows
\begin{align}\n
\mathcal{T}_3^2\lesssim&\lf({\bf a}_{m,n}\sum_{\ell=0}^{n-1}\binom{n}{\ell}  \lf\| \chi_{m+n}  q^n   \pav^{n-\ell+2}Z \   \pav  \left(\Gamma_k^{\ell}|k|^m\phe_k\right)\rg\|_{L^2}\rg)^2\\ 
&+\lf({\bf a}_{m,n}\sum_{\ell=0}^{n-1}\binom{n}{\ell}\lf\| \chi_{m+n}  q^n  \pav^{n-\ell+1}Z \ \pav^2 \left(\Gamma_k^{\ell}|k|^m\phe_k\right)\rg\|_{L^2}\rg)^2\\
=&\mathcal{T}_{31}^2+\mathcal{T}_{32}^2.\label{payT1_T3}
\end{align}
For the first term, we estimate it using \eqref{b_com} as follows:
\begin{align*}
\mathcal{T}_{31}\lesssim&{\bf a}_{m,n}\sum_{\ell=0}^{n-1}(\mathbbm{1}_{m\geq 1}+\mathbbm{1}_{\ell\geq 1})\binom{n}{\ell}  \lf\| \sqrt{\chi_{m+n}}  q^{n-\ell}  \pav^{n-\ell+2}v_y^2 \rg\| _{L^2}\lf \|  \sqrt{\chi_{m+n}}q^{\ell}\pav \left(\Gamma_k^{\ell}|k|^m\phe_k\right)\rg\|_{L^\infty}\\
&+{\bf a}_{0,n} \lf\| \chi_{n}  q^n   \pav^{n+2}Z \   \pav  \phe_k\rg\|_{L^2}\\
\lesssim&\sum_{\ell=0}^{n-1}\binom{n}{\ell}\frac{{\bf a}_{m,n}}{B_{0,n-\ell}{\bf a}_{m,\ell}} B_{0,n-\ell}\\
&\qquad\qquad\qquad\times \lf\| \sqrt{\chi_{m+n}}  q^{n-\ell}  \pav^2 (\chi_{n-\ell}  \pav  ^{n-\ell}v_y^2) \rg\| _{L^2} \ {\bf a}_{m,\ell} \lf \|  \sqrt{\chi_{m+n}}q^{\ell}\pav \left(\chi_{m+\ell}\Gamma_k^{\ell}|k|^m\phe_k\right)\rg\|_{L^\infty}\\
\lesssim&\sum_{\ell=0}^{n-1}\binom{m+n}{m+\ell}^{1-s}\varphi^{n-\ell}\\
&\times B_{0,n-\ell} \lf( \lf\|\pav^2 \lf(\chi_{n-\ell}q^{n-\ell}\pav^{n-\ell}v_y^2\rg) \rg\| _{L^2}+\lf\|    [ q^{n-\ell},\pav^2] \lf(\chi_{n-\ell} \pav^{n-\ell}v_y^2\rg) \rg\| _{L^2}\rg)\\
&\times {\bf a}_{m,\ell} \lf(\lf \| \pav \left(\chi_{m+\ell}q^{\ell}\Gamma_k^{\ell}|k|^m\phe_k\right)\rg\|_{L^\infty} +\lf \|[ q^{\ell}, \pav] \left(\chi_{m+\ell}\Gamma_k^{\ell}|k|^m\phe_k\right)\rg\|_{L^\infty} \rg).
\end{align*}

Now we invoke the commutator relations \eqref{cm_J}, \eqref{cm_qn_pv} to obtain that 
\begin{align*}
\mathcal{T}_{31}
\lesssim&\sum_{\ell=0}^{n-1}\binom{m+n}{m+\ell}^{-\sigma-\sigma_\ast}\varphi^{n-\ell}\ \
  B_{0,n-\ell} \lf( \sum_{a+b=2}\lf\|  \lf(\frac{n-\ell}{q}\rg)^{a}\pav^b \lf( \chi_{n-\ell} q^{n-\ell}\pav^{n-\ell}v_y^2\rg) \rg\| _{L^2}\rg)\\
&\times {\bf a}_{m,\ell} \lf(\lf \|\pav ^2\left(\chi_{m+\ell}q^{\ell}\Gamma_k^{\ell}|k|^m\phe_k\right)\rg\|_{L^2} +\lf \|\pav[ q^{\ell},\pav] \left(\chi_{m+\ell}\Gamma_k^{\ell}|k|^m\phe_k\right)\rg\|_{L^2} \rg)\\
\lesssim&\sum_{\ell=0}^{n-1}\binom{m+n}{m+\ell}^{-\sigma-\sigma_\ast}\varphi^{n-\ell}\\
&\times B_{0,n-\ell} \lf( \sum_{a+b=2}\lf\|  \lf(\frac{n-\ell}{q}\rg)^{a}\pav^b \lf( \chi_{n-\ell} q^{n-\ell}\pav^{n-\ell}v_y^2\rg) \rg\| _{L^2}\rg)\\
&\times {\bf a}_{m,\ell} \lf(\sum_{a+b=2} \lf \| \lf(\frac{m+\ell}{q}\rg)^a\pav^b \left(\chi_{m+\ell}q^\ell\Gamma_k^{\ell}|k|^m\phe_k\right)\rg\|_{L^2} \rg)?
\end{align*}

We can estimate the $\mathcal{T}_{32}$-term in a similar fashion and end up with 
\begin{align}
\mathcal{T}_{32}\lesssim&\sum_{\ell=0}^{n-1}\binom{m+n}{m+\ell}^{-\sigma-\sigma_\ast}\varphi^{n-\ell}\\
&\times B_{0,n-\ell} \lf( \sum_{a+b=2}\lf\|  \lf(\frac{n-\ell}{q}\rg)^{a}\pav^b \lf( \chi_{n-\ell} q^{n-\ell}\pav^{n-\ell}v_y^2\rg) \rg\| _{L^2}\rg) {\bf a}_{m,\ell} \lf(\sum_{a+b=2} \|J^{(a,b,0)}_{m,\ell}\phe_k\|_{L^2}\rg).?
\end{align}  
\ifx
\siming{We need the $\pa_{yy}\mathbb{T}_{(\cdot);m,n}$ and $\pa_y\mathbb{T}_{(\cdot)=2,3;m,n}$. The following relation might be helpful
\begin{align}
\sum_{a+b+c=j}\|\pa_y J^{(a,b,c)}_{m,n}\phe_k\|_{L^2}\lesssim&\sum_{a+b+c=j+1}\|J^{(a,b,c)}_{m,n}\phe_k\|_{L^2},\quad j\in\{2,3\},\quad n\geq 4.\label{JMtoM+1}
\end{align}
The estimate is the natural consequence of the following computation
\begin{align*}
\lf\|\pa_y\lf(\frac{m+n}{q}\rg)^a\pav^b|k|^c\phe_k \rg\|_{L^2}\lesssim& \lf\| \lf(\frac{m+n}{q}\rg)^a\pav^{b+1}|k|^c\phe_k \rg\|_{L^2}+\lf\|\lf[\pa_y,\lf(\frac{m+n}{q}\rg)^a\rg]\pav^b|k|^c\phe_k \rg\|_{L^2}\\
\lesssim&\|J^{(a,b+1,c)}_{m,n}\phe_k\|_{L^2}+\|J^{(a+1,b,c)}_{m,n}\phe_k\|_{L^2}.
\end{align*} 
}
\fi
\fi

\noindent
{\bf Step \# 2: Proof of \eqref{payT2}}
We can estimate the $\mathbb{T}_{2;m,n} $ in \eqref{T123} as follows 
\footnote{\siming{$\mathbb{T}_{2;m,n}= \chi_{m+n}(n(n-1)q^{n-2}(q')^2+nq^{n-1} q''+2nq^{n-1}q'\pa_y)|k|^m\Gamma_k^n\phe_k$}}
\begin{align} \n
{\bf a}_{m,n}\|\pa_y \mathbb{T}_{2;m,n}\|_{L^2}
\lesssim&{\bf a}_{m,n}\lf\|\chi_{m+n}'(n(n-1)q^{n-2}(q')^2+nq^{n-1} q''+2nq^{n-1}q'\pa_y)|k|^m\Gamma_k^n\phe_k\rg\|_{L^2}\\
\n &+{\bf a}_{m,n}\lf\|\chi_{m+n}\pa_y\lf( \frac{n^2}{q^2}|k|^mq^{n}\Gamma_k^n\phe_k\rg)\rg\|_{L^2}\\ 
\n&+
{\bf a}_{m,n}n\lf\|\chi_{m+n} \pa_y\lf(q^{n-1} \Gamma_k(\chi_{m+n-1}|k|^m\Gamma_k^{n-1}\phe_k)\rg)\rg\|_{L^2}
\\
  &+{\bf a}_{m,n}n\lf\|\chi_{m+n}\pa_y\lf(q^{n-1}(\pav\Gamma_k)(\chi_{m+n-1}|k|^m\Gamma_k^{n-1}\phe_k)\rg)\rg\|_{L^2} \n\\
=:& \sum_{\ell=1}^4\mathcal{T}_{2\ell}.\label{payT2_est}
\end{align}
We estimate the first term in \eqref{payT2_est} using the estimate \eqref{chi:prop:3} and the commutator estimates \eqref{cm_J}, \eqref{cm_qn_pv_1} as follows
\begin{align}
\n \mathcal{T}_{21}\lesssim& {\bf a}_{m,n}(m+n)^{1+\sigma}\\
\n &\times\bigg(\lf\|n^2 q^{-1}\pav\lf(\chi_{m+n-1}|k|^m q^{n-1}\Gamma_k^{n-1}\phe_k\rg)\rg\|_{L^2} +\lf\|n^2 q^{-1}[q^{n-1},\pav]\lf(\chi_{m+n-1}|k|^m \Gamma_k^{n-1}\phe_k\rg)\rg\|_{L^2}\\
\n &+\lf\|n \pav\lf(\chi_{m+n-1}|k|^m q^{n-1}\Gamma_k^{n-1}\phe_k\rg)\rg\|_{L^2} +\lf\|n[q^{n-1},\pav]\lf(\chi_{m+n-1}|k|^m \Gamma_k^{n-1}\phe_k\rg)\rg\|_{L^2}\\
\n &+\lf\|n\pav^2  \lf(\chi_{m+n-1}|k|^m q^{n-1}\Gamma_k^{n-1}\phe_k\rg)\rg\|_{L^2}
+t\lf\|n\pa_{v}|k| \lf(\chi_{m+n-1}|k|^m q^{n-1}\Gamma_k^{n-1}\phe_k\rg)\rg\|_{L^2}\\
\n &+\lf\|n [q^{n-1},\pav^2 +ikt\pav ]\lf(\chi_{m+n-1}|k|^m \Gamma_k^{n-1}\phe_k\rg)\rg\|_{L^2}\bigg)\\
\lesssim&\lambda^s{\bf a}_{m,n-1}(m+n)^{-\sigma_\ast}   \|J^{(\leq 3)}_{m,n-1}\phe\|_{L^2}.\label{payT2_1}
\end{align}
The estimate of the other three terms are similar, 
\begin{align*}
\mathcal T_{22}\lesssim& {\bf a}_{m,n}\lf(\lf\| \pav\lf( \frac{n^2}{q}\Gamma_k (\chi_{m+n-1}|k|^mq^{n-1}\Gamma_k^{n-1}\phe_k)\rg)\rg\|_{L^2}+\lf\| \pav\lf( \frac{n^2}{q}[q^{n-1},\pav] (\chi_{m+n-1}|k|^m\Gamma_k^{n-1}\phe_k)\rg)\rg\|_{L^2}\rg)\\ 
\lesssim&\lambda^s{\bf a}_{m,n-1}(m+n)^{-\sigma-\sigma_\ast}\sum_{a+b=3}\|J^{(a,b,0)}_{m,n-1}\phe_k\|_{L^2};\\
\n \mathcal T_{\myr{23\siming{?}}}\lesssim&
{\bf a}_{m,n}n\lf(\lf\|  \pav\lf(\Gamma_k(\chi_{m+n-1}|k|^mq^{n-1} \Gamma_k^{n-1}\phe_k)\rg)\rg\|_{L^2}+\lf\|  \pav\lf([q^{n-1}, \pav](\chi_{m+n-1}|k|^m\Gamma_k^{n-1}\phe_k)\rg)\rg\|_{L^2}\rg)\\
\lesssim&\lambda^s{\bf a}_{m,n-1}(m+n)^{-\sigma-\sigma_\ast}\|J^{( 3 )}_{m,n-1}\phe_k\|_{L^2};\\
\n \mathcal T_{24}\lesssim&{\bf a}_{m,n}n\lf(\lf\| \pav\lf((\pav\Gamma_k)(\chi_{m+n-1}|k|^mq^{n-1}\Gamma_k^{n-1}\phe_k)\rg)\rg\|_{L^2}+\lf\| \pav\lf([q^{n-1},\pav\Gamma_k](\chi_{m+n-1}|k|^m\Gamma_k^{n-1}\phe_k)\rg)\rg\|_{L^2}\rg)\\
\lesssim&\lambda^s{\bf a}_{m,n-1}(m+n)^{-\sigma-\sigma_\ast}{}\|J^{(3)}_{m,n-1}\phe_k\|_{L^2}.\siming{(Check?)}
\end{align*}
\myr{Now summing all the estimates in {\bf Step \# 2} yields \eqref{payT2}.}

\noindent
{\bf Step \# 3: Proof of \eqref{payT3}}
We estimate the $\mathbb{T}_{3;m,n}$ as follows
\begin{align}\n
{\bf a}_{m,n} \|\pa_y\mathbb{T}_{3;m,n}\|_{L^2} 
\leq &{\bf a}_{m,n}\|\pa_y(2\chi_{m+n}'\pa_y+\chi_{m+n}'') q^n\Gamma_k^n |k|^m\phe_k\|_{L^2}\\
 \n \lesssim&{\bf a}_{m,n}\bigg(\| \chi_{m+n}''\pa_y(q^n(\pav+ikt) \chi_{m+n-1}|k|^m\Gamma_k^{n-1} \phe_k)\|_{L^2}\\
\n &\qquad+\| \chi_{m+n}'\pa_{yy}(q^n (\pav+ikt)\chi_{m+n-1}|k|^m\Gamma_k^{n-1} \phe_k)\|_{L^2}\\
 \n &\qquad+ \|\chi_{m+n}''' q^n(\pav+ikt)^2 \chi_{m+n-2}|k|^m\Gamma_k^{n-2} \phe_k\|_{L^2}\\
\n &\qquad+ \|\chi_{m+n}'' \pa_y(q^n (\pav+ikt)^2\chi_{m+n-2}|k|^m\Gamma_k^{n-2} \phe_k)\|_{L^2}\bigg)\\
 =:&\sum_{\ell=1}^4 \mathcal T_{3\ell}.\label{payT3_est}
\end{align} 
Now we estimate the $\mathcal T_{31}$ term in the expression \eqref{payT3_est} 
\begin{align*}
\mathcal T_{31}\lesssim &{\bf a}_{m,n}(m+n)^{2(1+\sigma)}\lf( \|  \pa_y(q \Gamma_k  \mr\phe_{m,n-1;k})\|_{L^2}+\|  \pa_y(q[q^{n-1} ,\pav](\chi_{m+n-1}|k|^m\Gamma_k^{n-1} \phe_k))\|_{L^2}\rg)\\
\lesssim&\lambda^s{\bf a}_{m,n-1}(m+n)^{-\sigma_\ast+\sigma}\bigg(\lf\|\frac{m+n}{q}\pav^2\mr\phe_{m,n-1;k}\rg\|_{L^2}+ \lf\|\frac{m+n}{q}\pav|k|\mr\phe_{m,n-1;k}\rg\|_{L^2}\\
&\qquad+\lf\|\lf(\frac{m+n}{q}\rg)^2\pav\mr\phe_{m,n-1;k}\rg\|_{L^2}+\lf\|\lf(\frac{m+n}{q}\rg)^2\mr\phe_{m,n-1;k}\rg\|_{L^2}\bigg)\\
\lesssim &\lambda^s{\bf a}_{m,n-1}(m+n)^{-2\sigma} \|J^{(3)}_{m,n-1}\phe_k\|_{L^2}.
\end{align*}
Here in the last line, we have used that $\sigma_\ast\geq 4\sigma. $

Now we estimate the $\mathcal T_{32}$ term in the expression \eqref{payT3_est} 
\begin{align*}
\mathcal T_{32}\lesssim &{\bf a}_{m,n}(m+n)^{1+\sigma}\lf( \|  \pa_{yy}(q \Gamma_k  \mr\phe_{m,n-1;k})\|_{L^2}+\|  \pa_{yy}(q[q^{n-1} ,\pav](\chi_{m+n-1}|k|^m\Gamma_k^{n-1} \phe_k))\|_{L^2}\rg)\\
\lesssim &{\bf a}_{m,n-1}(m+n)^{-\sigma_\ast}\|J^{(3)}_{m,n-1}\phe_k\|_{L^2}.
\end{align*}

Next we estimate $\mathcal T_{33}$ term in the expression \eqref{payT3_est} with the commutator relations \eqref{cm_J}, \eqref{cm_qn_pv_1} as follows 
\begin{align*}
\mathcal T_{33}\lesssim &{\bf a}_{m,n}(m+n)^{3(1+\sigma)}\lf( \|  q^2 \Gamma_k^2  \mr\phe_{m,n-2;k}\|_{L^2}+\| q^2[q^{n-2} ,\pav^2+2ikt\pav](\chi_{m+n-2}|k|^m\Gamma_k^{n-2} \phe_k)\|_{L^2}\rg)\\
\lesssim&\lambda^{2s}{\bf a}_{m,n-2}(m+n)^{-2\sigma_\ast+\sigma}\bigg(\lf\|\frac{m+n}{q}\pav^2\mr\phe_{m,n-2;k}\rg\|_{L^2}+ \lf\|\frac{m+n}{q}\pav|k|\mr\phe_{m,n-2;k}\rg\|_{L^2}+\lf\|\frac{m+n}{q} |k|^2\mr\phe_{m,n-2;k}\rg\|_{L^2}\\
&\qquad+\sum_{a+b+c=3,a\geq 1}\lf\|\lf(\frac{m+n}{q}\rg)^a\pav^b|k|^c\mr\phe_{m,n-2;k}\rg\|_{L^2} \bigg)\\
\lesssim &\lambda^{2s}{\bf a}_{m,n-2}(m+n)^{-2\sigma}  \|J^{({3})}_{m,n-2}\phe_k\|_{L^2}.
\end{align*}
The estimate of the $\mathcal T_{34}$ term is similar to the estimate of the $\mathcal T_{33}$ term in \eqref{payT3_est}. We invoke the relations \eqref{cm_J}, \eqref{cm_qn_pv_2}\siming{HS: use \eqref{cm_pv_J_est}?} to obtain that 
\begin{align*}
\mathcal T_{34}\lesssim &{\bf a}_{m,n}(m+n)^{2(1+\sigma)}\lf( \|  \pa_y(q^2 \Gamma_k^2  \mr\phe_{m,n-2;k})\|_{L^2}+\| \pa_y(q^2[q^{n-2} ,\pav^2+2ikt\pav](\chi_{m+n-2}|k|^m\Gamma_k^{n-2} \phe_k))\|_{L^2}\rg)\\
\lesssim&\lambda^{2s}{\bf a}_{m,n-2}(m+n)^{-2\sigma_\ast}\bigg(\lf\| \pav(q^2\pav^2\mr\phe_{m,n-2;k})\rg\|_{L^2}+ \lf\|\pav({q}^2\pav|k|\mr\phe_{m,n-2;k})\rg\|_{L^2}+\lf\|\pav({q}^2 |k|^2\mr\phe_{m,n-2;k})\rg\|_{L^2}\\
&\qquad+\sum_{a+b+c=3,a\geq 1}\lf\|\lf(\frac{m+n}{q}\rg)^a\pav^b|k|^c\mr\phe_{m,n-2;k}\rg\|_{L^2} \bigg)\\
\lesssim &\lambda^{2s}{\bf a}_{m,n-2}(m+n)^{-2\sigma_\ast}\|J^{(3)}_{m,n-2}\phe_k\|_{L^2}.
\end{align*}
Now summing all the estimates in {\bf Step \# 3} yields \eqref{payT3}.
\ifx \begin{align*}
\n \lesssim&{\bf a}_{m,n}(m+n)\lf\|\pa_y \lf(  q^n\lf(\pav+ikt\rg)(\chi_{m+n-1}|k|^m\Gamma_k^{n-1} \phe_k)\rg)\rg\|_2 \\
\n &+{\bf a}_{m,n}(m+n)^2\lf\|q^n\lf(\pav+ikt\rg)^2(\chi_{m+n-2}|k|^{m}\Gamma_k^{n-2}\phe_k) \rg\|_2 
\end{align*}\fi
\end{proof}

As a consequence of the above lemma and the elliptic maximal estimate, we obtained the lemma:
\begin{lemma} Assume $m+n\geq 3,\, n\geq 2.$ Then, 
\begin{align}
\n
\sum_{m+n=3}^{M}&\mathbbm{1}_{n\geq 2}{\bf a}_{m,n}^2\sum_{b+c=3} \|J^{(0,b,c)}_{m,n} \phe_k\|_{L^2} ^2 \\ \n
\lesssim&  \sum_{m+n=0}^M\sum_{b+c=1} {\bf a}_{m,n}^2\|J^{(0,b,c)}_{m,n}\wwe_k\|_{L^2}^2\\
\n &+\sum_{\mf i=0}^{1}  \lf(\sum_{n=0}^{M}B_{0,n}^2\varphi^{2n} {\lf\|J^{(\leq 1+\mf i)}_{0,n}Z\rg\|_{L^2(\mathrm{supp}\wt \chi_1)}^2}\rg)\lf(\sum_{m+n=0}^{M}{\bf a}_{m,n}^2\lf\| J^{(\leq 3-\mf i)}_{m,n}\phe_k\rg\|_{L^2}^2\rg)\\
& +\lambda^{2s} \sum_{m+n=0}^{M-1}{\bf a}_{m,n}^2\|J^{(3)}_{m,n}\phe_k\|_{L^2}^2+\|\wt\ww_k^{(I)}\|_{\myr{H_k^1}}^2\sum_{n=0}^{M}B_{0,n}^2\varphi^{2n{+2}}\| J_{0,n}^{(\leq 2)}Z\|_{L^2(\mathrm{supp}\wt\chi_1)}^2. \label{J23_bg} 
\end{align}
Here the ${\bf a}_{m,n}$ is defined in \eqref{a:weight:neq} and $B_{m,n}$ is defined in \eqref{Bweight}.
\end{lemma}
\begin{proof}
The proof is similar to the proof of Lemma \ref{lem:ellp>3}. Since $n\geq 2$, we consider the equation \eqref{T123} and derive the following two equations 
\begin{align}\n
\de_k \pa_y(\chi_{m+n}|k|^m q^n\Gamma_k^n \phi_k^{(E)})
=&\pa_y(\chi_{m+n} |k|^mq^n\Gamma_k^n \wt\ww_k^{(E)})+\myr{\pa_y(\chi_{m+n} |k|^mq^n\Gamma_k^n(\wch(\pav^2-\pa_y^2)\phi_k^{(I)}))}\\
&+\pa_y\mathbb{T}_{1;m,n}+\pa_y\mathbb{T}_{2;m,n}+\pa_y\mathbb{T}_{3;m,n},\n \\
\quad \pa_y(\chi_{m+n}|k|^m q^n\Gamma_k^n \phi_k^{(E)})&\big|_{y=\pm1 }=0.\label{eq:pyphi_mnk}
\end{align} We invoke the maximal regularity estimate \eqref{mx_rg} on the equation \eqref{eq:pyphi_mnk} to obtain that 
\begin{align} \n
&\sum_{m+n=3}^M\sum_{b+c=2}{\bf a}_{m,n}^2\|\pa_y^{b+1}|k|^cJ_{m,n}^{(0)}\phe_{k}\|_{L^2}^2\\
&\lesssim \sum_{m+n=3}^M{\bf a}_{m,n}^2\lf(\|\pa_yJ_{m,n}^{(0)}\wwe_{k}\|_{L^2}^2+\|\pa_y\mathbb{T}_{1;m,n}\|_{L^2}^2+\|\pa_y\mathbb{T}_{2;m,n}\|_{L^2}^2+\|\pa_y\mathbb{T}_{3;m,n}\|_{L^2}^2+\|\pa_y{\ci_{m,n}}\|_{L^2}^2\rg).\label{J3_>3_T1}
\end{align}
The estimate of the second to fourth terms are summarized in \eqref{payT1}, \eqref{payT2} and \eqref{payT3}. For the last term, we invoke the estimate \eqref{ppv2-py2phI}. In conclusion, the left hand side of \eqref{J3_>3_T1} is consistent with \eqref{J23_bg}.

Next, from $J^{(2)}$-estimate \eqref{J2_est}, and the bound of $\ci_{m,n}$ \eqref{pv2-py2phI}, we have that 
\begin{align*}
\sum_{m+n=3}^M&\mathbbm{1}_{n\geq 2}\sum_{b+c=2}\|J^{(0,b,c+1)}_{m,n}\phe_k\|_{L^2}^2\\
\lesssim &\sum_{m+n=0}^M  {\bf a}_{m,n}^2\| |k|J_{m,n}^{(0)}\wwe_{k} \|_{L^2}^2+\lambda^{2s} \sum_{m+n=0}^{M-1}{\bf a}_{m,n}^2\sum_{a+b+c=2}\|J^{(a,b,c+1)}_{m,n}\phe_k\|_{L^2}^2\\
 &+   \lf(\sum_{n_1=0}^{M}{B}_{0,n_1}^2\varphi^{2n_1} \lf\|J^{(\leq 1)}_{0,n_1}Z\rg\|_{L^2}^2\rg)\lf(\sum_{m+n_2=0}^{M} {\bf a}_{m,n_2}^2\lf\| J^{(\leq 3)}_{m,n_2}\phe_k\rg\|_{L^2}^2\rg)\\
 &+\||k|\wt\ww_k^{(I)}\|_{L^2}^2\sum_{n=0}^{M}\sum_{a+b\leq 1}B_{0,n}^2\varphi^{2n{+2}}\|\mathbbm{1}_{\mathrm{supp}\wt\chi_1} J_{0,n}^{(a,b,0)}Z\|_{L^2}^2. 
\end{align*}
Combining all the estimates developed thus far, we obtained \eqref{J23_bg}. 
\end{proof}
\subsubsection{Conclusion} \label{sub2sct:con}


\begin{proof}[Proof of Proposition \ref{pro:IE_phi_ext} b)]
Recalling the induction relation \eqref{JtoJ_2}, the low-$n$ regularity estimates \eqref{est_n=0_J2}, \eqref{est_n=1_J2}, \eqref{est_J2}, the high-$n$ regularity estimate \eqref{J2_est} and the $\ci_{m,n}$-estimate \eqref{pv2-py2phI}, we obtain that 
\siming{(Check!)}
\begin{align*}
\n
\sum_{m+n=0}^{M} &{\bf a}_{m,n}^2\|J^{(\leq 2)}_{m,n} \phe_k\|_{L^2} ^2\\
  \lesssim&  \sum_{m+n=0}^M  {\bf a}_{m,n}^2\| \wwe_{m,n} \|_{L^2}^2+   \lf(\sum_{n_1=0}^{M}{B}_{0,n_1}^2\varphi^{2n_1} \lf\|J^{(\leq 1)}_{0,n_1}Z\rg\|_{L^2(\mathrm{supp}\wt\chi_1)}^2\rg)\lf(\sum_{m+n_2=0}^{M} {\bf a}_{m,n_2}^2\lf\| J^{(\leq 2)}_{m,n_2}\phe_k\rg\|_{L^2}^2\rg)\\ \n
& +\lambda^{2s} \sum_{m+n=0}^{M-1}{\bf a}_{m,n}^2\|J^{(2)}_{m,n}\phe_k\|_{L^2}^2 +\sum_{n=0}^{M}B_{n}^2\varphi^{2n}\|J^{(\leq 1)}_{0,n} Z\|_{L^2\myr{(\mathrm{supp}\wt\chi_1)}}^2  \|\wt\ww_k^{(I)}\|_{L^2}^2  .
\end{align*}

Now we apply the assumption that
\begin{align}\label{sml_vy2-1'}
\sum_{n=0}^{M}B_{0,n}^2\varphi^{2n}\|J^{(\leq 1)}_{0,n}Z \|_{L^2(\text{supp}\wt\chi_1)}^2\lesssim \exp\{-\nu^{-1/6}\},
\end{align} and take the parameters $\lambda,\, \nu$ small compared to universal constants to guarantee \eqref{J_ph_est1}.
   
To prove the estimate \eqref{J_phe_est}, the strategy is similar. We invoke the  induction relation \eqref{JtoJ_2}, the low-$n$ regularity estimates \eqref{est_n=0_J3}, \eqref{est_n=1_J3}, \eqref{est_J3}, the high-$n$ regularity estimate \eqref{J23_bg} and the $\ci_{m,n}$-estimate \eqref{ppv2-py2phI}, and apply the assumption \eqref{sml_vy2-1'} to obtain that \siming{(Check??)}
\begin{align*}
 &\sum_{m+n=0}^{M} {\bf a}_{m,n}^2\|J^{(\leq 3)}_{m,n} \phe_k\|_{L^2} ^2  \\
& \lesssim  \sum_{m+n=0}^M  {\bf a}_{m,n}^2\| J_{m,n}^{(0)}\wwe_{k} \|_{H_k^1}^2
+ \lf(\sum_{n=0}^{M}B_{0,n}^2\varphi^{2n} \lf\|J^{(2)}_{0,n}Z\rg\|_{L^2(\mathrm{supp}\wt\chi_1)}^2\rg)\lf(\sum_{m+n=0}^{M} {\bf a}_{m,n}^2\lf\| J^{(\leq 2)}_{m,n}\phe_k\rg\|_{L^2}^2\rg)\\
&\quad+(\exp\{-\nu^{-1/6}\}+\lambda^{2s}) \sum_{m+n=0}^{M} {\bf a}_{m,n}^2\|J^{(\leq 3)}_{m,n} \phe_k\|_{L^2} ^2+ \|\wt\ww_k^{(I)}\|_{H_k^1}^2\sum_{n=0}^{M}B_{0,n}^2\varphi^{2n}\|J^{(\leq 2)}_{0,n} Z\|_{L^2(\text{supp}\wt\chi_1)}^2\\
&\quad+|v_{yy}|\big|_{y=\pm 1}\|\oi_k\|_{L_v^2}.
\end{align*} 
Hence for $\lambda,\,\nu$ chosen small enough, we can apply the bound \eqref{J_ph_est1} to derive the following estimate 
\begin{align*}
 \sum_{m+n=0}^{M}& {\bf a}_{m,n}^2\|J^{({3})}_{m,n} \phe_k\|_{L^2} ^2 \\ \lesssim&  \sum_{m+n=0}^M  {\bf a}_{m,n}^2\|J_{m,n}^{(0)}\wwe_{k} \|_{H_k^1}^2
+ \lf(\sum_{n=0}^{M}B_{0,n}^2\varphi^{2n} \lf\|J^{(2)}_{0,n}Z\rg\|_{L^2(\mathrm{supp}\wt\chi_1)}^2\rg)\lf(\sum_{m+n=0}^\infty {\bf a}_{m,n}^2\| J_{m,n}^{(0)}\wwe_{k}\|_{L^2}^2+\|\wt\ww_k^{(I)}\|_{H_k^1}^2\rg)\\
&+ \exp\{-\nu^{-1/9}\}\|\wt\ww_k^{(I)}\|_{H_k^1}^2+|v_{yy}|\big|_{y=\pm 1}\|\oi_k\|_{L_v^2}.
\end{align*} 
 This concludes the proof. 
\ifx
{\bf Previous: }
We have the relation
\begin{align*}
\sum_{\ell=0}^{n}\binom{n}{\ell}^{-\sigma}|f_\ell|\leq (\sum_{\ell=0}^n\binom{n}{\ell}^{-2\sigma})^{1/2}(\sum| f_\ell|^2)^{1/2}\lesssim \|f_k\|_{\ell_k^2}
\end{align*}
Here the idea to handle the second term in \eqref{J23_bg} is that 
\begin{align*}
\leq&\sum_{m+n=0}^M\lf(\sum_{\ell=1}^n\binom{m+n}{m+\ell}^{-2\sigma}\rg)\lf(\sum_{\ell=1}^nB_{0,n-\ell}^2\varphi^{2n-2\ell}\|J^{(\mathfrak m_0-1)}\lf(\chi_{n-\ell}q^{n-\ell}\pav^{n-\ell}Z \rg)\|_{L^2}^2{\bf a}_{m,\ell}^2\|J_{m,\ell}^{(\mathfrak m_0)}\phe_k\|_2^2\rg)\\
\lesssim&\sum_{m=0}^{M}\sum_{n=0}^{M-m}\lf(\sum_{\ell=1}^nB_{0,n-\ell}^2\varphi^{2n-2\ell}\|J^{(\mathfrak m_0-1)}\lf(\chi_{n-\ell}q^{n-\ell}\pav^{n-\ell}Z \rg)\|_{L^2}^2{\bf a}_{m,\ell}^2\|J_{m,\ell}^{(\mathfrak m_0)}\phe_k\|_2^2\rg)\\
\lesssim&\sum_{m=0}^{M}\lf(\sum_{n=0}^{M-m}B_{0,n}^2\varphi^{2n}\|J^{(\mathfrak m_0-1)}\lf(\chi_{n}q^{n}\pav^{n}Z \rg)\|_{L^2}^2\rg)\lf(\sum_{n=0}^{M-m}{\bf a}_{m,n}^2\|J_{m,n}^{(\mathfrak m_0)}\phe_k\|_2^2\rg)\\
\lesssim&\lf(\sum_{n=0}^MB_{0,n}^2\varphi^{2n}\|J^{(\mathfrak m_0-1)}\lf(\chi_{n}q^{n}\pav^{n}Z \rg)\|_{L^2}^2\rg)\lf(\sum_{m+n=0}^M{\bf a}_{m,n}^2\|J_{m,n}^{(\mathfrak m_0)}\phe_k\|_2^2\rg)\\
\lesssim&\ep \sum_{m+n=0}^M{\bf a}_{m,n}^2\|J_{m,n}^{(\mathfrak m_0)}\phe_k\|_2^2.
\end{align*}
If $\ep$ is small enough, this term can be absorbed by the left hand side.  
\fi
\end{proof}

\subsection{Bounds on $\mathcal{F}_{ell}^{(E)}$}
In this section, we prove Proposition \ref{pro:ellptc_int} b). We take Fourier transform in $x$ and reformulate the problem \eqref{dmp_int2} as follows
\begin{align}\n
\Delta_k \Psi_k^{(E)} = &\ww_k^E-  (2+(1+\chi^I) H)\ H^E\ \pav^2 \psi_k +  (2+(1+\chi^I) H)\ H^E\ \pav^2 \Psi^{(E)}_k \\
\n& + \left(H^E\ \pav H + (H^I-1)\ \pav H^E\right)\ \pav \psi_k
-\left(H^E\ \pav H + (H^I-1)\ \pav H^E\right)\ \pav \Psi^{(E)}_k\\
=:& \ww_k^E+R_1+R_2+R_3+R_4, \label{PhiE}\\
\Psi^{(E)}_k(\pm 1) =& 0.\n
\end{align}
Here $F^I=\chi^I F,\ F^E=\chi^E F$ for functions $F\in\{w,H\}$. During the reformulation, we have used the decomposition that $\psi_k=\Psi_k^{(I)}+\Psi_k^{(E)}$. We further note that the support of the right hand side forcing are supported in the region where $\chi_{m+n}\equiv 1$ \eqref{chi} and hence one has sufficient regularity for the coordinate function $H$. 
   
To derive the conormal Gevrey estimates, we can write down a similar equation as \eqref{T123}, 
\begin{align}\n
&\hspace{-0.5cm} \de_k\lf( {\chi_{m+n}} |k|^m q^n\Gamma_k^n \Psi_k^{(E)}\rg) 
\\ =&  |k|^mq^n\Gamma_k^n \ww_k^E+ |k|^mq^n\Gamma_k^nR_1+ |k|^mq^n\Gamma_k^nR_2+ |k|^mq^n\Gamma_k^nR_3+ |k|^mq^n\Gamma_k^nR_4\n \\
\n &+ \chi_{m+n}|k|^m q^n[\pa_{yy}, \Gamma_k^n]\Psi_k^{(E)}+ \chi_{m+n}[\pa_{yy},q^n]\lf( |k|^m \Gamma_k^n \Psi_k^{(E)} \rg)+[\pa_{yy},\chi_{m+n}] \lf( |k|^mq^n\Gamma_k^n \Psi_k^{(E)} \rg) \\
=:&|k|^m q^n\Gamma_k^n \ww_k^E+\sum_{j=1}^7\mathbb{T}_{j;m,n}.\label{PhE_eq}
\end{align} 
Here we note that the $\ww^E_k$, $R_i,\ i\in\{1,2,3,4\}$ are supported in $\text{support}\chi^E$ and $\chi^E\chi_{m+n}=\chi^E$. Hence, we do not have $\chi_{m+n}$ cut-off for the first few terms. However, the solution $\Psi_k^{(E)}$ is supported on the whole channel, we have to include the cutoff $\chi_{m+n}$ in the last three terms. 
\begin{proof}[Proof of Proposition \ref{pro:ellptc_int} b)] 
We decompose the proof into steps:

\noindent
{\bf Step \# 1: General Plan.}
To estimate the solutions to \eqref{PhE_eq}, we invoke the exterior  estimates of $\ww_k,\ \psi_k$ and $H$ to control the $ \ww_k^E$, $\mathbb{T}_{1;m,n}$ and $\mathbb{T}_{3;m,n}$ contribution. Then we use the fact that $H$ is extremely small on the support of $\mathbb{T}_{2;m,n},\ \mathbb{T}_{4;m,n}$ to absorb their contribution. Finally, the treatment of the $\mathbb{T}_{5;m,n},\ \mathbb{T}_{6;m,n}$ and $\mathbb{T}_{7;m,n}$ are treated in the same fashion as previous subsections. 

With this plan in mind, we test the equation \eqref{PhE_eq} by $\chi_{m+n}|k|^m q^n \Gamma_k^n\Psi_k^{(E)}$ to obtain that for arbitrary $M\in\mathbb{N}$,\begin{align}\n
&\sum_{m+n=0}^M \lf(\frac{(2\widetilde\lambda)^{m+n}}{(m+n)!}\rg)^{2/r}\|{\chi_{m+n}}\mathring\Psi^{(E)}_{m,n;k}\|_{H_k^1}^2 \\
&\lesssim\sum_{m+n=0}^M\lf(\frac{(2\widetilde\lambda)^{m+n}}{(m+n)!}\rg)^{2/r}\|\chi_{m+n} \mathring{\ww}_{m,n;k}^E\|_{L^2}^2+\sum_{j=1}^7\sum_{m+n=0}^M\lf(\frac{(2\widetilde\lambda)^{m+n}}{(m+n)!}\rg)^{2/r}\lf\lan\mathbb{T}_{j;m,n},{\chi_{m+n}}\mathring\Psi^{(E)}_{m ,n;k}\rg\ran \n \\
&=: \sum_{m+n=0}^M\lf(\frac{(2\widetilde\lambda)^{m+n}}{(m+n)!}\rg)^{2/r}\|\chi_{m+n} \mathring{\ww}_{m,n;k}^E\|_{L^2}^2+\sum_{j=1}^7 T_j. \label{elp_int_T1-7}
\end{align}%
Here we recall the norms that we employed in \eqref{FE_ell} and the simplified notation $\mathring f_{m,n;k}=|k|^mq^n\Gamma_k^n f_k$, \eqref{f_mn_ring}. Hence to derive the bound \eqref{elp_int}, it is enough to estimate the right hand side of \eqref{elp_int_T1-7}.

\noindent 
{\bf Step \# 2: Estimates of $ {T}_{1},\ {T}_{2}$ terms. }

For the $ {T}_{1},\ {T}_{2}$ contributions in \eqref{elp_int_T1-7}, we estimate them through integration by parts. Since they have similar structures, we use $\mathfrak{B}_k$ to denote $\{\psi_k,\Psi_k^{(E)}\}$ and estimate them in the following unified fashion,
\begin{align*}
T_{j}&\mathbbm{1}_{j\in\{1,2\}}:=\sum_{m+n=0}^M\lf(\frac{(2\widetilde\lambda)^{m+n}}{(m+n)!}\rg)^{2/r}\lf\lan \mathbb{T}_{\iota;m,n},  {\chi_{m+n}}\mathring\Psi^{(E)}_{m,n;k}\rg\ran\mathbbm{1}_{j\in\{1,2\}}\\
=&\sum_{m+n=0}^M \lf(\frac{(2\widetilde\lambda)^{m+n}}{(m+n)!}\rg)^{2/r}\sum_{n_1=0}^n\binom{n}{n_1}\\
&\quad\bigg(
\lf\lan q^{n-n_1}\Gamma^{n-n_1}\lf((2+(1+\chi^I)H)\ H^E\rg)\   [q^{n_1},\pav]\pav(\Gamma_k^{n_1}|k|^m\mathfrak{B}_k),  {\chi_{m+n}}\mr \Psi^{(E)}_{m,n;k} \rg\ran\\ 
&\qquad-\lf\lan v_y^{-1}q^{n-n_1}\Gamma^{n-n_1}\lf((2+(1+\chi^I)H)\ H^E\rg)\  q^{n_1}\pav\lf(\Gamma_k^{n_1}|k|^m\mathfrak{B}_k\rg), \pa_y\lf({\chi_{m+n}} \mr\Psi^{(E)}_{m,n;k}\rg)\rg\ran\\ 
&\qquad-\lf\lan \pa_y\lf(v_y^{-1}q^{n-n_1}\Gamma^{n-n_1}\lf((2+(1+\chi^I)H)\ H^E\rg)\rg) \ q^{n_1}\pav\Gamma^{n_1}|k|^m \mathfrak{B}_k, {\chi_{m+n}}\mr\Psi^{(E)}_{m,n;k}\rg\ran
\end{align*}
Now we start to estimate the above expression. First we observe that the commutator term with $[q^{n_1},\pav]$ is non-vacuous only if $n\geq n_1\geq 1$. Hence, we apply the Sobolev embedding to estimate the coordinate factors in the above expressions, and suitably adjust the position of the $q$-weight to end up with the following 
\begin{align*}
T_{j}\mathbbm{1}_{j=1,2}&\lesssim\sum_{m+n=0}^M \sum_{n_1=0}^n\binom{n}{n_1}^{1-1/r}\lf(\frac{(2\widetilde\lambda)^{n-n_1}}{(n-n_1)!}\rg)^{1/r}\lf(\frac{(2\widetilde\lambda)^{m+n_1}}{(m+n_1)!}\rg)^{1/r}\lf(\frac{(2\widetilde\lambda)^{m+n}}{(m+n)!}\rg)^{1/r}  \\
&\bigg(  \|{q}[q^{n_1},\pav]\pav\Gamma^{n_1}|k|^m\mf B_k\|_{L^2([-1/2,1/2]^c)} \lf\|\frac{1}{q}{\chi_{m+n}} \mr\Psi^{(E)}_{m,n;k} \rg\|_{L^2}\mathbbm{1}_{n\geq n_1\geq 1}\\ 
&\quad +\|  q^{n_1}\pav(\Gamma^{n_1}|k|^m\mf B_k)\|_{L^2([-1/2,1/2]^c)}\lf\| \na_k\lf( {\chi_{m+n}}\mr\Psi^{(E)}_{m,n;k}\rg)\rg\|_{L^2}\\ 
&\hspace{-0.5cm}+\|q^{n_1}\pav({\chi_{m+n_1}}\Gamma^{n_1}|k|^m \mf B_k)\|_{L^2([-1/2,1/2]^c)}\|{\chi_{m+n}}\mr\Psi^{(E)}_{m,n;k}\|_{L^\infty}\bigg)\lf\|  q^{n-n_1}\pav^{n-n_1}\lf((2+(1+\chi^I)H)H^E\rg)\rg\|_{H^1}.
\end{align*}
Here in the last inequality, we have invoked a property of the Gevrey coefficient that $\binom{n}{n_1}\frac{B_{m,n}}{B_{m,n_1}B_{0,n-n_1}}\leq \binom{n}{n-n_1}\binom{m+n}{n-n_1}^{-1/r}\leq C\binom{n}{n_1}^{1-1/r}$, and the fact that $\chi_{m+n}\mathbbm{1}_{m+n\geq 1}\equiv 1$ on the support of $\chi^E.$ 
Now we recall the definition of the $J_{m,n}^{(a,b,c)}f_k$ \eqref{J_nq}, \eqref{J_vec}, the combinatorial bound \eqref{prod} and bound the above expression as follows:
\begin{align*}
T_{j}\mathbbm{1}_{j=1,2}&\lesssim\sum_{m+n=0}^M \sum_{n_1=0}^n\binom{n}{n_1}^{1-1/r}\lf(\frac{(2\widetilde\lambda)^{n-n_1}}{(n-n_1)!}\rg)^{1/r}\lf(\frac{(2\widetilde\lambda)^{m+n_1}}{(m+n_1)!}\rg)^{1/r}\lf(\frac{(2\widetilde\lambda)^{m+n}}{(m+n)!}\rg)^{1/r}\\
&\hspace{2cm}\times \lf(\sum_{a+b+c=1}\lf\|J^{(a,b,c)}_{m,n}\Psi^{(E)}_{k} \rg\|_{L^2}\rg)\lf\|  q^{n-n_1}\pav^{n-n_1}\lf((2+(1+\chi^I)H)H^E\rg)\rg\|_{H^1} \\
&\hspace{2cm}\times\bigg(\sum_{a+b=1}\|J^{(a,b,0)}_{m,n_1}\mf B_k\|_{L^2([-1/2,1/2]^c)}\bigg)\\
&\hspace{-0.8cm}\lesssim\lf(\sum_{m+n=0}^M \lf(\frac{(2\widetilde\lambda)^{m+n}}{(m+n)!}\rg)^{2/r}\lf\|J^{(1)}_{m,n}\Psi^{(E)}_{k} \rg\|_{L^2}^2\rg)^{1/2} \lf(\sum_{m+n=0}^M \lf(\sum_{n_1=0}^n\binom{n}{n_1}^{2-2/r}\rg) \sum_{n_1=0}^n \lf(\frac{(2\widetilde\lambda)^{n-n_1}}{(n-n_1)!}\rg)^{2/r}\rg.\\
&\hspace{-0.1cm}\lf.\times\lf\|  q^{n-n_1}\pav^{n-n_1}\lf((2+(1+\chi^I)H)H^E\rg)\rg\|_{H^1}^2\lf(\frac{(2\widetilde\lambda)^{m+n_1}}{(m+n_1)!}\rg)^{2/r}\|J^{(1)}_{m,n_1}\mf B_k\|_{L^2([-1/2,1/2]^c)}^2\rg)^{1/2}.
\end{align*}
In the first estimate, we have invoked the fact that on the support of $\chi^E$, it is permissible to introduce the cut-off $\chi_{m+n}$ freely to generate the quantities $J_{m,n}^{(a,b,c)}(\cdots)$. Thanks to the combinatorial bound \eqref{prod}, the sum $\sum_{n_1=0}^n\binom{n}{n_1}^{2-2/r}$ in the final expression is bounded by a constant $C_r.$ Furthermore, it is possible to switch the order of summation to estimate the last factor as we have done previously:
 \begin{align}\n
T_{j}\lesssim& \lf(\sum_{m+n=0}^M \lf(\frac{(2\widetilde\lambda)^{m+n}}{(m+n)!}\rg)^{2/r}\lf\|J^{(1)}_{m,n}\Psi^{(E)}_{k} \rg\|_{L^2}^2\rg)^{1/2} \lf(\sum_{m+n=0}^M\lf(\frac{(2\widetilde\lambda)^{m+n}}{(m+n)!}\rg)^{2/r}\|J^{(1)}_{m,n}\mf B_k\|_{L^2([-1/2,1/2]^c)}^2\rg)^{1/2}\\
&\hspace{1cm} \times \lf(\sum_{n=0}^M  \lf(\frac{(2\widetilde\lambda)^{n}}{n!}\rg)^{2/r}\lf\|  q^{n }\pav^{n }\lf((2+(1+\chi^I)H)H^E\rg)\rg\|_{H^1}^2\rg)^{1/2}.\label{FE_Tj12}
\end{align}
At this time, one might be tempted to apply the induction lemma \eqref{JtoJ_2}. However, the induction estimate is valid when applied with the Gevrey coefficients ${\bf a}_{m,n}$, which include $\varphi^n\sim \langle t\rangle^{-n}$ factors. The time-decaying $\varphi$ compensates for the additional $t$-factors concealed in $\Gamma_k$. Here, the Gevrey coefficients lack $\varphi$-factors. Nonetheless, thanks to the analysis in the proof of \eqref{JtoJ_2}, it is clear that invoking the induction will only result in extra $\langle t\rangle$-growing factors.

To estimate the ${T}_{1 }$ contribution, we invoke the product estimates \eqref{ga_prod}, \eqref{al_prod}, and the hypotheses \eqref{boot:H} to obtain that 
\begin{align*}\n
T_{1}\leq&  \frac{1}{8}\sum_{m+n=0}^M \lf(\frac{(2\widetilde\lambda)^{m+n}}{(m+n)!}\rg)^{2/r}\lf\|\na_k\lf({\chi_{m+n}}\mr \Psi^{(E)}_{m,n;k}\rg)\rg\|_{L^2}^2 \\ \n 
&+ C\nu^{4}\lan t\ran^2\sum_{m+n=0}^M \lf(\frac{(2\widetilde\lambda)^{m+n}}{(m+n)!}\rg)^{2/r}e^{-\frac{1}{8}\nu^{-2/3+\eta}}\||k|^m q^n \Gamma_k^n\psi_k\|_{H_k^1([-1/2,1/2]^c))}^2\lf(\mathcal{E}_{H}^{(\gamma)}+\mathcal{E}_{H}^{(\al)}\rg).
\end{align*}
Here in the last line, we have the extra factor $\exp\{-\frac{1}{8}\nu^{-2/3+\eta}\}$ because there is a weight $e^W$ in $\mathcal{E}_{H}^{(\al)},\ \mathcal{E}_{H}^{(\gamma)}
$. Now by Lemma \ref{lem:ExtToInt}, we can express the last term with Gevrey norms with $\varphi$ weights. Combining the decomposition $\psi_k=\phi_k^{(I)}+\phi_k^{(E)}$, the estimates \eqref{jell:2}, \eqref{eell:out}, yields the following, 
\begin{align}
T_{1}\leq& \frac{1}{8}\sum_{m+n=0}^M \lf(\frac{(2\widetilde\lambda)^{m+n}}{(m+n)!}\rg)^{2/r}\lf\|\na_k\lf({\chi_{m+n}}\mr\Psi^{(E)}_{m,n;k}\rg)\rg\|_{L^2}^2 
+ \exp\{-\nu^{-1/6}\}\lf(\mathcal{E}^{(\gamma)}[\ww_k]+\mathcal{E}_{\text{Int}}^{\rm low}[\ww_k]\rg)\lf(\mathcal{E}_{H}^{(\gamma)}+\mathcal{E}_{H}^{(\al)}\rg)^2.\label{est:R_1}
\end{align}

Next we consider the ${T}_{2}$ contributions . Now combining a variation of the induction relation \eqref{JtoJ_2} (with ${\bf a}_{m,n}$ replaced by $\frac{(2\wt\lambda)^{(m+n)/r}}{(m+n)!^{1/r}}$) \siming{(Checked once. But still need double check due to the extra time factors appearing here)}, the Sobolev inequality, Young's inequality and the combinatorial bound \eqref{prod}  yields that
\begin{align*}
T_{2}\leq  \sum_{m+n=0}^M &\lf(\frac{(2\widetilde\lambda)^{m+n}}{(m+n)!}\rg)^{2/r}\lf\|\na_k\lf(\chi_{m+n}\mr\Psi^{(E)}_{m,n;k}\rg)\rg\|_{L^2}^2 \\
&\times\lf(\frac{1}{16}+C\lan t\ran^2\sum_{n=0}^\infty \lf(\frac{(4\widetilde\lambda)^{n}}{n!}\rg)^{2/r}  \lf\|q^n\pav^n\lf((2+(1+\chi^I)H)H^E\rg)\rg\|_{H^1}^2\rg).
\end{align*}
Now we invoke the product estimates \eqref{ga_prod}, \eqref{al_prod}, the hypotheses \eqref{boot:H}, and Lemma \ref{lem:ExtToInt} to obtain that 
\begin{align}
T_{2}\leq \frac{1}{8} \sum_{m+n=0}^M &\lf(\frac{(2\widetilde\lambda)^{m+n}}{(m+n)!}\rg)^{2/r}\lf\|\na_k\lf(\chi_{m+n}\mr\Psi^{(E)}_{m,n;k}\rg)\rg\|_{L^2}^2.
\label{est:R_2}
\end{align}

This concludes Step \# 2.

\noindent
{\bf Step \# 3: Estimates of $ {T}_{3}$, ${T}_{4}$ terms.} The estimates in this step are similar to the ones in the last step. Hence we will only highlight the main differences. The contributions from the $ {T}_{3},\  {T}_{4}$ can be expressed in the following unified fashion (with $\mf B_k=\psi_k$ for $j=3$ and $\mf B_k=\Psi_k^{(E)}$ for $j=4$)
\begin{align*}
\n T_{j}\mathbbm{1}_{j=3,4}=&
\sum_{m+n=0}^M \lf(\frac{(2\widetilde\lambda)^{m+n}}{(m+n)!}\rg)^{2/r}\lf\lan \mathbb{T}_{j;m,n},  {\chi_{m+n}}\mr\Psi^{(E)}_{m,n;k}\rg\ran\mathbbm{1}_{j=3,4} \\
\n =&\sum_{m+n=0}^M \lf(\frac{(2\widetilde\lambda)^{m+n}}{(m+n)!}\rg)^{2/r}\sum_{n_1=0}^n\binom{n}{n_1}\\
&\quad\times\lf \lan q^{n-n_1}\Gamma_k^{n-n_1} \left(H^E\ \pav H + (H^I-1)\ \pav H^E\right)\ q^{n_1}\Gamma_k^{n_1}\pav |k|^m \mf B_k ,  {\chi_{m+n}}\mr\Psi^{(E)}_{m,n;k}\rg\ran. 
\end{align*}
We observe that this expression is similar to the $T_{1},\ T_2$ terms in Step \# 2. Hence, we apply the Sobolev embedding, the combinatorial bound \eqref{prod} to estimate the contributions as follows:
\begin{align}\n
&|T_{j=3,4}|\lesssim\\ \n
&\lf(\sum_{m+n=0}^M \lf(\frac{(2\widetilde\lambda)^{m+n}}{(m+n)!}\rg)^{2/r}\lf\| \chi_{m+n}\mr\Psi^{(E)}_{m,n;k} \rg\|_{L^\infty}^2\rg)^{1/2} \lf(\sum_{m+n=0}^M \lf(\sum_{n_1=0}^n\binom{n}{n_1}^{2-2/r}\rg) \sum_{n_1=0}^n \lf(\frac{(2\widetilde\lambda)^{n-n_1}}{(n-n_1)!}\rg)^{2/r}\rg.\\
&\lf.\times\lf\|  q^{n-n_1}\pav^{n-n_1}\left(H^E\ \pav H + (H^I-1)\ \pav H^E\right)\rg\|_{L^2}^2\lf(\frac{(2\widetilde\lambda)^{m+n_1}}{(m+n_1)!}\rg)^{2/r}\|J^{(1)}_{m,n_1}\mf B_k\|_{L^2([-1/2,1/2]^c)}^2\rg)^{1/2}\n \\
&\lesssim\lf(\sum_{m+n=0}^M \lf(\frac{(2\widetilde\lambda)^{m+n}}{(m+n)!}\rg)^{2/r}\lf\|\na_k( \chi_{m+n}\mr\Psi^{(E)}_{m,n;k}) \rg\|_{L^2}^2\rg)^{1/2} \lf(\sum_{m+n=0}^M \lf(\frac{(2\widetilde\lambda)^{m+n}}{(m+n)!}\rg)^{2/r}\|J^{(1)}_{m,n}\mf B_k\|_{L^2([-1/2,1/2]^c)}^2\rg)^{1/2}\n\\
&\quad \times\lf(\sum_{n=0}^M \lf(\frac{(2\widetilde\lambda)^{n}}{n!}\rg)^{2/r}\lf\|  q^{n}\pav^{n}\left(H^E\ \pav H + (H^I-1)\ \pav H^E\right)\rg\|_{L^2}^2\rg)^{1/2}.\label{FE_Tj_34}
\end{align}
Now we estimate the last factor which is associated with the coordinate system quantity $H$. We invoked a variant of the Gevrey product rule \eqref{prod_gen}: for  sufficiently smooth functions $f,g$, we have the following bound for $1/\mf r=1/\mf p+1/\mf q,$
\begin{align*}&\lf(\sum_{m+n=0}^\infty \lf(\frac{(2\wt \lambda)^n}{n!}\rg)^{2/r}\|q^n \pav^n (fg)\|_{L^{\mathfrak r}}^2\rg)\\
&\hspace{1cm}\lesssim\lf(\sum_{m+n=0}^\infty \lf(\frac{(2\wt \lambda)^n}{n!}\rg)^{2/r}\|q^n \pav^n f\|_{L^{\mathfrak p }}^2\rg)\lf(\sum_{m+n=0}^\infty \lf(\frac{(2\wt \lambda)^n}{n!}\rg)^{2/r}\|q^n \pav^n g\|_{L^{\mathfrak q}}^2\rg).
\end{align*}
The proof of this estimate is similar to the proof of \eqref{prod_gen} and we omit the details. By invoking the above product rule, we obtain the following
\begin{align*}
&\lf(\sum_{n=0}^M \lf(\frac{(2\widetilde\lambda)^{n}}{n!}\rg)^{2/r}\lf\|  q^{n}\pav^{n}\left(H^E\ \pav H + (H^I-1)\ \pav H^E\right)\rg\|_{L^2}^2\rg)^{1/2}\\
&\lesssim 
\lf(\sum_{n=0}^M \lf(\frac{(2\widetilde\lambda)^{n}}{n!}\rg)^{2/r}\lf\|  q^{n}\pav^{n} H^E \rg\|_{L^\infty}^2\rg)^{1/2}\
\lf(\sum_{n=0}^M \lf(\frac{(2\widetilde\lambda)^{n}}{n!}\rg)^{2/r}\lf\|  q^{n}\pav^{n+1}H \rg\|_{L^2(\text{support} {\chi^E})}^2\rg)^{1/2}\\
&\quad+\lf(\sum_{n=0}^M \lf(\frac{(2\widetilde\lambda)^{n}}{n!}\rg)^{2/r}\lf\|  q^{n}\pav^{n+1} H^E \rg\|_{L^2}^2\rg)^{1/2}\
\lf(\sum_{n=0}^M \lf(\frac{(2\widetilde\lambda)^{n}}{n!}\rg)^{2/r}\lf\|  q^{n}\pav^{n}(H^I-1) \rg\|_{L^\infty(\text{support} {\chi^E})}^2\rg)^{1/2}\\
&\lesssim \exp\{-\nu^{-1/5}\}\lf(\mathcal{E}_H^{(\al)}+\mathcal{E}_H^{(\gamma)}\rg).
\end{align*}
Here in the last line, we have invoked the assumption \eqref{asmp}. 

Now we are ready to conclude from the estimate \eqref{FE_Tj_34} and the coordinate control. Similar to the estimate of \eqref{est:R_1}, we have that 
\begin{align}
{T}_{3}\leq &\frac{1}{8}\sum_{m+n=0}^M \lf(\frac{(2\widetilde\lambda)^{m+n}}{(m+n)!}\rg)^{2/r}\lf\| \na_k\lf({\chi_{m+n}}\mr\Psi^{(E)}_{m,n;k}\rg)\rg\|_{L^2}^2+Ce^{-\nu^{-1/6}}\lf((\mathcal{E}_H^{\gamma})^2+(\mathcal{E}_H^{\al})^2\rg) \lf(\mathcal{E}^{(\gamma)}[\ww_k]+\mathcal{E}_{\text{Int}}[\ww_k]\rg). \label{est:R_3}
\end{align}

Similar to the estimate of \eqref{est:R_2}, we have that the ${T}_{4}$ contribution is bounded as follows
\begin{align}
T_4\leq& \lf(\frac{1}{8}+Ce^{-\nu^{-1/7}}((\mathcal{E}_H^{\gamma})^2+(\mathcal{E}_H^{\al})^2)\rg) \sum_{m+n=0}^M \lf(\frac{(2\widetilde\lambda)^{m+n}}{(m+n)!}\rg)^{2/r}\lf\|\na_k\lf({\chi_{m+n}}\mr\Psi^{(E)}_{m,n;k}\rg)\rg\|_{L^2}^2.\label{est:R_4}
\end{align}

This concludes Step \# 3.

\noindent
{\bf Step \# 4: Conclusion.}
The $\mathbb{T}_{5;m,n},\ \mathbb{T}_{6;m,n}$ and $\mathbb{T}_{7;m,n}$ terms in \eqref{PhE_eq} are similar to the terms in \eqref{T123}. Hence they can be estimated in a similar fashion as in the previous subsections. We omit further details for the sake of brevity. \footnote{\siming{\bf Check!} In previous sections, we treat these terms with maximal regularity trick $H^2$. Here we only do $H^1$ estimates. So should be easier.}  
Combining the above estimates (\eqref{elp_int_T1-7}, \eqref{est:R_1}, \eqref{est:R_2}, \eqref{est:R_3}, \eqref{est:R_4}) and the definitions of $\mathcal{E}^{(\gamma)}[\ww_k],\ \mathcal{E}_{\text{Int}}[\ww_k]$, we have obtained that 
\begin{align*}
\sum_{m+n=0}^M &\lf(\frac{(2\widetilde\lambda)^{m+n}}{(m+n)!}\rg)^{2/r}\|\myr{\chi_{m+n}}|k|^m q^n \Gamma_k^n\Psi^{(E)}_k\|_{H_k^1}^2\lesssim e^{-\nu^{-1/6}}\lf(\mathcal{E}^{(\gamma)}[\ww_k] + \mathcal{E}_{\text{Int}}[\ww_k]\right)\lf(1+ \mathcal{E}_H^{(\gamma)}+\mathcal{E}_H^{(\al)}\rg)^2.
\end{align*}
By taking $M\rightarrow \infty$, we have the result.
\end{proof}

\subsection{Proof of Proposition \ref{pro:ell:intro}}

We are now ready to bring together the bounds established in this section in order to prove Proposition \ref{pro:ell:intro}.

\begin{proof}[Proof of Proposition \ref{pro:ell:intro}]Combining Proposition \ref{pro:ellptc_int} and Proposition \ref{pro:IE_phi_ext} yields the result.
\end{proof}

\subsection{Technical Lemmas:  ICC-Method}

We first present a simple lemma, which is useful when we derive maximal regularity estimates.
\begin{lemma}\label{lem:max_reg}
Assume that $F_k\in L^2, \, k\neq 0$, and $\Psi_k$ solves the following elliptic equation
\begin{align}\label{mx_rg_eq}
-\pa_{yy}\Psi_k+|k|^2\Psi_k=F_k
\end{align}{subject to homogeneous boundary conditions $\Psi_k(y=\pm 1)=0$ or $\pa_y \Psi_k(y=\pm 1)=0$.}
Then the following estimates hold
\begin{align}
\|\Psi_k\|_{\dot H_k^2}:=&\|\pa_{yy}\Psi_k\|_{L^2}+\||k|\pa_y\Psi_k\|_{L^2}+\||k|^2\Psi_k\|_{L^2}\leq (3+\sqrt{2})\|F_k\|_{L^2};\label{mx_rg}\\
\|\Psi_k\|_{\dot H_{k,v}^2}:=&\|\pav^2 \Psi_k\|_{L^2}+\||k|\pav\Psi_k\|_{L^2}+\||k|^2\Psi_k\|_{L^2}\leq C(\|v_y^{-1}\|_{L_{t,y}^\infty},\|v_y\|_{L_t^\infty W^{1,\infty}_y})\|F_k\|_{L^2}.\label{mx_rg_v} 
\end{align}
\end{lemma}
\begin{proof}We first test the equation against $\overline{\Psi_k}$ and implement integration by parts. Thanks to the Dirichlet/Neumann boundary condition, we obtain  
\begin{align*}
\|\pa_y\Psi_k\|_2^2+|k|^2\|\Psi_k\|_2^2\leq \|F_k\|_2\|\Psi_k\|_2.
\end{align*}
As a result, we have that 
\begin{align*}
\||k|^2\Psi_k\|_2\leq \|F_k\|_2.
\end{align*}
By taking the $L^2$-norm on both side of the equation \eqref{mx_rg_eq}, we have that 
\begin{align}\n
\|\pa_{yy}\Psi_k\|_2\leq \|F_k\|_2+\||k|^2\Psi_k\|_2\leq 2\|F_k\|_2. 
\end{align}
Finally, we observe that, direct integration by parts yields the inequality
\begin{align}\n
\|k\pa_y\Psi_k\|_2^2=\int_{-1}^1 |k|\pa_y\Psi_k\overline{|k|\pa_y\Psi_k}dy=-\int_{-1}^1|k|^2\Psi_k\overline{\pa_{yy}\Psi_k}dy\leq \|\pa_{yy}\Psi_k\|_2\||k|^2\Psi_k\|_2\leq 2\|F_k\|_2^2.
\end{align}
Combining the estimates we developed so far, we obtained \eqref{mx_rg}. 

To derive \eqref{mx_rg_v}, it is enough to observe that 
\begin{align*}
\||k|\pav\Psi_k\|_{L_y^2}\leq& C\|\vyn\|_{L_{t,y}^\infty}\||k|\pa_y\Psi_k\|_{L_y^2},\\
\|\pav^2 \Psi_k\|_{L_y^2}\leq& C \|\vyn\|_{L^\infty_{t,y}}\lf(\|\pa_y\vyn\|_{L_{t,y}^\infty}\|\pa_y\Psi_k\|_{L_y^2}+\|\vyn\|_{L_{t,y}^\infty }\|\pa_{yy}\Psi_k\|_{L_y^2}\rg).
\end{align*}
Now direct application of \eqref{mx_rg} yields \eqref{mx_rg_v}.
\end{proof}

\begin{lemma}[Maximal regularity estimates with nonvanishing boundary condition]
Assume that $F_k\in L^2$. 

a) Consider the solution  $\Psi_k$ to the following equation
\begin{align}-\pa_{yy}\Psi_k+|k|^2\Psi_k=F_k,\quad \pa_y\Psi_k(y=\pm 1)=a_\pm\in \rr.\label{mx_rg_eq_b}
\end{align}
Then the following estimate holds
\begin{align}
\|\pa_{yy}\Psi_k\|_2+\||k|\pa_y\Psi_k\|_2+\| |k|^2\Psi_k\|_2\lesssim\|F_k\|_2+|k|^{1/2}|a_+|+|k|^{1/2}|a_-|.\label{mx_rg_b}
\end{align}

b) Consider the solution  $\Phi_k$ to the following equation
\begin{align}-\pa_{yy}\Phi_k+|k|^2\Phi_k=F_k,\quad \Phi_k(y=\pm 1)=b_\pm\in \rr.\label{mx_rg_eq_b2}
\end{align}
Then the following estimate holds
\begin{align}
\|\pa_{yy}\Phi_k\|_2+\||k|\pa_y\Phi_k\|_2+\| |k|^2\Phi_k\|_2\lesssim\|F_k\|_2+|k|^{3/2}|b_+|+|k|^{3/2}|b_-|.\label{mx_rg_b2}
\end{align}

c) In either of the scenarios above, if the following boundary constraint is satisfied, 
\begin{align}
\mathrm{Re} \Psi_k\overline{\pa_y\Psi_k}\big|_{y=\pm 1}=0,\label{b_vnsh}
\end{align} 
then the estimate can be improved to be \eqref{mx_rg}.
\end{lemma}
\begin{proof}
To derive the estimate with nonzero boundary conditions, we decompose the solution into the free part $\Psi_k^f$ and the boundary corrector $\Psi_k^b$, i.e.,
\begin{align}
-\de_k\Psi_k^f=F_k,\quad \pa_y\Psi_k^f(\pm 1)=0;\\
-\de_k\Psi_k^b=0,\quad \pa_y\Psi_k^b(\pm 1)=a_\pm.
\end{align}
The solution $\Psi_k^f$ can be estimated with \eqref{mx_rg}, whereas the boundary corrector can be solved explicitly:
\begin{align*}
\Psi_k^b=\frac{a_+}{|k|\sinh(2|k|)}\cosh(|k|(1+y))-\frac{a_-}{|k|\sinh(2|k|)}\cosh(|k|(1-y)). 
\end{align*}
Given the explicit form of the solution, one can derive the $\dot H_k^2$ bound for $\Psi_k^b$ as follows
\begin{align*}
\|\pa_{yy}\Psi_k^b\|_2+\||k|\pa_y \Psi_k^b\|_2+\||k|^2\Psi_k^b\|_2\lesssim |k|^{1/2}(|a_+|+|a_-|). 
\end{align*}
The proof of this inequality is direct and we omit the details for the sake of brevity. 
\siming{(For checking purpose only)Explicit estimate is as follows:
\begin{align*}
\int_{-1}^1 |\pa_{yy}\Psi_k^b|^2dy\lesssim &|a_+|^2\frac{|k|^4}{|k|^2\sinh^2(2|k|)}\int_{-1}^1 \cosh^2(|k|(1+y))dy+|a_-|^2\frac{|k|^4}{|k|^2\sinh^2(2|k|)}\int_{-1}^1 \cosh^2(|k|(1-y))dy\\
\lesssim&|a_+|^2\frac{|k|^4}{|k|^2e^{4|k|}}\int_{-1}^1 e^{2|k|(1+y)}dy+|a_-|^2\frac{|k|^4}{|k|^2e^{4|k|}}\int_{-1}^1 e^{2|k|(1-y)}dy\\
\lesssim&(|a_+|^2+|a_-|^2)|k|;\\
\int_{-1}^1 ||k|\pa_{y}\Psi_k^b|^2dy\lesssim &|a_+|^2\frac{|k|^4}{|k|^2\sinh^2(2|k|)}\int_{-1}^1 \sinh^2(|k|(1+y))dy+|a_-|^2\frac{|k|^4}{|k|^2\sinh^2(2|k|)}\int_{-1}^1 \sinh^2(|k|(1-y))dy\\
\lesssim&|a_+|^2\frac{|k|^4}{|k|^2e^{4|k|}}\int_{-1}^1 e^{2|k|(1+y)}dy+|a_-|^2\frac{|k|^4}{|k|^2e^{4|k|}}\int_{-1}^1 e^{2|k|(1-y)}dy\\
\lesssim&(|a_+|^2+|a_-|^2)|k|.
\end{align*}
The estimate of the $|k|^2\|\Psi_k^b\|_{L^2}$ is similar and we omitted. 
}
Combining this estimate of the boundary corrector and the estimate of the free part, yields the estimate \eqref{mx_rg_b}.

The proof for the \eqref{mx_rg_b2} is similar and we omit the details. 
\siming{(For checking purpose only) We decompose the solution into the free part $\Phi_k^f$ and the boundary corrector $\Phi_k^b$, i.e.,
\begin{align}
-\de_k\Phi_k^f=F_k,\quad \Phi_k^f(\pm 1)=0;\\
-\de_k\Phi_k^b=0,\quad \Phi_k^b(\pm 1)=b_\pm.
\end{align}
The solution $\Phi_k^f$ can be estimated with \eqref{mx_rg}, whereas the boundary corrector can be solved explicitly:
\begin{align*}
\Phi_k^b=\frac{a_+}{\sinh(2|k|)}\sinh(|k|(1+y))+\frac{a_-}{\sinh(2|k|)}\sinh(|k|(1-y)). 
\end{align*}
Given the explicit form of the solution, we can derive the $\dot H_k^2$ bound for $\Phi_k^b$ as follows
\begin{align*}
\|\pa_{yy}\Phi_k^b\|_2+\||k|\pa_y \Phi_k^b\|_2+\||k|^2\Phi_k^b\|_2\lesssim |k|^{3/2}(|b_+|+|b_-|).
\end{align*}
Explicit computation is as follows:
\begin{align*}
\int_{-1}^1 |\pa_{yy}\Phi_k^b|^2dy\lesssim &|b_+|^2\frac{|k|^4}{ \sinh^2(2|k|)}\int_{-1}^1 \sinh^2(|k|(1+y))dy+|b_-|^2\frac{|k|^4}{\sinh^2(2|k|)}\int_{-1}^1 \sinh^2(|k|(1-y))dy\\
\lesssim&|b_+|^2\frac{|k|^4}{e^{4|k|}}\int_{-1}^1 e^{2|k|(1+y)}dy+|b_-|^2\frac{|k|^4}{e^{4|k|}}\int_{-1}^1 e^{2|k|(1-y)}dy\\
\lesssim&(|b_+|^2+|b_-|^2)|k|^3;\\
\int_{-1}^1 ||k|\pa_{y}\Phi_k^b|^2dy\lesssim &|b_+|^2\frac{|k|^4}{\sinh^2(2|k|)}\int_{-1}^1 \cosh^2(|k|(1+y))dy+|b_-|^2\frac{|k|^4}{\sinh^2(2|k|)}\int_{-1}^1 \cosh^2(|k|(1-y))dy\\
\lesssim&|b_+|^2\frac{|k|^4}{e^{4|k|}}\int_{-1}^1 e^{2|k|(1+y)}dy+|b_-|^2\frac{|k|^4}{e^{4|k|}}\int_{-1}^1 e^{2|k|(1-y)}dy\\
\lesssim&(|b_+|^2+|b_-|^2)|k|^3.
\end{align*}
The estimate of the $|k|^2\|\Phi_k^b\|_{L^2}$ is similar and we omitted. 
}
\end{proof}

Before implementing the estimations, we recall Lemma \ref{lem:com:BA}. The product rule (without $c$ correction) should be simplified to a simple combination fact
\begin{lemma}For any positive $\zeta>0$, there exists a constant $C=C(\zeta)$ such that the following estimates hold
\begin{align}
\sum_{\ell=0}^n \binom{n}{\ell}^{-\zeta}\leq& C(\zeta); \label{prod}
\\
\sum_{\ell=1}^{n-1}\binom{n}{\ell}^{-\zeta}\leq& C(\zeta) n^{-\zeta },\quad \forall n\in \mathbb{N}\backslash \{0\}.\label{prod2}
\end{align}
\end{lemma}
\begin{proof}
First we observe that \eqref{prod} is a direct consequence of \eqref{prod2} thanks to the fact that the $\ell=0$ and $\ell=n$ term are bounded by $1$. Hence, in the remaining part of the proof, we focus on \eqref{prod2}.  
For a fixed $\zeta>0$ in \eqref{prod2}, we first choose a large enough $N(\zeta):=\max\{ 10, \lceil\frac{2}{\zeta}\rceil+1\}$. For $n\leq 4 N$, the estimate \eqref{prod2} is direct because $n^{-\zeta}$ has a fixed lower bound. Hence we only consider the $n>4N$ case. Now we expand the combinatoric numbers
\begin{align}
\sum_{\ell=1}^{n-1} &\left(\frac{(n-\ell)!\ell!}{n!}\right)^{\zeta}
=\left(\sum_{\ell=1}^N +\sum_{\ell=N+1}^{\lfloor\frac{n}{4}\rfloor}+\sum_{\ell=\lfloor\frac{n}{4}\rfloor}^{\lceil\frac{3n}{4}\rceil} +\sum_{\ell=\lceil\frac{n}{4}\rceil}^{n-N}+\sum_{\ell=n-N+1}^{n-1}\right) \left(\frac{(n-\ell)!\ell!}{n!}\right)^{\zeta} 
=: \sum_{j=1}^5\mathcal{I}_j.\label{prod_I}
\end{align}
Thanks to the symmetric property of the $\binom{n}{\ell}$, it is enough to estimate the first three terms in \eqref{prod_I}. For the first term, we use $ \binom{n}{\ell}\geq n$ for ${\ell\geq 1}$ to obtain that 
\begin{align}
\mathcal{I}_1\leq &N n^{-\zeta}\leq C(\zeta)n^{-\zeta}. \label{prod_I1}
\end{align}
For the second term, we recall the choice of $N$ to obtain that 
\begin{align}\nonumber
\mathcal{I}_2=&\sum_{\ell= N+1}^{\lfloor \frac{n}{4}\rfloor}\lf(\frac{\ell!}{n(n-1)...(n-\ell+1)}\rg)^{\zeta}\\ \nonumber
\leq&\sum_{\ell= N+1}^{\lfloor \frac{n}{4}\rfloor}\lf(\frac{1\cdot 2\cdot (...)\cdot N}{n(n-1)(n-2)...(n-N+1)}\rg)^{\zeta}\lf(\frac{\ell(\ell-1)...(N+1)}{(n-N)....(n-\ell+1)}\rg)^{\zeta}\\
 \leq& \sum_{\ell= N+1}^{\lfloor \frac{n}{4}\rfloor}\frac{ 2^{ \zeta N}(N!)^{\zeta}}{n^{\zeta N }}\leq \sum_{\ell=N +1}^{\lfloor \frac{n}{4}\rfloor}\frac{ 2^{ \zeta N}(N!)^{\zeta}}{n^{2+\zeta}}\leq C(\zeta)n^{-\zeta}. \label{prod_I2}
\end{align}
For the third term,  we apply the Stirling's formula $n!\sim n^{n+1/2}e^{-n}$ to obtain that for $n,\ell,n-\ell\geq N,\, \zeta>0$
 \begin{align}\nonumber
\sum_{\ell=\lfloor\frac{n}{4}\rfloor}^{\lceil\frac{3n}{4}\rceil}&\binom{ n}{ \ell} ^{-\zeta}\leq C \sum_{\ell=\lfloor\frac{n}{4}\rfloor}^{\lceil\frac{3n}{4}\rceil}\lf(\frac{\ell^{\ell+1/2}(n-\ell)^{n-\ell+1/2}}{n^{n+1/2}}\rg)^{\zeta}\\
\leq& C\sum_{\ell=\lfloor\frac{n}{4}\rfloor}^{\lceil\frac{3n}{4}\rceil}\lf(\frac{7}{8}\rg)^{n\zeta}(\sqrt{n})^{\zeta}\leq C(\zeta)n^{-\zeta} .\label{prod_I3} 
\end{align}Here we have used the fact that in the range of the sum, $\ell,n-\ell\leq  \frac7 8 n$. 
Combining the decomposition \eqref{prod_I} and the argument after, the estimates \eqref{prod_I1}, \eqref{prod_I2}, and \eqref{prod_I3}, we have proven the estimate \eqref{prod2}. 
\end{proof}
In the following, we will present four general lemmas regarding the ICC-method (Lemma \ref{lem:pavq}, Lemma \ref{lem:cm_J}, Lemma \ref{lem:induction}, Lemma \ref{lem:JtoD}). 
\begin{lemma}\label{lem:pavq}The $j$-th derivative ($1\leq j\leq 4$) of the $q^n$ has the following representation
\begin{align}
\pav^j (q^n)=\wt{q_{n,j}}\, n^{j}\,  q^{(n-j)_+}.\label{pav_j_qn}
\end{align}
Here $(n-j)_+:=\max\{n-j,0\}$ and the $\wt{q_{n,j}}$ are bounded functions such that \begin{subequations}\label{wtq_all}
\begin{align}
\|\wt {q_{n,j}}\|_{L^\infty}\leq& C( \|q\|_{W^{j,\infty}} ,\|v_{y}^{-1}\|_{L^\infty},\|v_y\|_{W^{j-1,\infty}});\label{wtq}\\
\lf\|\pa_y\wt {q_{n,j}}\rg\|_{L^\infty}\leq&  C( \|q\|_{W^{j+1,\infty}} ,\|v_y^{-1}\|_{L^\infty},\|v_{y} \|_{W^{j,\infty}}). \label{pay_wtq}
\end{align}
\end{subequations}
Here the constants are independent of $n.$  \siming{HS: As a result, the highest coefficients we can tolerate are $\wt {q_{n,3}}$ and $\pa_y \wt{q_{n,2}}$. }

As a consequence, the following commutator relation holds
\begin{align}
[\pav^j,q^n]f=\sum_{\ell=1}^j\binom{j}{\ell}\wt{q_{n,\ell}}n^\ell q^{(n-\ell)_+} \pav^{j-\ell}f. \label{cm_qn_pavj}
\end{align}
\end{lemma}
\begin{proof}
We start the proof by recalling the Fa\`a di Bruno's formula:
\begin{align}
\frac{d^j}{dx^j}f(g(x))=&\sum_ {(m_1,m_2,..., m_j)\in S_j}\frac{j!}{m_1!(1!)^{m_1}m_2! (2!)^{m_2}...m_j! (j!)^{m_j}} f^{(m_1+...+m_j)}(g(x))\cdot \prod_{\ell=1}^n(g^{(\ell)}(x))^{m_\ell},\label{FaadiBruno}\\
\n &\quad S_j:=\left\{(m_1,m_2,...,m_j)\in \mathbb{N}^j\bigg|\sum_{\ell=1}^j \ell m_\ell=j\right\}. 
\end{align} For a fixed $t\geq 0$, we write the physical variable $y$ as a function of the variable $v$, i.e., $y=y(v)$. A direct application of the formula \eqref{FaadiBruno} yields the expression of the quantity $\pav^j q(y),\, j\leq 4$,
\begin{align*}\n
\pav^j q(y)=&(\pav)^j q(y(v))=\sum_{(m_1,m_2,...,m_j)\in S_j}\frac{j!}{m_1!m_2!...m_j!}q^{(m_1+m_2+...+m_j)}(y(v))\prod_{\ell=1}^j\lf(\frac{\pav^{\ell }y(v)}{\ell!}\rg)^{m_\ell}\\
=&\sum_{(m_1,m_2,...,m_j)\in S_j}\frac{j!}{m_1!m_2!...m_j!}q^{(m_1+m_2+...+m_j)}(y(v))\prod_{\ell=1}^j\lf(\frac{\pav^{\ell-1 }v_y^{-1}}{\ell!}\rg)^{m_\ell}. 
\end{align*}
We observe that $\sum_{\ell=1}^jm_j\leq j\leq 4$. Hence the $j$-th $\pav$-derivative of the $q$-weight is bounded, 
\begin{align}\label{pa_vjq}
\|\pav^j q\|_{L^\infty}\leq C\lf(\|q\|_{W^{j,\infty}},\|v_y^{-1}\|_{L^\infty}, \| v_y\|_{W^{j-1,\infty}}\rg),\quad j\leq 4.
\end{align}  

{We apply the Fa\`a di Bruno formula again  to derive the following 
\begin{align*}
\pav^j (q^n(y))=&\pav ^j (q^n(y(v))) \\
=&\sum_{(m_1,m_2,...,m_j)\in S_j}\frac{j!}{m_1! m_2! \ ...\ m_j!}\lf(\frac{d}{dq}\rg)^{(\sum_{\ell=1}^j m_\ell)}(q^n)\prod_{\ell=1}^{j}\lf(\frac{\pav^{\ell}q}{\ell!}\rg)^{m_\ell} \\
=&\sum_{(m_1,m_2,...,m_j)\in S_j}\frac{j!}{m_1! m_2! \ ...\ m_j!}  \frac{n!}{(n-m_1-m_2...-m_j)!}q^{ n-m_1-...-m_j}\\
&\quad\quad\quad\quad\quad\quad\times\mathbbm{1}_{ n-m_1-...-m_j\geq 0} \prod_{\ell=1}^{j}\lf(\frac{\pav^{\ell}q}{\ell!}\rg) ^{m_\ell}\\
=:&\wt {q_{n,j}}\,  {n^ j }\, q^{\max\{n-j, 0\}},\quad   j\leq 4. 
\end{align*}
Here the explicit form of the coefficients are as follows 
\begin{align}\n
\wt {q_{n,j}}:=&\sum_{\substack{(m_1,m_2,...,m_j)\in S_j\\ m_1+m_2+...+m_j\leq n}}\frac{j!}{m_1! m_2! \ ...\ m_j!}  \frac{n!}{n^j(n-m_1-m_2...-m_j)!}\\
&\quad\quad\quad\quad\quad\quad\times q^{n-\max\{n-j,0\}-m_1-...-m_j}\prod_{\ell=1}^j\lf(\frac{\pav^{\ell}q}{\ell!}\rg) ^{m_\ell}.\label{wt_qnj}
\end{align}
Since for $(m_1,...,m_j)\in S_j$, the relation $\sum_{\ell=1}^j m_j\leq j\leq 4$ is satisfied  \eqref{FaadiBruno}. Combining this fact with the upper bound \eqref{pa_vjq} yields that
\begin{align}
|\pav^j (q^n)|=& |\wt{q_{n,j}}| n^j q^{\max\{n-j,0\}},\quad \|\wt {q_{n,j}}\|_{L^\infty}\leq C(\|q\|_{W^{j,\infty}},\|v_y^{-1}\|_{L^\infty},\|v_y\|_{W^{j-1,\infty}}), \quad 1\leq j\leq 4.
\end{align}\siming{Double check! (Check once.)} Hence we conclude the proof of \eqref{pav_j_qn}, \eqref{wtq}. 
Next we note that there are at most $2j$ copies of $\pav^{(\cdots)}q$ factors in the above expression \eqref{wt_qnj}, hence the estimate \eqref{pay_wtq}
is a direct consequence of differentiating the expression \eqref{wt_qnj}. Moreover, an application of the product rule and the relation \eqref{pav_j_qn} yields \eqref{cm_qn_pavj}. This concludes the proof of the lemma.  \siming{Double check the regularity requirements in \eqref{pay_wtq}. I mainly focus on the $n,q$ loss here! \myr{Check once.}}
}
\ifx The $q$-weight might also not be there.
Previous:
\siming{
Next, we prove the estimate \eqref{pay_wtq}. One can check that for $j\leq i\leq 4$, the bound is a direct consequence of the formula derived above. For $j>i$, we use the relation, 
\begin{align}
\wt{q_{j,i+1}}\frac{j^{i+1}}{q^{i+1}}q^j=\pav\lf( \pav^{i}q^j\rg)=\pav\lf(\wt{q_{j,i}}\frac{j^{i}}{q^{i}}q^j\rg)=\pav\wt{q_{j,i}}\frac{j^i}{q^i}q^j+\wt{q_{j,i}}\frac{(j-i) q'}{j v_y} \frac{j^{i+1}}{q^{i+1}}q^j. 
\end{align} 
Hence, 
\begin{align}
\lf\|\frac{q}{v_y j}\pa_y\wt{q_{j,i}}\rg\|_{L^\infty}\leq \|\wt{q_{j,i+1}}\|_{L^\infty}+\|q'\vyn\wt{q_{j,i}}\|_{L^\infty}\leq C.
\end{align}
Here the bound $C$ is independent of $j$.  This is \eqref{pay_wtq}.
}
{\bf Previous:} 
First, we consider the case where $j>i$, and obtain
\begin{align*}
(\vyn&\pa_y)^i (q^j(y))= (\pav)^ i q^j(y(v))\\
=&\sum_ {(m_1,m_2,..., m_i)\in S_i}\frac{i!}{m_1!(1!)^{m_1}m_2! (2!)^{m_2}...m_i! (i!)^{m_i}} \pa_y^{(m_1+...+m_i)}q^j(y)\bigg|_{y=y(v)}\cdot \prod_{\ell=1}^i(\pav^{(\ell)}y(v))^{m_\ell}\\
=&\sum_ {(m_1,m_2,..., m_i)\in S_i}\frac{i!}{m_1!(1!)^{m_1}m_2! (2!)^{m_2}...m_i! (i!)^{m_i}} \pa_y^{(m_1+...+m_i)}q^j(y)\bigg|_{y=y(v)}\cdot \prod_{\ell=1}^i(\pav^{(\ell-1)}\vyn)^{m_\ell}.
\end{align*}
Since we have that 
\begin{align}
\pa_y^{(n)}q^j=h  j^n q^{j-n}, \quad \|h\|_\infty\leq C,
\end{align}
the result \eqref{pav_i_qj} is a direct consequence.  

{\color{blue} For the $j\leq i\leq 4$ case, we apply the Fa\`a di Bruno formula to derive the following 
\begin{align*}
\|\pav^i (q^j(y(v)))\|_{L^\infty}=&\lf\|\sum_{(m_1,m_2,...,m_i)\in S_i}\frac{i!}{m_1! m_2! \ ...\ m_i!}\lf(\frac{d}{dq}\rg)^{(\sum_{\ell=1}^i m_\ell)}(q^j)\prod_{\ell=1}^{i}\lf(\frac{\pav^{(\ell)}q}{\ell!}\rg)^{m_\ell}\rg\|_\infty.
\end{align*}
Here the summation is implemented on the index set $S_i:=\{(m_1,m_2,...,m_i)|\sum_{\ell=1}^i \ell m_\ell =i,\quad \sum_{\ell=1}^im_\ell\leq j \}. $ Since $j\leq i\leq 4$, the combinatorial numbers are bounded and the norm is bounded as follows 
\begin{align*}
\|\pav^i (q^j(y(v)))\|_{L^\infty}\lesssim&\sum_{(m_1,m_2,...,m_i)\in S_i} \prod_{\ell=1}^{i}\lf\|\frac{\pav^{(\ell)}q}{\ell!}\rg\|_\infty^{m_\ell}\leq C(\|v_{y}^{-1}\|_{W^{i-1,\infty}}),\quad j\leq i\leq 4. 
\end{align*} 
}\fi
\end{proof}
\begin{lemma}\label{lem:cm_J}
The following estimates hold for any $n\geq 1$ and $j\in\{1,2,3\}$, %
 \siming{\footnote{\siming{(HS: For $k\neq 0$, we can replace the $a+b+c\leq j$ on the right hand side by $a+b+c= j$. For $k=0$, we don't have the $|k|^c$ component on the right hand side and we might need to include other terms in $a+b+c\leq j$.)}}}
\begin{subequations}\label{cm_J}
\begin{align}\n
&B_{m,n}\varphi^{n+1}\|[\pav^j,q^n](\chi_{m+n}|k|^m \Gamma_k^n f_k)\|_{L^2}
\\ &\lesssim  \sum_{n'=(n-2)_+}^{n}  B_{m,n'}\lambda^{(n-n')s}\varphi^{n'+1}\lf(\frac{n'!}{n!}\rg)^\sigma\|J_{m,n'}^{(\leq j)}f_k\|_{L^2} \n \\
&\quad+\mathbbm{1}_{n<j}\sum_{m'=(m-1)_+}^m\sum_{n'=0}^{n-1}\sum_{b+c\leq j}B_{m',n'}\lambda^{(|m,n|-|m',n'|)s}\varphi^{n'+1}\lf(\frac{|m',n'|!}{|m,n|!}\rg)^\sigma\|J_{m',n'}^{(0,b,c)}f_k\|_{L^2}.\label{cm_J_est} 
\end{align}
Here, $(n-2)_+=\max\{n-2,0\}$, $|m,n|=m+n$ and the implicit constants depend only on $\|v_y^{-1}\|_{L^\infty},\ \|v_y\|_{W^{2,\infty}}$ and independent of $m,n$.  Moreover, \begin{align} 
B_{m,n}&\varphi^{n+1}\lf\|\pav\lf([\pav,q^n](\chi_{m+n}|k|^m \Gamma_k^n f_k)\rg)\rg\|_{L^2} 
\lesssim  \sum_{n'=(n-1)_+}^{n}B_{m,n'}\varphi^{n'+1}\|J_{m,n'}^{(\leq 2)}f_k\|_{L^2};
\label{cm_pv_J_est}\\ 
  {B}_{m,n}&\varphi^{n+1}\lf\|\pav\lf([\pav^2,q^n]\lf(\chi_{m+n}|k|^m\Gamma_k^nf_k\rg)\rg)\rg\|_{L^2}  \lesssim  \sum_{m'=(m-1)_+}^m\sum_{n'=(n-2)_+}^{n}B_{m',n'}\varphi^{n'+1}\|J^{(\leq 3)}_{m',n'}f_k\|_{L^2}.
\label{cm_pv2J_est}
\end{align}\end{subequations}
Here the implicit constants depends only on $\|v_y^{-1}\|_{L^\infty}, \|v_y-1\|_{W^{2,\infty}}$.  

\ifx
?\\
\fi
\end{lemma}
\begin{proof}\ifx
We prove the relation \eqref{cm_qn_pv} with the expression \eqref{pav_i_qj}, 
\begin{align*}
[q^n&,\pav^3](\chi_n\Gamma_k^n\phi_k)\\
=&\sum_{i_0=1}^3\binom{3}{i_0}\pav^{i_0}(q^n)\ \pav^{3-i_0}(\chi_n\Gamma_k^n \phi_k)\\
=&\sum_{i_0=1}^3\binom{3}{i_0}\wt {q_{n,i_0}}n^{i_0} q^{n-i_0}\ \pav^{3-i_0}(\chi_n\Gamma_k^n \phi_k)\\
=&\sum_{i_0=1}^3\binom{3}{i_0}\wt {q_{n,i_0}}n^{i_0} q^{ -i_0}\ \pav^{3-i_0}(q^n \Gamma_k^n \phi_k)+\sum_{i_0=1}^3\binom{3}{i_0}\wt {q_{n,i_0}}n^{i_0} q^{ -i_0}\ [q^n, \pav^{3-i_0}](\chi_n \Gamma_k^n \phi_k)\\
=&\sum_{i_0=1}^3\binom{3}{i_0}\wt {q_{n,i_0}}J^{(i_0,3-i_0,0)}+\sum_{i_0=1}^3\sum_{i_1=1}^{3-i_0}\binom{3}{i_0}\binom{3-i_0}{i_1}\wt {q_{n,i_0}}\wt {q_{n,i_1}}n^{i_0} n^{i_1}q^{n -i_0-i_1} \pav^{3-i_0-i_1}(\chi_n \Gamma_k^n \phi_k)\\
=&\sum_{i_0=1}^3\binom{3}{i_0}\wt {q_{n,i_0}}J^{(i_0,3-i_0,0)}+\sum_{i_0=1}^3\sum_{i_1=1}^{3-i_0}\binom{3}{i_0}\binom{3-i_0}{i_1}\wt {q_{n,i_0}}\wt {q_{n,i_1}}n^{i_0} n^{i_1}q^{ -i_0-i_1} \pav^{3-i_0-i_1}(\chi_nq^n \Gamma_k^n \phi_k)\\
&+\sum_{i_0=1}^3\sum_{i_1=1}^{3-i_0}\binom{3}{i_0}\binom{3-i_0}{i_1}\wt {q_{n,i_0}}\wt {q_{n,i_1}}n^{i_0} n^{i_1}q^{ -i_0-i_1} [q^n,\pav^{3-i_0-i_1}](\chi_n \Gamma_k^n \phi_k)\mathbbm{1}_{i_0=i_1=1}\\
=&\sum_{i_0=1}^3\binom{3}{i_0}\wt {q_{n,i_0}}J^{(i_0,3-i_0,0)}+\sum_{i_0=1}^3\sum_{i_1=1}^{3-i_0}\binom{3}{i_0}\binom{3-i_0}{i_1}\wt {q_{n,i_0}}\wt {q_{n,i_1}}J_{m,n}^{(i_0+i_1,3-i_0-i_1,0)}+6\wt {q_{n,1}}^3 J_{m,n}^{(3,0,0)}.
\end{align*}\fi
{\bf Step \# 0: Setup.}
Without loss of generality, we assume that $k\geq 1$ in our proof. The $k=0$ and $k<0$ cases are treated in a similar fashion. 
We note that if $a+b+c=j>n$, the parameter $a$ in the $J^{(a,b,c)}_{m,n}f_k$ \eqref{J_nq} must be zero. Hence there is a natural distinction between the $j\leq n$ and $j>n$ cases. We prove the following relations in these two different regimes 
\begin{subequations}\label{cm_qn_pv}
\begin{align}
\label{cm_qn_pv_1}[\pav^{j},q^n]\lf(\chi_{m+n}|k|^m\Gamma_k^nf_k\rg)=&\sum_{a+b=j,a>0}\mf A_{m,n}^{a,b,0;j}J^{(a,b,0)}_{m,n}f_k,\, \quad j\in\{1,2,3\},\quad n\geq j;\\ 
{\varphi[\pav^{2},q]\lf(\chi_{m+1}|k|^m\Gamma_kf_k\rg)= }&\sum_{m'=(m-1)_+}^{m}\sum_{b+c\leq 2}\lan m\ran^{\myr{(m-m')_+}(1+\sigma)}\mf A_{m',0}^{0,b,c;2}J^{(0,b,c)}_{m',0}f_k;\n\\ 
|\varphi^{n}[\pav^{3},q^n]\lf(\chi_{m+n}|k|^m\Gamma_k^nf_k\rg)|\lesssim &\sum_{m'=(m-1)_+}^m\sum_{n'=(n-2)_+}^{n-1}\sum_{b+c\leq 3}\lan m\ran^{(n-n')(1+\sigma)}\varphi^{n'}|J^{(0,b,c)}_{m',n'}f_k|,\quad 1\leq n\leq2.\n
\end{align} Here, most of the indices in the coefficients $\mf A_{m,n}^{a,b,c;j}$ match their companion $J_{m,n}^{(a,b,c)}f_k$, and the extra index `$j$' highlights the associated commutator (e.g., $j=2$ if one is dealing with $[\pav^2,q^n]$). These coefficients satisfy the following estimates for $\ell\in\{0,1\},\, \ell+j\leq 3$,
\begin{align}\label{cm_qn_pv_1.5}
\|\pa_y^{\ell}\mf A_{m,n}^{a,b,0;j}\|_{L_y^\infty}\leq& C\lf(\|v_y^{-1}\|_{L^\infty_y},\|v_{y}-1\|_{W_y^{j-1+\ell,\infty}}\rg), \quad  j\leq n;\\
 \|\pa_y^{\ell}\mf A_{m,n}^{0,b,c;j}\|_{L_y^\infty}\leq& C\lf(\|v_y^{-1}\|_{L^\infty_y},\|v_{y}-1\|_{W_y^{j-1+\ell,\infty}}\rg)\lan m\ran^{\ell(1+\sigma)}, \quad \, j> n .\n 
\end{align}
\siming{{Double check the two set of estimates below! Check once.}} Moreover, \begin{align}
\label{cm_qn_pv_2}\varphi^n\lf|\pav\lf([\pav,q^n]\lf(\chi_{m+n}|k|^m\Gamma_k^nf_k\rg)\rg)\rg|\leq&\sum_{n'=(n-1)_+}^n\sum_{a+b+c\leq 2}\mf B_{m,n'}^{a,b,c}\varphi^{n'}\lf|J^{(a,b,c)}_{m,n'}f_k\rg|,\quad \|\mf B_{m,n}^{a,b,c}\|_{L_y^\infty}\lesssim 1,\quad n\geq 1.
\end{align}
Moreover, \begin{align}
\label{cm_qn_pv_3}\varphi^n\lf|\pav\lf([\pav^2,q^n]\lf(\chi_{m+n}|k|^m\Gamma_k^nf_k\rg)\rg)\rg|\leq&\sum_{n'=(n-2)_+}^n\sum_{a+b+c\leq3}\mf C_{m,n'}^{a,b,c}\varphi^{n'}\lf|J^{(a,b,c)}_{m,n'}f_k\rg|,\quad\|\mf C_{m,n}^{a,b,c}\|_{L_y^\infty}\lesssim 1, \quad n\geq 3;\\  \n
\varphi^n \lf|\pav[\pav^{2},q^n]\lf(\chi_{m+n}|k|^m\Gamma_k^n f_k\rg)\rg|\lesssim & \sum_{m'=(m-1)_+}^m\sum_{n'=(n-2)_+}^n\sum_{b+c\leq 3}\lan m\ran^{(m+n-m'-n')(1+\sigma)}\lf|J^{(0,b,c)}_{m',n'}f_k\rg|, \quad n\in\{1,2\}.
\end{align}
\end{subequations} Here the implicit constant only depends on the $\|v_y^{-1}\|_{L^\infty_y},\|v_{y}-1\|_{H_y^3}$. 

These estimates, combined with 
the bound \eqref{gevbd1212}, yields the results \eqref{cm_J_est}, \eqref{cm_pv_J_est} and \eqref{cm_pv2J_est}. For example, to derive the third estimate in  \eqref{cm_J_est}, we invoke the \eqref{gevbd1212} to get that 
\begin{align*}
&B_{m,n}\varphi^{n+1} \|[\pav^3,q^n](\chi_{m+n}|k|^m\Gamma_k^nf_k)\|_{L^2}\lesssim \sum_{a+b=3,a>0}\|\mf A_{m,n}^{a,b,0;3}\|_{L^\infty}B_{m,n}\varphi^{n+1}\|J_{m,n}^{(a,b,0)}f_k\|_{L^2}\\ &\hspace{3cm}\lesssim \sum_{a+b=3,a>0}B_{m,n}\varphi^{n+1}\|J_{m,n}^{(a,b,0)}f_k\|_{L^2},\quad n\geq 3;\\
&B_{m,n}\varphi^{n+1} \|[\pav^3,q^n](\chi_{m+n}|k|^m\Gamma_k^nf_k)\|_{L^2}\\
&\hspace{0.5 cm}\lesssim \sum_{m'=(m-1)_+}^m\sum_{n'=(n-2)_+}^{n-1}\sum_{b+c=3}\lan m\ran^{(n-n')(1+\sigma)}\frac{B_{m,n}}{B_{m',n'}}B_{m',n'}\varphi^{n'+1}\|J_{m',n'}^{(0,b,c)}f_k\|_{L^2}\\ 
&\hspace{0.5 cm}\lesssim \sum_{m'=(m-1)_+}^m\sum_{n'=(n-2)_+}^{n-1}\sum_{b+c=3}B_{m',n'}\lambda^{(|m,n|-|m',n'|)s}\lf(\frac{|m',n'|!}{|m,n|!}\rg)^\sigma\varphi^{n'+1}\|J_{m,n'}^{(0,b,c)}f_k\|_{L^2},\quad n\leq 2.
\end{align*}
We omit the derivation of other inequalities since they are similar.  

Thanks to the above discussion, the remaining task is to develop \eqref{cm_qn_pv}. It turns out that the most technical step is  the proof of \eqref{cm_qn_pv_1}, \eqref{cm_qn_pv_1.5}. Because to get the estimates \eqref{cm_qn_pv_2}, \eqref{cm_qn_pv_3}, it is also necessary to derive the explicit form of the coefficient $\mf A_{m,n}^{a,b,c;j}$ in \eqref{cm_qn_pv_1}. Hence, we further decompose the proof of \eqref{cm_qn_pv_1} into the high regularity regime and the low regularity regime.


\noindent
{\bf Step \# 1a: Proof of $\eqref{cm_qn_pv_1}_{j\leq n}, \ \eqref{cm_qn_pv_1.5}_{j\leq n}$. }
We apply induction to prove the result:
\begin{align}\label{indc_hp_cm}
\lf[\pav^j,q^n\rg]&(\chi_{m+n}|k|^m\Gamma_k^n f_k)=\sum_{\substack{a+b=j
}} \mf A^{a,b,0;j}_{m,n} J^{(a,b,0)}_{m,n}f_k, \quad 1\leq j\leq 3,\\
\mf A^{a,b,0;j}_{m,n}=&{\lf(\frac{n}{m+n}\rg)^a}\sum_{\substack{i_1+\cdots+i_h=a\\ i_1,\cdots, i_h\neq 0}}(-1)^{h+1}\binom{j}{i_1}\binom{j-i_1}{i_2}\binom{j-i_1-i_2}{i_3}\cdots\binom{j -i_1-\cdots-i_{h-1}}{i_{h}}\prod_{h'=1}^h \wt{q_{n,i_{h'}}}. \label{C_exp}
\end{align}
Here $(i_1,\cdots, i_h)$ is a partition of the integer $a$ and the terms with $a=0$ are vacuous. 
The $\wt{q_{n,i_h}}$ are defined in \eqref{wt_qnj}.  Moreover, thanks to the bounds \eqref{wtq}, \eqref{pay_wtq}, the coefficients $\mf A_{m,n}^{a,b,0;j}$ has the following bound for $ n\geq j,\ m \in \mathbb{N}$,
\begin{align}\label{mf_A_bound}
\|\mf A_{m,n}^{a,b,0;j}\|_{L_y^\infty}\leq& C(\|v_y^{-1}\|_{L^\infty_y},\|v_{y}-1\|_{W_y^{j-1,\infty}}),\quad a+b=j\in \{1,2,3\};\\ 
\|\pa_y\mf A_{m,n}^{a,b,0;j}\|_{L_y^\infty}\leq& C(\|v_y^{-1}\|_{L^\infty_y},\|v_{y}-1\|_{W_y^{j,\infty}}),\quad a+b=j\in\{1,2\}.
\end{align}
This is the first estimate in \eqref{cm_qn_pv_1.5}. \siming{(Check!)\myr{Check once.}} 

As a start, it is clear to see that for $j=1$, the commutator matches the expression \eqref{indc_hp_cm},  
\begin{align*}
[\pav,q^n](\chi_{m+n}|k|^m\Gamma_k^n f_k)=|k|^m\pav (q^n)\chi_{m+n}\Gamma_k^nf_k=\wt {q_{n,1}}\frac{n}{q}(\chi_{m+n} |k|^m q^n\Gamma_k^nf_k)=\wt {q_{n,1}}\lf(\frac{n}{m+n}\rg) J^{(1,0,0)}_{m,n}f_k. 
\end{align*}

Assume that the expression \eqref{indc_hp_cm} holds for the $1,\cdots, (j-1)$-th levels. Then for the $j$-th level, we have that an application of the Lemma \ref{lem:commutator_AB} and the expressions  \eqref{pav_j_qn}, \eqref{cm_qn_pavj} yields the relation 
\begin{align*}
\big[\pav^{j}&,q^n\big](\chi_{m+n}|k|^m\Gamma_k^nf_k)\\
=&\sum_{i_1=1}^{j}\binom{{j}}{i_1}\wt {q_{n,i_1}}\frac{n^{i_1}}{q^{i_1}} q^n\ \pav^{{j}-i_1}(\chi_{m+n}|k|^m\Gamma_k^n f_k)\\
=&\sum_{i_1=1}^{j}\binom{{j}}{i_1}\wt {q_{n,i_1}}\lf(\frac{n}{m+n}\rg)^{i_1}\frac{(m+n)^{i_1}}{q^{ i_1}}\ \pav^{{j}-i_1}(\chi_{m+n}|k|^mq^n \Gamma_k^n f_k)\\
&+\sum_{i_1=1}^{j}\binom{{j}}{i_1}\wt {q_{n,i_1}}\lf(\frac{n}{m+n}\rg)^{i_1}\frac{(m+n)^{i_1}}{ q^{ i_1}}\ [q^n, \pav^{{j}-i_1}](\chi_{m+n} |k|^m\Gamma_k^nf_k). 
\end{align*} Now we invoke the induction hypothesis \eqref{indc_hp_cm} and the relation $\frac{(m+n)^{i_1}}{q^{i_1}}J_{m,n}^{(a,b,0)}f_k\mathbbm{1}_{a+i_1+b\leq n}=J_{m,n}^{(a+i_1,b,0)}f_k\mathbbm{1}_{a+i_1+b\leq n}$ to obtain
\begin{align*}\bigg[\pav^{j}&,q^n\bigg](\chi_{m+n}|k|^m\Gamma_k^nf_k)\\
=&\sum_{i_1=1}^{j}\myr{\lf(\frac{n}{m+n}\rg)^{i_1}}\binom{{j}}{i_1}\wt {q_{n,i_1}}J^{(i_1,{j}-i_1,0)}_{m,n}f_k-\sum_{i_1=1}^{j}\sum_{\substack{a+b=j-i_1}}\lf(\frac{n}{m+n}\rg)^{i_1}\binom{{j}}{i_1}\mf A_{m,n}^{a,b,0;j-i_1}\wt {q_{n,i_1}}\frac{(m+n)^{i_1}}{ q^{ i_1}}J_{m,n}^{(a,b,0)}f_k\\
=&\sum_{i_1=1}^{j}\lf(\frac{n}{m+n}\rg)^{i_1}\binom{{j}}{i_1}\wt {q_{n,i_1}}J^{(i_1,{j}-i_1,0)}_{m,n}f_k\\
&+\sum_{i_1=1}^{j}\sum_{{(a+i_1)+b=j
}}\sum_{\substack{i_2+\cdots+i_{h}=a\\ i_2,\cdots,i_{h}\neq 0}}(-1)^{h+1}\lf(\frac{n}{m+n}\rg)^{i_1+a}\binom{{j}}{i_1}\binom{j-i_1}{i_2}\\
&\qquad\qquad\qquad\qquad\times\binom{j-i_1-i_2}{i_3} \cdots\binom{j-i_1 -i_2-\cdots-i_{h-1}}{i_{h}}\wt {q_{n,i_1}}\lf(\prod_{h'=2}^h \wt{q_{n,i_{h'}}}\rg)  J_{m,n}^{(a+i_1,b,0)}f_k\\
=&\sum_{ {a+b=j}}\sum_{ \substack{i_1+\cdots+i_{h}=a\\ i_1,\cdots, i_{h}\nq0}}\lf(\frac{n}{m+n}\rg)^{a}(-1)^{h+1}\binom{{j}}{i_1}\binom{j-i_1}{i_2}\binom{j-i_1-i_2}{i_3} \cdots\binom{j-i_1 -i_1-\cdots-i_{h-1}}{i_{h}}\\
&\qquad\qquad\qquad\qquad\times\lf(\prod_{h'=1}^h \wt{q_{n,i_h}}\rg)  J_{m,n}^{(a,b,0)}f_k\\
=&\myr{\sum_{\substack{a+b=j}}\mf A_{m,n}^{a,b,0; j} J^{(a,b,0)}_{m,n}f_k}.
\end{align*}
Here in the second to last line, we have redefined $a$. 
\siming{We should still double check whether it is equal. Some details about decomposition is needed. }
As a result, we have obtained \eqref{cm_qn_pv_1} for $j\leq n$. 

\noindent
{\bf Step \# 1b: Proof of $\eqref{cm_qn_pv_1}_{j> n}, \ \eqref{cm_qn_pv_1.5}_{j> n}$.} Next we consider the $j>n$ case. Since $j\leq 3$, we have $n\in\{1,2\}.$ Hence we only have $5$-sub-cases to consider. The main technical difficulty here is to keep track of the coefficients before $J^{(0,b,c)}_{m,n}f_k$. As the number of derivative $\pav$ increases, one will quickly find that the expression of the coefficients becomes complicated. To make the presentation clear, we will first do a precise computation on one of the cases. For other cases, we will use the following notation convention: We say that the coefficient $\mf F_{m,n}^{b,c;j}$ (associated with $[\pav^j,q^n]$ commutators) is \emph{admissible} if it satisfies the following two properties:
\begin{itemize}
\item $\mf F_{m,n}^{b,c;j}$ depends only on $v_y^{-1}, \lf\{\pa_y^\ell v_y\rg\}_{\ell=0}^{j-1}, |k|^{-1},t;$
\item $\|\pav^\ell \mf F_{m,n}^{b,c;j}\|_{L^\infty}\lesssim \lan m\ran^{\ell(1+\sigma)}, \ \ell=0,1$. Here the implicit constant depends only on $\|v_y^{-1}\|_{L_{t,y}^\infty}, \ \|v_y\|_{L_t^\infty W_y^{2,\infty}}$. 
\end{itemize} 
When we do the computation, the admissible coefficients will be changing line by line to incorporate all the contributions. We choose to drop the $a$ index because it is automatically zero in the low regularity setting. \siming{To keep track of the computation, we always absorb the $Jf_k$-terms generated along the way by a term of the form ${\mf F}Jf_k$. And we will keep the commutator terms on the back of the expression. }

 We invoke the commutator relation \eqref{comm_rel}, the cutoff bound \eqref{chi:prop:3}, the definition \eqref{pav_j_qn} and the computation facts that $\text{ad}_{\pav}^\ell(q)=\widetilde{q_{1,\ell}}\lesssim 1$, $\text{ad}_{\pav}^\ell(q^2)=\widetilde{q_{2,\ell}}2^{\ell}q^{(2-\ell)_+}\lesssim 1$ to obtain that 
\begin{align*}
\varphi^{n}[\pav,q^n]&(\chi_{m+n}|k|^m\Gamma_k^nf_k)=\varphi^n n\wt {q_{n,1}} \chi_{m+n} q^{n-1}(\pav+ikt)\lf(|k|^m \Gamma_k^{n-1} f_k\rg)\\
=&\varphi^n \lf(n\wt {q_{n,1}} \chi_{m+n}\rg)\lf(\mathbbm{1}_{n=2} [q^{n-1},\pav]\lf( \Gamma_k^{n-1}(\chi_{ m }|k|^mf_k)\rg)+J_{m,n-1}^{(0,1,0)}f_k+itJ_{m,n-1}^{(0,0,1)}f_k\rg)\\
=&\varphi^n\lf(n\wt {q_{n,1}} \chi_{m+n}\rg) J_{m,n-1}^{(0,1,0)}f_k +\varphi^{n}i t\lf(n\wt {q_{n,1}} \chi_{m+n}\rg) J_{m,n-1}^{(0,0,1)}f_k\\
&-\mathbbm{1}_{n=2}(2\widetilde {q_{2,1}}\widetilde{q_{1,1}}\chi_{m+2})\varphi^2\lf(J_{m,0}^{(0,1,0)}f_k+i tJ_{m,0}^{(0,0,1)}f_k\rg).
\end{align*}
It is direct to see that 
\begin{align*}
|\varphi^{n}[\pav,q^n](\chi_{m+n}|k|^m\Gamma_k^nf_k)|\lesssim&\sum_{n'=(n-2)_+}^{n-1}\sum_{b+c=1}\lan m\ran ^{(n-1-n')_+(1+\sigma)}\varphi^{n'}\lf|J^{(0,b,c)}_{m,n'} f_k\rg|,\quad n\in\{1,2\}.
\end{align*}
Moreover, we can clearly see that the coefficients are \emph{admissible} because $n=1,2$ and the derivative of $\chi_{m+n}$ is bounded by $C\lan m\ran^{1+\sigma}$,  \eqref{chi:prop:3}. Next we consider the following case, 
 \begin{align}\n
\varphi& [\pav^2,q](\chi_{m+1}|k|^m\Gamma_kf_k)\\  \n
&= \varphi \lf(\text{ad}_{\pav}^2[q](\chi_{m+1}\Gamma_k(\chi_m|k|^m f_k))+2\text{ad}_{\pav}[q]\pav (\chi_{m+1}\Gamma_k( |k|^mf_k))\rg)\\ \n
&=  \wt {q_{1,2}}\chi_{m+1} \lf(\varphi J_{m,0}^{(0,1,0)}f_k+i(\varphi t) J^{(0,0,1)}_{m,0}f_k\rg)+ 2\wt{q_{1,1}}\chi_{m+1}\lf(\varphi J_{m,0}^{(0,2,0)}f_k+it\varphi J^{(0,1,1)}_{m,0}f_k\rg)\\ \n
&\quad+ \mathbbm{1}_{m\geq 1}2\wt{q_{1,1}}[\pav\chi_{m+1}] \lf(\varphi J_{m-1,0}^{(0,1,1)}f_k+i(\varphi t) J^{(0,0,2)}_{m-1,0}f_k\rg)+ \mathbbm{1}_{m=0}2\wt{q_{1,1}}[\pav\chi_{1}] \lf(\varphi J_{0,0}^{(0,1,0)}f_k+i(\varphi t) J^{(0,0,1)}_{0,0}f_k\rg) \\
 &=\sum_{m'=(m-1)_+}^m\sum_{b+c\leq 2}\lan m\ran^{(m-m')_+(1+\sigma)}\mf F_{m',0}^{b,c;2} J^{(0,b,c)}_{m',0}f_k. \label{cm_qn_pv2}
 \end{align}
\siming{HS: The indexing system is not perfect because we will need to hide the starting $n$ (which appears on the left hand side) in $\mf F$ and record only the ending $n'$ ($\mf F_{m,n'}J_{m,n'}f_k$). } This is consistent with \eqref{cm_qn_pv_1}. 
With the above two examples, we can summarize the idea. First of all, we apply the commutator expansion \eqref{comm_rel} to expand the commutator properly. Then we commute the cut-off function $\chi_{\cdots}$ out of the main battle ground. Finally, we continuously commute the boundary weight $q^n$ to obtain various $J^{(0,b,c)}_{m,n}f_k$ terms. Consequently,\begin{align}\n
\varphi^2&[\pav^2,q^2](\chi_{m+2}|k|^m\Gamma_k^2f_k)\\ \n
&=\varphi^2 \text{ad}_{\pav}^2[q^2]\lf(\chi_{m+2}\Gamma_k^2(\chi_m |k|^mf_k)\rg) +2\varphi^2\text{ad}_{\pav}[q^2]\lf(\pav\lf(\chi_{m+2}\Gamma_k^2(|k|^mf_k)\rg)\rg)\\ \n
&=\sum_{b+c=2}\mf F_{m,0}^{b,c;2}J_{m,0}^{(0,b,c)}f_k+4\wt {q_{1,1}} \varphi^2  (\pav \chi_{m+2}) q\Gamma_k^2(|k|^mf_k)+4\wt{ q_{1,1}}\varphi^2 \chi_{m+2} q(\pav^2+ikt\pav)\Gamma_k (|k|^m f_k)\\ \n
&=\sum_{\substack{n'\in\{0,1\}\\ b+c=2}}\lan m\ran^{(\myr{1-n'})(1+\sigma)}\mf F_{m,n'}^{b,c;2}\varphi^{n'}J_{m,n'}^{(0,b,c)}f_k  +4\wt{ q_{1,1}}\varphi^2 \chi_{m+2} ([q,\pav^2]+ikt[q,\pav])\Gamma_k (\chi_m |k|^m f_k)\\
&=\sum_{n'=0}^1\sum_{b+c=2}\lan m\ran^{(\myr{1-n'})(1+\sigma)}\mf F_{m,n'}^{b,c;2}\varphi^{n'}J_{m,n'}^{(0,b,c)}f_k.\label{cm_q2_pv2}
\end{align}\siming{It is extremely important to keep the power of $\lan m\ran $ right.}
This is consistent with \eqref{cm_qn_pv_1}. 
\ifx
\siming{Explicit Expansion (Checking Purpose only): 
\begin{align*} & \lf(-4\wt {q_{2,1}}\wt{q_{1,1}}+4\wt{q_{2,2}}\rg)\chi_{m+2}\lf(J^{(0,2,0)}_{m,0}f_k+2it J^{(0,1,1)}_{m,0}f_k-t^2J^{(0,0,2)}_{m,0}f_k\rg)\\
&+4\wt{q_{2,1}}\vyn\chi_{m+2}'\lf(J^{(0,1,0)}_{m,0}f_k+itJ^{(0,0,1)}_{m,0}f_k\rg)+4\wt {q_{2,1}}\chi_{m+2}\lf(J^{(0,2,0)}_{m,1}f_k+it J^{(0,1,1)}_{m,1}f_k\rg)\\
&-4\lf(\wt{q_{2,1}}\wt{q_{1,1}}\vyn\chi_{m+2}'+\wt{q_{2,1}}\wt {q_{1,2}}\chi_{m+2}\rg)\lf(J^{(0,1,0)}_{m,0}f_k+it J^{(0,0,1)}_{m,0}f_k\rg);
\end{align*}}
\fi
Through similar but more complicated computations, we obtain the estimate of the other two commutator. Since we do not need to keep track of their derivatives, we choose to bound them directly (with implicit constants depends only on $\|v_y^{-1}\|_{L_{t,y}^\infty}, \ \|v_y\|_{L_t^\infty W^{2,\infty}_y}$),
\begin{align*}
\varphi |[\pav^3,&q](\chi_{m+1}|k|^m\Gamma_kf_k)|\lesssim \sum_{\ell=0}^2|\pav^{3-\ell} q|\varphi \lf|\pav^{\ell}(\chi_{m+1}\Gamma_k |k|^m f_k)\rg|\\
\lesssim&\sum_{\ell=0}^2\sum_{\ell_1=0}^{\ell }\varphi \lf|\pav^{\ell-\ell_1}\chi_{m+1}\rg|\lf|\lf(\pav^{1+\ell_1} +ikt\pav^{\ell_1}\rg)(|k|^m f_k)\rg|\\
\lesssim &\sum_{\ell=0}^2\sum_{\ell_1=0}^{ \ell }\mathbbm{1}_{(\ell,\ell_1)\neq (2,0)}\lan m\ran^{(1+\sigma)}\sum_{b+c=1+\ell_1}\lf|J^{(0,b,c)}_{m,0} f_k\rg|+\mathbbm{1}_{m\geq 1}\lan m\ran^{2(1+\sigma)}\sum_{b+c=1}|J^{(0,b,c+1)}_{m-1, 0} f_k|\\
&+\mathbbm{1}_{m=0}\sum_{b+c=1}|J^{(0,b,c)}_{0,0} f_k|.
\end{align*}
Finally, we estimate the most complicated expression:
\begin{align*}
\varphi^2& \big|[\pav^3,q^2](\chi_{m+2}|k|^m\Gamma_k^2f_k)\big|\lesssim \sum_{\ell=0}^2|\pav^{3-\ell} q^2|\varphi^2 \lf|\pav^{\ell}(\chi_{m+2}\Gamma_k^2 |k|^m f_k)\rg|\\
\lesssim&\sum_{\ell=0}^1\varphi^2|\pav^\ell(\chi_{m+2}|k|^m\Gamma_k^2f_k)|\ +\ \varphi^2|\pav q||q\pav^2(\chi_{m+2}|k|^m\Gamma_k^2f_k)|\\
\lesssim&\sum_{b+c=3}\lan m\ran^{(1+\sigma)}\lf|J_{m,0}^{(0,b,c)}f_k\rg|\ +\ \varphi^2|\pav q||q[\pav^2,\chi_{m+2}](|k|^m\Gamma_k^2f_k)|+\ \varphi^2|\pav q|\chi_{m+2}|q\pav^2(|k|^m\Gamma_k^2f_k)|\\
\lesssim&\sum_{b+c=3}\lan m\ran^{(1+\sigma)}\lf|J_{m,0}^{(0,b,c)}f_k\rg|\ +\ \varphi^2\sum_{\ell=0}^1|\pav^{2-\ell}\chi_{m+2}||\pav^\ell(|k|^m\Gamma_k^2f_k)|\\
&+\ \varphi^2\chi_{m+2}|[q,\pav^2\Gamma_k](|k|^m\Gamma_kf_k)|+\ \varphi^2\chi_{m+2}|\pav^2\Gamma_k(\chi_{m+1}|k|^mq\Gamma_kf_k)|\\
\lesssim&\sum_{n'=0}^1\sum_{b+c=3}\lan m\ran^{(n-n')(1+\sigma)}\varphi^{n'}\lf|J_{m,n'}^{(0,b,c)}f_k\rg|\ +\ \varphi^2\sum_{\ell=0}^1\lan m\ran^{(1+\sigma)}|\pav^\ell\Gamma_k^2(\chi_{m}|k|^mf_k)|\mathbbm{1}_{\text{supp }\chi_{m+2}}\\
&+\ \varphi^2\chi_{m+2}|[q,\pav^2\Gamma_k]\Gamma_k(\chi_m|k|^mf_k)|\\
\lesssim&\sum_{n'=0}^1\sum_{b+c\leq 3}\lan m\ran^{(n-n')(1+\sigma)}\varphi^{n'}\lf|J_{m,n'}^{(0,b,c)}f_k\rg|. 
\end{align*}
This concludes the proof of \eqref{cm_qn_pv_1} in the $j<n$ case. 
\ifx
\myr{\siming{Explicit Expansion (Checking Purpose only): }
\begin{align*}
[\pav^3,&q](\chi_{m+1}|k|^m\Gamma_kf_k)\\
=&3\wt{q_{1,1}}\chi_{m+1}\lf(J^{(0,3,0)}_{m,0}f_k+itJ_{m,0}^{(0,2,1)}f_k\rg)+3\wt{q_{1,1}} \pav^2\chi_{m+1}(J^{(0,1,0)}_{m,0}f_k+itJ^{(0,0,1)}_{m,0}f_k)\\
 &+6\wt{q_{1,1}} \pav\chi_{m+1}(J^{(0,2,0)}_{m,0}f_k+itJ^{(0,1,1)}_{m,0}f_k)+3\wt{q_{1,2}}\vyn\chi_{m+1}'\lf(J_{m,0}^{(0,1,0)}f_k+itJ^{0,0,1}_{m,0}f_k\rg)\\
 &+3\wt{q_{1,2}}\chi_{m+1}\lf(J_{m,0}^{(0,2,0)}f_k+itJ^{0,1,1}_{m,0}f_k\rg)+\wt{q_{1,3}}\chi_{m+1}\lf(J_{m,0}^{(0,1,0)}f_k+itJ^{0,0,1}_{m,0}f_k\rg);\\
[\pav^3,&q^2](\chi_{m+2}|k|^m\Gamma_k^2f_k)\\
=&-6\wt{q_{2,1}}\pav^2(\chi_{m+2}\pav q)(J_{m,0}^{(0,1,0)}f_k+itJ_{m,0}^{(0,0,1)}f_k)-12\wt{q_{2,1}}\pav(\chi_{m+2}\pav q)(J_{m,0}^{(0,2,0)}f_k+itJ_{m,0}^{(0,1,1)}f_k)\\
&-6\wt{q_{2,1}}\chi_{m+2}\pav q(J_{m,0}^{(0,3,0)}f_k+itJ_{m,0}^{(0,2,1)}f_k)\\
&+6\wt{q_{2,1}}\pav^2 \chi_{m+2}(J_{m,0}^{(0,1,0)}f_k+itJ_{m,0}^{(0,0,1)}f_k)+12\wt{q_{2,1}}\pav \chi_{m+2}(J_{m,0}^{(0,2,0)}f_k+itJ_{m,0}^{(0,1,1)}f_k)\\
&+6\wt{q_{2,1}}\chi_{m+2}(J_{m,0}^{(0,3,0)}f_k+itJ_{m,0}^{(0,2,1)}f_k)\\
&-6\wt{q_{2,1}}\wt{q_{1,2}}\chi_{m+2}(J_{m,0}^{(0,2,0)}f_k+2itJ_{m,0}^{(0,1,1)}f_k-t^2J_{m,0}^{(0,0,2)}f_k)\\
&-12\wt{q_{2,1}}\wt{q_{1,1}}\pav\chi_{m+2}(J_{m,0}^{(0,2,0)}f_k+2itJ_{m,0}^{(0,1,1)}f_k-t^2J_{m,0}^{(0,0,2)}f_k)\\
&-12\wt{q_{2,1}}\wt{q_{1,1}}\chi_{m+2}(J_{m,0}^{(0,3,0)}f_k+2itJ_{m,0}^{(0,2,1)}f_k-t^2J_{m,0}^{(0,1,2)}f_k)\\
&+12\wt{q_{2,2}}  \chi_{m+2}(J_{m,0}^{(0,3,0)}f_k+2itJ_{m,0}^{(0,2,1)}f_k-t^2J_{m,0}^{(0,1,2)}f_k)\\
&+12\wt{q_{2,2}} \pav \chi_{m+2}(J_{m,0}^{(0,2,0)}f_k+2itJ_{m,0}^{(0,1,1)}f_k-t^2J_{m,0}^{(0,0,2)}f_k)\\
&+8\wt{q_{2,3}}  \chi_{m+2}(J_{m,0}^{(0,2,0)}f_k+2itJ_{m,0}^{(0,1,1)}f_k-t^2J_{m,0}^{(0,0,2)}f_k).
\end{align*}
Hence direct estimate yields \eqref{cm_qn_pv_1}. }

\fi

\noindent{\bf Step \# 2: Proof of \eqref{cm_qn_pv_2}.} 
We apply the expression $\eqref{cm_qn_pv_1}_{j=1}$ to derive the following  
\begin{align}\n
\varphi^n\pav  [\pav,q^n]&(\chi_{m+n} |k|^m \Gamma_k^nf_k) 
=\varphi^n\pav \mf A_{m,n}^{1,0,0;1} J^{(1,0,0)}_{m,n}f_k+ \varphi^n\mf A_{m,n}^{1,0,0;1}  \pav J^{(1,0,0)}_{m,n}f_k\\
\n =&\begin{cases} \varphi^n\pav \mf A_{m,n}^{1,0,0;1}J^{(1,0,0)}_{m,n}f_k -\varphi^n\mf A_{m,n}^{1,0,0;1}\frac{\pav q}{m+n}J^{(2,0,0)}_{m,n}f_k+\varphi^n\mf A_{m,n}^{1,0,0;1}   J^{(1,1,0)}_{m,n}f_k,&\quad n\geq 2;\\
& \\
\varphi\pav \mf A_{m,1}^{1,0,0;1}J^{(1,0,0)}_{m,1}f_k +\mf A_{m,1}^{1,0,0;1} (  \varphi J^{(0,2,0)}_{m,0}f_k+i(\varphi t) J^{(0,1,1)}_{m,0}f_k),&\quad n= 1,\end{cases}\\
=:&\sum_{n'=(n-1)_+}^n\sum_{ a+b+c\leq 2}\mf B_{m,n'}^{a,b,c}\varphi^{n'}J^{(a,b,c)}_{m,n'}f_k.\label{dfn_B}
\end{align}
Thanks to the estimate \eqref{cm_qn_pv_1.5}, we have that 
\begin{align*}
\lf\|\mf A_{m,n}^{1,0,0;1}\rg\|_{L^\infty}+\lf\|\pa_y \mf A_{m,n}^{1,0,0;1}\rg\|_{L^\infty}\leq C(\|v_y^{-1}\|_{L^\infty}, \|v_y-1\|_{L_t^\infty W_y^{1,\infty}}).
\end{align*} 
Combining this estimate with the expression  \eqref{dfn_B}, we have obtained \eqref{cm_qn_pv_2} for $n\geq 1$.

\noindent{\bf Step \# 3: Proof of  \eqref{cm_qn_pv_3}.}  
We start by considering the case $n\geq 3$. We apply the expression $\eqref{cm_qn_pv_1}_{j=2}$ to derive the following  
\begin{align*}\n
\varphi^n\pav  [\pav^2,q^n]&(\chi_{m+n} |k|^m \Gamma_k^nf_k)\\ 
&\hspace{-2cm}=\varphi^n\pav \mf A_{m,n}^{2,0,0;2} J^{(2,0,0)}_{m,n}f_k+\varphi^n \mf A_{m,n}^{2,0,0;2} \pav J^{(2,0,0)}_{m,n}f_k+ \varphi^n\pav \mf A_{m,n}^{1,1,0;2}   J^{(1,1,0)}_{m,n}f_k+ \varphi^n\mf A_{m,n}^{1,1,0;2}  \pav J^{(1,1,0)}_{m,n}f_k\\
\n &\hspace{-2cm}=  \varphi^n\pav \mf A_{m,n}^{2,0,0;2}J^{(2,0,0)}_{m,n}f_k -\varphi^n\mf A_{m,n}^{2,0,0;2}\frac{2\pav q}{m+n}J^{(3,0,0)}_{m,n}f_k+\varphi^n \pav\mf A_{m,n}^{1,1,0;2}   J^{(1,1,0)}_{m,n}f_k\\
&\hspace{-2cm} \quad-\varphi^n  \mf A_{m,n}^{1,1,0;2} \frac{\pav q}{m+n}  J^{(2,1,0)}_{m,n}f_k +\varphi^n \mf A_{m,n}^{1,1,0;2}   J^{(1,2,0)}_{m,n}f_k\\
&\hspace{-2cm}=:\sum_{a+b\leq 3}\mf C_{m,n}^{(a,b,0)}f_k.
\end{align*}
Thanks to the estimate \eqref{cm_qn_pv_1.5}, we have that 
\begin{align*}
\sum_{\ell=0}^1\sum_{a+b=2,a>0}\lf\|\pa_y^\ell \mf A_{m,n}^{a,b,0;2}\rg\|_{L^\infty}\leq C(\|v_y^{-1}\|_{L^\infty}, \|v_y-1\|_{L_t^\infty W_y^{2,\infty}}).
\end{align*} 
Hence the coefficients $\{\mf C_{m,n}^{(a,b,0)}\}_{n\geq 3}$ are consistent with \eqref{cm_qn_pv_3}. 
 
For the $n\leq 2,\ j=1,2$ case, we invoke the second relation in \eqref{cm_qn_pv_1}, and the bound \eqref{cm_qn_pv_1.5},  
\begin{align*}
&\varphi^{n}\lf|\pav[\pav^{j},q^n]\lf(\chi_{m+n}|k|^m\Gamma_k^nf_k\rg)\rg|\\
&\lesssim \sum_{m'=(m-1)_+}^{m}\sum_{n'=(n-2)_+}^{n-1}\sum_{b+c=j}\lan m\ran^{(m+n-1-m'-n')_+(1+\sigma)}\lf(|\pav\mf A_{m',n'}^{0,b,c;j}|\varphi^{n'}|J^{(0,b,c)}_{m',n'}f_k|+| \mf A_{m',n'}^{0,b,c;j}|\varphi^{n'}|J^{(0,b+1,c)}_{m',n'}f_k|\rg)\\
&\lesssim \sum_{m'=(m-1)_+}^{m}\sum_{n'=(n-2)_+}^{n-1}\sum_{b+c=j+1}\lan m\ran^{(m+n-m'-n')_+(1+\sigma)}\varphi^{n'}|J^{(0,b,c)}_{m',n'}f_k|.
\end{align*}
This is consistent with \eqref{cm_qn_pv_3} and hence complete the proof.
\ifx
\siming{Previous:  
For the $n= 2$ case, we invoke the expression \eqref{cm_q2_pv2} to obtain that that
\begin{align*}\varphi^2\pav  [\pav^2,q^2]&(\chi_{m+2} |k|^m \Gamma_k^2f_k) =\sum_{n'=0}^1\sum_{b+c=2}\lan m\ran^{(1-n')(1+\sigma)}\pav \lf(\mf F_{m,n'}^{b,c;2}\varphi^{n'}J_{m,n'}^{(0,b,c)}f_k \rg) 
\\ &\hspace{-2cm} =\sum_{n'=0}^1\sum_{b+c=2}\lan m\ran^{(2-n')(1+\sigma)}\frac{\pav  \mf F_{m,n'}^{b,c;2}}{\lan m\ran^{1+\sigma}}\varphi^{n'}J_{m,n'}^{(0,b,c)}f_k  + \sum_{n'=0}^1\sum_{b+c=2}\lan m\ran^{(1-n')(1+\sigma)} \mf F_{m,n'}^{b,c;2}\varphi^{n'}J_{m,n'}^{(0,b+1,c)}f_k.
\end{align*}
Thanks to the fact that the coefficients $\mf F _{m,n}^{b,c;2}$ are \emph{admissible} (as defined in Step \# 1b), we have the estimate $\eqref{cm_qn_pv_3}_{n=2}$. 
\\
For the $n=1,\ m\geq 1$ case, we need to explicitly write it out as in \eqref{cm_qn_pv2} and estimate it as follows:
\begin{align*}\n
\big|\varphi& \pav[\pav^2,q](\chi_{m+1}|k|^m\Gamma_kf_k)\big|
= \big|\varphi \pav\lf(\text{ad}_{\pav}^2[q](\chi_{m+1}\Gamma_k( |k|^m f_k))+2\text{ad}_{\pav}[q]\pav (\chi_{m+1}\Gamma_k(|k|^mf_k))\rg)\big| \\  
&= \bigg|\pav(\wt {q_{1,2}}\chi_{m+1})\lf(\varphi J_{m,0}^{(0,1,0)}f_k+i(\varphi t) J^{(0,0,1)}_{m,0}f_k\rg)+ \wt {q_{1,2}}\chi_{m+1}\lf(\varphi J_{m,0}^{(0,2,0)}f_k+i(\varphi t) J^{(0,1,1)}_{m,0}f_k\rg)\\
&\quad+ 2\pav(\wt{q_{1,1}}\chi_{m+1})\lf(\varphi J_{m,0}^{(0,2,0)}f_k+it\varphi J^{(0,1,1)}_{m,0}f_k\rg)+2 \wt{q_{1,1}}\chi_{m+1}\lf(\varphi J_{m,0}^{(0,3,0)}f_k+it\varphi J^{(0,2,1)}_{m,0}f_k\rg)\\
&\quad+\pav(
 2\wt{q_{1,1}}\pav\chi_{m+1})\lf(\varphi J_{m-1,0}^{(0,1,1)}f_k+i(\varphi t) J^{(0,0,2)}_{m-1,0}f_k\rg)+2\wt{q_{1,1}}\pav\chi_{m+1} \lf(\varphi J_{m,0}^{(0,2,0)}f_k+i(\varphi t) J^{(0,1,1)}_{m,0}f_k\rg)\bigg|\\
&\lesssim  \sum_{m'=(m-1)_+}^m \lan m\ran^{(m+1-m')(1+\sigma)}|J_{m',0}^{(0,b,c)}f_k|,\quad m\geq 1.
\end{align*}
This is consistent with \eqref{cm_qn_pv_3}. If $n=1,\ m= 0$, we have that 
\begin{align*}
 \big|\varphi& \pav[\pav^2,q](\chi_{1} \Gamma_kf_k)\big|
= \big|\varphi \pav\lf(\text{ad}_{\pav}^2[q](\chi_{1}\Gamma_k( f_k))+2\text{ad}_{\pav}[q]\pav (\chi_{1}\Gamma_k(f_k))\rg)\big| \lesssim  \sum_{b+c\leq 3}|J_{0,0}^{(0,b,c)}f_k|.
\end{align*}
This is consistent with \eqref{cm_qn_pv_3} and hence complete the proof.}\fi
\end{proof}

The following lemma help us to control various type of $J_{m,n}^{(a,b,c)}$-terms.
\begin{lemma}[Induction lemma]\label{lem:induction}
Assume that $\sigma_\ast\geq 10\sigma$ and fix an arbitrary pair $(m,n)\in \mathbb{N}\times (\mathbb{N}\backslash\{0\})$. The $J^{(a,b,c)}_{m,n}f_k$ can be estimated as follows
\begin{align} 
\n &{\bf a}_{m,n}^2\sum_{\substack{a+b+c=\mf m_0,\\ a\geq 1}}\|J^{(a,b,c)}_{m,n}f_k\|_{L^2}^2\\
 &\lesssim\sum_{n'= (n-\mf m_0-2)_+}^{n-1}\lambda^{2s(n-n')} \lf(\frac{|m,n'|!}{|m,n|!}\rg)^{2\sigma}{\bf a}_{m,n'}^2\|J^{( \mf m_0)}_{m,n'}f_k\|_{L^2}^2\n \\
&\quad+\mathbbm{1}_{n<\mf m_0}\sum_{m'=(m-1)_+}^m\sum_{n'= 0}^{1}\lambda^{2s(|m,n|-|m',n'|)} \lf(\frac{|m',n'|!}{|m,n|!}\rg)^{2\sigma}{\bf a}_{m',n'}^2\sum_{b+c\leq \mf m_0}\|J^{(0,b,c)}_{m',n'}f_k\|_{L^2}^2.  \label{J_al>0_l2}
\end{align}
Here, we recall the notations $J^{( \mf m_0)}_{m,n}f_k
,\ J^{(\leq \mf m_0)}_{m,n}f_k
$ \eqref{J_vec}, the index $\mf m_0\in\{1,2,3\}$ and highlight that the implicit constant depends on $\|v_y^{-1}\|_{L_{t,y}^\infty},\ \|v_y\|_{L_t^\infty W_y^{2,\infty}}$.
\end{lemma}
\begin{proof}

\noindent
{\bf Step \#1: Setup.} We will derive the following bound which implies the estimate \eqref{J_al>0_l2}
\begin{align}\n
 &{\bf a}_{m,n}\sum_{\substack{a+b+c=\mf m_0\\ \mf m_0\leq n,\ a\geq 1}}\|J^{(a,b,c)}_{m,n}f_k\|_{L^2}\\ 
 &\lesssim 
\sum_{n'=(n-\mf m_0-2)_+ }^{n-1}\lambda^{s(n-n')} \lf(\frac{|m,n'|!}{|m,n|!}\rg)^{\sigma}{\bf a}_{m,n'}\|J^{(\mf m_0)}_{m,n'}f_k\|_{L^2}\n \\
&\quad+\mathbbm{1}_{n<\mf m_0}\sum_{m'=(m-1)_+}^m\sum_{n'= 0}^{1}\lambda^{s(|m,n|-|m',n'|)} \lf(\frac{|m',n'|!}{|m,n|!}\rg)^{\sigma}{\bf a}_{m',n'} \sum_{b+c\leq \mf m_0}\|J^{(0,b,c)}_{m',n'}f_k\|_{L^2},
\quad \mf m_0\in\{1,2,3\} . \label{J_al>0}
\end{align}
Thanks to the definition of $J_{m,n}^{(a,b,c)}$ \eqref{J_nq} and the restriction of its indices  \eqref{Snell}, we observe that the left hand side of \eqref{J_al>0} indeed contains all the nontrivial terms on the left hand side of the relation \eqref{J_al>0_l2}. Now, the implication is a direct consequence of the fact that there are only at most twenty terms in the sum ($\mathfrak{m}_0\in\{1,2,3\}$) and the H\"older inequality. 

The plan now is to decompose the left hand side of \eqref{J_al>0} into manageable pieces. We will commute the $\chi_{m+n}$-cutoff out first and then identify four main contributions. We start by rewriting the $J^{(a,b,c)}_{m,n}f_k$ term, for $a+b+c=\mf m_0,\quad \mf m_0\in\{1,2,3\}$, as follows,
\begin{align*}\n 
\eqref{J_al>0}_{L.H.S.}& = \sum_{a+b+c=\mf m_0}\mathbbm{1}_{\substack{a+b+c\leq n\\ a\geq 1}}\ {\bf a}_{m,n}\lf\||k|^c\lf(\frac{m+n}{q}\rg)^a\pav^b( \chi_{m+n}  q^n\Gamma_k^n|k|^m f_k)\rg\|_{L^2}\\ 
\n &\leq
\sum_{a+b+c=\mf m_0}\ina\ {\bf a}_{m,n}\lf\|\chi_{m+n} |k|^c\lf(\frac{m+n}{q}\rg)^a\pav^b(  q^n\Gamma_k^n|k|^m f_k)\rg\|_{L^2}\\ \n
&\quad+\sum_{a+b+c=\mf m_0}\ina\ {\bf a}_{m,n}\lf\||k|^c \lf(\frac{m+n}{q}\rg)^a[\pav^b,\chi_{m+n}](q^n \Gamma_k^n|k|^mf_k)\rg\|_{L^2}. 
\end{align*} 
Here we apply a variant of the $\chi_{m+n}$-estimate \eqref{chi:prop:3}
\begin{align}
\pav^\ell \chi_{m+n}\leq C(\|v_y^{-1}\|_{L_{t,y}^\infty},\|v_y\|_{L_t^\infty W_y^{2,\infty}})(m+n)^{\ell(1+\sigma)}\chi_{m+n-1},\quad \sigma_\ast\geq 10\sigma,\quad \ell\leq\mathfrak{m}_0\leq 3 ,
\end{align} 
and Lemma \ref{lem:commutator_AB} 
to conclude that 
\begin{align}\n
&\eqref{J_al>0}_{L.H.S.}\\ \n
 &\lesssim\sum_{a+b+c=\mf m_0}\ina \ {\bf a}_{m,n}\lf\|\mathbbm{1}_{\text{supp}(\chi_{m+n}) }|k|^c\lf(\frac{m+n}{q}\rg)^a\pav^b(q^n\Gamma_k^n|k|^m f_k)\rg\|_{L^2}\\ \n
&\quad +\sum_{a+b+c=\mf m_0}\ina\ {\bf a}_{m,n}\sum_{\ell=1}^{b}{(m+n)^{\sigma \ell}}\lf\|\mathbbm{1}_{\text{supp}(\chi_{m+n})}|k|^c\lf(\frac{m+n}{q}\rg)^{a+\ell}\pav^{b-\ell}( q^n\Gamma_k^n|k|^m f_k)\rg\|_{L^2}\\
& \lesssim{\bf a}_{m,n}(m+n)^{3\sigma}\sum_{a+b+c=\mf m_0}\ina\lf\|\mathbbm{1}_{\text{supp}(\chi_{m+n})}|k|^c\lf(\frac{m+n}{q}\rg)^{a}\pav^{b}( q^n\Gamma_k^n|k|^m f_k)\rg\|_{L^2}.\label{J_rf}
\end{align} 
This concludes the setup.

\siming{Previously, there is another argument here. Can undo the $\backslash ifx$ and check. }
\ifx
To estimate this expression, we distinguish between two cases: {\bf Case a)} $n-a\geq\mf m_0$; {\bf Case b)} $n-a<\mf m_0$.

To treat the first {\bf Case a)}, we note the following relation: 
\begin{align*}
q^{-a}\pav^b(q^{a}\, q^{n-a}\Gamma_k^a (g))=& \pav^b( \Gamma_k^a (q^{n-a}g))+q^{-a}[\pav^b, q^a]( \Gamma_k^a (q^{n-a}g))\\
&- \pav^b [ \Gamma_k^a, q^{n-a}](g) -q^{-a}[\pav^b, q^a]( [\Gamma_k^a,q^{n-a}] (g)),\quad 1\leq a\leq n.
\end{align*} 
Hence, we rewrite the expression \eqref{J_rf} as follows ($n-a\geq\mf m_0$):
\begin{align}\n
&\hspace{-1cm}\eqref{J_al>0}_{L.H.S.}\\ \n
\lesssim &{\bf a}_{m,n}(m+n)^{3\sigma}\sum_{\substack{a+b+c=\mf m_0\\ \mf m_0\leq n,\, a\geq 1}}\lf\|\mathbbm{1}_{\text{supp}(\chi_{m+n})}|k|^c\lf(\frac{m+n}{q}\rg)^a\pav^b(q^a \ q^{n-a}\Gamma_k^{a}\ (\chi_{m+n-a}\Gamma_k^{n-a}|k|^m f_k)\rg\|_2 \\
\n\lesssim&{\bf a}_{m,n}(m+n)^{3\sigma}\sum_{\substack{a+b+c=\mf m_0\\ \mf m_0\leq n\ a\geq 1}}\bigg(\lf\|\mathbbm{1}_{\text{supp}(\chi_{m+n})}|k|^c(m+n)^a\pav^b\ \lf(\Gamma_k^{a}\ (\chi_{m+n-a}q^{n-a}\Gamma_k^{n-a}|k|^mf_k)\rg)\rg\|_{L^2}\\
\n &+\lf\|\mathbbm{1}_{\text{supp}(\chi_{m+n})}|k|^c(m+n)^a q^{-a}\ [\pav^b,q^a]\ \lf( \Gamma_k^{a}\ (\chi_{m+n-a}q^{n-a}\Gamma_k^{n-a}|k|^m f_k)\rg)\rg\|_{L^2}\\
\n&+\lf\|\mathbbm{1}_{\text{supp}(\chi_{m+n})}|k|^c(m+n)^a\pav^b\ \lf([q^{n-a},\Gamma_k^{a}]\ (\chi_{m+n-a}\Gamma_k^{n-a}|k|^m f_k)\rg)\rg\|_{L^2}\\ \n
&+\lf\|\mathbbm{1}_{\text{supp}(\chi_{m+n})}|k|^c(m+n)^a q^{-a}\ [\pav^b,q^a ]\ \lf([q^{n-a},\Gamma_k^{a}]\ (\chi_{m+n-a}\Gamma_k^{n-a}|k|^m f_k)\rg)\rg\|_{L^2}\bigg)\\
=:&\sum_{j=1}^4T_{j}.\label{Jabg_I}
\end{align}
Hence to derive the estimate \eqref{J_al>0}, it is enough to estimate all the terms in \eqref{Jabg_I}. 

For {\bf Case b)}, we will treat it separately.

\noindent
{\bf Step \# 2: Case a).} Now we estimate the four terms in \eqref{Jabg_I}. First, we estimate the first term with the estimate \eqref{gevbd1212} as follows
\begin{align*}
T_{1}\lesssim&{\bf a}_{m,n}\sum_{\substack{a+b+c=\mf m_0\\ \mf m_0\leq n,\ a\geq 1}}(m+n)^{a+3\sigma}\sum_{i_1=0}^{a}\binom{a}{i_1} \lf\|\mathbbm{1}_{\text{supp}(\chi_{m+n})} t^{i_1} \pav^{b+a-i_1}|k|^{i_1+c}(\chi_{m+n-a}q^{n-a}\Gamma_k^{n-a}|k|^mf_k)\rg\|_{L^2}\\
\lesssim&\sum_{\substack{a+b+c=\mf m_0\\ \mf m_0\leq n,\ a\geq 1}}{\bf a}_{m,n}(m+n)^{3\sigma+a}  \sum_{i_1=0}^a t^{i_1}\|J^{(0,b+a-i_1,i_1+c)}_{m,n-a}f_k\|_{L^2}\\
\lesssim&\sum_{\substack{a+b+c=\mf m_0\\ \mf m_0\leq n,\ a\geq 1}}\myr{\lambda^{as}{\bf a}_{m,n-a}\frac{(m+n-a)!^{\sigma}} {(m+n)!^{\sigma}}(m+n)^{-a \sigma_\ast+3\sigma} }\sum_{i_1=0}^a (\varphi^a t^{i_1})\|J^{(0,b+a-i_1,i_1+c)}_{m,n-a}f_k\|_{L^2}\\
\lesssim& \sum_{ n'=(n-\mf m_0)_+}^{n-1}  \sum_{b+c=\mf m_0}\myr{\frac{{\bf a}_{m,n'}\lambda^{(n-n')s}(m+n')!^\sigma} {(m+n)!^{\sigma}}}\|J^{( 0 ,b ,c)}_{m,n'}f_k\|_{L^2}. 
\end{align*}
Next, we estimate the $T_2$ term in \eqref{Jabg_I} with the $q$-bounds \eqref{pav_j_qn}, \eqref{wtq} in Lemma \ref{lem:pavq} as follows
\begin{align*}
T_2\lesssim&{\bf a}_{m,n}(m+n)^{3\sigma}\sum_{\substack{a+b+c=\mf m_0\\ \mf m_0\leq n,\ a\geq 1}}\sum_{i_1=0}^{b-1}\sum_{i_2=0}^{a}\binom{b}{i_1} \binom{a}{i_2}|k|^c\\
&\lf\|\mathbbm{1}_{\text{supp}(\chi_{m+n})}\lf(\frac{m+n}{q}\rg)^a  (\pav^{b-i_1}q^{a})\ \pav^{i_1+i_2}|k|^{a-i_2}t^{a-i_2}\lf(\chi_{m+n-a}q^{n-a}\Gamma_k^{n-a}|k|^mf_k\rg) \rg\|_{L^2}\\
\lesssim&{\bf a}_{m,n}(m+n)^{3\sigma+a}\sum_{\substack{a+b+c=\mf m_0\\ \mf m_0\leq n,\ a\geq 1}}\sum_{i_1=0}^{b-1}\sum_{i_2=0}^{a}|k|^c\lf\|\wt{q_{a,b-i_1}}\frac{1}{q^{a-(a-b+i_1)_+}}\pav^{i_1+i_2}(|k|t)^{a-i_2}\lf(\chi_{m+n-a}q^{n-a}\Gamma_k^{n-a}|k|^mf_k\rg)\rg\|_{L^2}\\
\lesssim&\sum_{\substack{a+b+c=\mf m_0\\ \mf m_0\leq n,\ a\geq 1}}\myr{{\bf a}_{m,n-a}\lambda^{as} \frac{(m+n-a)!^\sigma}{(m+n)!^{\sigma}}}\sum_{i_1=0}^{b-1}\sum_{i_2=0}^{a}(\varphi^a t ^{a-i_2}) \lf\|J^{(a-(a-b+i_1)_+,i_1+i_2,a-i_2+c)}_{m,n-a}f_k \rg\|_{L^2}\\
\lesssim& \sum_{n'=(n-\mf m_0)_+}^{n-1}\myr{\sum_{\substack{a+b+c=\mf m_0}}{\bf a}_{m,n'}\lambda^{(n-n')s}\frac{(m+n')!^{\sigma}}{ (m+n)!^{\sigma}}}\|J^{( a ,b ,c)}_{m,n'}f_k\|_{L^2}. 
\end{align*}
Here the implicit constant depends on $\|v_y^{-1}\|_{L_{t,y}^\infty},\, \|v_y\|_{L_t^\infty W^{2,\infty}}$. \siming{We need to discuss the case $n-a\leq \mf m_0$ separately because in this case $((m+n)q^{-1})^a\pav^b|k|^c f_k\neq J^{(a,b,c)}_{m,n}f_k$.}

Next,  we estimate the $T_3$ term in \eqref{Jabg_I} as follows
\begin{align*}
T_3\lesssim& {\bf a}_{m,n}(m+n)^{3\sigma}\sum_{\substack{a+b+c=\mf m_0\\ \mf m_0\leq n,\ a\geq 1}}\sum_{i_1=0}^{a-1}\binom{a}{i_1}|k|^c \\
&\quad\quad\quad\quad\lf\|\mathbbm{1}_{\text{supp}(\chi_{m+n})}(m+n)^a \pav^b \lf(\Gamma_k^{a-i_1}q^{n-a}\ \ \Gamma_k^{i_1}(\chi_{m+n-a}\Gamma_k^{n-a}|k|^mf_k)\rg) \rg\|_{L^2}\\
\lesssim&  {\bf a}_{m,n}(m+n)^{3\sigma+a}\sum_{\substack{a+b+c=\mf m_0\\ \mf m_0\leq n,\ a\geq 1}}\sum_{i_1=0}^{a-1} \sum_{i_1=0}^{a-1}\sum_{i_2=0}^{b}  \lf\||k|^c\pav^{b-i_2}\Gamma_k^{a-i_1}q^{n-a}\ \ \pav^{i_2}\Gamma_k^{i_1}\lf(\chi_{m+n-a}\Gamma_k^{n-a}|k|^mf_k\rg)\rg\|_{L^2}\\
\lesssim&\sum_{\substack{a+b+c=\mf m_0\\ \mf m_0\leq n,\ a\geq 1}}{\bf a}_{m,n-a}\myr{\lambda^{as} \frac{(m+n-a)!^\sigma}{(m+n)!^{\sigma}}}\sum_{i_1=0}^{a-1}\sum_{i_2=0}^{b}\sum_{i_3=0}^{a-i_1}\sum_{i_4=0}^{i_1}\\
&\quad\quad\quad\quad\varphi^a \lf\|  |kt|^{a-i_1-i_3}|k|^c\pav^{b-i_2+i_3}q^{n-a}\ \ |kt|^{i_1-i_4}\pav^{i_2+i_4}\lf (\chi_{m+n-a}\Gamma_k^{n-a}|k|^mf_k \rg)\rg\|_{L^2}.
\end{align*}
Now we invoke the representation \eqref{pav_j_qn}, the estimate \eqref{wtq}, and the commutator relations \eqref{cm_J},  \eqref{cm_qn_pv_1}, to obtain that
\begin{align*}
T_3\lesssim&  \sum_{\substack{a+b+c=\mf m_0\\ \mf m_0\leq n,\ a\geq 1}}{\bf a}_{m,n-a}\myr{\lambda^{as} \frac{(m+n-a)!^\sigma}{(m+n)!^{\sigma}}}\sum_{i_1=0}^{a-1}\sum_{i_2=0}^{b}\sum_{i_3=0}^{a-i_1}\sum_{i_4=0}^{i_1} \varphi^{i_3+i_4}\\
&\quad\quad\quad\quad\lf\||k|^{a -i_3-i_4+c}((m+n)q^{-1})^{b-i_2+i_3}\ \ q^{n-a} \pav^{i_2+i_4} (\chi_{m+n-a}\Gamma_k^{n-a}|k|^mf_k)\rg\|_{L^2}\\
\lesssim&\sum_{\substack{a+b+c=\mf m_0\\ \mf m_0\leq n,\ a\geq 1}}{\bf a}_{m,n-a}\myr{\lambda^{as} \frac{(m+n-a)!^\sigma}{(m+n)!^{\sigma}}}\sum_{i_1=0}^{a-1}\sum_{i_2=0}^{b}\sum_{i_3=0}^{a-i_1}\sum_{i_4=0}^{i_1} \varphi^{i_3+i_4}\\
&\quad\quad\quad\quad\lf\|  |k|^{a -i_3-i_4+c}((m+n)q^{-1})^{b-i_2+i_3}\ \ \pav^{i_2+i_4} (\chi_{m+n-a}q^{n-a} \Gamma_k^{n-a}|k|^mf_k)\rg\|_{L^2}\\
&+\sum_{\substack{a+b+c=\mf m_0\\ \mf m_0\leq n,\ a\geq 1}}{\bf a}_{m,n-a}\myr{\lambda^{as} \frac{(m+n-a)!^\sigma}{(m+n)!^{\sigma}}}\sum_{i_1=0}^{a-1}\sum_{i_2=0}^{b}\sum_{i_3=0}^{a-i_1}\sum_{i_4=0}^{i_1}\varphi^{i_3+i_4}\\
&\quad\quad\quad\quad \lf\| |k|^{a -i_3-i_4+c}((m+n)q^{-1})^{b-i_2+i_3}\ \ [q^{n-a},\pav^{i_2+i_4}] (\chi_{m+n-a} \Gamma_k^{n-a}|k|^mf_k)\rg\|_{L^2}\\
\lesssim&\sum_{\substack{a+b+c=\mf m_0\\ \mf m_0\leq n,\ a\geq 1}}{\bf a}_{m,n-a}\myr{\lambda^{as} \frac{(m+n-a)!^\sigma}{(m+n)!^{\sigma}}}\sum_{i_1=0}^{a-1}\sum_{i_2=0}^{b}\sum_{i_3=0}^{a-i_1}\sum_{i_4=0}^{i_1}\varphi^{i_3+i_4} \lf\|  J^{(b-i_2+i_3,i_2+i_4,a -i_3-i_4+c)}_{m,n-a }f_k \rg\|_{L^2}\\
&+\sum_{\substack{a+b+c=\mf m_0\\ \mf m_0\leq n,\ a\geq 1}}{\bf a}_{m,n-a}\myr{\lambda^{as} \frac{(m+n-a)!^\sigma}{(m+n)!^{\sigma}}}\sum_{i_1=0}^{a-1}\sum_{i_2=0}^{b}\sum_{i_3=0}^{a-i_1}\sum_{i_4=0}^{i_1}\sum_{i_5+i_6=i_2+i_4}\varphi^{i_3+i_4}\\
&\quad\quad\quad\quad \lf\|  |k|^{a -i_3-i_4+c}((m+n)q^{-1})^{b-i_2+i_3+i_5} \pav^{i_6}  (\chi_{n-a} q^{n-a}\Gamma_k^{n-a}|k|^mf_k)\rg\|_{L^2}\\
\lesssim& \sum_{\substack{a+b+c=\mf m_0\\ \mf m_0\leq n,\ a\geq 1}}{\bf a}_{m,n-a}\myr{\lambda^{as} \frac{(m+n-a)!^\sigma}{(m+n)!^{\sigma}}}\sum_{i_1=0}^{a-1}\sum_{i_2=0}^{b}\sum_{i_3=0}^{a-i_1}\sum_{i_4=0}^{i_1}\lf\|J^{(b-i_2+i_3,i_2+i_4,a -i_3-i_4+c)}_{m,n-a }f_k \rg\|_{L^2}\\
&+\sum_{\substack{a+b+c=\mf m_0\\ \mf m_0\leq n,\ a\geq 1}}{\bf a}_{m,n-a}\myr{\lambda^{as} \frac{(m+n-a)!^\sigma}{(m+n)!^{\sigma}}}\sum_{i_1=0}^{a-1}\sum_{i_2=0}^{b}\sum_{i_3=0}^{a-i_1}\sum_{i_4=0}^{i_1}\sum_{i_5+i_6=i_2+i_4} \varphi^{i_3+i_4}  \lf\| J^{(b-i_2+i_3+i_5, i_6 ,a -i_3-i_4+c)}_{m,n-a}f_k\rg\|_{L^2}\\
\lesssim& \sum_{n'=n-\mf m_0}^{n-1}\myr{\sum_{\substack{a+b+c=\mf m_0}}}{\bf a}_{m,n'}\myr{\lambda^{(n-n')s} \frac{(m+n')!^\sigma}{(m+n)!^{\sigma}}}\|J^{( a ,b ,c)}_{m,n'}f_k\|_{L^2}. 
\end{align*}
Here the implicit constant depends on $\|v_y^{-1}\|_{L_{t,y}^\infty},\, \|v_y\|_{L_t^\infty W^{2,\infty}}$. 
Finally, we estimate the $T_4$ term in \eqref{Jabg_I} with Lemma \ref{lem:pavq} as follows
\begin{align*}
T_4\lesssim&  {\bf a}_{m,n}(m+n)^{3\sigma}\sum_{\substack{a+b+c=\mf m_0\\ \mf m_0\leq n,\ a\geq 1}}\sum_{i_1=0}^{b-1}\sum_{i_2=0}^{a-1} \\
&\qquad\qquad\qquad |k|^c\lf\|\lf(\frac{m+n}{q}\rg)^a(\pav^{b-i_1}q^a)\ \pav^{i_1}\lf((\Gamma_k^{a-i_2}q^{n-a})(\chi_{m+n-a+i_2} \Gamma_k^{n-a+i_2}|k|^mf_k)\rg)\rg\|_{L^2}\\
\lesssim& {\bf a}_{m,n}(m+n)^{3\sigma}\sum_{\substack{a+b+c=\mf m_0\\ \mf m_0\leq n,\ a\geq 1}}\sum_{i_1=0}^{b-1}\sum_{i_2=0}^{a-1}\sum_{i_3=0}^{i_1} \\ 
& \qquad\qquad|k|^c\lf\|\lf(\frac{m+n}{q}\rg)^a (\pav^{b-i_1}q^a)\ \lf((\pav^{i_1-i_3}\Gamma_k^{a-i_2}q^{n-a})( \pav^{i_3}(\chi_{m+n-a+i_2}\Gamma_k^{n-a+i_2}|k|^mf_k))\rg)\rg\|_{L^2}\\
\lesssim&{\bf a}_{m,n}(m+n)^{3\sigma}\sum_{\substack{a+b+c=\mf m_0\\ \mf m_0\leq n,\ a\geq 1}} \sum_{i_1=0}^{b-1}\sum_{i_2=0}^{a-1}\sum_{i_3=0}^{i_1}\\
&\qquad\qquad|k|^c\lf\|\lf(\frac{m+n}{q}\rg)^a(\pav^{b-i_1}q^a)\  (\pav^{i_1-i_3}(\pav+ikt)^{a-i_2}q^{n-a})(\pav^{i_3} (\chi_{m+n-a+i_2}\Gamma_k^{n-a+i_2}|k|^mf_k)\rg\|_{L^2}\\
\lesssim& {\bf a}_{m,n}(m+n)^{3\sigma}\sum_{\substack{a+b+c=\mf m_0\\ \mf m_0\leq n,\ a\geq 1}} \sum_{i_1=0}^{b-1}\sum_{i_2=0}^{a-1}\sum_{i_3=0}^{i_1}\sum_{i_4=0}^{a-i_2}(|k|t)^{ i_4 }\\
&\qquad\qquad\times|k|^c\lf\|\lf(\frac{m+n}{q}\rg)^a(\pav^{b-i_1}q^a)\ \lf((\pav^{i_1-i_3+a-i_2-i_4}q^{n-a})\pav^{i_3}( \chi_{m+n-a+i_2}\Gamma_k^{n-a+i_2}|k|^mf_k)\rg)\rg\|_{L^2}\\
\lesssim&{\bf a}_{m,n}(m+n)^{3\sigma}\sum_{\substack{a+b+c=\mf m_0\\ \mf m_0\leq n,\ a\geq 1}} \sum_{i_1=0}^{b-1}\sum_{i_2=0}^{a-1}\sum_{i_3=0}^{i_1}\sum_{i_4=0}^{a-i_2} (|k|t)^{ i_4 }\\
&\qquad\times|k|^c\lf\|\lf(\frac{m+n}{q}\rg)^a\frac{(m+n)^{b-i_1}}{q^{b-i_1}} q^a \frac{(m+n)^{i_1-i_3+a-i_2-i_4}}{q^{i_1-i_3+a-i_2-i_4}}q^{n-a}\pav^{i_3}\lf(\chi_{m+n-a+i_2} \Gamma_k^{n-a+i_2 }|k|^mf_k\rg)\rg\|_{L^2}\\
\lesssim&{\bf a}_{m,n}(m+n)^{3\sigma}\sum_{\substack{a+b+c=\mf m_0\\ \mf m_0\leq n,\ a\geq 1}} \sum_{i_1=0}^{b-1}\sum_{i_2=0}^{a-1}\sum_{i_3=0}^{i_1}\sum_{i_4=0}^{a-i_2} (|k|t)^{ i_4 }(m+n)^{a-i_2}\\
&\qquad\times|k|^c\lf\|  \frac{(m+n)^{a +b-i_3-i_4}}{q^{a+b-i_3 -i_4}}\ q^{n-a+i_2}\pav^{i_3}\lf(\chi_{m+n-a+i_2} \Gamma_k^{n-a+i_2 }|k|^mf_k\rg)\rg\|_{L^2}.
\end{align*}
Here the implicit constant depends on $\|v_y^{-1}\|_{L_{t,y}^\infty},\, \|v_y\|_{L_t^\infty W^{2,\infty}}$. 
Now we invoke the commutator estimate \eqref{cm_J}, \eqref{cm_qn_pv_1} to obtain that\begin{align*}
T_4\lesssim&\sum_{\substack{a+b+c=\mf m_0\\ \mf m_0\leq n,\ a\geq 1}} \sum_{i_1=0}^{b-1}\sum_{i_2=0}^{a-1}\sum_{i_3=0}^{i_1}\sum_{i_4=0}^{a-i_2} {\bf a}_{m,n-a+i_2}\myr{\lambda^{(a-i_2)s} \frac{(m+n-a+i_2)!^\sigma}{(m+n)!^{\sigma}}} \varphi^{a-i_2-i_4} \\
 &\qquad\qquad \times\lf\|\frac{(m+n)^{a+b-i_3-i_4}}{q^{a+b-i_3-i_4}} |k|^{i_4+c}\pav^{i_3}\lf( \chi_{m+n-a+i_2}q^{n-a+i_2}\Gamma_k^{n-a+i_2}|k|^mf_k\rg)\rg\|_{L^2}\\ 
&+ \sum_{\substack{a+b+c=\mf m_0\\ \mf m_0\leq n,\ a\geq 1}}\sum_{i_1=0}^{b-1}\sum_{i_2=0}^{a-1}\sum_{i_3=0}^{i_1}\sum_{i_4=0}^{a-i_2} {\bf a}_{m,n-a+i_2}\myr{\lambda^{(a-i_2)s} \frac{(m+n-a+i_2)!^\sigma}{(m+n)!^{\sigma}}} \varphi^{a-i_2-i_4}\\ 
& \qquad\qquad\times \lf\|\frac{(m+n)^{a+b-i_3-i_4}}{q^{a+b-i_3-i_4}} |k|^{i_4+c}[q^{n-a+i_2},\pav^{i_3}]\lf( \chi_{m+n-a+i_2}\Gamma_k^{n-a+i_2}|k|^mf_k\rg)\rg\|_{L^2}\\ 
\lesssim& \sum_{\substack{a+b+c=\mf m_0\\ \mf m_0\leq n,\ a\geq 1}} \sum_{i_1=0}^{b-1}\sum_{i_2=0}^{a-1}\sum_{i_3=0}^{i_1}\sum_{i_4=0}^{a-i_2}{\bf a}_{m,n-a+i_2}\myr{\lambda^{(a-i_2)s} \frac{(m+n-a+i_2)!^\sigma}{(m+n)!^{\sigma}}} \varphi^{a-i_2-i_4}\|J^{(a+b-i_3-i_4,i_3,i_4+c)}_{m,n-a+i_2}f_k \|_{L^2}\\
 &+  \sum_{\substack{a+b+c=\mf m_0\\ \mf m_0\leq n,\ a\geq 1}}\sum_{i_1=0}^{b-1}\sum_{i_2=0}^{a-1}\sum_{i_3=0}^{i_1}\sum_{i_4=0}^{a-i_2}\sum_{i_5+i_6=i_3, i_5>0} 
 {\bf a}_{m,n-a+i_2} \myr{\lambda^{(a-i_2)s} \frac{(m+n-a+i_2)!^\sigma}{(m+n)!^{\sigma}}}\varphi^{a-i_2-i_4}\\ 
 &\qquad\qquad\times \lf\| \frac{(m+n)^{a+b-i_3-i_4}}{q^{a+b-i_3-i_4}} |k|^{i_4+c}\lf(\frac{m+n}{q}\rg)^{i_5}\pav^{i_6}\lf( \chi_{m+n-a+i_2}q^{n-a+i_2}\Gamma_k^{n-a+i_2}|k|^mf_k\rg)\rg\|_{L^2}\\
\lesssim& \sum_{\substack{a+b+c=\mf m_0\\ \mf m_0\leq n,\ a\geq 1}} \sum_{i_1=0}^{b-1}\sum_{i_2=0}^{a-1}\sum_{i_3=0}^{i_1}\sum_{i_4=0}^{a-i_2}{\bf a}_{m,n-a+i_2}\myr{\lambda^{(a-i_2)s} \frac{(m+n-a+i_2)!^\sigma}{(m+n)!^{\sigma}}} \varphi^{a-i_2-i_4}\|J^{(a+b-i_3-i_4,i_3,i_4+c)}_{m,n-a+i_2}f_k \|_{L^2}\\
 &+\sum_{\substack{a+b+c=\mf m_0\\ \mf m_0\leq n,\ a\geq 1}}\sum_{i_1=0}^{b-1}\sum_{i_2=0}^{a-1}\sum_{i_3=0}^{i_1}\sum_{i_4=0}^{a-i_2}\sum_{i_5+i_6=i_3, i_5>0} 
 {\bf a}_{m,n-a+i_2} \myr{\lambda^{(a-i_2)s} \frac{(m+n-a+i_2)!^\sigma}{(m+n)!^{\sigma}}}\varphi^{a-i_2-i_4}\\
&\qquad\qquad \times \lf\|J^{(a+b-i_3-i_4+i_5,i_6,i_4+c)} _{m,\myr{n-a+i_2\siming{?}}}f_k \rg\|_{L^2}\\
 \lesssim& \sum_{n'=n-\mf m_0}^{n-1}\sum_{\substack{a+b+c=\mf m_0}}{\bf a}_{m,n'}\myr{\lambda^{(n-n')s} \frac{(m+n')!^\sigma}{(m+n)!^{\sigma}}}\|J^{( a ,b ,c)}_{m,n'}f_k\|_{L^2}. 
\end{align*}
Here the implicit constant depends on $\|v_y^{-1}\|_{L_{t,y}^\infty},\, \|v_y\|_{L_t^\infty W^{2,\infty}}$. Combining all the estimates developed so far yields \eqref{J_al>0}. 
\fi

\noindent
{\bf Step \# 2: Estimates.} We estimate the right hand side of \eqref{J_rf} with simple commutator:
\begin{align*}
 &\eqref{J_rf}_{R.H.S.}\lesssim{\bf a}_{m,n}(m+n)^{3\sigma}\sum_{\substack{a+b+c=\mf m_0\\ 
\mf m_0\leq n,\ a\geq 1}}\sum_{i_1=0}^b \lf\|\mathbbm{1}_{\text{supp}(\chi_{m+n})}|k|^c\lf(\frac{m+n}{q}\rg)^{a}\pav^{b-i_1}(q^n)\ \pav^{i_1}\Gamma_k^n(|k|^m f_k)\rg\|_{L^2}\\
&\lesssim {\bf a}_{m,n}(m+n)^{3\sigma+a}\sum_{\substack{a+b+c=\mf m_0\\ 
\mf m_0\leq n,\ a\geq 1}}\sum_{i_1=0}^b \lf\|\mathbbm{1}_{\text{supp}(\chi_{m+n})}|k|^c q^{n-a-b+i_1}\ \Gamma_k^{a+b-i_1}\pav^{i_1} (\Gamma_k^{n-a-b+i_1}|k|^m f_k)\rg\|_{L^2}\\
&\lesssim 
{\bf a}_{m,n}(m+n)^{3\sigma+a}\sum_{\substack{a+b+c=\mf m_0\\ 
\mf m_0\leq n,\ a\geq 1}}\sum_{i_1=0}^b \lf\|\mathbbm{1}_{\text{supp}(\chi_{m+n})}|k|^c \Gamma_k^{a+b-i_1}\pav^{i_1} (q^{n-a-b+i_1} \Gamma_k^{n-a-b+i_1}|k|^m f_k)\rg\|_{L^2}\\
&\quad+{\bf a}_{m,n}(m+n)^{3\sigma+a}\sum_{\substack{a+b+c=\mf m_0\\ 
\mf m_0\leq n,\ a\geq 1}}\sum_{i_1=0}^b \lf\|\mathbbm{1}_{\text{supp}(\chi_{m+n})}|k|^c \lf[q^{n-a-b+i_1},\Gamma_k^{a+b-i_1}\pav^{i_1}\rg] (\Gamma_k^{n-a-b+i_1}|k|^m f_k)\rg\|_{L^2}\\
&=:T_1+T_2.
\end{align*}
For the $T_1$ term, we apply the facts that $10\sigma\leq \sigma_\ast$, $\varphi t\lesssim 1$ and estimate as follows
\begin{align*}
T_1\lesssim&\sum_{\substack{a+b+c=\mf m_0\\ 
\mf m_0\leq n,\ a\geq 1}}\sum_{i_1=0}^b\sum_{i_2=0}^{a+b-i_1}{\bf a}_{m,n-a-b+i_1}\frac{(m+n-a-b+i_1)!^{\sigma}}{(m+n)!^{\sigma}}(m+n)^{-\sigma_\ast+3\sigma}(\varphi t)^{a+b-i_1-i_2}\lambda^{(a+b-i_1)s}\\
&\hspace{2.83cm}\times \lf\|\mathbbm{1}_{\text{supp}(\chi_{m+n})}|k|^{a+b+c-i_1-i_2}\pav^{i_1+i_2} (\chi_{n-a-b+i_1}|k|^mq^{ n-a-b+i_1}\Gamma_k^{ n-a-b+i_1}f_{k})\rg\|_{L^2}\\
\lesssim& \sum_{n'=(n-\mf m_0)_+}^{n-1}\sum_{b+c=\mf m_0} {\bf a}_{m,n'}\frac{(m+n')!^{\sigma}}{(m+n)!^\sigma}\lambda^{(n-n')s}\|J_{m,n'}^{(0,b,c)}f_k\|_{L^2}.
\end{align*}
Here in the last line, we use the observation that the sum of the order of the $x$-derivative ($a+b+c-i_1-i_2$) and the order of the $\pav$-derivative ($i_1+i_2$) is $\mf m_0$ and they can be combined to $J^{(0,b,c)}_{m,n}$-form. Then we relabel the summation indices to compactify the notation. 
 
To treat the $T_2$-term, we invoked the Gevrey coefficient relation \eqref{gevbd1212} and the product rule to decompose it as follows
\begin{align*}
T_2\lesssim&\sum_{\substack{a+b+c=\mf m_0\\ 
\mf m_0\leq n,\ a\geq 1}}\sum_{i_1=0}^b \sum_{i_2=0}^{a+b-i_1}{\bf a}_{m,n-a-b+i_1}\lf(\frac{|m,n-a-b+i_1|!}{|m,n|!}\rg)^\sigma \lambda^{(a+b-i_1)s}(m+n)^{3\sigma-\sigma_\ast} \varphi^{a+b-i_1}\\
&\hspace{2cm}\times |k|^{a+b+c -i_1-i_2} t^{a+b-i_1-i_2}\lf\|[q^{n-a-b+i_1},\pav^{i_1+i_2}] (\Gamma_k^{n-a-b+i_1}|k|^m f_k)\rg\|_{L^2(\text{supp}\{\chi_{m+n}\})}.
\end{align*} 
Thanks to the assumption that $10\sigma\leq \sigma_\ast$, the factor $(m+n)^{3\sigma-\sigma_\ast}\leq 1$. Moreover, we observe that the last factor can be estimated by the commutator estimate \eqref{cm_J}. Finally, since $a>0$, $b-i_1\geq 0$, we have that $n-a-b+i_1\leq n-1$. Hence  $\text{support}\{\chi_{m+n-a-b+i_1}\}\subset\text{support}\{\chi_{m+n}\}$, and the resulting $J_{m,n'}f_k$ terms will have lower $n'$-regularity (which guarantees that the induction process will finish  in finite steps). By combining all the observations above, we have that 
\begin{align*}
&\hspace{-0.3cm} T_2\\
\lesssim& \sum_{\substack{a+b+c=\mf m_0\\ 
\mf m_0\leq n,\ a\geq 1}}\sum_{i_1=0}^b \sum_{i_2=0}^{a+b-i_1} \lf(\frac{|m,n-a-b+i_1|!}{|m,n|!}\rg)^\sigma \lambda^{(a+b-i_1)s} |k|^{\mf m_0-i_1-i_2} (\varphi t)^{a+b-i_1-i_2}\\
&\hspace{-0.2cm}\times\Bigg(\sum_{n'=(n-a-b+i_1-2)_+}^{n-a-b+i_1}{\bf a}_{m,n'}\lf(\frac{|m,n'|!}{|m,n-a-b+i_1|!}\rg)^\sigma \lambda^{(n-a-b+i_1-n')s}\lf\|J^{(\leq i_1+i_2)}_{m,n'}f_k\rg\|_{L^2(\text{supp}\{\chi_{m+n}\})}\\
&\hspace{0.3cm}+\mathbbm{1}_{n-a-b+i_1<\mf m_0}\sum_{m'=(m-1)_+}^m\sum_{n'= {0} }^{n-a-b+i_1-1}{\bf a}_{m',n'}\lf(\frac{|m',n'|!}{|m,n-a-b+i_1|!}\rg)^\sigma \\
&\hspace{5.7cm}\times\lambda^{(|m,n-a-b+i_1|-|m',n'|)s}\sum_{b'+c'\leq i_1+i_2}\lf\|J^{(0,b',c')}_{m',n'}f_k\rg\|_{L^2(\text{supp}\{\chi_{m+n}\})}\Bigg)\\
\lesssim & \sum_{n'=(n-\mf m_0-2)_+}^{n-1}\sum_{a+b+c=\mf m_0} {\bf a}_{m,n'}\frac{|m,n'|!^{\sigma}}{|m,n|!^\sigma}\lambda^{(n-n')s}\|J_{m,n'}^{(a,b,c)}f_k\|_{L^2}\\
&\quad+\mathbbm{1}_{n<\mf m_0}\sum_{m'=(m-1)_+}^m\sum_{n'= 0}^{1}\sum_{b+c\leq \mf m_0}\lambda^{s(|m,n|-|m',n'|)} \lf(\frac{|m',n'|!}{|m,n|!}\rg)^{\sigma}{\bf a}_{m',n'}\|J^{(0,b,c)}_{m',n'}f_k\|_{L^2}.
\end{align*}
Here in the last line, we have used the fact that $|k|^{\mf m_0-i_1-i_2}\|J^{(\leq i_1+i_2)}_{m,n}f_k\|_{L^2}\leq  \|J^{(\leq  \mf m_0 )}_{m,n}f_k\|_{L^2}$.  
This concludes the proof. 
\end{proof}

\begin{lemma}[Translation Lemma]\label{lem:JtoD}Assume that $\sigma_\ast\geq 8\sigma$. 

\noindent 
{\bf a)} Assume that $\mf m_0\in\{1,2\}$. There exists a universal $\mathfrak{C}\geq 1$ such that the following two estimates hold\begin{subequations}\label{JtoD}
\begin{align}
\label{JtoD_1}
&\hspace{-0.4cm}\bold{a}_{m,n}^2\sum_{a+b+c=\mathfrak{m}_0} \lf\| J_{m,n}^{(a,b,c)} f_k\rg\|_{L^2} ^2 \\
\lesssim& \sum_{n'=0}^{n}\sum_{b+c\leq\mathfrak{m}_0} (\mathfrak C\lambda^s)^{2(n-n')}\lf(\frac{|m,n'|!}{|m,n|!}\rg)^{2\sigma}  \bold{a}_{m,n'}^2  \lf\| \chi_{m+n'}\pa_y^b|k|^c(|k|^mq^{n'}\Gamma_k^{n'}  f_k)\rg\|_{L^2} ^2 \n \\
&\hspace{-0.3cm}+\mathbbm{1}_{n<\mf m_0 }\sum_{m'=(m-1)_+}^m\sum_{n'=0}^{1}\sum_{b+c\leq\mathfrak{m}_0} (\mathfrak C\lambda^s)^{2(|m,n|-|m',n'|)}\lf(\frac{|m',n'|!}{|m,n|!}\rg)^{2\sigma}  \bold{a}_{m',n'}^2  \lf\| \chi_{m'+n'}\pa_y^b|k|^c(|k|^{m'}q^{n'}\Gamma_k^{n'}  f_k)\rg\|_{L^2} ^2.\n
\end{align}
Here the implicit constant depends on $\|\vyn\|_{L_{t,y}^\infty},\, \|v_y-1\|_{L_t^\infty H_y^3}$ and is independent of $m+n$. Moreover, if $\lambda^s\leq \frac{1}{2\mathfrak C}$, then \begin{align} \label{JtoD_2} 
\quad \sum_{m+n=0}^M\sum_{a+b+c=\mathfrak{m}_0}\bold{a}_{m,n}^2  \lf\| J_{m,n}^{(a,b,c)}f_k\rg\|_{L^2} ^2 \lesssim & \sum_{m+n=0}^{M}\sum_{b+c\leq\mathfrak{m}_0}\bold{a}_{m,n}^2 \lf\| \chi_{m+n}\pa_y^b|k|^c(|k|^mq^n\Gamma_k^n f_k)\rg\|_{L^2}^2, \ \quad \mathfrak{m}_0\in\{1,2\}. 
\end{align}
\end{subequations}Here the implicit constant depends on $\|\vyn\|_{L_{t,y}^\infty},\, \|v_y-1\|_{L_t^\infty H_y^3}$ and is independent of $M$. 

\noindent {\bf b)} Assume that $\mf m_0\in\{1,2,3\}.$ There exists a universal constant $\mf C\geq 1$ such that the following estimate holds\begin{subequations}\label{JtoJ} 
\begin{align} \label{JtoJ_1}
& \bold{a}_{m,n}^2\sum_{\substack{a+b+c=\mathfrak{m}_0\\ a\geq 1 }} \lf\| J_{m,n}^{(a,b,c)} f_k\rg\|_{L^2} ^2\\ &\lesssim \sum_{n'=0}^{n-1}\sum_{b+c\leq\mathfrak{m}_0}  (\mathfrak C\lambda^s)^{2(n-n')}\lf(\frac{|m,n'|!}{|m,n|!}\rg)^{2\sigma}  \bold{a}_{m,n'}^2  \lf\| J^{(0,b,c)}_{m,n'}f_k\rg\|_{L^2}^2\n\\
&\quad+\mathbbm{1}_{n< \mf m_0}\sum_{m'=(m-1)_+}^m\sum_{n'=0}^{1}\sum_{b+c\leq\mathfrak{m}_0}  (\mathfrak C\lambda^s)^{2(|m,n|-|m',n'|)}\lf(\frac{|m',n'|!}{|m,n|!}\rg)^{2\sigma}  \bold{a}_{m',n'}^2  \lf\| J^{(0,b,c)}_{m',n'}f_k\rg\|_{L^2}^2.\n
\end{align}
Here the implicit constant depends on $\|\vyn\|_{L_{t,y}^\infty},\, \|v_y-1\|_{L_t^\infty H_y^3}$ and is independent of $m+n$.  Moreover,  \begin{align}\label{JtoJ_2} 
\quad \sum_{m+n=0}^M\sum_{\substack{a+b+c=\mathfrak{m}_0\\ a\geq1} }\bold{a}_{m,n}^2  \lf\| J_{m,n}^{(a,b,c)}f_k\rg\|_{L^2} ^2 \lesssim & \siming{\lambda^{2s}}\sum_{m+n=0}^{{M-1}}\sum_{b+c\leq\mathfrak{m}_0}\bold{a}_{m,n}^2 \lf\| J^{(0,b,c)}_{m,n}f_k\rg\|_{L^2}^2, \ \quad \mathfrak{m}_0\in\{1,2,3\}. 
\end{align}
\siming{There should be $a>0$, otherwise, the $M-1$ is questionable. And do we have $\lambda$ here?}
\end{subequations}
Here the implicit constant depends on $\|\vyn\|_{L_{t,y}^\infty},\, \|v_y-1\|_{L_t^\infty H_y^3}$ and is independent of $M$.\end{lemma}
\begin{remark}
Here we highlight that the estimates \eqref{JtoD} cost more regularity than \eqref{JtoJ}. The \eqref{JtoD} estimates are used to convert the $J$-estimates to the diffusive bounds $\mathcal{D}$. On the other hand, the \eqref{JtoJ} estimates are used to treat various $\{J^{(a,b,c)}_{m,n}\phe_k\}_{a\neq 0}$ terms in the elliptic section.
\end{remark}
\begin{proof}
To make the flow of the proof more accessible, we organize the material in three steps. In Step \# 1, we prove estimate \eqref{JtoJ_1} because it is a direct consequence of the previous induction lemma and combinatorics; In Step \# 2, we use the trick in Step \# 1 to handle more complicated situation  \eqref{JtoD_1}, \eqref{JtoD_2}; In the final Step \# 3, we will comment that the remaining estimate \eqref{JtoJ_2} is similar to the treatment in Step \# 2.

\noindent
{\bf Step \# 1: Proof of the estimate \eqref{JtoJ_1}. } We plan to apply the induction relation \eqref{J_al>0_l2}. For presentation purposes, we first consider the case where $n\geq 5(\mf m_0+2)$ to avoid complicated lower order terms that might appear in the first few steps of induction. The $n<5(\mf m_0+2)$ case will be naturally handled when we complete the induction argument. We apply the relation \eqref{J_al>0_l2},
\begin{align*} 
{\bf a}_{m,n}^2&\sum_{\substack{a+b+c=\mf m_0,\\ a\geq 1}}\|J^{(a,b,c)}_{m,n}f_k\|_{L^2}^2\lesssim \sum_{n_1= (n-\mathfrak{m}_0-2)_+}^{n-1}\lambda^{2(n-n_1)s}{\bf a}_{m,n_1}^2\lf(\frac{|m,n_1|!}{|m,n|!}\rg)^{2\sigma}\|J^{(\mf m_0)}_{m,n_1}f_k\|_{L^2}^2\\
\lesssim&\sum_{n_1= n-1}\lambda^{2(n-n_1)s}{\bf a}_{m,n_1}^2\lf(\frac{|m,n_1|!}{|m,n|!}\rg)^{2\sigma}\sum_{b+c= \mf m_0}\|J^{(0,b,c)}_{m,n_1}f_k\|_{L^2}^2\\
&+\sum_{n_1= (n-\mathfrak{m}_0-2)_+}^{n-2}\lambda^{2(n-n_1)s}{\bf a}_{m,n_1}^2\lf(\frac{|m,n_1|!}{|m,n|!}\rg)^{2\sigma}\sum_{b+c= \mf m_0}\|J^{(0,b,c)}_{m,n_1}f_k\|_{L^2}^2\\
&+\sum_{n_1= (n-\mathfrak{m}_0-2)_+}^{n-1}\sum_{n_2=(n_1-\mathfrak{m}_0-2)_+}^{n_1-1}\lambda^{2(n-n_2)s}{\bf a}_{m,n_2}^2\lf(\frac{|m,n_1|! }{|m,n|! } \frac{|m,n_2|! }{|m,n_1|!}\rg)^{2\sigma}\|J^{(\mf m_0)}_{m,n_2}f_k\|_{L^2}^2.
\end{align*}Here in the second inequality, we single out the terms with $n_1=n-1$, identify the $J^{(a,b,c)}$-terms with $a\neq 0$ and apply the induction estimate \eqref{J_al>0_l2}. First, we note that the $n_1=n-1$ will no longer appear in the latter expansion because $n_2$ ends at $n_1-1\leq n-2$. Next we observe that the $|m,n_1|!$ cancels out in the last term. Hence, the coefficients in the sum only ``remembers'' the starting position $(m,n)$ and the ending position $(m,n_2)$. With these observations, we do another induction and end up with
\begin{align*}
{\bf a}_{m,n}^2&\sum_{\substack{a+b+c=\mf m_0,\\ a\geq 1}}\|J^{(a,b,c)}_{m,n}f_k\|_{L^2}^2\\
\lesssim&\sum_{\ell= n-2}^{n-1}\mathfrak{C}^{n-\ell} \lambda^{2(n-\ell)s}{\bf a}_{m,\ell}^2\lf(\frac{|m,\ell|!}{|m,n|!}\rg)^{2\sigma}\sum_{b+c=\mf m_0}\|J^{(0,b,c)}_{m,\ell}f_k\|_{L^2}^2\\
&+\sum_{n_1= (n-\mathfrak{m}_0-2)_+}^{n-3}\lambda^{2(n-n_1)s}{\bf a}_{m,n_1}^2\frac{|m,n_1|!^{2\sigma}}{|m,n|!^{2\sigma}}\sum_{b+c= \mf m_0}\|J^{(0,b,c)}_{m,n_1}f_k\|_{L^2}^2\\
&+\sum_{n_1= (n-\mathfrak{m}_0-2)_+}^{n-1}\sum_{n_2=(n_1-\mathfrak{m}_0-2)_+}^{n_1-1}\mathbbm{1}_{n_2\neq n-2}\lambda^{2(n-n_2)s}{\bf a}_{m,n_2}^2\lf(  \frac{|m,n_2|! }{|m,n|! }\rg)^{2\sigma}\sum_{b+c=\mf m_0}\|J^{(0,b,c)}_{m,n_2}f_k\|_{L^2}^2\\
&+\sum_{n_1= (n-\mathfrak{m}_0-2)_+}^{n-1}\sum_{n_2=(n_1-\mathfrak{m}_0-2)_+}^{n_1-1}\sum_{n_3=(n_2-\mathfrak{m}_0-2)_+}^{n_2-1}\lambda^{2(n-n_3)s}{\bf a}_{m,n_3}^2\lf(\frac{|m,n_3|! }{|m,n|! }\rg)^{2\sigma}\|J^{(\mf m_0)}_{m,n_3}f_k\|_{L^2}^2.
\end{align*} This process will be stopped after $\approx n$ steps. 
Here the constant $\mathfrak{C}=\mathfrak{C}(\mf m_0)$ is introduced to account for the repetition of the $J_{m,\ell}^{(0,b,c)} f_k$ terms. For example, in the first sum, the term $\ell=n-2$ appears at most $\mathfrak m_0+3$ times and it will not appear in the latter induction expansion. For other terms $J_{m,\ell}^{(0,b,c)}$, the total number of repetition increases as the distance $|m,n|-|m,\ell|$ increases. Thanks to the iterative sum structure in the last term, the number of repetition is exponential in terms of $|m,n|-|m,\ell|$. Hence,
\begin{align*}
{\bf a}_{m,n}^2\mathbbm{1}_{\mf   a\geq 1}\|J^{(\mf m_0)}_{m,n}f_k\|_{L^2}^2 
\lesssim &\sum_{\ell=0}^{n-1}\sum_{b+c= \mf m_0}{\mf C^{n-\ell}}\lambda^{2(n-\ell)s}{\bf a}_{m,\ell}^2\lf(\frac{|m,\ell|!}{|m,n|!}\rg)^{2\sigma}\|J^{(0,b,c)}_{m,\ell}f_k\|_{L^2}^2\\
&\hspace{-1cm}+\sum_{m'=(m-1)_+}^m\sum_{\ell=0}^{1}\sum_{b+c\leq\mathfrak{m}_0}  \mathfrak C^{|m,n|-|m',\ell|}\lambda^{2s(|m,n|-|m',\ell|)}\lf(\frac{|m',\ell|!}{|m,n|!}\rg)^{2\sigma}  \bold{a}_{m',\ell}^2  \lf\| J^{(0,b,c)}_{m',\ell}f_k\rg\|_{L^2}^2.  
\end{align*} 
Now all the terms appeared in the sum only have $J^{(0,b,c)}_{m,n}f_k$ components and are consistent with \eqref{JtoJ_1}. Since the $n<5(\mf m_0+2)$ case can be treated in an identical fashion, this concludes the proof of \eqref{JtoJ_1}.
 \siming{The $\mf C^{n-\ell}$ factor comes from counting. \footnote{\siming{HS: Double check this statement because there is a counting problem here. It can be phrased like this: Each time you take $1,2,3$ steps, and how many different ways to get to $n-\ell$ level. We can write this as a recursion problem: $a(n)=$ ways to get to $n$. We check that $a(1)=1,\,a(2)=2,\, a(3)=4$. To reach $n$, first you need to reach $(n-1)$, $(n-2)$ or $(n-3)$ and then take $1$, $2$ or $3$ steps. Hence, $a(n)=a(n-1)+a(n-2)+a(n-3)$. When I search on wolfram alpha, the bound is exponential $a(n)\lesssim \mf C^{n}$.}}}

    \noindent
{\bf Step \# 2: Proof of the  estimates \eqref{JtoD_1}, \eqref{JtoD_2}. } Thanks to the induction relation \eqref{J_al>0_l2}, we can first focus on the case where $a=0,\  b+c=\mf m_0$ because the $a>0$ case can be reduced to this case:
\begin{align}\n
&\hspace{-0.5cm}{\bf a}_{m,n}\sum_{b+c=\mf m_0}\|J_{m,n}^{(0,b,c)}f_k\|_{L^2}\\
\n\leq& {\bf a}_{m,n}\sum_{b+c=\mf m_0}\|\chi_{m+n}\pav^b |k|^c(q^n \Gamma_k^n|k|^m f_k)\|_{L^2}+{\bf a}_{m,n}\sum_{b+c=\mf m_0}\| |k|^c[\pav^b,\chi_{m+n}](q^{n} \Gamma_k^{n}|k|^m f_k)\|_{L^2}\\
\n\lesssim& {\bf a}_{m,n}\sum_{b+c=\mf m_0}\|\chi_{m+n}\pav^b |k|^c(q^n \Gamma_k^n|k|^m f_k)\|_{L^2}\\
\n&+{\bf a}_{m,n}\sum_{b+c=\mf m_0}\sum_{i_1=1}^{b}\binom{b}{i_1}\sum_{i_2=0}^{i_1}\binom{i_1}{i_2}\lf\| |k|^c \lf(\pav^{i_1}\chi_{m+n} \rg) \pav^{b-i_1}(q^{i_1}q^{n-i_1}\pav^{i_2}(ikt)^{i_1-i_2} \Gamma_k^{n-i_1}|k|^m f_k)\rg\|_{L^2}\\
\n\lesssim& {\bf a}_{m,n}\sum_{b+c\leq\mf m_0}\|\chi_{m+n}\pa_y^b |k|^c(q^n \Gamma_k^n|k|^m f_k)\|_{L^2}\\
\n&+\sum_{b+c=\mf m_0}\sum_{i_1=1}^b\sum_{i_2=0}^{i_1}\frac{{\bf a}_{m,n-i_1}\varphi^{i_1}t^{i_1-i_2}\lambda^{i_1s}}{(m+n)^{i_1\sigma_\ast}} |k|^{c+i_1-i_2} \|\chi_{m+n-i_1} \pav^{b-i_1}(q^{i_1} \pav^{i_2}(q^{n-i_1}\Gamma_k^{n-i_1}|k|^m f_k))\|_{L^2}\\
&+\sum_{b+c=\mf m_0}\sum_{i_1=1}^b\sum_{i_2=0}^{i_1}\frac{{\bf a}_{m,n-i_1}\varphi^{i_1}t^{i_1-i_2}\lambda^{i_1s}}{(m+n)^{i_1\sigma_\ast}}|k|^{c +i_1-i_2} \| \pav^{b-i_1}(q^{i_1}[\pav^{i_2} ,q^{n-i_1}]\Gamma_k^{n-i_1}|k|^m f_k))\|_{L^2(\text{supp}\{\chi_{m+n}\})}\n \\
=:&T_1+T_2+T_3.\label{JtoD_pf1}
\end{align}%
The $T_1$ term is consistent with the \eqref{JtoD_1}. For the $T_2$-term, we estimate it as follows 
\begin{align}\n
T_2\lesssim& \sum_{b+c=\mf m_0}\sum_{i_1=1}^b\sum_{i_2=0}^{i_1}\sum_{i_3=0}^{b-i_1}\frac{{\bf a}_{m,n-i_1}\lambda^{i_1s}}{(m+n)^{i_1\sigma_\ast}}  \|\chi_{m+n-i_1}  |k|^{c+i_1-i_2}\pav^{i_2+i_3}(q^{n-i_1}\Gamma_k^{n-i_1}|k|^m f_k))\|_{L^2}\\
\lesssim&\sum_{n'=(n-\mf m_0)_+}^{n-1}\sum_{b+c\leq\mf m_0}{\bf a}_{m,n'}\lambda^{(n-n')s}\lf(\frac{|m,n'|!}{|m,n|!}\rg)^{\sigma_\ast} \|\chi_{m+n'}  \pa_y^{b}|k|^{c}(q^{n'}\Gamma_k^{n'}|k|^m f_k)\|_{L^2}.\label{JtoD_pf2}
\end{align}
Next we estimate the $T_3$-term using the product rule,
\begin{align*}\n
T_3\lesssim \sum_{b+c=\mf m_0}\sum_{i_1=1}^b\sum_{i_2=0}^{i_1}\sum_{i_3=0}^{b-i_1}&\frac{{\bf a}_{m,n-i_1}\varphi^{i_2}\lambda^{i_1s}}{(m+n)^{i_1\sigma_\ast}}|k|^{c+i_1-i_2}\\
&\times\lf\|q^{(i_1-i_3)_+}\pav^{(b-i_1-i_3)}\lf(\lf[\pav^{i_2},q^{n-i_1}\rg] |k|^m\Gamma_k^{n-i_1}f_k \rg)\rg\|_{L^2(\text{supp}\ \chi_{m+n})}.
\end{align*}
Here we note that $(b-i_1-i_3)+i_2\leq b\leq \mf m_0\leq 2$. Hence we invoke the estimates $\eqref{cm_qn_pv_1}_{j=2}$, $\eqref{cm_qn_pv_2}$ 
to obtain that 
\begin{align}\n
T_3\lesssim &\sum_{n'=(n-\mf m_0-1)_+}^{n-1}\sum_{i_1=1}^b\frac{{\bf a}_{m,n-i_1}\lambda^{i_1s}}{(m+n)^{i_1\sigma_\ast}}\lf\|J^{(\leq \mf m_0)}_{m,n-i_1}f_k \rg\|_{L^2(\text{supp}\ \chi_{m+n})}\\
&+\mathbbm{1}_{n<2\mf m_0}\sum_{m'=(m-1)_+}^m{\bf a}_{m',0}\lambda^{(|m,n|-m')s}\lf(\frac{m'!}{|m,n|!}\rg)^{\sigma_\ast}\sum_{b+c\leq \mf m_0}\lf\|\chi_{m'}\pa_y^b|k|^c(|k|^{m'}f_k) \rg\|_{L^2(\text{supp}\ \chi_{m+n})}.\label{JtoD_pf3}
\end{align} 
Hence, combining the decomposition \eqref{JtoD_pf1} for the $\{a=0, b+c=\mf m_0\}$ case and the estimates \eqref{JtoD_pf2}, \eqref{JtoD_pf3}, we further apply the induction relation \eqref{J_al>0_l2} to take the $\{a\neq 0, a+b+c=\mf m_0\}$ cases into account and end up with the following
\begin{align*}
&\hspace{-0.5cm}{\bf a}_{m,n}\sum_{a+b+c=\mf m_0}\|J_{m,n}^{(a,b,c)}f_k\|_{L^2}\\
\lesssim& \sum_{n'=(n-\mf m_0-1)_+}^{n}\sum_{b+c\leq\mf m_0}{\bf a}_{m,n'}\lambda^{(n-n')s}\lf(\frac{|m,n'|!}{|m,n|!}\rg)^{\sigma_\ast} \|\chi_{m+n'}  \pa_y^{b}|k|^{c}(q^{n'}\Gamma_k^{n'}|k|^m f_k)\|_{L^2}\\
&+\sum_{n'=(n-\mf m_0-1)_+}^{n-1}\sum_{a+b+c\leq\mf m_0}{\bf a}_{m,n'}\lambda^{(n-n')s}\lf(\frac{|m,n'|!}{|m,n|!}\rg)^{\sigma_\ast} \| J^{(a,b,c)}_{m,n'} f_k\|_{L^2}\ \mathbbm{1}_{a\geq 1}\\
&+\mathbbm{1}_{n<2\mf m_0}\sum_{m'=(m-1)_+}^m{\bf a}_{m',0}\lambda^{(|m,n|-m')s}\lf(\frac{m'!}{|m,n|!}\rg)^{\sigma_\ast}\sum_{b+c\leq \mf m_0}\lf\|\chi_{m'}\pa_y^b|k|^c(|k|^{m'}f_k) \rg\|_{L^2}.&&
\end{align*}
Now we apply the induction estimates \eqref{J_al>0_l2} on the second term in the expression to translate all the terms in to the form $J^{(0,b,c)}_{m',n'}f_k$. We further note that each distinct $J_{m,n}^{(0,b,c)}f_k$ can only appear at most exponentially in $m+n$ many times during the inductive expansion. Now we apply the H\"older inequality and the combinatorial estimate \eqref{sum_comb} to obtain \eqref{JtoD_1}. 
\ifx\begin{align*}
{\bf a}_{m,n}^2\sum_{(a,b,c)\in \myr{\mathbb  S_n^{\mf m_0}}}\|J_{m,n}^{(a,b,c)}f_k\|_{L^2}^2\leq & C\lf(\sum_{\ell=0}^{n}\sum_{b+c=\mathfrak{m}_0} \mathfrak C^{n-\ell
+1}\lambda^{(n-\ell)s}\lf(\frac{(m+\ell)!}{(m+n)!}\rg)^{\sigma}  \bold{a}_{m,\ell}  \lf\| \chi_{m+\ell}\pa_y^b|k|^c(|k|^mq^\ell\Gamma_k^\ell f_k)\rg\|_{L^2}\rg)^2\\
\leq&C \sum_{\ell=0}^{n}\sum_{b+c=\mathfrak{m}_0} \mathfrak (2\mathfrak{C}\lambda^s)^{2(n-\ell)}\lf(\frac{(m+\ell)!}{(m+n)!}\rg)^{2\sigma}  \bold{a}_{m,\ell}^2  \lf\| \chi_{m+\ell}\pa_y^b|k|^c(|k|^mq^\ell\Gamma_k^\ell f_k)\rg\|_{L^2}^2.
\end{align*}This is the result \eqref{JtoD_1}.
\fi
To derive \eqref{JtoD_2}, we sum over both side of the expression  \eqref{JtoD_1} and apply the Fubini 
\begin{align*}
&\hspace{-0.5cm}\sum_{m+n=0}^M{\bf a}_{m,n}^2\sum_{a+b+c=\mf m_0}\|J_{m,n}^{(a,b,c)}f_k\|_{L^2}^2\\
\leq &C \sum_{m=0}^M\sum_{n=0}^{M-m}\sum_{\ell=0}^{n}\sum_{b+c\leq\mathfrak{m}_0} \mathfrak ( \mathfrak{C}\lambda^s)^{2(n-\ell)}\lf(\frac{|m,\ell|!}{|m,n|!}\rg)^{2\sigma}  \bold{a}_{m,\ell}^2  \lf\| \chi_{m+\ell}\pa_y^b|k|^c(|k|^mq^\ell\Gamma_k^\ell f_k)\rg\|_{L^2}^2\\
&+C\sum_{m=0}^M\sum_{n=0}^{M-m}\sum_{m'=(m-1)_+}^m\sum_{b+c\leq\mathfrak{m}_0}(\mf C\lambda^s)^{2(|m,n|-m')}\lf(\frac{m'!}{|m,n|!}\rg)^{2\sigma}{\bf a}_{m',0}^2\lf\|\chi_{m'}\pa_y^b|k|^c(|k|^{m'}f_k )\rg\|_{L^2}^2\\
\leq&C \sum_{m=0}^M\sum_{\ell=0}^{M-m}\sum_{n=\ell}^{M-m}\sum_{b+c\leq\mathfrak{m}_0} \mathfrak ( \mathfrak{C}\lambda^s)^{2(n-\ell)} \bold{a}_{m,\ell}^2  \lf\| \chi_{m+\ell}\pa_y^b|k|^c(|k|^mq^\ell\Gamma_k^\ell f_k)\rg\|_{L^2}^2\\
&+C\sum_{m=0}^M \sum_{b+c\leq\mathfrak{m}_0}{\bf a}_{m,0}^2\lf\|\chi_{m} \pa_y^b |k|^c(|k|^m f_k) \rg\|_{L^2}^2\\
\leq& C \sum_{m+\ell=0}^M\sum_{b+c\leq\mathfrak{m}_0} \bold{a}_{m,\ell}^2  \lf\| \chi_{m+\ell}\pa_y^b|k|^c(|k|^mq^\ell\Gamma_k^\ell f_k)\rg\|_{L^2}^2.
\end{align*} 
This is \eqref{JtoD_2}.

\noindent
{\bf Step \# 3: Proof of the estimate \eqref{JtoJ_2}. } Now the derivation of \eqref{JtoJ_2} is similar to {\bf Step \# 2} and we omit the details.  
\end{proof}

\ifx
\siming{
\begin{lemma}\label{Lem:H&vy2-1}
Assume \eqref{asmp}. The following estimate holds
\begin{align} 
\sum_{n=0}^\infty &\sum_{a+b=0}^1 B_{0,n}^2\varphi^{2n}\|J^{(a,b,0)}_{0,n}(\wt \chi_1(v_y^2-1))\|_{L^2}^2 
\lesssim  \nu^{100}\mathcal{E}_H^{(\al)}\lf(\mathcal{E}_H^{(\al)}+1\rg);\label{Hvy^2-1-H1}\\
\sum_{n=0}^\infty& \sum_{a+b=2} B_{0,n}^2\varphi^{2n}\|J^{(a,b,0)}_{0,n}(\wt \chi_1(v_y^2-1))\|_{L^2}^2 
\lesssim\nu^{100}(\mathcal{E}_H^{(\al)}\mathcal{D}_H^{(\al)}+\mathcal{E}_H^{(\al)} +\mathcal{D}_H^{(\al)}).\label{Hvy^2-1-H2}
\end{align}
\end{lemma}
\begin{proof}
To prove the estimate \eqref{Hvy^2-1-H1}, we reformulate the left hand for $n\neq 0$ side as follows:\siming{??}
\begin{align*}
&\sum_{n=0}^\infty \sum_{a+b=0}^1 B_{0,n}^2\varphi^{2n}\lf\|J_{0,n}^{(a,b,0)}(\wt \chi_1 (v_y^2-1))\rg\|_{L^2}^2\\
&=\sum_{n=0}^\infty \sum_{a+b=0}^1 B_{0,n}^2\varphi^{2n}\lf\|\lf(\frac{n}{q}\rg)^a\pav^b \lf(\chi_n q^n \pav^n((H+2)H\rg)\rg\|_{L^2}^2
\\ &\lesssim\sum_{n=0}^\infty \sum_{a+b=0}^1 B_{0,n}^2\varphi^{2n}\lf(\lf\|\sum_{i_2=1}^{n-1}\lf(\frac{n}{q}\rg)^a[\pav^b,\chi_n]\lf(q^{n-i_2}\pav^{n-i_2}H \ q^{i_2}\pav^{i_2}(H+2)\rg)\rg\|_{L^2}\rg)^2\\
&+\sum_{n=0}^\infty \sum_{a+b=0}^1 B_{0,n}^2\varphi^{2n}\lf(\sum_{i_1=0}^b\sum_{i_2=0}^n \binom{b}{i_1}\binom{n}{i_2}\lf\|\mathbbm{1}_{\text{supp}\chi_n}\lf(\frac{n}{q}\rg)^a\pav^{b-i_1}(\chi_{n-i_2} q^{n-i_2}\pav^{n-i_2}H) \ \pav^{i_1}(\chi_{i_2}q^{i_2}\pav^{i_2}(H+2))\rg\|_{L^2}\rg)^2\\
&\lesssim\mathbbm{1}_{b=1}\sum_{n=0}^\infty   B_{0,n}^2\varphi^{2n}\lf\|\chi_n' v_y^{-1}\sum_{i_2=1}^{n-1} \binom{n}{i_2} q^{n-i_2}\pav^{n-i_2}H \ q^{i_2}\pav^{i_2}(H+2)\rg\|_{L^2}^2\\
&\quad+\sum_{n=0}^\infty \sum_{a+b=0}^1\sum_{i_1=0}^b B_{0,n}^2\varphi^{2n}\lf(\sum_{i_2=0}^n\binom{n}{i_2}\lf\|\mathbbm{1}_{\text{supp}\chi_n}\lf(\frac{n}{q}\rg)^a \pav^{b-i_1}(\chi_{n-i_2}q^{n-i_2}\pav^{n-i_2}H)\ \pav^{i_1}(\chi_{i_2}q^{i_2}\pav^{i_2}(H+2))\rg\|_{L^2}\rg)^2\n \\
&=:T_1+T_2.
\end{align*}
Now we estimate the $T_1$ term in the above expression. To this end, we invoke the estimates $\|\vyn\|_{L_{t,y}^\infty}\lesssim 1$,  \eqref{chi:prop:3}, and the combinatorial estimate \eqref{prod} to obtain that \siming{(Sketch, the index might have issue for $n=0, etc.$)}
\begin{align*}
 &T_1\lesssim\sum_{n=0}^\infty   B_{0,n}^2\varphi^{2n}\lf(\sum_{i_2=0}^n \binom{n}{i_2}n^{1+\sigma}\lf\|  q^{n-i_2}\pav^{n-i_2}H \ q^{i_2}\pav^{i_2}(H+2)\rg\|_{L^2}\rg)^2\\
 &\lesssim\sum_{n=0}^\infty   \lf(\sum_{i_2=0}^{\lfloor n/2 \rfloor} \frac{B_{0,n-1} n^{-\sigma_\ast}}{B_{0,n-i_2-1}B_{0,i_2}} \binom{n-1}{i_2}\frac{n}{n-i_2} B_{0,n-i_2-1}\varphi^{n-i_2}\lf\|\chi_{n-i_2-1}  q^{n-i_2-1}\pav\ \pav^{n-i_2-1}H\rg\|_{L^2}\rg.\\
&\quad\lf. \times B_{0,i_2} \varphi^{i_2}\lf\| \chi_{i_2}q^{i_2}\pav^{i_2}(H+2)\rg\|_{L^\infty}\rg)^2\\
 &\quad+\sum_{n=0}^\infty   \lf(\sum_{i_2=\lfloor n/2 \rfloor}^n \frac{B_{0,n-1} n^{-\sigma_\ast}}{B_{0,n-i_2}B_{0,i_2-1}} \binom{n-1}{i_2-1}\frac{n}{i_2} B_{0,n-i_2}\varphi^{n-i_2}\lf\|\chi_{n-i_2}  q^{n-i_2}\pav^{n-i_2}H\rg\|_{L^\infty}\rg.\\
&\quad\lf. \times B_{0,i_2-1} \varphi^{i_2}\lf\| \chi_{i_2-1}q^{i_2-1}\pav\ \pav^{i_2-1}(H+2)\rg\|_{L^2}\rg)^2\\
 &\lesssim\sum_{n=0}^\infty   \lf(\sum_{i_2=0}^{\lfloor n/2 \rfloor} \binom{n-1}{i_2}^{-2\sigma_\ast}\rg)\sum_{i_2=0}^{\lfloor n/2 \rfloor} B_{0,n-i_2-1}^2\varphi^{2n-2i_2}\lf\|\chi_{n-i_2-1}  q^{n-i_2-1}\pav\ \pav^{n-i_2-1}H\rg\|_{L^2}^2\\
 &\qquad\qquad\times B_{0,i_2}^2 \varphi^{2i_2}\lf\| \chi_{i_2}q^{i_2}\pav^{i_2}(H+2)\rg\|_{L^\infty}^2\\
 & \quad+\sum_{n=0}^\infty   \lf(\sum_{i_2=\lfloor n/2 \rfloor}^n \binom{n-1}{i_2-1}^{-2\sigma_\ast}\rg) \sum_{i_2=\lfloor n/2 \rfloor}^n  B_{0,n-i_2}^2\varphi^{2n-2i_2}\lf\|\chi_{n-i_2}  q^{n-i_2}\pav^{n-i_2}H\rg\|_{L^\infty}^2\\
&\qquad\qquad \times B_{0,i_2-1} ^2\varphi^{2i_2}\lf\| \chi_{i_2-1}q^{i_2-1}\pav\ \pav^{i_2-1}(H+2)\rg\|_{L^2}^2\\
&\lesssim \lf(\sum_{n=0}^\infty\sum_{a+b=0}^1B_{0,n}^2\varphi^{2n}\|J_{0,n}^{(a,b,0)} H \|_{L^2}^2\rg)\lf(\sum_{n=0}^\infty\sum_{a+b=0}^1B_{0,n}^2\varphi^{2n}\|J_{0,n}^{(a,b,0)} H \|_{L^2}^2+1\rg).
\end{align*}
 To estimate the $T_2$ term, we apply similar strategy,\siming{(Double check the $\chi_0$!)}
\begin{align*}
T_2\lesssim&\sum_{n=0}^\infty \sum_{a+b=0}^1\sum_{i_1=0}^b\lf( \sum_{i_2=0}^n \binom{n}{i_2}B_{0,n} \varphi^{n}\lf\|\lf(\frac{n}{q}\rg)^a \pav^{b-i_1}(\chi_{n-i_2}q^{n-i_2}\pav^{n-i_2}H)\ \pav^{i_1}(\chi_{i_2}q^{i_2}\pav^{i_2}(H+2))\rg\|_{L^2}\rg)^2\\\lesssim&\sum_{n=0}^\infty \sum_{a+b=0}^1\sum_{i_1=0}^b\lf( \sum_{i_2=1}^{n-1} \binom{n}{i_2}B_{0,n} \varphi^{n}\lf\|\lf(\frac{n-i_2}{q}\rg)^a \pav^{b-i_1}(\chi_{n-i_2}q^{n-i_2}\pav^{n-i_2}H)\ \pav^{i_1}(\chi_{i_2}q^{i_2}\pav^{i_2}(H+2))\rg\|_{L^2}\rg)^2\\
&+\sum_{n=0}^\infty \sum_{a+b=0}^1\sum_{i_1=0}^b\lf( \sum_{i_2=1}^{n-1} \binom{n}{i_2}B_{0,n} \varphi^{n}\lf\| \pav^{b-i_1}(\chi_{n-i_2}q^{n-i_2}\pav^{n-i_2}H)\ \lf(\frac{i_2}{q}\rg)^a\pav^{i_1}(\chi_{i_2}q^{i_2}\pav^{i_2}(H+2))\rg\|_{L^2}\rg)^2\\
&+\sum_{n=0}^\infty \sum_{a+b=0}^1\sum_{i_1=0}^b\lf( \sum_{i_2=0,n} \binom{n}{i_2}B_{0,n} \varphi^{n}\lf\|\lf(\frac{n}{q}\rg)^a \pav^{b-i_1}(\chi_{n-i_2}q^{n-i_2}\pav^{n-i_2}H)\ \pav^{i_1}(\chi_{i_2}q^{i_2}\pav^{i_2}(H+2))\rg\|_{L^2}\rg)^2\\
=:&T_{21}+T_{22}+T_{23}.
\end{align*}
We first estimate the $T_{21}, \ T_{22}$ terms
\begin{align*}
T_{21}+T_{22}\lesssim&\sum_{n=0}^\infty  \lf( \sum_{i_2=0}^n \binom{n}{i_2}^{-\sigma}B_{0,n-i_2} \varphi^{n-i_2}\lf\|J^{(1,0,0)}_{0,n-i_2}H\rg\|_{L^2}  B_{0,i_2} \varphi^{i_2}\lf\|\chi_{i_2}q^{i_2}\pav^{i_2}(H+2)\rg\|_{L^\infty}\rg)^2\\
&+\sum_{n=0}^\infty \sum_{i_1=0}^1\lf( \sum_{i_2=0}^n \binom{n}{i_2}^{-\sigma} \lf\|\lf(B_{0,n-i_2} \varphi^{n-i_2}J^{(0,1-i_1,0)}_{0,n-i_2}H\rg)\lf (B_{0,i_2} \varphi^{i_2} J^{(0,i_1,0)}_{0,i_2}(H+2)\rg)\rg\|_{L^2}\rg)^2\\
&+\sum_{n=0}^\infty  \lf( \sum_{i_2=0}^n \binom{n}{i_2}^{-\sigma}B_{0,n-i_2} \varphi^{n-i_2}\lf\|J^{(0,0,0)}_{0,n-i_2}H\rg\|_{L^\infty}  B_{0,i_2} \varphi^{i_2}\lf\|J^{(1,0,0)}_{0,i_2}(H+2)\rg\|_{L^2}\rg)^2\\
&+\sum_{n=0}^\infty \sum_{i_1=0}^1\lf( \sum_{i_2=0}^n \binom{n}{i_2}^{-\sigma} \lf\|\lf(B_{0,n-i_2} \varphi^{n-i_2}J^{(0,1-i_1,0)}_{0,n-i_2}H\rg)\lf (B_{0,i_2} \varphi^{i_2} J^{(0,i_1,0)}_{0,i_2}(H+2)\rg)\rg\|_{L^2}\rg)^2\\
\lesssim& \lf(\sum_{n=0}^\infty\sum_{a+b=0}^1B_{0,n}^2\varphi^{2n}\|J_{0,n}^{(a,b,0)} H \|_{L^2}^2\rg)\lf(\sum_{n=0}^\infty\sum_{a+b=0}^1B_{0,n}^2\varphi^{2n}\|J_{0,n}^{(a,b,0)} H \|_{L^2}^2+1\rg)\lesssim \nu^{100}\mathcal{E}_H^{(\al)}\lf(\mathcal{E}_H^{(\al)}+1\rg).
\end{align*} 
Here in the last line, we have invoked Lemma \ref{lem:JtoD} and the definition of $\mathcal{E}_H^{(\al)}$ \eqref{E_coord}. 
The remaining term can be estimated as follows
\begin{align*}
T_{23}
\lesssim&\sum_{n=0}^\infty \sum_{a+b=0}^1\sum_{i_1=0}^b\sum_{i_2=0,n}\lf(B_{0,n} \varphi^{n}\lf\|\lf(\frac{\max\{n-i_2, i_2\}}{q}\rg)^a \pav^{b-i_1}(\chi_{n-i_2}q^{n-i_2}\pav^{n-i_2}H)\ \pav^{i_1}(\chi_{i_2}q^{i_2}\pav^{i_2}(H+2))\rg\|_{L^2}\rg)^2\\
\lesssim& \lf(\sum_{n=0}^\infty\sum_{a+b=0}^1B_{0,n}^2\varphi^{2n}\|J_{0,n}^{(a,b,0)} H \|_{L^2}^2\rg)\lf(\sum_{n=0}^\infty\sum_{a+b=0}^1B_{0,n}^2\varphi^{2n}\|J_{0,n}^{(a,b,0)} H \|_{L^2}^2+1\rg)\lesssim \nu^{100}\mathcal{E}_H^{(\al)}\lf(\mathcal{E}_H^{(\al)}+1\rg).
\end{align*}
This, when combined with the above estimates, yields the bound \eqref{Hvy^2-1-H1}.

The treatment of \eqref{Hvy^2-1-H2} is similar \siming{(Have not done? When we spread the $(nq^{-1})^a$, we have to be careful when $i_2\in\{0,...,n\}$ is too small ($\{0,1\}$) or too large ($\{n-1,n\}$). Thanks to the fact that $(nq^{-1})^a\approx(\max\{n-i_2,i_2\}q^{-1})^a$ in these two scenarios, the factor will naturally seeks the $q^{\max\{n-i_2,i_2\}}\pav^{\max\{n-i_2,i_2\}}$ part of the product and $J^{...}_{0,\max\{n-i_2,i_2\}}H$ is well-defined.)} and we omit the details.
\end{proof}
}
\fi

As an application of the ICC method, we present two lemmas that provide bounds on the coordinate quantity $Z=v_y^2-1$.

\begin{lemma} \label{lem:vy2-1est}
Assume $\sigma_\ast\geq 3\sigma. $ 
The following estimate holds for $Z=v_y^2-1$ and $n\geq 1$,
\begin{align}
B_{0,n }\|\chi_n q^n\pav^{n+1}Z\|_{L^2} \lesssim &B_{0,n }\|J^{(0,1,0)}_{0,n} Z\|_{L^2}+B_{0,n-1} n^{-\sigma_\ast}\sum_{a+b= 1}\|\mathbbm{1}_{\mathrm{supp}\chi_n}J^{(a,b,0)}_{0,n-1}Z\|_{L^2};\label{vy2_1est3} \\
B_{0,n }\|\chi_n q^n\pav^{n+2}Z\|_{L^2} \lesssim &B_{0,n }\|J^{(0,2,0)}_{0,n} Z\|_{L^2}+B_{0,n-1} n^{-\sigma}\sum_{a+b= 2}\|\mathbbm{1}_{\mathrm{supp}\chi_n}J^{(a,b,0)}_{0,n-1}Z\|_{L^2}\n\\
&+\mathbbm{1}_{n\in\{1,2\}}\underbrace{\sum_{\ell=0}^2\|\mathbbm{1}_{\mathrm{supp}\chi_1}\pav^\ell Z\|_{L^2}}_{=\sum_{0\leq b\leq 2}\|\mathbbm{1}_{\mathrm{supp}\chi_1}J^{(0,b,0)}_{0,0}Z\|_{L^2}}.\label{vy2_1est4} 
\end{align}
\end{lemma}
\begin{proof}
To derive the inequality \eqref{vy2_1est3}, 
\begin{align*}
B_{0,n }&\|\chi_n q^n\pav^{n+1}Z\|_{L^2}\\\lesssim&B_{0,n } \|\pav(\chi_n q^n\pav^{n }Z)\|_{L^2}+B_{0,n }\| \pav(\chi_n q^n) \pav (\chi_{n-1}\pav^{n-1 }Z)\|_{L^2}\\
\lesssim&B_{0,n }\|J^{(0,1,0)}_{0,n}Z\|_{L^2}+B_{0, n}n^{1+\sigma}\|\mathbbm{1}_{\mathrm{supp}\chi_n}q^{n-1}\pav (\chi_{n-1}\pav^{n-1}Z)\|_{L^2}\\
\lesssim&B_{0,n }\|J^{(0,1,0)}_{0,n}Z\|_{L^2}+\frac{\lambda^s B_{0,n-1} }{n^{\sigma_\ast}}\lf(\lf\|\mathbbm{1}_{\mathrm{supp}\chi_n}J^{(0,1,0)}_{0,n-1}Z\rg\|_{L^2}+\lf\|\mathbbm{1}_{\mathrm{supp}\chi_n}[q^{n-1},\pav](\chi_{n-1}\pav^{n-1}Z)\rg\|_{L^2}\rg)\\
\lesssim&B_{0,n }\|J^{(0,1,0)}_{0,n}Z\|_{L^2}+\frac{\lambda^s B_{0,n-1} }{n^{\sigma_\ast}}\sum_{a+b=1}\|\mathbbm{1}_{\mathrm{supp}\chi_n}J^{(a,b,0)}_{0,n-1}Z\|_{L^2}.
\end{align*}
To derive the inequality \eqref{vy2_1est4}, we expand the expression and obtain that, 
\begin{align}\n
B_{0,n }&\|\chi_n q^n\pav^{n+2}Z\|_{L^2}\\ \n
\lesssim&B_{0,n } \|\pav^2(\chi_n q^n\pav^{n }Z)\|_{L^2}+B_{0,n }\| \pav(\chi_n q^n) \pav^2 (\chi_{n-1}\pav^{n-1 }Z)\|_{L^2}\\ \n
&+B_{0,n }\| \pav^2(\chi_n q^n) \pav (\chi_{n-1}\pav^{n-1 }Z)\|_{L^2}\\ \n
\lesssim&B_{0,n }\|J^{(0,2,0)}_{0,n} Z\|_{L^2}+B_{0,n}\ n^{1+\sigma}\ \|\mathbbm{1}_{\mathrm{supp}\chi_n}q^{n-1}\pav^2(\chi_{n-1}\pav^{n-1}Z)\|_{L^2}\\ \n
&+B_{0,n }\myr{ n^{2+2\sigma}\|\mathbbm{1}_{\mathrm{supp} \chi_n}q^{(n-2)_+}} \pav (\chi_{n-1}\pav^{n-1 }Z)\|_{L^2}\\ \n
\lesssim&B_{0,n }\|J^{(0,2,0)}_{0,n}Z\|_{L^2}+B_{0,n }n^{1+\sigma}\ \lf(\|\mathbbm{1}_{\mathrm{supp}\chi_n}  J^{(0,2,0)}_{0,n-1}Z\|_{L^2}+\lf\|\mathbbm{1}_{\mathrm{supp}\chi_n}[q^{n-1},\pav^2](\chi_{n-1} \pav^{n-1}Z)\rg\|_{L^2}\rg)\\ 
&+B_{0,n }\myr{ n^{2+2\sigma}\|\mathbbm{1}_{\mathrm{supp} \chi_n}q^{(n-2)_+}} \pav (\chi_{n-1}\pav^{n-1 }Z)\|_{L^2}=:T_1+T_2+T_3.\label{vy2_1est_4_pf}
\end{align}
The first term is consistent with the result \eqref{vy2_1est4}. Then we  can   estimate the last term $T_3$ in \eqref{vy2_1est_4_pf} with the relation \eqref{s:prime}, and the facts that $\sigma_\ast\geq 3\sigma$ to as follows,
\begin{align*}
T_3\lesssim& \mathbbm{1}_{n\geq 3}\frac{\lambda^s B_{0,n-1}}{n^{\sigma}}\lf\|\mathbbm{1}_{\mathrm{supp}\chi_n}\frac{n}{q}\pav\lf(\chi_{n-1} q^{n-1}\pav^{n-1}Z\rg)\rg\|_{L^2}+\mathbbm{1}_{n=2}B_{0,0}\|\mathbbm{1}_{\mathrm{supp}\chi_n}\pav^2 Z\|_{L^2}\\
&+\mathbbm{1}_{n=1}B_{0,0}\|\mathbbm{1}_{\mathrm{supp}\chi_n}\pav Z\|_{L^2}\\
=&\mathbbm{1}_{n\geq 3}\frac{\lambda^s B_{0,n-1}}{n^{\sigma}}\lf\|\mathbbm{1}_{\mathrm{supp}\chi_n}J^{(1,1,0)}_{0,n-1}Z\rg\|_{L^2}+\mathbbm{1}_{n=2}B_{0,0}\|\mathbbm{1}_{\mathrm{supp}\chi_n}\pav^2 Z\|_{L^2}+\mathbbm{1}_{n=1}B_{0,0}\|\mathbbm{1}_{\mathrm{supp}\chi_n}\pav Z\|_{L^2}.
\end{align*}
Now we invoke the relation \eqref{cm_qn_pavj} and the estimate \eqref{wtq} to obtain a bound for the second term
\begin{align*}
T_2\lesssim& \frac{\lambda^s B_{0,n-1} }{n^{\sigma}}\| \mathbbm{1}_{\mathrm{supp}\chi_n}J^{(0,2,0)}_{0,n-1}\|_{L^2}+B_{0,n}n^{2+\sigma}\| \mathbbm{1}_{\mathrm{supp}\chi_n}q^{(n-2)_+}\pav(\chi_{n-1}\pav^{n-1}Z)\|_{L^2}\\
&+\frac{\lambda^s B_{0,n-1} }{n^{\sigma}}(n-1)^2\| \mathbbm{1}_{\mathrm{supp}\chi_n}q^{(n-3)_+-(n-1)}(\chi_{n-1}q^{n-1}\pav^{n-1}Z)\|_{L^2}.
\end{align*}
We observe that the second term can be estimated in a similar fashion as $T_3$ in \eqref{vy2_1est_4_pf}. Hence, 
\begin{align*}
T_2\lesssim &\frac{\lambda^s B_{0,n-1} }{n^{\sigma}}\| \mathbbm{1}_{\mathrm{supp}\chi_n}J^{(0,2,0)}_{0,n-1}\|_{L^2}+\mathbbm{1}_{n\geq 3}\frac{\lambda^s B_{0,n-1}}{n^{\sigma}}\sum_{a+b=2}\lf\|\mathbbm{1}_{\mathrm{supp}\chi_n}J^{(a,b,0)}_{0,n-1}Z\rg\|_{L^2}+\mathbbm{1}_{n\in\{1,2\}}\sum_{\ell=0}^2\|\mathbbm{1}_{\mathrm{supp}\chi_1}\pav^\ell Z\|_{L^2}. 
\end{align*}

\end{proof}
\begin{lemma}\label{Lem:H&vy2-1}
Assume \eqref{asmp}.
The following estimates hold
\begin{align}\label{crd_implJ1}
\sum_{n=0}^\infty \sum_{a+b=1} B_{0,n}^2\varphi^{2n}\lf\|J^{(a,b,0)}_{0,n}(\wt \chi_1 (v_y^2-1))\rg\|_{2}^2\lesssim& \sum_{n=0}^\infty \sum_{b=0}^{1}B_{0,n}^2\varphi^{2n}\lf\|\mathbbm{1}_{{\rm supp}\{  \wt\chi_1\}}J_{0,n}^{(0,b,0)}H\rg \|_2^2\lesssim e^{-\nu^{-1/9}}(\mathcal{E}_H^{(\gamma)}+\mathcal{E}_H^{(\al)});\\
\sum_{n=0}^\infty \sum_{a+b=2} B_{0,n}^2\varphi^{2n}\|J^{(a,b,0)}_{0,n}(\wt \chi_1 (v_y^2-1))\|_{2}^2\lesssim&\sum_{n=0}^\infty\sum_{b=0}^{2 }B_{0,n}^2 \varphi^{2n}\lf\|\mathbbm{1}_{{\rm supp}\{\wt\chi_1\}}J_{0,n}^{(0,b,0)}H\rg \|_2^2\lesssim e^{-\nu^{-1/9}}\sum_{\iota=\al,
\gamma}(\mathcal{E}_H^{(\iota)}+\mathcal{D}_H^{(\iota)}). \label{crd_implJ2}
\end{align}
\end{lemma}
\begin{proof}

First of all, we observe that the last inequalities in \eqref{crd_implJ1} and \eqref{crd_implJ2} are natural consequences of the definitions of $\mathcal{E}^{(\al)}_H$ \eqref{E_coord} and $\mathcal{D}^{(\al)}_H$ \eqref{D_coord}. We can explicitly estimate the left hand side of \eqref{crd_implJ1} with the translation relation \eqref{JtoD_2} \siming{(Should also holds for $k=0$, Check!)} 
as follows
\begin{align*}
\text{L.H.S of }&\eqref{crd_implJ1}\lesssim \sum_{n=0}^\infty\sum_{b=0}^1 B_{0,n}^2\varphi^{2n}\lf\|\chi_{n}\pav^{b}\lf( q^n\pav^n(\wt \chi_1(v_y^2-1))\rg)\rg\|_{2}^2
\end{align*}We recall that $v_y^2-1=H(H+2).$ 
Hence to prove \eqref{crd_implJ1}, it is enough to prove
\begin{align}\label{crd_impl_1}
\sum_{n=0}^\infty\sum_{b=0}^1 B_{0,n}^2\varphi^{2n}\lf\|\chi_n\pav^b(q^n\pav^{n}(\wt \chi_1 H(H+2)))\rg\|_{2}^2\lesssim&\sum_{n=0}^\infty \sum_{b= 0}^1 B_{0,n}^2\varphi^{2n}\lf\|\mathbbm{1}_{{\rm supp}\ \wt\chi_1}J^{(0,b,0)}_{0,n} H\rg\|_{L^2}^2.
\end{align}
Similarly, by the translation  estimate \eqref{JtoD_2}, \siming{(Should also holds for $k=0$, Check!)} the following estimate guarantees \eqref{crd_implJ2}
\begin{align}
\sum_{n=0}^\infty B_{0,n}^2\varphi^{2n}\|\pav^2(\chi_nq^n\pav^{n}(\wt \chi_1H(H+2)))\|_{2}^2\lesssim&\sum_{n=0}^\infty \sum_{b=0}^2 B_{0,n}^2\varphi^{2n}\|\mathbbm{1}_{{\rm supp}\ \wt\chi_1}J^{(0,b,0)}_{0,n} H\|_{L^2}^2.\label{crd_impl_2} 
\end{align}
We comment that the estimates \eqref{crd_impl_1} and \eqref{crd_impl_2} are simple exercises in Gevrey product rule and we omit further details for the sake of brevity.

\siming
{HS: For checking purpose: 

\noindent
{\bf Step \# 2: Proof of the estimate \eqref{crd_impl_1}.}  Consider two cases, $n=0$ and $n\geq 1$. In the first case, the estimate is direct consequence of the assumption \eqref{asmp}:
\begin{align*}
B_0^2\| &\pav(\wt\chi_1(v_y^2-1))\|_2^2=B_0^2\|\pav(\wt \chi_1H(H+2))\|_2^2\\
\lesssim&\|\pav \wt\chi_1\|_\infty^2\|\mathbbm{1}_{{\rm supp}\ \wt\chi_1}H\|_{2}^2(\|\mathbbm{1}_{{\rm supp}\ \wt\chi_1}H\|_{\infty}^2+1)+ \| \wt\chi_1\pav H\|_{2}^2(\|H\|_{\infty}^2+1)\\
\lesssim&\|\mathbbm{1}_{{\rm supp}\ \wt\chi_1}J^{(0,0,0)}_{0,0}H\|_{2}^2+\|\mathbbm{1}_{{\rm supp}\ \wt\chi_1}J^{(0,1,0)}_{0,0}H\|_{2}^2.
\end{align*}
This is consistent with \eqref{crd_impl_1}. For the $n\geq 1$ case, the general scheme is to replace $\chi_n$ by $\wt \chi_1^2\chi_n$ and  use the product rule. For the factor $\pav H$, we measure it in $L^2$, and for the factor $H$, we put it in $L^\infty$ and invoke embedding theorem. We decompose the left hand side of \eqref{crd_impl_1} as follows,
\begin{align}\n
\sum_{n=1}^\infty& B_{0,n}^2\varphi^{2n}\|\wt\chi_1^2\chi_n \pav(q^n\pav^{n}(\wt \chi_1 H(H+2)))\|_{2}^2\\
\lesssim&\sum_{n=1}^\infty B_{n}^2\varphi^{2n}\|\wt \chi_1^2\chi_n\pav(q^n\pav^{n}(\wt \chi_1 H^2 ))\|_{2}^2+\sum_{n=1}^\infty B_{n}^2\varphi^{2n}\|\wt\chi_1^2\chi_n\pav(q^n\pav^{n}(\wt \chi_1 H ))\|_{2}^2=:T_1+T_2.\label{crd_implT12} 
\end{align}
Now we focus on the first term $T_1$, the estimate of the $T_2$-term is similar. \siming{(Wikipedia)} We recall the following general Leibniz rule: 
\begin{align}
\pav^n(f_1f_2\cdots f_m)=\sum_{n_1+n_2+\cdots +n_m=n}\binom{n}{n_1,n_2,\cdots,n_m}\prod_{1\leq i\leq m}f_i^{(n_i)}.
\end{align}
Here $\binom{n}{n_1,n_2,\cdots, n_m}=\frac{n!}{n_1!n_2!\cdots n_m!}$ are the multinomial coefficients. As a result, 
\begin{align}\n
T_1
\lesssim &\sum_{n=0}^\infty\bigg(\sum_{n_1+n_2+n_3= n}\binom{n}{n_1,n_2,n_3} \frac{B_{0,n}}{B_{0,n_1}B_{0,n_2}B_{0,n_3}}\\ \n
&\quad\times\lf(B_{0,n_1}\|q^{n_1}\pav^{n_1}\wt \chi_1\|_\infty\rg)\lf(B_{0,n_2}\varphi^{n_2}\|{\wt \chi_1}\chi_{n_2}q^{n_2}\pav^{n_2}H\|_{\infty}\rg)\lf(B_{0,n_3}\varphi^{n_3}\|{\wt \chi_1}\pav(\chi_{n_3}q^{n_3}\pav^{n_3}H)\|_{2}\rg)\bigg)^2\\ 
&+\sum_{n=0}^\infty\bigg(\sum_{n_1+n_2+n_3= n}\binom{n}{n_1,n_2,n_3} \frac{B_{0,n}}{B_{0,n_1}B_{0,n_2}B_{0,n_3}}\n \\ 
&\quad\times\lf(B_{0,n_1}\|\pav(q^{n_1}\pav^{n_1}\wt \chi_1)\|_\infty\rg)\lf(B_{0,n_2}\varphi^{n_2}\|\wt \chi_1\chi_{n_2}q^{n_2}\pav^{n_2}H\|_{\infty}\rg)\lf(B_{0,n_3}\varphi^{n_3}\|{\wt \chi_1}\chi_{n_3}q^{n_3}\pav^{n_3}H\|_{2}\rg)\bigg)^2\n \\
=:&T_{11}+T_{12}.\label{crd_implT112}
\end{align}
Now we use the fact that 
\begin{align*}
\binom{n}{n_1,n_2,n_3}&\frac{B_{0,n}}{B_{0,n_1}B_{0,n_2}B_{0,n_3}}=\frac{n!}{n_1!n_2!n_3!}\frac{(n_1!)^s(n_2!)^s(n_3!)^s}{(n!)^{s}}=\binom{n}{n_1, n_2, n_3}^{-\sigma-\sigma_\ast}\\
=&\binom{n}{n_1}^{-\sigma-\sigma_\ast}\binom{n-n_1}{n_2}^{-\sigma-\sigma_\ast}\binom{n-n_1-n_2}{n_3}^{-\sigma-\sigma_\ast}=\binom{n}{n_1}^{-\sigma-\sigma_\ast}\binom{n-n_1}{n_2}^{-\sigma-\sigma_\ast}.
\end{align*}
Now we estimate the $T_{11}$ term in \eqref{crd_implT112} as follows
\begin{align*}
T_{11}\lesssim&\sum_{n=1}^\infty\bigg(B_{0,n-n_1-n_2}\|\pav^{n-n_1-n_2}\wt \chi_1\|_\infty \sum_{n_1=0}^{n}\binom{n}{n_1}^{-\sigma-\sigma_\ast}\lf(B_{0,n_1}\varphi^{n_1}\|\pav( { \wt \chi_1}\chi_{n_1}q^{n_1}\pav^{n_1}H)\|_{2}\rg)\\ 
&\quad\qquad\qquad\times\sum_{n_2=0}^{n-n_1} \binom{n-n_1}{n_2}^{-\sigma-\sigma_\ast}\lf(B_{0,n_2}\varphi^{n_2}\| { \wt \chi_1}\pav(\chi_{n_2}q^{n_2}\pav^{n_2}H)\|_{2}\rg)\bigg)^2. 
\end{align*}Now we invoke the estimate \eqref{vy2_1est3} in Lemma \ref{lem:vy2-1est} to obtain \siming{(Check the low $n$ case. We should be able to use the Sobolev estimates of $H$ to get similar outcome.)}
\begin{align*}
T_{11}\lesssim&\sum_{n=1}^\infty\Bigg[\lf(\sum_{n_1=0}^{n}\binom{n}{n_1}^{-2\sigma-2\sigma_\ast}\rg)\Bigg(\sum_{n_1=0}^{n}\sum_{b=0}^1B_{0,n_1}^2\varphi^{2n_1}\|\mathbbm{1}_{\text{supp} \wt \chi_1}J^{(0,b,0)}_{0,n_1}H\|_{2}^2\\ 
&\quad\times\lf(\sum_{n_2=0}^{n-n_1} \binom{n-n_1}{n_2}^{-2\sigma-2\sigma_\ast}\rg)\lf(\sum_{n_2=0}^{n-n_1}\sum_{b'=0}^1B_{0,n_2}^2\varphi^{2n_2}\|\mathbbm{1}_{\text{supp} \wt \chi_1}J^{(0,b',0)}_{0,n_2} H\|_{2}^2 \rg)\Bigg)\Bigg]\\
&\qquad\times B_{0,n-n_1-n_2}^2\|\pav^{n-n_1-n_2}\wt \chi_1\|_\infty^2 \siming{\myr{??}}\\
\lesssim&\sum_{n=0}^\infty  \lf(\sum_{n_1=0}^{n}\sum_{b=0}^1B_{0,n_1}^2\varphi^{2n_1}\|\mathbbm{1}_{\text{supp} \wt \chi_1}J^{(0,b,0)}_{0,n_1}H\|_{2}^2\rg)\lf(\sum_{n_2=0}^{n-n_1}\sum_{b'=0}^1B_{0,n_2}^2\varphi^{2n_2}\|\mathbbm{1}_{\text{supp} \wt \chi_1}J^{(0,b',0)}_{0,n_2}H\|_{2}^2 \rg)\\ 
&\quad\times\lf(B_{0,n-n_1-n_2}^2\|\pav^{n-n_1-n_2}\wt \chi_1\|_\infty^2\rg)\\
\lesssim& \lf(\sum_{n_1=0}^{\infty}\sum_{b=1}B_{0,n_1}^2\varphi^{2n_1}\|\mathbbm{1}_{\text{supp} \wt \chi_1} J^{(0,b,0)}_{0,n_1}H\|_{2}^2\rg)\lf(\sum_{n_2=0}^{\infty}\sum_{b'=1}B_{0,n_2}^2\varphi^{2n_2}\|\mathbbm{1}_{\text{supp} \wt \chi_1}J^{(0,b',0)}_{0,n_2}H\|_{2}^2 \rg)\\ 
&\quad\times \lf(\sum_{n=n_1+n_2}^\infty  B_{0,n-n_1-n_2}^2\|\pav^{n-n_1-n_2}\wt \chi_1\|_\infty^2\rg)\\
\lesssim& \lf(\sum_{n=0}^{\infty}\sum_{b=0}^{1}B_{0,n}^2\varphi^{2n}\|\mathbbm{1}_{\text{supp} \wt \chi_1}J^{(0,b,0)}_{0,n}H\|_{2}^2\rg)^2. 
\end{align*}
Here, in the last line, we have invoked the fact that $\wt\chi_1$ is very smooth. This is consistent with \eqref{crd_impl_1}. 

To treat the $T_{12}$-term in \eqref{crd_implT112}, we observe that since $\pa_y\wt\chi_1\equiv 0$ on the support of $q'$, hence $\sup_n B_{0,n}\|\pav(q^{n}\pav^{n}\wt\chi_1)\|_\infty\lesssim 1.$ Now by similar argument as before, 
 we obtain the bound
\begin{align*}
T_2+T_{12}\lesssim \sum_{n=0}^{\infty}\sum_{b=0}^{1}B_{0,n}^2\varphi^{2n}\|\mathbbm{1}_{\text{supp} \wt \chi_1}J^{(0,b,0)}_{0,n}H\|_{2}^2.
\end{align*}
Combining these estimates and the decomposition \eqref{crd_implT12}, \eqref{crd_implT112}, we obtain the estimate \eqref{crd_impl_1}. 

\noindent
{\bf Step \# 3: Proof of \eqref{crd_impl_2}.} The proof of \eqref{crd_impl_2} is similar. Consider two cases, $n=0$ and $n\geq 1$. In the first case, the estimate is direct a consequence of the assumption \eqref{asmp}:
\begin{align*}
B_0^2\|&\pa_{v}^2(\wt\chi_1(v_y^2-1))\|_2^2=B_0^2\|\pav^2(\wt \chi_1H(H+2))\|_2^2\\
\lesssim&\|\pav^2 \wt\chi_1\|_\infty^2\|\mathbbm{1}_{{\rm supp}\ \wt\chi_1} H\|_{2}^2(\|H\|_{\infty}^2+1) 
 +\|\pav \wt\chi_1\|_\infty^2\|\mathbbm{1}_{{\rm supp}\ \wt\chi_1}\pa_{v}H\|_{2}^2(\| H\|_{\infty}^2+1)\\ &+ \| \wt\chi_1\pav^2 H\|_{2}^2(\|H\|_{\infty}^2+1)+\| \wt\chi_1\pav H\|_{2}^2\|\pav H\|_{\infty}^2\\
\lesssim&\sum_{b=0}^2\|\mathbbm{1}_{{\rm supp}\ \wt\chi_1}J^{(0,b,0)}_{0,0}H\|_{2}^2.
\end{align*}
The remaining estimates are similar to step \# 1 and we omit further details. \siming{Check!}}
\end{proof}

To conclude the subsection, we present the last two technical lemma.
\begin{lemma}[Gevrey Product Rule]\label{lem:Gprod} Consider two Gevrey smooth functions $f,g\in G^{r;\Lambda}([-1,1])$. Let \[
\overline{B}_{m,n} :=\frac{\Lambda^{m+n}}{|m,n|!^{1/r}}.\]
The following two product estimates hold
\begin{align}\n
\sum_{m+n=0}^\infty &\overline B_{m,n}^2\varphi^{2n}\|q^n\pa_x^m\Gamma^n(fg)\|_{L^2}^2\\ \n
\lesssim &\lf(\sum_{m+n=0}^\infty\overline B_{m,n}^2\varphi^{2n}\|q^n\pa_x^m\Gamma^nf\|_{L^2}^2\rg)\lf(\sum_{m+n=0}^\infty\overline B_{m,n}^2c^{2n}\varphi^{2n}\|\pa_y(q^n\pa_x^m\Gamma^n g)\|_{L^2}^2\rg)\\
&+\lf(\sum_{m+n=0}^\infty\overline B_{m,n}^2c^{2n}\varphi^{2n}\|\pa_y(q^n\pa_x^m\Gamma^nf)\|_{L^2}^2\rg)\lf(\sum_{m+n=0}^\infty\overline B_{m,n}^2\varphi^{2n}\|q^n\pa_x^m\Gamma^n g\|_{L^2}^2\rg);\label{ga_prod}\\
\n
\sum_{m+n=0}^\infty &\overline B_{m,n}^2\varphi^{2n}\|\pa_y(q^n\pa_x^m\Gamma^n(fg))\|_{L^2}^2\\ 
\lesssim &\lf(\sum_{m+n=0}^\infty\overline B_{m,n}^2\varphi^{2n}\|\pa_y(q^n\pa_x^m\Gamma^nf)\|_{L^2}^2\rg)\lf(\sum_{m+n=0}^\infty\overline B_{m,n}^2\varphi^{2n}\|\pa_y(q^n\pa_x^m\Gamma^n g)\|_{L^2}^2\rg) .\label{al_prod}
\end{align}
Here $0<c=c(r)<1$.
 
In general, for $1/\mathfrak{p}_1+1/\mathfrak{q}_1=1/\mathfrak{p}_2+1/\mathfrak{q}_2=1/\mathfrak{r}$,\begin{align}\n
\sum_{m+n=0}^\infty &\overline B_{m,n}^2\varphi^{2n}\|q^n\pa_x^m\Gamma^n(fg)\|_{L^\mathfrak{r}}^2\\ \n
\lesssim &\lf(\sum_{m+n=0}^\infty\overline B_{m,n}^2\varphi^{2n}\|q^n\pa_x^m\Gamma^nf\|_{L^{\mathfrak p_1}}^2\rg)\lf(\sum_{m+n=0}^\infty\overline B_{m,n}^2c^{2n}\varphi^{2n}\|q^n\pa_x^m\Gamma^n g\|_{L^{\mathfrak{q}_1}}^2\rg)\\
&+\lf(\sum_{m+n=0}^\infty\overline B_{m,n}^2c^{2n}\varphi^{2n}\|q^n\pa_x^m\Gamma^nf\|_{L^{\mathfrak{p}_2}}^2\rg)\lf(\sum_{m+n=0}^\infty\overline B_{m,n}^2\varphi^{2n}\|q^n\pa_x^m\Gamma^n g\|_{L^{\mathfrak{q}_2}}^2\rg).\label{prod_gen}
\end{align}Here $0<c=c(r)<1$.
\end{lemma}\begin{remark}
We comment that the lemma has another form in our companion paper \cite{BHIW24a}, see, e.g., Lemma 6.27.  
\end{remark}
\begin{proof}  Throughout the proof, we use the shorthand notation $|m,n|:=m+n$. 
We recall the relation 
\begin{align*}
\binom{|m,n|}{|m',n'|}\geq \binom{m}{m'}\binom{n}{n'},\quad m'\leq m,\ n'\leq n.
\end{align*}
Now we prove the estimate \eqref{prod_gen}. To this end, we consider a finite sum (Here 
 $0<\zeta\leq \frac{1}{16}(1/r-1)$) and apply the combinatorial bound \eqref{prod}, 
\begin{align*}
&\sum_{m+n=0}^M\overline{B}_{m,n}^2\varphi^{2n} \|\pa_x^mq^n\Gamma^n(fg)\|_{L^\mathfrak{r}}^2\\
&\lesssim\sum_{m=0}^M\sum_{n=0}^{M-m}\frac{\Lambda^{2m+2n}}{|m,n|!^{2/r}}\varphi^{2n}\lf(\sum_{m'=0}^m\sum_{n'=0}^n\binom{m}{m'}^{1+\zeta-\zeta}\binom{n}{n'}^{1+\zeta-\zeta}\|\pa_{x}^{m'}q^{n'}\Gamma^{n'}f\pa_{x}^{m-m'}q^{n-n'}\Gamma^{n-n'}g\|_{L^\mathfrak{r}}\rg)^2\\
&\lesssim\sum_{m=0}^M\sum_{n=0}^{M-m}\sum_{\substack{m'\leq m,\ n'\leq n\\|m',n'|\leq 1/2|m,n|}}\frac{\Lambda^{2m+2n}}{|m,n|!^{2/r}}\varphi^{2n}\binom{m}{m'}^{2+2\zeta}\binom{n}{n'}^{2+2\zeta}\|\pa_{x}^{m'}q^{n'}\Gamma^{n'}f\|_{L^{\mathfrak{p_1}}}^2\|\pa_{x}^{m-m'}q^{n-n'}\Gamma^{n-n'}g\|_{L^{\mathfrak{q}_1}}^2\\
&\quad+\sum_{m=0}^M\sum_{n=0}^{M-m}\sum_{\substack{m'\leq m,\ n'\leq n\\|m',n'|> 1/2|m,n|}}\frac{\Lambda^{2m+2n}}{|m,n|!^{2/r}}\varphi^{2n}\binom{m}{m'}^{2+2\zeta}\binom{n}{n'}^{2+2\zeta}\|\pa_{x}^{m'}q^{n'}\Gamma^{n'}f\|_{L^{\mathfrak{p}_2}}^2\|\pa_{x}^{m-m'}q^{n-n'}\Gamma^{n-n'}g\|_{L^{\mathfrak{q}_2}}^2\\
&=: T_1+T_2.
\end{align*}
We focus on the $T_1$ term, and treat the $T_2$ term in a similar fashion. 
\ifx We estimate the $T_1$-term as follows:
\begin{align*}
T_1\lesssim&\sum_{|m,n|=0}^M \sum_{|m',n'|=0}^{ |m,n| }\frac{\overline{B}_{m,n}^2}{\overline{B}_{m-m',n-n'}^2\overline{B}_{m',n'}^2 }\binom{|m,n|}{|m',n'|}^{2+2\zeta}\\
&\times\overline{B}_{m',n'}^2\varphi^{2n'}\|\pa_{x}^{m'}q^{n'}\Gamma^{n'}f\|_{\dot H_y^1}^2\overline{B}_{m-m',n-n'}^2\varphi^{2n-2n'}\|\pa_{x}^{m-m'}q^{n-n'}\Gamma^{n-n'}g\|_{L^2}^2\\
\lesssim&\sum_{|m,n|=0}^M\sum_{|m',n'|=0}^{|m,n|}\binom{|m,n|}{|m',n'|}^{-2(1/r-1)+2\zeta}\overline{B}_{m',n'}^2\varphi^{2n'}\|\pa_{x}^{m'}q^{n'}\Gamma^{n'}f\|_{\dot H_y^1}^2\overline{B}_{m-m',n-n'}^2\varphi^{2n-2n'}\|\pa_{x}^{m-m'}q^{n-n'}\Gamma^{n-n'}g\|_{L^2}^2\\
\lesssim&\lf(\sum_{|m,n|=0}^M\overline{B}_{m,n}^2\varphi^{2n}\|\pa_{x}^{m}q^{n}\Gamma^{n}f\|_{\dot H^1}^2\rg)\lf(\sum_{|m,n|=0}^{M}\overline{B}_{m,n}^2\varphi^{2n}\|\pa_{x}^{m}q^{n}\Gamma^{n}g\|_{L^2}^2\rg). 
\end{align*}\myr{To obtain better estimate, we do the following}\fi  

We further decompose the $T_1$ term as follows:\begin{align*}
T_1\leq&\sum_{|m,n|=0}^M\sum_{|m',n'|=0}^{|m,n|}\mathbbm{1}_{m'\leq m,\ n'\leq n}\frac{\overline{B}_{m,n}^2}{\overline{B}_{m-m',n-n'}^2\overline{B}_{m',n'}^2 }\binom{|m,n|}{|m',n'|}^{2+2\zeta}\\
&\times\overline{B}_{m',n'}^2\varphi^{2n'}\|\pa_{x}^{m'}q^{n'}\Gamma^{n'}f\|_{L^{\mathfrak{p}_1}}^2\overline{B}_{m-m',n-n'}^2\varphi^{2n-2n'}\|\pa_{x}^{m-m'}q^{n-n'}\Gamma^{n-n'}g\|_{L^{\mathfrak{q}_1}}^2\\
\lesssim&\sum_{|m,n|=0}^M\lf(\sum_{|m',n'|=0}^{\min\{N,\lfloor|m,n|/8\rfloor\}}+\sum_{\min\{N,\lfloor|m,n|/8\rfloor\}+1}^{\lceil|m,n|/8\rceil}+\sum_{|m',n'|=\lceil|m,n|/8\rceil}^{|m,n|}\rg)\mathbbm{1}_{m'\leq m,\ n'\leq n}\binom{|m,n|}{|m',n'|}^{-2(1/r-1)+2\zeta}\\
&\times\overline{B}_{m',n'}^2\varphi^{2n'}\|\pa_{x}^{m'}q^{n'}\Gamma^{n'}f\|_{L^{\mathfrak{p}_1}}^2\overline{B}_{m-m',n-n'}^2\varphi^{2n-2n'}\|\pa_{x}^{m-m'}q^{n-n'}\Gamma^{n-n'}g\|_{L^{\mathfrak{q}_1}}^2\\
=:&T_{11}+T_{12}+T_{13}.
\end{align*}
We will pick a constant $N=N(r)$ large. The estimate of the $T_{11}$-term is direct:  
\begin{align*}
T_{11}
\leq& C(N,c)\sum_{{|m,n|}=0}^{M}\sum_{{|m',n'|}=0}^{N} \mathbbm{1}_{m'\leq m,\ n'\leq n}\overline{B}_{m',n'}^2c^{2n'}\varphi^{2n'}\|\pa_x^{m'}q^{n'}\Gamma^{n'} f\|_{L^{\mathfrak{p}_1}}^2\\ &\qquad\qquad\qquad\qquad\times \overline{B}_{m-m',n-n'}^2 \varphi^{2(n-n')}\|\pa_x^{m-m'}q^{n-n'}\Gamma^{n-n'} g\|_{L^{\mathfrak{q}_2}}^2\\
\leq &C(N,c)\lf(\sum_{|m,n|=0}^M\overline{B}_{m,n}^2c^{2n}\varphi^{2n}\|\pa_{x}^{m}q^{n}\Gamma^{n}f\|_{L^{\mathfrak{p}_1}}^2\rg)\lf(\sum_{|m,n|=0}^{M}\overline{B}_{m,n}^2\varphi^{2n}\|\pa_{x}^{m}q^{n}\Gamma^{n}g\|_{L^{\mathfrak{q}_1}}^2\rg).
\end{align*}
Here in the last line, we commute the summation order. 
Now we estimate the $T_{12}$ term. We note the following relation 
\begin{align*}
\binom{|m,n|}{|m',n'|}^{-1}\mathbbm{1}_{|m',n'|\leq |m,n|/8}=\frac{|m',n'|!}{|m,n|(|m,n|-1)\cdots(|m,n|-|m',n'|+1)}\mathbbm{1}_{|m',n'|\leq |m,n|/8}\leq  {7}^{-|m',n'|}.
\end{align*}As a result, we obtain 
\begin{align*}T_{12}\lesssim& \sum_{{|m,n|}=0}^{M}\sum_{{|m',n'|}=0}^{\lfloor |m,n|/8\rfloor}\mathbbm{1}_{m'\leq m,\ n'\leq n}7^{-2|m',n'|(1/r-1-\zeta)}\overline{B}_{m',n'}^2\varphi^{2n'}\|\pa_x^{m'}q^{n'}\Gamma^{n'} f\|_{L^{\mathfrak{p}_1}}^2\\
&\times \overline{B}_{m-m',n-n'}^2 \varphi^{2(n-n')}\|\pa_x^{m-m'}q^{n-n'}\Gamma^{n-n'} g\|_{L^{\mathfrak{q}_1}}^2\\ 
\lesssim &\lf(\sum_{|m,n|=0}^M\overline{B}_{m,n}^2c^{2n}\varphi^{2n}\|\pa_{x}^{m}q^{n}\Gamma^{n}f\|_{L^{\mathfrak{p}_1}}^2\rg)\lf(\sum_{|m,n|=0}^{M}\overline{B}_{m,n}^2\varphi^{2n}\|\pa_{x}^{m}q^{n}\Gamma^{n}g\|_{L^{\mathfrak{q}_1}}^2\rg).
\end{align*} 
Finally, we estimate $T_{13}$. Note the Stirling's formula $|m,n|!\sim |m,n|^{|m,n|+1/2}e^{-|m,n|}$ implies that for ${|m,n|},{|m',n'|},{|m,n|}-{|m',n'|}\geq N\gg 1,{1/r}>1$ 
\begin{align*}&\binom{|m,n|}{|m',n'|}^{-{1/r}+1+\zeta}\lesssim \lf(\frac{{|m',n'|}^{{|m',n'|}+1/2}({|m,n|}-{|m',n'|})^{{|m,n|}-{|m',n'|}+1/2}}{{|m,n|}^{{|m,n|}+1/2}}\rg)^{{1/r}-1-\zeta}\\
&\lesssim \lf(\frac{7}{8}\rg)^{{|m,n|}({1/r}-1-\zeta)}(\sqrt{{|m,n|}})^{{1/r}-1-\zeta}\lesssim  \lf(\frac{7}{8}\rg)^{\frac{{|m,n|}({1/r}-1-\zeta)}{2}} ,\quad {|m',n'|},|{|m,n|}-{|m',n'|}|\leq 7{|m,n|}/8.
\end{align*}
Hence,  \begin{align*}T_{13}
\lesssim &\lf(\sum_{|m,n|=0}^M\overline{B}_{m,n}^2c^{2n}\varphi^{2n}\|\pa_{x}^{m}q^{n}\Gamma^{n}f\|_{L^{\mathfrak{p}_1}}^2\rg)\lf(\sum_{|m,n|=0}^{M}\overline{B}_{m,n}^2\varphi^{2n}\|\pa_{x}^{m}q^{n}\Gamma^{n}g\|_{L^{\mathfrak{q}_1}}^2\rg).
\end{align*}  Apply similar argument to the $T_2$ term yields the similar bounds. 
Combining the estimates above and taking $M\rightarrow\infty$ yield \eqref{prod_gen}. \myr{The estimate \eqref{ga_prod} can be derived using \eqref{prod_gen} and  Sobolev embedding.} The estimate for the \eqref{al_prod} is similar. 
\end{proof} 
\begin{lemma} \label{lem:ExtToInt}
  Consider $s\in(1,5/4),\, t\in[0, \nu^{-1/3-\eta}],\, \eta \leq \frac{1}{78}$. Then, 
\begin{align} 
\sum_{m+n=0}^\infty& \frac{\wt \lambda ^{2(m+n)\wt s}}{ ((m+n)!)^{2\wt s}}\lf\||k|^m q^n\Gamma_k^nf_k\rg\|_{L^2([-1/2,1/2]^c)}^2  \lesssim \sum_{m+n=0}^\infty{\bf a}^2_{m,n}(t)\|\chi_{m+n}|k|^m q^n \Gamma_k^n f_k \|_{L^2}^2\exp\lf\{ \frac{1}{\mathfrak{C}}\nu^{-2/3+\eta}\rg\},\label{glu_rl}\\
&\hspace{5cm}\frac{1}{2}<\frac{1}{\wt s}=\frac{34}{19(2s +1)},\quad\wt\lambda =\mathcal{C}(\wt s,s,\mathfrak{C},\lambda ).\n
\end{align} 
Note that if $s\leq 5/4$, we have that  $1/\wt s =\frac{34}{19(2\frac{5}{4}+{1})}\geq \frac{68}{133}>0.51>1/2$. 
\end{lemma}
\begin{proof}
On the time interval  $t\in[0,\nu^{-1/3-\eta}]$, we have the following estimate:
\begin{align*}
\exp\lf\{\frac{1}{\mathfrak{C}}\nu^{-2/3+\eta}\rg\}\varphi^{2n+2}\geq &\exp\lf\{\frac{1}{\mathfrak{C}}\nu^{-2/3+\eta}\rg\}\varphi^{2(m+n)+2}\geq \frac{1}{C}\exp\lf\{\frac{1}{\mathfrak{C}}\nu^{-2/3+\eta}\rg\}(1+t^2)^{-m-n-1}\\
\geq & \frac{1}{C} \exp\lf\{\frac{1}{2\mathfrak{C}}\nu^{-2/3+\eta}\rg\}\ \frac{1}{N!}\frac{1}{(2\mathfrak{C})^N\nu^{\lf(\frac{2}{3}-\eta\rg)N}}\  \frac{1}{(1+t^2)^{ m+n }}\ \frac{1}{1+t^2}\\
\geq & \frac{1}{C}\exp\lf\{\frac{1}{4\mathfrak{C}}\nu^{-2/3+\eta}\rg\}\frac{1}{N!}\frac{1}{(2\mathfrak{C})^N2^{m+n}}\nu^{-\lf(\frac{2}{3}-\eta\rg)N+\lf(\frac{2}{3}+2\eta\rg)(m+n) }. 
\end{align*}
By setting $N=\lf\lfloor\frac{2/3+2\eta}{2/3-\eta}(m+n)\rg\rfloor$ and $\mathcal{G}=2\mathfrak C$, we have
\begin{align*}
\exp\lf\{\frac{1}{\mathfrak{C}}\nu^{-2/3+\eta}\rg\}\varphi^{2n+2}
\geq\frac{1}{C\mathcal{G}^{\lf\lfloor\frac{2/3+2\eta}{2/3-\eta}(m+n)\rg\rfloor}2^{m+n}\lf\lfloor \frac{2/3+2\eta}{2/3-\eta}(m+n)\rg\rfloor!}. 
\end{align*}
By invoking the facts that $\chi_{m+n}\equiv 1,\ \forall \ m+n\in \mathbb{N},  \ |y|\in[-1/2,1/2]^c$, we obtain the following
\begin{align}\sum_{m+n=0}^\infty&\frac{\lambda^{2(m+n)s}\varphi^{2n+2}}{((m+n)!)^{2s}}\| \chi_{m+n}|k|^m q^n \Gamma_k^n f_k \|_{L^2}^2 \exp\lf\{\frac{1}{\mathfrak{C}}\nu^{-2/3+\eta}\rg\} \n \\
\gtrsim& \sum_{ n+m=0}^\infty \frac{\lambda ^{2(m+n)s} }{\mathcal{G}^{\lf\lfloor\frac{2/3+2\eta}{2/3-\eta}(m+n)\rg\rfloor}2^{m+n}((m+n)!)^{2s}\lf\lfloor{\frac{2/3+2\eta}{2/3-\eta}(m+n)}\rg\rfloor!}\|\chi_{m+n}|k|^m \Gamma_k^nf_k\|_{L^2}^2\n \\
\gtrsim &\sum_{ n+m=0}^\infty \frac{\lambda^{2(m+n)s} }{(\sqrt{2}\mathcal{G})^{2(m+n)}\lf(\lf\lfloor\frac{2/3+2\eta}{2/3-\eta}(m+n)\rg\rfloor!\rg)^{2s+1}} \|\mathbbm{1}_{|y|\geq1/2} |k|^m q^n\Gamma_k^nf_k\|_{L^2}^2. \label{Gmm_fnc}
\end{align}

To estimate the right hand side of \eqref{Gmm_fnc}, we use the Gamma function $\Gamma(n)=(n-1)!,\quad n\in \mathbb{N}\backslash \{0\}$ and the log convexity of the Gamma function:
\begin{align*}
\Gamma(\theta x_1+(1-\theta)x_2)\leq \Gamma(x_1)^{\theta}\Gamma(x_2)^{1-\theta},\quad \theta\in[0,1], x_1,x_2>0.
\end{align*}
Since $\eta=\frac{1}{78} $,  so $\frac{2/3+2\eta}{2/3-\eta}=\frac{18}{17}$, and
\begin{align*}
\left\lfloor \frac{2/3+2\eta}{2/3-\eta}(m+n)\right\rfloor!=&\Gamma\lf(\lf\lfloor\frac{18}{17}(m+n)\rg\rfloor+1\rg)\\
\leq& \Gamma(m+n+1)^{\theta}\Gamma(2(m+n)+1)^{1-\theta},\quad \theta = \frac{2(m+n)-\lfloor\frac{18}{17}(m+n)\rfloor}{m+n}\geq\frac{16}{17}.
\end{align*}
Now we estimate the $\Gamma(2(m+n)+1)=(2(m+n))!$:
\begin{align*}
(2(m+n))!=\prod_{\ell_1=1}^{(m+n)}(2\ell_1)\prod_{\ell_2=0}^{(m+n-1)}(2\ell_2+1)\leq 2^{2(m+n)}((m+n)!)^2.
\end{align*}
Hence,
\begin{align*}
\left\lfloor \frac{2/3+2\eta}{2/3-\eta}(m+n)\right\rfloor!\leq& ((m+n)!)^{\theta}[(2(m+n))!]^{1-\theta} 
\leq  (m+n)![(2(m+n))!]^{\frac{1}{17}}\\
\leq & 2^{(m+n)\frac{ 2}{17} }((m+n)!)^{\frac{19}{17} }.
\end{align*}
Therefore, the right hand side of  \eqref{Gmm_fnc} has the following lower bound,
\begin{align*}
  \sum_{ m+n=0}^\infty &\frac{\lambda ^{2(m+n)s} }{(\sqrt{2}\mathcal{G})^{2(m+n)}\lf(\lf\lfloor\frac{2/3+2\eta}{2/3-\eta}(m+n)\rg\rfloor!\rg)^{2s+1}}
  \|\mathbbm{1}_{|y|  \geq1/2}|k|^m q^n\Gamma_k^nf_k\|_{L^2}^2\\
\gtrsim&\sum_{ m+n=0}^\infty \frac{\lambda ^{2(m+n)s} }{(\sqrt{2}\mathcal{G})^{2(m+n)} 2^{\frac{2}{17}(m+n)(2s+1)}((m+n)!)^{\frac{19}{17}(2s+1)}}\|\mathbbm{1}_{|y|\geq 1/2 } |k|^m q^n\Gamma_k^nf_k\|_{L^2}^2\\
\gtrsim&\sum_{ m+n=0}^\infty  \frac{ \lambda^{2(m+n) s}  }{ ({2}\mathcal{G})^{2(m+n)} ((m+n)!)^{\frac{19}{17}(2s+1)}}\|\mathbbm{1}_{|y|\geq 1/2} |k|^m q^n\Gamma_k^nf_k\|_{L^2}^2.
\end{align*} 
Hence the resulting Gevrey index $\wt s$ and radius of analyticity are 
$${2\wt s}=\frac{19}{17}\lf({2s}+1\rg)\Rightarrow \frac{1}{\wt s}=\frac{34}{19(2s+1)},\ \
\wt \lambda =\wt \lambda(\lambda,s,\wt s,\mathfrak C),$$ 
which completes the proof. 
\end{proof}

\ifx
\begin{lemma}Consider a function $g(\cdot)\in C_c^\infty(\mathbb{T}\times(-\pi,\pi))$. There exists a universal constant $\mathfrak{C}>1$ such that
\begin{align}\|\exp\{(|k|^2+|\eta|^2)^{r/2}\}\widehat{g}(k,\eta)\|_{L_\eta^2}^2
\lesssim\sum_{m+n=0}^\infty \frac{\mathfrak {C}^{m+n}}{((m+n)!)^{2/r}}\||k|^m|\eta|^n \widehat{g}(k,\eta)\|_{L_\eta^2}^2?
\end{align}
\end{lemma}
\begin{proof}%
First of all, we observe the relation $(x+y)^{r/2}\leq x^{r/2}+y^{r/2},\, x, y\geq 0,\, r\in(0,1)$. This is a consequence of the concavity of the function $(\cdot)^{r/2}$. 

We estimate the left hand side as follows,
\begin{align*}
\|\exp\{(|k|^2+|\eta|^2&)^{r/2}\}\widehat g( k,\eta)\|_{L^2_\eta}^2\leq \|\exp\{|k|^r+|\eta|^r\}\widehat g(k,\eta)\|_{L_\eta^2}^2\\
=&\lf\|\sum_{m=0}^\infty \sum_{n=0}^\infty \frac{|k|^{rm}}{m!}\frac{|\eta|^{rn}}{n!}\widehat{g}(k,\eta)\rg\|_{L_\eta^2}^2\\
\leq&\lf(\sum_{m=0}^\infty \sum_{n=0}^\infty \lf(\frac{2}{ 3 }\rg)^{m+n}\rg) \lf(\sum_{m=0}^\infty \sum_{n=0}^\infty\lf(\frac{3}{2}\rg)^{m+n}\frac{1}{(m!)^2(n!)^2} \lf\||k|^{rm}  |\eta|^{rn} \widehat{g}(k,\eta)\rg\|_{L_\eta^2}^2\rg)\\
\lesssim&  \lf(\sum_{m=0}^\infty \sum_{n=0}^\infty\lf(\frac{3}{2}\rg)^{m+n}\frac{1}{(m!)^2(n!)^2} \lf\||k|^{m}  |\eta|^{n} \widehat{g}(k,\eta)\rg\|_{L_\eta^2}^{2r}\|\widehat{g}(k,\eta)\|_{L_\eta^2}^{2-2r}\rg).
\end{align*}
Here in the last line, we use the interpolation, which is a direct consequence of H\"older inequality. Hence we will not lose $(m,n)$ dependent constants. As a result, 
\begin{align*}
\|\exp\{(|k|^2+|\eta|^2&)^{r/2}\}\widehat g( k,\eta)\|_{L_\eta^2}^2\\
\lesssim &\lf(\sum_{m=0}^\infty \sum_{n=0}^\infty \frac{ 3^{(m+n)/r}}{(m!)^{2/r}(n!)^{2/r}} \lf\||k|^{m}  |\eta|^{n} \widehat{g}(k,\eta)\rg\|_2^{2}\rg)^{r}\lf(\sum_{m=0}^\infty\sum_{n=0}^\infty \frac{1}{2^{(m+n)/(1-r)}}\|\widehat{g}(k,\eta)\|_{L_\eta^2}^{2}\rg)^{{1-r}}\\\lesssim &\sum_{m=0}^\infty \sum_{n=0}^\infty \frac{ 3^{(m+n)/r}}{(m+n)!^{2/r}}\underbrace{\frac{(m+n)!^{2/r}}{(m!)^{2/r}(n!)^{2/r}}}_{=\binom{m+n}{n}^{2/r}} \lf\||k|^{m}  |\eta|^{n} \widehat{g}(k,\eta)\rg\|_{L_\eta^2}^{2}.
\end{align*}
Now we note that $\sum_{\ell=0}^M \binom{M}{\ell}=2^{M}$, so $\binom{m+n}{n}\leq 2^{m+n}$. As a result,
\begin{align*}
\|\exp\{(|k|^2+|\eta|^2&)^{r/2}\}\widehat g( k,\eta)\|_{L_\eta^2}^2\lesssim \sum_{m=0}^\infty\sum_{n=0}^\infty \frac{12^{(m+n)/r}}{(m+n)!^{2/r}} \lf\||k|^{m}  |\eta|^{n} \widehat{g}(k,\eta)\rg\|_{L_\eta^2}^{2}.
\end{align*}
 \ifx
\begin{align}
\|\exp\{(|k|^2+|\eta|^2)^{r/2}\}\widehat g_k(\eta)\|_2^2=&\lf\|\sum_{M=0}^\infty\frac{(|k|^2+|\eta|^2)^{\frac{rM}{2}}}{M!}\widehat f_{k}(\eta)\rg\|_{2}^2\\
\lesssim&\sum_{M=0}^\infty\frac{(3/2)^{M}}{M!^2}\lf\|(|k|^2+|\eta|^2)^{\frac{rM}{2}}\widehat f_{k}(\eta)\rg\|_{2}^2\\
\lesssim&\sum_{\lceil rM\rceil=0}^\infty\frac{(3/2)^M}{(\lceil r M\rceil/r)!^2}\lf\|(|k|+|\eta|)^{{\lceil rM\rceil}{}}\widehat f_{k}(\eta)\rg\|_{2}^2.
\end{align}
Now we might need to use some sort of log convexity of the Gamma function to get that 
\begin{align}
\cdots\lesssim&\sum_{\lceil \wt M\rceil=0}^\infty\frac{(3/2)^{\wt M/r}}{\wt M !^{2/r}}\lf(\sum_{\ell=0}^{\wt M}\binom{\wt M}{\ell}\lf\||k|^{\ell}|\eta|^{\wt M -\ell} f_{k}(\eta)\rg\|_{2}\rg)^2\\
\lesssim&\sum_{\lceil \wt M\rceil=0}^\infty\frac{3^{\wt M/r}}{\wt M !^{2/r'}}\lf(\sum_{\ell=0}^{\wt M} \binom{\wt M}{\ell}^2\lf\||k|^{\ell}|\eta|^{\wt M -\ell} f_{k}(\eta)\rg\|_{2}^2\rg)\\
\lesssim&\sum_{\lceil \wt M\rceil=0}^\infty\frac{3^{\wt M/r}}{\wt M !^{2/r}}
\lf(\sum_{\ell=0}^{\wt M} 2^{2\wt M}\lf\||k|^{\ell}|\eta|^{\wt M -\ell} f_{k}(\eta)\rg\|_{2}^2\rg)\\
=&C\sum_{\lceil \wt M\rceil=0}^\infty\sum_{\ell=0}^{\wt M} \frac{3^{\wt M/r} 2^{2\wt M}}{\wt M !^{2/r}} \lf\||k|^{\ell}|\eta|^{\wt M -\ell} f_{k}(\eta)\rg\|_{2}^2 \\
=&C\sum_{m+n=0}^\infty \frac{\mathcal{C}^{m+n}}{(m+n) !^{2/r}} \lf\||k|^{m}\Gamma^{n } f_{k}\rg\|_{2}^2?
\end{align} \fi 
\end{proof}
\fi

\ifx
\subsubsection{Previous Argument for a Different Equation}
{\color{blue} For the $\Phi^E$, we can write down a similar equation as \eqref{T123}, 
\begin{align}
\de_k( |k|^m q^n\Gamma_k^n \Phi_k^{(E)}) 
=& |k|^mq^n\Gamma_k^n (\chi^E\ww_k)+|k|^m q^n[\pa_{yy}, \Gamma_k^n]\Phi_k^{(E)}\n \\
\n &+[\pa_{yy},q^n](|k|^m\Gamma_k^n \Phi_k^{(E)}) \\
=:&|k|^m q^n\Gamma_k^n( \chi^E\ww_k)+\mathbb{T}_{1;m,n}+\mathbb{T}_{2;m,n}.\label{PhiE_eq}
\end{align} 
Note that this time, we don't have $\chi_{m+n}$ commutators!! And that is the key. Since we have obtained the conormal almost analytic regularity of $\chi^E\ww_k$ in the exterior estimate, we can apply a similar argument to derive the Gevrey conormal regularity of $\Phi^E$ in the whole domain (since there is no $\chi_{m+n}$ cut-off.). Further note that when we really need to use this estimate of $\Phi^E$, the conormal weight $q$ is of order $1$! This will conclude the estimate of the influence from the boundary vorticity to the interior domain.  

\begin{proposition}\label{Pro:ext_to_int}
Consider the solution to the equation \eqref{PhiE_eq}. Assume the condition \eqref{sml_vy2-1}. Then the following estimate holds
\begin{align}\label{ell_extin_1}
\sum_{m+n=0}^\infty\sum_{a+b+c\leq 2}\mathbbm{1}_{\substack{a+b+c\leq n\\  \mathrm{or}\ a=0\ \ }} {\bf a}_{m,n}^2\lf\|\lf(\frac{m+n}{q}\rg)^a \pav^b|k|^c (|k|^m q^n\Gamma_k^n\Phi_k^{(E)})\rg\|_{L^2}^2\lesssim \sum_{m+n=0}^\infty {\bf a}_{m,n}^2\||k|^m q^n\Gamma_k^n(\chi^E \ww_k)\|_{L^2}^2.
\end{align}
Moreover, 
\begin{align}\label{ell_extin_2}
\sum_{m+n=0}^\infty (B_{m,n}^{in})^2\lf\|\mathbbm{1}_{|y|\leq \frac{31}{40}} |k|^m \Gamma_k^n\Phi_k^{(E)} \rg\|_{L^2}^2\lesssim \sum_{m+n=0}^\infty {\bf a}_{m,n}^2\||k|^m q^n\Gamma_k^n(\chi^E \ww_k)\|_{L^2}^2\exp\lf \{\frac{1}{\wt K\nu t}\rg\}.
\end{align}
Here $\wt K\geq 10 K$ ($K$ is defined in the weight \eqref{defndW}) is a universal constant and $B_{m,n}^{(in)}$ is defined in \eqref{B_in_mn}.  
\end{proposition}
\begin{remark}
The exterior vorticity has size $\leq Ce^{-\frac{5}{\wt K\nu t}}$. Hence, the contribution of $\Phi^E$ is exponentially small (${o}(e^{-\nu^{-1/9}/C})$) in $\nu$ in the  Gevrey-$1/r$ space.  %
\end{remark}
\begin{proof}
The proof of \eqref{ell_extin_1} is similar to the proof of \eqref{J_ph_est1}. \siming{(Check!)} We omit further details for the sake of brevity.

Next we focus on \eqref{ell_extin_2}. We consider the $a=b=c=0$ case in \eqref{ell_extin_1} and recall the fact that $q(y)=1,\ |y|\leq 31/40$ \myr{(???)} to obtain the bound 
\begin{align}
\sum_{m+n=0}^\infty \frac{\lambda^{2s(m+n)}}{((m+n)!)^{2s}}\varphi^{2n+2}\lf\|\mathbbm{1}_{|y|\leq 31/40}|k|^m \Gamma_k^n\Phi_k^{(E)}\rg\|_{L^2}^2\lesssim
\sum_{m+n=0}^\infty {\bf a}_{m,n}^2\||k|^mq^n\Gamma_k^n (\chi^E\ww_k)\|_2^2. \label{ell_exttoint2}
\end{align} 
Recall the definition of $B_{m,n}^{{in}}$ \eqref{B_in_mn}, we invoke the bound \eqref{ell_exttoint2} and a variant of the estimate \eqref{glu_rl} to obtain  
\begin{align*}
\sum_{m+n=0}^\infty (B_{m,n}^{in})^2\|\mathbbm{1}_{|y|\leq 31/40}|k|^m \Gamma_k^n\Phi^{E}_k\|_2^2\underbrace{\lesssim}_{\eqref{glu_rl}}&\sum_{m+n=0}^\infty {\bf a}_{m,n}^2\|\mathbbm{1}_{|y|\leq 31/40}|k|^m\Gamma_k^n\Phi^{E}_k\|_2^2\exp\lf\{\frac{1}{\wt K\nu t}\rg\}\\
\lesssim&\sum_{m+n=0}^\infty {\bf a}_{m,n}^2\||k|^m q^n\Gamma_k^n(\chi_{E}\ww_k)\|_2^2\exp\lf\{\frac{1}{\wt K\nu t}\rg\}.
\end{align*}
This concludes the proof of \eqref{ell_extin_2}.

\end{proof}
}
\fi


%
%
%
%
%
%




%
%
%
%
%
%


\vspace{2 mm}

\noindent \textbf{Acknowledgments:} JB was supported by NSF Award DMS-2108633. JB would also like to thank Ryan Arbon for helpful discussions. The work of  SH is supported in part by NSF grants DMS 2006660, DMS 2304392, DMS 2006372. SH would like to thank Ruth Luo for teaching him many fancy combinatoric tricks. SH would also like to thank Yiyue Zhang for many suggestion on Lie Algebra. The work of SI is supported by NSF DMS-2306528 and a UC Davis Society of Hellman Fellowship award. The work of FW is supported by the National Natural Science Foundation of China (No. 12101396, 12161141004, and 12331008).

\addcontentsline{toc}{section}{References}
\bibliographystyle{abbrv}
\bibliography{bibliography}

\end{document}